\documentclass[twoside,10pt,reqno]{amsart}
\usepackage{bbm}
\usepackage{amssymb}
\usepackage{fullpage}
\usepackage{graphicx}
\usepackage{multicol}
\usepackage[flushleft]{threeparttable}
\usepackage{booktabs,caption}
\usepackage{pdflscape}
\usepackage{longtable}
\usepackage{multirow}
\usepackage{makecell}
\usepackage{siunitx}
\usepackage{babel}
\usepackage{array, longtable, booktabs, makecell}

\usepackage{makecell}
\usepackage{xltabular}
\usepackage{chngcntr}
\usepackage{placeins}
\usepackage{multicol}
\usepackage{float}
\restylefloat{table}
\counterwithin{table}{section}
\usepackage{tikz}
\usetikzlibrary{arrows,decorations.markings}
\usetikzlibrary{graphs}
\usetikzlibrary{backgrounds}

\usepackage[pdftex,colorlinks,backref,pagebackref,hypertexnames=false,
           linkcolor=blue,urlcolor=blue,citecolor=blue]{hyperref}

\theoremstyle{remark}

\numberwithin{equation}{section}

\newcommand*{\scale}[2][4]{\scalebox{#1}{$#2$}}%

\DeclareMathOperator{\diag}{diag}
\DeclareMathOperator{\codim}{codim}

\newtheorem{prop}{Proposition}[section]

\newtheorem{rem}[prop]{Remark}

\newtheorem{lem}[prop]{Lemma}

\newtheorem{theorem}[prop]{Theorem}

\DeclareMathOperator{\GL}{GL}

\DeclareMathOperator{\SO}{SO}
\DeclareMathOperator{\Sp}{Sp}

\DeclareMathOperator{\ch}{ch}
\DeclareMathOperator{\rank}{rank}
\DeclareMathOperator{\ZG}{Z}

\DeclareMathOperator{\SW}{S^{2}}
\DeclareMathOperator{\SWT}{S^{3}}
\DeclareMathOperator{\id}{id}

\DeclareMathOperator{\chr}{char}
\DeclareMathOperator{\X}{X}
\DeclareMathOperator{\Y}{Y}
\DeclareMathOperator{\I}{I}

\makeatletter
\newcommand{\bigperp}{%
  \mathop{\mathpalette\bigp@rp\relax}%
  \displaylimits
}

\newcommand{\bigp@rp}[2]{%
  \vcenter{
    \m@th\hbox{\scalebox{\ifx#1\displaystyle2.1\else1.5\fi}{$#1\perp$}}
  }%
}
\makeatother

\title{\boldmath Minimum codimension of eigenspaces in irreducible representations of simple classical linear algebraic groups}
\author{Ana-M.~Retegan }
\address{School of Mathematics,
University of Birmingham,
Birmingham, B15 2TT,
UK}
\email{retegan.ana1@gmail.com}

\begin{document}

\begin{abstract}
Let $k$ be an algebraically closed field of characteristic $p \geq 0$, let $G$ be a simple simply connected classical linear algebraic group of rank $\ell$ and let $T$ be a maximal torus in $G$ with rational character group $\X(T)$. For a nonzero $p$-restricted dominant weight $\lambda\in \X(T)$, let $V$ be the associated irreducible $kG$-module. Define $\nu_{G}(V)$ to be the minimum codimension of eigenspaces corresponding to non-central elements of $G$ on $V$. In this paper, we calculate $\nu_{G}(V)$ for $G$ of type $A_{\ell}$, $\ell\geq 16$, and $\dim(V)\leq \frac{\ell^{3}}{2}$; for $G$ of type $B_{\ell}$, respectively $C_{\ell}$, $\ell\geq 14$, and $\dim(V)\leq 4\ell^{3}$; and for $G$ of type $D_{\ell}$, $\ell\geq 16$, and $\dim(V)\leq 4\ell^{3}$. Moreover, for the groups of smaller rank and their corresponding irreducible modules with dimension satisfying the above bounds, we determine lower-bounds for $\nu_{G}(V)$.
\end{abstract}

\maketitle

\section{Introduction}
Let $k$ be an algebraically closed field of characteristic $p\geq 0$, let $V$ be a finite-dimensional $k$-vector space and let $H$ be a group acting linearly on $V$. For $h\in H$ denote by $V_{h}(\mu)$ the eigenspace corresponding to the eigenvalue $\mu\in k^{*}$ of $h$ on $V$, and set $\nu_{H}(V)=\min\{\dim(V)-\dim(V_{h}(\mu))\mid h\in H\setminus \ZG(H) \text{ and }\mu\in k^{*}\}$. In \cite{Gor90}, one can find the classification of groups $H$ acting linearly, irreducibly and primitively on a vector space $V$ (over a field of characteristic zero) that contain an element $h$ for which $\nu_{H}(V)$ is small when compared to $\dim(V)$. The following year, Hall, Liebeck and Seitz, \cite{HLS_92}, expanded on Gordeev's result by working over algebraically closed fields of arbitrary characteristic, and they proved that, in the case of linear algebraic groups, if $H$ is classical, we have $\nu_{H}(V)\geq \frac{n}{8(2\ell+1)}$, where $\ell$ is the rank of $H$ and $V$ is a faithful rational irreducible $kH$-module of dimension $n$; while, if $H$ is not of classical type, then $\nu_{H}(V)>\frac{\sqrt{n}}{12}$. Now, with the lower-bounds for $\nu_{H}(V)$ known, the following natural step was to start the classification of pairs $(H,V)$ with bounded $\nu_{H}(V)$ from above, in particular the pairs $(H,V)$ with $\nu_{H}(V)=1$ or $\nu_{H}(V)=2$ have been of great interest, see for example \cite{Kac1982}, \cite{Kemper1997} and \cite{Verbitsky1999}. In \cite{GuralnickSaxl_2003}, the irreducible subgroups $H$ of $\GL(V)$, where $V$ is a finite-dimensional $k$-vector space of dimension $n>1$, which act primitively and tensor-indecomposably on $V$ and $\nu_{H}(V)\leq \max\{2,\frac{\sqrt{n}}{2}\}$ have been classified. 

Let $k$ be an algebraically closed field of characteristic $p\geq 0$, let $G$ be a simple simply connected classical linear algebraic group of rank $\ell$, $\ell\geq 1$, over $k$ and let $V$ be a nontrivial rational irreducible tensor-indecomposable $kG$-module. In this paper, we aim to determine $\nu_{G}(V)$ in the following cases:
\begin{enumerate}
\item  $G$ is of type $A_{\ell}$, $\ell\geq 16$, and $\dim(V)\leq \frac{\ell^{3}}{2}$;
\item  $G$ is of type $B_{\ell}$, $\ell\geq 14$, and $\dim(V)\leq 4\ell^{3}$;
\item  $G$ is of type $C_{\ell}$, $\ell\geq 14$, and $\dim(V)\leq 4\ell^{3}$;
\item  $G$ is of type $D_{\ell}$, $\ell\geq 16$, and $\dim(V)\leq 4\ell^{3}$.
\end{enumerate} 
Moreover, for the groups of smaller rank and their corresponding irreducible tensor-indecomposable modules with dimensions satisfying the above bounds, we find improved lower-bounds for $\nu_{G}(V)$. The origin of this paper is the PhD thesis of the author, in which the classification of pairs $(G,V)$ with $\nu_{G}(V)\leq \sqrt{\dim(V)}$ was established. We now state the main results of this paper. The notation used will be introduced in Section \ref{SectionNotation}.

\begin{theorem}\label{ResultsAl}
Let $k$ be an algebraically closed field of characteristic $p\geq 0$ and let $G$ be a simple simply connected linear algebraic group of type $A_{\ell}$, $\ell\geq 1$. Let $T$ be a maximal torus in $G$ and let $V=L_{G}(\lambda)$, where $\lambda\in \X(T)^{+}_{p}$, be a nontrivial irreducible $kG$-module with $\dim(V)\leq \frac{\ell^{3}}{2}$. Then $\scale[0.95]{\displaystyle \max_{s\in T\setminus\ZG(G)}\dim(V_{s}(\mu))}$, $\scale[0.95]{\displaystyle \max_{u\in G_{u}\setminus \{1\}}\dim(V_{u}(1))}$ and $\nu_{G}(V)$ are given Table \ref{TableAl}.

\begin{table}[h!]
\centering 
\begin{tabular}{| c | c | c | c | c | c | c |}
\hline
$V$ & Char. & Rank & $\displaystyle \max_{s\in T\setminus\ZG(G)}\dim(V_{s}(\mu))$ & $\displaystyle \max_{u\in G_{u}\setminus \{1\}}\dim(V_{u}(1))$ & $\nu_{G}(V)$ \\ \hline
$L_{G}(\omega_{1})$ & $p\geq 0$ & $\ell\geq 1$ &  $\ell$ & $\ell$ & $1$ \\ \hline
$L_{G}(\omega_{2})$ & $p\geq 0$ & $\ell\geq 3$ & $\frac{\ell(\ell-1)}{2}+\varepsilon_{\ell}(3)$ & $\frac{\ell^{2}-\ell+2}{2}$ & $\ell-1$ \\ \hline
$L_{G}(2\omega_{1})$ & $p\neq 2$ & $\ell\geq 1$ & $\frac{\ell^{2}+\ell+2}{2}$ & $\binom{\ell+1}{2}$ & $\ell$ \\ \hline
$L_{G}(\omega_{1}+\omega_{\ell})$ & $p\geq 0$ & $\ell\geq 2$ & $\ell^{2}-\varepsilon_{p}(\ell+1)+\varepsilon_{p}(3)\varepsilon_{\ell}(2)$ & $\ell^{2}-\varepsilon_{p}(\ell+1)$ & $2\ell-\varepsilon_{p}(3)\varepsilon_{\ell}(2)$  \\ \hline
$L_{G}(\omega_{3})$ & $p\geq 0$ & $\ell\geq 5$ & $\leq \binom{\ell}{3}+2$ & $\binom{\ell}{3}+\ell-1$ & $\binom{\ell-1}{2}$\\ \hline
$L_{G}(3\omega_{1})$ & $p\neq 2,3$ & $\ell\geq 1$ & $\binom{\ell+2}{3}+\ell$ & $\binom{\ell+2}{3}$ & $\frac{\ell^{2}+\ell+2}{2}$ \\ \hline
\multirow{2}{5.5em}{$L_{G}(\omega_{1}+\omega_{2})$} & $p=3$ & $\ell\geq 3$ & $\binom{\ell+2}{3}$ & $\binom{\ell+2}{3}-\ell+1$ & $\binom{\ell+1}{2}$ \\ \cline{2-6}
                  																	  & $p\neq 3$ & $\ell\geq 3$ & $ \frac{\ell^{3}+2\ell}{3}-\ell\varepsilon_{p}(2)$ & $\frac{\ell^{3}+2\ell}{3}$ & $ \ell^{2}$ \\ \hline
$L_{G}(\omega_{4})$ & $p\geq 0$ & $7\leq \ell\leq 14$ & $\leq \binom{\ell}{4}+2\ell-5$ & $\binom{\ell-1}{4}+\binom{\ell-1}{3}+\binom{\ell-1}{2}$ & $\binom{\ell-1}{3}$ \\ \hline
$L_{G}(\omega_{5})$& $p\geq 0$ & $9\leq \ell\leq 10$ & $\leq \binom{\ell}{5}+\ell^{2}-6\ell+9$ & $\binom{\ell-1}{5}+\binom{\ell-1}{4}+\binom{\ell-1}{3}$ & $\binom{\ell-1}{4}$ \\ \hline
\end{tabular}
\caption{\label{TableAl}The value of $\nu_{G}(V)$ for groups of type $A_{\ell}$.}
\end{table}
\end{theorem}

\begin{theorem}\label{ResultsCl}
Let $k$ be an algebraically closed field of characteristic $p\geq 0$ and let $G$ be a simple simply connected linear algebraic group of type $C_{\ell}$, $\ell\geq 2$. Let $T$ be a maximal torus in $G$ and let $V=L_{G}(\lambda)$, where $\lambda\in \X(T)^{+}_{p}$, be a nontrivial irreducible $kG$-module with $\dim(V)\leq 4\ell^{3}$. Then $\displaystyle \max_{s\in T\setminus\ZG(G)}\dim(V_{s}(\mu))$, $\displaystyle \max_{u\in G_{u}\setminus \{1\}}\dim(V_{u}(1))$ and $\nu_{G}(V)$ are as in Table \ref{TableCl}, where the $x_{i}$'s and $y_{i}$'s are given at the end of the table.

\begin{table}[h!]
\centering 
\begin{tabular}{| c | c | c | c | c | c | c |}
\hline
$V$ & Char. & Rank & $\displaystyle \max_{s\in T\setminus\ZG(G)}\dim(V_{s}(\mu))$ & $\displaystyle \max_{u\in G_{u}\setminus \{1\}}\dim(V_{u}(1))$ & $\nu_{G}(V)$ \\ \hline
$L_{G}(\omega_{1})$ &$p\geq 0$ & $\ell\geq 2$ &  $2\ell-2$ & $2\ell-1$ & $1$ \\ \hline
$L_{G}(\omega_{2})$ & $p\geq 0$ & $\ell\geq 2$ & $x_{1}$ & $y_{1}$ & $2\ell-2-\varepsilon_{\ell}(2)$ \\ \hline
$L_{G}(2\omega_{1})$ & $p\neq 2$ & $\ell\geq 2$ & $2\ell^{2}-3\ell+4$ & $2\ell^{2}-\ell$ & $2\ell$ \\ \hline
\multirow{3}{3.5em}{$L_{G}(\omega_{3})$} & \multirow{3}{3.5em}{$p\geq 0$} & $\ell=3$ & $10-6\varepsilon_{p}(2)$ & $9-3\varepsilon_{p}(2)$ & $4-2\varepsilon_{p}(2)$ \\ \cline{3-6}
									   &	& $\ell=4$ & $\scale[0.8]{28-2\varepsilon_{p}(3)-8\varepsilon_{p}(2)}$ & $\scale[0.8]{34-7\varepsilon_{p}(3)}$ & $14-\varepsilon_{p}(3)$ \\\cline{3-6}
& & $\ell\geq 5$ & $\leq x_{2}$ & $y_{2}$ &  $\scale[0.8]{2\ell^{2}-5\ell+2-\varepsilon_{p}(\ell-1)}$ \\ \hline					
$L_{G}(3\omega_{1})$ & $p\neq 2,3$ & $\ell\geq 2$ & $\binom{2\ell}{3}+3(2\ell-2)$ & $\binom{2\ell+1}{3}$ & $2\ell^{2}+\ell$ \\ \hline
\multirow{3}{4.5em}{$\scale[0.8]{L_{G}(\omega_{1}+\omega_{2})}$} & $p\geq 0$ & $\ell=2$ & $8-2\varepsilon_{p}(5)-2\varepsilon_{p}(2)$ & $8-3\varepsilon_{p}(5)$ & $8-2\varepsilon_{p}(5)$ \\\cline{2-6}
                           & $p=3$ & $\ell\geq 3$ & $\leq \frac{4\ell^{3}-6\ell^{2}+14\ell-12}{3}$ & $\frac{4\ell^{3}-7\ell+6}{3}$ & $2\ell^{2}+\ell-2$\\\cline{2-6}
                          & $p\neq 3$ & $\ell\geq 3$ & $\leq x_{3}$ & $\leq y_{3}$ & $\scale[0.8]{\geq 4\ell^{2}-4\ell-\varepsilon_{p}(2\ell+1)-2\ell\varepsilon_{p}(2)}^{\dagger}$ \\ \hline
$L_{G}(2\omega_{2})$& $p\neq 2$ & $\ell=2$ & $10-\varepsilon_{p}(5)$ & $\leq 6-\varepsilon_{p}(5)+2\varepsilon_{p}(3)$ & $4$\\ \hline
$\scale[0.8]{L_{G}(\omega_{1}+2\omega_{2})}$& $p=7$ & $\ell=2$ & $12$ & $7$ & $12$\\ \hline
$L_{G}(3\omega_{2})$& $p=7$ & $\ell=2$ & $16$ & $7$ & $9$\\ \hline
$\scale[0.8]{L_{G}(2\omega_{1}+\omega_{2})}$& $p=3$ & $\ell=2$ & $16$ & $\leq 11$ & $9$\\ \hline
$\scale[0.8]{L_{G}(\omega_{1}+\omega_{3})}$ & $p\geq 0$ & $\ell=3$ & $\scale[0.8]{40-8\varepsilon_{p}(3)-20\varepsilon_{p}(2)}$ & $\scale[0.8]{\leq 40-9\varepsilon_{p}(3)-12\varepsilon_{p}(2)}$ & $30-5\varepsilon_{p}(3)-10\varepsilon_{p}(2)$ \\ \hline 
$\scale[0.8]{L_{G}(\omega_{2}+\omega_{3})}$ & $p=5$ & $\ell=3$ & $36$ & $25$ & $26$ \\ \hline 
$L_{G}(2\omega_{3})$ & $p=5$ & $\ell=3$ & $39$ & $25$ & $24$ \\ \hline 
$L_{G}(2\omega_{2})$ & $p=5$ & $\ell=3$ & $50$  & $\leq 50$ & $40$ \\ \hline 
$L_{G}(\omega_{4})$& $p\neq 2$ & $\ell=4$ & $28$ & $28-\varepsilon_{p}(3)$ & $\scale[0.9]{14-\varepsilon_{p}(3)}$\\ \hline
$\scale[0.8]{L_{G}(\omega_{1}+\omega_{\ell})}$ & $p=2$ & $\scale[0.8]{4\leq \ell\leq 6}$ & $(2\ell-1)\cdot 2^{\ell-1}$ & $(3\ell-2)\cdot 2^{\ell-1}$ & $(\ell+2)\cdot 2^{\ell-1}$ \\ \hline 
$\scale[0.8]{L_{G}(\omega_{1}+\omega_{4})}$ & $p=7$ & $\ell=4$ & $\leq 144$ & $142$ & $\geq 96$ \\ \hline 
$\scale[0.8]{L_{G}(\omega_{1}+\omega_{3})}$ & $p=2$ & $\ell=4$ & $\leq 130$ & $\leq 144$ & $\geq 102$\\ \hline 
$L_{G}(\omega_{4})$& $p\geq 0$ & $\ell=5$ & $\scale[0.8]{\leq 100-20\varepsilon_{p}(3)-24\varepsilon_{p}(2)}$ & $\scale[0.8]{\leq 117-36\varepsilon_{p}(3)+3\varepsilon_{p}(2)}$ & $\geq 48-8\varepsilon_{p}(3)-4\varepsilon_{p}(2)$ \\ \hline
$L_{G}(\omega_{5})$ & $p\neq 2$ & $\ell=5$ & $\leq 84-2\varepsilon_{p}(3)$ & $90-9\varepsilon_{p}(3)$ & $42-2\varepsilon_{p}(3)$ \\ \hline 
$L_{G}(\omega_{4})$& $p\geq 0$ & $\ell=6$ & $\leq x_{4}$ & $\leq y_{4}$ & $\geq 110-14\varepsilon_{p}(2)$ \\ \hline
$L_{G}(\omega_{5})$ & $p\geq 0$ & $\ell=6$ & $\scale[0.8]{\leq 350-108\varepsilon_{p}(3)-86\varepsilon_{p}(2)}$ & $\scale[0.8]{\leq 407-164\varepsilon_{p}(3)+5\varepsilon_{p}(2)}$ & $\geq 165-44\varepsilon_{p}(3)-17\varepsilon_{p}(2)$ \\ \hline 
$L_{G}(\omega_{6})$ & $p\neq 2$ & $\ell=6$ & $\leq 268-24\varepsilon_{p}(3)$ & $297-54\varepsilon_{p}(3)$ & $132-11\varepsilon_{p}(3)$ \\ \hline 
\end{tabular}
\end{table}

\begin{table}[h!]
\begin{center}
\begin{tabular}{| c | c | c | c | c | c | c |}
\hline
$L_{G}(\omega_{4})$& $p\geq 0$ & $\ell=7$ & $\leq x_{5}$ & $\leq y_{5}$ & $\geq 208-12\varepsilon_{p}(5)-2\varepsilon_{p}(2)$ \\ \hline
$L_{G}(\omega_{6})$ & $p=3$ & $\ell=7$ & $728$ & $729$ & $364$ \\ \hline 
$L_{G}(\omega_{7})$ & $p=3$ & $\ell=7$ & $730$ & $729$ & $364$ \\ \hline 
$L_{G}(\omega_{5})$ & $p=2$ & $\ell=7$ & $\leq 608$  & $\leq 948$  & $\geq 340$ \\ \hline
$L_{G}(\omega_{4})$& $p\geq 0$ & $\ell=8$ & $\leq x_{6}$ & $\leq y_{6}$ & $\scale[0.9]{\geq 350-14\varepsilon_{p}(3)-16\varepsilon_{p}(2)}$ \\ \hline
$L_{G}(\omega_{4})$& $p\geq 0$ & $\ell=9$ & $\leq x_{7}$ & $\leq y_{7}$ & $\scale[0.9]{\geq 544-16\varepsilon_{p}(7)-2\varepsilon_{p}(2)}$ \\\hline
$L_{G}(\omega_{\ell})$& $p=2$ & $\scale[0.8]{4\leq \ell\leq 13}$ & $2^{\ell-1}$ & $3\cdot 2^{\ell-2}$ & $2^{\ell-2}$\\ \hline
$\scale[0.8]{L_{G}(2\omega_{1}+\omega_{\ell})}^{\ddagger}$& $p=2$ & $\scale[0.8]{3\leq \ell\leq 6}$ & $(2\ell-1)\cdot 2^{\ell-1}$ & $(3\ell-2)\cdot 2^{\ell-1}$ & $(\ell+2)\cdot 2^{\ell-1}$\\
\hline
\end{tabular}
\begin{tablenotes}
      \small
\item $\scale[0.8]{x_{1}=2\ell^{2}-5\ell+3-\varepsilon_{p}(\ell)+(3-\varepsilon_{p}(2))\varepsilon_{\ell}(2)+(2+\varepsilon_{p}(3))\varepsilon_{\ell}(3)+(1-\varepsilon_{p}(2))\varepsilon_{\ell}(4)}$, $\scale[0.8]{y_{1}=2\ell^{2}-3\ell+1-\varepsilon_{p}(\ell)+\varepsilon_{\ell}(2)\varepsilon_{p}(2)}$;
\item $\scale[0.8]{x_{2}=\binom{2\ell-2}{3}+10-(2\ell-2)\varepsilon_{p}(\ell-1)+4\varepsilon_{p}(3)-10\varepsilon_{p}(2)}$, $\scale[0.8]{y_{2}=\binom{2\ell-1}{3}-1-(2\ell-1)\varepsilon_{p}(\ell-1)}$;
\item $\scale[0.8]{x_{3}=16\binom{\ell}{3}+8\ell+4-2(\ell-1)\varepsilon_{p}(2\ell+1)-(4\ell-2)\varepsilon_{p}(2)-(6-6\varepsilon_{p}(2))\varepsilon_{\ell}(3)-(2-2\varepsilon_{p}(2))\varepsilon_{\ell}(4)}$, $\scale[0.8]{y_{3}=2\binom{2\ell}{3}-(2\ell-1)\varepsilon_{p}(2\ell+1)+2\ell\varepsilon_{p}(2)}$;
\item $\scale[0.8]{x_{4}= 260-\varepsilon_{p}(5)+13\varepsilon_{p}(3)-88\varepsilon_{p}(2)}$, $\scale[0.8]{y_{4}=319-\varepsilon_{p}(5)-51\varepsilon_{p}(2)}$;
\item $\scale[0.8]{x_{5}=565-66\varepsilon_{p}(5)+23\varepsilon_{p}(3)-65\varepsilon_{p}(2)}$,  $\scale[0.8]{y_{5}= 702-78\varepsilon_{p}(5)-\varepsilon_{p}(3)+2\varepsilon_{p}(2)}$;
\item $\scale[0.8]{x_{6}= 1091-\varepsilon_{p}(7)-59\varepsilon_{p}(3)-175\varepsilon_{p}(2)}$, $\scale[0.8]{y_{6}=1350-\varepsilon_{p}(7)-105\varepsilon_{p}(3)-102\varepsilon_{p}(2)}$;
\item $\scale[0.8]{x_{7}= 1930-119\varepsilon_{p}(7)+40\varepsilon_{p}(3)-106\varepsilon_{p}(2)}$, $\scale[0.8]{y_{7}= 2363-136\varepsilon_{p}(7)+\varepsilon_{p}(2)}$.
\item$^{\dagger}$ equality holds for $p\neq 2$.
\item$^{\ddagger}$ we added this weight because we made the choice not to treat algebraic groups of type $B_{\ell}$ over fields of characteristic $2$, see Theorem \ref{ResultsBl} and Remark \ref{IsogenyBltoCl}.
    \end{tablenotes}
\caption{\label{TableCl}The value of $\nu_{G}(V)$ for groups of type $C_{\ell}$.}
\end{center}
\end{table}
\end{theorem}

\begin{theorem}\label{ResultsBl}
Let $k$ be an algebraically closed field of characteristic $p\neq 2$ and let $G$ be a simple simply connected linear algebraic group of type $B_{\ell}$, $\ell\geq 3$. Let $T$ be a maximal torus in $G$ and let $V=L_{G}(\lambda)$, where $\lambda\in \X(T)^{+}_{p}$, be a nontrivial irreducible $kG$-module with $\dim(V)\leq 4\ell^{3}$. Then $\displaystyle \max_{s\in T\setminus\ZG(G)}\dim(V_{s}(\mu))$, $\displaystyle \max_{u\in G_{u}\setminus \{1\}}\dim(V_{u}(1))$ and $\nu_{G}(V)$ are as in Table \ref{TableBl}, where the $x_{i}$'s and $y_{i}$'s are given at the end of the table.

\begin{table}[h!]
\centering 
\begin{tabular}{| c | c | c | c | c | c | c |}
\hline
$V$ & Char. & Rank & $\displaystyle \max_{s\in T\setminus\ZG(G)}\dim(V_{s}(\mu))$ & $\displaystyle \max_{u\in G_{u}\setminus \{1\}}\dim(V_{u}(1))$ & $\nu_{G}(V)$ \\ \hline
$L_{G}(\omega_{1})$ & $p\neq 2$ & $\ell\geq 3$ & $2\ell$ & $2\ell-1$ & $1$ \\ \hline
$L_{G}(\omega_{2})$ & $p\neq 2$ & $\ell\geq 3$ & $2\ell^{2}-\ell$ & $2\ell^{2}-3\ell+4$ & $2\ell$  \\ \hline
$L_{G}(2\omega_{1})$ & $p\neq 2$ & $\ell\geq 3$ & $2\ell^{2}+\ell-\varepsilon_{p}(2\ell+1)$ & $2\ell^{2}-\ell-\varepsilon_{p}(2\ell+1)$ & $2\ell$ \\ \hline
$L_{G}(\omega_{3})$& $p\neq 2$ & $\ell\geq 3$ & $\frac{4\ell^{3}-6\ell^{2}+2\ell}{3}$ & $\frac{4\ell^{3}-12\ell^{2}+29\ell-27}{3}$ & $2\ell^{2}-\ell$ \\ \hline
$L_{G}(3\omega_{1})$ & $p\neq 2,3$ & $\ell\geq 3$ &  $x_{1}$ & $y_{1}$  & $2\ell^{2}+\ell-\varepsilon_{p}(2\ell+3)$ \\ \hline
$\scale[0.8]{L_{G}(\omega_{1}+\omega_{2})}$ & $p\neq 2$ & $\ell\geq 3$ & $x_{2}$ & $y_{2}$ & $\scale[0.8]{4\ell^{2}-1-\varepsilon_{p}(\ell)-(2\ell^{2}-\ell)\varepsilon_{p}(3)}$ \\ \hline
$L_{G}(2\omega_{3})$& $p\neq 2$ & $\ell=3$ & $20$ & $21$ & $14$ \\ \hline
$\scale[0.8]{L_{G}(\omega_{2}+\omega_{3})}$& $p=3$ & $\ell=3$ & $52$ & $\leq 54$ & $\geq 50$ \\ \hline
$\scale[0.8]{L_{G}(\omega_{2}+\omega_{3})}$& $p=5$ & $\ell=3$ & $32$ & $30$ & $32$ \\ \hline
$L_{G}(3\omega_{3})$& $p=5$ & $\ell=3$ & $52$ & $50$ & $52$ \\ \hline
$L_{G}(2\omega_{4})$& $p\neq 2$ & $\ell=4$ & $\leq 74$ & $76$ & $50$ \\ \hline 
$L_{G}(\omega_{4})$ & $p\neq 2$ & $\ell=5$ & $\leq 214$ & $202$ & $\geq 116$ \\ \hline 
$L_{G}(2\omega_{5})$ & $p\neq 2$ & $\ell=5$ & $\leq 278$ & $279$ & $183$ \\ \hline 
$L_{G}(\omega_{4})$ & $p\neq 2$ & $\ell=6$ & $\leq 499$ & $453$ & $\geq 216$ \\ \hline 
$L_{G}(\omega_{4})$ & $p\neq 2$ & $\ell=7$ & $\leq 1005$ & $897$ & $\geq 360$ \\ \hline
$\scale[0.8]{L_{G}(\omega_{1}+\omega_{\ell})}$& $p\neq 2$ & $\scale[0.8]{3\leq \ell\leq 6}$ & $\scale[0.8]{\ell\cdot 2^{\ell}-2^{\ell-1} \varepsilon_{p}(2\ell+1)}$ & $\scale[0.8]{\leq (3\ell-2)\cdot 2^{\ell-1}-3\cdot 2^{\ell-2}\varepsilon_{p}(2\ell+1)}$ & $\scale[0.8]{\geq (\ell+2)\cdot 2^{\ell-1}-2^{\ell-2}\varepsilon_{p}(2\ell+1)}$ \\ \hline
$L_{G}(\omega_{\ell})$& $p\neq 2$ & $\scale[0.8]{4\leq \ell\leq 13}$ & $2^{\ell-1}$ & $3\cdot 2^{\ell-2}$ & $2^{\ell-2}$ \\ 
\hline
\end{tabular}
\begin{tablenotes}
      \small
\item $\scale[0.8]{x_{1}=\binom{2\ell+1}{3}+2{\ell}^{2}+\ell-2\ell\varepsilon_{p}(2\ell+3)}$, $\scale[0.8]{y_{1}=\binom{2\ell+1}{3}-(2\ell-1)\varepsilon_{p}(2\ell+3)}$;
\item $\scale[0.8]{x_{2}=2\binom{2\ell+1}{3}-2\ell\varepsilon_{p}(\ell)-\binom{2\ell}{3}\varepsilon_{p}(3)}$, $\scale[0.8]{y_{2}=2[\binom{2\ell}{3}+3\ell+5]-(2\ell-1)\varepsilon_{p}(\ell)-(\frac{4\ell^{3}-12\ell^{2}+29\ell-30}{3})\varepsilon_{p}(3)}$.
    \end{tablenotes}
\caption{\label{TableBl}The value of $\nu_{G}(V)$ for groups of type $B_{\ell}$.}
\end{table}
\end{theorem}

\begin{theorem}\label{ResultsDl}
Let $k$ be an algebraically closed field of characteristic $p\geq 0$ and let $G$ be a simple simply connected linear algebraic group of type $D_{\ell}$, $\ell\geq 4$. Let $T$ be a maximal torus in $G$ and let $V=L_{G}(\lambda)$, where $\lambda\in \X(T)^{+}_{p}$, be a nontrivial irreducible $kG$-module with $\dim(V)\leq 4\ell^{3}$. Then $\displaystyle \max_{s\in T\setminus\ZG(G)}\dim(V_{s}(\mu))$, $\displaystyle \max_{u\in G_{u}\setminus \{1\}}\dim(V_{u}(1))$ and $\nu_{G}(V)$ are as in Table \ref{TableDl}, where the $x_{i}$'s and $y_{i}$'s are given at the end of the table.

\begin{table}[h!]
\centering 
\begin{tabular}{| c | c | c | c | c | c | c |}
\hline
$V$ & Char. & Rank & $\displaystyle \max_{s\in T\setminus\ZG(G)}\dim(V_{s}(\mu))$ & $\displaystyle \max_{u\in G_{u}\setminus \{1\}}\dim(V_{u}(1))$ & $\nu_{G}(V)$ \\ \hline
$L_{G}(\omega_{1})$ & $p\geq 0$ & $\ell\geq 4$ & $2\ell-2$ & $2\ell-2$ & $2$ \\ \hline
$L_{G}(\omega_{2})$ & $p\geq 0$ & $\ell\geq 4$ & $x_{1}$ & $y_{1}$ & $4\ell-6$ \\ \hline
$L_{G}(2\omega_{1})$ & $p\neq 2$ & $\ell\geq 4$ & $\scale[0.8]{2\ell^{2}-3\ell+3-\varepsilon_{p}(\ell)}$ & $\scale[0.8]{2\ell^{2}-3\ell+1-\varepsilon_{p}(\ell)}$ & $4\ell-4$  \\ \hline
$L_{G}(\omega_{3})$ & $p\geq 0$ & $\ell\geq 5$ & $\leq x_{2}$ & $y_{2}$ & $\scale[0.8]{4\ell^{2}-14\ell+14-2(1+\varepsilon_{2}(\ell-1))\varepsilon_{p}(2)}$ \\ \hline					
$L_{G}(3\omega_{1})$ & $p\neq 2,3$ & $\ell\geq 4$ & $x_{3}$ & $y_{3}$ & $\scale[0.8]{4\ell^{2}-6\ell+4-2\varepsilon_{p}(\ell+1)}$ \\ \hline
\multirow{2}{4.5em}{$\scale[0.8]{L_{G}(\omega_{1}+\omega_{2})}$} & $p=3$ & $\ell\geq 4$ & $x_{4}$ & $\leq y_{4}$ & $\scale[0.8]{4\ell^{2}-6\ell+2-2\varepsilon_{3}(2\ell-1)}$  \\ \cline{2-6}
& $p\neq 3$ & $\ell\geq 4$ & $x_{5}$ & $\leq y_{5}$ & $\scale[0.8]{\geq 8(\ell-1)^{2}-2\varepsilon_{p}(2\ell-1)-(4\ell-6)\varepsilon_{p}(2)^{\dagger}}$  \\ \hline
$\scale[0.8]{L_{G}(\omega_{1}+\omega_{4})}$& $p\geq 0$ & $\ell=4$ & $\leq 34-14\varepsilon_{p}(2)$ & $34-6\varepsilon_{p}(2)$ & $22-2\varepsilon_{p}(2)$ \\ \hline
$L_{G}(2\omega_{2})$& $p=3$ & $\ell=4$ & $\leq 111$ & $\leq 93$ & $\geq 84$ \\ \hline
$\scale[0.8]{L_{G}(2\omega_{1}+\omega_{3})}$& $p\neq 2$ & $\ell=4$ & $\scale[0.8]{\leq 128-34\varepsilon_{p}(5)+8\varepsilon_{p}(3)}$ & $\scale[0.8]{\leq 114-34\varepsilon_{p}(5)+6\varepsilon_{p}(3)}$ & $\geq 96-22\varepsilon_{p}(5)-8\varepsilon_{p}(3)$ \\ \hline
$\scale[0.8]{L_{G}(\omega_{1}+\omega_{3}+\omega_{4})}$& $p=2$ & $\ell=4$ &  $\leq 130$ & $\leq 144$ & $\geq 102$ \\ \hline
$L_{G}(2\omega_{5})$& $p\neq 2$ & $\ell=5$ & $\leq 80$ & $76$ & $\geq 46$ \\ \hline
$\scale[0.8]{L_{G}(\omega_{4}+\omega_{5})}$ & $p\geq 0$ & $\ell=5$ & $\leq 130-50\varepsilon_{p}(2)$ & $\leq 128-28\varepsilon_{p}(2)$ & $\geq 80-16\varepsilon_{p}(2)$ \\ \hline 
$\scale[0.8]{L_{G}(\omega_{1}+\omega_{5})}$ & $p\geq 0$ & $\ell=5$ & $\scale[0.8]{\leq 92-12\varepsilon_{p}(5)-16\varepsilon_{p}(2)}$ & $\leq 92-12\varepsilon_{p}(5)$ & $\geq 52-4\varepsilon_{p}(5)$ \\ \hline 
$\scale[0.8]{L_{G}(\omega_{2}+\omega_{5})}$ & $p=2$ & $\ell=5$ & $\leq 192$ & $\leq 252$ & $\geq 164$ \\ \hline 
$L_{G}(2\omega_{6})$ & $p\neq 2$ & $\ell=6$ & $\leq 290$ & $280$ & $\geq 172$ \\ \hline 
$\scale[0.8]{L_{G}(\omega_{5}+\omega_{6})}$ & $p\geq 0$ & $\ell=6$ & $\leq 492-216\varepsilon_{p}(2)$ & $\leq 484-140\varepsilon_{p}(2)$ & $\geq 300-84\varepsilon_{p}(2)$ \\ \hline 
$\scale[0.8]{L_{G}(\omega_{1}+\omega_{6})}$ & $p\geq 0$ & $\ell=6$ & $\leq x_{6}$ & $\leq 232-24\varepsilon_{p}(6)$ & $\geq 120-8\varepsilon_{p}(3)-24\varepsilon_{p}(2)$ \\ \hline 
$L_{G}(\omega_{4})$ & $p\geq 0$ & $\ell=6$ & $\leq 303-127\varepsilon_{p}(2)$ & $\leq 311-91\varepsilon_{p}(2)$ & $\geq 184-40\varepsilon_{p}(2)$ \\ \hline 
$L_{G}(\omega_{5})$ & $p=2$ & $\ell=7$ & $\leq 608$ & $\leq 948$ & $\geq 340$ \\ \hline 
$\scale[0.8]{L_{G}(\omega_{1}+\omega_{7})}$ & $p\geq 0$ & $\ell=7$ & $\leq x_{7}$ & $\leq 560-48\varepsilon_{p}(7)$ & $\geq 272-16\varepsilon_{p}(2)$ \\ \hline 
$L_{G}(\omega_{4})$ & $p\geq 0$ & $\ell=7$ & $\leq 625-119\varepsilon_{p}(2)$ & $\leq 651-65\varepsilon_{p}(2)$ & $\geq 350-26\varepsilon_{p}(2)$ \\ \hline 
$L_{G}(\omega_{4})$ & $p\geq 0$ & $\ell=8$ & $\leq 1172-250\varepsilon_{p}(2)$ & $\leq 1224-182\varepsilon_{p}(2)$ & $\geq 596-56\varepsilon_{p}(2)$ \\ \hline 
$\scale[0.8]{L_{G}(\omega_{1}+\omega_{8})}$ & $p\geq 0$ & $\ell=8$ & $\leq x_{8}$ & $\leq 1312-96\varepsilon_{p}(2)$ & $\geq 508-142\varepsilon_{p}(2)$ \\ \hline 
$L_{G}(\omega_{4})$ & $p=2$ & $\ell=9$ & $\leq 1824$ & $\leq 2364$ & $\geq 542$ \\ \hline 
$L_{G}(\omega_{\ell})$ & $p\geq 0$ & $\scale[0.8]{5\leq \ell\leq 15}$ & $5\cdot 2^{\ell-4}$ &  $3\cdot 2^{\ell-3}$ & $2^{\ell-3}$ \\ 
\hline
\end{tabular}
\caption{\label{TableDl}The value of $\nu_{G}(V)$ for groups of type $D_{\ell}$.}
\begin{tablenotes}
      \small
\item $\scale[0.8]{x_{1}=2\ell^{2}-5\ell+4-(1+\varepsilon_{2}(\ell))\varepsilon_{p}(2)}$, $\scale[0.8]{y_{1}=2\ell^{2}-5\ell+6-(1+\varepsilon_{2}(\ell))\varepsilon_{p}(2)}$;
\item $\scale[0.8]{x_{2}=\binom{2\ell-2}{3}+2\ell+6-(2\ell+6+(2\ell-2)\varepsilon_{2}(\ell-1))\varepsilon_{p}(2)}$, $\scale[0.8]{y_{2}=\binom{2\ell-2}{3}+6\ell-10-(2\ell-2)(1+\varepsilon_{2}(\ell-1))\varepsilon_{p}(2)}$;
\item $\scale[0.8]{x_{3}=\binom{2\ell}{3}+4\ell-4-(2\ell-2)\varepsilon_{p}(\ell+1)}$, $\scale[0.8]{y_{3}=\binom{2\ell}{3}-(2\ell-2)\varepsilon_{p}(\ell+1)}$;
\item $\scale[0.8]{x_{4}=\binom{2\ell}{3}+(2\ell-2)(1-\varepsilon_{3}(2\ell-1))}$, $\scale[0.8]{y_{4}=\binom{2\ell}{3}-2\ell+6-(2\ell-2)\varepsilon_{3}(2\ell-1)}$;
\item $\scale[0.8]{x_{5}=2^{4}\binom{\ell}{3}+(2\ell-2)(4-\varepsilon_{p}(2\ell-1)-2\varepsilon_{p}(2))}$, $\scale[0.8]{y_{5}=2^{4}\binom{\ell}{3}+8\ell-10-(2\ell-2)(\varepsilon_{p}(2\ell-1)-2\varepsilon_{p}(2))}$;
\item $\scale[0.8]{x_{6}=224-4\varepsilon_{p}(5)-20\varepsilon_{p}(3)-52\varepsilon_{p}(2)}$;
\item $\scale[0.8]{x_{7}=528-40\varepsilon_{p}(7)-8\varepsilon_{p}(5)-64\varepsilon_{p}(2)}$;
\item $\scale[0.8]{x_{8}=1216-16\varepsilon_{p}(5)-208\varepsilon_{p}(2)}$.
\item$^{\dagger}$ equality holds for $p\neq 2$.
    \end{tablenotes}
\end{table}
\end{theorem}

\section{Notation}\label{SectionNotation}

Throughout the text $k$ is an algebraically closed field of characteristic $p\geq 0$. Note that when we write $p\neq p_{0}$, for some prime $p_{0}$, we allow $p=0$. Let $G$, $T$ and $\X(T)$ be as before. Let $\Y(T)$ be the group of rational cocharacters of $T$ and let $\langle -, - \rangle$ be the natural pairing on $\X(T)\times \Y(T)$. Now, we denote by $\Phi$ the root system of $G$ corresponding to $T$ and by $\Delta=\{\alpha_{1},\dots,\alpha_{\ell}\}$ a set of simple roots in $\Phi$, where we use the standard Bourbaki labeling as given in \cite[11.4, p.58]{humphreys_1972introduction}. Moreover, let $\Phi^{+}$ be the set of positive roots of $G$. Following \cite[Section 2.1]{carter1989simple}, we fix a total order $\preceq$ on $\Phi$: for $\alpha,\beta\in \Phi$ we have $\alpha\preceq \beta$ if and only if $\alpha=\beta$, or $\scale[0.9]{\displaystyle \beta-\alpha=\sum_{i=1}^{r}a_{i}\alpha_{i}}$ with $1\leq r\leq \ell$,  $a_{i}\in \mathbb{Z}$, $1\leq i\leq r$, and $a_{r}>0$. 

Now, for each $\alpha\in \Phi$, let $h_{\alpha}\in \Y(T)$ be its corresponding coroot, let $U_{\alpha}$ be its corresponding root subgroup and let $x_{\alpha}:k\to U_{\alpha}$ be an isomorphism of algebraic groups with the property that $tx_{\alpha}(c)t^{-1}=x_{\alpha}(\alpha(t)c)$ for all $t\in T$ and all $c\in k$. Let $G_{s}$ be the set of semisimple elements of $G$ and let $G_{u}$ be the set of unipotent elements of $G$. Any $s\in T$ can be written $\scale[0.9]{\displaystyle s=\prod_{\alpha_{i}\in \Delta}h_{\alpha_{i}}(c_{\alpha_{i}}})$, where $c_{\alpha_{i}}\in k^{*}$, respectively any $u\in G_{u}$ can be written as $\scale[0.9]{\displaystyle u=\prod_{\alpha\in \Phi^{+}}x_{\alpha}(c_{\alpha})}$, where $c_{\alpha}\in k$ and the product respects $\preceq$. Lastly, let $B$ be the positive Borel subgroup of $G$, $\mathcal{W}$ be the Weyl group of $G$ corresponding to $T$, and $w_{0}\in \mathcal{W}$ be the longest word.

The set of dominant weights of $G$ with respect to $\Delta$ is denoted by $\X(T)^{+}$, and the set of $p$-restricted dominant weights by $\X(T)^{+}_{p}$. We adopt the usual convention that when $p=0$, all weights are $p$-restricted. For $\lambda\in \X(T)^{+}$, we denote by $L_{G}(\lambda)$ the irreducible $kG$-module of highest weight $\lambda$. Further, we denote by $\omega_{i}$, $1\leq i\leq \ell$, the fundamental dominant weight of $G$ with respect to $\alpha_{i}$. 

All representations and all modules of a linear algebraic group are assumed to be rational and nontrivial. For a $kG$-module $V$ we will use the notation $V=W_{1}\mid W_{2}\mid \cdots \mid W_{m}$ to express that $V$ has a composition series $V=V_{1}\supset V_{2}\supset \cdots \supset V_{m}\supset V_{m+1}=0$ with composition factors $W_{i}\cong V_{i}/V_{i+1}$, $1\leq i\leq m$. Another notation we will use is $V^{m}$ for the direct sum $V\oplus \cdots \oplus V$, in which $V$ occurs $m$ times. Lastly, when $p>0$, we denote by $V^{(p^{i})}$ the $kG$-module obtained from $V$ by twisting the action of $G$ with the Frobenius endomorphism $i$ times. Lastly, we will denote the natural module of $G$ by $W$.

For $m\in \mathbb{Z}_{\geq 0}$, we define $\varepsilon_{m}:\mathbb{Z}_{\geq 0}\to \{0,1\}$ by $\varepsilon_{m}(n)=0$ if $m\nmid n$ and $\varepsilon_{m}(n)=1$ if $m\mid n$. 

\section{Preliminary results}\label{Preliminary}
To begin, we prove the following result which gives us a strategy on how to calculate $\nu_{G}(V)$.
\begin{prop}\label{Lemmaoneigenvaluesuniposs}
Let $G$ be a simple linear algebraic group and let $V$ be an irreducible $kG$-module. Then 
$$\displaystyle\nu_{G}(V)=\dim(V)-\max\{\max_{s\in T\setminus \ZG(G)}\dim(V_{s}(\mu)),\max_{u\in G_{u}\setminus \{1\}}\dim(V_{u}(1))\}.$$
\end{prop}

\begin{proof}
Let $\rho:G\to \GL(V)$ be the associated representation of $G$ into $\GL(V)$ and let $g\in G\setminus \ZG(G)$. We write the Jordan decomposition of $g=g_{s}g_{u}=g_{u}g_{s}$, where $g_{s}\in G_{s}$ and $g_{u}\in G_{u}$. By \cite[Theorem 2.5]{Malle_testerman_2011}, $\rho(g)=\rho(g_{s})\rho(g_{u})=\rho(g)_{s}\rho(g)_{u}$ is the Jordan decomposition of $\rho(g)$ in $\GL(V)$. We choose a basis of $V$ with the property that $\rho(g)$ is written in its Jordan normal form. Then, with respect to this basis, $\rho(g)_{s}$ is the diagonal matrix whose entries are just the diagonal entries of $\rho(g)$, while $\rho(g)_{u}$ is the unipotent matrix obtained from the Jordan normal form of $\rho(g)$ by dividing all entries of each Jordan block by the diagonal element.  We distinguish the following two cases:

\underline{Case 1}: Assume $g_{s}\in \ZG(G)$.  First, we remark that $g_{u}\neq 1$, as $g\notin \ZG(G)$. Secondly, as $g_{s}\in \ZG(G)$, it follows that $\rho(g)_{s}=\diag(c,c,\dots,c)$ for some $c\in k^{*}$. Thereby, $c$ is the sole eigenvalue of $\rho(g)$ on $V$ and we have $\displaystyle \dim(V_{g}(c))=\dim(V_{g_{u}}(1))\leq \max_{u\in G_{u}\setminus \{1\}}\dim(V_{u}(1))$.

\underline{Case 2}: Assume $g_{s}\notin \ZG(G)$. Then, since $\rho(g)_{s}$ is a diagonal matrix with entries the diagonal entries of $\rho(g)$, we determine that $\rho(g)$ and $\rho(g)_{s}$ have the same eigenvalues on $V$ and, for any such eigenvalue $c\in k^{*}$ we have $\displaystyle\dim(V_{g}(c))\leq \dim(V_{g_{s}}(c))\leq \max_{s\in G_{s}\setminus \ZG(G)}\dim(V_{s}(\mu))=\max_{s'\in T\setminus \ZG(G)}\dim(V_{s'}(\mu'))$, where the last equality follows by \cite[Corollary 4.5 and Theorem 4.4]{Malle_testerman_2011}. 

In conclusion, for any $\scale[0.9]{g\in G\setminus \ZG(G)}$ and any eigenvalue $c\in k^{*}$ of $\rho(g)$ on $V$ we have
$$\dim(V_{g}(c))\leq \max\{\max_{u\in G_{u}\setminus \{1\}}\dim(V_{u}(1)), \max_{s\in T\setminus \ZG(G)}\dim(V_{s}(\mu))\}.$$
\end{proof}

\subsection{Group isogenies and irreducible modules}\label{SubsectGroupIsogens}

In this subsection, we will assume that the simple algebraic group $G$ is not simply connected, and we let $\tilde{G}$ be its simple simply connected cover. Fix a central isogeny $\phi:\tilde{G}\to G$ with $\ker(\phi)\subseteq \ZG(\tilde{G})$ and $d\phi\neq 0$. Let $\tilde{T}$ be a maximal torus in $\tilde{G}$ with the property that $\phi(\tilde{T})=T$ and, similarly, let $\tilde{B}$ be the Borel subgroup of $\tilde{G}$ given by $\phi(\tilde{B})=B$. Let $V=L_{G}(\lambda)$, where $\lambda\in \X(T)^{+}$. Since $\X(T)\subseteq \X(\tilde{T})$, we will denote by $\tilde{\lambda}$ the weight $\lambda$ when viewing it as an element of $\X(\tilde{T})$. By \cite[II.2.10]{Jantzen_2007representations}, as $\lambda\in \X(T)^{+}$, it follows that $\tilde{\lambda}\in \X(\tilde{T})^{+}$. Moreover, by the same result, we have that $V$ is a simple $k\tilde{G}$-module and $V\cong L_{\tilde{G}}(\tilde{\lambda})$ as $k\tilde{G}$-modules.

\begin{lem}\label{LtildeGtildelambdaandLGlambda}
Let $\tilde{V}=L_{\tilde{G}}(\tilde{\lambda})$. Then $\displaystyle \max_{\tilde{s}\in \tilde{T}\setminus \ZG(\tilde{G})}\dim(\tilde{V}_{\tilde{s}}(\tilde{\mu}))=\max_{s\in T\setminus \ZG(G)}\dim(V_{s}(\mu))$ and $\displaystyle \max_{\tilde{u}\in \tilde{G}_{u}\setminus \{1\}}\dim(\tilde{V}_{\tilde{u}}(1))=$ $\displaystyle\max_{u\in G_{u}\setminus \{1\}}\dim(V_{u}(1))$. In particular, we have $\nu_{G}(V)=\nu_{\tilde{G}}(\tilde{V})$.
\end{lem}
 
We end this subsection with the following remark which justifies why we treat groups of type $B_{\ell}$ and their respective modules, only over fields of characteristic different than $2$.
\begin{rem}\label{IsogenyBltoCl}
Assume that $p=2$. For $C$, a simple simply connected linear algebraic group of type $C_{\ell}$, and $B$, a simple simply connected linear algebraic group of type $B_{\ell}$,  there exists an exceptional isogeny $\phi:C\to B$ between the two groups, see \emph{\cite[Theorem 28]{Steinberg:2016}}. Consequently, we can induce irreducible $kC$-modules from irreducible $kB$-modules by twisting with the isogeny $\phi$. More specifically, for any $2$-restricted dominant weight $\scale[0.9]{\displaystyle \mu=\sum_{i=1}^{\ell}d_{i}\omega_{i}^{B}}$, where the $\omega_{i}^{B}$'s are the fundamental dominant weights of $B$, we have that 
$$\scale[0.9]{\displaystyle L_{B}(\mu)\cong L_{C}(2\sum_{i=1}^{\ell-1}d_{i}\omega_{i}^{C}+d_{\ell}\omega_{\ell}^{C})\cong L_{C}(\sum_{i=1}^{\ell-1}d_{i}\omega_{i}^{C})^{(2)}\otimes L_{C}(d_{\ell}\omega_{\ell}^{C})},$$
where the $\omega_{i}^{C}$'s are the fundamental dominant weights of $C$. Thus, for any $2$-restricted dominant weight $\scale[0.9]{\displaystyle \mu= \sum_{i=1}^{\ell-1}d_{i}\omega_{i}^{B}}$ of $B$, we have $\nu_{B}(L_{B}(\mu))=\nu_{C}(L_{C}(2\lambda))=\nu_{C}(L_{C}(\lambda)^{(2)})=\nu_{C}(L_{C}(\lambda))$, where $\displaystyle \lambda= \sum_{i=1}^{\ell-1}d_{i}\omega_{i}^{C}$. Similarly, for the weight $\omega_{\ell}^{B}$, we have $\nu_{B}(L_{B}(\omega_{\ell}^{B}))=\nu_{C}(L_{C}(\omega_{\ell}^{C}))$. Lastly, in the case of weights of the form $\scale[0.9]{\displaystyle \mu= \sum_{i=1}^{\ell}d_{i}\omega_{i}^{B}}$, where there exists $1\leq i\leq \ell-1$ such that $d_{i}=1$ and $d_{\ell}=1$, we will calculate $\scale[0.9]{\nu_{C}(L_{C}(\displaystyle \sum_{i=1}^{\ell-1}d_{i}\omega_{i}^{C})^{2}\otimes L_{C}(\omega_{\ell}))}$ to determine $\nu_{B}(L_{B}(\mu))$.
\end{rem}

\subsection{Restriction to Levi subgroups}
We now drop the assumption we made on $G$ in the previous subsection and we let $G$ be simply connected. For each $1\leq i\leq \ell$, let $P_{i}$ be the maximal parabolic subgroup of $G$ corresponding to $\Delta_{i}:=\Delta\setminus \{\alpha_{i}\}$, and let $L_{i}=\langle T,U_{\pm \alpha_{1}},\dots,$ $U_{\pm \alpha_{i-1}},U_{\pm \alpha_{i+1}}, $ $\dots, U_{\pm \alpha_{\ell}}\rangle$ be a Levi subgroup of $P_{i}$. The root system of $L_{i}$ is $\Phi_{i}=\Phi\cap (\mathbb{Z}\alpha_{1}+\cdots+\mathbb{Z}\alpha_{i-1}+\mathbb{Z}\alpha_{i+1}+\cdots+\mathbb{Z}\alpha_{\ell})$, in which $\Delta_{i}$ is a set of simple roots. Now, we have that $L_{i}=\ZG(L_{i})^{\circ}[L_{i},L_{i}]$, where $\scale[0.9]{\displaystyle \ZG(L_{i})^{\circ}=\bigg(\bigcap_{j\neq i}\ker(\alpha_{j})\bigg)^{\circ}}$ is a one-dimensional subtorus of $G$ and $[L_{i},L_{i}]$ is a semisimple simply connected linear algebraic group of rank $\ell-1$. Lastly, let $T_{i}=T\cap [L_{i},L_{i}]$ be a maximal torus in $[L_{i},L_{i}]$, contained in the Borel subgroup $B_{i}=B\cap [L_{i},L_{i}]$. We will abuse notation and denote the fundamental dominant weights of $L_{i}$ corresponding to $\Delta_{i}$ by $\omega_{1},\dots, \omega_{i-1},\omega_{i+1}, \dots, \omega_{\ell}$.

Let $\lambda\in \X(T)^{+}$, $\scale[0.9]{\lambda=\displaystyle \sum_{i=1}^{\ell}d_{i}\omega_{i}}$, let $V=L_{G}(\lambda)$ be the associated irreducible $kG$-module, and let $\Lambda(V)$ be the set of weights of $V$. Fix some $1\leq i\leq \ell$. We say that a weight $\mu\in \Lambda(V)$ has \textit{$\alpha_{i}$-level} $j$ if $\scale[0.9]{\mu=\lambda-j\alpha_{i}-\displaystyle \sum_{r\neq i}c_{r}\alpha_{r}}$, where $c_{r}\in \mathbb{Z}_{\geq 0}$. The \textit{maximum $\alpha_{i}$-level of weights} in $V$ will be denoted by $e_{i}(\lambda)$. By \cite[II, Proposition 2.4(b)]{Jantzen_2007representations}, we have that $e_{i}(\lambda)$ is equal to the $\alpha_{i}$-level of $w_{0}(\lambda)$. Now, consider the Levi subgroup $L_{i}$ of $P_{i}$. For each $0\leq j\leq e_{i}(\lambda)$, define the subspace $\scale[0.9]{\displaystyle V^{j}:=\bigoplus_{\gamma\in \mathbb{N}\Delta_{i}}V_{\lambda-j\alpha_{i}-\gamma}}$ of $V$ and note that $V^{j}$ is invariant under $L_{i}$. Then, as a $k[L_{i},L_{i}]$-module, $V$ admits the following decomposition:
$$\scale[0.9]{\displaystyle V\mid_{[L_{i},L_{i}]}=\bigoplus_{j=0}^{e_{i}(\lambda)}V^{j}.}$$
Further, by \cite[Proposition]{Smith_82}, $\scale[0.9]{\displaystyle V^{0}=\bigoplus_{\gamma\in \mathbb{N}\Delta_{i}}V_{\lambda-\gamma}}$ is the irreducible $k[L_{i},L_{i}]$-module of highest weight $\lambda\mid_{T_{i}}$.

\begin{lem}\label{dualitylemma}
Assume $V$ is a self-dual $kG$-module. Then, for all $0\leq j\leq \left\lfloor \frac{e_{i}(\lambda)}{2}\right\rfloor$, we have $V^{e_{i}(\lambda)-j}\cong (V^{j})^{*}$, as $k[L_{i},L_{i}]$-modules.
\end{lem}
\begin{proof}
We note that, as $V$ is self-dual, we have $w_{0}(\lambda)=-\lambda$ and $V$ is equipped with a nondegenerate bilinear form $(-,-)$. Let $\mu,\mu^{'}\in \Lambda(V)$ be such that $\mu^{'}\neq -\mu$. Let $v\in V_{\mu}$ and $v^{'}\in V_{\mu^{'}}$. Then $(v,v^{'})=(t\cdot v,t\cdot v^{'})=(\mu(t)v,\mu^{'}(t)v^{'})=(\mu+\mu^{'})(t)(v,v^{'}), \text{ for all }t\in T$. Therefore $(v,v^{'})=0$, as $\mu^{'}\neq -\mu$, and so $V_{\mu^{'}}\subset V_{\mu}^{\perp}$. Moreover, as $(-,-)$ is nondegenerate, it follows that $-\mu\in \Lambda(V)$ for all $\mu\in \Lambda(V)$.

Secondly, let $\mu\in \Lambda(V)$ be a weight of $\alpha_{i}$-level $j$, where $0\leq j\leq e_{i}(\lambda)$. We will show that $-\mu$ has $\alpha_{i}$-level $e_{i}(\lambda)-j$. On one hand, we know that $e_{i}(\lambda)$ is equal to the $\alpha_{i}$-level of $w_{0}(\lambda)$, thus $\scale[0.9]{\displaystyle w_{0}(\lambda)=\lambda-e_{i}(\lambda)\alpha_{i}-\sum_{r\neq i}a_{r}\alpha_{r}}$, where $a_{r}\in \mathbb{Z}_{\geq 0}$. On the other hand, since $\scale[0.87]{\displaystyle \mu=\lambda-j\alpha_{i}-\sum_{r\neq i}c_{r}\alpha_{r}$, for $c_{r}\in \mathbb{Z}_{\geq 0}}$, we have $\scale[0.9]{\displaystyle -\mu=-\lambda+j\alpha_{i}+\sum_{r\neq i}c_{r}\alpha_{r}}$ $\scale[0.9]{\displaystyle=\lambda-(e_{i}(\lambda)-j)\alpha_{i}-\sum_{r\neq i}b_{r}\alpha_{r}}$, where $b_{r}\in \mathbb{Z}_{\geq 0}$ for all $r\neq i$. Therefore, $-\mu$ has $\alpha_{i}$-level equal to $e_{i}(\lambda)-j$. In particular, as $V_{\mu'}\subseteq (V_{\mu})^{\perp}$ for all $\mu'\neq -\mu$, it follows that $\scale[0.9]{\displaystyle (V^{j})^{\perp}\supseteq \bigoplus_{r\neq e_{i}(\lambda)-j}V^{r}}$.

Lastly, as $\scale[0.9]{\displaystyle V\mid_{[L_{i},L_{i}]}=\bigoplus_{j=0}^{e_{i}(\lambda)}V^{j}}$ is self-dual, it follows that $\scale[0.9]{\displaystyle V\mid_{[L_{i},L_{i}]}\cong \bigoplus_{j=0}^{e_{i}(\lambda)}(V^{j})^{*}}$. Furthermore, as $V$ is equipped with a nondegenerate bilinear form, we have that $(V^{j})^{*}\cong V/(V^{j})^{\perp}$, for all $0\leq j\leq e_{i}(\lambda)$.  As $\scale[0.87]{\displaystyle (V^{j})^{\perp}\supseteq \bigoplus_{r\neq e_{i}(\lambda)-j}V^{r}}$, it follows that $\dim((V^{j})^{*})\leq \dim(V^{e_{i}(\lambda)-j})$.  By the same argument, this time applied to $V^{e_{i}(\lambda)-j}$, we determine that $\dim((V^{e_{i}(\lambda)-j})^{*})\leq \dim(V^{j})$.  Therefore, $\dim((V^{j})^{*})=\dim(V^{e_{i}(\lambda)-j})$, thus $\displaystyle (V^{j})^{\perp}=\bigoplus_{r\neq e_{i}(\lambda)-j}V^{r}$, and we conclude that $(V^{j})^{*}\cong V^{e_{i}(\lambda)-j}$.
\end{proof}

\begin{rem}
Applying Lemma \ref{dualitylemma}, let $V=L_{G}(\lambda)$ for some $\scale[0.9]{\displaystyle \lambda=\sum_{i=1}^{\ell}d_{i}\omega_{i}\in \X(T)^{+}}$. As $V^{*}\cong L_{G}(-w_{0}(\lambda))$, it follows that $V$ is self-dual if $-w_{0}(\lambda)=\lambda$. Thus, for groups of type $A_{\ell}$, $V$ is self-dual if $d_{i}=d_{\ell+1-i}$ for all $1\leq i\leq \ell$. For groups of type $B_{\ell}$ and $C_{\ell}$, as $w_{0}=-1$, all irreducible $kG$-modules are self-dual. Lastly, for groups of type $D_{\ell}$ with $\ell$ even, all irreducible $kG$-modules are self-dual, while for groups of type $D_{\ell}$ with $\ell$ odd, $V$ is self-dual if $d_{\ell-1}=d_{\ell}$.
\end{rem}

In what follows, we give a formula for $e_{1}(\lambda)$, the maximum $\alpha_{1}$-level of weights in $L_{G}(\lambda)$, for the classical linear algebraic groups. Further, for groups of type $C_{\ell}$, we also give a formula for $e_{\ell}(\lambda)$.

\begin{lem}\label{weightlevelAl}
Let $G$ be of type $A_{\ell}$ and let $\scale[0.9]{\displaystyle \lambda=\sum_{i=1}^{\ell}d_{i}\omega_{i}}\in \X(T)^{+}$. Then $\scale[0.9]{\displaystyle e_{1}(\lambda)=\sum_{j=1}^{\ell}d_{j}}$.
\end{lem}

\begin{proof}
In order to determine $e_{1}(\lambda)$ we have to calculate the $\alpha_{1}$-level of $w_{0}(\lambda)$. We have that
\begin{equation}\label{wo(lambda)Al}
\scale[0.9]{\displaystyle w_{0}(\lambda) = \lambda-(\lambda-w_{0}(\lambda))=\lambda-\sum_{r=1}^{\ell}d_{r}(\omega_{r}-w_{0}(\omega_{r}))= \lambda-\sum_{r=1}^{\ell}d_{r}(\omega_{r}+\omega_{\ell-r+1})}.
\end{equation}
Using \cite[Table $1$, p.$69$]{humphreys_1972introduction}, we write the $\omega_{i}$'s, $1\leq i\leq \ell$, in terms of the simple roots $\alpha_{j}$, $1\leq j\leq \ell$,  and we see that for $1\leq r\leq \left\lfloor\frac{\ell}{2}\right\rfloor$, we have $\scale[0.8]{\displaystyle \omega_{r}+\omega_{\ell-r+1}=\sum_{j=1}^{r-1}j\alpha_{j}+r\sum_{j=r}^{\ell-r+1}\alpha_{j}+\sum_{j=\ell-r+2}^{\ell}(\ell+1-j)\alpha_{j}}$; and if $\ell$ is odd, we have $\scale[0.8]{\displaystyle \omega_{\frac{\ell+1}{2}}=\frac{1}{2}\bigg[\alpha_{1}+2\alpha_{2}+\cdots+\frac{\ell-1}{2}\cdot\alpha_{\frac{\ell-1}{2}}+\frac{\ell+1}{2}\cdot \alpha_{\frac{\ell+1}{2}}+\frac{\ell-1}{2}\cdot\alpha_{\frac{\ell+1}{2}+1}+\cdots+\alpha_{\ell}\bigg]}$. Substituting in \eqref{wo(lambda)Al}, we determine that $\scale[0.9]{\displaystyle e_{1}(\lambda)=\sum_{j=1}^{\ell}d_{j}}$.
\end{proof}

\begin{lem}\label{weightlevelCl}
Let $G$ be of type $C_{\ell}$ and let $\scale[0.9]{\displaystyle \lambda=\sum_{i=1}^{\ell}d_{i}\omega_{i}}\in \X(T)^{+}$. Then $\scale[0.9]{\displaystyle e_{1}(\lambda)= 2\sum_{j=1}^{\ell}d_{j}}$ and $\scale[0.9]{e_{\ell}(\lambda)=\displaystyle \sum_{j=1}^{\ell}jd_{j}}$.
\end{lem}

\begin{proof}
Note that we have $w_{0}=-1$, hence $w_{0}(\lambda)=-\lambda$. We write the $\omega_{i}$'s, $1\leq i\leq \ell$, in terms of the simple roots $\alpha_{j}$, $1\leq j\leq \ell$, see \cite[Table $1$, p.$69$]{humphreys_1972introduction}, and we get:
\begin{equation*}
\scale[0.9]{\displaystyle w_{0}(\lambda)=-\lambda=\lambda-2\lambda= \lambda-2(d_{1}+\cdots +d_{\ell})\alpha_{1}- 2\bigg(d_{1} +2\sum_{j=2}^{\ell}d_{j}\bigg)\alpha_{2}-\cdots -\bigg(\sum_{j=1}^{\ell}jd_{j}\bigg)\alpha_{\ell}.}
\end{equation*}
We remark that the coefficient of each $\alpha_{i}$ is a nonnegative integer and the result follows.
\end{proof}

\begin{lem}\label{weightlevelBl}
Let $G$ be of type $B_{\ell}$ and let $\scale[0.9]{\displaystyle \lambda=\sum_{i=1}^{\ell}d_{i}\omega_{i}}\in \X(T)^{+}$. Then $\scale[0.9]{\displaystyle e_{1}(\lambda)=2\bigg[\sum_{j=i}^{\ell-1}d_{j}\bigg]+d_{\ell}}$.
\end{lem}

\begin{proof}
We have that $w_{0}=-1$, hence $w_{0}(\lambda)=-\lambda$. Writing the fundamental dominant weights $\omega_{i}$ in terms of the simple roots $\alpha_{j}$, we see that 
\begin{equation*}
\scale[0.9]{\displaystyle w_{0}(\lambda) =-\lambda=\lambda-2\lambda= \lambda-2\sum_{i=1}^{\ell}\bigg[\sum_{j=1}^{i-1}jd_{j}+i\bigg(\sum_{j=i}^{\ell-1}d_{j}+\frac{1}{2}d_{\ell}\bigg)\bigg]\alpha_{i},}
\end{equation*}
therefore $\scale[0.9]{e_{1}(\lambda)=2\displaystyle \bigg[\sum_{j=i}^{\ell-1}d_{j}\bigg]+d_{\ell}}$.
\end{proof}

\begin{lem}\label{weightlevelDl}
Let $G$ be of type $D_{\ell}$ and let $\scale[0.9]{\displaystyle \lambda=\sum_{i=1}^{\ell}d_{i}\omega_{i}}\in \X(T)^{+}$. Then $\scale[0.9]{\displaystyle e_{1}(\lambda)=2\bigg[\sum_{j=1}^{\ell-2}d_{j}+\frac{1}{2}i(d_{\ell-1}+d_{\ell})\bigg]}$.
\end{lem}

\begin{proof}
We first assume that $\ell$ is even. Then $w_{0}=-1$, hence $w_{0}(\lambda)=-\lambda$, and so
\begin{equation*}
\scale[0.9]{\begin{split}
w_{0}(\lambda) & = \lambda- 2\sum_{j=1}^{\ell-2}d_{j}\omega_{j}-2d_{\ell-1}\omega_{\ell-1}-2d_{\ell}\omega_{\ell} =\lambda-\sum_{r=1}^{\ell-2}2\bigg[\sum_{j=1}^{r-1}jd_{j}+r\sum_{j=r}^{\ell-2}d_{j}+\frac{1}{2}r(d_{\ell-1}+d_{\ell})\bigg]\alpha_{r}\\
						& -\bigg[\sum_{j=1}^{\ell-2}jd_{j}+\frac{1}{2}(\ell d_{\ell-1}+(\ell-2)d_{\ell})\bigg]\alpha_{\ell-1}-\bigg[\sum_{j=1}^{\ell-2}jd_{j}+\frac{1}{2}((\ell-2)d_{\ell-1}+\ell d_{\ell})\bigg]\alpha_{\ell}. \end{split}}
\end{equation*}
Thus $\scale[0.9]{e_{1}(\lambda)=\displaystyle 2\bigg[\sum_{j=1}^{\ell-2}d_{j}+\frac{1}{2}(d_{\ell-1}+d_{\ell})\bigg]}$. We now assume that $\ell$ is odd. We note that $w_{0}(\omega_{j})=-\omega_{j}$, for all $1\leq j\leq \ell-2$, $w_{0}(\omega_{\ell-1})=-\omega_{\ell}$ and $w_{0}(\omega_{\ell})=-\omega_{\ell-1}$. It follows that:
\begin{equation*}
\scale[0.9]{\begin{split}
w_{0}(\lambda) & =-\sum_{j=1}^{\ell-2}d_{j}\omega_{j}-d_{\ell-1}\omega_{\ell}-d_{\ell}\omega_{\ell-1}=\lambda-2\sum_{j=1}^{\ell-2}d_{j}\omega_{j}-(d_{\ell-1}+d_{\ell})(\omega_{\ell-1}+\omega_{\ell})\\
                           &= \lambda-\sum_{r=1}^{\ell-2}2\bigg[\sum_{j=1}^{r-1}jd_{j}+r\sum_{j=r}^{\ell-2}d_{j}+\frac{1}{2}r(d_{\ell-1}+d_{\ell})\bigg]\alpha_{r}-\bigg[\sum_{j=1}^{\ell-2}jd_{j}+\frac{1}{2}(\ell-1)(d_{\ell-1}+d_{\ell})\bigg](\alpha_{\ell-1}+\alpha_{\ell})
\end{split}}
\end{equation*}
and so $\scale[0.9]{e_{1}(\lambda)=\displaystyle 2\bigg[\sum_{j=1}^{\ell-2}d_{j}+\frac{1}{2}(d_{\ell-1}+d_{\ell})\bigg]}$.
\end{proof}

\subsection{The algorithm for semisimple elements}\label{algosselems}

In what follows, we outline an inductive algorithm which calculates $\displaystyle \max_{s\in T\setminus \ZG(G)}\dim(V_{s}(\mu))$. Recall that $V=L_{G}(\lambda)$ for some $\scale[0.9]{\displaystyle\lambda= \sum_{i=1}^{\ell}d_{i}\omega_{i}}\in \X(T)^{+}$. For this, we consider the restriction $\scale[0.9]{\displaystyle V \mid_{[L_{i},L_{i}]}=\bigoplus_{j=0}^{e_{i}(\lambda)}V^{j}}$, where $1\leq i\leq \ell$ and $\scale[0.9]{\displaystyle V^{j}=\bigoplus_{\gamma\in \mathbb{N}\Delta_{i}}V_{\lambda-j\alpha_{i}-\gamma}}$, for all $1\leq j\leq e_{i}(\lambda)$. Now, let $s\in T\setminus \ZG(G)$. Then, in particular, $s\in L_{i}$ and so $s=z\cdot h$, where $z\in \ZG(L_{i})^{\circ}$ and $h\in [L_{i},L_{i}]$.  As $z\in \ZG(L_{i})^{\circ}$ and $\ZG(L_{i})^{\circ}$ is a one-dimensional torus, there exists $c\in k^{*}$ and $k_{r}\in \mathbb{Z}$, $1\leq r\leq\ell$, such that $\scale[0.9]{\displaystyle z=\prod_{r=1}^{\ell}h_{\alpha_{r}}(c^{k_{r}})}$. Moreover, we have $\alpha_{j}(z)=1$ for all $1\leq j\leq \ell$, $j\neq i$. On the other hand, as $h\in [L_{i},L_{i}]$, we have $\scale[0.9]{\displaystyle h=\prod_{r\neq i}h_{\alpha_{r}}(a_{r})}$, where $a_{r}\in k^{*}$ for all $r\neq i$. Now, as $z\in \ZG(L_{i})^{\circ}$, $z$ acts on each $V^{j}$, $0\leq j\leq e_{i}(\lambda)$, as scalar multiplication by $s_{z}^{j}$, where
\begin{equation}\label{actionofzonVj}
\scale[0.9]{\displaystyle s_{z}^{j}:=(\lambda-j\alpha_{i}-\gamma)(z) = (\lambda-j\alpha_{i})(\prod_{r=1}^{\ell}h_{\alpha_{r}}(c^{k_{r}}))= \prod_{r=1}^{\ell}\bigg(c^{k_{r}d_{r}}\bigg)\cdot \prod_{r=1}^{\ell}c^{-jk_{r}\langle\alpha_{i},\alpha_{r}\rangle}.}
\end{equation}
Lastly, let $\mu^{j}_{1},\dots,\mu^{j}_{t_{j}}$, $t_{j}\geq 1$, be the distinct eigenvalues of $h$ on $V^{j}$, $0\leq j\leq e_{i}(\lambda)$, and let $n^{j}_{1},\dots, n^{j}_{t_{j}}$ be their respective multiplicities. Then, as $s=z\cdot h$, it follows that the eigenvalues of $s$ on $V^{j}$ are $s_{z}^{j}\mu^{j}_{1},\dots,s_{z}^{j}\mu^{j}_{t_{j}}$ and they are distinct, as the $\mu^{j}_{r}$'s are, with respective multiplicities $n^{j}_{1},\dots, n^{j}_{t_{j}}$. This proves the following:

\begin{lem}\label{eigenvaluesofsonVj}
Let $s\in T\setminus \ZG(G)$ and write $s=z\cdot h$ with $z\in \ZG(L_{i})^{\circ}$ and $h\in [L_{i},L_{i}]$. Let $\mu^{j}_{1},\dots,\mu^{j}_{t_{j}}$, $t_{j}\geq 1$, be the distinct eigenvalues of $h$ on $V^{j}$, $0\leq j\leq e_{i}(\lambda)$, with respective multiplicities $n^{j}_{1},\dots, n^{j}_{t_{j}}$. Then:
\begin{enumerate}
\item  $z$ acts on $V^{j}$ as scalar multiplication by $s_{z}^{j}$, where $s_{z}^{j}$ is given in \eqref{actionofzonVj};
\item the distinct eigenvalues of $s$ on $V^{j}$ are $s_{z}^{j}\mu^{j}_{1},\dots,s_{z}^{j}\mu^{j}_{t_{j}}$, with multiplicities $n^{j}_{1},\dots, n^{j}_{t_{j}}$;
\item the eigenvalues of $s$ on $V$ are $s_{z}^{j}\mu^{j}_{1},\dots,s_{z}^{j}\mu^{j}_{t_{j}}$, $0\leq j\leq e_{i}(\lambda)$, with respective multiplicities at least $n^{j}_{1},\dots, n^{j}_{t_{j}}$.
\end{enumerate}
\end{lem}

\subsection{The algorithm for unipotent elements}\label{algounipelems}

Similar to the previous subsection, we outline an inductive algorithm to calculate $\displaystyle \max_{u\in G_{u}\setminus \{1\}}\dim(V_{u}(1))$. For $u\in G_{u}$, we denote by $k[u]$ the group algebra of $\langle u\rangle$ over $k$. 

\begin{lem}\label{LemmaonfiltrationofV}
Let $u\in G$ be a unipotent element and let $V$ be a finite-dimensional $kG$-module. Let $V=M_{t}\supseteq M_{t-1}\supseteq \cdots\supseteq M_{1}\supseteq M_{0}=0$, where $t\geq 1$, be a filtration of $k[u]$-submodules of $V$. Then:
\begin{equation*}
\scale[0.9]{\displaystyle\dim(V_{u}(1))\leq \sum_{i=1}^{t}\dim((M_{i}/M_{i-1})_{u}(1)).}
\end{equation*}
Moreover, suppose that for each $i$, we have a $u$-invariant decomposition $M_{i}=M_{i-1}\oplus M'_{i-1}$ with $M_{i-1}'\cong M_{i}/M_{i-1}$ as $k[u]$-modules. Then $\scale[0.9]{\displaystyle\dim(V_{u}(1))= \sum_{i=1}^{t}\dim((M_{i}/M_{i-1})_{u}(1))}$.
\end{lem}

\begin{proof}
For each $1\leq i\leq t$, we fix a basis in $M_{i}$ with the property that the matrix $(u)_{M_{i}/M_{i-1}}$ associated to the action of $u$ on $M_{i}/M_{i-1}$ is upper-triangular.  Then, the matrix $(u)_{V}$ of the action of $u$ on $V$ is the block upper-triangular matrix:

$$\scale[0.9]{\newcommand*{\tempb}{\multicolumn{1}{|c}{}}
(u)_{V}=\begin{pmatrix} 
(u)_{M_{1}} & \tempb & \star &  & \star &  & \cdots &  &  \star \\ 
\cline{1-3}
0  & \tempb & (u)_{M_{2}/M_{1}} & \tempb & \star &  & \cdots &  &  \star \\
\cline{2-5}
0  &  & 0& \tempb & (u)_{M_{3}/M_{2}} & \tempb & \cdots &  &  \star \\
\cline{4-5}
  &  & \vdots &  & \vdots &  & \ddots &  &  \vdots \\
\cline{8-9}
0 &  & 0 &  & 0 &  & \cdots & \tempb &  (u)_{M_{t}/M_{t-1}}\\
\end{pmatrix}}.$$
Using $(u)_{V}$, we calculate the matrix of the action of $u-\id_{V}$ on $V$:
$$\scale[0.9]{(u-\id_{V})_{V}=\begin{pmatrix}
(u-\id_{M_{1}})_{M_{1}} & \star              & \star     & \cdots &  \star \\ 
0                                       & (u-\id_{M_{2}/M_{1}})_{M_{2}/M_{1}}  & \star     & \cdots & \star \\
0                     & 0                    & \ddots  & \cdots & \star \\
\vdots             & \vdots            & \vdots  & \ddots &  \vdots \\
0                     &0                     & 0          & \cdots & (u-\id_{M_{t}/M_{t-1}})_{M_{t}/M_{t-1}}\\
\end{pmatrix}},$$
where $(u-\id_{M_{i}/M_{i-1}})_{M_{i}/M_{i-1}}$ is the matrix of the action of $u-\id_{M_{i}/M_{i-1}}$ on $M_{i}/M_{i-1}$, $1\leq i\leq t$, with respect to the basis of $M_{i}$ we have previously fixed. It follows that $\scale[0.9]{\displaystyle \rank(u-\id_{V})\geq \sum_{i=1}^{t}\rank((u-\id_{M_{i}/M_{i-1}})_{M_{i}/M_{i-1}} )}$ and, consequently, we have $\scale[0.9]{\displaystyle \dim(\ker(u-\id_{V}))\leq}$ $\scale[0.9]{\displaystyle \sum_{i=1}^{t}\dim(\ker((u-\id_{V})\mid_{M_{i}/M_{i-1}}))}$. Now, as $V_{u}(1)=\ker(u-\id_{V})$ we determine that $\scale[0.9]{\displaystyle \dim(V_{u}(1))\leq \sum_{i=1}^{t}\dim((M_{i}/M_{i-1})_{u}(1))}$.

Lastly, for all $1\leq i\leq t$, assume that there exists a $k[u]$-submodule $M'_{i-1}$ of $M_{i}$ such that $M_{i}=M_{i-1}\oplus M'_{i-1}$. Then $V\mid_{k[u]}=M_{0}'\oplus \cdots \oplus M_{t-1}'\cong M_{1}\oplus M_{2}/M_{1}\oplus \cdots \oplus M_{t}/M_{t-1}$, and so there exists a basis of $V$ with the property that:
$$\scale[0.9]{(u-\id_{V})_{V}=\begin{pmatrix}
(u-\id_{M_{1}})_{M_{1}} & 0              & 0     & \cdots &  0 \\ 
0                                       & (u-\id_{M_{2}/M_{1}})_{M_{2}/M_{1}}  & 0     & \cdots & 0 \\
0                     & 0                    & \ddots  & \cdots & 0 \\
\vdots             & \vdots            & \vdots  & \ddots &  \vdots \\
0                     &0                     & 0          & \cdots & (u-\id_{M_{t}/M_{t-1}})_{M_{t}/M_{t-1}}\\
\end{pmatrix}},$$
thereby $\scale[0.9]{\displaystyle \rank(u-\id_{V})= \sum_{i=1}^{t}\rank((u-\id_{M_{i}/M_{i-1}})_{M_{i}/M_{i-1}})}$. Arguing as above, we establish that  $\dim(V_{u}(1))=$ $\scale[0.9]{\displaystyle \sum_{i=1}^{t}\dim((M_{i}/M_{i-1})_{u}(1))}$.
\end{proof}

Let $P_{i}=L_{i}\cdot Q_{i}= \langle T,U_{\beta}\mid \beta\in \Phi_{i}\rangle\cdot \langle U_{\alpha}\mid \alpha\in \Phi^{+}\setminus \Phi_{i}\rangle$ be the Levi decomposition of the maximal parabolic subgroup $P_{i}$ of $G$. Let $u\in G_{u}$, $\scale[0.9]{u=\displaystyle \prod_{\alpha\in \Phi^{+}}x_{\alpha}(c_{\alpha})}$, where the product respects the total order $\preceq$ on $\Phi$ and $c_{\alpha}\in k$.  Now, as $u\in B$ and $B\subseteq P_{i}$, it follows that $u$ admits a decomposition $\scale[0.9]{\displaystyle u=\prod_{\alpha\in \Phi_{i}}x_{\alpha}(c'_{\alpha})\cdot \prod_{\alpha\in \Phi^{+}\setminus \Phi_{i}}x_{\alpha}(c'_{\alpha})}$,
where each of the products respects $\preceq$ and $c'_{\alpha}\in k$, for all $\alpha\in \Phi^{+}$. We set $\scale[0.9]{\displaystyle u_{L_{i}}=\prod_{\alpha\in \Phi_{i}}x_{\alpha}(c'_{\alpha})}$ and $\scale[0.9]{\displaystyle u_{Q_{i}}=\prod_{\alpha\in \Phi^{+}\setminus \Phi_{i}}x_{\alpha}(c'_{\alpha})}$, and we note that $u_{L_{i}}\in L_{i}$ and $u_{Q_{i}}\in Q_{i}$. Recall that $V=L_{G}(\lambda)$ for some $\lambda\in \X(T)^{+}$ and that $\scale[0.9]{\displaystyle V\mid_{[L_{i},L_{i}]}=\bigoplus_{j=0}^{e_{i}(\lambda)}V^{j}}$. Let $\mu\in \Lambda(V)$, with corresponding weight space $V_{\mu}$, and let $\alpha\in \Phi$. As $\scale[0.9]{\displaystyle U_{\alpha}V_{\mu}\subseteq \bigoplus_{r\in\mathbb{Z}_{\geq 0}}V_{\mu+r\alpha}}$, see \cite[Lemma $15.4$]{Malle_testerman_2011}, we have 
$\scale[0.9]{\displaystyle u_{L_{i}}\cdot V^{j}\subseteq V^{j}}$, $\scale[0.9]{\displaystyle u_{Q_{i}}\cdot V^{j}\subseteq \bigoplus_{r=0}^{j}V^{r}}$ and $\scale[0.9]{\displaystyle(u_{Q_{i}}-1)\cdot V^{j}\subseteq \bigoplus_{r=0}^{j-1}V^{r}, \text{ for all }0\leq j\leq e_{i}(\lambda)}$. Therefore, $V$ admits a filtration $V=M_{e_{i}(\lambda)}\supseteq M_{e_{i}(\lambda)-1}\supseteq \cdots \supseteq M_{1}\supseteq M_{0}\supseteq 0$ of $k[u]$-submodules, where $\scale[0.9]{\displaystyle M_{j}=\bigoplus_{r=0}^{j}V^{r}$ for all $0\leq j\leq e_{i}(\lambda)}$. We see that $u$ acts on each $M^{j}/M^{j-1}$, $1\leq j\leq e_{i}(\lambda)$, as $u_{L_{i}}$ and so, by Lemma \ref{LemmaonfiltrationofV}, we determine that $\scale[0.9]{\displaystyle\dim(V_{u}(1))\leq \sum_{j=0}^{e_{i}(\lambda)}\dim(V^{j}_{u_{L_{i}}}(1))=\dim(V_{u_{L_{i}}}(1))}$. Therefore, if we identify the $kL_{i}$-composition factors of each $V^{j}$, $0\leq j\leq e_{i}(\lambda)$, then using already proven results and Lemma \ref{LemmaonfiltrationofV} , we can establish an upper-bound for each $\dim(V^{j}_{u_{L_{i}}}(1))$. Now, assuming that $u_{L_{i}}\neq 1$, the upper-bound we obtain for $\dim(V_{u_{L_{i}}}(1))$, hence for $\dim(V_{u}(1))$, will be strictly smaller than $\dim(V)$. Lastly, we remark that if $u=u_{L_{i}}$, i.e. $u_{Q_{i}}=1$, then $u\cdot V^{j}\subseteq V^{j}$, for all $0\leq j\leq e_{i}(\lambda)$, and thus, by Lemma \ref{LemmaonfiltrationofV}, it follows that $\dim(V_{u}(1))=\dim(V_{u_{L_{i}}}(1))$.

We end this subsection with two lemmas concerning the behavior of unipotent elements. The first one is due to Guralnick and Lawther, \cite{GurLaw19}, and tells us which unipotent conjugacy classes in $G$ afford the largest dimensional eigenspaces. 

\begin{lem}\emph{\cite[p.$19$ and Lemmas 1.4.1 and 1.4.4]{GurLaw19}}\label{uniprootelems}
We have $\dim(V_{u_{2}}(1))\leq \dim(V_{u_{1}}(1))$, if $u_{1}\in G_{u}$ belongs to a unipotent conjugacy class of root elements and $u_{2}\in G_{u}$ belongs to any nontrivial unipotent class, unless $e(\Phi)>1$ and one of the following holds:
\begin{enumerate}
\item[\emph{$(1)$}] $G=C_{\ell}$, $p=2$, $u_{1}$ belongs to the unipotent conjugacy class of $x_{\alpha_{\ell}}(1)$ and $u_{2}$ belongs to the unipotent conjugacy class of $x_{\alpha_{1}}(1)$.
\item[\emph{$(2)$}] $G=C_{\ell}$, $u_{1}$ belongs to the unipotent conjugacy class of $x_{\alpha_{1}}(1)$ and $u_{2}$ belongs to the unipotent conjugacy class of $x_{\alpha_{\ell}}(1)$.
\item[\emph{$(3)$}] $G=B_{\ell}$, $u_{1}$ belongs to the unipotent conjugacy class of $x_{\alpha_{\ell}}(1)$ and $u_{2}$ belongs to the unipotent conjugacy class of $x_{\alpha_{1}}(1)$.
\end{enumerate}
\end{lem}

The second lemma gives us $\dim(\wedge^{2}(V)_{u}(1))$, where $V$ is a finite-dimensional $k$-vector space and the unipotent element $u$ acts as a single Jordan block in $\GL(V)$, when $\chr(k)=2$. For each $i\geq 0$, let $V_{i}$ be the indecomposable $k[u]$-module with $\dim(V_{i})=i$ and on which $u$ acts as the full Jordan block $J_{i}$ of size $i$.  We note that $\{V_{i}\mid i\geq 0\}$ is a set of representatives of the isomorphisms classes of indecomposable $k[u]$-modules.

\begin{lem}\label{Lemma on fixed space for wedge in p=2}
Let $k$ be a field of characteristic $p=2$ and let $V$ be a vector space of dimension $i\geq 1$ over $k$. Let $u$ be a unipotent element acting as a single Jordan block in $\GL(V)$. Then $\scale[0.9]{\dim((\wedge^{2}(V))_{u}(1))=\left\lfloor\frac{i}{2}\right\rfloor}$.
\end{lem}

\begin{proof}
We will prove the result by induction on $i\geq 1$. First,  we note that both cases $i=1$ and $i=2$ follow directly from the structure of $\wedge^{2}(V)$.  Hence, we assume that $i\geq 3$ and that the result holds for all $1\leq r<i$. Let $m$ be the unique nonnegative integer for which $2^{m-1}<i\leq 2^{m}$ and set $q=2^{m}$. Now, up to isomorphism, there exist exactly $q$ indecomposable $k[u]$-modules: $V_{1},V_{2},\dots, V_{q}$, where $\dim(V_{j})=j$ and $u$ acts on $V_{j}$ as $J_{j}$. Therefore, as $k[u]$-modules, we have  $V\cong V_{i}$. Now, by \cite[Theorem $2$]{GowLaffey_06}, by we have the following isomorphism of $k[u]$-modules $\wedge^{2}(V_{i})=\wedge^{2}(V_{q-i})\oplus (i-\frac{q}{2}-1)V_{q}\oplus V_{3\frac{q}{2}-i}$. This gives
\begin{equation}\label{unipotentwedgedimeqforAlp=2}
\dim((\wedge^{2}(V_{i}))_{u}(1))=\dim((\wedge^{2}(V_{q-i}))_{u}(1))+(i-\frac{q}{2}-1)\dim((V_{q})_{u}(1)) + \dim((V_{3\frac{q}{2}-i})_{u}(1)).
\end{equation}
As $3\frac{q}{2}-i<q$ and as $u$ acts as a single Jordan block on $V_{q}$ and $V_{3\frac{q}{2}-i}$, respectively, it follows that $\dim((V_{q})_{u}(1))=1$ and $\dim((V_{3\frac{q}{2}-i})_{u}(1))=1$. Furthermore, we note that, as $\frac{q}{2}<i$, we have $q-i<i$ and, by applying induction, it follows that $ \dim((\wedge^{2}(V_{q-i}))_{u}(1))=\left\lfloor\frac{q-i}{2}\right\rfloor$. Substituting in  \eqref{unipotentwedgedimeqforAlp=2} we obtain $\dim((\wedge^{2}(V_{i}))_{u}(1))=\left\lfloor\frac{q-i}{2}\right\rfloor+i-\frac{q}{2}-1+1=\left\lfloor\frac{i}{2}\right\rfloor$.
\end{proof}

\section{The list of modules}\label{Sectionthelistofrepresentations}

\subsection{$G$ of type $A_{\ell}$, $\ell\geq 1$}\label{SubsecrListofRepresentationsAl}

Let $\lambda\in \X(T)^{+}_{p}$. Using \cite[Theorem 1.2]{Martinez2018}, we see that for $\ell\geq 16$ the only $kG$-modules $L_{G}(\lambda)$ with $\dim (L_{G}(\lambda))\leq \frac{\ell^{3}}{2}$, up to duality, have highest weight $\lambda\in \{\omega_{1},\omega_{2},2\omega_{1},\omega_{1}+\omega_{\ell},\omega_{3},3\omega_{1},\omega_{1}+\omega_{2}\}$. Further, for $\ell\leq 15$, the additional $kG$-modules $L_{G}(\lambda)$, up to duality, with $\dim (L_{G}(\lambda))\leq \frac{\ell^{3}}{2}$ are given in Table \ref{AdditionalReprAl}.
\begin{table*}[h!]
\centering
\begin{tabular}{ c | c | c | c }
\toprule
Rank & $\lambda$ & $p$ & $\dim(L_{G}(\lambda))$ \\
\hline
$7\leq \ell\leq 14$  & $\omega_{4}$ & all & $\binom{\ell+1}{4}$\\\hline
$9\leq \ell\leq 10$  & $\omega_{5}$ & all & $\binom{\ell+1}{5}$ \\
\bottomrule
\end{tabular}
\caption{\label{AdditionalReprAl} The additional highest weight modules for groups of type $A_{\ell}$.}
\end{table*}

\subsection{$G$ of type $C_{\ell}$, $\ell\geq 2$}\label{SubsecrListofRepresentationsCl}

Let $\lambda\in \X(T)^{+}_{p}$. By \cite[Theorem 1.2]{Martinez2018}, for $\ell\geq 14$ the only $kG$-modules $L_{G}(\lambda)$ with $\dim (L_{G}(\lambda))\leq 4\ell^{3}$ have highest weight $\lambda\in \{\omega_{1},\omega_{2},2\omega_{1},\omega_{3},$ $3\omega_{1},\omega_{1}+\omega_{2}\}$. Further, for $\ell\leq 13$, the additional $kG$-modules $L_{G}(\lambda)$ with $\dim (L_{G}(\lambda))\leq 4\ell^{3}$ have highest weights given in Table \ref{AdditionalReprCl}.

\begin{table*}[h!]
\centering
\begin{tabular}{ c | c | c | c }
\toprule
Rank & $\lambda$ & $p$ & $\dim(L_{G}(\lambda))$ \\
\hline
\multirow{4}{2.5em}{$\ell=2$} & $2\omega_{2}$ & $p\neq 2$ & $14-\varepsilon_{p}(5)$\\
                              & $\omega_{1}+2\omega_{2}$ & $p=7$ & $24$ \\
                              & $3\omega_{2}$ & $p\neq 2,3$ & $30-5\varepsilon_{p}(7)$\\
                              & $2\omega_{1}+\omega_{2}$ & $p=3$ & $25$ \\\hline
\multirow{5}{2.5em}{$\ell=3$} & $\omega_{1}+\omega_{3}$ & all & $70-13\varepsilon_{p}(3)-22\varepsilon_{p}(2)$\\
						   & $2\omega_{1}+\omega_{3}^{\dagger}$ & $p=2$ & $48$\\
                              & $\omega_{2}+\omega_{3}$ & $p=5$ & $62$ \\
                              & $2\omega_{3}$ & $p\neq 2$ & $84-21\varepsilon_{p}(5)$\\
                              & $2\omega_{2}$ & $p\neq 2$ & $90-\varepsilon_{p}(7)$ \\\hline
\multirow{3}{2.5em}{$\ell=4$} & $\omega_{4}$ & $p\neq 2$ & $42-\varepsilon_{p}(3)$\\
                              & $\omega_{1}+\omega_{4}$ & $p=2$ & $128$ \\
                              & $2\omega_{1}+\omega_{4}^{\dagger}$ & $p=2$ & $128$\\\hline
\end{tabular}
\end{table*}

\begin{table*}[h!]
\centering
\begin{tabular}{ c | c | c | c }
\hline
\multirow{2}{2.5em}{$\ell=4$} & $\omega_{1}+\omega_{4}$ & $p=7$ & $240$ \\
                              & $\omega_{1}+\omega_{3}$ & $p=2$ & $246$ \\\hline
\multirow{4}{2.5em}{$\ell=5$} & $\omega_{4}$ & all & $165-\varepsilon_{p}(2)-44\varepsilon_{p}(3)$\\
                              & $\omega_{5}$ & $p\neq 2$ & $132-10\varepsilon_{p}(3)$ \\
                              & $\omega_{1}+\omega_{5}$ & $p=2$ & $320$ \\
                              & $2\omega_{1}+\omega_{5}^{\dagger}$ & $p=2$ & $320$\\\hline
\multirow{5}{2.5em}{$\ell=6$} & $\omega_{4}$ & all & $429-\varepsilon_{p}(5)-65\varepsilon_{p}(2)$\\
						   & $\omega_{6}$ & $p\neq 2$ & $429-64\varepsilon_{p}(3)$ \\
                              & $\omega_{5}$ & all & $572-12\varepsilon_{p}(2)-208\varepsilon_{p}(3)$\\
                              & $\omega_{1}+\omega_{6}$ & $p=2$ & $768$\\
                              & $2\omega_{1}+\omega_{6}^{\dagger}$ & $p=2$ & $768$\\\hline
\multirow{4}{2.5em}{$\ell=7$} & $\omega_{4}$ & all & $910-\varepsilon_{p}(3)-90\varepsilon_{p}(5)$\\
                              & $\omega_{6}$ & $p=3$ & $1093$ \\
                              & $\omega_{7}$ & $p=3$ & $1094$ \\
                              & $\omega_{5}$ & $p=2$ & $1288$\\\hline
$\ell=8$& $\omega_{4}$ & all & $1700-\varepsilon_{p}(7)-119\varepsilon_{p}(3)-118\varepsilon_{p}(2)$\\\hline
$\ell=9$ & $\omega_{4}$ & all & $2907-152\varepsilon_{p}(7)-\varepsilon_{p}(2)$\\\hline
$4\leq \ell\leq 13$ & $\omega_{\ell}$ & $p=2$ & $2^{\ell}$\\
\bottomrule
\end{tabular}
\begin{center}
\begin{tablenotes}
      \small
\item$^{\dagger}$ The weights $2\omega_{1}+\omega_{\ell}$, $\ell=3,4,5,6$, have been added to the table as we made the choice to not treat groups of type $B_{\ell}$ in characteristic $2$, and, in view of Remark \ref{IsogenyBltoCl}, these are the only modules for groups of type $B_{\ell}$ with dimension $\leq 4\ell^{3}$ which, when viewed as a module for $C_{\ell}$, are not isomorphic to (a twist of) a module already listed. 
    \end{tablenotes}
\end{center}
\caption{\label{AdditionalReprCl} The additional highest weight modules for groups of type $C_{\ell}$.}
\end{table*}

\subsection{$G$ of type $B_{\ell}$, $\ell\geq 3$}\label{SubsecrListofRepresentationsBl}

Let $\lambda\in \X(T)^{+}_{p}$. By \cite[Theorem 1.2]{Martinez2018}, for $\ell\geq 14$ the only $kG$-modules $L_{G}(\lambda)$ with $\dim (L_{G}(\lambda))\leq 4\ell^{3}$ have highest weight $\lambda\in \{\omega_{1},\omega_{2},2\omega_{1},\omega_{3},$ $3\omega_{1},\omega_{1}+\omega_{2}\}$. Further, for $\ell\leq 13$, the additional $kG$-modules $L_{G}(\lambda)$ with $\dim (L_{G}(\lambda))\leq 4\ell^{3}$ have highest weights given in Table \ref{AdditionalReprBl}.

\begin{table*}[h!]
\centering
\begin{tabular}{ c | c | c | c }
\toprule
Rank & $\lambda$ & $p$ & $\dim(L_{G}(\lambda))$ \\
\hline
\multirow{3}{2.5em}{$\ell=3$} & $2\omega_{3}$ & $p\neq 2$ & $35$\\
                              & $\omega_{2}+\omega_{3}$ & $p=3,5$ & $104-40\varepsilon_{p}(5)$\\
                              & $3\omega_{3}$ & $p=5$ & $104$ \\\hline
$\ell=4$ & $2\omega_{4}$ & $p\neq 2$ & $126$ \\\hline
$\ell=5$ & $2\omega_{5}$ & $p\neq 2$ & $462$ \\\hline
$3\leq \ell\leq 6$ & $\omega_{1}+\omega_{\ell}$ & $p\neq 2$ & $\ell\cdot 2^{\ell+1}-2^{\ell}\varepsilon_{p}(2\ell+1)$\\\hline
$5\leq \ell\leq 7$& $\omega_{4}$ & $p\neq 2$ & $\binom{2\ell+1}{4}$\\\hline
$4\leq \ell\leq 13$ & $\omega_{\ell}$ & $p\neq 2$ & $2^{\ell}$\\
\bottomrule
\end{tabular}
\caption{\label{AdditionalReprBl} The additional highest weight modules for groups of type $B_{\ell}$.}
\end{table*}

\subsection{$G$ of type $D_{\ell}$, $\ell\geq 4$}\label{SubsecrListofRepresentationsDl}

Let $\lambda\in \X(T)^{+}_{p}$. By \cite[Theorem 1.2]{Martinez2018}, for $\ell\geq 16$ the only $kG$-modules $L_{G}(\lambda)$ with $\dim (L_{G}(\lambda))\leq 4\ell^{3}$ have highest weight $\lambda\in \{\omega_{1},\omega_{2},2\omega_{1},\omega_{3},$ $3\omega_{1},\omega_{1}+\omega_{2}\}$. Further, for $\ell\leq 15$, the additional $kG$-modules $L_{G}(\lambda)$, up to duality or outer automorphisms, with $\dim (L_{G}(\lambda))\leq 4\ell^{3}$ have highest weights given in Table \ref{AdditionalReprDl}.

\begin{table*}[h!]
\centering
\begin{tabular}{ c | c | c | c }
\toprule
Rank & $\lambda$ & $p$ & $\dim(L_{G}(\lambda))$ \\
\hline
\multirow{4}{2.5em}{$\ell=4$} & $\omega_{1}+\omega_{4}$ & all & $56-8\varepsilon_{p}(2)$ \\
                              & $2\omega_{2}$ & $p=3$ & $195$ \\
                              & $2\omega_{1}+\omega_{3}$ & all & $224-56\varepsilon_{p}(5)$ \\
                              & $\omega_{1}+\omega_{3}+\omega_{4}$ & $p=2$ & $246$ \\\hline                                         
\end{tabular}
\end{table*}

\begin{table*}[h!]
\centering
\begin{tabular}{ c | c | c | c }
\hline                                         
\multirow{4}{2.5em}{$\ell=5$} & $2\omega_{5}$ & $p\neq 2$ & $126$ \\
                                                   & $\omega_{1}+\omega_{5}$ & all & $144-16\varepsilon_{p}(5)$ \\
                                                   & $\omega_{4}+\omega_{5}$ & all & $210-46\varepsilon_{p}(2)$ \\
                                                   & $\omega_{2}+\omega_{5}$ & $p=2$ & $416$ \\
\hline                                         
\multirow{4}{2.5em}{$\ell=6$} & $\omega_{1}+\omega_{6}$ & all & $352-32\varepsilon_{p}(2)-32\varepsilon_{p}(3)$ \\
                                                    & $\omega_{4}$ & all & $495-131\varepsilon_{p}(2)$\\
                                                    & $2\omega_{6}$ & $p\neq 2$ & $462$\\
                                                    & $\omega_{5}+\omega_{6}$ & all & $792-232\varepsilon_{p}(2)$\\
\hline                                         
\multirow{3}{2.5em}{$\ell=7$} & $\omega_{1}+\omega_{7},$ & all & $832-64\varepsilon_{p}(7)$ \\
                                                    & $\omega_{4}$ & all & $1001-91\varepsilon_{p}(2)$\\
                                                    & $\omega_{5}$ & $p=2$ & $1288$\\
\hline                                         
\multirow{2}{2.5em}{$\ell=8$}  & $\omega_{4}$ & all & $1820-238\varepsilon_{p}(2)$\\
											   & $\omega_{1}+\omega_{8}$ & all & $1920-128\varepsilon_{p}(2)$ \\
\hline                                         
$\ell=9$ & $\omega_{4}$ & $p=2$ & $2906$\\
\hline
$5\leq \ell\leq 15$ & $\omega_{\ell}$ & all & $2^{\ell-1}$\\
\bottomrule
\end{tabular}
\caption{\label{AdditionalReprDl} The additional highest weight modules for groups of type $D_{\ell}$.}
\end{table*}

\section{Proof of Theorem \ref{ResultsAl}}

In this section $G$ is a simple, simply connected  linear algebraic group of type $A_{\ell}$, $\ell\geq 1$. In order to prove Theorem \ref{ResultsAl}, we will sometimes use the algorithm from \cite[Section 2.2]{GurLaw19}, which gives lower bounds for $\nu_{G}(V)$, where $V=L_{G}(\lambda)$ for some $\lambda\in \X(T)^{+}_{p}$. We give a brief description of it in what follows. Let $\Psi$ be a standard subsystem of $\Phi$ and let $\mathcal{W}(\Psi)$ be its Weyl group. We define $r_{\Psi}=\frac{|\mathcal{W}:\mathcal{W}(\Psi)|\cdot |\Phi\setminus \Psi|}{2|\Phi_{s}|}$. For $\scale[0.9]{\displaystyle \lambda^{'}=\sum_{i=1}^{\ell}a_{i}\omega_{i}}\in \X(T)^{+}$, set $\Psi(\lambda^{'})=\langle\alpha_{i}\mid a_{i}=0\rangle$, and for $\lambda\in \X(T)^{+}_{p}$, define $\scale[0.9]{\displaystyle s_{\lambda}=\sum_{0\preceq \lambda^{'}\preceq \lambda}r_{\lambda^{'}}}$. By \cite[Prop. 2.2.1]{GurLaw19}, we have $\nu_{G}(L_{G}(\lambda))\geq s_{\lambda}$. Now, as $\scale[0.9]{\displaystyle s_{\lambda}=\sum_{0\preceq \mu\preceq \lambda}r_{\mu}}$, it will prove extremely useful to give a formula for $r_{\omega_{i}}$, $1\leq i\leq \ell$. To this end, we first note that $\Psi(\omega_{i})$, $1\leq i\leq \ell$, is of type $A_{i-1}A_{\ell-i}$, thus $|\Psi(\omega_{i})|=\ell^{2}-2\ell i+2i^{2}+\ell-2i$ and $|\Phi\setminus \Psi(\omega_{i})|=2i(\ell-i+1)$. Moreover, as $|\mathcal{W}(\Psi(\omega_{i}))|=i!(\ell-i+1)!$, we have $|\mathcal{W}:\mathcal{W}(\Psi(\omega_{i}))|=\binom{\ell+1}{i}$, therefore
\begin{equation}\label{Al_r_omega_i}
\scale[0.9]{r_{\omega_{i}}=\binom{\ell+1}{i}\cdot \frac{2i(\ell-i+1)}{2\ell(\ell+1)}=\binom{\ell-1}{i-1}}.
\end{equation}

Lastly, we fix the following hypothesis on the semisimple elements of $G$: 
\begin{equation*}
\scale[0.9]{\begin{split}
(^{\dagger}H_{s}): & \text{ any }s\in T\setminus \ZG(G) \text{ is such that }s=\diag(\mu_{1}\cdot \I_{n_{1}},\mu_{2}\cdot \I_{n_{2}},\dots, \mu_{m}\cdot \I_{n_{m}}),\text{ where }m\geq 2, \\
                   &\mu_{i}\neq \mu_{j} \text{ for all }i<j, \ \ell\geq n_{1}\geq \cdots \geq n_{m}\geq 1 \text{ and }\prod_{i=1}^{m}\mu_{i}^{n_{i}}=1.
\end{split}}
\end{equation*}

\begin{prop}\label{PropositionAlnatural}
Let $V=L_{G}(\omega_{1})$. Then $\nu_{G}(V)=1$.  Moreover, we have:
\begin{enumerate}
\item[\emph{$(1)$}] $\scale[0.9]{\displaystyle \max_{u\in G_{u}\setminus \{1\}}\dim(V_{u}(1))=\ell}$.
\item[\emph{$(2)$}] $\scale[0.9]{\displaystyle \max_{s\in T\setminus\ZG(G)}\dim(V_{s}(\mu))=\ell}$, where the maximum is attained if and only if $s=\diag(d,d,$ $\dots,d,d^{-\ell})$ with $d^{\ell+1}\neq 1$ and $\mu=d$.
\end{enumerate}
\end{prop}

\begin{proof}
To begin, we note that $V\cong W$ as $kG$-modules. Thus, for all $(s,\mu)\in T\setminus \ZG(G)\times k^{*}$ we have $\dim(V_{s}(\mu))\leq \ell$ and equality holds if and only if $s$ and $\mu$ are as in the statement of the result. In the case of unipotent elements, we have $\scale[0.8]{\displaystyle \max_{u\in G_{u}\setminus \{1\}}\dim(V_{u}(1))=\dim(V_{x_{\alpha_{1}}(1)}(1))=\ell}$, by Lemma \ref{uniprootelems}, and so $\nu_{G}(V)=1$.
\end{proof}

\begin{prop}\label{PropositionAlwedge}
Let $\ell\geq 3$ and let $V=L_{G}(\omega_{2})$. Then $\nu_{G}(V)=\ell-1$. Moreover, we have:
\begin{enumerate}
\item[\emph{$(1)$}] $\scale[0.9]{\displaystyle \max_{u\in G_{u}\setminus \{1\}}\dim(V_{u}(1))}=\frac{\ell^{2}-\ell+2}{2}$.
\item[\emph{$(2)$}] $\scale[0.9]{\displaystyle \max_{s\in T\setminus\ZG(G)}\dim(V_{s}(\mu))}=\frac{\ell(\ell-1)}{2}+\varepsilon_{\ell}(3)$, where the maximum is attained if and only if
\begin{enumerate}
\item[\emph{$(2.1)$}] $\ell=3$ and, up to conjugation, $s=\diag(d,d,\pm d^{-1},\pm d^{-1})$ with $d^{2}\neq \pm 1$, and $\mu=\pm 1$.
\item[\emph{$(2.2)$}] $\ell=4$ and, up to conjugation, $s=\diag(d,d,d,e,e)$ with $d\neq e$ and $d^{3}=e^{-2}$, and $\mu=de$.
\item[\emph{$(2.3)$}] $\ell\geq 4$ and, up to conjugation, $s=\diag(d,\dots,d, d^{-\ell})$ with $d^{\ell+1}\neq 1$, and $\mu=d^{2}$. 
\end{enumerate}
\end{enumerate}
\end{prop}

\begin{proof}
To begin, we note that $V\cong \wedge^{2}(W)$, see \cite[Proposition $4.2.2$]{mcninch_1998}. Then, for the unipotent elements, in view of Lemma \ref{uniprootelems}, we have $\scale[0.9]{\displaystyle \max_{u\in G_{u}\setminus \{1\}}\dim(V_{u}(1))}$ $= \dim((\wedge^{2}(W))_{x_{\alpha_{1}}(1)}(1))$. We write $W=W_{1}\oplus W_{2}$, where $\dim(W_{1})=2$ and $x_{\alpha_{1}}(1)$ acts on $W_{1}$ as $J_{2}$, and $\dim(W_{2})=\ell-1$ and $x_{\alpha_{1}}(1)$ acts trivially on $W_{2}$. Then, since $\displaystyle \wedge^{2}(W_{1}\oplus W_{2})\cong \wedge^{2}(W_{1})\oplus [W_{1}\otimes W_{2}]\oplus \wedge^{2}(W_{2})$, we get $ \dim((\wedge^{2}(W))_{x_{\alpha_{1}}(1)}(1))=\dim((\wedge^{2}(W_{1}))_{x_{\alpha_{1}}(1)}(1))+\dim((W_{1}\otimes W_{2})_{x_{\alpha_{1}}(1)}(1))+\dim((\wedge^{2}(W_{2}))_{x_{\alpha_{1}}(1)}(1))$. We use \cite[Lemma $3.4$]{liebeck_2012unipotent} and Lemma \ref{Lemma on fixed space for wedge in p=2}, to determine that $\scale[0.9]{\displaystyle \max_{u\in G_{u}\setminus \{1\}}\dim(V_{u}(1))=\dim((\wedge^{2}(W))_{x_{\alpha_{1}}(1)}(1))}=\frac{\ell^{2}-\ell+2}{2}$. 

We now calculate $\scale[0.9]{\displaystyle \max_{s\in T\setminus\ZG(G)}\dim(V_{s}(\mu))}$. Let $s\in T\setminus \ZG(G)$ be as in hypothesis $(^{\dagger}H_{s})$. As $V\cong \wedge^{2}(W)$, we deduce that the eigenvalues of $s$ on $V$, not necessarily distinct, are:
\begin{equation}\label{Al_enum_P2}
\scale[0.8]{\begin{cases}
\mu_{i}^{2}$, $1\leq i\leq m$, with multiplicity at least $\frac{n_{i}(n_{i}-1)}{2};\\
\mu_{i}\mu_{j}$, $1\leq i<j\leq m$, with multiplicity at least $n_{i}n_{j}.
\end{cases}}
\end{equation}

Suppose there exists $i$ such that $n_{i}\geq 2$ and consider the eigenvalue $\mu_{i}^{2}$ of $s$ on $V$.  Now, since the $\mu_{r}$'s are distinct, it follows that $\mu_{i}^{2}\neq \mu_{i}\mu_{j}$ for all $i\neq j$, hence $\scale[0.9]{\dim(V_{s}(\mu_{i}^{2}))\leq \frac{\ell(\ell+1)}{2}-(\ell+1-n_{i})n_{i}}$. 

Let $\ell=3$ and assume $\dim(V_{s}(\mu_{i}^{2}))\geq 4$. Then $2-(4-n_{i})n_{i}=(n_{i}-2)^{2}-2\geq 0$, which does not hold as $n_{i}\leq 3$. Thus, we let $\ell\geq 4$ and assume that $\dim(V_{s}(\mu_{i}^{2}))\geq \frac{\ell(\ell-1)}{2}$. Then, $(\ell-n_{i})(1-n_{i})\geq 0$ and, since $\ell\geq n_{i}\geq 2$, the inequality holds if and only if $n_{i}=\ell$.  Hence, $m=2$, $n_{1}=\ell$, $n_{2}=1$, $\mu_{2}=\mu_{1}^{-\ell}$ and $\mu_{1}^{\ell+1}\neq 1$, as $\mu_{1}\neq \mu_{2}$ and $\mu_{1}^{\ell}\mu_{2}=1$. This gives $\dim(V_{s}(\mu_{i}^{2}))\leq \frac{\ell(\ell-1)}{2}$ for all $s\in T\setminus \ZG(G)$, where equality holds if and only if $i=1$ and, up to conjugation, $s=\diag(\mu_{1},\dots, \mu_{1},\mu_{1}^{-\ell})$ with $\mu_{1}^{\ell+1}\neq 1$,  as in $(2.3)$. 

Now, fix $ i<j$ and consider the eigenvalue $\mu_{i}\mu_{j}$ of $s$ on $V$. Since the $\mu_{r}$'s are distinct, we remark that:
\begin{equation*}
\scale[0.8]{\begin{cases}
\mu_{i}\mu_{j}\neq \mu_{i}^{2}$ and $\mu_{i}\mu_{j}\neq \mu_{j}^{2};\\
\mu_{i}\mu_{j}\neq \mu_{i}\mu_{r}$, for $i<r\leq m$ and $r\neq j$, and $\mu_{i}\mu_{j}\neq \mu_{r}\mu_{i}$, for $1\leq r<i;\\
\mu_{i}\mu_{j}\neq \mu_{r}\mu_{j}$, for $1\leq r<j$ and $r\neq i$, and $\mu_{i}\mu_{j}\neq \mu_{j}\mu_{r}$, for $j<r \leq m.
\end{cases}}
\end{equation*}
By \eqref{Al_enum_P2}, these account for at least $\frac{n_{i}(n_{i}-1)}{2}+\frac{n_{j}(n_{j}-1)}{2}+(n_{i}+n_{j})(\ell+1-n_{i}-n_{j})$ eigenvalues of $s$ on $V$ different than $\mu_{i}\mu_{j}$. Hence, we have $\scale[0.9]{\dim(V_{s}(\mu_{i}\mu_{j}))\leq \frac{\ell(\ell+1)}{2}-\frac{n_{i}(n_{i}-1)}{2}-\frac{n_{j}(n_{j}-1)}{2}-(n_{i}+n_{j})(\ell+1-n_{i}-n_{j})}$.

Let $\ell=3$. Then, since $n_{i}+n_{j}\leq 4$, we have $(n_{i},n_{j})\in \{(3,1), (2,2), (2,1), (1,1)\}$. Assume that $\dim(V_{s}(\mu_{i}\mu_{j}))\geq 4$. Then $2-\frac{n_{i}(n_{i}-1)}{2}-\frac{n_{j}(n_{j}-1)}{2}-(n_{i}+n_{j})(4-n_{i}-n_{j})\geq 0$. Now, the inequality holds if and only if $n_{i}=n_{j}=2$, i.e. $m=2$, $n_{1}=n_{2}=2$, $\mu_{2}=\pm \mu_{1}^{-1}$, as $\mu_{1}^{2}\mu_{2}^{2}=1$, and $\mu_{1}^{2}\neq \pm 1$, as $\mu_{1}\neq \mu_{2}$. Thus, $\dim(V_{s}(\mu_{i}\mu_{j}))\leq 4$ for all $s\in T\setminus \ZG(G)$ and all $i<j$, where equality holds if and only if, up to conjugation, $s=\diag(\mu_{1},\mu_{1},\pm \mu_{1}^{-1},\pm \mu_{1}^{-1})$ with $\mu_{1}^{2}\neq \pm 1$ and $\mu_{i}\mu_{j}=\pm 1$, as in $(2.1)$. We now let $\ell\geq 4$ and assume $\dim(V_{s}(\mu_{i}\mu_{j}))\geq \frac{\ell(\ell-1)}{2}$. Then
\begin{equation}\label{Alwedineq1}
\scale[0.9]{(\ell-n_{i}-n_{j})(1-n_{i}-n_{j})-\frac{n_{i}(n_{i}-1)}{2}-\frac{n_{j}(n_{j}-1)}{2}\geq 0}.
\end{equation}
Since $n_{i}\geq n_{j}\geq 1$, we have $\frac{n_{i}(n_{i}-1)}{2}+\frac{n_{j}(n_{j}-1)}{2}\geq 0$ and $1-n_{i}-n_{j}<0$,  therefore, by inequality \eqref{Alwedineq1},  it follows that $\ell-n_{i}-n_{j}\leq  0$. If $n_{i}+n_{j}=\ell$, then, for inequality \eqref{Alwedineq1} to hold, we must have $\frac{n_{i}(n_{i}-1)}{2}+\frac{n_{j}(n_{j}-1)}{2}=0$, hence $n_{i}=n_{j}=1$, contradicting $\ell\geq 4$. On the other hand, if $n_{i}+n_{j}=\ell+1$, then $m=2$ and, by \eqref{Alwedineq1}, we determine that $n_{1}(3-n_{1})-(n_{2}-1)(n_{2}-2)\geq 0$. Now, the inequality holds if and only if $n_{1}=3$ and $n_{2}=2$, as $\ell\geq 4$. In this case, $\ell=4$ and $s=\diag(\mu_{1},\mu_{1},\mu_{1},\mu_{2},\mu_{2})$ with $\mu_{1}\neq \mu_{2}$ and $\mu_{1}^{3}=\mu_{2}^{-2}$. Therefore, $\dim(V_{s}(\mu_{i}\mu_{j}))\leq \frac{\ell(\ell-1)}{2}$ for all $s\in T\setminus \ZG(G)$ and all $i<j$, where equality holds if and only if $\ell=4$, $\mu_{i}\mu_{j}=\mu_{1}\mu_{2}$ and, up to conjugation, $s=\diag(\mu_{1},\mu_{1},\mu_{1},\mu_{2},\mu_{2})$ with $\mu_{1}\neq \mu_{2}$ and $\mu_{1}^{3}=\mu_{2}^{-2}$, as in $(2.2)$. In conclusion, we showed that $\scale[0.9]{\displaystyle \max_{s\in T\setminus\ZG(G)}\dim(V_{s}(\mu))}=\frac{\ell(\ell-1)}{2}+\varepsilon_{\ell}(3)$ and, as $\scale[0.9]{\displaystyle \max_{u\in G_{u}\setminus \{1\}}\dim(V_{u}(1))}=\frac{\ell^{2}-\ell+2}{2}$, we determine that $\nu_{G}(V)=\ell-1$.
\end{proof}

\begin{prop}\label{PropositionAlsymmetric}
Assume $p\neq 2$ and let $V=L_{G}(2\omega_{1})$. Then $\nu_{G}(V)=\ell$.  Moreover, we have:
\begin{enumerate}
\item[\emph{$(1)$}] $\scale[0.9]{\displaystyle \max_{u\in G_{u}\setminus \{1\}}\dim(V_{u}(1))}=\binom{\ell+1}{2}$.
\item[\emph{$(2)$}] $\scale[0.9]{\displaystyle \max_{s\in T\setminus\ZG(G)}\dim(V_{s}(\mu))}=\frac{\ell^{2}+\ell+2}{2}$, where the maximum is attained if and only if, up to conjugation, $s=\diag(d,\dots,d,d^{-\ell})$ with $d^{\ell+1}=-1$, and $\mu=d^{2}$. 
\end{enumerate}
\end{prop}

\begin{proof}
To begin, we note that $V\cong \SW(W)$, see \cite[Proposition $4.2.2$]{mcninch_1998}. Thus, for the unipotent elements, in view of Lemma \ref{uniprootelems}, we have $\scale[0.9]{\displaystyle \max_{u\in G_{u}\setminus \{1\}}\dim(V_{u}(1))}$ $=\dim((\SW(W))_{x_{\alpha_{1}}(1)}(1))$. We write $W=W_{1}\oplus W_{2}$, where $\dim(W_{1})=2$ and $x_{\alpha_{1}}(1)$ acts on $W_{1}$ as $J_{2}$, and $\dim(W_{2})=\ell-1$ and $x_{\alpha_{1}}(1)$ acts trivially on $W_{2}$. Then, since $\SW(W_{1}\oplus W_{2})\cong \SW(W_{1})\oplus [W_{1}\otimes W_{2}]\oplus \SW(W_{2})$, we have $\dim((\SW(W))_{x_{\alpha_{1}}(1)}(1))=\dim((\SW(W_{1}))_{x_{\alpha_{1}}(1)}(1))+\dim((W_{1}\otimes W_{2})_{x_{\alpha_{1}}(1)}(1))+\dim((\SW(W_{2}))_{x_{\alpha_{1}}(1)}(1))$. By \cite[Lemma $3.4$]{liebeck_2012unipotent}, we determine that $\scale[0.9]{\displaystyle \max_{u\in G_{u}\setminus \{1\}}\dim(V_{u}(1))=\dim((\SW(W))_{x_{\alpha_{1}}(1)}(1))}$ $=\binom{\ell+1}{2}$.

We now calculate $\scale[0.9]{\displaystyle \max_{s\in T\setminus\ZG(G)}\dim(V_{s}(\mu))}$. Let $s\in T\setminus \ZG(G)$ be as in hypothesis $(^{\dagger}H_{s})$.  Since, $V\cong \SW(W)$, the eigenvalues of $s$ on $V$, not necessarily distinct, are:
\begin{equation}\label{Al_enum_P1}
\scale[0.8]{\begin{cases}
\mu_{i}^{2}$, $1\leq i\leq m$, with multiplicity at least $\frac{n_{i}(n_{i}+1)}{2};\\
\mu_{i}\mu_{j}$, $1\leq i<j\leq m$, with multiplicity at least $n_{i}n_{j}.
\end{cases}}
\end{equation}

Fix some $i$ and consider the eigenvalue $\mu_{i}^{2}$ of $s$ on $V$. Since the $\mu_{r}$'s are distinct, we deduce that $\dim(V_{s}(\mu_{i}^{2}))\leq \frac{(\ell+1)(\ell+2)}{2}-n_{i}(\ell+1-n_{i})$. Assume $\dim(V_{s}(\mu_{i}^{2}))\geq \frac{\ell^{2}+\ell+2}{2}$. Then $(\ell-n_{i})(1-n_{i})\geq 0$, which holds if and only if  $n_{i}\in \{1,\ell\}$. In both cases, we get $\dim(V_{s}(\mu_{i}^{2}))\leq \frac{\ell^{2}+\ell+2}{2}$. Now, equality holds if and only if $\mu_{i}^{2}=\mu_{j}^{2}$ for all $j\neq i$, in which case, we get $m=2$, as the $\mu_{r}$'s are distinct, and $s=\diag(\mu_{1},\dots,\mu_{1},\mu_{1}^{-\ell})$ with $\mu_{1}^{\ell+1}=-1$, as $\mu_{1}\neq \mu_{2}$, $\mu_{1}^{\ell}\mu_{2}=1$ and $\mu_{1}^{2}=\mu_{2}^{2}$. 

Lastly, fix some $i<j$ and consider the eigenvalue $\mu_{i}\mu_{j}$ of $s$ on $V$. Since the $\mu_{r}$'s are distinct, we have
\begin{equation*}
\scale[0.8]{\begin{cases}
\mu_{i}\mu_{j}\neq \mu_{i}^{2}$ and $\mu_{i}\mu_{j}\neq \mu_{j}^{2};\\
\mu_{i}\mu_{j}\neq \mu_{i}\mu_{r}$, for $i<r\leq m$ and $r\neq j$, and $\mu_{i}\mu_{j}\neq \mu_{r}\mu_{i}$, for $1\leq r<i;\\
\mu_{i}\mu_{j}\neq \mu_{r}\mu_{j}$, for $1\leq r<j$ and $r\neq i$, and $\mu_{i}\mu_{j}\neq \mu_{j}\mu_{r}$, for $j<r \leq m.
\end{cases}}
\end{equation*}
By \eqref{Al_enum_P1}, these account for at least $\frac{n_{i}(n_{i}+1)}{2}+\frac{n_{j}(n_{j}+1)}{2}+(n_{i}+n_{j})(\ell+1-n_{i}-n_{j})$ eigenvalues of $s$ on $V$ which are different than $\mu_{i}\mu_{j}$. Hence, we have $\dim(V_{s}(\mu_{i}\mu_{j}))\leq \frac{(\ell+1)(\ell+2)}{2}-\frac{n_{i}(n_{i}+1)}{2}-\frac{n_{j}(n_{j}+1)}{2}-(n_{i}+n_{j})(\ell+1-n_{i}-n_{j})$. Assume $\dim(V_{s}(\mu_{i}\mu_{j}))\geq \frac{\ell^{2}+\ell+2}{2}$. Then 
\begin{equation}\label{Alsymeq1}
\scale[0.9]{(\ell-n_{i}-n_{j})(1-n_{i}-n_{j})-\frac{n_{i}(n_{i}+1)+n_{j}(n_{j}+1)}{2}\geq 0}.
\end{equation}
As $n_{i}\geq n_{j}\geq 1$, we have $\frac{n_{i}(n_{i}+1)+n_{j}(n_{j}+1)}{2}\geq 2$, therefore
$(\ell-n_{i}-n_{j})(1-n_{i}-n_{j})-2\geq 0$.  If $n_{i}+n_{j}\leq \ell$, then $(\ell-n_{i}-n_{j})(1-n_{i}-n_{j})\leq 0$ and the inequality does not hold. If $n_{i}+n_{j}=\ell+1$, then $m=2$, $n_{2}=\ell+1-n_{1}$ and substituting in \eqref{Alsymeq1} gives $0\leq \ell-\frac{n_{1}(n_{1}+1)+(\ell+1-n_{1})(\ell+2-n_{1})}{2}=-\frac{[(\ell-n_{1})^{2}+(n_{1}-1)^{2}+\ell+1]}{2}$. Therefore, $\dim(V_{s}(\mu_{i}\mu_{j}))<\frac{\ell^{2}+\ell+2}{2}$ for all $s\in T\setminus \ZG(G)$ and all $i<j$. We conclude that $\scale[0.9]{\displaystyle \max_{s\in T\setminus\ZG(G)}\dim(V_{s}(\mu))}$ $=\frac{\ell^{2}+\ell+2}{2}$ and, consequently, $\nu_{G}(V)=\ell$.
\end{proof}

\begin{prop}\label{PropositionAlom1+oml}
Let $\ell\geq 2$ and $V=L_{G}(\omega_{1}+\omega_{\ell})$. Then $\nu_{G}(V)=2\ell-\varepsilon_{p}(3)\varepsilon_{\ell}(2)$. Moreover, we have:
\begin{enumerate}
\item[\emph{$(1)$}] $\scale[0.9]{\displaystyle \max_{u\in G_{u}\setminus \{1\}}\dim(V_{u}(1))=\ell^{2}-\varepsilon_{p}(\ell+1)}$.
\item[\emph{$(2)$}] $\scale[0.9]{\displaystyle \max_{s\in T\setminus\ZG(G)}\dim(V_{s}(\mu))=\ell^{2}-\varepsilon_{p}(\ell+1)+\varepsilon_{p}(3)\varepsilon_{\ell}(2)}$, where the maximum is attained if and only if
\begin{enumerate}
\item[\emph{$(2.1)$}] $p\neq 2$, $\ell=2$, $\mu=-1$ and, up to conjugation, $\scale[0.9]{s=\diag(d,d,d^{-2})}$ with $d^{3}=-1$.
\item[\emph{$(2.2)$}] $p\neq 3$, $\ell=2$, $\mu=1$ and, up to conjugation, $\scale[0.9]{s=\diag(d,d,d^{-2})}$ with $d^{3}\neq 1$.
\item[\emph{$(2.3)$}] $\ell\geq 3$, $\mu=1$ and, up to conjugation, $\scale[0.9]{s=\diag(d,\dots, d,d^{-\ell})}$ with $d^{\ell+1}\neq 1$. 
\end{enumerate}
\end{enumerate}
\end{prop}

\begin{proof}
To begin, we note that if $\varepsilon_{p}(\ell+1)=0$, then $W\otimes W^{*}\cong V\oplus L_{G}(0)$, while, if $\varepsilon_{p}(\ell+1)=1$, then $W\otimes W^{*}\cong L_{G}(0)\mid V\mid L_{G}(0)$, see \cite[Proposition $4.6.10$]{mcninch_1998}. For the unipotent elements, in view of Lemma \ref{uniprootelems}, we have $\scale[0.9]{\displaystyle \max_{u\in G_{u}\setminus \{1\}}\dim(V_{u}(1))}$ $=\dim(V_{x_{\alpha_{1}}(1)}(1))$. In what follows, we calculate $\dim((W\otimes W^{*})_{x_{\alpha_{1}}(1)}(1))$. We write $W=W_{1}\oplus W_{2}$, where $\dim(W_{1})=2$ and $x_{\alpha_{1}}(1)$ acts on $W_{1}$ as $J_{2}$, and $\dim(W_{2})=\ell-1$ and $x_{\alpha_{1}}(1)$ acts trivially on $W_{2}$. Then, since $(W_{1}\oplus W_{2})\otimes (W_{1}\oplus W_{2})^{*}\cong [W_{1}\otimes W_{1}^{*}]\oplus  [W_{1}\otimes W_{2}^{*}]\oplus  [W_{2}\otimes W_{1}^{*}]\oplus  [W_{2}\otimes W_{2}^{*}]$, using \cite[Lemma $3.4$]{liebeck_2012unipotent}, we determine that $\dim((W\otimes W^{*})_{x_{\alpha_{1}}(1)}(1))=\ell^{2}+1$. Then, by the structure of $W\otimes W^{*}$ as a $kG$-module and \cite[Theorem $6.1$]{Korhonen_2019}, we determine that $\displaystyle \max_{u\in G_{u}\setminus \{1\}}\dim(V_{u}(1))=\dim(V_{x_{\alpha_{1}}(1)}(1))=\dim((W\otimes W^{*})_{x_{\alpha_{1}}(1)}(1))-1-\varepsilon_{p}(\ell+1)=\ell^{2}-\varepsilon_{p}(\ell+1)$. 

We now focus on the semisimple elements. Let $s\in T\setminus \ZG(G)$ be as in hypothesis $(^{\dagger}H_{s})$. By the structure of $W\otimes W^{*}$ as a $kG$-module, we determine that the eigenvalues of $s$ on $V$, not necessarily distinct, are:

\begin{equation}\label{Al_enum_P3}
\scale[0.8]{\begin{cases}
1$ with multiplicity $\displaystyle \sum_{i=1}^{m}n_{i}^{2}-1-\varepsilon_{p}(\ell+1);\\
\mu_{i}\mu_{j}^{-1}$ and $\mu_{i}^{-1}\mu_{j}$, where $1\leq i<j \leq m$, each with multiplicity at least $n_{i}n_{j}.
\end{cases}}
\end{equation}

We first consider the eigenvalue $1$ of $s$ on $V$. Since the $\mu_{i}$'s are distinct, it follows that $\scale[0.9]{\dim(V_{s}(1))=\displaystyle \sum_{i=1}^{m}n_{i}^{2}-1}$ $\scale[0.9]{\displaystyle -\varepsilon_{p}(\ell+1)= (\displaystyle \sum_{i=1}^{m}n_{i})^{2}-2\displaystyle \sum_{i<j}n_{i}n_{j}-1-\varepsilon_{p}(\ell+1)= \ell^{2}+2\ell-2\displaystyle \sum_{i<j}n_{i}n_{j}-\varepsilon_{p}(\ell+1)}$. Assume $\dim(V_{s}(1))\geq \ell^{2}-\varepsilon_{p}(\ell+1)+\varepsilon_{p}(3)\varepsilon_{\ell}(2)$. Then $\scale[0.9]{2\displaystyle\ell-2\sum_{i<j}n_{i}n_{j}-\varepsilon_{p}(3)\varepsilon_{\ell}(2)\geq 0}$ and, since $\scale[0.9]{\ell=\displaystyle \sum_{i=1}^{m}n_{i}-1}$, we have that:
\begin{equation}\label{eqn: Al_P3_1}
\scale[0.9]{\displaystyle 2(1-n_{2})(n_{1}-1)+2\sum_{i=3}^{m}n_{i}(1-\sum_{j=1}^{i-1}n_{j})-\varepsilon_{p}(3)\varepsilon_{\ell}(2)\geq 0.}
\end{equation}
But $\scale[0.9]{\displaystyle \sum_{i=3}^{m}n_{i}(1-\sum_{j=1}^{i-1}n_{j})\leq 0}$ and $(1-n_{2})(n_{1}-1)\leq 0$, since $n_{i}\geq 1$ for all $1\leq i\leq m$, and so (\ref{eqn: Al_P3_1}) holds if and only if either $\ell=2$, $p\neq 3$,  $m=2$, $n_{2}=1$ and $n_{1}=2$; or $\ell\geq 3$, $m=2$, $n_{2}=1$ and $n_{1}=\ell$. In both cases, as $\mu_{1}\neq \mu_{2}$ and $\mu_{1}^{\ell}\mu_{2}=1$, we get $\mu_{2}=\mu_{1}^{-\ell}$ and $\mu_{1}^{\ell+1}\neq 1$. Thus, we have shown that $\dim(V_{s}(1))\leq \ell^{2}-\varepsilon_{p}(\ell+1)+\varepsilon_{p}(3)\varepsilon_{\ell}(2)$ for all $s\in T\setminus \ZG(G)$ and that equality holds if and only if $p$, $\ell$, $s$ and $\mu$ are as in $(2.2)$, or $(2.3)$. 

We now fix $i<j$ and consider the eigenvalue $\mu_{i}\mu_{j}^{-1}$ of $s$ on $V$. If $\mu_{i}\mu_{j}^{-1}\neq \mu_{i}^{-1}\mu_{j}$, then $\dim(V_{s}(\mu_{i}\mu_{j}^{-1}))\leq \dim(V)-\dim(V_{s}(1))-\dim(V_{s}(\mu_{i}^{-1}\mu_{j}))$. Since $n_{r}\geq 1$ for all $1\leq r\leq m$, we have $\scale[0.9]{\displaystyle \sum_{r=1}^{m}n_{r}^{2}\geq \displaystyle \sum_{r=1}^{m}n_{r}=\ell+1}$, and, since $V$ is self-dual, we deduce that $\dim(V_{s}(\mu_{i}\mu_{j}^{-1}))\leq \frac{(\ell+1)^{2}-(\ell+1)}{2}=\frac{\ell(\ell+1)}{2}<\ell^{2}-\varepsilon_{p}(\ell+1)+\varepsilon_{p}(3)\varepsilon_{\ell}(2)$ for all eigenvalues $\mu_{i}\mu_{j}^{-1}\neq \mu_{i}^{-1}\mu_{j}$. We thus assume that $p\neq 2$ and $\mu_{i}\mu_{j}^{-1}=-1$. Since the $\mu_{r}$'s are distinct, we remark that:

\begin{equation*}
\scale[0.8]{\begin{cases}
\mu_{i}\mu_{r}^{-1}\neq -1$ and $\mu_{i}^{-1}\mu_{r}\neq -1$, where $i<r\leq m$, $r\neq j$; and $\mu_{r}^{-1}\mu_{i}\neq -1$ and $\mu_{r}\mu_{i}^{-1}\neq -1$, where $1\leq r<i;\\
\mu_{r}\mu_{j}^{-1}\neq -1$ and $\mu_{r}^{-1}\mu_{j}\neq -1$, where $1\leq r <j$, $r\neq i$; and $\mu_{j}^{-1}\mu_{r}\neq -1$ and $\mu_{j}\mu_{r}^{-1}\neq -1$, where $j< r \leq m.
\end{cases}}
\end{equation*}
By \eqref{eqn: Al_P3_1}, it follows that $\dim(V_{s}(-1))\leq \dim(V)-\dim(V_{s}(1))-2(n_{i}+n_{j})(\ell+1-n_{i}-n_{j})$. Assume that $\dim(V_{s}(-1))\geq \ell^{2}-\varepsilon_{p}(\ell+1)+\varepsilon_{p}(3)\varepsilon_{\ell}(2)$. Then:
\begin{equation}\label{eqn: Al_P3_4}
\scale[0.9]{\displaystyle 2(\ell-n_{i}-n_{j})(1-n_{i}-n_{j})-\sum_{r=1}^{m}n_{r}^{2}+1+\varepsilon_{p}(\ell+1)-\varepsilon_{p}(3)\varepsilon_{\ell}(2)\geq 0.}
\end{equation}
Since $\scale[0.9]{\displaystyle \sum_{r=1}^{m}n_{r}^{2}\geq \ell+1}$ and $\ell\geq 2$, for inequality \eqref{eqn: Al_P3_4} to hold, we must have $(\ell-n_{i}-n_{j})(1-n_{i}-n_{j})>0$. But then, as $n_{i}\geq n_{j}\geq 1$, we must have $\ell-n_{i}-n_{j}<0$ and so $m=2$ and $n_{1}+n_{2}=\ell+1$.  Substituting in \eqref{eqn: Al_P3_4} gives $-(n_{2}-1)^{2}+n_{1}(2-n_{1})+\varepsilon_{p}(\ell+1)-\varepsilon_{p}(3)\varepsilon_{\ell}(2)\geq 0$. Now, if $n_{1}\geq 3$, the inequality does not hold. Therefore $n_{1}=2$ and $n_{2}\in \{1,2\}$. If $n_{2}=2$, then $\ell=3$ and, as $p\neq 2$, we get $-(n_{2}-1)^{2}+n_{1}(2-n_{1})+\varepsilon_{p}(\ell+1)-\varepsilon_{p}(3)\varepsilon_{\ell}(2)=-1$. On the other hand, if $n_{2}=1$, then $\ell=2$ and $-(n_{2}-1)^{2}+n_{1}(2-n_{1})+\varepsilon_{p}(\ell+1)-\varepsilon_{p}(3)\varepsilon_{\ell}(2)=0$. We conclude that $\dim(V_{s}(-1))\leq \ell^{2}-\varepsilon_{p}(\ell+1)+\varepsilon_{p}(3)\varepsilon_{\ell}(2)$, where equality holds if and only if $p$, $\ell$, $s$ and $\mu$ are as in $(2.1)$. Therefore, $\displaystyle \max_{s\in T\setminus\ZG(G)}\dim(V_{s}(\mu))=\ell^{2}-\varepsilon_{p}(\ell+1)+\varepsilon_{p}(3)\varepsilon_{\ell}(2)$ and, as $\displaystyle \max_{u\in G_{u}\setminus \{1\}}\dim(V_{u}(1))$ $=\ell^{2}-\varepsilon_{p}(\ell+1)$, it follows that $\nu_{G}(V)=2\ell-\varepsilon_{p}(3)\varepsilon_{\ell}(2)$.
\end{proof}

\begin{prop}\label{PropositionAlwedgecube}
Let $\ell\geq 5$ and $V=L_{G}(\omega_{3})$. Then $\nu_{G}(V)=\binom{\ell-1}{2}$. Moreover, we have $\scale[0.9]{\displaystyle \max_{u\in G_{u}\setminus \{1\}}\dim(V_{u}(1))}$ $=\binom{\ell}{3}+\ell-1$ and $\scale[0.9]{\displaystyle \max_{s\in T\setminus\ZG(G)}\dim(V_{s}(\mu))}$ $\leq \binom{\ell}{3}+2$.
\end{prop}

\begin{proof}
As $\omega_{3}$ is a minuscule weight, we have $s_{\omega_{3}}=r_{\omega_{3}}=\binom{\ell-1}{2}$, see \eqref{Al_r_omega_i}, thus $\nu_{G}(V)\geq \binom{\ell-1}{2}$. Further, by \cite[Proposition $4.2.2$]{mcninch_1998}, we have $V\cong \wedge^{3}(W)$. 

For the unipotent elements, in view of Lemma \ref{uniprootelems}, we have $\displaystyle \max_{u\in G_{u}\setminus \{1\}}\dim(V_{u}(1))=\dim((\wedge^{3}(W))_{x_{\alpha_{1}}(1)}(1))$. We write $W=W_{1}\oplus W_{2}$, where $\dim(W_{1})=2$ and $x_{\alpha_{1}}(1)$ acts as $J_{2}$ on $W_{1}$, and $\dim(W_{2})=\ell-1$ and $x_{\alpha_{1}}(1)$ acts trivially on $W_{2}$. We have $\wedge^{3}(W)\cong \wedge^{3}(W_{1})\oplus [\wedge^{2}(W_{1})\otimes W_{2}]\oplus [W_{1}\otimes \wedge^{2}(W_{2})]\oplus \wedge^{3}(W)$ and so $\dim((\wedge^{3}(W))_{x_{\alpha_{1}}(1)}(1)) = \dim((\wedge^{3}(W_{1}))_{x_{\alpha_{1}}(1)}(1))+\dim((\wedge^{2}(W_{1})\otimes W_{2})_{x_{\alpha_{1}}(1)}(1))+\dim((W_{1}\otimes \wedge^{2}(W_{2}))_{x_{\alpha_{1}}(1)}(1))+\dim((\wedge^{3}(W_{2}))_{x_{\alpha_{1}}(1)}(1))$. Using \cite[Lemma $3.4$]{liebeck_2012unipotent} and Lemma \ref{Lemma on fixed space for wedge in p=2}, we determine that $\displaystyle \max_{u\in G_{u}\setminus \{1\}}\dim(V_{u}(1))$ $=\binom{\ell}{3}+\ell-1$.

We now focus on the semisimple elements. Set $\lambda=\omega_{3}$ and $L=L_{1}$. By Lemma \ref{weightlevelAl}, we have $e_{1}(\lambda)=1$, therefore $\displaystyle V\mid_{[L,L]}=V^{0}\oplus V^{1}$. By \cite[Proposition]{Smith_82}, we have $V^{0}\cong L_{L}(\omega_{3})$ and, since the weight $\displaystyle (\lambda-\alpha_{1}-\alpha_{2}-\alpha_{3})\mid_{T_{1}}=\omega_{4}$ admits a maximal vector in $V^{1}$, by dimensional considerations, it follows that
\begin{equation}\label{DecompVAlom3}
V\mid_{[L,L]}\cong L_{L}(\omega_{3})\oplus L_{L}(\omega_{4}).
\end{equation}

Let $s\in T\setminus \ZG(G)$. If $\dim(V^{i}_{s}(\mu))=\dim(V^{i})$ for some eigenvalue $\mu$ of $s$ on $V$, where $i=0, 1$, then $s\in \ZG(L)^{\circ}\setminus \ZG(G)$ and acts on $V^{i}$ as scalar multiplication by $c^{\ell-2-i(\ell+1)}$. As $c^{\ell+1}\neq 1$, we have $\dim(V_{s}(\mu))\leq \binom{\ell}{3}$ for all eigenvalues $\mu$ of $s$ on $V$. We thus assume that $\dim(V^{i}_{s}(\mu))<\dim(V^{i})$ for all eigenvalues $\mu$ of $s$ on $V$ and for both $i=0,1$.  We write $s=z\cdot h$, where $z\in \ZG(L)^{\circ}$ and $h\in [L,L]$. First, let $\ell=5$. Using \eqref{DecompVAlom3} and Proposition \ref{PropositionAlwedge}, we determine that $\dim(V_{s}(\mu))\leq 12$ for all eigenvalues $\mu$ of $s$ on $V$. We now assume that $\ell\geq 6$, and, by \eqref{DecompVAlom3} and Proposition \ref{PropositionAlwedge}, we have $\dim(V_{s}(\mu))\leq \frac{(\ell-1)(\ell-2)}{2}+\dim((L_{L}(\omega_{4}))_{h}(\mu_{h}))$. Recursively and using the result for $\ell=5$, we get $\scale[0.9]{\displaystyle\dim(V_{s}(\mu))\leq \sum_{j=5}^{\ell-1}}\frac{j(j-1)}{2}+12=\binom{\ell}{3}+2$ for all eigenvalues $\mu$ of $s$ on $V$. In conclusion, we have shown that $\scale[0.9]{\displaystyle \max_{s\in T\setminus\ZG(G)}\dim(V_{s}(\mu))}\leq \binom{\ell}{3}+2$ and, as $\scale[0.9]{\displaystyle \max_{u\in G_{u}\setminus \{1\}}\dim(V_{u}(1))}=\binom{\ell}{3}+\ell-1$, we determine that $\nu_{G}(V)=\binom{\ell-1}{2}$.
\end{proof}

\begin{prop}\label{PropositionAlsymmcube}
Assume $p\neq 2,3$ and let $V=L_{G}(3\omega_{1})$. Then $\nu_{G}(V)=\frac{\ell^{2}+\ell+2}{2}$. Moreover, we have $\scale[0.9]{\displaystyle \max_{u\in G_{u}\setminus \{1\}}\dim(V_{u}(1))}=\binom{\ell+2}{3}$ and $\scale[0.9]{\displaystyle \max_{s\in T\setminus\ZG(G)}\dim(V_{s}(\mu))}=\binom{\ell+2}{3}+\ell$.
\end{prop}

\begin{proof}
By \cite{LuTables}, the sub-dominant weights in $V$ are $3\omega_{1}$, $\omega_{1}+\omega_{2}$ and $\omega_{3}$. Therefore, $s_{3\omega_{1}}=r_{3\omega_{1}}+r_{\omega_{1}+\omega_{2}}+r_{\omega_{3}}$. As $\Psi(\omega_{1}+\omega_{2})$ is of type $A_{\ell-2}$, we have $|\Phi\setminus \Psi(\omega_{1}+\omega_{2})_{s}|=2(2\ell-1)$ and $|\mathcal{W}:\mathcal{W}(\Psi(\omega_{1}+\omega_{2}))|=\ell(\ell+1)$, hence $r_{\omega_{1}+\omega_{2}}=2\ell-1$. Lastly, using \eqref{Al_r_omega_i}, we get $s_{3\omega_{1}}=\frac{\ell^{2}+\ell+2}{2}$, therefore $\nu_{G}(V)\geq \frac{\ell^{2}+\ell+2}{2}$.

We begin with the unipotent elements. As $V\cong \SWT(W)$, see \cite[Proposition $4.2.2$]{mcninch_1998}, by Lemma \ref{uniprootelems}, it follows that $\displaystyle \max_{u\in G_{u}\setminus \{1\}}\dim(V_{u}(1))=\dim(( \SWT(W))_{x_{\alpha_{1}}(1)}(1))$. We write $W=W_{1}\oplus W_{2}$, where $\dim(W_{1})=2$ and $x_{\alpha_{1}}(1)$ acts as $J_{2}$ on $W_{1}$, and $\dim(W_{2})=\ell-1$ and $x_{\alpha_{1}}(1)$ acts trivially on $W_{2}$. We have that $\SWT(W)\cong \SWT(W_{1})\oplus [\SW(W_{1})\otimes W_{2}]\oplus [W_{1}\otimes \SW(W_{2})]\oplus \SWT(W)$, thus $\dim((\SWT(W))_{x_{\alpha_{1}}(1)}(1)) = \dim((\SWT(W_{1}))_{x_{\alpha_{1}}(1)}(1))+\dim((\SW(W_{1})\otimes W_{2})_{x_{\alpha_{1}}(1)}(1))+\dim((W_{1}\otimes \SW(W_{2}))_{x_{\alpha_{1}}(1)}(1)) +\dim((\SWT(W_{2}))_{x_{\alpha_{1}}(1)}(1))$. Using\cite[Lemma $3.4$]{liebeck_2012unipotent}, we determine that $\displaystyle \max_{u\in G_{u}\setminus \{1\}}\dim(V_{u}(1))$ $=\binom{\ell+2}{3}$.

We now focus on the semisimple elements. Let $\lambda=3\omega_{1}$ and $L=L_{1}$. By Lemma \ref{weightlevelAl}, we have $e_{1}(\lambda)=3$, therefore $\displaystyle V\mid_{[L,L]}=V^{0}\oplus V^{1}\oplus V^{2}\oplus V^{3}$. First, by \cite[Proposition]{Smith_82}, we have $V^{0}\cong L_{L}(0)$. Secondly, the weight $\displaystyle (\lambda-\alpha_{1})\mid_{T_{1}}=\omega_{2}$ admits a maximal vector in $V^{1}$, thus $V^{1}$ has a composition factor isomorphic to $L_{L}(\omega_{2})$. Similarly, since the weight $\displaystyle (\lambda-2\alpha_{1})\mid_{T_{1}}=2\omega_{2}$ admits a maximal vector in $V^{2}$, it follows that $V^{2}$ has a composition factor isomorphic to $L_{L}(2\omega_{2})$. Lastly, the weight $\displaystyle (\lambda-3\alpha_{1})\mid_{T_{1}}=3\omega_{2}$ admits a maximal vector in $V^{3}$, thus $V^{3}$ has a composition factor isomorphic to $L_{L}(3\omega_{2})$. By dimensional considerations, we determine that:
\begin{equation}\label{DecompVLAl3om1}
V\mid_{[L,L]}\cong L_{L}(0)\oplus L_{L}(\omega_{2})\oplus L_{L}(2\omega_{2})\oplus L_{L}(3\omega_{2}).
\end{equation}

Let $s\in T\setminus \ZG(G)$. If $\dim(V^{i}_{s}(\mu))=\dim(V^{i})$ for some eigenvalue $\mu$ of $s$ on $V$, where $0\leq i\leq 3$, then $s\in \ZG(L)^{\circ}\setminus \ZG(G)$ and acts on $V^{i}$ as scalar multiplication by $c^{3\ell-i(\ell+1)}$. As $c^{\ell+1}\neq 1$, we determine that $\dim(V_{s}(\mu))\leq \binom{\ell+2}{3}+\ell$ for all eigenvalues $\mu$ of $s$ on $V$, where equality holds for $\mu=c^{-3}$ and $c^{\ell+1}=-1$. Since we have found a pair $(s,\mu)\in T\setminus\ZG(G)\times k^{*}$ with $\codim(V_{s}(\mu))=\frac{\ell^{2}+\ell+2}{2}$ and since $\nu_{G}(V)\geq \frac{\ell^{2}+\ell+2}{2}$, it follows that $\displaystyle \max_{s\in T\setminus \ZG(G)}\dim(V_{s}(\mu))$ $=\binom{\ell+2}{3}+\ell$ and $\nu_{G}(V)=\frac{\ell^{2}+\ell+2}{2}$.
\end{proof}

\begin{prop}\label{PropositionAlom1+om2p=3}
Let $p=3$, $\ell\geq 3$ and let $V=L_{G}(\omega_{1}+\omega_{2})$. Then $\nu_{G}(V)=\binom{\ell+1}{2}$. Moreover, we have $\scale[0.9]{\displaystyle \max_{u\in G_{u}\setminus \{1\}}\dim(V_{u}(1))}=\binom{\ell+2}{3}-\ell+1$ and $\scale[0.9]{\displaystyle \max_{s\in T\setminus\ZG(G)}\dim(V_{s}(\mu))}=\binom{\ell+2}{3}$. 
\end{prop}

\begin{proof}
By \cite{LuTables}, the sub-dominant weights in $V$ are $\omega_{1}+\omega_{2}$ and $\omega_{3}$, therefore, $s_{\omega_{1}+\omega_{2}}=r_{\omega_{1}+\omega_{2}}+r_{\omega_{3}}=\binom{\ell+1}{2}$ and so $\nu_{G}(V)\geq \binom{\ell+1}{2}$. 

Let $\lambda=\omega_{1}+\omega_{2}$ and $L=L_{1}$. By Lemma \ref{weightlevelAl}, we have $e_{1}(\lambda)=2$, therefore $\displaystyle V\mid_{[L,L]}=V^{0}\oplus V^{1}\oplus V^{2}$. First, by \cite[Proposition]{Smith_82}, we have $V^{0}\cong L_{L}(\omega_{2})$. Secondly, the weight $\displaystyle (\lambda-\alpha_{1})\mid_{T_{1}}=2\omega_{2}$ admits a maximal vector in $V^{1}$, thus $V^{1}$ has a composition factor isomorphic to $L_{L}(2\omega_{2})$. Lastly, as the weight $\displaystyle (\lambda-2\alpha_{1}-\alpha_{2})\mid_{T_{1}}=\omega_{2}+\omega_{3}$ admits a maximal vector in $V^{2}$, it follows that $V^{2}$ has a composition factor isomorphic to $L_{L}(\omega_{2}+\omega_{3})$. By dimensional considerations, we determine that:
\begin{equation}\label{DecompVLAlom1+om2p=3}
V\mid_{[L,L]}\cong L_{L}(\omega_{2})\oplus L_{L}(2\omega_{2})\oplus L_{L}(\omega_{2}+\omega_{3}).
\end{equation}

We start with the semisimple elements. Let $s\in T\setminus \ZG(G)$. If $\dim(V^{i}_{s}(\mu))=\dim(V^{i})$ for some eigenvalue $\mu$ of $s$ on $V$, where $0\leq i\leq 2$, then $s\in \ZG(L)^{\circ}\setminus \ZG(G)$ and acts on $V^{i}$ as scalar multiplication by $c^{2\ell-1-i(\ell+1)}$. As $c^{\ell+1}\neq 1$, we determine that $\dim(V_{s}(\mu))\leq \binom{\ell+2}{3}$ for all eigenvalues $\mu$ of $s$ on $V$, where equality holds for $\mu=c^{-3}$ and $c^{\ell+1}=-1$. Since we have found a pair $(s,\mu)\in T\setminus\ZG(G)\times k^{*}$ with $\codim(V_{s}(\mu))=\binom{\ell+1}{2}$ and since $\nu_{G}(V)\geq \binom{\ell+1}{2}$, it follows that $\scale[0.9]{\displaystyle \max_{s\in T\setminus \ZG(G)}\dim(V_{s}(\mu))}=\binom{\ell+2}{3}$ and $\nu_{G}(V)=\binom{\ell+1}{2}$.

Lastly, we calculate $\displaystyle \max_{u\in G_{u}\setminus \{1\}}\dim(V_{u}(1))$. By Lemma \ref{uniprootelems}, we have $\displaystyle \max_{u\in G_{u}\setminus \{1\}}\dim(V_{u}(1))=\dim(V_{x_{\alpha_{\ell}}(1)}(1))$ and, using \eqref{DecompVLAlom1+om2p=3} and Propositions \ref{PropositionAlnatural} and \ref{PropositionAlsymmetric}, we deduce that $\dim(V_{x_{\alpha_{\ell}}(1)}(1))=\ell-1+\binom{\ell}{2}+\dim((L_{L}(\omega_{2}+\omega_{3}))_{x_{\alpha_{\ell}}(1)}(1))$. Recursively and using Proposition \ref{PropositionAlom1+oml} for the base case of $\ell=3$, we deduce that $ \dim(V_{x_{\alpha_{\ell}}(1)}(1))$ $\scale[0.9]{\displaystyle=3+\sum_{j=2}^{\ell-1}}$ $\binom{j+1}{2}+$ $\displaystyle \sum_{j=2}^{\ell-1}j$ $=\binom{\ell+2}{3}-\ell+1$, therefore $\scale[0.9]{\displaystyle \max_{u\in G_{u}\setminus \{1\}}\dim(V_{u}(1))}$ $=\binom{\ell+2}{3}-\ell+1$.
\end{proof}

\begin{prop}\label{PropositionAlom1+om2pneq3}
Let $p\neq 3$,  $\ell\geq 3$ and $V=L_{G}(\omega_{1}+\omega_{2})$. Then $\nu_{G}(V)=\ell^{2}$. Moreover, we have $\scale[0.9]{\displaystyle \max_{u\in G_{u}\setminus \{1\}}\dim(V_{u}(1))}$ $=\frac{\ell^{3}+2\ell}{3}$ and $\scale[0.9]{\displaystyle \max_{s\in T\setminus\ZG(G)}\dim(V_{s}(\mu))}$ $=\frac{\ell^{3}+2\ell}{3}$ $-\ell\varepsilon_{p}(2)$.
\end{prop}

\begin{proof}
Let $\lambda=\omega_{1}+\omega_{2}$ and $L=L_{1}$. By Lemma \ref{weightlevelAl}, we have $e_{1}(\lambda)=2$,therefore $\displaystyle V\mid_{[L,L]}=V^{0}\oplus V^{1}\oplus V^{2}$. First, by \cite[Proposition]{Smith_82}, we have $V^{0}\cong L_{L}(\omega_{2})$. Secondly, since the weight $\displaystyle (\lambda-\alpha_{1})\mid_{T_{1}}=2\omega_{2}$ admits a maximal vector in $V^{1}$, it follows that $V^{1}$ has a composition factor isomorphic to $L_{L}(2\omega_{2})$. Moreover, the weight $\displaystyle (\lambda-\alpha_{1}-\alpha_{2})\mid_{T_{1}}=\omega_{3}$ occurs with multiplicity $2$ in $V^{1}$ and is a sub-dominant weight in the composition factor of $V^{1}$ isomorphic to $L_{L}(2\omega_{2})$, in which it has multiplicity $1$, if and only if $p\neq 2$. Lastly, the weight $\displaystyle (\lambda-2\alpha_{1}-\alpha_{2})\mid_{T_{1}}=\omega_{2}+\omega_{3}$ admits a maximal vector in $V^{2}$, thus $V^{2}$ has a composition factor isomorphic to $L_{L}(\omega_{2}+\omega_{3})$. By dimensional considerations, we determine that $V^{2}\cong L_{L}(\omega_{2}+\omega_{3})$ and $V^{1}$ has exactly $2+\varepsilon_{p}(2)$ composition factors: one isomorphic to $L_{L}(2\omega_{2})$ and $1+\varepsilon_{p}(2)$ isomorphic to $L_{L}(\omega_{3})$.

We start with the semisimple elements. Let $s\in T\setminus \ZG(G)$. If $\dim(V^{i}_{s}(\mu))=\dim(V^{i})$ for some eigenvalue $\mu$ of $s$ on $V$, where $0\leq i\leq 2$, then $s\in \ZG(L)^{\circ}\setminus \ZG(G)$ and acts on $V^{i}$ as scalar multiplication by $c^{2\ell-1-i(\ell+1)}$. As $c^{\ell+1}\neq 1$, we determine that $\dim(V_{s}(\mu))\leq \frac{\ell^{3}+2\ell}{3}-\ell\varepsilon_{p}(2)$, where equality holds for $p\neq 2$, $\mu=c^{-3}$ and $c^{\ell+1}=-1$, respectively for $p=2$ and $\mu=c^{-3}$. We thus assume that $\dim(V^{i}_{s}(\mu))<\dim(V^{i})$ for all eigenvalues $\mu$ of $s$ on $V$ and all $0\leq i\leq 2$. We write $s=z\cdot h$, where $z\in \ZG(L)^{\circ}$ and $h\in [L,L]$. First, let $\ell=3$. Assume that, up to conjugation, $h=h_{\alpha_{2}}(d)h_{\alpha_{3}}(d^{2})$, where $d^{3}\neq 1$. One shows that, in this case, the eigenvalues of $s$ on $V$ are $c^{5}d$ with multiplicity at least $2$; $c^{5}d^{-2}$ with multiplicity at least $1$; $cd^{2}$ with multiplicity at least $4$; $cd^{-1}$ with multiplicity at least $4$; $cd^{-4}$ with multiplicity at least $1$; $c^{3}$ with multiplicity at least $4$; $c^{3}d$ with multiplicity at least $1$; and $c^{3}d^{-1}$ with multiplicity at least $1$, where $c\in k^{*}$. Thus, as $d^{3}\neq 1$, one shows that $\dim(V_{s}(\mu))\leq 11-3\varepsilon_{p}(2)$ for all eigenvalues $\mu$ of $s$ on $V$. On the other hand, if $h$ is not conjugate to $h=h_{\alpha_{2}}(d)h_{\alpha_{3}}(d^{2})$ with $d^{3}\neq 1$, then, by the structure of $V\mid_{[L,L]}$ and Propositions \ref{PropositionAlnatural}, \ref{PropositionAlsymmetric} and \ref{PropositionAlom1+oml}, it follows that $\dim(V_{s}(\mu))\leq (2+\varepsilon_{p}(2))\dim((L_{L}(\omega_{2}))_{h}(\mu_{h}))+\dim((L_{L}(2\omega_{2}))_{h}(\mu_{h}))+\dim((L_{L}(\omega_{2}+\omega_{3}))_{h}(\mu_{h}))\leq 8-\varepsilon_{p}(2)$ for all eigenvalues $\mu$ of $s$ on $V$. Therefore $\scale[0.8]{\displaystyle \max_{s\in T\setminus \ZG(G)}\dim(V_{s}(\mu))}=11-3\varepsilon_{p}(2)$. We thus assume that $\ell\geq 4$, and, by the structure of $V\mid_{[L,L]}$ and Propositions \ref{PropositionAlnatural}, \ref{PropositionAlwedge} and \ref{PropositionAlsymmetric}, we have $\dim(V_{s}(\mu))\leq \ell-1+\bigg[\frac{(\ell-1)^{2}+\ell-1+2}{2}-\frac{(\ell-1)^{2}-(\ell-1)+2}{2}\varepsilon_{p}(2)\bigg]+(1+\varepsilon_{p}(2))\binom{\ell-1}{2}+\dim((L_{L}(\omega_{2}+\omega_{3}))_{h}(\mu_{h}))$. Recursively and using the result for $\ell=3$, we determine that $\scale[0.8]{\displaystyle \dim(V_{s}(\mu))\leq 11-3\varepsilon_{p}(2)+\sum_{j=3}^{\ell-1}[j^{2}+j+1]-}$ $\scale[0.8]{\displaystyle\sum_{j=3}^{\ell-1}\varepsilon_{p}(2)=\frac{\ell^{3}+2\ell}{3}-\ell\varepsilon_{p}(2)}$. Therefore, $\scale[0.8]{\displaystyle \max_{s\in T\setminus \ZG(G)}\dim(V_{s}(\mu))}$ $=\frac{\ell^{3}+2\ell}{3}-\ell\varepsilon_{p}(2)$.

We now focus on the unipotent elements. By \cite[Proposition $4.6.10$]{mcninch_1998}, we have that $L_{G}(\omega_{1})\otimes L_{G}(\omega_{2})\cong V\oplus L_{G}(\omega_{3})$. Thus, in view of Lemma \ref{uniprootelems}, we have $\displaystyle \max_{u\in G_{u}\setminus \{1\}}\dim(V_{u}(1))=\dim((L_{G}(\omega_{1})\otimes L_{G}(\omega_{2}))_{x_{\alpha_{\ell}}(1)}(1))-\dim((L_{G}(\omega_{3}))_{x_{\alpha_{\ell}}(1)}(1))$. Now, as $x_{\alpha_{\ell}}(1)$ acts on $L_{G}(\omega_{1})$ as $J_{2}\oplus J_{1}^{\ell-1}$ and on $L_{G}(\omega_{2})$ as $J_{2}^{\ell-1}\oplus J_{1}^{\frac{\ell^{2}-3\ell+4}{2}}$, see proof of Proposition \ref{PropositionAlwedge}, using \cite[Lemma $3.4$]{liebeck_2012unipotent}, one shows that $\dim((L_{G}(\omega_{1})\otimes L_{G}(\omega_{2}))_{x_{\alpha_{\ell}}(1)}(1))=\frac{\ell^{3}-\ell^{2}+4\ell-2}{2}$. Further, by Propositions \ref{PropositionAlnatural} for $\ell=3$, \ref{PropositionAlwedge} for $\ell=4$ and \ref{PropositionAlwedgecube} for $\ell\geq 5$, we determine that $\scale[0.85]{\displaystyle \max_{u\in G_{u}\setminus \{1\}}\dim(V_{u}(1))}$ $=\frac{\ell^{3}+2\ell}{3}$. Therefore, we have $\nu_{G}(V)=\ell^{2}$.
\end{proof}

\begin{prop}\label{PropositionAlwedge4}
Let $\ell\geq 7$ and let $V=L_{G}(\omega_{4})$. Then $\nu_{G}(V)=\binom{\ell-1}{3}$. Moreover, we have $\scale[0.9]{\displaystyle \max_{u\in G_{u}\setminus \{1\}}\dim(V_{u}(1))}$ $=\binom{\ell-1}{4}+\binom{\ell-1}{3}+\binom{\ell-1}{2}$ and $\scale[0.9]{\displaystyle \max_{s\in T\setminus\ZG(G)}\dim(V_{s}(\mu))}$ $\leq \binom{\ell}{4}+2\ell-5$.
\end{prop}

\begin{proof}
As $\omega_{4}$ is a minuscule weight, we have $s_{\omega_{4}}=r_{\omega_{4}}=\binom{\ell-1}{3}$, see \eqref{Al_r_omega_i}, thus $\nu_{G}(V)\geq \binom{\ell-1}{3}$. Further, by \cite[Proposition $4.2.2$]{mcninch_1998}, we have $V\cong \wedge^{4}(W)$.

For the unipotent elements, by Lemma \ref{uniprootelems}, we have $\displaystyle \max_{u\in G_{u}\setminus \{1\}}\dim(V_{u}(1))=$ $\dim((\wedge^{4}(W))_{x_{\alpha_{1}}(1)}(1))$. We write $W=W_{1}\oplus W_{2}$, where $\dim(W_{1})=2$ and $x_{\alpha_{1}}(1)$ acts as $J_{2}$ on $W_{1}$, and $\dim(W_{2})=\ell-1$ and $x_{\alpha_{1}}(1)$ acts trivially on $W_{2}$. As $\wedge^{4}(W)\cong \wedge^{4}(W_{1})\oplus [\wedge^{3}(W_{1})\otimes W_{2}]\oplus [\wedge^{2}(W_{1})\otimes \wedge^{2}(W_{2})]\oplus [W_{1}\otimes \wedge^{3}(W_{2})]\oplus \wedge^{4}(W)$, it follows that $\dim((\wedge^{4}(W))_{x_{\alpha_{1}}(1)}(1)) = \dim((\wedge^{4}(W_{1}))_{x_{\alpha_{1}}(1)}(1))+\dim((\wedge^{3}(W_{1})\otimes W_{2})_{x_{\alpha_{1}}(1)}(1))+\dim((\wedge^{2}(W_{1})\otimes \wedge^{2}(W_{2}))_{x_{\alpha_{1}}(1)}(1))+\dim((W_{1}\otimes \wedge^{3}(W_{2}))_{x_{\alpha_{1}}(1)}(1))+\dim((\wedge^{4}(W_{2}))_{x_{\alpha_{1}}(1)}(1))$. Using \cite[Lemma $3.4$]{liebeck_2012unipotent} and Lemma \ref{Lemma on fixed space for wedge in p=2}, we deduce $\displaystyle \max_{u\in G_{u}\setminus \{1\}}\dim(V_{u}(1))$ $=\binom{\ell-1}{4}+\binom{\ell-1}{3}+\binom{\ell-1}{2}$.

We now focus on the semisimple elements. Let $\lambda=\omega_{4}$ and $L=L_{1}$. By Lemma \ref{weightlevelAl}, we have $e_{1}(\lambda)=1$, therefore $V\mid_{[L,L]}=V^{0}\oplus V^{1}$. Now, by \cite[Proposition]{Smith_82}, we have $V^{0}\cong L_{L}(\omega_{4})$. Further, the weight $\displaystyle (\lambda-\alpha_{1}-\cdots-\alpha_{4})\mid_{T_{1}}=\omega_{5}$ admits a maximal vector in $V^{1}$, thus $V^{1}$ has a composition factor isomorphic to $L_{L}(\omega_{5})$. By dimensional considerations, we determine that
\begin{equation}\label{DecompVAlom4}
V\mid_{[L,L]}\cong L_{L}(\omega_{4})\oplus L_{L}(\omega_{5}).
\end{equation}

Let $s\in T\setminus \ZG(G)$. If $\dim(V^{i}_{s}(\mu))=\dim(V^{i})$ for some eigenvalue $\mu$ of $s$ on $V^{i}$, where $i=0$ or $i=1$, then $s\in \ZG(L)^{\circ}\setminus \ZG(G)$. In this case $s$ acts as scalar multiplication by $c^{\ell-3-i(\ell+1)}$ on $V^{i}$ and, as $c^{\ell+1}\neq 1$, we determine that $\dim(V_{s}(\mu))\leq \binom{\ell}{4}$ for all eigenvalues $\mu$ of $s$ on $V$. We thus assume that $\dim(V^{i}_{s}(\mu))<\dim(V^{i})$ for all eigenvalues $\mu$ of $s$ on $V$ and for both $i=0,1$. We write $s=z\cdot h$, where $z\in \ZG(L)^{\circ}$ and $h\in [L,L]$. We first consider the case of $\ell=7$. Then by \eqref{DecompVAlom4} and Proposition \ref{PropositionAlwedgecube}, we determine that $\dim(V_{s}(\mu))\leq 44$ for all eigenvalues $\mu$ of $s$ on $V$. We can now assume that $\ell\geq 8$, and, by \eqref{DecompVAlom4} and Proposition \ref{PropositionAlwedgecube}, we determine that $\dim(V_{s}(\mu))\leq \binom{\ell-1}{3}+2+\dim((L_{L}(\omega_{5}))_{h}(\mu_{h}))$. Recursively and using the result for $\ell=7$, we determine that $\scale[0.9]{\displaystyle \dim(V_{s}(\mu))\leq \sum_{j=7}^{\ell-1}\binom{j}{3}+\sum_{j=7}^{\ell-1}2+44=\frac{1}{6}\bigg[\sum_{j=7}^{\ell-1}j^{3}-3\sum_{j=7}^{\ell-1}j^{2}+2\sum_{j=7}^{\ell-1}j\bigg]+2\ell+30=\binom{\ell}{4}+2\ell-5}$. 
In conclusion, we have shown that $\scale[0.9]{\displaystyle \max_{s\in T\setminus\ZG(G)}\dim(V_{s}(\mu))}$ $\leq \binom{\ell}{4}+2\ell-5$ and, as $\scale[0.9]{\displaystyle \max_{u\in G_{u}\setminus \{1\}}\dim(V_{u}(1))}$ $=\binom{\ell-1}{4}+ \binom{\ell-1}{3}+\binom{\ell-1}{2}$, we determine that $\scale[0.9]{\nu_{G}(V)=\binom{\ell-1}{3}}$.
\end{proof}

\begin{prop}\label{PropositionAlwedge5}
Let $\ell\geq 9$ and let $V=L_{G}(\omega_{5})$. Then $\nu_{G}(V)=\binom{\ell-1}{4}$. Moreover, we have $\scale[0.9]{\displaystyle \max_{u\in G_{u}\setminus \{1\}}\dim(V_{u}(1))}$ $=\binom{\ell-1}{5}+\binom{\ell-1}{4}+\binom{\ell-1}{3}$ and $\scale[0.9]{\displaystyle \max_{s\in T\setminus\ZG(G)}\dim(V_{s}(\mu))}$ $\leq \binom{\ell}{5}+\ell^{2}-6\ell+9$.
\end{prop}

\begin{proof}
As $\omega_{5}$ is a minuscule weight, we have $s_{\omega_{5}}=r_{\omega_{5}}=\binom{\ell-1}{4}$, see \eqref{Al_r_omega_i}, thus $\nu_{G}(V)\geq \binom{\ell-1}{4}$. Further, by \cite[Proposition $4.2.2$]{mcninch_1998}, we have $V\cong \wedge^{5}(W)$. 

For the unipotent elements, by Lemma \ref{uniprootelems}, we have $\displaystyle \max_{u\in G_{u}\setminus \{1\}}\dim(V_{u}(1))=\dim((\wedge^{5}(W))_{x_{\alpha_{1}}(1)}(1))$. We write $W=W_{1}\oplus W_{2}$, where $\dim(W_{1})=2$ and $x_{\alpha_{1}}(1)$ acts as $J_{2}$ on $W_{1}$, and $\dim(W_{2})=\ell-1$ and $x_{\alpha_{1}}(1)$ acts trivially on $W_{2}$. As $\wedge^{5}(W)\cong \wedge^{5}(W_{1})\oplus [\wedge^{4}(W_{1})\otimes W_{2}]\oplus [\wedge^{3}(W_{1})\otimes \wedge^{2}(W_{2})]\oplus [\wedge^{2}(W_{1})\otimes \wedge^{3}(W_{2})]\oplus [W_{1}\otimes \wedge^{4}(W)]\oplus \wedge^{5}(W_{2})$, it follows that $\dim((\wedge^{5}(W))_{x_{\alpha_{1}}(1)}(1)) = \binom{\ell-1}{5}+\binom{\ell-1}{4}+\binom{\ell-1}{3}$, by \cite[Lemma $3.4$]{liebeck_2012unipotent} and Lemma \ref{Lemma on fixed space for wedge in p=2}.

We now focus on the semisimple elements. Let $\lambda=\omega_{5}$ and $L=L_{1}$. By Lemma \ref{weightlevelAl}, we have $e_{1}(\lambda)=1$, therefore $V\mid_{[L,L]}=V^{0}\oplus V^{1}$. Now, by \cite[Proposition]{Smith_82}, we have $V^{0}\cong L_{L}(\omega_{5})$. Further, the weight $\displaystyle (\lambda-\alpha_{1}-\cdots-\alpha_{5})\mid_{T_{1}}=\omega_{6}$ admits a maximal vector in $V^{1}$, thus $V^{1}$ has a composition factor isomorphic to $L_{L}(\omega_{6})$. By dimensional considerations, we determine that
\begin{equation}\label{DecompVAlom5}
V\mid_{[L,L]}\cong L_{L}(\omega_{5})\oplus L_{L}(\omega_{6}).
\end{equation}

Let $s\in T\setminus \ZG(G)$. If $\dim(V^{i}_{s}(\mu))=\dim(V^{i})$ for some eigenvalue $\mu$ of $s$ on $V^{i}$, where $i=0$ or $i=1$, then $s\in \ZG(L)^{\circ}\setminus \ZG(G)$. In this case $s$ acts as scalar multiplication by $c^{\ell-4-i(\ell+1)}$ on $V^{i}$ and, as $c^{\ell+1}\neq 1$, we determine that $\dim(V_{s}(\mu))\leq \binom{\ell}{5}$ for all eigenvalues $\mu$ of $s$ on $V$. We thus assume that $\dim(V^{i}_{s}(\mu))<\dim(V^{i})$ for all eigenvalues $\mu$ of $s$ on $V$ and for both $i=0,1$. We write $s=z\cdot h$, where $z\in \ZG(L)^{\circ}$ and $h\in [L,L]$. We first let $\ell=9$. Then by \eqref{DecompVAlom5} and Proposition \ref{PropositionAlwedge4}, we determine that $\dim(V_{s}(\mu))\leq 162$ for all eigenvalues $\mu$ of $s$ on $V$. We now let $\ell\geq 10$, and, by \eqref{DecompVAlom5} and Proposition \ref{PropositionAlwedge4}, we have $\dim(V_{s}(\mu))\leq \binom{\ell-1}{4}+2(\ell-1)-5+\dim((L_{L}(\omega_{6}))_{h}(\mu_{h}))$. Recursively and using the result for $\ell=9$, we determine that $\scale[0.84]{\displaystyle \dim(V_{s}(\mu)) \leq\sum_{j=9}^{\ell-1}\binom{j}{4}+2\sum_{j=9}^{\ell-1}j+\sum_{j=9}^{\ell-1}5+162=\frac{1}{24}\bigg[\sum_{j=9}^{\ell-1}j^{4}-6\sum_{j=9}^{\ell-1}j^{3}+11\sum_{j=9}^{\ell-1}j^{2}-6\sum_{j=9}^{\ell-1}j\bigg]+\ell^{2}-6\ell+135=\binom{\ell}{5}+\ell^{2}-6\ell+9}$.
In conclusion, we have shown that $\scale[0.9]{\displaystyle \max_{s\in T\setminus\ZG(G)}\dim(V_{s}(\mu))}$ $\leq \binom{\ell}{5}+\ell^{2}-6\ell+9$ and, as $\scale[0.9]{\displaystyle \max_{u\in G_{u}\setminus \{1\}}\dim(V_{u}(1))}$ $=\binom{\ell-1}{5}+ \binom{\ell-1}{4}+\binom{\ell-1}{3}$, we determine that $\scale[0.9]{\nu_{G}(V)=\binom{\ell-1}{4}}$.
\end{proof}

\subsection{Supplementary results}
At this point, we have completed the proof of Theorem \ref{ResultsAl}. The results of this subsection will be required in the proofs of Theorems \ref{ResultsCl}, \ref{ResultsBl} and \ref{ResultsDl}. 

\begin{prop}\label{A1mom1}
Let $p=0$, or $p>m$, $\ell=1$ and $V=L_{G}(m\omega_{1})$, where $m\geq 4$. Then, we have $\scale[0.9]{\displaystyle \max_{s\in T\setminus\ZG(G)}\dim(V_{s}(\mu))}$ $=1+\left\lfloor \frac{\dim(V)-1}{2}\right\rfloor$ and $\scale[0.9]{\displaystyle \max_{u\in G_{u}\setminus \{1\}}\dim(V_{u}(1))}=1$. 
\end{prop}

\begin{proof}
The result for the unipotent elements follows from \cite[Theorem $1.9$]{suprunenko1983preservation}. Thus, let $s\in T\setminus \ZG(G)$ and let $\mu$ be an eigenvalue of $s$ on $V$. We note that the eigenvalues of $s$ on $V$, not necessarily distinct, are $\mu_{1}^{m},\mu_{1}^{m-2},\dots,\mu_{1}^{-m+2},$ $\mu_{1}^{-m}$, where $\mu_{1}\neq \pm 1$ and, consequently, $\mu_{1}^{i}\neq \mu_{1}^{i-2}$ for all $-m+2\leq i\leq m$. 

We first assume that $m$ is even. We remark that, in this case, $1$ occurs as an eigenvalue of $s$ on $V$. If $\mu\neq \mu^{-1}$, then, as $V$ is self-dual, we determine that $\dim(V_{s}(\mu))\leq \frac{\dim(V)-\dim(V_{s}(1))}{2}\leq \frac{m}{2}$. We now let $\mu=1$.  If $\mu_{1}^{i}=1$ for some $2\leq i\leq m$, then, as $\mu_{1}^{2}\neq 1$ and $\mu_{1}^{i}\neq \mu_{1}^{i-2}$, at most $\left\lfloor\frac{m-2}{4}\right\rfloor+\varepsilon$ of the eigenvalues $\mu_{1}^{m},\mu_{1}^{m-2},\dots,\mu_{1}^{4}$ can equal $1$, where $\varepsilon=1$ if $\varepsilon_{4}(m-2)=0$ and $\varepsilon=0$ if $\varepsilon_{4}(m-2)=1$. We deduce that $\dim(V_{s}(1))\leq 1+2(\left\lfloor\frac{m-2}{4}\right\rfloor+\varepsilon)=\varepsilon+\frac{m}{2}\leq 1+\left\lfloor\frac{m}{2}\right\rfloor$. Lastly, let $\mu=-1$. Then, as $\mu_{1}^{i}\neq \mu_{1}^{i-2}$, at most $\left\lfloor\frac{m}{4}\right\rfloor+\xi$ of the eigenvalues $\mu_{1}^{m},\mu_{1}^{m-2},\dots,\mu_{1}^{2}$ can equal $-1$, where $\xi=1$ if $\varepsilon_{4}(m)=0$ and $\xi=0$ if $\varepsilon_{4}(m)=1$. We deduce that $\dim(V_{s}(-1))\leq 2(\left\lfloor\frac{m}{4}\right\rfloor+\xi)=\xi+\frac{m}{2}\leq 1+ \left\lfloor\frac{m}{2}\right\rfloor=1+\left\lfloor \frac{\dim(V)-1}{2}\right\rfloor$.

We now assume that $m$ is odd. If $\mu\neq \mu^{-1}$, then, arguing as in the case of $m$ even, we deduce that $\dim(V_{s}(\mu))\leq \frac{m+1}{2}$. We now let $\mu=\pm 1$. If $\mu_{1}^{i}=\pm 1$ for some $1\leq i\leq m$, then, since $\mu_{1}\neq \pm 1$ and $\mu_{1}^{i}\neq \mu_{1}^{i-2}$, at most $\left\lfloor\frac{m-1}{4}\right\rfloor+\zeta$ of the eigenvalues $\mu_{1}^{m},\mu_{1}^{m-2}, \dots, \mu_{1}$ can equal $\pm 1$, where $\zeta=1$ if $\varepsilon_{4}(m-1)=0$ and $\zeta=0$ if $\varepsilon_{4}(m-1)=1$. We deduce that $\dim(V_{s}(\pm 1))\leq 2(\left\lfloor\frac{m-1}{4}\right\rfloor+\zeta)=\zeta+\frac{m-1}{2}\leq 1+\left\lfloor\frac{m}{2}\right\rfloor=1+\left\lfloor \frac{\dim(V)-1}{2}\right\rfloor$. 

Having considered all cases, we conclude that $\dim(V_{s}(\mu))\leq 1+\left\lfloor \frac{m}{2}\right\rfloor$ for all $(s,\mu)\in T\setminus \ZG(G)\times k^{*}$. We will now show that equality holds. For this, let $s=\diag(\mu_{1},\mu_{1}^{-1})\in T\setminus \ZG(G)$, where $\mu_{1}^{2}=-1$. Now, the eigenvalues of $s$ on $V$ are $\mu_{1}^{m},\mu_{1}^{m-2},\dots,\mu_{1}^{-m+2},\mu_{1}^{-m}$, and, as $\mu_{1}^{2}=-1$, it follows that $\dim(V_{s}(\mu_{1}^{m}))=1+\left\lfloor \frac{m}{2}\right\rfloor$. Therefore, $\scale[0.9]{\displaystyle \max_{s\in T\setminus\ZG(G)}\dim(V_{s}(\mu))}$ $=1+\left\lfloor \frac{\dim(V)-1}{2}\right\rfloor$.
\end{proof}

\begin{prop}\label{A22om1+om2}
Let $\ell=2$ and $V=L_{G}(2\omega_{1}+\omega_{2})$. Then, we have $\scale[0.9]{\displaystyle \max_{s\in T\setminus\ZG(G)}\dim(V_{s}(\mu))}=8-4\varepsilon_{p}(2)$ and $\scale[0.9]{\displaystyle \max_{u\in G_{u}\setminus \{1\}}\dim(V_{u}(1))}$ $\leq 6+\varepsilon_{p}(3)-\varepsilon_{p}(2)$, where equality holds for $p\neq 3$.
\end{prop}

\begin{proof}
Let $\lambda=2\omega_{1}+\omega_{2}$ and let $L=L_{1}$. By Lemma \ref{weightlevelAl}, we have $e_{1}(\lambda)=3$, therefore $\displaystyle V\mid_{[L,L]}=V^{0}\oplus \cdots \oplus V^{3}$. By \cite[Proposition]{Smith_82}, we have $V^{0}\cong L_{L}(\omega_{2})$. When $p\neq 2$, the weight $\displaystyle(\lambda-\alpha_{1})\mid_{T_{1}}=2\omega_{2}$ admits a maximal vector in $V^{1}$, thus $V^{1}$ has a composition factor isomorphic to $L_{L}(2\omega_{2})$. Moreover, the weight $\displaystyle(\lambda-\alpha_{1}-\alpha_{2})\mid_{T_{1}}=0$ occurs with multiplicity $2$ and is a sub-dominant weight in the composition factor of $V^{1}$ isomorphic to $L_{L}(2\omega_{2})$, in which it has multiplicity $1$. On the other hand, when $p=2$, the weight $\displaystyle(\lambda-\alpha_{1}-\alpha_{2})\mid_{T_{1}}=0$ admits a maximal vector in $V^{1}$, thus $V^{1}\cong L_{L})(0)$. We now let $p$ be arbitrary. In $V^{2}$ the weight $\displaystyle(\lambda-2\alpha_{1})\mid_{T_{1}}=3\omega_{2}$ admits a maximal vector, thus $V^{2}$ has a composition factor isomorphic to $L_{L}(3\omega_{2})$. Further, the weight $\displaystyle(\lambda- 2\alpha_{1}-\alpha_{2})\mid_{T_{1}}=\omega_{2}$ occurs with multiplicity $2-\varepsilon_{p}(2)$ and is a sub-dominant weight in the composition factor of $V^{2}$ isomorphic to $L_{L}(3\omega_{2})$, in which it has multiplicity $1-\varepsilon_{p}(3)$. Lastly, in $V^{3}$ the weight $\displaystyle(\lambda-3\alpha_{1}-\alpha_{2})\mid_{T_{1}}=2\omega_{2}$ admits a maximal vector, thus $V^{3}$ has a composition factor isomorphic to $L_{L}(2\omega_{2})$. By dimensional considerations and using \cite[II.2.14]{Jantzen_2007representations}, we determine that $V^{3}\cong L_{L}(2\omega_{2})$, $V^{1}\cong L_{L}(2\omega_{2})^{1-\varepsilon_{p}(2)}\oplus L_{L}(0)$ and $V^{2}$ has exactly $2+\varepsilon_{p}(3)$ composition factors: one isomorphic to $L_{L}(3\omega_{2})$ and $1+\varepsilon_{p}(3)$ isomorphic to $L_{L}(\omega_{2})$.

For the semisimple elements, let $s\in T\setminus \ZG(G)$. If $\dim(V^{i}_{s}(\mu))=\dim(V^{i})$ for some eigenvalue $\mu$ of $s$ on $V$, where $0\leq i\leq 3$, then $s\in \ZG(L)^{\circ}\setminus \ZG(G)$. In this case, as $s$ acts on each $V^{i}$ as scalar multiplication by $c^{5-3i}$ and $c^{3}\neq 1$, we determine that $\dim(V_{s}(\mu))\leq 8-4\varepsilon_{p}(2)$ for all eigenvalues $\mu$ of $s$ on $V$. Moreover, equality holds for $p\neq 2$, $c^{3}=-1$ and $\mu=-c^{2}$, respectively for $p=2$ and $\mu=c^{-1}$. We thus assume that $\dim(V^{i}_{s}(\mu))<\dim(V^{i})$ for all eigenvalues $\mu$ of $s$ on $V$ and all $0\leq i\leq 3$. We write $s=z\cdot h$, where $z\in \ZG(L)^{\circ}$ and $h\in [L,L]$, and, we have $\scale[0.9]{\displaystyle \dim(V_{s}(\mu))\leq \sum_{i=0}^{3}\dim(V^{i}_{h}(\mu_{h}^{i}))}$, where $\dim(V^{i}_{h}(\mu_{h}^{i}))<\dim(V^{i})$ for all eigenvalues $\mu_{h}^{i}$ of $h$ on $V^{i}$. We first treat the case of $p=2$. One shows that the eigenvalues of $s$ on $V$ are $c^{5}a^{\pm 1}$, $c^{2}$, $c^{-1}a^{\pm 3}$, $c^{-1}a^{\pm 1}$ and $c^{-4}a^{\pm 2}$, where $a\neq 1$. It follows that $\dim(V_{s}(\mu))\leq 4$ for all eigenvalues $\mu$ of $s$ on $V$. We thus assume $p\neq 2$, and, as $V^{1}\cong L_{L}(2\omega_{2})\oplus L_{L}(0)$, using \eqref{Al_enum_P1}, we determine that $\dim(V^{1}_{h}(\mu_{h}^{1}))\leq 2$ for all eigenvalues $\mu_{h}^{1}$ of $h$ on $V^{2}$. Thus, by the structure of $V\mid_{[L,L]}$ and Propositions \ref{PropositionAlnatural},  \ref{PropositionAlsymmetric} and \ref{PropositionAlsymmcube}, we determine that $\dim(V_{s}(\mu))\leq 8$ for all eigenvalues $\mu$ of $s$ on $V$. This shows that $\displaystyle \max_{s\in T\setminus\ZG(G)}\dim(V_{s}(\mu))=8-4\varepsilon_{p}(2)$. 

We now consider the case of the unipotent elements. By Lemma \ref{uniprootelems}, we have $\scale[0.9]{\displaystyle \max_{u\in G_{u}\setminus \{1\}}\dim(V_{u}(1))}$ $\scale[0.9]{\displaystyle=\dim(V_{x_{\alpha_{2}}(1)}(1))}$. If $p=2$, then, as $V\cong L_{G}(\omega_{1})^{(2)}\otimes L_{G}(\omega_{2})$, and, as $x_{\alpha_{2}}(1)$ acts as $J_{2}\oplus J_{1}$ on both $L_{G}(\omega_{1})^{(2)}$ and $L_{G}(\omega_{2})$, we determine that $\dim(V_{x_{\alpha_{2}}(1)}(1))=5$. On the other hand, if $p\neq 2$, then, using the structure of $V\mid_{[L,L]}$ and Propositions \ref{PropositionAlnatural},  \ref{PropositionAlsymmetric} and \ref{PropositionAlsymmcube}, we deduce that $\dim(V_{x_{\alpha_{2}}(1)}(1))\leq 6+\varepsilon_{p}(3)$, where equality holds for $p\neq 3$.  Therefore, $\scale[0.9]{\displaystyle \max_{u\in G_{u}\setminus \{1\}}\dim(V_{u}(1))\leq 6+\varepsilon_{p}(3)-\varepsilon_{p}(2)}$, where equality holds for $p\neq 3$.
\end{proof}

\begin{prop}\label{A22om1+2om2pneq2}
Let $p\neq 2$, $\ell=2$ and $V=L_{G}(2\omega_{1}+2\omega_{2})$. Then, we have $\scale[0.9]{\displaystyle \max_{s\in T\setminus\ZG(G)}\dim(V_{s}(\mu))}=15-4\varepsilon_{p}(5)$ and $\scale[0.9]{\displaystyle \max_{u\in G_{u}\setminus \{1\}}\dim(V_{u}(1))}\leq 9+4\varepsilon_{p}(3)-3\varepsilon_{p}(5)$, where equality holds for $p\neq 3$.
\end{prop}

\begin{proof}
Let $\lambda=2\omega_{1}+2\omega_{2}$ and let $L=L_{1}$. By Lemma \ref{weightlevelAl}, we have $e_{1}(\lambda)=4$, therefore $\displaystyle V\mid_{[L,L]}=V^{0}\oplus \cdots \oplus V^{4}$. By \cite[Proposition]{Smith_82} and Lemma \ref{dualitylemma}, we have $V^{0}\cong L_{L}(2\omega_{2})$ and $V^{4}\cong L_{L}(2\omega_{2})$. Now, the weight $\displaystyle(\lambda-\alpha_{1})\mid_{T_{1}}=3\omega_{2}$ admits a maximal vector in $V^{1}$, thus $V^{1}$ has a composition factor isomorphic to $L_{L}(3\omega_{2})$. Moreover, the weight $\displaystyle(\lambda-\alpha_{1}-\alpha_{2})\mid_{T_{1}}=\omega_{2}$ occurs with multiplicity $2-\varepsilon_{p}(5)$ and is a sub-dominant weight of multiplicity $1-\varepsilon_{p}(3)$ in the composition factor of $V^{1}$ isomorphic to $L_{L}(3\omega_{2})$. Similarly, in $V^{2}$ the weight $\displaystyle(\lambda-2\alpha_{1})\mid_{T_{1}}=4\omega_{2}$ admits a maximal vector, thus $V^{2}$ has a composition factor isomorphic to $L_{L}(4\omega_{2})$. Further, the weight $\displaystyle(\lambda- 2\alpha_{1}-\alpha_{2})\mid_{T_{1}}=\omega_{2}$ occurs with multiplicity $2-\varepsilon_{p}(5)$ and is a sub-dominant weight in the composition factor of $V^{1}$ isomorphic to $L_{L}(4\omega_{2})$, in which it has multiplicity $1$. Lastly, we note that the weight $\displaystyle(\lambda-2\alpha_{1}-2\alpha_{2})\mid_{T_{1}}=0$ occurs with multiplicity $3-2\varepsilon_{p}(5)$ in $V^{2}$. Thus, by dimensional considerations, we determine that $V^{2}$ has exactly $3-2\varepsilon_{p}(5)+\varepsilon_{p}(3)$ composition factors: one isomorphic to $L_{L}(4\omega_{2})$, $1-\varepsilon_{p}(5)$ to $L_{L}(2\omega_{2})$ and $1-\varepsilon_{p}(5)+\varepsilon_{p}(3)$ to $L_{L}(0)$. Further, $V^{1}$ and $V^{3}$ each has $2-\varepsilon_{p}(5)+\varepsilon_{p}(3)$ composition factors: one isomorphic to $L_{L}(3\omega_{2})$ and $1-\varepsilon_{p}(5)+\varepsilon_{p}(3)$ to $L_{L}(\omega_{2})$.

For the semisimple elements, let $s\in T\setminus \ZG(G)$. If $\dim(V^{i}_{s}(\mu))=\dim(V^{i})$ for some eigenvalue $\mu$ of $s$ on $V$, where $0\leq i\leq 4$, then $s\in \ZG(L)^{\circ}\setminus \ZG(G)$. In this case, as $s$ acts on each $V^{i}$ as scalar multiplication by $c^{6-3i}$ and $c^{3}\neq 1$, we determine that $\dim(V_{s}(\mu))\leq 15-4\varepsilon_{p}(5)$, where equality holds for $c^{3}=-1$ and $\mu=1$. We thus assume that $\dim(V^{i}_{s}(\mu))<\dim(V^{i})$ for all eigenvalues $\mu$ of $s$ on $V$ and all $0\leq i\leq 4$. We write $s=z\cdot h$, where $z\in \ZG(L)^{\circ}$ and $h\in [L,L]$. We have $\displaystyle \dim(V_{s}(\mu))\leq \sum_{i=0}^{4}\dim(V^{i}_{h}(\mu_{h}^{i}))$, where $\dim(V^{i}_{h}(\mu_{h}^{i}))<\dim(V^{i})$ for all eigenvalues $\mu_{h}^{i}$ of $h$ on $V^{i}$. By the structure of $V^{2}$, we determine that the eigenvalues of $h$ on $V^{2}$ are $d^{\pm 4}$ each with multiplicity at least $1$; $d^{\pm 2}$ each with multiplicity at least $2-\varepsilon_{p}(5)$; and $1$ with multiplicity at least $3-2\varepsilon_{p}(5)$, where $d^{2}\neq 1$. Therefore, $\dim(V^{2}(\mu_{h}^{2}))\leq 5-2\varepsilon_{p}(5)$ for all eigenvalues $\mu_{h}^{2}$ of $h$ on $V^{2}$, and, using Propositions \ref{PropositionAlnatural}, \ref{PropositionAlsymmetric}, \ref{PropositionAlsymmcube} and \ref{A1mom1}, we determine that $\dim(V_{s}(\mu))\leq 15-4\varepsilon_{p}(5)$ for all eigenvalues $\mu$ of $s$ on $V$. This shows that $\scale[0.9]{\displaystyle \max_{s\in T\setminus\ZG(G)}\dim(V_{s}(\mu))}=15-4\varepsilon_{p}(5)$. For the unipotent elements, by Lemma \ref{uniprootelems}, the structure of $V\mid_{[L,L]}$, \cite[Lemma $3.4$]{liebeck_2012unipotent} and Propositions \ref{PropositionAlnatural},  \ref{PropositionAlsymmetric} and \ref{PropositionAlsymmcube}, we have $\scale[0.9]{\displaystyle \max_{u\in G_{u}\setminus \{1\}}\dim(V_{u}(1))=\dim(V_{x_{\alpha_{2}}(1)}(1))\leq 8+3\varepsilon_{p}(3)-3\varepsilon_{p}(5)+\dim((L_{L}(4\omega_{2}))_{x_{\alpha_{2}}(1)}(1))}$. Now, if $p>3$, then, by Proposition \ref{A1mom1}, we determine that $\scale[0.9]{\displaystyle \max_{u\in G_{u}\setminus \{1\}}\dim(V_{u}(1))}=9-3\varepsilon_{p}(5)$. On the other hand, if $p=3$, then, as $x_{\alpha_{2}}(1)$ acts on $L_{L}(\omega_{2})$ and on $L_{L}(\omega_{2})^{(3)}$ as $J_{2}$, we deduce that $x_{\alpha_{2}}(1)$ acts on $L_{L}(4\omega_{2})$ as $J_{2}^{2}$, thus $\scale[0.9]{\displaystyle \max_{u\in G_{u}\setminus \{1\}}\dim(V_{u}(1))}\leq 13$. We conclude that $\scale[0.9]{\displaystyle \max_{u\in G_{u}\setminus \{1\}}\dim(V_{u}(1))}\leq 9+4\varepsilon_{p}(3)-3\varepsilon_{p}(5)$.
\end{proof}

\begin{prop}\label{A3om1+om2+om3}
Let $\ell=3$ and $V=L_{G}(\omega_{1}+\omega_{2}+\omega_{3})$. Then, we have $\scale[0.9]{\displaystyle \max_{s\in T\setminus\ZG(G)}\dim(V_{s}(\mu))}\leq 36-4\varepsilon_{p}(5)-12\varepsilon_{p}(3)$ and $\scale[0.9]{\displaystyle \max_{u\in G_{u}\setminus \{1\}}\dim(V_{u}(1))}\leq 30-4\varepsilon_{p}(5)-10\varepsilon_{p}(3)+8\varepsilon_{p}(2)$.
\end{prop}

\begin{proof}
Let $\lambda=\omega_{1}+\omega_{2}+\omega_{3}$ and let $L=L_{1}$. By Lemma \ref{weightlevelAl}, we have $e_{1}(\lambda)=3$, therefore $\displaystyle V\mid_{[L,L]}=V^{0}\oplus \cdots \oplus V^{3}$. By \cite[Proposition]{Smith_82} and Lemma \ref{dualitylemma}, we have $V^{0}\cong L_{L}(\omega_{2}+\omega_{3})$ and $V^{3}\cong L_{L}(\omega_{2}+\omega_{3})$. Now, the weight $\displaystyle(\lambda-\alpha_{1})\mid_{T_{1}}=2\omega_{2}+\omega_{3}$ admits a maximal vector in $V^{1}$, thus $V^{1}$ has a composition factor isomorphic to $L_{L}(2\omega_{2}+\omega_{3})$. Moreover, the weight $\displaystyle (\lambda-\alpha_{1}-\alpha_{2})\mid_{T_{1}}=2\omega_{3}$ occurs with multiplicity $2-\varepsilon_{p}(3)$ and is a sub-dominant weight in the composition factor of $V^{1}$ isomorphic to $L_{L}(2\omega_{2}+\omega_{3})$, in which it has multiplicity $1-\varepsilon_{p}(2)$. Further, the weight $\displaystyle (\lambda-\alpha_{1}-\alpha_{2}-\alpha_{3})\mid_{T_{1}}=\omega_{2}$ occurs with multiplicity $4-\varepsilon_{p}(5)-2\varepsilon_{p}(3)$ and is a sub-dominant weight in the composition factor of $V^{1}$ isomorphic to $L_{L}(2\omega_{2}+\omega_{3})$, in which it has multiplicity $2-\varepsilon_{p}(2)$. As $\dim(V^{1})=24-3\varepsilon_{p}(5)-9\varepsilon_{p}(3)$, we determine that $V^{1}$ has exactly $\scale[0.9]{3-\varepsilon_{p}(5)-2\varepsilon_{p}(3)+3\varepsilon_{p}(2)}$ composition factors: one isomorphic to $L_{L}(2\omega_{2}+\omega_{3})$, $\scale[0.9]{1-\varepsilon_{p}(3)+\varepsilon_{p}(2)}$ to $L_{L}(2\omega_{3})$ and $\scale[0.9]{1-\varepsilon_{p}(5)-\varepsilon_{p}(3)+2\varepsilon_{p}(2)}$ to $L_{L}(\omega_{2})$. Lastly, as $V^{2}\cong (V^{1})^{*}$, $V^{2}$ has $\scale[0.9]{3-\varepsilon_{p}(5)-2\varepsilon_{p}(3)+3\varepsilon_{p}(2)}$ composition factors: one isomorphic to $L_{L}(\omega_{2}+2\omega_{3})$, $\scale[0.9]{1-\varepsilon_{p}(3)+\varepsilon_{p}(2)}$ to $L_{L}(2\omega_{2})$ and $\scale[0.9]{1-\varepsilon_{p}(5)-\varepsilon_{p}(3)+2\varepsilon_{p}(2)}$ to $L_{L}(\omega_{3})$. 

For the semisimple elements, let $s\in T\setminus \ZG(G)$. If $\dim(V^{i}_{s}(\mu))=\dim(V^{i})$ for some eigenvalue $\mu$ of $s$ on $V$, where $0\leq i\leq 3$, then $s\in \ZG(L)^{\circ}\setminus \ZG(G)$. In this case, as $s$ acts on each $V^{i}$ as scalar multiplication by $c^{6-4i}$ and $c^{4}\neq 1$, we determine that $\dim(V_{s}(\mu))\leq 32-3\varepsilon_{p}(5)-10\varepsilon_{p}(3)-8\varepsilon_{p}(2)$ for all eigenvalues $\mu$ of $s$ on $V$. We thus assume that $\dim(V^{i}_{s}(\mu))<\dim(V^{i})$ for all eigenvalues $\mu$ of $s$ on $V$ and all $0\leq i\leq 3$. We write $s=z\cdot h$, where $z\in \ZG(L)^{\circ}$ and $h\in [L,L]$. If $p\neq 2$, using the structure of $V\mid_{[L,L]}$ and Propositions \ref{PropositionAlnatural}, \ref{PropositionAlsymmetric}, \ref{PropositionAlom1+oml} and \ref{A22om1+om2}, we determine that $\dim(V_{s}(\mu))\leq 2\dim((L_{L}(\omega_{2}+\omega_{3}))_{h}(\mu_{h}))+2\dim((L_{L}(2\omega_{2}+\omega_{3}))_{h}(\mu_{h}))+(2-2\varepsilon_{p}(3))\dim((L_{L}(2\omega_{2}))_{h}(\mu_{h}))+(2-2\varepsilon_{p}(5)-2\varepsilon_{p}(3))\dim((L_{L}(\omega_{2}))_{h}(\mu_{h}))\leq 36-4\varepsilon_{p}(5)-12\varepsilon_{p}(3)$ for all eigenvalues $\mu$ of $s$ on $V$. We now assume that $p=2$. Let $d_{1}$, $d_{2}$, $d_{3}$ be the eigenvalues of $h$ on $L_{L}(\omega_{2})$, not all equal. Then, the eigenvalues of $h$ on $L_{L}(2\omega_{2}+\omega_{3})\cong L_{L}(\omega_{2})^{(2)}\otimes L_{L}(\omega_{3})$ are: $d_{1}, d_{1}^{2}d_{2}^{-1},d_{1}^{2}d_{3}^{-1}, d_{2}, d_{2}^{2}d_{1}^{-1},d_{2}^{2}d_{3}^{-1}, d_{3}, d_{3}^{2}d_{1}^{-1},d_{3}^{2}d_{2}^{-1}$, and one shows that $\dim((L_{L}(2\omega_{2}+\omega_{3}))_{h}(\mu_{h}))\leq 4$ for all eigenvalues $\mu_{h}$ of $h$. Thus, by the structure of $V\mid_{[L,L]}$ and Propositions \ref{PropositionAlnatural} and \ref{PropositionAlom1+oml}, we determine that $\dim(V_{s}(\mu))\leq 2\dim((L_{L}(\omega_{2}+\omega_{3}))_{h}(\mu_{h}))+2\dim((L_{L}(2\omega_{2}+\omega_{3}))_{h}(\mu_{h}))+10\dim((L_{L}(\omega_{2}))_{h}(\mu_{h}))\leq 36$ for all eigenvalues $\mu$ of $s$ on $V$. This shows that $\scale[0.9]{\displaystyle \max_{s\in T\setminus\ZG(G)}\dim(V_{s}(\mu))}$ $\leq 36-4\varepsilon_{p}(5)-12\varepsilon_{p}(3)$. We now focus on the unipotent elements. First, suppose that $p\neq 2$. Then, by Lemmas \ref{uniprootelems} and \ref{LemmaonfiltrationofV}, the structure of $V\mid_{[L,L]}$ and Propositions \ref{PropositionAlnatural}, \ref{PropositionAlsymmetric}, \ref{PropositionAlom1+oml} and \ref{A22om1+om2}, we have $\scale[0.9]{\displaystyle \max_{u\in G_{u}\setminus \{1\}}\dim(V_{u}(1))=\dim(V_{x_{\alpha_{3}}(1)}(1))\leq 30-4\varepsilon_{p}(5)-10\varepsilon_{p}(3)}$. Secondly,  let $p=2$. Once more, by Lemma \ref{uniprootelems}, we have $\scale[0.9]{\displaystyle \max_{u\in G_{u}\setminus \{1\}}\dim(V_{u}(1))}=\dim(V_{x_{\alpha_{3}}(1)}(1))$. As $x_{\alpha_{3}}(1)$ acts as $J_{2}\oplus J_{1}$ on $L_{L}(\omega_{2})$ and $L_{L}(\omega_{3})$, respectively, we determine that $x_{\alpha_{3}}(1)$ acts on $L_{L}(2\omega_{2}+\omega_{3})$ as $J_{2}^{4}\oplus J_{1}$. Thus, by the structure of $V\mid_{[L,L]}$ and Propositions \ref{PropositionAlnatural} and \ref{PropositionAlom1+oml}, we determine that $\dim(V_{x_{\alpha_{3}}(1)}(1)\leq 38$. Therefore, $\scale[0.9]{\displaystyle \max_{u\in G_{u}\setminus \{1\}}\dim(V_{u}(1))\leq 30-4\varepsilon_{p}(5)-10\varepsilon_{p}(3)+8\varepsilon_{p}(2)}$.
\end{proof}

\begin{prop}\label{A32om2}
Let $p\neq 2$, $\ell=3$ and $V=L_{G}(2\omega_{2})$. Then, we have $\scale[0.9]{\displaystyle \max_{s\in T\setminus\ZG(G)}\dim(V_{s}(\mu))}=12$ and $\scale[0.9]{\displaystyle \max_{u\in G_{u}\setminus \{1\}}\dim(V_{u}(1))}=10$.
\end{prop}

\begin{proof}
Let $\lambda=2\omega_{2}$ and $L=L_{1}$. By Lemma \ref{weightlevelAl}, we have $e_{1}(\lambda)=2$, therefore $\displaystyle V\mid_{[L,L]}=V^{0}\oplus V^{1}\oplus V^{2}$. By \cite[Proposition]{Smith_82} and Lemma \ref{dualitylemma}, we have $V^{0}\cong L_{L}(2\omega_{2})$ and $V^{2}\cong L_{L}(2\omega_{3})$. Now, the weight $\displaystyle(\lambda-\alpha_{1}-\alpha_{2})\mid_{T_{1}}=\omega_{2}+\omega_{3}$ admits a maximal vector in $V^{1}$, therefore $V^{1}$ has a composition factor isomorphic to $L_{L}(\omega_{2}+\omega_{3})$. By dimensional considerations, we determine that
\begin{equation}\label{DecompVA32om2}
V\mid_{[L,L]}\cong L_{L}(2\omega_{2})\oplus L_{L}(\omega_{2}+\omega_{3})\oplus L_{L}(2\omega_{3}).
\end{equation}

For the semisimple elements, let $s\in T\setminus \ZG(G)$. If $\dim(V^{i}_{s}(\mu))=\dim(V^{i})$ for some eigenvalue $\mu$ of $s$ on $V$, where $0\leq i\leq 2$, then $s\in \ZG(L)^{\circ}\setminus \ZG(G)$. In this case, as $s$ acts on each $V^{i}$ as scalar multiplication by $c^{4-4i}$ and $c^{4}\neq 1$, we determine that $\dim(V_{s}(\mu))\leq 12$ for all eigenvalues $\mu$ of $s$ on $V$. Moreover, equality holds for $c^{4}=-1$ and $\mu=-1$. We thus assume that $\dim(V^{i}_{s}(\mu))<\dim(V^{i})$ for all eigenvalues $\mu$ of $s$ on $V$ and all $0\leq i\leq 2$. We write $s=z\cdot h$, where $z\in \ZG(L)^{\circ}$ and $h\in [L,L]$, and, by \eqref{DecompVA32om2} and Propositions \ref{PropositionAlsymmetric} and \ref{PropositionAlom1+oml}, we determine that $\dim(V_{s}(\mu))\leq 2\dim((L_{L}(2\omega_{2}))_{h}(\mu_{h}))+\dim((L_{L}(\omega_{2}+\omega_{3}))_{h}(\mu_{h}))\leq 12$ for all eigenvalues $\mu$ of $s$ on $V$. This shows that $\displaystyle \max_{s\in T\setminus\ZG(G)}\dim(V_{s}(\mu))=12$. For the unipotent elements, by Lemma \ref{uniprootelems}, \eqref{DecompVA32om2} and Propositions \ref{PropositionAlsymmetric} and \ref{PropositionAlom1+oml}, we have $\scale[0.9]{\displaystyle \max_{u\in G_{u}\setminus \{1\}}\dim(V_{u}(1))=\dim(V_{x_{\alpha_{3}}(1)}(1))=10}$.
\end{proof}

\begin{prop}\label{A33om2}
Let $p\neq 2,3$, $\ell=3$ and $V=L_{G}(3\omega_{2})$. Then, we have $\scale[0.9]{\displaystyle \max_{s\in T\setminus\ZG(G)}\dim(V_{s}(\mu))}\leq 28$ and $\scale[0.9]{\displaystyle \max_{u\in G_{u}\setminus \{1\}}\dim(V_{u}(1))}=20$.
\end{prop}

\begin{proof}
Let $\lambda=3\omega_{2}$ and $L=L_{1}$. By Lemma \ref{weightlevelAl}, we have $e_{1}(\lambda)=3$, therefore $\displaystyle V\mid_{[L,L]}=V^{0}\oplus \cdots \oplus V^{3}$. By \cite[Proposition]{Smith_82} and Lemma \ref{dualitylemma}, we have $V^{0}\cong L_{L}(3\omega_{2})$ and $V^{3}\cong L_{L}(3\omega_{3})$. Now, the weight $\displaystyle(\lambda-\alpha_{1}-\alpha_{2})\mid_{T_{1}}=2\omega_{2}+\omega_{3}$ admits a maximal vector in $V^{1}$, therefore $V^{1}$ has a composition factor isomorphic to $L_{L}(2\omega_{2}+\omega_{3})$. By dimensional considerations, we determine that $V^{1}\cong L_{L}(2\omega_{2}+\omega_{3})$, hence $V^{2}\cong L_{L}(\omega_{2}+2\omega_{3})$, and so:
\begin{equation}\label{DecompVA33om2}
V\mid_{[L,L]}\cong L_{L}(3\omega_{2})\oplus L_{L}(2\omega_{2}+\omega_{3})\oplus L_{L}(\omega_{2}+2\omega_{3})\oplus L_{L}(3\omega_{3}).
\end{equation}

For the semisimple elements, let $s\in T\setminus \ZG(G)$. If $\dim(V^{i}_{s}(\mu))=\dim(V^{i})$ for some eigenvalue $\mu$ of $s$ on $V$, where $0\leq i\leq 3$, then $s\in \ZG(L)^{\circ}\setminus \ZG(G)$. In this case, as $s$ acts on each $V^{i}$ as scalar multiplication by $c^{6-4i}$ and $c^{4}\neq 1$, we determine that $\dim(V_{s}(\mu))\leq 25$ for all eigenvalues $\mu$ of $s$ on $V$. We thus assume that $\dim(V^{i}_{s}(\mu))<\dim(V^{i})$ for all eigenvalues $\mu$ of $s$ on $V$ and all $0\leq i\leq 3$. We write $s=z\cdot h$, where $z\in \ZG(L)^{\circ}$ and $h\in [L,L]$. Using \eqref{DecompVA33om2} and Propositions \ref{PropositionAlsymmcube} and \ref{A22om1+om2}, we determine that $\dim(V_{s}(\mu))\leq 2\dim((L_{L}(3\omega_{2}))_{h}(\mu_{h}))+2\dim((L_{L}(2\omega_{2}+\omega_{3}))_{h}(\mu_{h}))\leq 28$ for all eigenvalues $\mu$ of $s$ on $V$. This shows that $\displaystyle \max_{s\in T\setminus\ZG(G)}\dim(V_{s}(\mu))\leq 28$. For the unipotent elements, by Lemma \ref{uniprootelems}, \eqref{DecompVA33om2} and Propositions \ref{PropositionAlsymmcube} and \ref{A22om1+om2}, we determine that $\scale[0.9]{\displaystyle \max_{u\in G_{u}\setminus \{1\}}\dim(V_{u}(1))=\dim(V_{x_{\alpha_{3}}(1)}(1))=20}$.
\end{proof}

\begin{prop}\label{A3om1+2om2}
Let $\ell=3$ and $V=L_{G}(\omega_{1}+2\omega_{2})$. Then, we have $\scale[0.9]{\displaystyle \max_{s\in T\setminus\ZG(G)}\dim(V_{s}(\mu))}=33-21\varepsilon_{p}(2)$ and $\scale[0.9]{\displaystyle \max_{u\in G_{u}\setminus \{1\}}\dim(V_{u}(1))}\leq 26+3\varepsilon_{p}(3)-12\varepsilon_{p}(2)$, where equality holds for $p\neq 3$.
\end{prop}

\begin{proof}
Let $\lambda=\omega_{1}+2\omega_{2}$ and let $L=L_{1}$. By Lemma \ref{weightlevelAl}, we have $e_{1}(\lambda)=3$, therefore $\displaystyle V\mid_{[L,L]}=V^{0}\oplus \cdots \oplus V^{3}$. By \cite[Proposition]{Smith_82}, we have $V^{0}\cong L_{L}(2\omega_{2})$. In $V^{1}$ the weight $\displaystyle(\lambda-\alpha_{1})\mid_{T_{1}}=3\omega_{2}$ admits a maximal vector, thus $V^{1}$ has a composition factor isomorphic to $L_{L}(3\omega_{2})$. Moreover, the weight $\displaystyle(\lambda-\alpha_{1}-\alpha_{2})\mid_{T_{1}}=\omega_{2}+\omega_{3}$ occurs with multiplicity $2-\varepsilon_{p}(2)$ and is a sub-dominant weights in the composition factor of $V^{1}$ isomorphic to $L_{L}(3\omega_{2})$, in which it has multiplicity $1-\varepsilon_{p}(3)$. We also note that the weight $\displaystyle(\lambda-\alpha_{1}-2\alpha_{2}-\alpha_{3})\mid_{T_{1}}=0$ occurs with multiplicity $3-3\varepsilon_{p}(2)$ in $V^{1}$. Now, when $p\neq 2$, in $V^{2}$ the weight $\displaystyle(\lambda-2\alpha_{1}-\alpha_{2})\mid_{T_{1}}=2\omega_{2}+\omega_{3}$ admits a maximal vector, thus $V^{2}$ has a composition factor isomorphic to $L_{L}(2\omega_{2}+\omega_{3})$. Moreover, the weight $\displaystyle(\lambda-2\alpha_{1}-2\alpha_{2})\mid_{T_{1}}=2\omega_{3}$ occurs with multiplicity $2$ and is a sub-dominant weights in the composition factor of $V^{2}$ isomorphic to $L_{L}(2\omega_{2}+\omega_{3})$, in which it has multiplicity $1$. On the other hand, when $p=2$ the weight $\displaystyle(\lambda-2\alpha_{1}-2\alpha_{2})\mid_{T_{1}}=2\omega_{3}$ admits a maximal vector in $V^{2}$, thus $V^{2}$ has a composition factor isomorphic to $L_{L}(2\omega_{3})$. We let $p$ be arbitrary. Lastly, in $V^{3}$ the weight $\displaystyle(\lambda-3\alpha_{1}-2\alpha_{2})\mid_{T_{1}}=\omega_{2}+2\omega_{3}$ admits a maximal vector, thus $V^{3}$ has a composition factor isomorphic to $L_{L}(\omega_{2}+2\omega_{3})$. As $
\dim(V^{3})\leq 15-6\varepsilon_{p}(2)$, we determine that $V^{3}\cong L_{L}(\omega_{2}+2\omega_{3})$, that $V^{2}\cong L_{L}(2\omega_{2}+\omega_{3})^{1-\varepsilon_{p}(2)}\oplus L_{L}(2\omega_{3})$, see \cite[II.2.14]{Jantzen_2007representations}, and that $V^{1}$ has exactly $2+2\varepsilon_{p}(3)-\varepsilon_{p}(2)$ composition factors: one isomorphic to $L_{L}(3\omega_{2})$, $1+\varepsilon_{p}(3)-\varepsilon_{p}(2)$ to $L_{L}(\omega_{2}+\omega_{3})$, and $\varepsilon_{p}(3)$ to $L_{L}(0)$. 

For the semisimple elements, let $s\in T\setminus \ZG(G)$. If $\dim(V^{i}_{s}(\mu))=\dim(V^{i})$ for some eigenvalue $\mu$ of $s$ on $V$, where $0\leq i\leq 3$, then $s\in \ZG(L)^{\circ}\setminus \ZG(G)$. In this case, as $s$ acts on each $V^{i}$ as scalar multiplication by $c^{7-4i}$ and $c^{4}\neq 1$, we determine that $\dim(V_{s}(\mu))\leq 33-21\varepsilon_{p}(2)$, where equality holds for $p\neq 2$, $c^{4}=-1$ and $\mu=c^{3}$, respectively for $p=2$, $c^{3}=1$ and $\mu=c$. We thus assume that $\dim(V^{i}_{s}(\mu))<\dim(V^{i})$ for all eigenvalues $\mu$ of $s$ on $V$ and all $0\leq i\leq 3$. We write $s=z\cdot h$, where $z\in \ZG(L)^{\circ}$ and $h\in [L,L]$, and we have $\scale[0.9]{\displaystyle \dim(V_{s}(\mu))\leq \sum_{i=0}^{3}\dim(V^{i}_{h}(\mu^{i}_{h}))}$, where $\dim(V^{i}_{h}(\mu^{i}_{h}))<\dim(V^{i})$ for all eigenvalues $\mu^{i}_{h}$ of $h$ on $V^{i}$. First, let $p\neq 2$. Assume that $h$ is conjugate to $h_{\alpha_{2}}(d)h_{\alpha_{3}}(d^{2})$ with $d^{3}=-1$. Then, using the weight structure of $V^{1}$, respectively of $V^{2}$, one shows that the eigenvalues of $h$ on $V^{1}$, respectively on $V^{2}$, are $1$ with $\dim(V^{1}_{h}(1))=8$ and $-1$ with $\dim(V^{1}(-1))=10$, respectively $d$ with $\dim(V^{2}_{h}(d))=10$ and $-d$ with $\dim(V^{2}_{h}(-d))=11$. Thus, by Proposition \ref{PropositionAlsymmetric}, respectively Proposition \ref{A22om1+om2}, we have $\dim(V^{0}_{h}(\mu^{0}_{h}))\leq 4$ for all eigenvalues $\mu_{h}^{0}$ of $h$ on $V^{0}$, respectively $\dim(V^{3}_{h}(\mu^{3}_{h}))\leq 8$ for all eigenvalues $\mu_{h}^{3}$ of $h$ on $V^{3}$, and so $\dim(V_{s}(\mu))\leq 33$ for all eigenvalues $\mu$ of $s$ on $V$. On the other hand, if $h$ is not conjugate to $h_{\alpha_{2}}(d)h_{\alpha_{3}}(d^{2})$ with $d^{3}=-1$, then, by the structure of $V\mid_{[L,L]}$ and Propositions \ref{PropositionAlnatural}, \ref{PropositionAlsymmetric}, \ref{PropositionAlsymmcube}, \ref{PropositionAlom1+oml} and \ref{A22om1+om2}, we determine that $\dim(V_{s}(\mu))\leq \dim((L_{L}(3\omega_{2}))_{h}(\mu_{h}))+(1+\varepsilon_{p}(3))\dim((L_{L}(\omega_{2}+\omega_{3}))_{h}(\mu_{h}))+2\dim((L_{L}(2\omega_{2}+\omega_{3}))_{h}(\mu_{h}))+2\dim((L_{L}(2\omega_{3}))_{h}(\mu_{h}))+\varepsilon_{p}(3)\dim((L_{L}(0))_{h}(\mu_{h}))\leq 32+\varepsilon_{p}(3)$ for all eigenvalues $\mu$ of $s$ on $V$. Therefore $\scale[0.9]{\displaystyle \max_{s\in T\setminus\ZG(G)}\dim(V_{s}(\mu))}=33$. We now consider the case when $p=2$. One sees that the eigenvalues of $h$ on $V^{1}$ are $d^{3}$, $d^{-1}e^{2}$, $de^{-2}$, $de$, $d^{-3}e^{3}$, $d^{-1}e^{-1}$, $d^{2}e^{-1}$, $d^{-2}e$ and $e^{-3}$, where $d,e\in k^{*}$ do not simultaneously satisfy $e^{3}=1$ and $d=e^{2}$. We deduce that $\dim(V^{1}_{h}(\mu_{h}^{1}))\leq 5$ for all eigenvalues $\mu_{h}^{1}$ of $h$ on $V^{1}$. Further, by Propositions \ref{PropositionAlnatural} and \ref{A22om1+om2}, we determine that $\dim(V_{s}(\mu))\leq 13$ for all eigenvalues $\mu$ of $s$ on $V$. Lastly, assume that there exist $(s,\mu)\in T\setminus \ZG(G)\times k^{*}$ such that $\dim(V_{s}(\mu))=13$. Then, by Proposition \ref{PropositionAlnatural}, we deduce that $h$ is conjugate to $h_{\alpha_{2}}(d)h_{\alpha_{3}}(d^{2})$ with $d^{3}\neq 1$. In this case, as $\dim(V^{1}_{h}(\mu_{h}^{1}))=5$, where $\mu_{h}^{1}=c^{-3}\mu$, we have  $\mu_{h}^{1}=d^{3}$, where $d^{9}=1$. Similarly, we deduce that $\mu_{h}^{0}=d^{2}$, where $\mu_{h}^{0}=c^{-7}\mu$, and $\mu_{h}^{2}=d^{-2}$, where $\mu_{h}^{2}=c\mu$. Lastly, as $c^{7}d^{2}=c^{3}d^{3}=c^{-1}d^{-2}$, it follows that $d^{3}=1$, a contradiction. Thus, we conclude that $\scale[0.9]{\displaystyle \max_{s\in T\setminus\ZG(G)}\dim(V_{s}(\mu))}=12$.

For the unipotent elements, by Lemma \ref{uniprootelems}, we have $\scale[0.9]{\displaystyle \max_{u\in G_{u}\setminus \{1\}}\dim(V_{u}(1))=\dim(V_{x_{\alpha_{3}}(1)}(1))}$. Now, if $p\neq 2$, then, using the structure of $V\mid_{[L,L]}$ and Propositions \ref{PropositionAlnatural}, \ref{PropositionAlsymmetric}, \ref{PropositionAlsymmcube}, \ref{PropositionAlom1+oml} and \ref{A22om1+om2}, we determine that $\scale[0.9]{\displaystyle \max_{u\in G_{u}\setminus \{1\}}\dim(V_{u}(1))}\leq 26+3\varepsilon_{p}(3)$. Lastly, if $p=2$, then, as $x_{\alpha_{3}}(1)$ acts on $L_{L}(\omega_{2})$, respectively on $L_{L}(\omega_{2})^{(2)}$, as $J_{2}\oplus J_{1}$, we determine that $x_{\alpha_{3}}(1)$ acts on $L_{L}(3\omega_{2})$ as $J_{2}^{4}\oplus J_{1}$, and so, using Propositions \ref{PropositionAlnatural} and \ref{A22om1+om2}, we conclude that $\scale[0.9]{\displaystyle \max_{u\in G_{u}\setminus \{1\}}\dim(V_{u}(1))}=14$. 
\end{proof}

\begin{prop}\label{A32om1+2om3}
Let $p\neq 2$, $\ell=3$ and $V=L_{G}(2\omega_{1}+2\omega_{3})$. Then, we have $\scale[0.9]{\displaystyle \max_{s\in T\setminus\ZG(G)}\dim(V_{s}(\mu))}=48-\varepsilon_{p}(5)-9\varepsilon_{p}(3)$ and $\scale[0.9]{\displaystyle \max_{u\in G_{u}\setminus \{1\}}\dim(V_{u}(1))}\leq 36-3\varepsilon_{p}(3)$.
\end{prop}

\begin{proof}
Let $\lambda=2\omega_{1}+2\omega_{3}$ and $L=L_{1}$. By Lemma \ref{weightlevelAl}, we have $e_{1}(\lambda)=4$, therefore $\displaystyle V\mid_{[L,L]}=V^{0}\oplus \cdots \oplus V^{4}$. By \cite[Proposition]{Smith_82} and Lemma \ref{dualitylemma}, we have $V^{0}\cong L_{L}(2\omega_{3})$ and $V^{4}\cong L_{L}(2\omega_{2})$. Now, the weight $\displaystyle(\lambda-\alpha_{1})\mid_{T_{1}}=\omega_{2}+2\omega_{3}$ admits a maximal vector in $V^{1}$, therefore $V^{1}$ has a composition factor isomorphic to $L_{L}(\omega_{2}+2\omega_{3})$. Moreover, the weight $\displaystyle(\lambda-\alpha_{1}-\alpha_{2}-\alpha_{3})\mid_{T_{1}}=\omega_{3}$ occurs with multiplicity $3-\varepsilon_{p}(3)$ and is a sub-dominant weight of multiplicity $2$ in the composition factor of $V^{1}$ isomorphic to $L_{L}(\omega_{2}+2\omega_{3})$. Similarly, in $V^{2}$, the weight $\displaystyle(\lambda-2\alpha_{1})\mid_{T_{1}}=2\omega_{2}+2\omega_{3}$ admits a maximal vector, thus $V^{2}$ has a composition factor isomorphic to $L_{L}(2\omega_{2}+2\omega_{3})$. Moreover, the weight $\displaystyle(\lambda-2\alpha_{1}-\alpha_{2}-\alpha_{3})\mid_{T_{1}}=\omega_{2}+\omega_{3}$ occurs with multiplicity $3-\varepsilon_{p}(3)$ and is a sub-dominant weight of multiplicity $2-\varepsilon_{p}(5)$ in the composition factor of $V^{2}$ isomorphic to $L_{L}(2\omega_{2}+2\omega_{3})$. Further, we note that the weight $\displaystyle(\lambda-2\alpha_{1} -2\alpha_{2}-2\alpha_{3})\mid_{T_{1}}=0$ occurs with multiplicity $6-\varepsilon_{p}(5)-3\varepsilon_{p}(3)$ in $V^{2}$. By dimensional considerations and \cite[II.2.14]{Jantzen_2007representations}, we determine that 
\begin{equation}\label{DecompVA32om1+2om3}
V\mid_{[L,L]}\cong L_{L}(2\omega_{3})\oplus L_{L}(\omega_{2}+2\omega_{3})\oplus L_{L}(\omega_{3})^{1-\varepsilon_{p}(3)}\oplus V^{2}\oplus L_{L}(2\omega_{2}+\omega_{3})\oplus L_{L}(\omega_{2})^{1-\varepsilon_{p}(3)}\oplus L_{L}(2\omega_{2}),
\end{equation}
where $V^{2}$ has exactly $3-2\varepsilon_{p}(3)$ composition factors: one isomorphic to $L_{L}(2\omega_{2}+2\omega_{3})$, $1+\varepsilon_{p}(5)-\varepsilon_{p}(3)$ to $L_{L}(\omega_{2}+\omega_{3})$ and $1-\varepsilon_{p}(5)-\varepsilon_{p}(3)$ to $L_{L}(0)$.

We start with the semisimple elements. Let $s\in T\setminus \ZG(G)$. If $\dim(V^{i}_{s}(\mu))=\dim(V^{i})$ for some eigenvalue $\mu$ of $s$ on $V$, where $0\leq i\leq 4$, then $s\in \ZG(L)^{\circ}\setminus \ZG(G)$. In this case, as $s$ acts on each $V^{i}$ as scalar multiplication by $c^{8-4i}$ and $c^{4}\neq 1$, we determine that $\dim(V_{s}(\mu))\leq 48-\varepsilon_{p}(5)-9\varepsilon_{p}(3)$, where equality holds for $c^{4}=-1$ and $\mu=1$. We thus assume that $\dim(V^{i}_{s}(\mu))<\dim(V^{i})$ for all eigenvalues $\mu$ of $s$ on $V$ and all $0\leq i\leq 4$. We write $s=z\cdot h$, where $z\in \ZG(L)^{\circ}$ and $h\in [L,L]$, and, by \eqref{DecompVA32om1+2om3} and Propositions \ref{PropositionAlnatural}, \ref{PropositionAlsymmetric}, \ref{PropositionAlom1+oml}, \ref{A22om1+om2} and \ref{A22om1+2om2pneq2}, we determine that $\dim(V_{s}(\mu))\leq 2\dim((L_{L}(2\omega_{3}))_{h}(\mu_{h}))+2\dim((L_{L}(2\omega_{2}+\omega_{3}))_{h}(\mu_{h}))+(2-2\varepsilon_{p}(3))\dim((L_{L}(\omega_{3}))_{h}(\mu_{h}))+\dim((L_{L}(2\omega_{2}+2\omega_{3}))_{h}(\mu_{h}))+(1+\varepsilon_{p}(5)-\varepsilon_{p}(3))\dim((L_{L}(\omega_{2}+\omega_{3}))_{h}(\mu_{h}))+(1-\varepsilon_{p}(5)-\varepsilon_{p}(3))\dim((L_{L}(0))_{h}(\mu_{h}))\leq 48-\varepsilon_{p}(5)-9\varepsilon_{p}(3)$ for all eigenvalues $\mu$ of $s$ on $V$. This shows that $\scale[0.9]{\displaystyle \max_{s\in T\setminus\ZG(G)}\dim(V_{s}(\mu))}=48-\varepsilon_{p}(5)-9\varepsilon_{p}(3)$. For the unipotent elements, by Lemma \ref{uniprootelems}, \eqref{DecompVA32om1+2om3} and Propositions \ref{PropositionAlnatural}, \ref{PropositionAlsymmetric}, \ref{PropositionAlom1+oml}, \ref{A22om1+om2} and \ref{A22om1+2om2pneq2}, we determine that $\scale[0.9]{\displaystyle \max_{u\in G_{u}\setminus \{1\}}\dim(V_{u}(1))}$ $\scale[0.9]{=\dim(V_{x_{\alpha_{3}}(1)}(1))\leq 36-3\varepsilon_{p}(3)}$.
\end{proof}

\section{Proof of Theorem \ref{ResultsCl}}

In this section $G$ is a simple, simply connected linear algebraic group of type $C_{\ell}$, $\ell\geq 2$. Before we begin, we fix the following hypothesis on the semisimple elements of $G$:
\begin{equation*}
\scale[0.9]{\begin{split}
(^{\dagger}H_{s}): & \text{ any }s\in T\setminus \ZG(G) \text{ is such that } s=\diag(\mu_{1}\cdot \I_{n_{1}},\dots,\mu_{m}\cdot \I_{n_{m}},\mu_{m}^{-1}\cdot \I_{n_{m}},\dots, \mu_{1}^{-1}\cdot \I_{n_{1}}),\text{where }\\
                   & \mu_{i}\neq \mu_{j}^{\pm 1} \text{ for all } i<j, \ \sum_{i=1}^{m}n_{i}=\ell \text{ and } n_{1}\geq n_{2}\geq \cdots \geq n_{m}\geq 1. \ \text{Moreover, if }m=1, \text{ then }\mu_{1}\neq \pm 1.
\end{split}}
\end{equation*}

\begin{prop}\label{PropositionClnatural}
Let $V=L_{G}(\omega_{1})$. Then $\nu_{G}(V)=1$.  Moreover, we have:
\begin{enumerate}
\item[\emph{$(1)$}] $\scale[0.9]{\displaystyle \max_{u\in G_{u}\setminus \{1\}}\dim(V_{u}(1))=2\ell-1}$.
\item[\emph{$(2)$}] $\scale[0.9]{\displaystyle \max_{s\in T\setminus\ZG(G)}\dim(V_{s}(\mu))=2\ell-2}$, where the maximum is attained if and only if $\mu=\pm 1$ and, up to conjugation, $s=\diag(\pm 1,\dots, \pm 1, d,d^{-1},$ $\pm 1,\dots, \pm 1)$ with $d\neq \pm 1$.
\end{enumerate}
\end{prop}

\begin{proof}
To begin, we note that $V\cong W$ as $kG$-modules. Let $s\in T\setminus \ZG(G)$ and let $\mu\in k^{*}$ be an eigenvalue of $s$ on $V$. If $\mu\neq \mu^{-1}$, then, $\dim(V_{s}(\mu))\leq \frac{\dim(V)}{2}= \ell$, as $V$ is self-dual. On the other hand, if $\mu=\pm 1$, as $s\notin \ZG(G)$,  we have $\dim(V_{s}(\pm 1))\leq 2\ell-2$. Now equality holds if and only if, up to conjugation, $s=\diag(\pm 1,\dots, \pm 1,d,d^{-1},\pm 1,\dots, \pm 1)$ with $d\neq \pm 1$. In the case of the unipotent elements,  by Lemma \ref{uniprootelems}, we have $\displaystyle \max_{u\in G_{u}\setminus \{1\}}\dim(V_{u}(1))=\max\{\dim(V_{x_{\alpha_{\ell}}(1)}(1)), \dim(V_{x_{\alpha_{1}}(1)}(1))\}$. Therefore, $\scale[0.9]{\displaystyle \max_{u\in G_{u}\setminus \{1\}}\dim(V_{u}(1))=2\ell-1}$, and, as $\scale[0.9]{\displaystyle \max_{s\in T\setminus\ZG(G)}\dim(V_{s}(\mu))=2\ell-2}$, it follows that $\nu_{G}(V)=1$. 
\end{proof}

\begin{prop}\label{PropositionClwedge}
Let $V=L_{G}(\omega_{2})$. Then $\nu_{G}(V)=2\ell-2-\varepsilon_{\ell}(2)$. Moreover, we have:
\begin{enumerate}
\item[\emph{$(1)$}] $\scale[0.9]{\displaystyle \max_{u\in G_{u}\setminus \{1\}}\dim(V_{u}(1))=2\ell^{2}-3\ell+1-\varepsilon_{p}(\ell)+\varepsilon_{\ell}(2)\varepsilon_{p}(2)}$.
\item[\emph{$(2)$}] $\ell=2$ and $\scale[0.9]{\displaystyle \max_{s\in T\setminus\ZG(G)}\dim(V_{s}(\mu))=4-2\varepsilon_{p}(2)}$, where the maximum is attained if and only if
\begin{enumerate}
\item[\emph{$(2.1)$}] $p\neq 2$, $\mu=-1$ and, up to conjugation, $s=\diag(1,-1,-1,1)$.
\item[\emph{$(2.2)$}] $p=2$, $\mu=1$ and, up to conjugation, $s=\diag(d,d,d^{-1},d^{-1})$ with $d\neq 1$.
\item[\emph{$(2.3)$}] $p=2$, $\mu=d^{\pm 1}$ and, up to conjugation, $s=\diag(d,1,1,d^{-1})$ with $d\neq 1$.
\end{enumerate}
\item[\emph{$(3)$}] $\ell=3$ and $\displaystyle \max_{s\in T\setminus\ZG(G)}\dim(V_{s}(\mu))=8$, where the maximum is attained if and only if
\begin{enumerate}
\item[\emph{$(3.1)$}] $p\neq 3$, $\mu=1$ and, up to conjugation, $s=\diag(d,d,d,d^{-1},d^{-1},d^{-1})$ with $d^{2}\neq 1$.
\item[\emph{$(3.2)$}] $p\neq 2$, $\mu=-1$ and, up to conjugation, $s=\pm \diag(1,1,-1,-1,1,1)$.
\end{enumerate}
\item[\emph{$(4)$}] $\ell=4$ and $\scale[0.9]{\displaystyle \max_{s\in T\setminus\ZG(G)}\dim(V_{s}(\mu))=16-2\varepsilon_{p}(2)}$, where the maximum is attained if and only if
\begin{enumerate}
\item[\emph{$(4.1)$}] $p\neq 2$, $\mu=-1$ and, up to conjugation, $s=\diag(1,1,-1,-1,-1,-1,1,1)$.
\item[\emph{$(4.2)$}] $p=2$, $\mu=1$ and, up to conjugation, $s=\diag(d,d,d,d,d^{-1},d^{-1},d^{-1},d^{-1})$ with $d\neq 1$.
\item[\emph{$(4.3)$}] $p=2$, $\mu=1$ and, up to conjugation, $s=\diag(1,1,1,d,d^{-1},1,1,1)$ with $d\neq 1$.
\end{enumerate}
\item[\emph{$(5)$}] $\ell\geq 5$ and $\scale[0.9]{\displaystyle \max_{s\in T\setminus\ZG(G)}\dim(V_{s}(\mu))=2\ell^{2}-5\ell+3-\varepsilon_{p}(\ell)}$, where the maximum is attained if and only if $\mu=1$ and, up to conjugation, $s=\pm \diag(1,\dots,1,d,d^{-1},1,\dots, 1)$ with $d\neq 1$.
\end{enumerate}
\end{prop}

\begin{proof}
To begin, we note that $\wedge^{2}(W)\cong L_{G}(0)\oplus V$, if $\varepsilon_{p}(\ell)=0$, while if $\varepsilon_{p}(\ell)=1$, we have $\wedge^{2}(W)\cong L_{G}(0)\mid V\mid L_{G}(0)$, see \cite[Lemma $4.8.2$]{mcninch_1998}. 

We first consider the case of the unipotent elements. By Lemma \ref{uniprootelems}, we have $\scale[0.9]{\displaystyle \max_{u\in G_{u}\setminus \{1\}}\dim(V_{u}(1))}$ $=\max\{\dim(V_{x_{\alpha_{\ell}}(1)}(1)), \dim(V_{x_{\alpha_{1}}(1)}(1))\}$. Now, by Proposition \ref{PropositionAlwedge}, we have $\dim((\wedge^{2}(W))_{x_{\alpha_{\ell}}(1)}(1))=2\ell^{2}-3\ell+2$ and so, by the structure of $\wedge^{2}(W)$ as a $kG$-module and \cite[Corollary $6.2$]{Korhonen_2019}, or \cite[Theorem B]{Korhonen_2020HesselinkNF}, we determine that $\dim(V_{x_{\alpha_{\ell}}(1)}(1))=2\ell^{2}-3\ell+1-\varepsilon_{p}(\ell)$. To calculate $\dim((\wedge^{2}(W))_{x_{\alpha_{1}}(1)}(1))$, we write $\scale[0.9]{\displaystyle W=W_{1}\oplus W_{2}}$, where $\dim(W_{1})=4$ and $x_{\alpha_{1}}(1)$ acts as $J_{2}^{2}$ on $W_{1}$, and $\dim(W_{2})=2\ell-4$ and $x_{\alpha_{1}}(1)$ acts trivially on $W_{2}$. Then, as $\scale[0.9]{\displaystyle \wedge^{2}(W)=\wedge^{2}(W_{1})\oplus [W_{1}\otimes W_{2}]\oplus \wedge^{2}(W_{2})}$, using \cite[Lemma $3.4$]{liebeck_2012unipotent} and Lemma \ref{Lemma on fixed space for wedge in p=2}, we determine that $\dim((\wedge^{2}(W))_{x_{\alpha_{1}}(1)}(1))=2\ell^{2}-5\ell+6$. Once more, by the structure of $\wedge^{2}(W)$ as a $kG$-module and \cite[Corollary $6.2$]{Korhonen_2019}, or \cite[Theorem B]{Korhonen_2020HesselinkNF}, we determine that $\dim(V_{x_{\alpha_{1}}(1)}(1))=2\ell^{2}-5\ell+5-\varepsilon_{p}(\ell)+\varepsilon_{p}(2)\varepsilon_{\ell}(2)$. It follows that $\scale[0.9]{\displaystyle \max_{u\in G_{u}\setminus \{1\}}\dim(V_{u}(1))=2\ell^{2}-3\ell+1-\varepsilon_{p}(\ell)+\varepsilon_{\ell}(2)\varepsilon_{p}(2)}$.
 
We now focus on the semisimple elements. To ease notation, set $B=4-2\varepsilon_{p}(2)$ if $\ell=2$; $B=8$ if $\ell=3$; $B=16-2\varepsilon_{p}(2)$ if $\ell=4$; and $B=2\ell^{2}-5\ell+3-\varepsilon_{p}(\ell)$ if $\ell\geq 5$. Let $s\in T\setminus \ZG(G)$ be as in hypothesis $(^{\dagger}H_{s})$. Using the structure of $\wedge^{2}(W)$ as a $kG$-module, we deduce that the eigenvalues of $s$ on $V$, not necessarily distinct, are:
\begin{equation}\label{enum Cl_P2}
\scale[0.9]{\begin{cases}
\mu_{i}^{2}$ and $\mu_{i}^{-2}$, $1\leq i\leq m$, each with multiplicity at least $\frac{n_{i}(n_{i}-1)}{2};\\
\mu_{i}\mu_{j}$ and $\mu_{i}^{-1}\mu_{j}^{-1}$, $1\leq i<j\leq m$, each with multiplicity at least $n_{i}n_{j};\\
\mu_{i}\mu_{j}^{-1}$ and $\mu_{i}^{-1}\mu_{j}$, $1\leq i<j\leq m$, each with multiplicity at least $n_{i}n_{j};\\
1$ with multiplicity at least $\displaystyle \sum_{i=1}^{m}n_{i}^2-1-\varepsilon_{p}(\ell).
\end{cases}}
\end{equation}

Let $\mu\in k^{*}$ be an eigenvalue of $s$ on $V$. If $\mu\neq \mu^{-1}$, then $\dim(V_{s}(\mu))\leq \dim(V)-\dim(V_{s}(1))-\dim(V_{s}(\mu^{-1}))$. Since $n_{i}\geq 1$ for all $1\leq i\leq m$, we have $\scale[0.9]{\dim(V_{s}(1))\geq \displaystyle \sum_{i=1}^{m}n_{i}^{2}-1-\varepsilon_{p}(\ell)\geq\ell-1-\varepsilon_{p}(\ell)}$. Further, since $V$ is self-dual, we deduce that $\dim(V_{s}(\mu))\leq \frac{2\ell^{2}-\ell-1-\varepsilon_{p}(\ell)-(\ell-1-\varepsilon_{p}(\ell))}{2}=\ell^{2}-\ell$. Thus, $\dim(V_{s}(\mu))\leq B$, for all $(s,\mu)\in T\setminus \ZG(G)\times k^{*}$ with $\mu\neq \mu^{-1}$, and equality holds if and only if $\ell$, $p$, $s$ and $\mu$ are as in $(2.3)$.

We thus assume that $\mu=\pm 1$. We begin with the case of $m=1$, i.e. $n_{1}=\ell$ and $\mu_{1}^{2}\neq 1$. By \eqref{enum Cl_P2}, we have $\dim(V_{s}(1))=\ell^{2}-1-\varepsilon_{p}(\ell)$ and $\dim(V_{s}(-1))\leq \ell^{2}-\ell$. Therefore, $\dim(V_{s}(\pm 1))\leq B$ for all $s\in T\setminus \ZG(G)$ with $m=1$, and equality holds if and only if $\ell$, $p$, $s$ and $\mu$ are as in $(2.2)$, $(3.1)$ and $(4.2)$. We now let $m\geq 2$. If $\ell=2$, then $m=2$, i.e. $n_{1}=n_{2}=1$, and, using \eqref{enum Cl_P2}, we determine that $\dim(V_{s}(1))=1-\varepsilon_{p}(\ell)$ and $\dim(V_{s}(-1))\leq 4$ where equality holds if and only if $\mu_{1}=-\mu_{2}$ and $\mu_{2}=\pm 1$, hence if and only if $p$, $s$ and $\mu$ are as in $(2.1)$. If $\ell=3$, then $m\leq 3$. First, let $m=2$, i.e. $n_{1}=2$ and $n_{2}=1$. Since $\mu_{1}\neq \mu_{2}^{\pm 1}$, by \eqref{enum Cl_P2}, it follows that $\dim(V_{s}(1))\leq 6-\varepsilon_{p}(\ell)$ and $\dim(V_{s}(-1))\leq 8$, where equality holds if and only if $\mu_{1}\mu_{2}^{\pm 1}=-1$, i.e.  if and only if $s$ is as in $(3.2)$. Now, let $m=3$. Since $\mu_{i}\neq \mu_{j}^{\pm 1}$ for all $i<j$, by \eqref{enum Cl_P2}, we get $\dim(V_{s}(1))=2-\varepsilon_{p}(\ell)$ and $\dim(V_{s}(-1))\leq 12$. Further, as $-1$ can equal at most one eigenvalue of the form $\mu_{i}\mu_{j}$ and at most one of the form $\mu_{i}\mu_{j}^{-1}$, we deduce that $\dim(V_{s}(-1))\leq 4$.

Having dealt with $\ell=2$ and $\ell=3$, let $\ell\geq 4$.  Recall that we are still in the case of $m\geq 2$ and $\mu=\pm 1$. For $\mu=1$, since $\mu_{i}\neq \mu_{j}^{\pm 1}$ for all $ i<j$, it follows that $\mu_{i}^{\pm 1}\mu_{j}^{\pm 1}\neq 1$ for all $i<j$, and so
\begin{equation}\label{eqn: (Cl_P2_1.x)}
\scale[0.9]{\displaystyle \dim(V_{s}(1))\leq 2\ell^{2}-\ell-1-\varepsilon_{p}(\ell)-4\sum_{i<j}n_{i}n_{j}.}
\end{equation}
Assume $\dim(V_{s}(1))\geq B$. Then, for $\ell=4$, we get $\scale[0.9]{\displaystyle 11+\varepsilon_{p}(\ell)-4\sum_{i<j}n_{i}n_{j}\geq 0}$. As $\scale[0.9]{\displaystyle \sum_{r=1}^{m}n_{r}=4}$, $m\geq 2$ and $n_{r}\geq n_{q}$ for all $r<q$, we have $\scale[0.9]{\displaystyle \sum_{i<j}n_{i}n_{j}\geq 3}$, therefore, the inequality $\scale[0.9]{\displaystyle 11+\varepsilon_{p}(\ell)-4\sum_{i<j}n_{i}n_{j}\geq 0}$ holds if and only if $p=2$, $m=2$, $n_{1}=3$ and $n_{2}=1$. Substituting in \eqref{eqn: (Cl_P2_1.x)}, gives $\dim(V_{s}(1))\leq 14$, where equality holds if and only if $\mu_{1}^{2}=1$, see \eqref{enum Cl_P2}, hence, if and only if $s$ is as in $(4.3)$. In the case of $\ell\geq 5$, by \eqref{eqn: (Cl_P2_1.x)}, we get:
\begin{equation}\label{Cl_wed_1mgeq2}
\scale[0.9]{\displaystyle\ell-1-\sum_{i<j}n_{i}n_{j}\geq 0}
\end{equation}
and, since $\scale[0.9]{\ell=\displaystyle \sum_{i=1}^{m}n_{i}}$, it follows that $\scale[0.9]{\displaystyle \sum_{i=1}^{m-2}n_{i}(1-\sum_{i<j}n_{j})+(n_{m-1}-1)(1-n_{m})\geq 0}$. But $\scale[0.9]{\displaystyle \sum_{i=1}^{m-2}n_{i}(1- \displaystyle \sum_{i<j}n_{j})\leq 0}$ and $(n_{m-1}-1)(1-n_{m})\leq 0$, as $n_{i}\geq 1$ for all $1\leq i\leq m$, therefore inequality (\ref{Cl_wed_1mgeq2}) holds if and only if $m=2$, $n_{2}=1$ and $n_{1}=\ell-1$. Substituting in  \eqref{eqn: (Cl_P2_1.x)}, gives $\dim(V_{s}(1))\leq 2\ell^{2}-5\ell+3-\varepsilon_{p}(\ell)$, where equality holds if and only if all eigenvalues of $s$ on $V$ different than $\mu_{1}^{\pm 1}\mu_{2}^{\pm 1}$ are equal to $1$. Consequently, $\dim(V_{s}(1))\leq 2\ell^{2}-5\ell+3-\varepsilon_{p}(\ell)$ for all $s$ with $m\geq 2$, where equality holds if and only if $s$ is as in $(5)$.

Lastly,  let $\mu=-1$. We remark that $\scale[0.9]{\dim(V_{s}(-1))\leq \dim(V)-\dim(V_{s}(1))\leq 2\ell^{2}-\ell-\displaystyle \sum_{r=1}^{m}n_{r}^{2}}$, see \eqref{enum Cl_P2}. If $\mu_{i}\mu_{j}\neq -1$ for all $i<j$, we have $\scale[0.9]{\displaystyle\dim(V_{s}(-1))\leq 2\ell^{2}-\ell-\sum_{r=1}^{m}n_{r}^{2}-2\sum_{i<j}n_{i}n_{j}}$ $\scale[0.9]{\displaystyle=\ell^{2}-\ell<B}$. We thus assume that there exist $i<j$ such that $\mu_{i}\mu_{j}=-1$. Then $\mu_{i}^{-1}\mu_{j}^{-1}=-1$ and, since the $\mu_{i}$'s are distinct, we have that:

\begin{equation}\label{enum Cl_P2_2}
\scale[0.9]{\begin{cases}
\mu_{i}^{2}\neq -1$ and $\mu_{j}^{2}\neq -1$, hence $\mu_{i}^{-2}\neq -1$ and $\mu_{j}^{-2}\neq -1;\\
\mu_{i}\mu_{r}\neq -1$ and $\mu_{i}^{-1}\mu_{r}^{-1}\neq -1$, where $i<r\leq m$, $r\neq j$; and $\mu_{r}\mu_{i}\neq -1$ and $\mu_{r}^{-1}\mu_{i}^{-1}\neq -1$, where $1\leq r<i;\\ 
\mu_{r}\mu_{j}\neq -1$ and $\mu_{r}^{-1}\mu_{j}^{-1}\neq -1$, where $1\leq r<j$, $r\neq i$; and $\mu_{j}\mu_{r}\neq -1$ and $\mu_{j}^{-1}\mu_{r}^{-1}\neq -1$,  where$j<r\leq m.\\
\end{cases}}
\end{equation}
By \eqref{enum Cl_P2}, all of the above account for at least $n_{i}(n_{i}-1)+n_{j}(n_{j}-1)+2(n_{i}+n_{j})(\ell-n_{i}-n_{j})$ additional eigenvalues of $s$ on $V$ different than $-1$. This gives
\begin{equation}\label{eqn: (Cl_P2_1)}
\scale[0.9]{\displaystyle \dim (V_{s}(-1))\leq 2\ell^{2}-\ell-\sum_{r=1}^{m}n_{r}^{2}-n_{i}(n_{i}-1)-n_{j}(n_{j}-1)-2(n_{i}+n_{j})(\ell-n_{i}-n_{j}).}
\end{equation}

Assume $\dim(V_{s}(-1))\geq B$. For $\ell=4$, as $p\neq 2$, we get $\scale[0.9]{\displaystyle 12-\sum_{r\neq i,j}n_{r}^{2}-7(n_{i}+n_{j})+4n_{i}n_{j}\geq 0}$, where $(n_{i},n_{j})\in \{(3,1), (2,2), (2,1), (1,1)\}$. For $(n_{i},n_{j})\in \{(3,1), (2,1)\}$, the inequality does not hold. For $(n_{i},n_{j})=(2,2)$, we get $\dim(V_{s}(-1))\leq 16$, see \eqref{eqn: (Cl_P2_1)}, where equality holds if and only if all eigenvalues of $s$ on $V$ different than $1$ and the ones listed in \eqref{enum Cl_P2_2} are equal to $-1$, hence if and only if $s$ is as in $(4.1)$. Lastly, for $(n_{i},n_{j})=(1,1)$, then $\scale[0.9]{\displaystyle\dim(V_{s}(-1))\leq 18-\sum_{r\neq i,j}n_{r}^{2}}$, see \eqref{eqn: (Cl_P2_1)}. Since we are assuming that $\dim(V_{s}(-1))\geq 16$, we must have $n_{r}=1$ for all $r\neq i,j$, therefore $m=4$ and $n_{i}=1$ for all $1\leq i\leq 4$. Substituting in \eqref{eqn: (Cl_P2_1)} gives $\dim(V_{s}(-1))\leq 16$, where equality holds  if and only if all eigenvalues of $s$ on $V$ different than $1$ and the ones listed in \eqref{enum Cl_P2_2} are equal to $-1$. However, as at most one eigenvalue of the form $\mu_{i}\mu_{r}^{-1}$, $r\neq i,j$ can equal $-1$, we see that the condition for equality cannot be satisfied. Having completed the case of $\ell=4$, we now assume that $\ell\geq 5$. Then, as $\dim(V_{s}(-1))\geq 2\ell^{2}-5\ell+3-\varepsilon_{p}(\ell)$, we get:

\begin{equation}\label{eqn: (Cl_P2_3)}
\scale[0.9]{\ell(4-n_{i}-n_{j})-3+\varepsilon_{p}(\ell)-\displaystyle \sum_{r\neq i,j}n_{r}^{2}-(n_{i}-n_{j})^{2}-(n_{i}+n_{j})(\ell-n_{i}-n_{j}-1)\geq 0.}
\end{equation}
If $\ell-n_{i}-n_{j}-1<0$, then, as $\scale[0.9]{\displaystyle \sum_{r=1}^{m}n_{r}=\ell}$, we have $m=2$ and so $\ell=n_{1}+n_{2}$. Substituting in (\ref{eqn: (Cl_P2_3)}) gives $\scale[0.9]{\displaystyle \ell(5-\ell)-3+\varepsilon_{p}(\ell)-\sum_{r\neq i,j}n_{r}^{2}-(2n_{1}-\ell)^{2}\geq 0}$, which does not hold as $\ell\geq 5$. If $\ell-n_{i}-n_{j}-1\geq 0$, then, for (\ref{eqn: (Cl_P2_3)}) to hold, we must have $\ell(4-n_{i}-n_{j})>0$, hence $(n_{i},n_{j})\in \{(2,1),(1,1)\}$. If $(n_{i},n_{j})=(2,1)$, inequality \eqref{eqn: (Cl_P2_3)} does not hold. If $(n_{i},n_{j})=(1,1)$, substituting in (\ref{eqn: (Cl_P2_3)}) gives $\scale[0.9]{\displaystyle 3+\varepsilon_{p}(\ell)-\sum_{r\neq i,j}n_{r}^{2}\geq 0}$. One checks that this inequality holds if and only if $\ell\leq 6$, $n_{r}=1$ for all $1\leq r\leq \ell$, and $\varepsilon_{p}(6)=1$ when $\ell=6$. In both cases, we can assume without loss of generality that $\mu_{1}\mu_{2}=-1$. As the $\mu_{r}$'s are distinct, at most one eigenvalue of each of the forms $\mu_{1}\mu_{r}^{-1}$, $\mu_{1}^{-1}\mu_{r}$, $\mu_{2}\mu_{r}^{-1}$ and $\mu_{2}^{-1}\mu_{r}$, $3\leq r\leq \ell$, can equal $-1$. This gives an additional $4(\ell-3)$ eigenvalues of $s$ on $V$ that are different than $-1$. Consequently, we have $\dim(V_{s}(-1))\leq 2\ell^{2}-10\ell+20<2\ell^{2}-5\ell+3-\varepsilon_{p}(\ell)$. Thus, we have shown that for $\ell\geq 5$ we have $\dim(V_{s}(-1))<2\ell^{2}-5\ell+3$ for all $s\in T\setminus\ZG(G)$ with $m\geq 2$. In conclusion, we have $\scale[0.9]{\displaystyle \max_{s\in T\setminus\ZG(G)}\dim(V_{s}(\mu))=B}$ and, as $\scale[0.9]{\displaystyle \max_{u\in G_{u}\setminus \{1\}}\dim(V_{u}(1))=2\ell^{2}-3\ell+1-\varepsilon_{p}(\ell)}$ $\scale[0.9]{\displaystyle+\varepsilon_{\ell}(2)\varepsilon_{p}(2)}$, we determine that $\nu_{G}(V)=2\ell-2-\varepsilon_{\ell}(2)$.
\end{proof}

\begin{prop}\label{PropositionClsymm}
Assume $p\neq 2$ and let $V=L_{G}(2\omega_{1})$. Then $\nu_{G}(V)=2\ell$. Moreover, we have:
\begin{enumerate}
\item[\emph{$(1)$}] $\scale[0.9]{\displaystyle \max_{u\in G_{u}\setminus \{1\}}\dim(V_{u}(1))=2\ell^{2}-\ell}$.
\item[\emph{$(2)$}] $\scale[0.9]{\displaystyle \max_{s\in T\setminus\ZG(G)}\dim(V_{s}(\mu))=2\ell^{2}-3\ell+4}$, where the maximum is attained if and only if 
\begin{enumerate}
\item[\emph{$(2.1)$}] $\ell=2$, $\mu=-1$ and, up to conjugation, $s=\diag(d,d,d^{-1},d^{-1})$ with $d^{2}=-1$. 
\item[\emph{$(2.2)$}] $\ell\geq 2$, $\mu=1$ and, up to conjugation, $s=\pm \diag(1,\dots,1,-1,-1,1,\dots,1)$.
\end{enumerate}
\end{enumerate}
\end{prop}

\begin{proof}
We begin with the unipotent elements. We note that $\SW(W)\cong V$ as $kG$-modules, see \cite[Proposition $4.2.2$]{mcninch_1998}. Now, by Lemma \ref{uniprootelems}, we have $\scale[0.9]{\displaystyle \max_{u\in G_{u}\setminus \{1\}}\dim(V_{u}(1))=\max\{\dim(V_{x_{\alpha_{\ell}}(1)}(1)), \dim(V_{x_{\alpha_{1}}(1)}(1))\}}$. By the above and Proposition \ref{PropositionAlsymmetric}, we determine that $\dim(V_{x_{\alpha_{\ell}}(1)}(1))=2\ell^{2}-\ell$. For $\dim(V_{x_{\alpha_{1}}(1)}(1))$, we write $\scale[0.9]{\displaystyle W=W_{1}\oplus W_{2}}$, where $\dim(W_{1})=4$ and $x_{\alpha_{1}}(1)$ acts as $J_{2}^{2}$ on $W_{1}$, and $\dim(W_{2})=2\ell-4$ and $x_{\alpha_{1}}(1)$ acts trivially on $W_{2}$. Then, as $\scale[0.9]{\displaystyle \SW(W)=\SW(W_{1})\oplus [W_{1}\otimes W_{2}]\oplus \SW(W_{2})}$, using \cite[Lemma $3.4$]{liebeck_2012unipotent}, we determine that $\dim(V_{x_{\alpha_{1}}(1)}(1))=2\ell^{2}-3\ell+2$. Therefore $\scale[0.9]{\displaystyle \max_{u\in G_{u}\setminus \{1\}}\dim(V_{u}(1))=2\ell^{2}-\ell}$.

We now focus on the semisimple elements. Let $s\in T\setminus \ZG(G)$ be as in hypothesis $(^{\dagger}H_{s})$. We remark that, as $V\cong \SW (W)$,  see \cite[Proposition $4.2.2$]{mcninch_1998}, the eigenvalues of $s$ on $V$, not necessarily distinct, are:

\begin{equation}\label{enum Cl_P1}
\scale[0.9]{\begin{cases}
\mu_{i}^{2}$ and $\mu_{i}^{-2}$, $1\leq i\leq m$, each with multiplicity at least $\frac{n_{i}(n_{i}+1)}{2};\\
\mu_{i}\mu_{j}$ and $\mu_{i}^{-1}\mu_{j}^{-1}$, $1\leq i<j\leq m$, each with multiplicity at least $n_{i}n_{j};\\
\mu_{i}\mu_{j}^{-1}$ and $\mu_{i}^{-1}\mu_{j}$, $1\leq i<j\leq m$, each with multiplicity at least $n_{i}n_{j};\\
1$ with multiplicity at least $\displaystyle \sum_{i=1}^{m}n_{i}^{2}.
\end{cases}}
\end{equation}

Let $\mu\in k^{*}$ be an eigenvalue of $s$ on $V$. If $\mu\neq \mu^{-1}$, then $\dim(V_{s}(\mu))\leq \dim(V)-\dim(V_{s}(1))-\dim(V_{s}(\mu^{-1}))$, and, since $\scale[0.9]{\dim(V_{s}(1))\geq \displaystyle \sum_{i=1}^{m}n_{i}^{2}\geq \ell}$ and $V$ is self-dual, we determine that $\dim(V_{s}(\mu))\leq \frac{2\ell^{2}+\ell-\ell}{2}=\ell^{2}<2\ell^{2}-3\ell+4$. We thus assume that $\mu=\mu^{-1}$. 

First, let $m=1$, i.e.  $n_{1}=\ell$ and $\mu_{1}^{2}\neq 1$. Now, by \eqref{enum Cl_P1}, we have $\dim(V_{s}(1))=\ell^{2}$ and $\dim(V_{s}(-1))\leq \ell^{2}+\ell$.  If $\ell\geq 3$,  then $\dim(V_{s}(\pm 1))<2\ell^{2}-3\ell+4$ for all $s\in T\setminus \ZG(G)$ with $m=1$. On the other hand, if $\ell=2$, then $\dim(V_{s}(1))<2\ell^{2}-3\ell+4$ and $\dim(V_{s}(-1))\leq 2\ell^{2}-3\ell+4$, where equality holds if and only if $s$ is as in $(2.1)$. We can now assume that $m\geq 2$. To calculate $\dim(V_{s}(1))$ we note that, as $\mu_{i}^{\pm 1}\mu_{j}^{\pm 1}\neq 1$ for all $i<j$, there are at least $\scale[0.9]{\displaystyle 4\sum_{i<j}n_{i}n_{j}}$ eigenvalues of $s$ on $V$ different than $1$, see \eqref{enum Cl_P1}. Therefore $\scale[0.9]{\displaystyle\dim(V_{s}(1))\leq 2\ell^{2}+\ell-4\sum_{i<j}n_{i}n_{j}}$. 

Assume $\dim(V_{s}(1))\geq 2\ell^{2}-3\ell+4$. Then $\scale[0.9]{\displaystyle \ell-1-\sum_{i<j}n_{i}n_{j}\geq 0}$, and, by \eqref{Cl_wed_1mgeq2}, this holds if and only if $m=2$, $n_{2}=1$ and $n_{1}=\ell-1$. Thus, $\dim(V_{s}(1))\leq 2\ell^{2}-3\ell+4$ for all $s\in T\setminus \ZG(G)$ with $m\geq 2$ and equality holds if and only if all eigenvalues of $s$ on $V$ different than $\mu_{1}^{\pm 1}\mu_{2}^{\pm 1}$ are equal to $1$, i.e. if and only if $s$ is as in $(2.2)$.

Lastly, let $\mu=-1$. First note that $\scale[0.9]{\displaystyle \dim(V_{s}(-1))\leq 2\ell^{2} +\ell-\sum_{i=1}^{m}n_{i}^{2}}$. If $\mu_{i}\mu_{j}\neq -1$ for all $i<j$, then, by \eqref{enum Cl_P1}, there are at least $\scale[0.9]{2\displaystyle \sum_{i<j}n_{i}n_{j}}$ additional eigenvalues of $s$ on $V$ different than $-1$. This gives 
\begin{equation}\label{eigen-1case1symminClforDl}
\scale[0.9]{\displaystyle \dim(V_{s}(-1)) \leq 2\ell^{2}+\ell-\sum_{i=1}^{m}n_{i}^{2}-2\sum_{i<j}n_{i}n_{j}=2\ell^{2}+\ell-(\sum_{i=1}^{m}n_{i})^2= \ell^{2}+\ell}.
\end{equation}
If $\ell\geq 3$, then $\dim(V_{s}(-1))<2\ell^{2}-3\ell+4$, while, for $\ell=2$, we have $\dim(V_{s}(-1))\leq 2\ell^{2}-3\ell+4$ where equality holds if and only if all eigenvalues of $s$ on $V$ different than $1$, $\mu_{1}\mu_{2}$ and $\mu_{1}^{-1}\mu_{2}^{-1}$ are equal to $-1$. But then, by \eqref{enum Cl_P1}, we must have $\mu_{1}^{2}=\mu_{1}\mu_{2}^{-1}$, a contradiction.

We thus assume that there exist $i<j$ such that $\mu_{i}\mu_{j}=-1$. In this case, we also have $\mu_{i}^{-1}\mu_{j}^{-1}=-1$. Furthermore, since the $\mu_{i}$'s are distinct, it follows that:
\begin{equation*}
\scale[0.9]{\begin{cases}
\mu_{i}^{2}\neq -1, \ \mu_{i}^{-2}\neq -1 \text{ and }\mu_{j}^{2}\neq -1, \ \mu_{j}^{-2}\neq -1;\\
\mu_{i}\mu_{r}\neq -1$  and $\mu_{i}^{-1}\mu_{r}^{-1}\neq -1$, $\text{ where }i<r\leq m, \ r\neq j; \text{ and }\mu_{r}\mu_{i}\neq -1$ and $\mu_{r}^{-1}\mu_{i}^{-1}\neq -1$,$\text{ where }1\leq r<i;\\
\mu_{j}\mu_{r}\neq -1$ and $\mu_{j}^{-1}\mu_{r}^{-1}\neq -1$, $\text{ where }j<r\leq m; \text{ and }\mu_{r}\mu_{j}\neq -1$ and $\mu_{r}^{-1}\mu_{j}^{-1}\neq -1, \text{ where }1\leq r<j, \ r\neq i.\\
\end{cases}}
\end{equation*}
By \eqref{enum Cl_P1}, these are at least $n_{i}(n_{i}+1)+n_{j}(n_{j}+1)+2(n_{i}+n_{j})(\ell-n_{i}-n_{j})$ additional eigenvalues of $s$ on $V$ different than $-1$. This gives 
\begin{equation}\label{eigen-1dimspacesymminClforDl}
\scale[0.9]{\displaystyle \dim(V_{s}(-1))\leq 2\ell^{2}+\ell-\sum_{r=1}^{m}n_{r}^{2}-n_{i}(n_{i}+1)-n_{j}(n_{j}+1)-2(n_{i}+n_{j})(\ell-n_{i}-n_{j})}.
\end{equation}
Assume $\dim(V_{s}(-1))\geq 2\ell^{2}-3\ell+4$. Then:
\begin{equation}\label{eqn: (Cl_P1_4)}
\scale[0.9]{\displaystyle\ell(4-n_{i}-n_{j})- \sum_{r\neq i,j}n_{r}^{2}- (n_{i}-n_{j})^{2}-(n_{i}+n_{j})(\ell+1-n_{i}-n_{j})-4\geq 0}.
\end{equation}
We remark that $\scale[0.9]{\big[-\displaystyle \sum_{r\neq i,j}n_{r}^{2}- (n_{i}-n_{j})^{2}-(n_{i}+n_{j})(\ell+1-n_{i}-n_{j})-4\big]<0}$, as $\ell+1>n_{i}+n_{j}$. Thus, by \eqref{eqn: (Cl_P1_4)}, we must have $\ell(4-n_{i}-n_{j})>0$, and so $(n_{i},n_{j})\in \{(2,1),(1,1)\}$. In both cases, inequality \eqref{eqn: (Cl_P1_4)} does not hold. Thus, $\dim(V_{s}(-1))<2\ell^{2}-3\ell+4$ for all $s\in T\setminus\ZG(G)$ with $m\geq 2$, and so $\scale[0.9]{\displaystyle \max_{s\in T\setminus\ZG(G)}\dim(V_{s}(\mu))=2\ell^{2}-3\ell+4}$. As $\scale[0.9]{\displaystyle \max_{u\in G_{u}\setminus \{1\}}\dim(V_{u}(1))=2\ell^{2}-\ell}$, we determine that $\nu_{G}(V)=2\ell$. 
\end{proof}

\begin{prop}\label{C3om3}
Let $\ell=3$ and $V=L_{G}(\omega_{3})$. Then $\nu_{G}(V)=4-2\varepsilon_{p}(2)$. Moreover, we have $\scale[0.9]{\displaystyle \max_{u\in G_{u}\setminus \{1\}}\dim(V_{u}(1))}$ $=9-3\varepsilon_{p}(2)$ and $\scale[0.9]{\displaystyle \max_{s\in T\setminus\ZG(G)}\dim(V_{s}(\mu))=10-6\varepsilon_{p}(2)}$.
\end{prop}

\begin{proof}
Set $\lambda=\omega_{3}$ and $L=L_{1}$. Then, by Lemma \ref{weightlevelCl}, we have $e_{1}(\lambda)=2$, therefore $\displaystyle V\mid_{[L,L]}=V^{0} \oplus V^{1}\oplus V^{2}$. Now, by \cite[Proposition]{Smith_82} and Lemma \ref{dualitylemma}, we have $V^{0}\cong L_{L}(\omega_{3})$ and $V^{2}\cong L_{L}(\omega_{3})$. If $p=2$, then $V^{1}=\{0\}$. However, if $p\neq 2$, then the weight $\displaystyle (\lambda-\alpha_{1}-\alpha_{2}-\alpha_{3})\mid_{T_{1}}=\omega_{2}$ admits a maximal vector in $V^{1}$, thus $V^{1}$ has a composition factor isomorphic to $L_{L}(\omega_{2})$. By dimensional considerations, we deduce that 
\begin{equation}\label{DecompVC3om3}
V\mid_{[L,L]}\cong L_{L}(\omega_{3})\oplus L_{L}(\omega_{2})^{1-\varepsilon_{p}(2)} \oplus L_{L}(\omega_{3}).
\end{equation}

Let $s\in T\setminus \ZG(G)$. If $\dim(V^{i}_{s}(\mu))=\dim(V^{i})$ for some eigenvalue $\mu$ of $s$ on $V$, where $p=2$ and $i=0,2$, or $p\neq 2$ and $0\leq i\leq 2$, then $s\in \ZG(L)^{\circ}\setminus \ZG(G)$. In this case, as $s$ acts on each $V^{i}$ as scalar multiplication by $c^{1-i}$ and $c\neq 1$, we determine that $\dim(V_{s}(\mu))\leq 10-6\varepsilon_{p}(2)$ for all eigenvalues $\mu$ of $s$ on $V$, where equality holds if and only if $p=2$ and $\mu=c^{\pm 1}$, or $p\neq 2$, $c=-1$ and $\mu=-1$. We thus assume that $\dim(V^{i}_{s}(\mu))<\dim(V^{i})$ for all eigenvalues $\mu$ of $s$ on $V$ and all $0\leq i\leq 2$. We write $s=z\cdot h$, where $z\in \ZG(L)^{\circ}$ and $h\in [L,L]$, and, by \eqref{DecompVC3om3} and Propositions \ref{PropositionClnatural} and \ref{PropositionClwedge}, we determine that $\dim(V_{s}(\mu))\leq 2\dim(L_{L}(\omega_{3})_{h}(\mu_{h}))+(1-\varepsilon_{p}(2))\dim(L_{L}(\omega_{2})_{h}(\mu_{h}))=10-6\varepsilon_{p}(2)$ for all eigenvalues $\mu$ of $s$ on $V$. Therefore, $\scale[0.9]{\displaystyle \max_{s\in T\setminus\ZG(G)}\dim(V_{s}(\mu))=10-6\varepsilon_{p}(2)}$. 

For the unipotent elements, we have $\scale[0.9]{\displaystyle \max_{u\in G_{u}\setminus \{1\}}\dim(V_{u}(1))=\max\{\dim(V_{x_{\alpha_{3}}(1)}(1)), \dim(V_{x_{\alpha_{2}}(1)}(1))\}}$, see Lemma \ref{uniprootelems}. Now, by \eqref{DecompVC3om3} and Propositions \ref{PropositionClnatural} and \ref{PropositionClwedge}, we have $\dim(V_{x_{\alpha_{i}}(1)}(1))\leq 9-3\varepsilon_{p}(2)$, where equality holds if $p\neq 2$ and $i=3$, or if $p=2$ and $i=2$. Therefore, $\scale[0.9]{\displaystyle \max_{u\in G_{u}\setminus \{1\}}\dim(V_{u}(1))}$ $=9-3\varepsilon_{p}(2)$, and, as $\scale[0.9]{\displaystyle \max_{s\in T\setminus\ZG(G)}\dim(V_{s}(\mu))=10-6\varepsilon_{p}(2)}$, we have $\nu_{G}(V)=4-2\varepsilon_{p}(2)$.  
\end{proof}

\begin{prop}\label{C4om3}
Let $\ell=4$ and $V=L_{G}(\omega_{3})$. Then $\nu_{G}(V)=14-\varepsilon_{p}(3)$. Moreover, we have $\scale[0.9]{\displaystyle \max_{u\in G_{u}\setminus \{1\}}\dim(V_{u}(1))}$ $=34-7\varepsilon_{p}(3)$ and $\scale[0.9]{\displaystyle \max_{s\in T\setminus\ZG(G)}\dim(V_{s}(\mu))=28-2\varepsilon_{p}(3)-8\varepsilon_{p}(2)}$.
\end{prop}

\begin{proof}
Set $\lambda=\omega_{3}$ and $L=L_{1}$. Then, by Lemma \ref{weightlevelCl}, we have $e_{1}(\lambda)=2$, therefore $\displaystyle V\mid_{[L,L]}=V^{0} \oplus V^{1}\oplus V^{2}$. Now, by \cite[Proposition]{Smith_82} and Lemma \ref{dualitylemma}, we have $V^{0}\cong L_{L}(\omega_{3})$ and $V^{2}\cong L_{L}(\omega_{3})$. In $V^{1}$, the weight $\displaystyle (\lambda-\alpha_{1}-\alpha_{2}-\alpha_{3})\mid_{T_{1}}=\omega_{4}$ admits a maximal vector, thus $V^{1}$ has a composition factor isomorphic to $L_{L}(\omega_{4})$. Further, the weight $\displaystyle (\lambda-\alpha_{1}-\alpha_{2}-2\alpha_{3}-\alpha_{4})\mid_{T_{1}}=\omega_{2}$ occurs with multiplicity $2-\varepsilon_{p}(3)$ and is a sub-dominant weigh in the composition factor of $V^{1}$ isomorphic to $L_{L}(\omega_{4})$, where it has multiplicity $1-\varepsilon_{p}(2)$. By dimensional considerations, we determine that $V^{1}$ has exactly $2-\varepsilon_{p}(3)+\varepsilon_{p}(2)$ composition factors: one isomorphic to $L_{L}(\omega_{4})$ and $1-\varepsilon_{p}(3)+\varepsilon_{p}(2)$ isomorphic to $L_{L}(\omega_{2})$. 

We begin with the semisimple elements. Let $s\in T\setminus \ZG(G)$. If $\dim(V^{i}_{s}(\mu))=\dim(V^{i})$ for some eigenvalue $\mu$ of $s$ on $V$, where $0\leq i\leq 2$, then $s\in \ZG(L)^{\circ}\setminus \ZG(G)$. In this case, as $s$ acts on each $V^{i}$ as scalar multiplication by $c^{1-i}$ and $c\neq 1$, we determine that $\dim(V_{s}(\mu))\leq 28-2\varepsilon_{p}(3)-8\varepsilon_{p}(2)$, where equality holds if and only if $p\neq 2$, $c=-1$ and $\mu=-1$, or $p=2$ and $\mu=1$. We thus assume that $\dim(V^{i}_{s}(\mu))<\dim(V^{i})$ for all eigenvalues $\mu$ of $s$ on $V$ and all $0\leq i\leq 2$. We write $s=z\cdot h$, where $z\in \ZG(L)^{\circ}$ and $h\in [L,L]$, and, by the structure of $V\mid_{[L,L]}$ and Propositions \ref{PropositionClnatural}, \ref{PropositionClwedge} and \ref{C3om3}, we have $\dim(V_{s}(\mu))\leq 2\dim((L_{L}(\omega_{3}))_{h}(\mu_{h}))+\dim((L_{L}(\omega_{4}))_{h}(\mu_{h}))+(1-\varepsilon_{p}(3)+\varepsilon_{p}(2))\dim((L_{L}(\omega_{2}))_{h}(\mu_{h}))\leq 30-4\varepsilon_{p}(3)-2\varepsilon_{p}(2)$ for all eigenvalues $\mu$ of $s$ on $V$. However, we will show that $\scale[0.9]{\displaystyle \max_{s\in T\setminus\ZG(G)}\dim(V_{s}(\mu))=28-2\varepsilon_{p}(3)-8\varepsilon_{p}(2)}$. 

First, for $p\neq 2,3$ assume there exist $(s,\mu)\in T\setminus \ZG(G)\times k^{*}$ such that $\dim(V_{s}(\mu))>28$. Then, by the structure of $V\mid_{[L,L]}$ and Propositions \ref{PropositionClnatural}, \ref{PropositionClwedge} and \ref{C3om3}, it follows that $\dim(L_{L}(\omega_{3})_{h}(\mu_{h}))=8$ for some eigenvalue $\mu_{h}$ of $h$ on $L_{L}(\omega_{3})$. Thus, up to conjugation, we have that $h\in \{ \diag(1,1,1,-1,-1,1,1,1), $ $\diag(1,-1,-1,1,1,-1,-1,1), \diag(1,d,d,d,d^{-1},d^{-1},d^{-1})\text{ with }d^{2}\neq 1\}$, see Proposition \ref{PropositionClwedge}. However, using the weight structures of $L_{L}(\omega_{2})$ and $L_{L}(\omega_{4})$, one shows that for $h=\diag(1,1,1,-1,-1,1,1,1)$ we have $\dim(V^{1}_{h}(-1))=12$ and $\dim(V^{1}_{h}(1))=8$, hence $\dim(V_{s}(\mu))\leq 28$ for all eigenvalues $\mu$ of $s$ on $V$; for $h=\diag(1,-1,-1,1,1,-1,-1,1)$ we have $\dim(V^{1}_{h}(1))=12$ and $\dim(V^{1}_{h}(-1))=8$, hence $\dim(V_{s}(\mu))\leq 28$ for all eigenvalues $\mu$ of $s$ on $V$; and for $h=\diag(1,d,d,d,d^{-1},d^{-1},d^{-1},1)$ with $d^{2}\neq 1$ we have $\dim(V^{1}_{h}(\mu^{1}_{h}))\leq 11$ for all eigenvalues $\mu^{1}_{h}$ of $h$ on $V^{1}$, hence $\dim(V_{s}(\mu))\leq 28$ for all eigenvalues $\mu$ of $s$ on $V$. Secondly, when $p=3$, we saw earlier that for $s\in \ZG(L)^{\circ}\setminus \ZG(G)$ with $c=-1$, we have $\dim(V_{s}(-1))=26$. We thus assume that $p=2$ and we will show that $\dim(V_{s}(\mu))\leq 20$ for all eigenvalues $\mu$ of $s$ on $V$. Now, in view of Proposition \ref{PropositionClnatural}, assume $\dim((L_{L}(\omega_{2}))_{h}(\mu_{h}))=4$ for some eigenvalue $\mu_{h}$. Then, by the same result, we have that, up to conjugation, $h=\diag(1,1,1,d,d^{-1},1,1,1)$ with $d\neq 1$. Further, by the structure of $V\mid_{[L,L]}$, \eqref{enum Cl_P2} and the weight structure of $L_{L}(\omega_{4})$, we determine that $\dim(V_{s}(\mu))\leq 20$ for all eigenvalues $\mu$ of $s$ on $V$. Thus, we assume $\dim((L_{L}(\omega_{2}))_{h}(\mu_{h}))=3$. One checks that, in this case, up to conjugation, we have $h=\diag(1,d,d,d,d^{-1},d^{-1},d^{-1},1)$ with $d\neq 1$. Then, by the structure of $V^{1}$, we see that the distinct eigenvalues of $s$ on $V^{1}$, as $z$ acts trivially on $V^{1}$, are $d^{\pm 1}$ each with multiplicity $9$; and $d^{\pm 3}$ each with multiplicity $1$. Further, the distinct eigenvalues of $s$ on $V^{0}$ are $c$ with multiplicity $8$; and $cd^{\pm 2}$ each with multiplicity $3$, while the distinct eigenvalues of $s$ on $V^{2}$ are $c^{-1}$ with multiplicity $8$; and $c^{-1}d^{\pm 2}$ each with multiplicity $3$. Keeping in mind that $d\neq 1$, one shows that $\dim(V_{s}(\mu))\leq 20$. We now assume $\dim((L_{L}(\omega_{2}))_{h}(\mu_{h}))=2$. In this case, one checks that, up to conjugation, $h=\diag(1,d,d,e,e^{-1},d^{-1},d^{-1})$ with $d\neq 1,e, e^{-1}$. Again, by the structure of $V\mid_{[L,L]}$, \eqref{enum Cl_P2} and the weight structure of $L_{L}(\omega_{4})$, we determine that $\dim(V_{s}(\mu))\leq 16$. Lastly, if $\dim((L_{L}(\omega_{2}))_{h}(\mu_{h}))=1$, then, by Propositions \ref{PropositionClwedge} and \ref{C3om3}, we determine that $\dim(V_{s}(\mu))\leq 18$ for all eigenvalues $\mu$. Having considered all cases, keeping in mind that for $s\in \ZG(L)^{\circ}\setminus \ZG(G)$ we have $\dim(V_{s}(1))=20$, we conclude that $\scale[0.9]{\displaystyle \max_{s\in T\setminus\ZG(G)}\dim(V_{s}(\mu))=20}$. Therefore, $\scale[0.9]{\displaystyle \max_{s\in T\setminus\ZG(G)}\dim(V_{s}(\mu))=28-2\varepsilon_{p}(3)-8\varepsilon_{p}(2)}$.

We now focus on the unipotent elements. We have that $\scale[0.9]{\displaystyle \max_{u\in G_{u}\setminus \{1\}}\dim(V_{u}(1))=\max_{i=3,4}\dim(V_{x_{\alpha_{i}}(1)}(1))}$, by Lemma \ref{uniprootelems}. Further, we note that $\wedge^{3}(W)\cong V\oplus L_{G}(\omega_{1})$, if $p\neq 3$, and $\wedge^{3}(W)\cong L_{G}(\omega_{1})\mid V\mid L_{G}(\omega_{1})$ if $p=3$, see \cite[Lemma $4.8.2$]{mcninch_1998}. For $x_{\alpha_{\ell}}(1)$, by Proposition \ref{PropositionAlwedgecube}, we have $\dim((\wedge^{3}(W))_{x_{\alpha_{4}}(1)}(1))=41$, and so, by the structure of $\wedge^{3}(W)$, Proposition \ref{PropositionClnatural} and Lemma \ref{LemmaonfiltrationofV}, we determine that $\dim(V_{x_{\alpha_{4}}(1)}(1))\geq 34-7\varepsilon_{p}(3)$, where equality holds for $p\neq 3$. Assume $p=3$. Then, by the structure of $V\mid_{[L,L]}$ and Propositions \ref{PropositionClnatural}, \ref{PropositionClwedge} and \ref{C3om3}, we have $\dim(V_{x_{\alpha_{4}}(1)}(1))\leq 27$, and so $\dim(V_{x_{\alpha_{4}}(1)}(1))=34-7\varepsilon_{p}(3)$. Similarly, one shows that $\dim(V_{x_{\alpha_{3}}(1)}(1))=22-6\varepsilon_{p}(3)$, thus $\scale[0.9]{\displaystyle \max_{u\in G_{u}\setminus \{1\}}\dim(V_{u}(1))}=34-7\varepsilon_{p}(3)$, and so $\nu_{G}(V)=14-\varepsilon_{p}(3)$.
\end{proof}

\begin{prop}\label{PropositionClwedgecube}
Let $\ell\geq 5$ and $V=L_{G}(\omega_{3})$. Then $\nu_{G}(V)=2\ell^{2}-5\ell+2-\varepsilon_{p}(\ell-1)$, $\scale[0.9]{\displaystyle \max_{u\in G_{u}\setminus \{1\}}\dim(V_{u}(1))=}$ $\binom{2\ell-1}{3}\scale[0.9]{-1-(2\ell-1)\varepsilon_{p}(\ell-1)}$ and $\scale[0.9]{\displaystyle \max_{s\in T\setminus\ZG(G)}\dim(V_{s}(\mu))}\leq \binom{2\ell-2}{3}\scale[0.9]{+10-2(\ell-1)\varepsilon_{p}(\ell-1)+4\varepsilon_{p}(3)-10\varepsilon_{p}(2)}$.
\end{prop}

\begin{proof}
We begin with the unipotent elements. We note that if $\varepsilon_{p}(\ell-1)=0$, we have $\wedge^{3}(W)\cong V\oplus L_{G}(\omega_{1})$, while, if $\varepsilon_{p}(\ell-1)=1$, then $\wedge^{3}(W)\cong L_{G}(\omega_{1})\mid V\mid L_{G}(\omega_{1})$, as $kG$-modules, see \cite[Lemma $4.8.2$]{mcninch_1998}. Further, by Lemma \ref{uniprootelems}, we have $\scale[0.9]{\displaystyle \max_{u\in G_{u}\setminus \{1\}}\dim(V_{u}(1))=\max\{\dim(V_{x_{\alpha_{\ell}}(1)}(1)), \dim(V_{x_{\alpha_{1}}(1)}(1))\}}$. Now, for $x_{\alpha_{\ell}}(1)$, by Proposition \ref{PropositionAlwedgecube}, we have $\dim((\wedge^{3}(W))_{x_{\alpha_{\ell}}(1)}(1))=\binom{2\ell-1}{3}+2\ell-2$. Thus, by the structure of $\wedge^{3}(W)$, Proposition \ref{PropositionClnatural} and Lemma \ref{LemmaonfiltrationofV}, we determine that $\dim(V_{x_{\alpha_{\ell}}(1)}(1))\geq \binom{2\ell-1}{3}-1-(2\ell-1)\varepsilon_{p}(\ell-1)$, where equality holds for $\varepsilon_{p}(\ell-1)=0$. For $x_{\alpha_{1}}(1)$, we write $\scale[0.9]{\displaystyle W=W_{1}\oplus W_{2}}$, where $\dim(W_{1})=4$ and $x_{\alpha_{1}}(1)$ acts as $J_{2}^{2}$ on $W_{1}$, and $\dim(W_{2})=2\ell-4$ and $x_{\alpha_{1}}(1)$ acts trivially on $W_{2}$. We have that $\scale[0.9]{\displaystyle \wedge^{3}(W)=\wedge^{3}(W_{1})\oplus [\wedge^{2}(W_{1})\otimes W_{2}]\oplus [W_{1}\otimes \wedge^{2}(W_{2})]\oplus \wedge^{3}(W_{2})}$, and so $\dim((\wedge^{3}(W))_{x_{\alpha_{1}}(1)}(1))=\binom{2\ell-1}{3}-2\ell^{2}+11\ell-13$. Similar to $x_{\alpha_{\ell}}(1)$, we use the structure of $\wedge^{3}(W)$, Proposition \ref{PropositionClnatural} and Lemma \ref{LemmaonfiltrationofV}, to determine that $\dim(V_{x_{\alpha_{1}}(1)}(1))\geq \binom{2\ell-1}{3}-2\ell^{2}+9\ell-11-(2\ell-2)\varepsilon_{p}(\ell-1)$, where equality holds for $\varepsilon_{p}(\ell-1)=0$. Therefore, $\scale[0.9]{\displaystyle \max_{u\in G_{u}\setminus \{1\}}\dim(V_{u}(1))}\geq \binom{2\ell-1}{3}-1-(2\ell-1)\varepsilon_{p}(\ell-1)$, where equality holds for $\varepsilon_{p}(\ell-1)=0$. In what follows, we show that equality also holds for $\varepsilon_{p}(\ell-1)=1$.

Let $\lambda=\omega_{3}$ and let $L=L_{1}$. We have $e_{1}(\lambda)=2$, therefore $\displaystyle V\mid_{[L,L]}=V^{0}\oplus V^{1}\oplus V^{2}$. Now, by \cite[Proposition]{Smith_82} and Lemma \ref{dualitylemma}, we have $V^{0}\cong L_{L}(\omega_{3})$ and $V^{2}\cong L_{L}(\omega_{3})$. The weight $\displaystyle (\lambda-\alpha_{1}-\alpha_{2}-\alpha_{3})\mid_{T_{1}}=\omega_{4}$ admits a maximal vector in $V^{1}$, thus $V^{1}$ has a composition factor isomorphic to $L_{L}(\omega_{4})$. Let $\varepsilon_{p}(\ell-1)=1$. Then, by dimensional considerations, we determine that:
\begin{equation}\label{DecompCl_omega3pmidell-1}
V\mid_{[L,L]}\cong L_{L}(\omega_{3})\oplus L_{L}(\omega_{4})\oplus L_{L}(\omega_{3}).
\end{equation}
This gives $\dim(V_{x_{\alpha_{\ell}}(1)}(1))=2\dim((L_{L}(\omega_{3}))_{x_{\alpha_{\ell}}(1)}(1))+\dim((L_{L}(\omega_{4}))_{x_{\alpha_{\ell}}(1)}(1))$, where, by the proof of Proposition \ref{PropositionClwedge}, we have $\dim((L_{L}(\omega_{3}))_{x_{\alpha_{\ell}}(1)}(1))=2\ell^{2}-7\ell+5$, and, as $\varepsilon_{p}(\ell-1)=1$, then $\varepsilon_{p}(\ell-2)=0$, and so, by the arguments of the previous paragraph, we have $\dim((L_{L}(\omega_{4}))_{x_{\alpha_{\ell}}(1)}(1))=\binom{2\ell-3}{3}-1$. This gives $\dim(V_{x_{\alpha_{\ell}}(1)}(1))=\binom{2\ell-1}{3}-2\ell$, and, consequently $\scale[0.9]{\displaystyle \max_{u\in G_{u}\setminus \{1\}}\dim(V_{u}(1))}=\binom{2\ell-1}{3}-2\ell$ for $\varepsilon_{p}(\ell-1)=1$.

We now focus on the semisimple elements. Since we have already treated $\ell=4$ in Proposition \ref{C4om3}, we assume $\ell\geq 5$. Moreover, since we have already determined the decomposition of $V\mid_{[L,L]}$ when $\varepsilon_{p}(\ell-1)=1$, we assume that $\varepsilon_{p}(\ell-1)=0$. We see that the weight $\displaystyle (\lambda-\alpha_{1}-\alpha_{2}-2\alpha_{3}-\cdots -2\alpha_{\ell-1}-\alpha_{\ell})\mid_{T_{1}}=\omega_{2}$ occurs with multiplicity $\ell-2$ in $V^{1}$ and is a sub-dominant weight in the composition factor of $V^{1}$ isomorphic to $L_{L}(\omega_{4})$, in which it has multiplicity $\ell-3-\varepsilon_{p}(\ell-2)$. By dimensional considerations, we conclude that $V^{1}$ has $2+\varepsilon_{p}(\ell-2)$ composition factors: one isomorphic to $L_{L}(\omega_{4})$ and $1+\varepsilon_{p}(\ell-2)$ isomorphic to $L_{L}(\omega_{2})$.

We return to the general case, where no assumption is made on $p$. If $\dim(V^{i}_{s}(\mu))=\dim(V^{i})$ for some eigenvalue $\mu$ of $s$ on $V$, where $0\leq i\leq 2$, then $s\in \ZG(L)^{\circ}\setminus \ZG(G)$ and acts on $V^{i}$ as scalar multiplication by $c^{1-i}$. As $c\neq 1$, we determine that $\dim(V_{s}(\mu))\leq \binom{2\ell-2}{3}-(2\ell-2)\varepsilon_{p}(\ell-1)$ for all eigenvalues $\mu$ of $s$ on $V$. We thus assume that $\dim(V^{i}_{s}(\mu))<\dim(V^{i})$ for all eigenvalues $\mu$ of $s$ on $V$ and for all $0\leq i\leq 2$. We write $s=z\cdot h$, where $z\in \ZG(L)^{\circ}$ and $h\in [L,L]$, and, by the structure of $V\mid_{[L,L]}$, we see that $\dim(V_{s}(\mu))\leq 2\dim((L_{L}(\omega_{3}))_{h}(\mu_{h}))+\dim((L_{L}(\omega_{4}))_{h}(\mu_{h}))+(1-\varepsilon_{p}(\ell-1)+\varepsilon_{p}(\ell-2))\dim((L_{L}(\omega_{2}))_{h}(\mu_{h}))$. For $\ell=5$, by Propositions \ref{PropositionClnatural}, \ref{PropositionClwedge} and \ref{C4om3}, we get $\dim(V_{s}(\mu))\leq 66+4\varepsilon_{p}(3)-18\varepsilon_{p}(2)$ for all eigenvalues $\mu$ of $s$ on $V$, therefore $\scale[0.9]{\displaystyle \max_{s\in T\setminus\ZG(G)}}\dim(V_{s}(\mu))\leq 66+4\varepsilon_{p}(3)-18\varepsilon_{p}(2)$. We now assume that $\ell\geq 6$. Recursively and using the result for $\ell=5$, we determine that $\dim(V_{s}(\mu))\leq 4(\ell-1)^{2}-8(\ell-1)+4-2(\ell-1)\varepsilon_{p}(\ell-1)+(2(\ell-1)-2)\varepsilon_{p}(\ell-2)+\dim((L_{L}(\omega_{4}))_{h}(\mu_{h}))\leq \scale[0.9]{\displaystyle4\sum_{j=5}^{\ell-1}j^{2}-}$ $\scale[0.9]{\displaystyle 8\sum_{j=5}^{\ell-1}j +\sum_{j=5}^{\ell-1}4-\sum_{j=5}^{\ell-1}2j\varepsilon_{p}(j)+\sum_{j=5}^{\ell-1}2(j-1)\varepsilon_{p}(j-1)}+66+4\varepsilon_{p}(3)-18\varepsilon_{p}(2)$ $=\binom{2\ell-2}{3}+10-2(\ell-1)\varepsilon_{p}(\ell-1)+4\varepsilon_{p}(3)-10\varepsilon_{p}(2)$. Thus, we have proven that $\dim(V_{s}(\mu))\leq \binom{2\ell-2}{3}+10-2(\ell-1)\varepsilon_{p}(\ell-1)+4\varepsilon_{p}(3)-10\varepsilon_{p}(2)$ for all $(s,\mu)\in T\setminus \ZG(G)\times k^{*}$. Further, as $\displaystyle \max_{u\in G_{u}\setminus \{1\}}\dim(V_{u}(1))$ $=\binom{2\ell-1}{3}-1-\varepsilon_{p}(\ell-1)(2\ell-1)$, we conclude that $\nu_{G}(V)=2\ell^{2}-5\ell+2-\varepsilon_{p}(\ell-1)$.
\end{proof}

\begin{prop}\label{PropositionClsymmcube}
Assume $p\neq 2,3$ and let $V=L_{G}(3\omega_{1})$. Then $\nu_{G}(V)=2\ell^{2}+\ell$.  Moreover, we have $\scale[0.9]{\displaystyle \max_{u\in G_{u}\setminus \{1\}}\dim(V_{u}(1))=\binom{2\ell+1}{3}}$ and $\scale[0.9]{\displaystyle \max_{s\in T\setminus\ZG(G)}\dim(V_{s}(\mu))=\binom{2\ell}{3}+3(2\ell-2)}$.
\end{prop}

\begin{proof}
We start with the unipotent elements. We note that $\SWT(W)\cong V$, see \cite[Proposition $4.2.2$]{mcninch_1998}, and so, by Lemma \ref{uniprootelems}, we have $\displaystyle \max_{u\in G_{u}\setminus \{1\}}\dim(V_{u}(1))=\max\{\dim((\SWT(W))_{x_{\alpha_{\ell}}(1)}(1)), \dim((\SWT(W))_{x_{\alpha_{1}}(1)}(1))\}$. Now, by Proposition \ref{PropositionAlsymmcube}, we have $\dim((\SWT(W))_{x_{\alpha_{\ell}}(1)}(1))=\binom{2\ell+1}{3}$. For $x_{\alpha_{1}}(1)$, we write $\scale[0.9]{\displaystyle W=W_{1}\oplus W_{2}}$, where $\dim(W_{1})=4$ and $x_{\alpha_{1}}(1)$ acts as $J_{2}^{2}$ on $W_{1}$, and $\dim(W_{2})=2\ell-4$ and $x_{\alpha_{1}}(1)$ acts trivially on $W_{2}$. We have that $\scale[0.9]{\displaystyle \SWT(W)=\SWT(W_{1})\oplus [\SW(W_{1})\otimes W_{2}]\oplus [W_{1}\otimes \SW(W_{2})]\oplus \SWT(W_{2})}$, and so $\dim((\SWT(W))_{x_{\alpha_{1}}(1)}(1))=\binom{2\ell+1}{3}-2\ell^{2}+3\ell-2$. Therefore, $\scale[0.9]{\displaystyle \max_{u\in G_{u}\setminus \{1\}}\dim(V_{u}(1))}=\binom{2\ell+1}{3}$.

We now focus on the semisimple elements. Set $\lambda=3\omega_{1}$ and $L=L_{1}$. We have $e_{1}(\lambda)=6$, therefore $\displaystyle V\mid_{[L,L]}=V^{0}\oplus \cdots \oplus V^{6}$. Now, by \cite[Proposition]{Smith_82} and Lemma \ref{dualitylemma}, we have $V^{0}\cong L_{L}(0)$ and $V^{6}\cong L_{L}(0)$. As the weight $\displaystyle (\lambda-\alpha_{1})\mid_{T_{1}}=\omega_{2}$ admits a maximal vector in $V^{1}$, it follows that $V^{1}$ has a composition factor isomorphic to $L_{L}(\omega_{2})$. Similarly, the weight $\displaystyle (\lambda-2\alpha_{1})\mid_{T_{1}}=2\omega_{2}$ admits a maximal vector in $V^{2}$, thus $V^{2}$ has a composition factor isomorphic to $L_{L}(2\omega_{2})$. Moreover, the weight $\displaystyle (\lambda-2\alpha_{1}-\cdots-2\alpha_{\ell-1}-\alpha_{\ell})\mid_{T_{1}}=0$ occurs with multiplicity $\ell$ and is a sub-dominant weight in the composition factor of $V^{2}$ isomorphic to $L_{L}(2\omega_{2})$ in which it has multiplicity $\ell-1$. Now, the weight $\displaystyle (\lambda-3\alpha_{1})\mid_{T_{1}}=3\omega_{2}$ admits a maximal vector in $V^{3}$, thus $V^{3}$ has a composition factor isomorphic to $L_{L}(3\omega_{2})$. Moreover, the weight $\displaystyle (\lambda-3\alpha_{1}-2\alpha_{2}-\cdots -2\alpha_{\ell-1}-\alpha_{\ell})\mid_{T_{1}}=\omega_{2}$ occurs with multiplicity $\ell$ and is a sub-dominant weight in the composition factor of $V^{3}$ isomorphic to $L_{L}(2\omega_{2})$ in which it has multiplicity $\ell-1$.  Now, as $\dim(V^{3})\leq \binom{2\ell}{3}+2\ell-2$, we deduce that $V^{3}$ has exactly $2$ composition factors: one isomorphic to $L_{L}(3\omega_{2})$ and one to $L_{L}(\omega_{2})$, thus $V^{3}\cong L_{L}(3\omega_{2})\oplus L_{L}(\omega_{2})$, see \cite[II.$2.14$]{Jantzen_2007representations}. Further, $V^{2}$ and $V^{4}$ each has exactly $2$ composition factors: one isomorphic to $L_{L}(2\omega_{2})$ and one to $L_{L}(0)$, therefore $V^{2}\cong L_{L}(2\omega_{2})\oplus L_{L}(0)$ and $V^{4}\cong L_{L}(2\omega_{2})\oplus L_{L}(0)$, by \cite[II.$2.14$]{Jantzen_2007representations}. Lastly, we have $V^{1}\cong L_{L}(\omega_{2})$ and $V^{5}\cong L_{L}(\omega_{2})$, and so
\begin{equation}\label{DecompCl_3omega1}
V\mid_{[L,L]}\cong L_{L}(0)^{4}\oplus L_{L}(\omega_{2})^{3} \oplus  L_{L}(2\omega_{2})^{2} \oplus L_{L}(3\omega_{2}).
\end{equation}

Let $s\in T\setminus \ZG(G)$. If $\dim(V^{i}_{s}(\mu))=\dim(V^{i})$ for some eigenvalue $\mu$ of $s$ on $V$, where $0\leq i\leq 6$, then $s\in \ZG(L)^{\circ}\setminus \ZG(G)$ and acts on $V^{i}$ as scalar multiplication by $c^{3-i}$. As $c\neq 1$, we determine that $\dim(V_{s}(\mu))\leq \binom{2\ell}{3}+3(2\ell-2)$, where equality holds if and only if $\ell=2$, $c=-1$ and $\mu=\pm 1$, or $\ell\geq 3$, $c=-1$ and $\mu=1$. We thus assume that $\dim(V^{i}_{s}(\mu))<\dim(V^{i})$ for all eigenvalues $\mu$ of $s$ on $V$ and for all $0\leq i\leq 6$. First let $\ell=2$. Using the structure of each $V^{i}$, one shows that the eigenvalues of $s$ on $V$ are $c^{\pm 3}$ each with multiplicity at least $1$; $c^{\pm 2}d^{\pm 1}$ each with multiplicity at least $1$; $c^{\pm 1}d^{\pm 2}$ each with multiplicity at least $1$; $c^{\pm 1}$ each with multiplicity at least $2$; and $d^{\pm 1}$ each with multiplicity at least $2$, where $d^{2}\neq 1$ and $c\in k^{*}$. One determines that $\dim(V_{s}(\mu))\leq 10$ for all eigenvalues $\mu$ of $s$ on $V$. We now let $\ell\geq 3$ and write $s=z\cdot h$, where $z\in \ZG(L)^{\circ}$ and $h\in [L,L]$. Using \eqref{DecompCl_3omega1} and Propositions \ref{PropositionClnatural} and \ref{PropositionClsymm}, we deduce that $\dim(V_{s}(\mu))\leq 4+3(2(\ell-1)-2)+2(2(\ell-1)^{2}-3(\ell-1)+4)+\dim((L_{L}(3\omega_{2}))_{h}(\mu_{h})= 4(\ell-1)^{2}+6+\dim((L_{L}(3\omega_{2}))_{h}(\mu_{h}))$. Recursively and using the result for $\ell=2$, it follows that $\scale[0.9]{\displaystyle \dim(V_{s}(\mu))\leq 4\sum_{j=2}^{\ell-1}j^{2}+\sum_{j=2}^{\ell-1}6+10}$ $=\binom{2\ell}{3}+6\ell-6$, for all eigenvalues $\mu$ of $s$ on $V$. Therefore $\displaystyle \max_{s\in T\setminus\ZG(G)}\dim(V_{s}(\mu))$ $=\binom{2\ell}{3}+3(2\ell-2)$ and, as $\displaystyle \max_{u\in G_{u}\setminus \{1\}}\dim(V_{u}(1))$ $=\binom{2\ell+1}{3}$, we have shown that $\nu_{G}(V)=2\ell^{2}+\ell$.
\end{proof}

\begin{prop}\label{PropositionC2om1+om2}
Let $\ell=2$ and $V=L_{G}(\omega_{1}+\omega_{2})$. Then $\nu_{G}(V)=8-2\varepsilon_{p}(5)$.  Moreover we have $\displaystyle \max_{u\in G_{u}\setminus \{1\}}\dim(V_{u}(1))=8-3\varepsilon_{p}(5)$ and $\displaystyle \max_{s\in T\setminus\ZG(G)}\dim(V_{s}(\mu))=8-2\varepsilon_{p}(5)-2\varepsilon_{p}(2)$.
\end{prop}

\begin{proof}
Set $\lambda=\omega_{1}+\omega_{2}$. We first consider the case when $p=5$ and we determine $V\mid_{[L_{i},L_{i}]}$ for $i=1,2$. We start with $i=2$ and we have $e_{2}(\lambda)=3$, see Lemma \ref{weightlevelCl}, therefore $\displaystyle V\mid_{[L_{2},L_{2}]}=V_{2}^{0}\oplus \cdots\oplus V_{2}^{3}$, where $V_{2}^{i}=\displaystyle \bigoplus_{\gamma \in \mathbb{N}\Delta_{2}}V_{\lambda-i \alpha_{2}-\gamma}$. By \cite[Proposition]{Smith_82} and Lemma \ref{dualitylemma}, we have $V_{2}^{0}\cong L_{L_{2}}(\omega_{1})$ and $V_{2}^{3}\cong L_{L_{2}}(\omega_{1})$. Now, the weight $\displaystyle(\lambda-\alpha_{2})\mid_{T_{2}}=3\omega_{1}$ admits a maximal vector in $V_{2}^{1}$, therefore $V_{2}^{1}$ has a composition factor isomorphic to $L_{L_{2}}(3\omega_{1})$. By dimensional considerations, we determine that 
\begin{equation}\label{DecompC2om1+om2p=5L2}
V\mid_{[L_{2},L_{2}]}\cong L_{L_{2}}(\omega_{1})\oplus L_{L_{2}}(3\omega_{1})\oplus L_{L_{2}}(3\omega_{1})\oplus L_{L_{2}}(\omega_{1}).
\end{equation}
We will now determine $V\mid_{[L_{1},L_{1}]}$. By Lemma \ref{weightlevelCl}, we have $e_{1}(\lambda)=4$, therefore $\displaystyle V\mid_{[L_{1},L_{1}]}=V_{1}^{0}\oplus \cdots \oplus V_{1}^{4}$, where $V_{1}^{i}=\displaystyle \bigoplus_{\gamma \in \mathbb{N}\Delta_{1}}V_{\lambda-i \alpha_{1}-\gamma}$. By \cite[Proposition]{Smith_82} and Lemma \ref{dualitylemma}, we have $V_{1}^{0}\cong L_{L_{1}}(\omega_{2})$ and $V_{1}^{4}\cong L_{L_{1}}(\omega_{2})$. Now, the weight $\displaystyle(\lambda-\alpha_{1})\mid_{T_{1}}=2\omega_{2}$ admits a maximal vector in $V_{1}^{1}$, therefore $V_{1}^{1}$ has a composition factor isomorphic to $L_{L_{1}}(2\omega_{2})$. Similarly, the weight $\displaystyle (\lambda-2\alpha_{1}-\alpha_{2})\mid_{T_{1}}=\omega_{2}$ admits a maximal vector in $V_{1}^{2}$, thus $V_{1}^{2}$ has a composition factor isomorphic to $L_{L_{1}}(\omega_{2})$. By dimensional considerations, we have 
\begin{equation}\label{DecompC2om1+om2p=5L1}
V\mid_{[L_{1},L_{1}]}\cong L_{L_{1}}(\omega_{2})\oplus L_{L_{1}}(2\omega_{2})\oplus L_{L_{1}}(\omega_{2})\oplus L_{L_{1}}(2\omega_{2})\oplus L_{L_{1}}(\omega_{2}).
\end{equation}

In view of Lemma \ref{uniprootelems}, we have $\displaystyle \max_{u\in G_{u}\setminus \{1\}}\dim(V_{u}(1))=\max\{\dim(V_{x_{\alpha_{1}}(1)}(1)), \dim(V_{x_{\alpha_{2}}(1)}(1))\}$. Now, using Propositions \ref{PropositionAlnatural}, \ref{PropositionAlsymmetric} and \ref{PropositionAlsymmcube}, it follows that $\dim(V_{x_{\alpha_{1}}(1)}(1))=4$, see \eqref{DecompC2om1+om2p=5L2}, and $\dim(V_{x_{\alpha_{2}}(1)}(1))=5$, see \eqref{DecompC2om1+om2p=5L1}.  This shows that $\displaystyle \max_{u\in G_{u}\setminus \{1\}}\dim(V_{u}(1))=5$.

Let $s\in T\setminus \ZG(G)$. If $\dim((V_{2}^{i})_{s}(\mu))=\dim(V_{2}^{i})$ for some eigenvalue $\mu$ of $s$ on $V$, where $0\leq i\leq 3$, then $s\in \ZG(L_{2})^{\circ}\setminus \ZG(G)$. In this case, as $s$ acts on each $V_{2}^{i}$ as scalar multiplication by $c^{3-2i}$ and $c^{2}\neq 1$, we determine that $\dim(V_{s}(\mu))\leq 6$ for all eigenvalues $\mu$ of $s$ on $V$, where equality holds if and only if $c^{2}=-1$ and $\mu=\pm c$. We can now assume that $\dim((V_{2}^{i})_{s}(\mu))<\dim(V_{2}^{i})$ for all eigenvalues $\mu$ of $s$ on $V$ and all $0\leq i\leq 3$.  We write $s=z'\cdot h'$, where $z'\in \ZG(L_{2})^{\circ}$ and $h'\in [L_{2},L_{2}]$. Using \eqref{DecompC2om1+om2p=5L2} and Propositions \ref{PropositionAlnatural} and \ref{PropositionAlsymmcube}, we determine that $\dim(V_{s}(\mu))\leq 6$ for all eigenvalues $\mu$ of $s$ on $V$. Therefore $\displaystyle \max_{s\in T\setminus\ZG(G)}\dim(V_{s}(\mu))=6$.

We now consider the case of $p\neq 5$. We have $e_{1}(\lambda)=4$, therefore $\displaystyle V\mid_{[L_{1},L_{1}]}=V_{1}^{0}\oplus\cdots\oplus V_{1}^{4}$. By \cite[Proposition]{Smith_82} and Lemma \ref{dualitylemma}, we have $V_{1}^{0}\cong L_{L_{1}}(\omega_{2})$ and $V_{1}^{4}\cong L_{L_{1}}(\omega_{2})$.  Now, the weight $\displaystyle(\lambda-\alpha_{1})\mid_{T_{1}}=2\omega_{2}$ admits a maximal vector in $V_{1}^{1}$, thus $V_{1}^{1}$ has a composition factor isomorphic to $L_{L_{1}}(2\omega_{2})$.  Moreover, the weight $\displaystyle(\lambda-\alpha_{1}-\alpha_{2})\mid_{T_{1}}=0$ occurs with multiplicity $2$ and is a sub-dominant weight in the composition factor of $V_{1}^{1}$ isomorphic to $L_{L_{1}}(2\omega_{2})$, in which it has multiplicity $1-\varepsilon_{p}(2)$. Similarly, the weight $\displaystyle(\lambda-2\alpha_{1}-\alpha_{2})\mid_{T_{1}}=\omega_{2}$ is the highest weight in $V_{1}^{2}$, in which it has multiplicity $2$ and admits a maximal vector. By dimensional considerations, we deduce that $V_{1}^{2}$ has exactly two composition factors, both isomorphic to $L_{L_{1}}(\omega_{2})$, and, by \cite[II.2.12]{Jantzen_2007representations}, we have $V_{1}^{2}\cong L_{L_{1}}(\omega_{2})^{2}$. Further, $V_{1}^{1}$ and $V_{1}^{3}$ each has exactly $2+\varepsilon_{p}(2)$ composition factors: one isomorphic to $L_{L_{1}}(2\omega_{2})$ and $1+\varepsilon_{p}(2)$ to $L_{L_{1}}(0)$. 

Let $s\in T\setminus \ZG(G)$. If $\dim((V_{1}^{i})_{s}(\mu))=\dim(V_{1}^{i})$ for some eigenvalue $\mu$ of $s$ on $V$, where $0\leq i\leq 4$, then $s\in \ZG(L_{1})^{\circ}\setminus \ZG(G)$. In this case, as $s$ acts on each $V_{1}^{i}$ as scalar multiplication by $c^{2-i}$ and $c\neq 1$, we determine that $\dim(V_{s}(\mu))\leq 8-2\varepsilon_{p}(2)$ for all eigenvalues $\mu$ of $s$ on $V$, where equality holds for $p\neq 2$, $c=-1$ and $\mu=\pm 1$, respectively for $p=2$, $c^{3}=1$ and $\mu=c^{\pm 1}$. We thus assume that $\dim((V_{1}^{i})_{s}(\mu))<\dim(V_{1}^{i})$ for all eigenvalues $\mu$ of $s$ on $V$ and all $0\leq i\leq 4$. We write $s=z\cdot h$, where $z\in \ZG(L)^{\circ}$ and $h\in [L,L]$. We have $\scale[0.9]{\displaystyle \dim(V_{s}(\mu))=\sum_{i=0}^{4}\dim((V_{1}^{i})_{h}(\mu_{h}^{i}))}$, where $\dim((V^{i}_{1})_{h}(\mu_{h}^{i}))<\dim(V_{1}^{i})$ for all eigenvalues $\mu_{h}^{i}$ of $h$ on $V_{1}^{i}$. Using \eqref{Al_enum_P1}, we determine that the eigenvalues of $h$ on $V_{1}^{1}$ are $d^{2},1,1,d^{-2}$, where $d^{2}\neq 1$. Thus, $\dim((V_{1}^{1})_{h}(\mu_{h}^{1}))\leq 2$ for all eigenvalues $\mu_{h}^{1}$ of $h$ on $V_{1}^{1}$. This also gives $\dim((V_{1}^{3})_{h}(\mu_{h}^{3}))\leq 2$ for all eigenvalues $\mu_{h}^{3}$ of $h$ on $V_{1}^{3}$. Lastly, using Proposition \ref{PropositionAlnatural}, we determine that $\dim(V_{s}(\mu))\leq 8$ for all eigenvalues $\mu$ of $s$ on $V$. Therefore, when $p\neq 2$, we have shown that $\displaystyle \max_{s\in T\setminus\ZG(G)}\dim(V_{s}(\mu))=8$. Now, in the case of $p=2$, using the structure of each $V_{1}^{i}$, one sees that the eigenvalues of $s$ on $V$ are $c^{\pm 2}d^{\pm1}$ each with multiplicity at least $1$; $c^{\pm 1}d^{\pm 2}$ each with multiplicity $1$; $c^{\pm 1}$ each with multiplicity $2$; and $d^{\pm 1}$ each with multiplicity $2$, where $d\neq 1$ and $c\in k^{*}$. Thereby, we have $\dim(V_{s}(\mu))\leq 6$ for all eigenvalues $\mu$ of $s$ on $V$ and so $\displaystyle \max_{s\in T\setminus\ZG(G)}\dim(V_{s}(\mu))=6$. 

For the unipotent elements, we have $\displaystyle \max_{u\in G_{u}\setminus \{1\}}\dim(V_{u}(1))\leq \max\{\dim(V_{x_{\alpha_{1}}(1)}(1)), \dim(V_{x_{\alpha_{2}}(1)}(1))\}$, by Lemma \ref{uniprootelems}. First, let $p=2$. In this case, we have $V=L_{G}(\omega_{1})\otimes L_{G}(\omega_{2})$, as $kG$-modules, see \cite[(1.6)]{seitz1987maximal}. Knowing that the Jordan form of $x_{\alpha_{1}}(1)$ on $L_{G}(\omega_{1})$ is $J_{2}^{2}$, one shows that $x_{\alpha_{1}}(1)$ acts on $\wedge^{2}(W)$ as $J_{2}^{2}\oplus J_{1}^{2}$. Further, using \cite[Theorem B]{Korhonen_2020HesselinkNF}, one determines that $x_{\alpha_{1}}(1)$ acts on $L_{G}(\omega_{2})$ as $J_{2}\oplus J_{1}^{2}$, and, consequently, as $J_{2}^{8}$ on $V$. Therefore, $\dim(V_{x_{\alpha_{1}}(1)}(1))=8$. Analogously, one shows that $\dim(V_{x_{\alpha_{2}}(1)}(1))=8$. We now let $p\neq 2$. First, by the structure of $V\mid_{[L_{1},L_{1}]}$ and Propositions \ref{PropositionAlnatural} and \ref{PropositionAlsymmetric}, we see that $\dim(V_{x_{\alpha_{2}}(1)}(1))=8$. To determine $\dim(V_{x_{\alpha_{1}}(1)}(1))$, we have to establish $V\mid_{[L_{2},L_{2}]}$. By Lemma \ref{weightlevelCl}, we have $e_{2}(\lambda)=3$, therefore $\displaystyle V\mid_{[L_{2},L_{2}]}=V_{2}^{0}\oplus \cdots \oplus V_{2}^{3}$, where $\scale[0.9]{V_{2}^{i}=\displaystyle \bigoplus_{\gamma \in \mathbb{N}\Delta_{2}}V_{\lambda-i \alpha_{2}-\gamma}}$. By \cite[Proposition]{Smith_82} and Lemma \ref{dualitylemma}, we have $V_{2}^{0}\cong L_{L_{2}}(\omega_{1})$ and $V_{2}^{3}\cong L_{L_{2}}(\omega_{1})$. Now, the weight $\displaystyle (\lambda-\alpha_{2})\mid_{T_{2}}=3\omega_{1}$ admits a maximal vector in $V_{2}^{1}$, thus $V_{2}^{1}$ has a composition factor isomorphic to $L_{L_{2}}(3\omega_{1})$. Moreover, the weight $\displaystyle(\lambda-\alpha_{1}-\alpha_{2})\mid_{T_{2}}=\omega_{1}$ occurs with multiplicity $2$ and is a sub-dominant weight in the composition factor of $V_{2}^{1}$ isomorphic to $L_{L_{2}}(3\omega_{1})$, in which it has multiplicity $1$, if and only if $p\neq 3$. By dimensional considerations, we determine that $V_{2}^{1}$ and $V_{2}^{2}$ each has exactly $2+\varepsilon_{p}(3)$ composition factors: one isomorphic to $L_{L_{2}}(3\omega_{1})$ and $1+\varepsilon_{p}(3)$ isomorphic to $L_{L_{2}}(\omega_{1})$. Thereby, using Propositions \ref{PropositionAlnatural} and \ref{PropositionAlsymmetric}, we determine that $\dim(V_{x_{\alpha_{1}}(1)}(1))\leq 6+2\varepsilon_{p}(3)$. To conclude, we have shown that $\displaystyle \max_{u\in G_{u}\setminus \{1\}}\dim(V_{u}(1))=8-3\varepsilon_{p}(5)$ and, as $\displaystyle \max_{s\in T\setminus\ZG(G)}\dim(V_{s}(\mu))=8-2\varepsilon_{p}(5)-2\varepsilon_{p}(2)$, it follows that $\nu_{G}(V)=8-2\varepsilon_{p}(5)$.
\end{proof}

\begin{prop}\label{PropositionClom1+om2p=3}
Let $p=3$, $\ell\geq 3$ and $V=L_{G}(\omega_{1}+\omega_{2})$. Then $\nu_{G}(V)=2\ell^{2}+\ell-2$.  Moreover, we have $\displaystyle \max_{u\in G_{u}\setminus \{1\}}\dim(V_{u}(1))$ $=\frac{4\ell^{3}-7\ell+6}{3}$ and $\displaystyle \max_{s\in T\setminus\ZG(G)}\dim(V_{s}(\mu))$ $\leq \frac{4\ell^{3}-6\ell^{2}+14\ell-12}{3}$.
\end{prop}

\begin{proof}
Let $\lambda=\omega_{1}+\omega_{2}$ and $L=L_{1}$. We have $e_{1}(\lambda)=4$, therefore $\displaystyle V\mid_{[L,L]}=V^{0}\oplus \cdots \oplus V^{4}$. Now, by \cite[Proposition]{Smith_82} and Lemma \ref{dualitylemma}, we have $V^{0}\cong L_{L}(\omega_{2})$ and $V^{4}\cong L_{L}(\omega_{2})$. As the weight $\displaystyle (\lambda-\alpha_{1})\mid_{T_{1}}=2\omega_{2}$ admits a maximal vector in $V^{1}$, it follows that $V^{1}$ has a composition factor isomorphic to $L_{L}(\omega_{2})$. Moreover, the weight $\displaystyle (\lambda-\alpha_{1}-2\alpha_{2}-\cdots-2\alpha_{\ell-1}-\alpha_{\ell})\mid_{T_{1}}=0$ occurs with multiplicity $\ell$ and is a a sub-dominant weight in the composition factor of $V^{1}$ isomorphic to $L_{L}(2\omega_{2})$ in which it has multiplicity $\ell-1$. Similarly, the weight $\displaystyle (\lambda-2\alpha_{1}-\alpha_{2})\mid_{T_{1}}=\omega_{2}+\omega_{3}$ admits a maximal vector in $V^{2}$, thus $V^{2}$ has a composition factor isomorphic to $L_{L}(\omega_{2}+\omega_{3})$. Moreover, the weight $\displaystyle (\lambda-2\alpha_{1}-\cdots -2\alpha_{\ell-1}-\alpha_{\ell})\mid_{T_{1}}=\omega_{2}$ occurs with multiplicity $\ell$ and is a sub-dominant weight in the composition factor of $V^{2}$ isomorphic to $L_{L}(\omega_{2}+\omega_{3})$ in which it has multiplicity $\ell-1$.  Now, as $\dim(V^{2})\leq 16\binom{\ell}{3}-\binom{2\ell-2}{3}+4\ell-4$, we deduce that $V^{2}$ has exactly $2$ composition factors: one isomorphic to $L_{L}(\omega_{2}+\omega_{3})$ and one to $L_{L}(\omega_{2})$, therefore $V^{2}\cong L_{L}(\omega_{2}+\omega_{3})\oplus L_{L}(\omega_{2})$, see \cite[II.$2.14$]{Jantzen_2007representations}. Further, $V^{1}$ and $V^{3}$ each has exactly $2$ composition factors: one isomorphic to $L_{L}(2\omega_{2})$ and one to $L_{L}(0)$, therefore $V^{1}\cong L_{L}(2\omega_{2})\oplus L_{L}(0)$ and $V^{3}\cong L_{L}(2\omega_{2})\oplus L_{L}(0)$, see \cite[II.$2.14$]{Jantzen_2007representations}. We have shown that
\begin{equation}\label{DecompCl_omega1+omega2p=3}
V\mid_{[L,L]}\cong L_{L}(0)^{2}\oplus L_{L}(\omega_{2})^{3} \oplus  L_{L}(2\omega_{2})^{2} \oplus L_{L}(\omega_{2}+\omega_{3}).
\end{equation}

For the unipotent elements, we have $\displaystyle \max_{u\in G_{u}\setminus \{1\}}\dim(V_{u}(1))=\max\{\dim(V_{x_{\alpha_{\ell}}(1)}(1)), \dim(V_{x_{\alpha_{\ell-1}}(1)}(1))\}$, by Lemma \ref{uniprootelems}. Recursively, using \eqref{DecompCl_omega1+omega2p=3} and Propositions \ref{PropositionClnatural}, \ref{PropositionClsymm} and \ref{PropositionC2om1+om2} for the base case of $\ell=2$, we determine that $\scale[0.9]{\displaystyle \dim(V_{x_{\alpha_{i}}(1)}(1))\leq 4\sum_{j=2}^{\ell-1}j^{2}+4\sum_{j=2}^{\ell-1}j-\sum_{j=2}^{\ell-1}1+8}$ $=\frac{4\ell^{3}-7\ell+6}{3}$, where equality holds for $i=\ell$.

We now focus on the semisimple elements. Let $s\in T\setminus \ZG(G)$. If $\dim(V^{i}_{s}(\mu))=\dim(V^{i})$ for some eigenvalue $\mu$ of $s$ on $V$, where $0\leq i\leq 4$, then $s\in \ZG(L)^{\circ}\setminus \ZG(G)$ and acts on $V^{i}$ as scalar multiplication by $c^{2-i}$. As $c\neq 1$, we determine that $\dim(V_{s}(\mu))\leq 16\binom{\ell}{3}-\binom{2\ell-2}{3}+8\ell-8=\frac{4\ell^{3}-6\ell^{2}+8\ell-6}{3}$. We thus assume that $\dim(V^{i}_{s}(\mu))<\dim(V^{i})$ for all eigenvalues $\mu$ of $s$ on $V$ and for all $0\leq i\leq 4$. We write $s=z\cdot h$, where $z\in \ZG(L)^{\circ}$ and $h\in [L,L]$, and, recursively, using \eqref{DecompCl_omega1+omega2p=3} and Propositions \ref{PropositionClnatural},  \ref{PropositionClsymm} and \ref{PropositionC2om1+om2} for the base case of $\ell=2$, we determine that $\scale[0.9]{\displaystyle\dim(V_{s}(\mu))\leq 4\sum_{j=2}^{\ell-1}[j^{2}+1]+8}$ $=\frac{4\ell^{3}-6\ell^{2}+14\ell-12}{3}$ for all eigenvalues $\mu$ of $s$ on $V$. Thus, we have shown that $\displaystyle \max_{s\in T\setminus\ZG(G)}\dim(V_{s}(\mu))$ $\leq \frac{4\ell^{3}-6\ell^{2}+14\ell-12}{3}$ and, as $\displaystyle \max_{u\in G_{u}\setminus \{1\}}\dim(V_{u}(1))$ $=\frac{4\ell^{3}-7\ell+6}{3}$, it follows that $\nu_{G}(V)=2\ell^{2}+\ell-2$.
\end{proof}

\begin{prop}\label{PropositionClom1+om2pneq3}
Let $p\neq 3$, $\ell\geq 3$ and $V=L_{G}(\omega_{1}+\omega_{2})$. Then $\nu_{G}(V)\geq 4\ell^{2}-4\ell-\varepsilon_{p}(2\ell+1)-2\ell\varepsilon_{p}(2)$, where equality holds for $p\neq 2$. Moreover, we have $\scale[0.9]{\displaystyle \max_{u\in G_{u}\setminus \{1\}}\dim(V_{u}(1))}$ $\leq 2\binom{2\ell}{3}-(2\ell-1)\varepsilon_{p}(2\ell+1)+2\ell\varepsilon_{p}(2)$, where equality holds for $p\neq 2$, and $\scale[0.9]{\displaystyle \max_{s\in T\setminus\ZG(G)}\dim(V_{s}(\mu))}$ $\leq 16\binom{\ell}{3}+8\ell+4-2(\ell-1)\varepsilon_{p}(2\ell+1)-(4\ell-2)\varepsilon_{p}(2)-(6-6\varepsilon_{p}(2))\varepsilon_{\ell}(3)-(2-2\varepsilon_{p}(2))\varepsilon_{\ell}(4)$.
\end{prop}

\begin{proof}
Set $\lambda=\omega_{1}+\omega_{2}$ and $L=L_{1}$. Once again, we have $e_{1}(\lambda)=4$, see Lemma \ref{weightlevelCl}, therefore $\displaystyle V\mid_{[L,L]}=V^{0}\oplus \cdots \oplus V^{4}$. As in the proof of Proposition \ref{PropositionClom1+om2p=3}, we have $V^{0}\cong L_{L}(\omega_{2})$ and $V^{4}\cong L_{L}(\omega_{2})$. In $V^{2}$ the weight $\displaystyle (\lambda-2\alpha_{1}-\alpha_{2})\mid_{T_{1}}=\omega_{2}+\omega_{3}$ admits a maximal vector, thus $V^{2}$ has a composition factor isomorphic to $L_{L}(\omega_{2}+\omega_{3})$. Moreover, the weight $\displaystyle (\lambda-2\alpha_{1}-\cdots -2\alpha_{\ell-1}-\alpha_{\ell})\mid_{T_{1}}=\omega_{2}$ occurs with multiplicity $2\ell-2-\varepsilon_{p}(2\ell+1)$ and is a sub-dominant weight in the composition factor of $V^{2}$ isomorphic to $L_{L}(\omega_{2}+\omega_{3})$ in which it has multiplicity $2\ell-4-\varepsilon_{p}(2\ell-1)$. Now, in $V^{1}$ the weight $\displaystyle (\lambda-\alpha_{1})\mid_{T_{1}}=2\omega_{2}$ admits a maximal vector, thus $V^{1}$ has a composition factor isomorphic to $L_{L}(2\omega_{2})$. Moreover, the weight $\displaystyle (\lambda-\alpha_{1}-\alpha_{2})\mid_{T_{1}}=\omega_{3}$ occurs with multiplicity $2$ and is a a sub-dominant weight in the composition factor of $V^{1}$ isomorphic to $L_{L}(2\omega_{2})$, in which it has multiplicity $1$, if and only if $p\neq 2$. Lastly, we note that the weight $\displaystyle (\lambda-\alpha_{1}-2\alpha_{2}-\cdots-2\alpha_{\ell-1}-\alpha_{\ell})\mid_{T_{1}}=0$ occurs with multiplicity $2\ell-2-\varepsilon_{p}(2\ell+1)$ in $V^{1}$. As $\dim(V^{1})\leq 4(\ell-1)^{2}-\varepsilon_{p}(2\ell+1)$, we determine that $V^{1}$ has exactly $3-\varepsilon_{p}(2\ell+1)+\varepsilon_{p}(\ell-1)+[2+\varepsilon_{p}(\ell-1)]\varepsilon_{p}(2)$ composition factors: one isomorphic to $L_{L}(2\omega_{1})$, $1+\varepsilon_{p}(2)$ to $L_{L}(\omega_{3})$ and $1-\varepsilon_{p}(2\ell+1)+\varepsilon_{p}(\ell-1)+[1+\varepsilon_{p}(\ell-1)]\varepsilon_{p}(2)$ to $L_{L}(0)$. Further, $V^{3}$ has the same composition factors as $V^{1}$, and $V^{2}$ has exactly $3+\varepsilon_{p}(2\ell-1)-\varepsilon_{p}(2\ell+1)$ composition factors: one isomorphic to $L_{L}(\omega_{2}+\omega_{3})$ and $2+\varepsilon_{p}(2\ell-1)-\varepsilon_{p}(2\ell+1)$ to $L_{L}(\omega_{2})$.

We begin with the semisimple elements. Let $s\in T\setminus \ZG(G)$. If $\dim(V^{i}_{s}(\mu))=\dim(V^{i})$ for some eigenvalue $\mu$ of $s$ on $V$, where $0\leq i\leq 4$, then $s\in \ZG(L)^{\circ}\setminus \ZG(G)$ and acts on $V^{i}$ as scalar multiplication by $c^{2-i}$. As $c\neq 1$, we determine that $\dim(V_{s}(\mu))\leq 16\binom{\ell}{3}+4\ell-4-(2\ell-2)\varepsilon_{p}(2\ell+1)+(4\ell-4)(1-\varepsilon_{p}(2))+2\varepsilon_{\ell}(3)\varepsilon_{p}(7)$. We thus assume that $\dim(V^{i}_{s}(\mu))<\dim(V^{i})$ for all eigenvalues $\mu$ of $s$ on $V$ and for all $0\leq i\leq 4$. We write $s=z\cdot h$, where $z\in \ZG(L)^{\circ}$ and $h\in [L,L]$. By the structure of $V\mid_{[L,L]}$ and Propositions \ref{PropositionClnatural} and \ref{PropositionClsymm}, we have
\begin{equation}\label{dimV_s(mu)-om1+om2pneq3}
\scale[0.87]{\begin{split}\dim(V_{s}(\mu))& \leq (4+\varepsilon_{p}(2\ell-1)-\varepsilon_{p}(2\ell+1))\dim((L_{L}(\omega_{2}))_{h}(\mu_{h}))+2\dim((L_{L}(2\omega_{2}))_{h}(\mu_{h}))+(2+2\varepsilon_{p}(2))\dim((L_{L}(\omega_{3}))_{h}(\mu_{h}))\\
& +[2-2\varepsilon_{p}(2\ell+1)+2\varepsilon_{p}(\ell-1)+(2+2\varepsilon_{p}(\ell-1))\varepsilon_{p}(2)]\dim((L_{L}(0))_{h}(\mu_{h}))+\dim((L_{L}(\omega_{2}+\omega_{3}))_{h}(\mu_{h}))\\
& \leq 4(\ell-1)^{2}+2(\ell-1)+2-2(\ell-1)\varepsilon_{p}(2\ell+1)+[2(\ell-1)-2]\varepsilon_{p}(2\ell-1)+2\varepsilon_{p}(\ell-1)-2[2(\ell-1)^{2}-5(\ell-1)+5\\
&-\varepsilon_{p}(\ell-1)]\varepsilon_{p}(2)+(2+2\varepsilon_{p}(2))\dim((L_{L}(\omega_{3}))_{h}(\mu_{h}))+\dim((L_{L}(\omega_{2}+\omega_{3}))_{h}(\mu_{h})).
\end{split}}
\end{equation}
We will use recurrence to determine the bound for $\dim(V_{s}(\mu))$ in \eqref{dimV_s(mu)-om1+om2pneq3}. We first handle the cases $\ell=3,4, 5$. Let $\ell=3$. Using \eqref{dimV_s(mu)-om1+om2pneq3} and Propositions \ref{PropositionClwedge} and \ref{PropositionC2om1+om2}, we determine that $\dim(V_{s}(\mu))\leq 38-4\varepsilon_{p}(7)-4\varepsilon_{p}(2)$ for all eigenvalues $\mu$ of $s$ on $V$. Therefore, as $\dim(V_{s}(\mu))\leq 32-2\varepsilon_{p}(7)$ for all $(s,\mu)\in \ZG(L)^{\circ}\setminus \ZG(G)\times k^{*}$, we have $\displaystyle \max_{s\in T\setminus\ZG(G)}÷\dim(V_{s}(\mu))\leq 38-4\varepsilon_{p}(7) -4\varepsilon_{p}(2)$. Analogously, one shows that for $\ell=4$ we have $\displaystyle \max_{s\in T\setminus\ZG(G)}\dim(V_{s}(\mu))\leq 98-12\varepsilon_{p}(2)$, while, for $\ell=5$ we have $\displaystyle \max_{s\in T\setminus\ZG(G)}\dim(V_{s}(\mu))\leq 204-8\varepsilon_{p}(11)-18\varepsilon_{p}(2)$. We now let $\ell\geq 6$. Recursively and using Proposition \ref{PropositionClwedge} and the result we just established for $\ell=5$, by \eqref{dimV_s(mu)-om1+om2pneq3}, we have $\dim(V_{s}(\mu))\leq 8[(\ell-1)^{2}-(\ell-1)+1]+(2(\ell-1)-2)\varepsilon_{p}(2(\ell-1)+1)-2(\ell-1)\varepsilon_{p}(2(\ell-1)+3)-4\varepsilon_{p}(2)+\dim((L_{L}(\omega_{2}+\omega_{3}))_{h}(\mu_{h}))$ 
$\scale[0.9]{\displaystyle\leq 8\sum_{j=5}^{\ell-1}[j^{2}-j+1]+\sum_{j=5}^{\ell-1}(2j-2)\varepsilon_{p}(2j+1)-\sum_{j=5}^{\ell-1}2j\varepsilon_{p}(2j+3)-4\sum_{j=5}^{\ell-1}\varepsilon_{p}(2)}+204-8\varepsilon_{p}(11)-18\varepsilon_{p}(2)=16\binom{\ell}{3}+8\ell+4-2(\ell-1)\varepsilon_{p}(2\ell+1)-(4\ell-2)\varepsilon_{p}(2)$. Therefore, $\scale[0.9]{\displaystyle \max_{s\in T\setminus\ZG(G)}\dim(V_{s}(\mu))}\leq 16\binom{\ell}{3}+8\ell+4-2(\ell-1)\varepsilon_{p}(2\ell+1)-(4\ell-2)\varepsilon_{p}(2)$.

In the case of the unipotent elements, by Lemma \ref{uniprootelems}, we have $\scale[0.9]{\displaystyle \max_{u\in G_{u}\setminus \{1\}}\dim(V_{u}(1))=\max_{i=\ell-1,\ell}\{\dim(V_{x_{\alpha_{i}}(1)}(1))\}}$. We first consider the case when $p\neq 2$. Now, by the structure of $V\mid_{[L,L]}$ and Propositions \ref{PropositionClnatural} and \ref{PropositionClsymm}, for both $i=\ell-1$ and $i=\ell$, we have
\begin{equation}\label{dimV_u(1)-om1+om2pneq3}
\scale[0.84]{\begin{split}
\dim(V_{x_{\alpha_{i}}(1)}(1))& \leq 4(\ell-1)^{2}+6(\ell-1)-2+[2(\ell-1)-1]\varepsilon_{p}(2\ell-1)-[2(\ell-1)+1]\varepsilon_{p}(2\ell+1)+2\varepsilon_{p}(\ell-1)\\
&+2\dim((L_{L}(\omega_{3}))_{x_{\alpha_{i}}(1)}(1))+\dim((L_{L}(\omega_{2}+\omega_{3}))_{x_{\alpha_{i}}(1)}(1)).
\end{split}}
\end{equation}
We will use recurrence to determine the bound for $\dim(V_{x_{\alpha_{i}}(1)}(1))$, $i=\ell-1,\ell$, in \eqref{dimV_u(1)-om1+om2pneq3}. For this, we first treat the case of $\ell=3$. By Propositions \ref{PropositionClwedge} and \ref{PropositionC2om1+om2}, we determine that $\displaystyle \max_{u\in G_{u}\setminus \{1\}}\dim(V_{u}(1))\leq 40-5\varepsilon_{p}(7)$. We now let $\ell\geq 4$. Recursively and using Proposition \ref{PropositionClwedge} and the result we just established for $\ell=3$, by \eqref{dimV_u(1)-om1+om2pneq3}, we have $\dim(V_{x_{\alpha_{i}}(1)}(1)) \leq 8(\ell-1)^{2}+[2(\ell-1)-1]\varepsilon_{p}(2\ell-1)-[2(\ell-1)+1]\varepsilon_{p}(2\ell+1)+\dim((L_{L}(\omega_{2}+\omega_{3}))_{x_{\alpha_{i}}(1)}(1)) \leq$ $\scale[0.9]{\displaystyle 8\sum_{j=3}^{\ell-1}j^{2}+\sum_{j=3}^{\ell-1}(2j-1)\varepsilon_{p}(2j+1)-\sum_{j=3}^{\ell-1}(2j+1)\varepsilon_{p}(2j+3)+40-5\varepsilon_{p}(7)}$ $=2\binom{2\ell}{3}-(2\ell-1)\varepsilon_{p}(2\ell+1)$, where $i=\ell-1,\ell$. On the other hand, by \cite[Lemmas $4.9.1$ and $4.9.2$]{mcninch_1998}, we know that $V$ is a composition factor of the $kG$-module $W\otimes \wedge^{2}(W)$, where $\ch(W\otimes \wedge^{2}(W))=\chi(\omega_{1}+\omega_{2})+\chi(\omega_{3})+2\chi(\omega_{1})$. Note that by $\chi(\lambda^{'})$ we understand the character of the Weyl $kG$-module of highest weight $\lambda^{'}\in \X(T)$. Therefore, in view of Lemma \ref{LemmaonfiltrationofV} and \cite[Lemmas $4.8.2$ and $4.9.2$]{mcninch_1998}, we determine that $\dim(V_{u}(1))\geq \dim((W\otimes \wedge^{2}(W))_{u}(1))-\dim((L_{L}(\omega_{3}))_{u}(1))-(2+\varepsilon_{p}(\ell-1)+\varepsilon_{p}(2\ell+1))\dim((L_{L}(\omega_{1}))_{u}(1))$ for all $u\in G_{u}$. Now, since $x_{\alpha_{\ell}}(1)$ acts on $W$ as $J_{2}\oplus J_{1}^{2\ell-2}$ and on $\wedge^{2}(W)$ as $J_{2}^{2\ell-2}\oplus J_{1}^{2\ell^{2}-5\ell+4}$, see proof of Proposition \ref{PropositionClnatural}, using \cite[Lemma $3.4$]{liebeck_2012unipotent}, we determine that $\dim((W\otimes \wedge^{2}(W))_{x_{\alpha_{\ell}}(1)}(1))=4\ell^{3}-8\ell^{2}+9\ell-4$. We now apply Propositions \ref{PropositionClnatural} and \ref{PropositionClwedgecube} and see that $\dim(V_{x_{\alpha_{\ell}}(1)}(1))\geq 2\binom{2\ell}{3}-(2\ell-1)\varepsilon_{p}(2\ell+1)$. Therefore, we have shown that when $p\neq 2$, we have $\scale[0.9]{\displaystyle \max_{u\in G_{u}\setminus \{1\}}\dim(V_{u}(1))}=2\binom{2\ell}{3}-(2\ell-1)\varepsilon_{p}(2\ell+1)$, and, as $\scale[0.9]{\displaystyle \max_{s\in T\setminus\ZG(G)}\dim(V_{s}(\mu))}\leq 16\binom{\ell}{3}+8\ell+4-2(\ell-1)\varepsilon_{p}(2\ell+1)$, we determine that $\nu_{G}(V)=4\ell^{2}-4\ell-\varepsilon_{p}(2\ell+1)$.  

Lastly, we consider the case of $p=2$. Then, by the structure of $V\mid_{[L,L]}$, for both $i=\ell-1$ and $i=\ell$, we have $\dim(V_{x_{\alpha_{i}}(1)}(1))\leq 6\dim((L_{L}(\omega_{2}))_{x_{\alpha_{i}}(1)}(1))+4\dim((L_{L}(\omega_{3}))_{x_{\alpha_{i}}(1)}(1))+[4+4\varepsilon_{2}(\ell-1)]\dim((L_{L}(0))_{x_{\alpha_{i}}(1)}(1))+\dim((L_{L}(\omega_{2}+\omega_{3}))_{x_{\alpha_{i}}(1)}(1))$. Let $\ell=3$. Then, using Propositions \ref{PropositionClnatural}, \ref{PropositionClwedge} and \ref{PropositionC2om1+om2}, one shows that $\scale[0.9]{\displaystyle \max_{u\in G_{u}\setminus \{1\}}\dim(V_{u}(1))}\leq 46$. We now assume that $\ell\geq 4$. Recursively, using Propositions \ref{PropositionClnatural} and \ref{PropositionClwedge} and the result for $\ell=3$, one shows that $\dim(V_{x_{\alpha_{i}}(1)}(1))\leq 8(\ell-1)^{2}+2+\dim((L_{L}(\omega_{2}+\omega_{3}))_{x_{\alpha_{i}}(1)}(1))\leq$ $\scale[0.9]{\displaystyle 8\sum_{j=3}^{\ell-1}j^{2}+2(\ell-3)+46}=2\binom{2\ell}{3}+2\ell$. Therefore, $\scale[0.9]{\displaystyle\max_{u\in G_{u}\setminus \{1\}}\dim(V_{u}(1))}\leq 2\binom{2\ell}{3}+2\ell$ and, as $\scale[0.9]{\displaystyle \max_{s\in T\setminus\ZG(G)}\dim(V_{s}(\mu))}\leq 16\binom{\ell}{3}+4\ell+6$, it follows that $\nu_{G}(V)\geq 4\ell^{2}-6\ell$. 
\end{proof}

\begin{prop}\label{C22om2}
Let $p\neq 2$, $\ell=2$ and $V=L_{G}(2\omega_{2})$. Then $\nu_{G}(V)=4$.  Moreover, we have $\scale[0.9]{\displaystyle \max_{u\in G_{u}\setminus \{1\}}\dim(V_{u}(1))}$ $\leq 6-\varepsilon_{p}(5)+2\varepsilon_{p}(3)$, where equality holds for $p\neq 3$, and $\scale[0.9]{\displaystyle \max_{s\in T\setminus\ZG(G)}\dim(V_{s}(\mu))}=10-\varepsilon_{p}(5)$.
\end{prop}

\begin{proof}
Set $\lambda=2\omega_{2}$ and $L=L_{1}$. By Lemma \ref{weightlevelCl}, we have $e_{1}(\lambda)=4$, therefore $\displaystyle V\mid_{[L,L]}=V^{0}\oplus \cdots\oplus V^{4}$. By \cite[Proposition]{Smith_82} and Lemma \ref{dualitylemma}, we have $V^{0}\cong L_{L}(2\omega_{2})$ and $V^{4}\cong L_{L}(2\omega_{2})$.  Now, the weight $\displaystyle(\lambda-\alpha_{1}-\alpha_{2})\mid_{T_{1}}=\omega_{2}$ admits a maximal vector in $V^{1}$, thereby $V^{1}$ has a composition factor isomorphic to $L_{L}(\omega_{2})$. Similarly, in $V^{2}$ the weight $\displaystyle(\lambda-2\alpha_{1}-\alpha_{2})\mid_{T_{1}}=2\omega_{2}$ admits a maximal vector, thus $V^{2}$ has a composition factor isomorphic to $L_{L}(2\omega_{2})$. Further, as $\dim(V^{2})\leq 4-\varepsilon_{p}(5)$, it follows that $V^{2}\cong  L_{L}(2\omega_{2})\oplus L_{L}(0)^{1-\varepsilon_{p}(5)}$, see \cite[II.$2.14$]{Jantzen_2007representations}, and, consequently:
\begin{equation}\label{DecompVC22om2}
V\mid_{[L,L]}\cong L_{L}(2\omega_{2})\oplus L_{L}(\omega_{2})\oplus L_{L}(2\omega_{2})\oplus L_{L}(0)^{1-\varepsilon_{p}(5)}\oplus L_{L}(\omega_{2})\oplus L_{L}(2\omega_{2}).
\end{equation}

We start with the semisimple elements. Let $s\in T\setminus \ZG(G)$. If $\dim(V^{i}_{s}(\mu))=\dim(V^{i})$ for some eigenvalue $\mu$ of $s$ on $V$, where $0\leq i\leq 4$, then $s\in \ZG(L)^{\circ}\setminus \ZG(G)$. In this case, as $s$ acts on each $V^{i}$ as scalar multiplication by $c^{2-i}$ and $c\neq 1$, we determine that $\dim(V_{s}(\mu))\leq 10-\varepsilon_{p}(5)$, where equality holds for $c=-1$ and $\mu=1$. We thus assume that $\dim(V^{i}_{s}(\mu))<\dim(V^{i})$ for all eigenvalues $\mu$ of $s$ on $V$ and all $0\leq i\leq 4$. We write $s=z\cdot h$, where $z\in \ZG(L)^{\circ}$ and $h\in [L,L]$, and, by Propositions \ref{PropositionAlnatural} and \ref{PropositionAlsymmetric}, we determine that $\dim(V_{s}(\mu))\leq 9-\varepsilon_{p}(5)$ for all eigenvalues $\mu$ of $s$ on $V$. This shows that $\scale[0.9]{\displaystyle \max_{s\in T\setminus\ZG(G)}\dim(V_{s}(\mu))}=10-\varepsilon_{p}(5)$.

We now focus on the unipotent elements. By Lemma \ref{uniprootelems}, we have $\scale[0.9]{\displaystyle \max_{u\in G_{u}\setminus \{1\}}\dim(V_{u}(1))=\max_{i=1,2}\dim(V_{x_{\alpha_{i}}(1)}(1))}$. For $x_{\alpha_{2}}(1)$, by \eqref{DecompVC22om2} and Propositions \ref{PropositionAlnatural} and \ref{PropositionAlsymmetric}, we have $\dim(V_{x_{\alpha_{2}}(1)}(1))=6-\varepsilon_{p}(5)$. To calculate $\dim(V_{x_{\alpha_{1}}(1)}(1))$, we need to know $V\mid_{[L_{2},L_{2}]}$. By Lemma \ref{weightlevelCl}, we have $e_{2}(\lambda)=4$, therefore $\displaystyle V\mid_{[L_{2},L_{2}]}=V_{2}^{0}\oplus \cdots \oplus V_{2}^{4}$, where $\scale[0.9]{V_{2}^{i}=\displaystyle \bigoplus_{\gamma \in \mathbb{N}\Delta_{2}}V_{\lambda-i \alpha_{2}-\gamma}}$. By \cite[Proposition]{Smith_82} and Lemma \ref{dualitylemma}, we have $V_{2}^{0}\cong L_{L_{2}}(0)$ and $V_{2}^{4}\cong L_{L_{2}}(0)$. Now, the weight $\displaystyle(\lambda-\alpha_{2})\mid_{T_{2}}=2\omega_{1}$ admits a maximal vector in $V_{2}^{1}$, thereby $V_{2}^{1}$ has a composition factor isomorphic to $L_{L_{2}}(2\omega_{1})$. In $V^{2}$, the weight $\displaystyle (\lambda-2\alpha_{2})\mid_{T_{2}}=4\omega_{1}$ admits a maximal vector, thereby $V_{2}^{2}$ has a composition factor isomorphic to $L_{L_{2}}(4\omega_{1})$. We also note that the weight $\displaystyle (\lambda-2\alpha_{1}-2\alpha_{2})\mid_{T_{2}}=0$ occurs with multiplicity $2-\varepsilon_{p}(5)$ and is a sub-dominant weight in the composition factor of $V_{2}^{2}$ isomorphic to $L_{L_{2}}(4\omega_{1})$, in which it has multiplicity $1$, if and only if $p\neq 3$. Now, if $p\neq 3$, then, by dimensional considerations and \cite[II.2.14]{Jantzen_2007representations}, we have $V_{2}^{2}\cong L_{L_{2}}(4\omega_{1})\oplus L_{L_{2}}(0)^{1-\varepsilon_{p}(5)}$, $V_{2}^{1}\cong L_{L_{2}}(2\omega_{1})$ and $V_{2}^{3}\cong L_{L_{2}}(2\omega_{1})$, and so $\dim(V_{x_{\alpha_{1}}(1)}(1))=6-\varepsilon_{p}(5)$, by Propositions \ref{PropositionAlsymmetric} and \ref{PropositionAlsymmcube}. If $p=3$, then $V_{2}^{2}$ has $3$ composition factors: one isomorphic to $L_{L_{2}}(4\omega_{1})\cong L_{L_{2}}(\omega_{1})\otimes L_{L_{2}}(\omega_{1})^{(3)}$ and two to $L_{L_{2}}(0)$; $V_{2}^{1}\cong L_{L_{2}}(2\omega_{1})$ and $V_{2}^{3}\cong L_{L_{2}}(2\omega_{1})$. Thus, by Proposition \ref{PropositionAlsymmetric}, it follows that $\dim(V_{x_{\alpha_{1}}(1)}(1))=4+\dim((V^{2}_{2})_{x_{\alpha_{1}}(1)}(1))$. Now, as $x_{\alpha_{1}}(1)$ acts as $J_{2}$ on $L_{L_{2}}(\omega_{1})$ and on $ L_{L_{2}}(\omega_{1})^{(3)}$, respectively, by \cite[Lemma $3.4$]{liebeck_2012unipotent} and Lemma \ref{LemmaonfiltrationofV}, we determine that $\dim((V_{2}^{2})_{x_{\alpha_{1}}(1)}(1))\leq 4$, hence $\dim(V_{x_{\alpha_{1}}(1)}(1))\leq 8$. In conclusion,  we showed that $\scale[0.9]{\displaystyle \max_{u\in G_{u}\setminus \{1\}}\dim(V_{u}(1))}\leq 6-\varepsilon_{p}(5)+2\varepsilon_{p}(3)$, where equality holds for $p\neq 3$, and, as $\scale[0.9]{\displaystyle \max_{s\in T\setminus\ZG(G)}\dim(V_{s}(\mu))}=10-\varepsilon_{p}(5)$, we determine that $\nu_{G}(V)=4$.
\end{proof}

\begin{prop}\label{C22om1+om2}
Let $p=7$, $\ell=2$ and $V=L_{G}(\omega_{1}+2\omega_{2})$. Then $\nu_{G}(V)=12$.  Moreover, we have $\scale[0.9]{\displaystyle \max_{u\in G_{u}\setminus \{1\}}\dim(V_{u}(1))}=7$ and $\scale[0.9]{\displaystyle \max_{s\in T\setminus\ZG(G)}\dim(V_{s}(\mu))}=12$.
\end{prop}

\begin{proof}
Set $\lambda=\omega_{1}+2\omega_{2}$ and $L=L_{1}$. By Lemma \ref{weightlevelCl}, we have $e_{1}(\lambda)=6$, therefore $\displaystyle V\mid_{[L,L]}=V^{0}\oplus \cdots \oplus V^{6}$. By \cite[Proposition]{Smith_82} and Lemma \ref{dualitylemma}, we have $V^{0}\cong L_{L}(2\omega_{2})$ and $V^{6}\cong L_{L}(2\omega_{2})$. Now, in $V^{1}$, the weight $\displaystyle (\lambda-\alpha_{1})\mid_{T_{1}}=3\omega_{2}$ admits a maximal vector, thus $V^{1}$ has a composition factor isomorphic to $L_{L}(3\omega_{2})$. Similarly, the weight $\displaystyle (\lambda-2\alpha_{1}-\alpha_{2})\mid_{T_{1}}=2\omega_{2}$ admits a maximal vector in $V^{2}$, thus $V^{2}$ has a composition factor isomorphic to $L_{L}(2\omega_{2})$. Lastly, the weight $\displaystyle (\lambda- 3\alpha_{1}-\alpha_{2})\mid_{T_{1}}=3\omega_{2}$ admits a maximal vector in $V^{3}$, thereby $V^{3}$ has a composition factor isomorphic to $L_{L}(3\omega_{2})$. By dimensional considerations, it follows that
\begin{equation}\label{DecompVC2om1+2om2p=7}
V\mid_{[L,L]}\cong L_{L}(2\omega_{2})\oplus L_{L}(3\omega_{2})\oplus L_{L}(2\omega_{2}) \oplus L_{L}(3\omega_{2})\oplus L_{L}(2\omega_{2}) \oplus L_{L}(3\omega_{2})\oplus L_{L}(2\omega_{2}).
\end{equation}

We start with the semisimple elements. Let $s\in T\setminus \ZG(G)$. If $\dim(V^{i}_{s}(\mu))=\dim(V^{i})$ for some eigenvalue $\mu$ of $s$ on $V$, where $0\leq i\leq 6$, then $s\in \ZG(L)^{\circ}\setminus \ZG(G)$. In this case, as $s$ acts on each $V^{i}$ as scalar multiplication by $c^{3-i}$ and $c\neq 1$, we determine that $\dim(V_{s}(\mu))\leq 12$, where equality holds for $c=-1$ and $\mu=\pm 1$. We thus assume that $\dim(V^{i}_{s}(\mu))<\dim(V^{i})$ for all eigenvalues $\mu$ of $s$ on $V$ and all $0\leq i\leq 6$. We write $s=z\cdot h$, where $z\in \ZG(L)^{\circ}$ and $h\in [L,L]$, and we have $\scale[0.9]{\displaystyle \dim(V_{s}(\mu))\leq \sum_{i=0}^{6}\dim(V^{i}_{h}(\mu_{h}^{i}))}$, where $\dim(V^{i}_{h}(\mu_{h}^{i}))<\dim(V^{i})$ for all eigenvalues $\mu_{h}^{i}$ of $h$ on $V^{i}$. Assume there exist $(s,\mu)\in T\setminus \ZG(G)\times k^{*}$ with $\dim(V_{s}(\mu))>12$. Then, as $V$ is a self-dual $kG$-module and $\dim(V)=24$, it follows that $\mu=\pm 1$. Moreover, since $\dim(V_{s}(\pm 1))>12$, by Propositions \ref{PropositionAlsymmetric} and \ref{PropositionAlsymmcube}, it follows that there exist at least six $V^{i}$'s such that $\dim(V^{i}_{s}(\pm 1))=2$. Furthermore, as $V^{6-i}\cong (V^{i})^{*}$ for all $0\leq i\leq 6$, we determine that $\dim(V^{i}_{s}(\pm 1))=2$ for all $i\neq 3$. Let $z=h_{\alpha_{1}}(c)h_{\alpha_{2}}(c)$ and also let $\mu_{h}^{i}$, $0\leq i\leq 6$, be the eigenvalue of $h$ on $V^{i}$ with the property that $\mu=c^{3-i}\mu_{h}^{i}$. We have that $\dim(V^{i}_{h}(\mu_{h}^{i}))=2$ for all $i\neq 3$. In particular, by Proposition \ref{PropositionAlsymmetric}, we get $\mu_{h}^{2}=-1$, thereby $c=-\mu$, i.e. $c=\mp 1$. Similarly, by the proof of Proposition \ref{PropositionAlsymmcube}, we have $\mu_{h}^{1}=d^{\pm 1}$, where $d^{2}=-1$, thereby $\mu=c^{2}d^{\pm 1}=d^{\pm 1}$, contradicting the fact that $\mu=\pm 1$. This shows that there do not exist pairs $(s,\mu)\in T\setminus \ZG(G)\times k^{*}$ with $\dim(V_{s}(\mu))>12$. It follows that $\scale[0.9]{\displaystyle \max_{s\in T\setminus\ZG(G)}\dim(V_{s}(\mu))}=12$.

We now focus on the unipotent elements. By Lemma \ref{uniprootelems}, we have $\scale[0.9]{\displaystyle \max_{u\in G_{u}\setminus \{1\}}\dim(V_{u}(1))=\max_{i=1,2}\dim(V_{x_{\alpha_{i}}(1)}(1))}$. For $\dim(V_{x_{\alpha_{2}}(1)}(1))$, we see that by \eqref{DecompVC2om1+2om2p=7} and Propositions \ref{PropositionAlsymmetric} and \ref{PropositionAlsymmcube}, we have $\dim(V_{x_{\alpha_{2}}(1)}(1))=7$. To calculate $\dim(V_{x_{\alpha_{1}}(1)}(1))$, we need to know the structure of $V\mid_{[L_{2},L_{2}]}$. By Lemma \ref{weightlevelCl}, we have $e_{2}(\lambda)=5$, therefore $\displaystyle V\mid_{[L_{2},L_{2}]}=V_{2}^{0}\oplus\cdots \oplus V_{2}^{5}$, where $\scale[0.9]{V_{2}^{i}=\displaystyle \bigoplus_{\gamma \in \mathbb{N}\Delta_{2}}V_{\lambda-i \alpha_{2}-\gamma}}$. By \cite[Proposition]{Smith_82} and Lemma \ref{dualitylemma}, we have $V_{2}^{0}\cong L_{L_{2}}(\omega_{1})$ and $V_{2}^{5}\cong L_{L_{2}}(\omega_{1})$. Now, in $V_{2}^{1}$, the weight $\displaystyle (\lambda-\alpha_{2})\mid_{T_{2}}=3\omega_{1}$ admits a maximal vector, thus $V_{2}^{1}$ has a composition factor isomorphic to $L_{L_{2}}(3\omega_{1})$. Similarly, the weight $\displaystyle (\lambda-2\alpha_{2})\mid_{T_{2}}=5\omega_{1}$ admits a maximal vector in $V_{2}^{2}$, thus $V_{2}^{2}$ has a composition factor isomorphic to $L_{L_{2}}(5\omega_{1})$. As $\dim(L_{L_{2}}(5\omega_{1}))=6$, we have $V_{2}^{2}\cong L_{L_{2}}(5\omega_{1})$, $V_{2}^{3}\cong L_{L_{2}}(5\omega_{1})$, $V_{2}^{1}\cong L_{L_{2}}(3\omega_{1})$, $V_{2}^{4}\cong L_{L_{2}}(3\omega_{1})$, and so
$$V\mid_{[L_{2},L_{2}]}\cong L_{L_{2}}(\omega_{1})\oplus L_{L_{2}}(3\omega_{1})\oplus L_{L_{2}}(5\omega_{1})  \oplus L_{L_{2}}(5\omega_{1}) \oplus L_{L_{2}}(3\omega_{1}) \oplus L_{L_{2}}(\omega_{1}).$$
Using Propositions \ref{PropositionAlnatural} and \ref{A1mom1} we determine that $\dim(V_{x_{\alpha_{1}}(1)}(1))=6$. To conclude, we have shown that $\scale[0.9]{\displaystyle \max_{u\in G_{u}\setminus \{1\}}\dim(V_{u}(1))}=7$ and, as $\scale[0.9]{\displaystyle \max_{s\in T\setminus\ZG(G)}\dim(V_{s}(\mu))}=12$, we have $\nu_{G}(V)=12$.
\end{proof}

\begin{prop}\label{C23om2p=7}
Let $p\neq 2,3$, $\ell=2$ and $V=L_{G}(3\omega_{2})$. Then $\nu_{G}(V)=10-\varepsilon_{p}(7)$.  Moreover, we have $\scale[0.9]{\displaystyle \max_{u\in G_{u}\setminus \{1\}}\dim(V_{u}(1))}\leq 10+2\varepsilon_{p}(5)-3\varepsilon_{p}(7)$, where equality holds for $p\neq 5$, and $\scale[0.9]{\displaystyle \max_{s\in T\setminus\ZG(G)}\dim(V_{s}(\mu))}= 20-4\varepsilon_{p}(7)$.
\end{prop}

\begin{proof}
Set $\lambda=3\omega_{2}$ and $L=L_{1}$. By Lemma \ref{weightlevelCl}, we have $e_{1}(\lambda)=6$, therefore $\displaystyle V\mid_{[L,L]}=V^{0}\oplus \cdots \oplus V^{6}$. By \cite[Proposition]{Smith_82} and Lemma \ref{dualitylemma}, we have $V^{0}\cong L_{L}(3\omega_{2})$ and $V^{6}\cong L_{L}(3\omega_{2})$. Now, the weight $\displaystyle (\lambda-\alpha_{1}-\alpha_{2})\mid_{T_{1}}=2\omega_{2}$ admits a maximal vector in $V^{1}$, thus $V^{1}$ has a composition factor isomorphic to $L_{L}(2\omega_{2})$. The weight $\displaystyle (\lambda-2\alpha_{1}-\alpha_{2})\mid_{T_{1}}=3\omega_{2}$ admits a maximal vector in $V^{2}$, thereby $V^{2}$ has a composition factor isomorphic to $L_{L}(3\omega_{2})$. Moreover, the weight $\displaystyle (\lambda-2\alpha_{1}-2\alpha_{2})\mid_{T_{1}}=\omega_{2}$ occurs with multiplicity $2-\varepsilon_{p}(7)$ and is a sub-dominant weight in the composition factor of $V^{2}$ isomorphic to $L_{L}(3\omega_{2})$, in which it has multiplicity $1$. Lastly, the weight $\displaystyle (\lambda-3\alpha_{1}-2\alpha_{2})\mid_{T_{1}}=2\omega_{2}$ admits a maximal vector in $V^{3}$, thus $V^{3}$ has a composition factor isomorphic to $L_{L}(2\omega_{2})$. Further, the weight $\displaystyle (\lambda-3\alpha_{1}-3\alpha_{2})\mid_{T_{1}}=0$, which occurs with multiplicity $2-\varepsilon_{p}(7)$ in $V^{3}$, is a sub-dominant weight in the composition factor of $V^{3}$ isomorphic to $L_{L}(2\omega_{2})$, in which it has multiplicity $1$. Thus, by dimensional considerations and \cite[II.$2.14$]{Jantzen_2007representations}, we have
\begin{equation}\label{DecompVC23om2p=7}
V\mid_{[L,L]}\cong L_{L}(3\omega_{2})^{4}\oplus L_{L}(2\omega_{2})^{3}\oplus L_{L}(\omega_{2})^{2-2\varepsilon_{p}(7)} \oplus L_{L}(0)^{1-\varepsilon_{p}(7)}.
\end{equation}

We start with the semisimple elements. Let $s\in T\setminus \ZG(G)$. If $\dim(V^{i}_{s}(\mu))=\dim(V^{i})$ for some eigenvalue $\mu$ of $s$ on $V$, where $0\leq i\leq 6$, then $s\in \ZG(L)^{\circ}\setminus \ZG(G)$. In this case, as $s$ acts on each $V^{i}$ as scalar multiplication by $c^{3-i}$ and $c\neq 1$, we determine that $\dim(V_{s}(\mu))\leq  20-4\varepsilon_{p}(7)$, where equality holds for $c=-1$ and $\mu=-1$. We thus assume that $\dim(V^{i}_{s}(\mu))<\dim(V^{i})$ for all eigenvalues $\mu$ of $s$ on $V$ and all $0\leq i\leq 6$. We write $s=z\cdot h$, where $z\in \ZG(L)^{\circ}$ and $h\in [L,L]$, and, by \eqref{DecompVC23om2p=7} and Propositions \ref{PropositionAlsymmetric} and \ref{PropositionAlsymmcube}, we determine that $\dim(V_{s}(\mu))\leq 17-3\varepsilon_{p}(7)$ for all eigenvalues $\mu$ of $s$ on $V$. Therefore $\scale[0.9]{\displaystyle \max_{s\in T\setminus\ZG(G)}\dim(V_{s}(\mu))}= 20-4\varepsilon_{p}(7)$.

We now focus on the unipotent elements. By Lemma \ref{uniprootelems}, we have $\scale[0.9]{\displaystyle \max_{u\in G_{u}\setminus \{1\}}\dim(V_{u}(1))=\max_{i=1,2}\dim(V_{x_{\alpha_{i}}(1)}(1))}$. For $\dim(V_{x_{\alpha_{2}}(1)}(1))$, we see that by \eqref{DecompVC23om2p=7} and Propositions \ref{PropositionAlsymmetric} and \ref{PropositionAlsymmcube}, we have $\dim(V_{x_{\alpha_{2}}(1)}(1))=10-3\varepsilon_{p}(7)$.To calculate $\dim(V_{x_{\alpha_{1}}(1)}(1))$, we need to know the structure of $V\mid_{[L_{2},L_{2}]}$. By Lemma \ref{weightlevelCl}, we have $e_{2}(\lambda)=6$, therefore $\displaystyle V\mid_{[L_{2},L_{2}]}=V_{2}^{0}\oplus \cdots \oplus V_{2}^{6}$, where $\scale[0.9]{V_{2}^{i}=\displaystyle \bigoplus_{\gamma \in \mathbb{N}\Delta_{2}}V_{\lambda-i \alpha_{2}-\gamma}}$. By \cite[Proposition]{Smith_82} and Lemma \ref{dualitylemma}, we have $V_{2}^{0}\cong L_{L_{2}}(0)$ and $V_{2}^{6}\cong  L_{L_{2}}(0)$. Now, the weight $\displaystyle (\lambda-\alpha_{2})\mid_{T_{2}}=2\omega_{1}$ admits a maximal vector in $V_{2}^{1}$, thus $V_{2}^{1}$ has a composition factor isomorphic to $L_{L_{2}}(2\omega_{1})$. Similarly, the weight $\displaystyle (\lambda-2\alpha_{2})\mid_{T_{2}}=4\omega_{1}$ admits a maximal vector in $V_{2}^{2}$, thereby $V_{2}^{2}$ has a composition factor isomorphic to $L_{L_{2}}(4\omega_{1})$. Further, the weight $\displaystyle (\lambda-2\alpha_{1}-\alpha_{2})\mid_{T_{2}}=0$ occurs with multiplicity $2-\varepsilon_{p}(7)$ and is a sub-dominant weight in the composition factor of $V_{2}^{2}$ isomorphic to $L_{L_{2}}(4\omega_{1})$, in which it has multiplicity $1$. Lastly, the weight $\displaystyle (\lambda-3\alpha_{2})\mid_{T_{2}}=6\omega_{1}$ admits a maximal vector in $V_{2}^{3}$, thus $V_{2}^{3}$ has a composition factor isomorphic to $L_{L_{2}}(6\omega_{1})$. Moreover, the weight $\displaystyle (\lambda-2\alpha_{1}-3\alpha_{2})\mid_{T_{2}}=2\omega_{1}$ occurs with multiplicity $2-\varepsilon_{p}(7)$ and is a sub-dominant weight in the composition factor of $V_{2}^{3}$ isomorphic to $L_{L_{2}}(6\omega_{1})$, in which it has multiplicity $1$, if and only if $p\neq 5$. Thus, by dimensional considerations and \cite[II.$2.14$]{Jantzen_2007representations}, we determine that if $p\neq 5$, we have $V_{2}^{3}\cong L_{L_{2}}(6\omega_{1})\oplus L_{L_{2}}(2\omega_{1})^{1-\varepsilon_{p}(7)}$,  while, if $p\neq 5$, then $V_{2}^{3}$ has three composition factors: one isomorphic to $L_{L_{2}}(\omega_{1})\otimes L_{L_{2}}(\omega_{1})^{(5)}$ and $2$ isomorphic to $L_{L_{2}}(2\omega_{1})$. Further, in both cases, $V_{2}^{1}\cong L_{L_{2}}(2\omega_{1})$, $V_{2}^{5}\cong L_{L_{2}}(2\omega_{1})$, $V_{2}^{2}\cong L_{L_{2}}(4\omega_{1})\oplus L_{L_{2}}(0)^{1-\varepsilon_{p}(7)}$ and $V_{2}^{4}\cong L_{L_{2}}(4\omega_{1})\oplus L_{L_{2}}(0)^{1-\varepsilon_{p}(7)}$. Now, if $p\neq 5$, then, by Propositions \ref{PropositionAlsymmetric} and \ref{A1mom1}, we have $\dim(V_{x_{\alpha_{2}}(1)}(1))=10-3\varepsilon_{p}(7)$. Similarly, if $p=5$, using the same results and the fact that $x_{\alpha_{1}}(1)$ acts as $J_{2}$ on $L_{L_{2}}(\omega_{1})$, we determine that $\dim(V_{x_{\alpha_{2}}(1)}(1))\leq 12$. To conclude, we have shown that $\scale[0.9]{\displaystyle \max_{u\in G_{u}\setminus \{1\}}\dim(V_{u}(1))}\leq 10-3\varepsilon_{p}(7)+2\varepsilon_{p}(5)$, where equality holds for $p\neq 5$, and, as $\scale[0.9]{\displaystyle \max_{s\in T\setminus\ZG(G)}\dim(V_{s}(\mu))}= 20-4\varepsilon_{p}(7)$, we have $\nu_{G}(V)=10-\varepsilon_{p}(7)$.
\end{proof}

\begin{prop}\label{C22om1+om2p=3}
Let $p\neq 2$, $\ell=2$ and $V=L_{G}(2\omega_{1}+ \omega_{2})$. Then $\nu_{G}(V)=15-6\varepsilon_{p}(3)$.  Moreover, we have $\scale[0.9]{\displaystyle \max_{u\in G_{u}\setminus \{1\}}\dim(V_{u}(1))}\leq 15-4\varepsilon_{p}(3)$ and $\scale[0.9]{\displaystyle \max_{s\in T\setminus\ZG(G)}\dim(V_{s}(\mu))}=20-4\varepsilon_{p}(3)$.
\end{prop}

\begin{proof}
Set $\lambda=2\omega_{1}+\omega_{2}$ and $L=L_{2}$. By Lemma \ref{weightlevelCl}, we have $e_{2}(\lambda)=4$, therefore $\displaystyle V\mid_{[L,L]}=V^{0}\oplus \cdots \oplus V^{4}$. By \cite[Proposition]{Smith_82} and Lemma \ref{dualitylemma}, we have $V^{0}\cong L_{L}(2\omega_{1})$ and $V^{4}\cong L_{L}(2\omega_{1})$. Now, the weight $\displaystyle(\lambda-\alpha_{2})\mid_{T_{2}}=4\omega_{1}$ admits a maximal vector in $V^{1}$, thus $V^{1}$ has a composition factor isomorphic to $L_{L}(4\omega_{1})$. Further, the weight $\displaystyle (\lambda-\alpha_{1}-\alpha_{2})\mid_{T_{2}}=2\omega_{1}$ occurs with multiplicity $2-\varepsilon_{p}(3)$ in $V^{1}$ and is a sub-dominant weight in the composition factor of $V^{1}$ isomorphic to $L_{L}(4\omega_{1})$ in which it has multiplicity $1$. Moreover, we note that the weight $\displaystyle (\lambda-2\alpha_{1}-\alpha_{2})\mid_{T_{2}}=0$ occurs with multiplicity $3-\varepsilon_{p}(3)$ in $V^{1}$. Similarly, the weight $\displaystyle (\lambda-\alpha_{1}-2\alpha_{2})\mid_{T_{2}}=4\omega_{1}$ admits a maximal vector in $V^{2}$, thus $V^{2}$ has a composition factor isomorphic to $L_{L}(4\omega_{1})$. Moreover, the weight $\displaystyle (\lambda-2\alpha_{1} -2\alpha_{2})\mid_{T_{2}}=2\omega_{1}$ occurs with multiplicity $3-\varepsilon_{p}(3)$ in $V^{2}$ and is a sub-dominant weight in the composition factor of $V^{2}$ isomorphic to $L_{L}(4\omega_{1})$, in which it has multiplicity $1$. Thus, by dimensional considerations, we determine that $V^{2}$ has exactly $3-\varepsilon_{p}(3)$ composition factors: one isomorphic to $L_{L}(4\omega_{1})$ and $2-\varepsilon_{p}(3)$ to $L_{L}(2\omega_{1})$. Further, $V^{1}$ and $V^{3}$ each has exactly three composition factors: one isomorphic to $L_{L}(4\omega_{1})$, $1-\varepsilon_{p}(3)$ to $L_{L}(2\omega_{1})$ and  $1+\varepsilon_{p}(3)$ to $L_{L}(0)$. 

We start with the semisimple elements. Let $s\in T\setminus \ZG(G)$. If $\dim(V^{i}_{s}(\mu))=\dim(V^{i})$ for some eigenvalue $\mu$ of $s$ on $V$, where $0\leq i\leq 4$, then $s\in \ZG(L)^{\circ}\setminus \ZG(G)$. In this case, as $s$ acts on each $V^{i}$ as scalar multiplication by $c^{4-2i}$ and $c^{2}\neq 1$, we determine that $\dim(V_{s}(\mu))\leq 18-5\varepsilon_{p}(3)$ for all eigenvalues $\mu$ of $s$ on $V$. We thus assume that $\dim(V^{i}_{s}(\mu))<\dim(V^{i})$ for all eigenvalues $\mu$ of $s$ on $V$ and all $0\leq i\leq 4$. We write $s=z'\cdot h'$, where $z'\in \ZG(L)^{\circ}$ and $h'\in [L,L]$, and, we have $\scale[0.9]{\displaystyle \dim(V_{s}(\mu))\leq \sum_{i=0}^{4}\dim(V^{i}_{h'}(\mu^{i}_{h'}))}$, where $\dim(V^{i}_{h'}(\mu^{i}_{h'}))<\dim(V^{i})$ for all eigenvalues $\mu^{i}_{h'}$ of $h'$ on $V^{i}$. To calculate $\dim(V^{1}_{h'}(\mu^{1}_{h'}))$, we note that the eigenvalues of $h'$ on $V^{1}$ are $d^{\pm 4}$ each with multiplicity at least $1$; $d^{\pm 2}$ each with multiplicity at least $2-\varepsilon_{p}(3)$; and $1$ with multiplicity at least $3-\varepsilon_{p}(3)$, where $d^{2}\neq 1$. It follows that $\dim(V^{1}_{h'}(\mu^{1}_{h'}))\leq 5-\varepsilon_{p}(3)$ for all eigenvalues $\mu^{1}_{h'}$ of $h'$ on $V^{1}$, and, as $V^{3}\cong (V^{1})^{*}$, we also have $\dim(V^{3}_{h'}(\mu^{3}_{h'}))\leq 5-\varepsilon_{p}(3)$ for all eigenvalues $\mu^{3}_{h'}$ of $h'$ on $V^{3}$. Similarly, to calculate $\dim(V^{2}_{h'}(\mu^{2}_{h'}))$, we see that the eigenvalues of $h'$ on $V^{2}$, are $d^{\pm 4}$ each with multiplicity at least $1$; $d^{\pm 2}$ each with multiplicity at least $3-\varepsilon_{p}(3)$; and $1$ with multiplicity at least $3-2\varepsilon_{p}(3)$, where $d^{2}\neq 1$. It follows that $\dim(V^{2}_{h'}(\mu^{2}_{h'}))\leq 6-2\varepsilon_{p}(3)$ for all eigenvalues $\mu^{2}_{h'}$ of $h'$ on $V^{2}$. Thus, by Proposition \ref{PropositionAlsymmetric}, we determine that $\dim(V_{s}(\mu))\leq 20-4\varepsilon_{p}(3)$ for all eigenvalues $\mu$ of $s$ on $V$. Further, for $s=\diag(1,-1,-1,1)\in T\setminus \ZG(G)$, using Propositions \ref{PropositionAlsymmetric} and \ref{A1mom1}, one shows that $\dim(V_{s}(-1))=20-4\varepsilon_{p}(3)$. Therefore, $\scale[0.9]{\displaystyle \max_{s\in T\setminus\ZG(G)}\dim(V_{s}(\mu))}=20-4\varepsilon_{p}(3)$.

We now focus on the unipotent elements. By Lemma \ref{uniprootelems}, we have $\scale[0.9]{\displaystyle \max_{u\in G_{u}\setminus \{1\}}\dim(V_{u}(1))=\max_{i=1,2}\dim(V_{x_{\alpha_{i}}(1)}(1))}$. For $\dim(V_{x_{\alpha_{1}}(1)}(1))$, by the structure of $V\mid_{[L,L]}$ and Propositions \ref{PropositionAlsymmetric} and \ref{A1mom1}, we have $\dim(V_{x_{\alpha_{1}}(1)}(1))\leq 11+2\varepsilon_{p}(3)$. Note that, for $p=3$, we have used the fact that $x_{\alpha_{1}}(1)$ acts on $L_{L}(4\omega_{1})$ as $J_{2}^{2}$. To calculate $\dim(V_{x_{\alpha_{2}}(1)}(1))$, we need to know $V\mid_{[L_{1},L_{1}]}$. By Lemma \ref{weightlevelCl}, we have $e_{1}(\lambda)=6$, therefore $\displaystyle V\mid_{[L_{1},L_{1}]}=V_{1}^{0}\oplus \cdots \oplus V_{1}^{6}$, where $\scale[0.9]{V_{1}^{i}=\displaystyle \bigoplus_{\gamma \in \mathbb{N}\Delta_{1}}V_{\lambda-i \alpha_{1}-\gamma}}$. By \cite[Proposition]{Smith_82} and Lemma \ref{dualitylemma}, we have $V_{1}^{0}\cong L_{L_{1}}(\omega_{2})$ and $V_{1}^{6}\cong  L_{L_{1}}(\omega_{2})$. Now, the weight $\displaystyle(\lambda-\alpha_{1})\mid_{T_{1}}=2\omega_{2}$ admits a maximal vector in $V_{1}^{1}$, thus $V_{1}^{1}$ has a composition factor isomorphic to $L_{L_{1}}(2\omega_{2})$. Moreover, the weight $\displaystyle (\lambda-\alpha_{1} -\alpha_{2})\mid_{T_{1}}=0$ occurs with multiplicity $2-\varepsilon_{p}(3)$ and is a sub-dominant weight in the composition factor of $V_{1}^{1}$ isomorphic to $L_{L_{1}}(2\omega_{2})$, in which it has multiplicity $1$. Now, in $V_{1}^{2}$ the weight $\displaystyle (\lambda-2\alpha_{1})\mid_{T_{1}}=3\omega_{2}$ admits a maximal vector, thus $V_{1}^{2}$ has a composition factor isomorphic to $L_{L_{1}}(3\omega_{2})$. Further, the weight $\displaystyle (\lambda-2\alpha_{1} -\alpha_{2})\mid_{T_{1}}=\omega_{2}$ occurs with multiplicity $3-\varepsilon_{p}(3)$ and is a sub-dominant weight in the composition factor of $V_{1}^{2}$ isomorphic to $L_{L_{1}}(3\omega_{2})$ if and only if $p\neq 3$. Lastly, in $V_{1}^{3}$, the weight $\displaystyle (\lambda-3\alpha_{1} -\alpha_{2})\mid_{T_{1}}=2\omega_{2}$ occurs with multiplicity $2-\varepsilon_{p}(3)$ and admits a maximal vector, thus $V_{1}^{3}$ has $2-\varepsilon_{p}(3)$ composition factors isomorphic to $L_{L_{1}}(2\omega_{2})$. Moreover, the weight $\displaystyle (\lambda-3\alpha_{1} -2\alpha_{2})\mid_{T_{1}}=0$ has multiplicity $3-2\varepsilon_{p}(3)$ and is a sub-dominant weight of multiplicity $1$ in each of the $2-\varepsilon_{p}(3)$ composition factors of $V_{1}^{3}$ isomorphic to $L_{L_{1}}(2\omega_{2})$. Therefore, by dimensional considerations, we determine that $V_{1}^{3}$ has exactly $3-2\varepsilon_{p}(3)$ composition factors: $2-\varepsilon_{p}(3)$ isomorphic to $L_{L_{1}}(2\omega_{2})$ and $1-\varepsilon_{p}(3)$ to $L_{L_{1}}(0)$. Further, $V_{1}^{2}$ and $V_{1}^{4}$ each has $3$ composition factors: one isomorphic to $L_{L_{1}}(3\omega_{2})$ and two to $L_{L_{1}}(\omega_{2})$, and, lastly, $V_{1}^{1}$ and $V_{1}^{5}$ each has $2-\varepsilon_{p}(3)$ composition factors: one isomorphic to $L_{L_{1}}(2\omega_{2})$ and $1-\varepsilon_{p}(3)$ to $L_{L_{1}}(0)$. Having determined the structure $V\mid_{[L_{1},L_{1}]}$, we use Propositions \ref{PropositionAlnatural}, \ref{PropositionAlsymmetric} and \ref{PropositionAlsymmcube}, to show that $\dim(V_{x_{\alpha_{2}}(1)}(1))\leq 15-4\varepsilon_{p}(3)$. In conclusion, we have $\scale[0.9]{\displaystyle \max_{u\in G_{u}\setminus \{1\}}\dim(V_{u}(1))}\leq 15-4\varepsilon_{p}(3)$ and, as $\scale[0.9]{\displaystyle \max_{s\in T\setminus\ZG(G)}\dim(V_{s}(\mu))}=20-4\varepsilon_{p}(3)$, it follows that $\nu_{G}(V)=15-6\varepsilon_{p}(3)$.
\end{proof}

\begin{prop}\label{Cl_oml_p=2}
Let $p=2$, $\ell\geq 4$ and $V=L_{G}(\omega_{\ell})$. Then $\nu_{G}(V)=2^{\ell-2}$, $\scale[0.9]{\displaystyle \max_{u\in G_{u}\setminus \{1\}}\dim(V_{u}(1))}=3\cdot 2^{\ell-2}$ and $\scale[0.9]{\displaystyle \max_{s\in T\setminus\ZG(G)}\dim(V_{s}(\mu))}=2^{\ell-1}$.
\end{prop}

\begin{proof}
Set $\lambda=\omega_{\ell}$ and $L=L_{1}$. We have $e_{1}(\lambda)=2$, see Lemma \ref{weightlevelCl}, therefore $\displaystyle V\mid_{[L,L]}=V^{0} \oplus V^{1}\oplus V^{2}$. Now, as $p=2$, we have $V^{1}=\{0\}$, and by \cite[Proposition]{Smith_82} and Lemma \ref{dualitylemma}, we deduce that
\begin{equation}\label{DecompVClomlp=2}
V\mid_{[L,L]}\cong L_{L}(\omega_{\ell}) \oplus L_{L}(\omega_{\ell}).
\end{equation}

We start with the semisimple elements. Let $s\in T\setminus \ZG(G)$. If $\dim(V^{i}_{s}(\mu))=\dim(V^{i})$ for some eigenvalue $\mu$ of $s$ on $V$, where $i=0, 2$, then $s\in \ZG(L)^{\circ}\setminus \ZG(G)$. In this case, as $s$ acts on each $V^{i}$ as scalar multiplication by $c^{1-i}$ and $c\neq 1$, we determine that $\dim(V_{s}(\mu))=2^{\ell-1}$ and $\mu=c^{\pm 1}$. We thus assume that $\dim(V^{i}_{s}(\mu))<\dim(V^{i})$ for all eigenvalues $\mu$ of $s$ on $V$ and for both $i=0, 2$. We write $s=z\cdot h$, where $z\in \ZG(L)^{\circ}$ and $h\in [L,L]$. By \eqref{DecompVClomlp=2}, recursively and using Proposition \ref{C3om3} for the base case of $\ell=4$, one shows that $\dim(V_{s}(\mu))\leq 2^{\ell-1}$ for all eigenvalues $\mu$ of $s$ on $V$. Therefore, $\scale[0.9]{\displaystyle \max_{s\in T\setminus\ZG(G)}\dim(V_{s}(\mu))}=2^{\ell-1}$. 

We focus on the unipotent elements. By Lemma \ref{uniprootelems}, we have $\scale[0.85]{\displaystyle \max_{u\in G_{u}\setminus \{1\}}\dim(V_{u}(1))=\max_{i=\ell-1,\ell}\dim(V_{x_{\alpha_{i}}(1)}(1))}$. By \eqref{DecompVClomlp=2}, recursively and using Proposition \ref{C3om3} for the base case of $\ell=4$, one shows that $\dim(V_{x_{\alpha_{i}}(1)}(1))\leq 3\cdot  2^{\ell-2}$, where equality holds for $x_{\alpha_{\ell-1}}(1)$. In conclusion, we have $\scale[0.9]{\displaystyle \max_{u\in G_{u}\setminus \{1\}}\dim(V_{u}(1))}=3\cdot 2^{\ell-2}$ and, as $\scale[0.9]{\displaystyle \max_{s\in T\setminus\ZG(G)}\dim(V_{s}(\mu))}=2^{\ell-1}$, it follows that $\nu_{G}(V)=2^{\ell-2}$.
\end{proof}

\begin{prop}\label{C3om1+om3}
Let $\ell=3$ and $V=L_{G}(\omega_{1}+\omega_{3})$. Then $\nu_{G}(V)=30-5\varepsilon_{p}(3)-10\varepsilon_{p}(2)$. Moreover, we have $\scale[0.9]{\displaystyle \max_{u\in G_{u}\setminus \{1\}}\dim(V_{u}(1))}\leq 40-9\varepsilon_{p}(3)-12\varepsilon_{p}(2)$, where equality holds for $p\neq 5$, and $\scale[0.9]{\displaystyle \max_{s\in T\setminus\ZG(G)}\dim(V_{s}(\mu))}=40-8\varepsilon_{p}(3)-20\varepsilon_{p}(2)$. 
\end{prop}

\begin{proof}
Set $\lambda=\omega_{1}+\omega_{3}$ and $L=L_{1}$. By Lemma \ref{weightlevelCl}, we have $e_{1}(\lambda)=4$, therefore $\displaystyle V\mid_{[L,L]}=V^{0} \oplus \cdots\oplus V^{4}$. By \cite[Proposition]{Smith_82} and Lemma \ref{dualitylemma}, we have $V^{0}\cong L_{L}(\omega_{3})$ and $V^{4}\cong L_{L}(\omega_{3})$. Since the weight $\displaystyle (\lambda-\alpha_{1})\mid_{T_{1}}=\omega_{2}+\omega_{3}$ admits a maximal vector in $V^{1}$,  it follows that $V^{1}$ has a composition factor isomorphic to $L_{L}(\omega_{2}+\omega_{3})$. Further, the weight $\displaystyle (\lambda-\alpha_{1}-\alpha_{2}-\alpha_{3})\mid_{T_{1}}=\omega_{2}$ occurs with multiplicity $3-\varepsilon_{p}(6)$ and is a sub-dominant weight in the composition factor of $V^{1}$ isomorphic to $L_{L}(\omega_{2}+\omega_{3})$, in which it has multiplicity $2-\varepsilon_{p}(5)$. Now, if $p=2$, the weight $\displaystyle(\lambda- 2\alpha_{1}-2\alpha_{2}-\alpha_{3})\mid_{T_{1}}=\omega_{3}$ occurs with multiplicity $2$ and admits a maximal vector in $V^{2}$. In this case, as $\dim(V^{2})\leq 8$, we determine that $V^{2}\cong L_{L}(\omega_{3})^{2}$, see \cite[II.$2.12$]{Jantzen_2007representations}. If $p\neq 2$, then, the weight $\displaystyle(\lambda- 2\alpha_{1}-\alpha_{2}-\alpha_{3})\mid_{T_{1}}=2\omega_{2}$ admits a maximal vector, thus $V^{2}$ has a composition factor isomorphic to $L_{L}(2\omega_{2})$. Moreover, the weight $\displaystyle(\lambda- 2\alpha_{1}-2\alpha_{2}-\alpha_{3})\mid_{T_{1}}=\omega_{3}$ occurs with multiplicity $3-\varepsilon_{p}(3)$ and is a sub-dominant weight in the composition factor of $V^{2}$ isomorphic to $L_{L}(2\omega_{2})$ in which it has multiplicity $1$. Thus, as $\dim(V^{2})\leq 20-5\varepsilon_{p}(3)$, we use \cite[II.$2.14$]{Jantzen_2007representations} to determine that $V^{2}\cong L_{L}(2\omega_{2})\oplus L_{L}(\omega_{3})^{2-\varepsilon_{p}(3)}$. Further, in both cases, $V^{1}$ and $V^{3}$ each has $2-\varepsilon_{p}(6)+\varepsilon_{p}(5)$ composition factors: one isomorphic to $L_{L}(\omega_{2}+\omega_{3})$ and $1-\varepsilon_{p}(6)+\varepsilon_{p}(5)$ to $L_{L}(\omega_{2})$. 

We start with the semisimple elements. Let $s\in T\setminus \ZG(G)$. If $\dim(V^{i}_{s}(\mu))=\dim(V^{i})$ for some eigenvalue $\mu$ of $s$ on $V$, where $0\leq i\leq 4$, then $s\in \ZG(L)^{\circ}\setminus \ZG(G)$. In this case, as $s$ acts on each $V^{i}$ as scalar multiplication by $c^{2-i}$ and $c\neq 1$, we determine that $\dim(V_{s}(\mu))\leq 40-8\varepsilon_{p}(3)-20\varepsilon_{p}(2)$, where equality holds for $p\neq 2$, $c=-1$ and $\mu=1$, respectively for $p=2$, $c^{3}=1$ and $\mu=c^{\pm 1}$. We thus assume that $\dim(V^{i}_{s}(\mu))<\dim(V^{i})$ for all eigenvalues $\mu$ of $s$ on $V$ and all $0\leq i\leq 4$. We write $s=z\cdot h$, where $z\in \ZG(L)^{\circ}$ and $h\in [L,L]$ and we have $\scale[0.9]{\displaystyle \dim(V_{s}(\mu))\leq \sum_{i=0}^{4}\dim(V^{i}_{h}(\mu^{i}_{h}))}$, where $\dim(V^{i}_{h}(\mu^{i}_{h}))<\dim(V^{i})$ for all eigenvalues $\mu^{i}_{h}$ of $h$ on $V^{i}$. Now, by the structure of $V^{2}$, we have $\dim(V^{2}_{h}(\mu^{2}_{h}))\leq (1-\varepsilon_{p}(2))\dim((L_{L}(2\omega_{2}))_{h}(\mu^{2}_{h}))+(2-\varepsilon_{p}(3))\dim((L_{L}(\omega_{3}))_{h}(\mu^{2}_{h}))$ for all eigenvalues $\mu^{2}_{h}$ of $h$ on $V^{2}$. If $p=2$, then $\dim(V^{2}_{h}(\mu_{h}^{2}))\leq 4$ for all $\mu_{h}^{2}$, by Proposition \ref{PropositionClwedge}. Thus, we let $p\neq 2$. Assume that $h$ is conjugate to either $\diag(1,d,d,d^{-1},d^{-1},1)$ with $d^{2}=-1$, or to $\diag(1,1,-1,-1,1,1)$. In both cases, one uses Propositions \ref{PropositionClwedge} and \ref{PropositionClsymm} to show that $\dim(V^{2}_{h}(\mu_{h}^{2}))\leq 12-4\varepsilon_{p}(3)$ for all eigenvalues $\mu_{h}^{2}$ of $h$ on $V^{2}$. On the other hand, if $h$ is not conjugate to either $\diag(1,1,-1,-1,1,1)$, or $\diag(1,d,d,d^{-1},d^{-1},1)$ with $d^{2}=-1$, then, by the same results, it follows that $\dim(V^{2}_{h}(\mu^{2}_{h}))\leq 11-3\varepsilon_{p}(3)$ for all eigenvalues $\mu_{h}^{2}$ of $h$ on $V^{2}$. Thus, we have shown that $\dim(V^{2}_{h}(\mu^{2}_{h}))\leq 12-4\varepsilon_{p}(3)-8\varepsilon_{p}(2)$ for all all eigenvalues $\mu_{h}^{2}$ of $h$ on $V^{2}$. Lastly, using the structure of $V\mid_{[L,L]}$ and Propositions \ref{PropositionClnatural}, \ref{PropositionClwedge} and \ref{PropositionC2om1+om2}, we determine that $\dim(V_{s}(\mu))\leq 40-8\varepsilon_{p}(3)-20\varepsilon_{p}(2)$ for all eigenvalues $\mu$ of $s$ on $V$. Therefore, $\scale[0.9]{\displaystyle \max_{s\in T\setminus\ZG(G)}\dim(V_{s}(\mu))}\leq 40-8\varepsilon_{p}(3)-20\varepsilon_{p}(2)$.

For the unipotent elements, by Lemma \ref{uniprootelems}, we have $\scale[0.9]{\displaystyle \max_{u\in G_{u}\setminus \{1\}}\dim(V_{u}(1))=\max_{i=2,3}\dim(V_{x_{\alpha_{i}}(1)}(1))}$, and, by the structure of $V\mid_{[L,L]}$ and Propositions \ref{PropositionClnatural}, \ref{PropositionClwedge}, \ref{PropositionClsymm} and \ref{PropositionC2om1+om2}, it follows that $\dim(V_{x_{\alpha_{i}}(1)}(1))\leq (4-\varepsilon_{p}(3))\dim((L_{L}(\omega_{3}))_{x_{\alpha_{i}}(1)}(1))+(1-\varepsilon_{p}(2))\dim((L_{L}(2\omega_{2}))_{x_{\alpha_{i}}(1)}(1))+2\dim((L_{L}(\omega_{2}+\omega_{3}))_{x_{\alpha_{i}}(1)}(1))+(2-2\varepsilon_{p}(6)+2\varepsilon_{p}(5))\dim((L_{L}(\omega_{2}))_{x_{\alpha_{i}}(1)}(1))\leq 40-9\varepsilon_{p}(3)-12\varepsilon_{p}(2)$, where equality holds for $x_{\alpha_{2}}(1)$ and $p=2$, or for $x_{\alpha_{3}}(1)$ and $p\neq 2,5$. In conclusion, we have shown that $\scale[0.9]{\displaystyle \max_{u\in G_{u}\setminus \{1\}}\dim(V_{u}(1))}\leq 40-9\varepsilon_{p}(3)-12\varepsilon_{p}(2)$, where equality holds for $p\neq 5$, and as $\scale[0.9]{\displaystyle \max_{s\in T\setminus\ZG(G)}\dim(V_{s}(\mu))}= 40-8\varepsilon_{p}(3)-20\varepsilon_{p}(2)$, we determine that $\nu_{G}(V)=30-5\varepsilon_{p}(3)-10\varepsilon_{p}(2)$.
\end{proof}

\begin{prop}\label{Clom1+omlp=2}
Let $p=2$, $\ell\geq 4$ and $V=L_{G}(\omega_{1}+\omega_{\ell})$. Then $\nu_{G}(V)=(\ell+2)\cdot 2^{\ell-1}$. Moreover, we have $\scale[0.9]{\displaystyle \max_{u\in G_{u}\setminus \{1\}}\dim(V_{u}(1))}=(3\ell-2)\cdot 2^{\ell-1}$ and $\scale[0.9]{\displaystyle \max_{s\in T\setminus\ZG(G)}\dim(V_{s}(\mu))}=(2\ell-1)\cdot 2^{\ell-1}$. 
\end{prop}

\begin{proof}
Set $\lambda=\omega_{1}+\omega_{\ell}$ and $L=L_{1}$. By Lemma \ref{weightlevelCl}, we have $e_{1}(\lambda)=4$, therefore $\displaystyle V\mid_{[L,L]}=V^{0} \oplus \cdots\oplus V^{4}$. By \cite[Proposition]{Smith_82} and Lemma \ref{dualitylemma}, we have $V^{0}\cong L_{L}(\omega_{\ell})$ and $V^{4}\cong L_{L}(\omega_{\ell})$. Since the weight $\displaystyle (\lambda-\alpha_{1})\mid_{T_{1}}=\omega_{2}+\omega_{\ell}$ admits a maximal vector in $V^{1}$,  it follows that $V^{1}$ has a composition factor isomorphic to $L_{L}(\omega_{2}+\omega_{\ell})$. Similarly, in $V^{2}$ the weight $\displaystyle(\lambda- 2\alpha_{1}-\cdots-2\alpha_{\ell-1}-\alpha_{\ell})\mid_{T_{1}}=\omega_{\ell}$ occurs with multiplicity $2$ and admits a maximal vector. Thus, as $\dim(V^{2})\leq 2^{\ell}$, using \cite[II.$2.12$]{Jantzen_2007representations}, we determine that:
\begin{equation}\label{DecompVClom1+omlp=2}
V\mid_{[L,L]}\cong L_{L}(\omega_{\ell}) \oplus L_{L}(\omega_{2}+\omega_{\ell}) \oplus L_{L}(\omega_{\ell})\oplus L_{L}(\omega_{\ell}) \oplus L_{L}(\omega_{2}+\omega_{\ell}) \oplus L_{L}(\omega_{\ell}).
\end{equation}

We start with the semisimple elements. Let $s\in T\setminus \ZG(G)$. If $\dim(V^{i}_{s}(\mu))=\dim(V^{i})$ for some eigenvalue $\mu$ of $s$ on $V$, where $0\leq i\leq 4$, then $s\in \ZG(L)^{\circ}\setminus \ZG(G)$. In this case, as $s$ acts on each $V^{i}$ as scalar multiplication by $c^{2-i}$ and $c\neq 1$, we determine that $\dim(V_{s}(\mu))\leq (2\ell-1)\cdot 2^{\ell-1}$, where equality holds for $c^{3}=1$ and $\mu=c^{\pm 1}$. We thus assume that $\dim(V^{i}_{s}(\mu))<\dim(V^{i})$ for all eigenvalues $\mu$ of $s$ on $V$ and all $0\leq i\leq 4$. We write $s=z\cdot h$, where $z\in \ZG(L)^{\circ}$ and $h\in [L,L]$, and, by \eqref{DecompVClom1+omlp=2} and Proposition \ref{Cl_oml_p=2}, we determine that $\dim(V_{s}(\mu))\leq 4\dim((L_{L}(\omega_{\ell}))_{h}(\mu_{h}))+2\dim((L_{L}(\omega_{2}+\omega_{\ell}))_{h}(\mu_{h}))\leq 2^{\ell}+2\dim((L_{L}(\omega_{2}+\omega_{\ell}))_{h}(\mu_{h}))$. Recursively and using Proposition \ref{C3om1+om3} for the base case of $\ell=4$, we get $\dim(V^{s}(\mu))\leq (\ell-3)\cdot 2^{\ell}+20\cdot 2^{\ell-3}=(2\ell-1)\cdot 2^{\ell-1}$ for all eigenvalues $\mu$ of $s$ on $V$. Therefore, $\scale[0.9]{\displaystyle \max_{s\in T\setminus\ZG(G)}\dim(V_{s}(\mu))}=(2\ell-1)\cdot 2^{\ell-1}$.

For the unipotent elements, by Lemma \ref{uniprootelems}, we have $\scale[0.9]{\displaystyle \max_{u\in G_{u}\setminus \{1\}}\dim(V_{u}(1))=\max_{i=\ell-1,\ell}\dim(V_{x_{\alpha_{i}}(1)}(1))}$, and recursively, using \eqref{DecompVClom1+omlp=2} and Propositions \ref{Cl_oml_p=2} and \ref{C3om1+om3} for the base case of $\ell=4$, we show that $\dim(V_{x_{\alpha_{i}}(1)}(1))$ $\leq 3\cdot 2^{\ell-1}+2\dim((L_{L}(\omega_{2}+\omega_{\ell}))_{x_{\alpha_{i}}(1)}(1))\leq 3(\ell-3)\cdot 2^{\ell-1}+28\cdot 2^{\ell-1}=(3\ell-2)\cdot 2^{\ell-1}$, where equality holds for $x_{\alpha_{\ell-1}}(1)$. In conclusion, we have $\scale[0.9]{\displaystyle \max_{u\in G_{u}\setminus \{1\}}\dim(V_{u}(1))}=(3\ell-2)\cdot 2^{\ell-1}$, and so $\nu_{G}(V)=(\ell+2)\cdot 2^{\ell-1}$.
\end{proof}

\begin{prop}\label{Cl2om1+omlp=2}
Let $p=2$, $\ell\geq 3$ and $V=L_{G}(2\omega_{1}+\omega_{\ell})$. Then $\nu_{G}(V)=(\ell+2)\cdot 2^{\ell-1}$. Moreover, we have $\scale[0.9]{\displaystyle \max_{u\in G_{u}\setminus \{1\}}\dim(V_{u}(1))}=(3\ell-2)\cdot 2^{\ell-1}$ and $\scale[0.9]{\displaystyle \max_{s\in T\setminus\ZG(G)}\dim(V_{s}(\mu))}=(2\ell-1)\cdot 2^{\ell-1}$. 
\end{prop}

\begin{proof}
To begin, we note that $V\cong L_{G}(\omega_{1})^{(2)}\otimes L_{G}(\omega_{\ell})$. Now, as $p=2$, by \cite[(1.6)]{seitz1987maximal}, we have $L_{G}(\omega_{1}+\omega_{\ell})\cong L_{G}(\omega_{1})\otimes L_{G}(\omega_{\ell})$. Thus, for the unipotent elements the result follows by Proposition \ref{Clom1+omlp=2}. We now focus on the case of the semisimple elements.

Set $\lambda=2\omega_{1}+\omega_{\ell}$ and $L=L_{1}$. By Lemma \ref{weightlevelCl}, we have $e_{1}(\lambda)=6$, therefore $\displaystyle V\mid_{[L,L]}=V^{0} \oplus \cdots\oplus V^{6}$. By \cite[Proposition]{Smith_82} and Lemma \ref{dualitylemma}, we have $V^{0}\cong L_{L}(\omega_{\ell})$ and $V^{6}\cong L_{L}(\omega_{\ell})$. Since $p=2$, we have $V^{1}=V^{3}=V^{5}=\{0\}$. Now, in $V^{2}$ the weight $\displaystyle (\lambda-2\alpha_{1})\mid_{T_{1}}=2\omega_{2}+\omega_{\ell}$ admits a maximal vector, thus $V^{2}$ has a composition factor isomorphic to $L_{L}(2\omega_{2}+\omega_{\ell})$. Further, the weight $\displaystyle(\lambda- 2\alpha_{1}-\cdots-2\alpha_{\ell-1}-\alpha_{\ell})\mid_{T_{1}}=\omega_{\ell}$ occurs with multiplicity $\ell$ and is a sub-dominant weight in the composition factor of $V^{2}$ isomorphic to $L_{L}(2\omega_{2}+\omega_{\ell})$, in which it has multiplicity $\ell-1$. As $\dim(V^{2})=(2\ell-1)\cdot 2^{\ell-1}$, we deduce that $V^{2}$ and $V^{4}$ each has two composition factors: one isomorphic to $L_{L}(2\omega_{2}+\omega_{\ell})$ and one to $L_{L}(\omega_{\ell})$.

Let $s\in T\setminus \ZG(G)$. If $\dim(V^{i}_{s}(\mu))=\dim(V^{i})$ for some eigenvalue $\mu$ of $s$ on $V$, where $0\leq i\leq 6$, $i\neq 1,3,5$, then $s\in \ZG(L)^{\circ}\setminus \ZG(G)$. In this case, as $s$ acts on each $V^{i}$ as scalar multiplication by $c^{3-i}$ and $c\neq 1$, we determine that $\dim(V_{s}(\mu))\leq (2\ell-1)\cdot 2^{\ell-1}$, where equality holds for $c^{3}=1$ and $\mu=c^{\pm 1}$. We now assume that $\dim(V^{i}_{s}(\mu))<\dim(V^{i})$ for all eigenvalues $\mu$ of $s$ on $V$ and all $0\leq i\leq 6$, $i\neq 1,3,5$. We write $s=z\cdot h$, where $z\in \ZG(L)^{\circ}$ and $h\in [L,L]$, and we have $\dim(V_{s}(\mu))\leq \dim(V^{0}_{h}(\mu_{h}^{0}))+\dim(V^{2}_{h}(\mu_{h}^{2}))+\dim(V^{4}_{h}(\mu_{h}^{4}))+\dim(V^{6}_{h}(\mu_{h}^{6}))$, where $\dim(V^{i}_{h}(\mu_{h}^{i}))<\dim(V^{i})$ for all eigenvalues $\mu_{h}^{i}$ of $h$ on $V^{i}$, $i=0,2,4,6$. First, let $\ell=3$. We will show that $\dim(V^{2}_{h}(\mu_{h}^{2}))\leq 8$ for all eigenvalues $\mu_{h}^{2}$ of $h$ on $V^{2}$. Using the weight structures of $L_{L}(2\omega_{2}+\omega_{3})$ and $L_{L}(\omega_{3})$, we determine that the eigenvalues of $h$ on $V^{2}$ are $d^{\pm 3}e^{\pm 1}$ each with multiplicity at least $1$; $d^{\pm 1}e^{\pm 3}$ each with multiplicity at least $1$; and $d^{\pm 1}e^{\pm 1}$ each with multiplicity at least $3$, where $d,e\in k^{*}$ not both simultaneously equal to $1$. Thereby, $\dim(V^{2}_{h}(\mu_{h}^{2}))\leq 8$ for all eigenvalues $\mu_{h}^{2}$ of $h$ on $V^{2}$, and, by the structure of $V\mid_{[L,L]}$ and Proposition \ref{C3om3}, it follows that $\dim(V_{s}(\mu))\leq 20$ for all eigenvalues $\mu$ of $s$ on $V$. We now assume that $\ell\geq 4$. Recursively, using the structure of $V\mid_{[L,L]}$, Proposition \ref{Cl_oml_p=2} and the result for $\ell=3$, we determine that $\displaystyle \dim(V_{s}(\mu))\leq 4\dim((L_{L}(\omega_{\ell}))_{h}(\mu_{h}))+2\dim(( L_{L}(2\omega_{2}+\omega_{\ell}))_{h}(\mu_{h}))\leq (\ell-3)2^{\ell}+20\cdot 2^{\ell-3}=(2\ell-1)2^{\ell-1}$ for all eigenvalues $\mu$ of $s$ on $V$. Thus, we have shown that $\scale[0.9]{\displaystyle \max_{s\in T\setminus\ZG(G)}\dim(V_{s}(\mu))}=(2\ell-1)2^{\ell-1}$, and, as $\scale[0.9]{\displaystyle \max_{u\in G_{u}\setminus \{1\}}\dim(V_{u}(1))}=(3\ell-2)\cdot 2^{\ell-1}$, it follows that $\nu_{G}(V)=(\ell+2)\cdot 2^{\ell-1}$.
\end{proof}

\begin{prop}\label{C3om2+om3p=5}
Let $p=5$, $\ell=3$ and $V=L_{G}(\omega_{2}+\omega_{3})$. Then $\nu_{G}(V)=26$. Moreover, we have $\scale[0.9]{\displaystyle \max_{u\in G_{u}\setminus \{1\}}\dim(V_{u}(1))}=25$ and $\scale[0.9]{\displaystyle \max_{s\in T\setminus\ZG(G)}\dim(V_{s}(\mu))}=36$. 
\end{prop}

\begin{proof}
Set $\lambda=\omega_{2}+\omega_{3}$ and $L=L_{1}$. By Lemma \ref{weightlevelCl}, we have $e_{1}(\lambda)=4$, therefore $\displaystyle V\mid_{[L,L]}=V^{0} \oplus \cdots\oplus V^{4}$. By \cite[Proposition]{Smith_82} and Lemma \ref{dualitylemma}, we have $V^{0}\cong L_{L}(\omega_{2}+\omega_{3})$ and $V^{4}\cong L_{L}(\omega_{2}+\omega_{3})$. The weight $\displaystyle (\lambda-\alpha_{1}-\alpha_{2})\mid_{T_{1}}=2\omega_{3}$ admits a maximal vector in $V^{1}$, thus $V^{1}$ has a composition factor isomorphic to $L_{L}(2\omega_{3})$. Similarly, the weight $\displaystyle(\lambda- 2\alpha_{1}-2\alpha_{2}-\alpha_{3})\mid_{T_{1}}=\omega_{2}+\omega_{3}$ admits a maximal vector in $V^{2}$, thus $V^{2}$ has a composition factor isomorphic to $L_{L}(\omega_{2}+\omega_{3})$. By dimensional considerations, we determine that
\begin{equation}\label{DecompVC3om2+om3p=5}
V\mid_{[L,L]}\cong L_{L}(\omega_{2}+\omega_{3}) \oplus L_{L}(2\omega_{3})\oplus L_{L}(\omega_{2}+\omega_{3}) \oplus L_{L}(2\omega_{3})\oplus L_{L}(\omega_{2}+\omega_{3}).
\end{equation}

We start with the semisimple elements. Let $s\in T\setminus \ZG(G)$. If $\dim(V^{i}_{s}(\mu))=\dim(V^{i})$ for some eigenvalue $\mu$ of $s$ on $V$, where $0\leq i\leq 4$, then $s\in \ZG(L)^{\circ}\setminus \ZG(G)$. In this case, as $s$ acts on each $V^{i}$ as scalar multiplication by $c^{2-i}$ and $c\neq 1$, we determine that $\dim(V_{s}(\mu))\leq 36$, where equality holds for $c=-1$ and $\mu=1$. We now assume that $\dim(V^{i}_{s}(\mu))<\dim(V^{i})$ for all eigenvalues $\mu$ of $s$ on $V$ and all $0\leq i\leq 4$. We write $s=z\cdot h$, where $z\in \ZG(L)^{\circ}$ and $h\in [L,L]$, and, by \eqref{DecompVC3om2+om3p=5} and Propositions \ref{PropositionC2om1+om2} and \ref{C22om2}, we determine that $\displaystyle \dim(V_{s}(\mu))\leq 3\dim(( L_{L}(\omega_{2}+\omega_{3}))_{h}(\mu_{h}))+2\dim(( L_{L}(2\omega_{3}))_{h}(\mu_{h}))\leq 36$. Thus, we have shown that $\scale[0.9]{\displaystyle \max_{s\in T\setminus\ZG(G)}\dim(V_{s}(\mu))}=36$. 

For the unipotent elements, by Lemma \ref{uniprootelems}, we have $\scale[0.85]{\displaystyle \max_{u\in G_{u}\setminus \{1\}}\dim(V_{u}(1))=\max_{i=2,3}\dim(V_{x_{\alpha_{i}}(1)}(1))}$ and, using \eqref{DecompVC3om2+om3p=5} and Propositions \ref{PropositionC2om1+om2} and \ref{C22om2}, we determine that $\scale[0.9]{\displaystyle \max_{u\in G_{u}\setminus \{1\}}\dim(V_{u}(1))}=25$. Lastly, as $\scale[0.9]{\displaystyle \max_{s\in T\setminus\ZG(G)}\dim(V_{s}(\mu))}$ $=36$, we have $\nu_{G}(V)=26$.
\end{proof}

\begin{prop}\label{C32om3}
Let $p\neq 2$, $\ell=3$ and $V=L_{G}(2\omega_{3})$. Then $\nu_{G}(V)=32-8\varepsilon_{p}(5)$. Moreover, we have $\scale[0.9]{\displaystyle \max_{u\in G_{u}\setminus \{1\}}\dim(V_{u}(1))\leq 40-15\varepsilon_{p}(5)+4\varepsilon_{p}(3)}$, where equality holds for $p\neq 3$ and $\scale[0.9]{\displaystyle \max_{s\in T\setminus\ZG(G)}\dim(V_{s}(\mu))=52-13\varepsilon_{p}(5)}$. 
\end{prop}

\begin{proof}
Set $\lambda=2\omega_{3}$ and $L=L_{1}$. By Lemma \ref{weightlevelCl}, we have $e_{1}(\lambda)=4$, therefore $\displaystyle V\mid_{[L,L]}=V^{0} \oplus \cdots\oplus V^{4}$. By \cite[Proposition]{Smith_82} and Lemma \ref{dualitylemma}, we have $V^{0}\cong L_{L}(2\omega_{3})$ and $V^{4}\cong L_{L}(2\omega_{3})$. Now, the weight $\displaystyle (\lambda-\alpha_{1}-\alpha_{2}-\alpha_{3})\mid_{T_{1}}=\omega_{2}+\omega_{3}$ admits a maximal vector in $V^{1}$, thus $V^{1}$ has a composition factor isomorphic to $L_{L}(\omega_{2}+\omega_{3})$. Similarly, the weight $\displaystyle(\lambda- 2\alpha_{1}-2\alpha_{2}-\alpha_{3})\mid_{T_{1}}=2\omega_{3}$ admits a maximal vector in $V^{2}$, thus $V^{2}$ has a composition factor isomorphic to $L_{L}(2\omega_{3})$. Further, the weight $\displaystyle(\lambda- 2\alpha_{1}-2\alpha_{2}-2\alpha_{3})\mid_{T_{1}}=2\omega_{2}$ occurs with multiplicity $2-\varepsilon_{p}(5)$ in $V^{2}$ and is a sub-dominant weight in the composition factor of $V^{2}$ isomorphic to $L_{L}(2\omega_{3})$, in which it has multiplicity $1$. As $\dim(V^{2})\leq 24-11\varepsilon_{p}(5)$, using \cite[II.2.14]{Jantzen_2007representations}, it follows that
\begin{equation}\label{DecompVC32om3}
V\mid_{[L,L]}\cong L_{L}(2\omega_{3}) \oplus L_{L}(\omega_{2}+\omega_{3})\oplus L_{L}(2\omega_{3})\oplus L_{L}(2\omega_{2})^{1-\varepsilon_{p}(5)} \oplus L_{L}(\omega_{2}+\omega_{3})\oplus L_{L}(2\omega_{3}).
\end{equation}

We start with the semisimple elements. Let $s\in T\setminus \ZG(G)$. If $\dim(V^{i}_{s}(\mu))=\dim(V^{i})$ for some eigenvalue $\mu$ of $s$ on $V$, where $0\leq i\leq 4$, then $s\in \ZG(L)^{\circ}\setminus \ZG(G)$. In this case, as $s$ acts on each $V^{i}$ as scalar multiplication by $c^{2-i}$ and $c\neq 1$, we determine that $\dim(V_{s}(\mu))\leq 52-13\varepsilon_{p}(5)$, where equality holds for $c=-1$ and $\mu=1$. We thus assume that $\dim(V^{i}_{s}(\mu))<\dim(V^{i})$ for all eigenvalues $\mu$ of $s$ on $V$ and all $0\leq i\leq 4$. We write $s=z\cdot h$, where $z\in \ZG(L)^{\circ}$ and $h\in [L,L]$, and, by \eqref{DecompVC32om3} and Propositions \ref{PropositionClsymm}, \ref{PropositionC2om1+om2} and \ref{C22om2}, we determine that $\displaystyle \dim(V_{s}(\mu))\leq 3\dim((L_{L}(2\omega_{3}))_{h}(\mu_{h}))+2\dim(( L_{L}(\omega_{2}+\omega_{3}))_{h}(\mu_{h}))+(1-\varepsilon_{p}(5))\dim((L_{L}(2\omega_{2}))_{h}(\mu_{h}))\leq 52-13\varepsilon_{p}(5)$. Thus, we have shown that $\displaystyle \max_{s\in T\setminus\ZG(G)}\dim(V_{s}(\mu))=52-13\varepsilon_{p}(5)$. 

We now focus on the unipotent elements. By Lemma \ref{uniprootelems}, we have $\scale[0.85]{\displaystyle \max_{u\in G_{u}\setminus \{1\}}\dim(V_{u}(1))=\max_{i=2,3}\dim(V_{x_{\alpha_{i}}(1)}(1))}$. First, let $p\neq 3$. Then, using \eqref{DecompVC32om3} and Propositions \ref{PropositionClsymm},  \ref{PropositionC2om1+om2} and \ref{C22om2}, we determine that $\dim(V_{x_{\alpha_{i}}(1)}(1))\leq 40-15\varepsilon_{p}(5)$, where equality holds for $i=3$. We now assume that $p=3$, and we use the same results from before to show that $\dim(V_{x_{\alpha_{3}}(1)}(1))\leq 40$ and that $\dim(V_{x_{\alpha_{2}}(1)}(1))\leq 44$. Therefore, $\scale[0.9]{\displaystyle \max_{u\in G_{u}\setminus \{1\}}\dim(V_{u}(1))}\leq 40-15\varepsilon_{p}(5)+4\varepsilon_{p}(3)$, where equality holds for $p\neq 3$, and, as $\scale[0.9]{\displaystyle \max_{s\in T\setminus\ZG(G)}\dim(V_{s}(\mu))}=52-13\varepsilon_{p}(5)$, we determine that $\nu_{G}(V)=32-8\varepsilon_{p}(5)$.
\end{proof}

\begin{prop}\label{C32om2}
Let $p\neq 2$, $\ell=3$ and $V=L_{G}(2\omega_{2})$. Then $\nu_{G}(V)=40$. Moreover, we have $\scale[0.9]{\displaystyle \max_{u\in G_{u}\setminus \{1\}}\dim(V_{u}(1))}$ $\leq 50-\varepsilon_{p}(7)$ and $\scale[0.9]{\displaystyle \max_{s\in T\setminus\ZG(G)}\dim(V_{s}(\mu))}=50-\varepsilon_{p}(7)$. 
\end{prop}

\begin{proof}
Set $\lambda=2\omega_{2}$ and $L=L_{1}$. By Lemma \ref{weightlevelCl}, we have $e_{1}(\lambda)=4$, therefore $\displaystyle V\mid_{[L,L]}=V^{0} \oplus \cdots\oplus V^{4}$. By \cite[Proposition]{Smith_82} and Lemma \ref{dualitylemma}, we have $V^{0}\cong L_{L}(2\omega_{2})$ and $V^{4}\cong L_{L}(2\omega_{2})$. In $V^{1}$ the weight $\displaystyle (\lambda-\alpha_{1}-\alpha_{2})\mid_{T_{1}}=\omega_{2}+\omega_{3}$ admits a maximal vector, thus $V^{1}$ has a composition factor isomorphic to $L_{L}(\omega_{2}+\omega_{3})$. Further, the weight $\displaystyle (\lambda-\alpha_{1}-2\alpha_{2}-\alpha_{3})\mid_{T_{1}}=\omega_{2}$ occurs with multiplicity $3$ and is a sub-dominant weight in the composition factor of $V^{1}$ isomorphic to $L_{L}(\omega_{2}+\omega_{3})$, in which it has multiplicity $2-\varepsilon_{p}(5)$. Similarly, the weight $\displaystyle(\lambda- 2\alpha_{1}-2\alpha_{2})\mid_{T_{1}}=2\omega_{3}$ admits a maximal vector in $V^{2}$, thus $V^{2}$ has a composition factor isomorphic to $L_{L}(2\omega_{3})$. Further, the weight $\displaystyle(\lambda- 2\alpha_{1}-2\alpha_{2}-\alpha_{3})\mid_{T_{1}}=2\omega_{2}$ occurs with multiplicity $2$ and is a sub-dominant weight in the composition factor of $V^{2}$ isomorphic to $L_{L}(2\omega_{3})$, in which it has multiplicity $1$. Moreover, the weight $\displaystyle(\lambda- 2\alpha_{1}-3\alpha_{2}-\alpha_{3})\mid_{T_{1}}=\omega_{3}$ occurs with multiplicity $3$ and is a sub-dominant weight in the composition factor of $V^{2}$ isomorphic to $L_{L}(2\omega_{3})$, in which it has multiplicity $1$. Lastly, we note that the weight $\displaystyle(\lambda- 2\alpha_{1}-4\alpha_{2}-2\alpha_{3})\mid_{T_{1}}=0$ occurs with multiplicity $6-\varepsilon_{p}(7)$ in $V^{2}$. As $\dim(V^{2})\leq 30-\varepsilon_{p}(7)$, we determine that $V^{2}$ has exactly $4-\varepsilon_{p}(7)+\varepsilon_{p}(5)$ composition factors: one isomorphic to $L_{L}(2\omega_{3})$, one to $L_{L}(2\omega_{2})$, one  to $L_{L}(\omega_{3})$ and $1-\varepsilon_{p}(7)+\varepsilon_{p}(5)$ to $L_{L}(0)$. Further, $V^{1}$ and $V^{3}$ each has $2+\varepsilon_{p}(5)$ composition factors: one isomorphic to $L_{L}(\omega_{2}+\omega_{3})$ and $1+\varepsilon_{p}(5)$ to $L_{L}(\omega_{2})$.

We start with the semisimple elements. Let $s\in T\setminus \ZG(G)$. If $\dim(V^{i}_{s}(\mu))=\dim(V^{i})$ for some eigenvalue $\mu$ of $s$ on $V$, where $0\leq i\leq 4$, then $s\in \ZG(L)^{\circ}\setminus \ZG(G)$. In this case, as $s$ acts on each $V^{i}$ as scalar multiplication by $c^{2-i}$ and $c\neq 1$, we determine that $\dim(V_{s}(\mu))\leq 50-\varepsilon_{p}(7)$, where equality holds for $c=-1$ and $\mu=1$. We thus assume that $\dim(V^{i}_{s}(\mu))<\dim(V^{i})$ for all eigenvalues $\mu$ of $s$ on $V$ and all $0\leq i\leq 4$. We write $s=z\cdot h$, where $z\in \ZG(L)^{\circ}$ and $h\in [L,L]$. Using the structure of $V\mid_{[L,L]}$ and Propositions \ref{PropositionClnatural}, \ref{PropositionClwedge}, \ref{PropositionClsymm}, \ref{PropositionC2om1+om2} and \ref{C22om2}, we determine that $\displaystyle \dim(V_{s}(\mu))\leq 3\dim((L_{L}(2\omega_{2}))_{h}(\mu_{h}))+2\dim(( L_{L}(\omega_{2}+\omega_{3}))_{h}(\mu_{h}))+(2+2\varepsilon_{p}(5))\dim((L_{L}(\omega_{2}))_{h}(\mu_{h}))+\dim((L_{L}(2\omega_{3}))_{h}(\mu_{h}))+\dim((L_{L}(\omega_{3}))_{h}(\mu_{h}))+(1-\varepsilon_{p}(7)+\varepsilon_{p}(5))\dim((L_{L}(0))_{h}(\mu_{h}))\leq 53-\varepsilon_{p}(7)$ for all eigenvalues $\mu$ of $s$ on $V$. However, assume there exist $(s,\mu)\in T\setminus \ZG(G)\times k^{*}$ with $\dim(V_{s}(\mu))>50-\varepsilon_{p}(7)$. By the above and the structure of $V\mid_{[L,L]}$, $\dim(V_{s}(\mu))>50-\varepsilon_{p}(7)$ if and only if $\dim((L_{L}(2\omega_{2}))_{h}(\mu_{h}))=6$ for some eigenvalue $\mu_{h}$ of $h$ on $L_{L}(2\omega_{2})$. Then, by Proposition \ref{PropositionClsymm}, we have that, up to conjugation, $h=\diag(1,1,-1,-1,1,1)$ and $\mu_{h}=1$, or $h=\diag(1,d,d,d^{-1},d^{-1},1)$ with $d^{2}=-1$, and $\mu_{h}=-1$. If $h=\diag(1,1,-1,-1,1,1)$, then, using the structure of $V\mid_{[L,L]}$, \eqref{enum Cl_P2} and the weights of $L_{L}(2\omega_{3})$, one shows that the eigenvalues of $h$ on $V^{2}$ are $1$ and $-1$ with $\dim(V^{2}_{h}(1))=18-\varepsilon_{p}(7)$ and $\dim(V^{2}_{h}(-1))=13$. This gives $\dim(V_{s}(\mu))\leq 50-\varepsilon_{p}(7)$, contradicting our assumption. Analogously, for $h=\diag(1,d,d,d^{-1},d^{-1},1)$, one shows that $\dim(V^{2}_{h}(-1))=18-\varepsilon_{p}(7)$ and $\dim(V^{2}_{h}(1))=13$, therefore $\dim(V_{s}(\mu))\leq 49-\varepsilon_{p}(7)$. We determine that $\dim(V_{s}(\mu))\leq 50-\varepsilon_{p}(7)$ for all eigenvalues $\mu$ of $s$ on $V$, and so $\scale[0.9]{\displaystyle \max_{s\in T\setminus\ZG(G)}\dim(V_{s}(\mu))}=50-\varepsilon_{p}(7)$. 

We now focus on the unipotent elements. By Lemma \ref{uniprootelems}, we have $\scale[0.85]{\displaystyle \max_{u\in G_{u}\setminus \{1\}}\dim(V_{u}(1))=\max_{i=2,3}\dim(V_{x_{\alpha_{i}}(1)}(1))}$. By the structure of $V\mid_{[L,L]}$ and Propositions \ref{PropositionClnatural}, \ref{PropositionClwedge}, \ref{PropositionClsymm}, \ref{PropositionC2om1+om2} and \ref{C22om2}, we determine that $\dim(V_{x_{\alpha_{3}}(1)}(1))\leq 50-\varepsilon_{p}(7)$ and $\dim(V_{x_{\alpha_{2}}(1)}(1))\leq 42-\varepsilon_{p}(7)-4\varepsilon_{p}(5)+2\varepsilon_{p}(2)$. Therefore $\scale[0.9]{\displaystyle \max_{u\in G_{u}\setminus \{1\}}\dim(V_{u}(1))}\leq 50-\varepsilon_{p}(7)$, and, as $\scale[0.9]{\displaystyle \max_{s\in T\setminus\ZG(G)}\dim(V_{s}(\mu))}=50-\varepsilon_{p}(7)$, it follows that $\nu_{G}(V)=40$.
\end{proof}

\begin{prop}\label{C4om1+om4p=7}
Let $p=7$, $\ell=4$ and $V=L_{G}(\omega_{1}+\omega_{4})$. Then $\nu_{G}(V)\geq 96$.  Moreover, we have $\scale[0.9]{\displaystyle \max_{u\in G_{u}\setminus \{1\}}\dim(V_{u}(1))}=142$ and $\scale[0.9]{\displaystyle \max_{s\in T\setminus\ZG(G)}\dim(V_{s}(\mu))}\leq 144$.
\end{prop}

\begin{proof}
Let $\lambda=\omega_{1}+\omega_{4}$ and let $L=L_{1}$. We have $e_{1}(\lambda)=4$, see Lemma \ref{weightlevelCl}, therefore $\displaystyle V\mid_{[L,L]}=V^{0} \oplus \cdots \oplus V^{4}$. By \cite[Proposition]{Smith_82} and Lemma \ref{dualitylemma}, we have $V^{0}\cong L_{L}(\omega_{4})$ and $V^{4}\cong L_{L}(\omega_{4})$. Now, the weight $\displaystyle (\lambda-\alpha_{1})\mid_{T_{1}}=\omega_{2}+\omega_{4}$ admits a maximal vector in $V^{1}$, thus $V^{1}$ has a composition factor isomorphic to $L_{L}(\omega_{2}+\omega_{4})$. Similarly, in $V^{2}$, the weight  $\displaystyle (\lambda-2\alpha_{1}-\alpha_{2}-\alpha_{3}-\alpha_{4})\mid_{T_{1}}=\omega_{2}+\omega_{3}$ admits a maximal vector, thus $V^{2}$ has a composition factor isomorphic to $L_{L}(\omega_{2}+\omega_{3})$. Moreover, the weight $\displaystyle (\lambda-2\alpha_{1}-2\alpha_{2}-2\alpha_{3}-\alpha_{4})\mid_{T_{1}}=\omega_{4}$ occurs with multiplicity $3$ and is a sub-dominant weight in the composition factor of $V^{2}$ isomorphic to $L_{L}(\omega_{2}+\omega_{3})$ in which it has multiplicity $2$. Thus, as $\dim(V^{2})\leq 72$, we determine that $V^{2}$ has exactly $2$ composition factors: one isomorphic to $L_{L}(\omega_{2}+\omega_{3})$ and one to $L_{L}(\omega_{4})$, and so $V^{2}\cong L_{L}(\omega_{2}+\omega_{3})\oplus L_{L}(\omega_{4})$, by \cite[II.$2.14$]{Jantzen_2007representations}. Moreover, we have
\begin{equation}\label{DecompVC4om1+om4p=7}
V\mid_{[L,L]}\cong L_{L}(\omega_{4}) \oplus L_{L}(\omega_{2}+\omega_{4}) \oplus L_{L}(\omega_{2}+\omega_{3}) \oplus L_{L}(\omega_{4})\oplus L_{L}(\omega_{2}+\omega_{4}) \oplus L_{L}(\omega_{4}).
\end{equation}

We  start with the semisimple elements. Let $s\in T\setminus \ZG(G)$. If $\dim(V^{i}_{s}(\mu))=\dim(V^{i})$ for some eigenvalue $\mu$ of $s$ on $V$, where $0\leq i\leq 4$, then $s\in \ZG(L)^{\circ}\setminus \ZG(G)$. In this case, $s$ acts on each $V^{i}$ as scalar multiplication by $c^{2-i}$ and, as $c\neq 1$, we determine that $\dim(V_{s}(\mu))\leq 140$ for all eigenvalues $\mu$ of $s$ on $V$. We thus assume that $\dim(V^{i}_{s}(\mu))<\dim(V^{i})$ for all eigenvalues $\mu$ of $s$ on $V$ and all $0\leq i\leq 4$. We write $s=z\cdot h$, where $z\in \ZG(L)^{\circ}$ and $h\in [L,L]$, and, by \eqref{DecompVC4om1+om4p=7} and Propositions \ref{C3om3}, \ref{PropositionClom1+om2pneq3} and \ref{C3om1+om3}, we get $\dim(V_{s}(\mu))\leq 3\dim((L_{L}(\omega_{4}))_{h}(\mu))+2\dim((L_{L}(\omega_{2}+\omega_{4}))_{h}(\mu))+\dim((L_{L}(\omega_{2}+\omega_{3}))_{h}(\mu))\leq 144$. Therefore, $\scale[0.9]{\displaystyle \max_{s\in T\setminus\ZG(G)}\dim(V_{s}(\mu))}\leq 144$.

For the unipotent elements, we have $\displaystyle \max_{u\in G_{u}\setminus \{1\}}\dim(V_{u}(1))=\max_{i=3,4}\dim(V_{x_{\alpha_{i}}(1)}(1))$, see Lemma \ref{uniprootelems}. Further, by \eqref{DecompVC4om1+om4p=7} and Propositions \ref{C3om3}, \ref{PropositionClom1+om2pneq3} and \ref{C3om1+om3}, we have that $\dim(V_{x_{\alpha_{i}}(1)}(1))\leq 142$, where equality holds for $i=4$. Therefore $\scale[0.9]{\displaystyle \max_{u\in G_{u}\setminus \{1\}}\dim(V_{u}(1))}=142$, and so $98\geq \nu_{G}(V)\geq 96$. 
\end{proof}

\begin{prop}\label{C4om1+om3p=2}
Let $p=2$, $\ell=4$ and $V=L_{G}(\omega_{1}+\omega_{3})$. Then $\nu_{G}(V)\geq 102$.  Moreover, we have $\scale[0.9]{\displaystyle \max_{u\in G_{u}\setminus \{1\}}\dim(V_{u}(1))}\leq 144$ and $\scale[0.9]{\displaystyle \max_{s\in T\setminus\ZG(G)}\dim(V_{s}(\mu))}\leq 130$.
\end{prop}

\begin{proof}
Let $\lambda=\omega_{1}+\omega_{3}$ and let $L=L_{1}$. We have $e_{1}(\lambda)=4$, see Lemma \ref{weightlevelCl}, therefore $\displaystyle V\mid_{[L,L]}=V^{0} \oplus \cdots \oplus V^{4}$. By \cite[Proposition]{Smith_82} and Lemma \ref{dualitylemma}, we have $V^{0}\cong L_{L}(\omega_{3})$ and so $V^{4}\cong L_{L}(\omega_{3})$. Now, the weight $\displaystyle (\lambda-\alpha_{1})\mid_{T_{1}}=\omega_{2}+\omega_{3}$ admits a maximal vector in $V^{1}$, thus $V^{1}$ has a composition factor isomorphic to $L_{L}(\omega_{2}+\omega_{3})$. Similarly, in $V^{2}$, the weight $\displaystyle (\lambda-2\alpha_{1}-\alpha_{2}-\alpha_{3})\mid_{T_{1}}=\omega_{2}+\omega_{4}$ admits a maximal vector, thus $V^{2}$ has a composition factor isomorphic to $L_{L}(\omega_{2}+\omega_{4})$. Further, the weight $\displaystyle (\lambda-2\alpha_{1}-\alpha_{2}-2\alpha_{3}-\alpha_{4})\mid_{T_{1}}=2\omega_{2}$ occurs with multiplicity $2$ in $V^{2}$ and is not a sub-dominant weight in the composition factor of $V^{2}$ isomorphic to $L_{L}(\omega_{2}+\omega_{4})$. However, the weight $\displaystyle (\lambda-2\alpha_{1}-2\alpha_{2}-2\alpha_{3}-\alpha_{4})\mid_{T_{1}}=\omega_{3}$ occurs with multiplicity $4$ in $V^{2}$ and is a sub-dominant weight in the composition factor of $V^{2}$ isomorphic to $L_{L}(\omega_{2}+\omega_{4})$, in which it has multiplicity $2$. Lastly, we note that the weight $\displaystyle (\lambda-2\alpha_{1}-3\alpha_{2}-4\alpha_{3}-2\alpha_{4})\mid_{T_{1}}=0$ occurs with multiplicity $6$ in $V^{2}$. As $\dim(V^{2})\leq 90$, it follows that $V^{2}$ has exactly $7$ composition factors: one isomorphic to $L_{L}(\omega_{2}+\omega_{4})$, two to $L_{L}(\omega_{2})^{(2)}$, two to $L_{L}(\omega_{3})$ and two to $L_{L}(0)$, and  $V^{1}\cong L_{L}(\omega_{2}+\omega_{3})$ and $V^{3}\cong L_{L}(\omega_{2}+\omega_{3})$.

We start with the semisimple elements. Let $s\in T\setminus \ZG(G)$. If $\dim(V^{i}_{s}(\mu))=\dim(V^{i})$ for some eigenvalue $\mu$ of $s$ on $V$, where $0\leq i\leq 4$, then $s\in \ZG(L)^{\circ}\setminus \ZG(G)$. In this case, $s$ acts on each $V^{i}$ as scalar multiplication by $c^{2-i}$ and, as $c\neq 1$, we determine that $\dim(V_{s}(\mu))\leq 128$ for all eigenvalues $\mu$ of $s$ on $V$. We thus assume that $\dim(V^{i}_{s}(\mu))<\dim(V^{i})$ for all eigenvalues $\mu$ of $s$ on $V$ and all $0\leq i\leq 4$. We write $s=z\cdot h$, where $z\in \ZG(L)^{\circ}$ and $h\in [L,L]$, and, by the structure of $V\mid_{[L,L]}$ and Propositions \ref{PropositionClnatural}, \ref{PropositionClwedge}, \ref{PropositionClom1+om2pneq3} and \ref{C3om1+om3}, we get $\dim(V_{s}(\mu))\leq 4\dim((L_{L}(\omega_{3}))_{h}(\mu))+2\dim((L_{L}(\omega_{2}+\omega_{3}))_{h}(\mu))+2\dim((L_{L}(2\omega_{2}))_{h}(\mu))+\dim((L_{L}(\omega_{2}+\omega_{4}))_{h}(\mu))+2\dim((L_{L}(0))_{h}(\mu))\leq 130$ for all eigenvalues $\mu$ of $s$ on $V$. Therefore, $\scale[0.9]{\displaystyle \max_{s\in T\setminus\ZG(G)}\dim(V_{s}(\mu))}\leq 130$.

For the unipotent elements, we have $\displaystyle \max_{u\in G_{u}\setminus \{1\}}\dim(V_{u}(1))=\max_{i=3,4}\dim(V_{x_{\alpha_{i}}(1)}(1))$, see Lemma \ref{uniprootelems}. Further, by the structure of $V\mid_{[L,L]}$ and Propositions \ref{PropositionClnatural}, \ref{PropositionClwedge}, \ref{PropositionClom1+om2pneq3} and \ref{C3om1+om3}, we get $\dim(V_{x_{\alpha_{i}}(1)}(1))\leq 144$, where $i=3,4$. Therefore, $\scale[0.9]{\displaystyle \max_{u\in G_{u}\setminus \{1\}}\dim(V_{u}(1))}\leq 144$, and so $\nu_{G}(V)\geq 102$.
\end{proof}

\begin{prop}\label{C4om4pneq2}
Let $p\neq 2$, $\ell=4$ and $V=L_{G}(\omega_{4})$. Then $\nu_{G}(V)=14-\varepsilon_{p}(3)$.  Moreover, we have $\scale[0.9]{\displaystyle \max_{u\in G_{u}\setminus \{1\}}\dim(V_{u}(1))}=28-\varepsilon_{p}(3)$ and $\scale[0.9]{\displaystyle \max_{s\in T\setminus\ZG(G)}\dim(V_{s}(\mu))}=28$.
\end{prop}

\begin{proof}
Set $\lambda=\omega_{4}$ and $L=L_{1}$. By Lemma \ref{weightlevelCl}, we have $e_{1}(\lambda)=2$, therefore $\displaystyle V\mid_{[L,L]}=V^{0} \oplus V^{1} \oplus V^{2}$. By \cite[Proposition]{Smith_82} and Lemma \ref{dualitylemma}, we have $V^{0}\cong L_{L}(\omega_{4})$ and $V^{2}\cong L_{L}(\omega_{4})$. Now, the weight $\displaystyle (\lambda-\alpha_{1}-\alpha_{2}-\alpha_{3}-\alpha_{4})\mid_{T_{1}}=\omega_{3}$ admits a maximal vector in $V^{1}$, thus $V^{1}$ has a composition factor isomorphic to $L_{L}(\omega_{3})$. As $\dim(V^{1})=14-\varepsilon_{p}(3)$, we deduce that $V^{1}\cong L_{L}(\omega_{3})$ and
\begin{equation}\label{DecompVC4om4pneq2}
V\mid_{[L,L]}\cong L_{L}(\omega_{4})\oplus L_{L}(\omega_{3}) \oplus L_{L}(\omega_{4}).
\end{equation}

We start with the semisimple elements. Let $s\in T\setminus \ZG(G)$. If $\dim(V^{i}_{s}(\mu))=\dim(V^{i})$ for some eigenvalue $\mu$ of $s$ on $V$, where $0\leq i\leq 2$, then $s\in \ZG(L)^{\circ}\setminus \ZG(G)$. In this case, as $s$ acts on each $V^{i}$ as scalar multiplication by $c^{1-i}$ and $c\neq 1$, it follows that $\dim(V_{s}(\mu))\leq 28$, where equality holds for $c=-1$ and $\mu=-1$. We thus assume that $\dim(V^{i}_{s}(\mu))<\dim(V^{i})$ for all eigenvalues $\mu$ of $s$ on $V$ and all $0\leq i\leq 2$. We write $s=z\cdot h$, where $z\in \ZG(L)^{\circ}$ and $h\in [L,L]$, and we use \eqref{DecompVC4om4pneq2} and Proposition \ref{PropositionClwedge} and \ref{C3om3}, to determine that $\dim(V_{s}(\mu))\leq \dim((L_{L}(\omega_{3}))_{h}(\mu_{h}))+2\dim((L_{L}(\omega_{4}))_{h}(\mu_{h})\leq 28$ for all eigenvalues $\mu$ of $s$ on $V$. This gives $\scale[0.9]{\displaystyle \max_{s\in T\setminus\ZG(G)}\dim(V_{s}(\mu))}=28$. 

We focus on the unipotent elements. By Lemma \ref{uniprootelems}, we have $\scale[0.85]{\displaystyle \max_{u\in G_{u}\setminus \{1\}}\dim(V_{u}(1))=\max_{i=3,4}\dim(V_{x_{\alpha_{i}}(1)}(1))}$. Then, by \eqref{DecompVC4om4pneq2} and Proposition \ref{PropositionClwedge} and \ref{C3om3}, it follows that $\dim(V_{x_{\alpha_{i}}(1)}(1))\leq 28-\varepsilon_{p}(3)$, where equality holds for $i=4$. Therefore $\scale[0.9]{\displaystyle \max_{u\in G_{u}\setminus \{1\}}\dim(V_{u}(1))}=28-\varepsilon_{p}(3)$ and so, $\nu_{G}(V)=14-\varepsilon_{p}(3)$.
\end{proof}

\begin{prop}\label{C5om4}
Let $\ell=5$ and $V=L_{G}(\omega_{4})$. Then $\nu_{G}(V)\geq 48-8\varepsilon_{p}(3)-4\varepsilon_{p}(2)$, where equality holds for $p\neq 2$. Moreover, we have $\scale[0.9]{\displaystyle \max_{u\in G_{u}\setminus \{1\}}\dim(V_{u}(1))}\leq 117-36\varepsilon_{p}(3)+3\varepsilon_{p}(2)$, where equality holds for $p\neq 2$, and $\scale[0.9]{\displaystyle \max_{s\in T\setminus\ZG(G)}\dim(V_{s}(\mu))}\leq 100-20\varepsilon_{p}(3)-24\varepsilon_{p}(2)$.
\end{prop}

\begin{proof}
Set $\lambda=\omega_{4}$ and $L=L_{1}$. Now, by Lemma \ref{weightlevelCl}, we have $e_{1}(\lambda)=2$, therefore $\displaystyle V\mid_{[L,L]}=V^{0} \oplus V^{1} \oplus V^{2}$. By \cite[Proposition]{Smith_82} and Lemma \ref{dualitylemma}, we have $V^{0}\cong L_{L}(\omega_{4})$ and $V^{2}\cong L_{L}(\omega_{4})$. Now, the weight $\displaystyle (\lambda-\alpha_{1}-\alpha_{2}-\alpha_{3}-\alpha_{4})\mid_{T_{1}}=\omega_{5}$ admits a maximal vector in $V^{1}$, thus $V^{1}$ has a composition factor isomorphic to $L_{L}(\omega_{5})$. Further, the weight $\displaystyle (\lambda-\alpha_{1}-\alpha_{2}-\alpha_{3}-2\alpha_{4}-\alpha_{5})\mid_{T_{1}}=\omega_{3}$ occurs with multiplicity $2-\varepsilon_{p}(3)$ and is a sub-dominant weight in the composition factor of $V^{1}$ isomorphic to $L_{L}(\omega_{5})$, in which it has multiplicity $1$, if and only if $p\neq 2$. Thus, since $\dim(V^{1})=69-28\varepsilon_{p}(3)-\varepsilon_{p}(2)$, we determine that $V^{1}$ has exactly $2-\varepsilon_{p}(3)+\varepsilon_{p}(2)$ composition factors: one isomorphic to $L_{L}(\omega_{5})$ and $1-\varepsilon_{p}(3)+\varepsilon_{p}(2)$ to $L_{L}(\omega_{3})$. Lastly, we note that if $p\neq 2$, then $V^{1}\cong L_{L}(\omega_{5})\oplus L_{L}(\omega_{3})^{1-\varepsilon_{p}(3)}$, see \cite[II.2.14]{Jantzen_2007representations}. 

We start with the semisimple elements. Let $s\in T\setminus \ZG(G)$. If $\dim(V^{i}_{s}(\mu))=\dim(V^{i})$ for some eigenvalue $\mu$ of $s$ on $V$, where $0\leq i\leq 2$, then $s\in \ZG(L)^{\circ}\setminus \ZG(G)$. In this case, as $s$ acts on each $V^{i}$ as scalar multiplication by $c^{1-i}$ and $c\neq 1$, it follows that $\dim(V_{s}(\mu))\leq 96-16\varepsilon_{p}(3)-28\varepsilon_{p}(2)$ for all eigenvalues $\mu$ of $s$ on $V$. We thus assume that $\dim(V^{i}_{s}(\mu))<\dim(V^{i})$ for all eigenvalues $\mu$ of $s$ on $V$ and all $0\leq i\leq 2$. We write $s=z\cdot h$, where $z\in \ZG(L)^{\circ}$ and $h\in [L,L]$, and using the structure of $V\mid_{[L,L]}$ and Propositions \ref{PropositionClwedge}, \ref{C4om3}, \ref{Cl_oml_p=2} and  \ref{C4om4pneq2}, we determine that $\dim(V_{s}(\mu))\leq  2\dim((L_{L}(\omega_{4}))_{h}(\mu_{h}))+\dim((L_{L}(\omega_{5}))_{h}(\mu_{h}))+(1-\varepsilon_{p}(3)+\varepsilon_{p}(2))\dim((L_{L}(\omega_{3}))_{h}(\mu_{h}))\leq 100-20\varepsilon_{p}(3)-24\varepsilon_{p}(2)$ for all eigenvalues $\mu$ of $s$ on $V$. Therefore,  $\scale[0.9]{\displaystyle \max_{s\in T\setminus\ZG(G)}\dim(V_{s}(\mu))}\leq 100-20\varepsilon_{p}(3)-24\varepsilon_{p}(2)$.

For the unipotent elements, by Lemma \ref{uniprootelems}, we have $\scale[0.85]{\displaystyle \max_{u\in G_{u}\setminus \{1\}}\dim(V_{u}(1))=\max_{i=4,5}\dim(V_{x_{\alpha_{i}}(1)}(1))}$. Using the structure of $V\mid_{[L,L]}$ and Propositions \ref{PropositionClwedge}, \ref{C4om3}, \ref{Cl_oml_p=2} and \ref{C4om4pneq2}, we determine that $\dim(V_{x_{\alpha_{i}}(1)}(1))$ $\leq 117-36\varepsilon_{p}(3)+3\varepsilon_{p}(2)$, where equality holds for $x_{\alpha_{5}}(1)$ and $p\neq 2$. Thus $\scale[0.9]{\displaystyle \max_{u\in G_{u}\setminus \{1\}}\dim(V_{u}(1))}\leq 117-36\varepsilon_{p}(3)+3\varepsilon_{p}(2)$, where equality holds for $p\neq 2$, and so $\nu_{G}(V)\geq 48-8\varepsilon_{p}(3)-4\varepsilon_{p}(2)$, where equality holds for $p\neq 2$.
\end{proof}

\begin{prop}\label{C5om5}
Let $p\neq 2$, $\ell=5$ and $V=L_{G}(\omega_{5})$. Then $\nu_{G}(V)=42-2\varepsilon_{p}(3)$. Moreover, we have $\scale[0.9]{\displaystyle \max_{u\in G_{u}\setminus \{1\}}\dim(V_{u}(1))}=90-9\varepsilon_{p}(3)$ and $\scale[0.9]{\displaystyle \max_{s\in T\setminus\ZG(G)}\dim(V_{s}(\mu))}\leq 84-2\varepsilon_{p}(3)$, where equality holds for $p=3$.
\end{prop}

\begin{proof}
Set $\lambda=\omega_{5}$ and $L=L_{1}$. By Lemma \ref{weightlevelCl}, we have $e_{1}(\lambda)=2$, therefore $\displaystyle V\mid_{[L,L]}=V^{0} \oplus V^{1} \oplus V^{2}$. By \cite[Proposition]{Smith_82} and Lemma \ref{dualitylemma}, we have $V^{0}\cong L_{L}(\omega_{5})$ and $V^{2}\cong L_{L}(\omega_{5})$. Now, the weight $\displaystyle (\lambda-\alpha_{1}-\cdots-\alpha_{5})\mid_{T_{1}}=\omega_{4}$ admits a maximal vector in $V^{1}$, thus $V^{1}$ has a composition factor isomorphic to $L_{L}(\omega_{4})$. Since $\dim(V^{1})=48-8\varepsilon_{p}(3)$, we determine that
\begin{equation}\label{DecompVC5om5}
V\mid_{[L,L]}\cong  L_{L}(\omega_{5})\oplus  L_{L}(\omega_{4})\oplus  L_{L}(\omega_{5}).
\end{equation}

We start with the semisimple elements. Let $s\in T\setminus \ZG(G)$. If $\dim(V^{i}_{s}(\mu))=\dim(V^{i})$ for some eigenvalue $\mu$ of $s$ on $V$, where $0\leq i\leq 2$, then $s\in \ZG(L)^{\circ}\setminus \ZG(G)$. In this case, as $s$ acts on each $V^{i}$ as scalar multiplication by $c^{1-i}$ and $c\neq 1$, it follows that $\dim(V_{s}(\mu))\leq 84-2\varepsilon_{p}(3)$, where equality holds for $c=-1$ and $\mu=-1$. We thus assume that $\dim(V^{i}_{s}(\mu))<\dim(V^{i})$ for all eigenvalues $\mu$ of $s$ on $V$ and all $0\leq i\leq 2$. We write $s=z\cdot h$, where $z\in \ZG(L)^{\circ}$ and $h\in [L,L]$, and, using \eqref{DecompVC5om5} and Propositions \ref{C4om3} and \ref{C4om4pneq2}, we determine that $\dim(V_{s}(\mu))\leq 2\dim((L_{L}(\omega_{5}))_{h}(\mu_{h}))+\dim((L_{L}(\omega_{4}))_{h}(\mu_{h}))\leq 84-2\varepsilon_{p}(3)$. Therefore, $\scale[0.9]{\displaystyle \max_{s\in T\setminus\ZG(G)}\dim(V_{s}(\mu))}\leq 84-2\varepsilon_{p}(3)$. 

We focus on the unipotent elements. By Lemma \ref{uniprootelems}, we have $\scale[0.85]{\displaystyle \max_{u\in G_{u}\setminus \{1\}}\dim(V_{u}(1))=\max_{i=4,5}\dim(V_{x_{\alpha_{i}}(1)}(1))}$, and, using \eqref{DecompVC5om5} and Propositions \ref{C4om3} and \ref{C4om4pneq2}, we determine that $\dim(V_{x_{\alpha_{i}}(1)}(1))\leq 90-9\varepsilon_{p}(3)$, where equality holds for $i=5$. Therefore $\scale[0.9]{\displaystyle \max_{u\in G_{u}\setminus \{1\}}\dim(V_{u}(1))}=90-9\varepsilon_{p}(3)$ and, as $\scale[0.9]{\displaystyle \max_{s\in T\setminus\ZG(G)}\dim(V_{s}(\mu))}\leq 84-2\varepsilon_{p}(3)$, where equality holds for $p=3$, we get $\nu_{G}(V)=42-2\varepsilon_{p}(3)$.
\end{proof}

\begin{prop}\label{C6om4}
Let $\ell=6$ and $V=L_{G}(\omega_{4})$. Then $\nu_{G}(V)\geq 110-14\varepsilon_{p}(2)$, where equality holds for $p\neq 2,3$. Moreover, we have $\scale[0.9]{\displaystyle \max_{u\in G_{u}\setminus \{1\}}\dim(V_{u}(1))}\leq 319-\varepsilon_{p}(5)-51\varepsilon_{p}(2)$, where equality holds for $p\neq 2,3$, and $\scale[0.9]{\displaystyle \max_{s\in T\setminus\ZG(G)}\dim(V_{s}(\mu))}\leq 260-\varepsilon_{p}(5)+16\varepsilon_{p}(3) -88\varepsilon_{p}(2)$.
\end{prop}

\begin{proof}
Set $\lambda=\omega_{4}$ and $L=L_{1}$. Now, by Lemma \ref{weightlevelCl}, we have $e_{1}(\lambda)=2$, therefore $\displaystyle V\mid_{[L,L]}=V^{0} \oplus V^{1} \oplus V^{2}$. We argue as we did in the proof of Proposition \ref{C5om4}, to show that $V^{0}\cong L_{L}(\omega_{4})$, $V^{2}\cong L_{L}(\omega_{4})$ and $V^{1}$ has $2+\varepsilon_{p}(3)-\varepsilon_{p}(2)$ composition factors: one isomorphic to $L_{L}(\omega_{5})$ and $1+\varepsilon_{p}(3)-\varepsilon_{p}(2)$ to $L_{L}(\omega_{3})$, where if $p\neq 3$, we have $V^{1}\cong L_{L}(\omega_{5})\oplus L_{L}(\omega_{3})^{1-\varepsilon_{p}(2)}$.

First, let $s\in T\setminus \ZG(G)$. We argue as we did in the proof of Proposition \ref{C5om4} to show that if $s\in \ZG(L)^{\circ}\setminus \ZG(G)$, then $\dim(V_{s}(\mu))\leq 220-56\varepsilon_{p}(2)$ for all eigenvalues $\mu$ of $s$ on $V$, while, if $s\notin \ZG(L)^{\circ}$, then, by Propositions \ref{PropositionClwedge}, \ref{PropositionClwedgecube} and \ref{C5om4}, we have $\dim(V_{s}(\mu))\leq  260-\varepsilon_{p}(5)+16\varepsilon_{p}(3) -88\varepsilon_{p}(2)$ for all eigenvalues $\mu$ of $s$ on $V$. Therefore, $\scale[0.9]{\displaystyle \max_{s\in T\setminus\ZG(G)}\dim(V_{s}(\mu))}\leq 260-\varepsilon_{p}(5)+16\varepsilon_{p}(3) -88\varepsilon_{p}(2)$. Secondly, for the unipotent elements, we argue as in the proof of Proposition \ref{C5om4} to show that $\scale[0.9]{\displaystyle \max_{u\in G_{u}\setminus \{1\}}\dim(V_{u}(1))}\leq 319-\varepsilon_{p}(5)-51\varepsilon_{p}(2)$, where equality holds for $p\neq 2,3$, and so $\nu_{G}(V)\geq 110-14\varepsilon_{p}(2)$, where equality holds for $p\neq 2,3$.
\end{proof}

\begin{prop}\label{C6om5}
Let $\ell=6$ and $V=L_{G}(\omega_{5})$. Then $\nu_{G}(V)\leq 165-44\varepsilon_{p}(3)-17\varepsilon_{p}(2)$, where equality holds for $p\neq 2$. Moreover, we have $\scale[0.9]{\displaystyle \max_{u\in G_{u}\setminus \{1\}}\dim(V_{u}(1))}\leq 407-164\varepsilon_{p}(3)+5\varepsilon_{p}(2)$, where equality holds for $p\neq 2$, and $\scale[0.9]{\displaystyle \max_{s\in T\setminus\ZG(G)}\dim(V_{s}(\mu))} \leq 350-108\varepsilon_{p}(3)-86\varepsilon_{p}(2)$.
\end{prop}

\begin{proof}
Set $\lambda=\omega_{5}$ and $L=L_{1}$. By Lemma \ref{weightlevelCl}, we have $e_{1}(\lambda)=2$, therefore $\displaystyle V\mid_{[L,L]}=V^{0} \oplus V^{1} \oplus V^{2}$. By \cite[Proposition]{Smith_82} and Lemma \ref{dualitylemma}, we have $V^{0}\cong L_{L}(\omega_{5})$ and $V^{2}\cong L_{L}(\omega_{5})$.  Now, the weight $\displaystyle (\lambda-\alpha_{1}-\cdots-\alpha_{5})\mid_{T_{1}}=\omega_{6}$ admits a maximal vector in $V^{1}$, thus $V^{1}$ has a composition factor isomorphic to $L_{L}(\omega_{6})$. Moreover, the weight $\displaystyle (\lambda-\alpha_{1}-\cdots-\alpha_{4}-2\alpha_{5}-\alpha_{6})\mid_{T_{1}}=\omega_{4}$ occurs with multiplicity $2-\varepsilon_{p}(3)$ and is a sub-dominant weight in the composition factor of $V^{1}$ isomorphic to $L_{L}(\omega_{6})$, in which it has multiplicity $1$, if and only if $p\neq 2$. Therefore, as $\dim(V^{1})=242-120\varepsilon_{p}(3)-10\varepsilon_{p}(2)$, we determine that $V^{1}$ has $2-\varepsilon_{p}(3)+\varepsilon_{p}(2)$ composition factors: one isomorphic to $L_{L}(\omega_{6})$ and $1-\varepsilon_{p}(3)+\varepsilon_{p}(2)$ to $L_{L}(\omega_{4})$. Moreover, when $p\neq 2$, by  \cite[II.2.14]{Jantzen_2007representations}, we have $V^{1}\cong L_{L}(\omega_{6})\oplus L_{L}(\omega_{4})^{1-\varepsilon_{p}(3)}$. 

We start with the semisimple elements. Let $s\in T\setminus \ZG(G)$. If $\dim(V^{i}_{s}(\mu))=\dim(V^{i})$ for some eigenvalue $\mu$ of $s$ on $V$, where $0\leq i\leq 2$, then $s\in \ZG(L)^{\circ}\setminus \ZG(G)$. In this case, as $s$ acts on each $V^{i}$ as scalar multiplication by $c^{1-i}$ and $c\neq 1$, it follows that $\dim(V_{s}(\mu))\leq 330-88\varepsilon_{p}(3)-98\varepsilon_{p}(2)$ for all eigenvalues $\mu$ of $s$ on $V$. We thus assume that $\dim(V^{i}_{s}(\mu))<\dim(V^{i})$ for all eigenvalues $\mu$ of $s$ on $V$ and all $0\leq i\leq 2$. We write $s=z\cdot h$, where $z\in \ZG(L)^{\circ}$ and $h\in [L,L]$, and, using the structure of $V\mid_{[L,L]}$ and Propositions \ref{PropositionClwedgecube}, \ref{Cl_oml_p=2}, \ref{C5om4} and \ref{C5om5}, we have $\dim(V_{s}(\mu))\leq 2\dim((L_{L}(\omega_{5}))_{h}(\mu_{h}))+\dim((L_{L}(\omega_{6}))_{h}(\mu_{h}))+(1-\varepsilon_{p}(3)+\varepsilon_{p}(2))\dim((L_{L}(\omega_{4}))_{h}(\mu_{h}))\leq 350-108\varepsilon_{p}(3)-86\varepsilon_{p}(2)$. Therefore, $\scale[0.9]{\displaystyle \max_{s\in T\setminus\ZG(G)}\dim(V_{s}(\mu))} \leq 350-108\varepsilon_{p}(3)-86\varepsilon_{p}(2)$.

For the unipotent elements, by Lemma \ref{uniprootelems}, we have $\scale[0.85]{\displaystyle \max_{u\in G_{u}\setminus \{1\}}\dim(V_{u}(1))=\max_{i=5,6}\dim(V_{x_{\alpha_{i}}(1)}(1))}$. Using the structure of $V\mid_{[L,L]}$ and Propositions \ref{PropositionClwedgecube}, \ref{Cl_oml_p=2}, \ref{C5om4} and \ref{C5om5}, we have $\dim(V_{x_{\alpha_{i}}(1)}(1)))\leq 407-164\varepsilon_{p}(3)+5\varepsilon_{p}(2)$, where equality holds for $x_{\alpha_{6}}(1)$ and $p\neq 2$. Therefore, $\scale[0.9]{\displaystyle \max_{u\in G_{u}\setminus \{1\}}\dim(V_{u}(1))}\leq 407-164\varepsilon_{p}(3)+5\varepsilon_{p}(2)$, where equality holds for $p\neq 2$, and so $\nu_{G}(V)\leq 165-44\varepsilon_{p}(3)-17\varepsilon_{p}(2)$, where equality holds for $p\neq 2$.
\end{proof}

\begin{prop}\label{C6om6}
Let $p\neq 2$, $\ell=6$ and $V=L_{G}(\omega_{6})$. Then $\nu_{G}(V)=132-11\varepsilon_{p}(3)$. Moreover, we have $\scale[0.9]{\displaystyle \max_{u\in G_{u}\setminus \{1\}}\dim(V_{u}(1))}=297-54\varepsilon_{p}(3)$ and $\scale[0.9]{\displaystyle \max_{s\in T\setminus\ZG(G)}\dim(V_{s}(\mu))}\leq 268-24\varepsilon_{p}(3)$, where equality holds for $p=3$. 
\end{prop}

\begin{proof}
Set $\lambda=\omega_{6}$ and $L=L_{1}$. By Lemma \ref{weightlevelCl}, we have $e_{1}(\lambda)=2$, therefore $\displaystyle V\mid_{[L,L]}=V^{0} \oplus V^{1} \oplus V^{2}$. By \cite[Proposition]{Smith_82} and Lemma \ref{dualitylemma}, we have $V^{0}\cong L_{L}(\omega_{6})$ and $V^{2}\cong L_{L}(\omega_{6})$. Now, the weight $\displaystyle (\lambda-\alpha_{1}-\cdots-\alpha_{6})\mid_{T_{1}}=\omega_{5}$ admits a maximal vector in $V^{1}$, thus $V^{1}$ has a composition factor isomorphic to $L_{L}(\omega_{5})$. As $\dim(V^{1})=165-44\varepsilon_{p}(3)$, we determine that:
\begin{equation}\label{DecompVB6om6pneq2}
V\mid_{[L,L]}\cong L_{L}(\omega_{6})\oplus L_{L}(\omega_{5})\oplus L_{L}(\omega_{6}).
\end{equation}

We start with the semisimple elements. Let $s\in T\setminus \ZG(G)$. If $\dim(V^{i}_{s}(\mu))=\dim(V^{i})$ for some eigenvalue $\mu$ of $s$ on $V$, where $0\leq i\leq 2$, then $s\in \ZG(L)^{\circ}\setminus \ZG(G)$. In this case, as $s$ acts on each $V^{i}$ as scalar multiplication by $c^{1-i}$ and $c\neq 1$, it follows that $\dim(V_{s}(\mu))\leq 264-20\varepsilon_{p}(3)$, where equality holds for $c=-1$ and $\mu=-1$. We thus assume that $\dim(V^{i}_{s}(\mu))<\dim(V^{i})$ for all eigenvalues $\mu$ of $s$ on $V$ and all $0\leq i\leq 2$. We write $s=z\cdot h$, where $z\in \ZG(L)^{\circ}$ and $h\in [L,L]$, and, using \eqref{DecompVB6om6pneq2} and Propositions \ref{C5om4}, and \ref{C5om5}, we have $\dim(V_{s}(\mu))\leq 2\dim((L_{L}(\omega_{6}))_{h}(\mu_{h}))+\dim((L_{L}(\omega_{5}))_{h}(\mu_{h}))\leq 268-24\varepsilon_{p}(3)$. Therefore,  $\scale[0.9]{\displaystyle \max_{s\in T\setminus\ZG(G)}\dim(V_{s}(\mu))}\leq 268-24\varepsilon_{p}(3)$, where equality holds for $p=3$. 

We focus on the unipotent elements. By Lemma \ref{uniprootelems}, we have $\scale[0.85]{\displaystyle \max_{u\in G_{u}\setminus \{1\}}\dim(V_{u}(1))=\max_{i=5,6}\dim(V_{x_{\alpha_{i}}(1)}(1))}$, and, using  \eqref{DecompVB6om6pneq2} and Propositions \ref{C5om4} and \ref{C5om5}, we have $\dim(V_{x_{\alpha_{i}}(1)}(1)))\leq 297-54\varepsilon_{p}(3)$, where equality holds for $i=6$. Therefore, $\scale[0.9]{\displaystyle \max_{u\in G_{u}\setminus \{1\}}\dim(V_{u}(1))}=297-54\varepsilon_{p}(3)$ and, as $\scale[0.9]{\displaystyle \max_{s\in T\setminus\ZG(G)}\dim(V_{s}(\mu))}\leq 268-24\varepsilon_{p}(3)$, where equality holds for $p=3$, it follows that $\nu_{G}(V)=132-11\varepsilon_{p}(3)$.
\end{proof}

\begin{prop}\label{C7om4}
Let $\ell=7$ and $V=L_{G}(\omega_{4})$. Then $\nu_{G}(V)\geq 208-12\varepsilon_{p}(5)-2\varepsilon_{p}(2)$, where equality holds for $p\neq 2,3$. Moreover, we have $\scale[0.9]{\displaystyle \max_{u\in G_{u}\setminus \{1\}}\dim(V_{u}(1))}\leq 702-78\varepsilon_{p}(5)-\varepsilon_{p}(3)+2\varepsilon_{p}(2)$, where equality holds for $p\neq 2,3$, and $\scale[0.9]{\displaystyle \max_{s\in T\setminus\ZG(G)}\dim(V_{s}(\mu))}\leq 565 -66\varepsilon_{p}(5)+23\varepsilon_{p}(3)-65\varepsilon_{p}(2)$.
\end{prop}

\begin{proof}
Set $\lambda=\omega_{4}$ and $L=L_{1}$. By Lemma \ref{weightlevelCl}, we have $e_{1}(\lambda)=2$, therefore $\displaystyle V\mid_{[L,L]}=V^{0} \oplus V^{1} \oplus V^{2}$. We argue as we did in the proof of Proposition \ref{C5om4}, to show that $V^{0}\cong L_{L}(\omega_{4})$, $V^{2}\cong L_{L}(\omega_{4})$ and that $V^{1}$ has $2-\varepsilon_{p}(5)+\varepsilon_{p}(2)$ composition factors: one isomorphic to $L_{L}(\omega_{5})$ and $1-\varepsilon_{p}(5)+\varepsilon_{p}(2)$ to $L_{L}(\omega_{3})$. Further, when $\varepsilon_{p}(2)=0$, we have $V^{1}\cong L_{L}(\omega_{5})\oplus L_{L}(\omega_{3})^{1-\varepsilon_{p}(5)}$.

First, let $s\in T\setminus \ZG(G)$. We argue as we did in the proof of Proposition \ref{C5om4} to show that for $s\in \ZG(L)^{\circ}\setminus \ZG(G)$ we have $\dim(V_{s}(\mu))\leq 494-74\varepsilon_{p}(5)-\varepsilon_{p}(3)$ for all eigenvalues $\mu$ of $s$ on $V$, while, for $s\notin \ZG(L)^{\circ}$, by Propositions \ref{PropositionClwedge}, \ref{PropositionClwedgecube} and \ref{C6om4}, we have $\dim(V_{s}(\mu))\leq 565 -66\varepsilon_{p}(5)+23\varepsilon_{p}(3)-65\varepsilon_{p}(2)$ for all eigenvalues $\mu$ of $s$ on $V$. Therefore, $\scale[0.9]{\displaystyle \max_{s\in T\setminus\ZG(G)}\dim(V_{s}(\mu))}\leq 565 -66\varepsilon_{p}(5)+23\varepsilon_{p}(3)-65\varepsilon_{p}(2)$. Secondly, in the case of the unipotent elements, we argue as in the proof of Proposition \ref{C5om4} to show that $\scale[0.9]{\displaystyle \max_{u\in G_{u}\setminus \{1\}}\dim(V_{u}(1))}\leq 702-78\varepsilon_{p}(5)-\varepsilon_{p}(3)+2\varepsilon_{p}(2)$, where equality holds for $p\neq 2,3$, and so $\nu_{G}(V)\geq 208-12\varepsilon_{p}(5)-2\varepsilon_{p}(2)$, where equality holds for $p\neq 2,3$.
\end{proof}

\begin{prop}\label{C7om6p=3}
Let $p=3$, $\ell=7$ and $V=L_{G}(\omega_{6})$. Then $\nu_{G}(V)=364$. Moreover, we have $\scale[0.9]{\displaystyle \max_{u\in G_{u}\setminus \{1\}}\dim(V_{u}(1))}$ $=729$ and $\scale[0.9]{\displaystyle \max_{s\in T\setminus\ZG(G)}\dim(V_{s}(\mu))}=728$.
\end{prop}

\begin{proof}
Set $\lambda=\omega_{6}$ and $L=L_{1}$. By Lemma \ref{weightlevelCl}, we have $e_{1}(\lambda)=2$, therefore $\displaystyle V\mid_{[L,L]}=V^{0} \oplus V^{1} \oplus V^{2}$. By \cite[Proposition]{Smith_82} and Lemma \ref{dualitylemma}, we have $V^{0}\cong L_{L}(\omega_{6})$ and $V^{2}\cong L_{L}(\omega_{6})$. Now, the weight $\displaystyle (\lambda-\alpha_{1}-\cdots-\alpha_{6})\mid_{T_{1}}=\omega_{7}$ admits a maximal vector in $V^{1}$, thus $V^{1}$ has a composition factor isomorphic to $L_{L}(\omega_{7})$. By dimensional considerations, we determine that:
\begin{equation}\label{DecompVB7om6p=3}
V\mid_{[L,L]}\cong L_{L}(\omega_{6})\oplus L_{L}(\omega_{7})\oplus L_{L}(\omega_{6}).
\end{equation}

We start with the semisimple elements. Let $s\in T\setminus \ZG(G)$. If $\dim(V^{i}_{s}(\mu))=\dim(V^{i})$ for some eigenvalue $\mu$ of $s$ on $V$, where $0\leq i\leq 2$, then $s\in \ZG(L)^{\circ}\setminus \ZG(G)$. In this case, as $s$ acts on each $V^{i}$ as scalar multiplication by $c^{1-i}$ and $c\neq 1$, it follows that $\dim(V_{s}(\mu))\leq 728$, where equality holds for $c=-1$ and $\mu=-1$. We thus assume that $\dim(V^{i}_{s}(\mu))<\dim(V^{i})$ for all eigenvalues $\mu$ of $s$ on $V$ and all $0\leq i\leq 2$. We write $s=z\cdot h$, where $z\in \ZG(L)^{\circ}$ and $h\in [L,L]$, and, using \eqref{DecompVB7om6p=3} and Propositions \ref{C6om5} and \ref{C6om6}, we have $\dim(V_{s}(\mu))\leq 2\dim((L_{L}(\omega_{6}))_{h}(\mu_{h}))+\dim((L_{L}(\omega_{7}))_{h}(\mu_{h}))\leq 728$. Therefore, $\scale[0.9]{\displaystyle \max_{s\in T\setminus\ZG(G)}\dim(V_{s}(\mu))}=728$.

We focus on the unipotent elements. By Lemma \ref{uniprootelems}, we have $\scale[0.85]{\displaystyle \max_{u\in G_{u}\setminus \{1\}}\dim(V_{u}(1))=\max_{i=6,7}\dim(V_{x_{\alpha_{i}}(1)}(1))}$, and, using \eqref{DecompVB7om6p=3} and Propositions \ref{C6om5} and \ref{C6om6}, we have $\dim(V_{x_{\alpha_{i}}(1)}(1)))\leq 729$, where equality holds for $i=7$. Therefore, $\scale[0.9]{\displaystyle \max_{u\in G_{u}\setminus \{1\}}\dim(V_{u}(1))}=729$, and so $\nu_{G}(V)=364$.
\end{proof}

\begin{prop}\label{C7om7p=3}
Let $p=3$, $\ell=7$ and $V=L_{G}(\omega_{7})$. Then $\nu_{G}(V)=364$. Moreover, we have $\scale[0.9]{\displaystyle \max_{u\in G_{u}\setminus \{1\}}\dim(V_{u}(1))}$ $=729$ and $\scale[0.9]{\displaystyle \max_{s\in T\setminus\ZG(G)}\dim(V_{s}(\mu))}=730$.
\end{prop}

\begin{proof}
Set $\lambda=\omega_{7}$ and $L=L_{1}$. By Lemma \ref{weightlevelCl}, we have $e_{1}(\lambda)=2$, therefore $\displaystyle V\mid_{[L,L]}=V^{0} \oplus V^{1} \oplus V^{2}$. By \cite[Proposition]{Smith_82} and Lemma \ref{dualitylemma}, we have $V^{0}\cong L_{L}(\omega_{7})$ and $V^{2}\cong L_{L}(\omega_{7})$. Now, the weight $\displaystyle (\lambda-\alpha_{1}-\cdots-\alpha_{7})\mid_{T_{1}}=\omega_{6}$ admits a maximal vector in $V^{1}$, thus $V^{1}$ has a composition factor isomorphic to $L_{L}(\omega_{6})$. By dimensional considerations, we determine that:
\begin{equation}\label{DecompVB7om7p=3}
V\mid_{[L,L]}\cong L_{L}(\omega_{7})\oplus L_{L}(\omega_{6})\oplus L_{L}(\omega_{7}).
\end{equation}

We start with the semisimple elements. Let $s\in T\setminus \ZG(G)$. If $\dim(V^{i}_{s}(\mu))=\dim(V^{i})$ for some eigenvalue $\mu$ of $s$ on $V$, where $0\leq i\leq 2$, then $s\in \ZG(L)^{\circ}\setminus \ZG(G)$. In this case, as $s$ acts on each $V^{i}$ as scalar multiplication by $c^{1-i}$ and $c\neq 1$, it follows that $\dim(V_{s}(\mu))\leq 730$, where equality holds for $c=-1$ and $\mu=-1$. We thus assume that $\dim(V^{i}_{s}(\mu))<\dim(V^{i})$ for all eigenvalues $\mu$ of $s$ on $V$ and all $0\leq i\leq 2$. We write $s=z\cdot h$, where $z\in \ZG(L)^{\circ}$ and $h\in [L,L]$, and, using \eqref{DecompVB7om7p=3} and Propositions \ref{C6om5} and \ref{C6om6}, we have $\dim(V_{s}(\mu))\leq 2\dim((L_{L}(\omega_{7}))_{h}(\mu_{h}))+\dim((L_{L}(\omega_{6}))_{h}(\mu_{h}))\leq 730$. Therefore, $\scale[0.9]{\displaystyle \max_{s\in T\setminus\ZG(G)}\dim(V_{s}(\mu))}=730$.

We focus on the unipotent elements. By Lemma \ref{uniprootelems}, we have $\scale[0.85]{\displaystyle \max_{u\in G_{u}\setminus \{1\}}\dim(V_{u}(1))=\max_{i=6,7}\dim(V_{x_{\alpha_{7}}(1)}(1))}$, and, using \eqref{DecompVB7om7p=3} and Propositions \ref{C6om5} and \ref{C6om6}, we have $\dim(V_{x_{\alpha_{i}}(1)}(1)))\leq 729$, where equality holds for $i=7$. Therefore, $\scale[0.9]{\displaystyle \max_{u\in G_{u}\setminus \{1\}}\dim(V_{u}(1))}=729$, and, as $\scale[0.9]{\displaystyle \max_{s\in T\setminus\ZG(G)}\dim(V_{s}(\mu))}=730$, it follows that $\nu_{G}(V)=364$.
\end{proof}

\begin{prop}\label{C7om5p=2}
Let $p=2$, $\ell=7$ and $V=L_{G}(\omega_{5})$. Then $\nu_{G}(V)\geq 340$. Moreover, we have $\scale[0.9]{\displaystyle \max_{u\in G_{u}\setminus \{1\}}\dim(V_{u}(1))}$ $\leq 948$ and $\scale[0.9]{\displaystyle \max_{s\in T\setminus\ZG(G)}\dim(V_{s}(\mu))}\leq 608$.
\end{prop}

\begin{proof}
Set $\lambda=\omega_{5}$ and $L=L_{1}$. By Lemma \ref{weightlevelCl}, we have $e_{1}(\lambda)=2$, therefore $\displaystyle V\mid_{[L,L]}=V^{0} \oplus V^{1} \oplus V^{2}$. By \cite[Proposition]{Smith_82} and Lemma \ref{dualitylemma}, we have $V^{0}\cong L_{L}(\omega_{5})$ and $V^{2}\cong L_{L}(\omega_{5})$. Now, the weight $\displaystyle (\lambda-\alpha_{1}-\cdots-\alpha_{5})\mid_{T_{1}}=\omega_{6}$ admits a maximal vector in $V^{1}$, thus $V^{1}$ has a composition factor isomorphic to $L_{L}(\omega_{6})$. By dimensional considerations, we determine that:
\begin{equation}\label{DecompVB7om5p=2}
V\mid_{[L,L]}\cong L_{L}(\omega_{5})\oplus L_{L}(\omega_{6})\oplus L_{L}(\omega_{5}).
\end{equation}

We start with the semisimple elements. Let $s\in T\setminus \ZG(G)$. If $\dim(V^{i}_{s}(\mu))=\dim(V^{i})$ for some eigenvalue $\mu$ of $s$ on $V$, where $0\leq i\leq 2$, then $s\in \ZG(L)^{\circ}\setminus \ZG(G)$. In this case, as $s$ acts on each $V^{i}$ as scalar multiplication by $c^{1-i}$ and $c\neq 1$, it follows that $\dim(V_{s}(\mu))\leq 560$ for all eigenvalues $\mu$ of $s$ on $V$. We thus assume that $\dim(V^{i}_{s}(\mu))<\dim(V^{i})$ for all eigenvalues $\mu$ of $s$ on $V$ and all $0\leq i\leq 2$. We write $s=z\cdot h$, where $z\in \ZG(L)^{\circ}$ and $h\in [L,L]$, and, using \eqref{DecompVB7om5p=2} and Propositions \ref{C6om4} and \ref{C6om5}, we have $\dim(V_{s}(\mu))\leq 2\dim((L_{L}(\omega_{5}))_{h}(\mu_{h}))+\dim((L_{L}(\omega_{6}))_{h}(\mu_{h}))\leq 608$. Therefore, $\scale[0.9]{\displaystyle \max_{s\in T\setminus\ZG(G)}\dim(V_{s}(\mu))}\leq 608$.

We focus on the unipotent elements. By Lemma \ref{uniprootelems}, we have $\scale[0.85]{\displaystyle \max_{u\in G_{u}\setminus \{1\}}\dim(V_{u}(1))=\max_{i=6,7}\dim(V_{x_{\alpha_{i}}(1)}(1))}$, and, using \eqref{DecompVB7om5p=2} and Propositions \ref{C6om4} and \ref{C6om5}, we have $\dim(V_{x_{\alpha_{i}}(1)}(1))\leq 948$, for $i=6,7$. Therefore, $\scale[0.9]{\displaystyle \max_{u\in G_{u}\setminus \{1\}}\dim(V_{u}(1))}\leq 948$, and so $\nu_{G}(V)\geq 340$.
\end{proof}

\begin{prop}\label{C8om4}
Let $\ell=8$ and $V=L_{G}(\omega_{4})$. Then $\nu_{G}(V)\geq 350-14\varepsilon_{p}(3)-16\varepsilon_{p}(2)$, where equality holds for $p\neq 2,3,5$. Moreover, we have $\scale[0.9]{\displaystyle \max_{u\in G_{u}\setminus \{1\}}\dim(V_{u}(1))}\leq 1350-\varepsilon_{p}(7)-105\varepsilon_{p}(3)-102\varepsilon_{p}(2)$, where equality holds for $p\neq 2,3,5$, and $\scale[0.9]{\displaystyle \max_{s\in T\setminus\ZG(G)}\dim(V_{s}(\mu))}\leq 1091-\varepsilon_{p}(7) -59\varepsilon_{p}(3)-175\varepsilon_{p}(2)$.
\end{prop}

\begin{proof}
Set $\lambda=\omega_{4}$ and $L=L_{1}$. By Lemma \ref{weightlevelCl}, we have $e_{1}(\lambda)=2$, therefore $\displaystyle V\mid_{[L,L]}=V^{0} \oplus V^{1} \oplus V^{2}$. We argue as we did in the proof of Proposition \ref{C5om4}, to show that $V^{0}\cong L_{L}(\omega_{4})$, $V^{2}\cong L_{L}(\omega_{4})$ and that $V^{1}$ has $2+\varepsilon_{p}(5)-\varepsilon_{p}(3)-\varepsilon_{p}(2)$ composition factors: one isomorphic to $L_{L}(\omega_{5})$ and $1+\varepsilon_{p}(5)-\varepsilon_{p}(3)-\varepsilon_{p}(2)$ to $L_{L}(\omega_{3})$. Further, when $\varepsilon_{p}(5)=0$, then $V^{1}\cong L_{L}(\omega_{5})\oplus L_{L}(\omega_{3})^{1-\varepsilon_{p}(3)-\varepsilon_{p}(2)}$.

First, let $s\in T\setminus \ZG(G)$. We argue as we did in the proof of Proposition \ref{C5om4} to show that for $s\in \ZG(L)^{\circ}\setminus \ZG(G)$ we have $\dim(V_{s}(\mu))\leq 1000-\varepsilon_{p}(7)-91\varepsilon_{p}(3)-90\varepsilon_{p}(2)$ for all eigenvalues $\mu$ of $s$ on $V$, while, for $s\notin \ZG(L)^{\circ}$, by Propositions \ref{PropositionClwedge}, \ref{PropositionClwedgecube} and \ref{C7om4}, we have $\dim(V_{s}(\mu))\leq 1091-\varepsilon_{p}(7) -59\varepsilon_{p}(3)-175\varepsilon_{p}(2)$ for all eigenvalues $\mu$ of $s$ on $V$. Therefore, $\scale[0.9]{\displaystyle \max_{s\in T\setminus\ZG(G)}\dim(V_{s}(\mu))}\leq 1091-\varepsilon_{p}(7) -59\varepsilon_{p}(3)-175\varepsilon_{p}(2)$. Secondly, in the case of the unipotent elements, we argue as in the proof of Proposition \ref{C5om4} to show that $\scale[0.9]{\displaystyle \max_{u\in G_{u}\setminus \{1\}}\dim(V_{u}(1))}\leq 1350-\varepsilon_{p}(7)-105\varepsilon_{p}(3)-102\varepsilon_{p}(2)$, where equality holds for $p\neq 2,3,5$, and so $\nu_{G}(V)\geq 350-14\varepsilon_{p}(3)-16\varepsilon_{p}(2)$, where equality holds for $p\neq 2,3,5$.
\end{proof}

\begin{prop}\label{C9om4}
Let $\ell=9$ and $V=L_{G}(\omega_{4})$. Then $\nu_{G}(V)\geq 544-16\varepsilon_{p}(7)-2\varepsilon_{p}(2)$, where equality holds for $p\neq 2,3,5$. Moreover, we have $\scale[0.9]{\displaystyle \max_{u\in G_{u}\setminus \{1\}}\dim(V_{u}(1))}\leq 2363-136\varepsilon_{p}(7)+\varepsilon_{p}(2)$, where equality holds for $p\neq 2,3,5$, and $\scale[0.9]{\displaystyle \max_{s\in T\setminus\ZG(G)}\dim(V_{s}(\mu))}\leq 1930-119\varepsilon_{p}(7) +40\varepsilon_{p}(3)-106\varepsilon_{p}(2)$.
\end{prop}

\begin{proof}
Set $\lambda=\omega_{4}$ and $L=L_{1}$. By Lemma \ref{weightlevelCl}, we have $e_{1}(\lambda)=2$, therefore $\displaystyle V\mid_{[L,L]}=V^{0} \oplus V^{1} \oplus V^{2}$. We argue as we did in the proof of Proposition \ref{C5om4}, to show that $V^{0}\cong L_{L}(\omega_{4})$, $V^{2}\cong L_{L}(\omega_{4})$ and that $V^{1}$ has $2-\varepsilon_{p}(7)+\varepsilon_{p}(3)+\varepsilon_{p}(2)$ composition factors: one isomorphic to $L_{L}(\omega_{5})$ and $1-\varepsilon_{p}(7)+\varepsilon_{p}(3)+\varepsilon_{p}(2)$ isomorphic to $L_{L}(\omega_{3})$. Further, when $\varepsilon_{p}(6)=0$, we have $V^{1}\cong L_{L}(\omega_{5})\oplus L_{L}(\omega_{3})^{1-\varepsilon_{p}(7)}$.

First, let $s\in T\setminus \ZG(G)$. We argue as we did in the proof of Proposition \ref{C5om4} to show that for $s\in \ZG(L)^{\circ}\setminus \ZG(G)$ we have $\dim(V_{s}(\mu))\leq 1819-120\varepsilon_{p}(7)-\varepsilon_{p}(2)$ for all eigenvalues $\mu$ of $s$ on $V$, while for $s\notin \ZG(L)^{\circ}$, by Propositions \ref{PropositionClwedge}, \ref{PropositionClwedgecube} and \ref{C8om4}, we have $\dim(V_{s}(\mu))\leq 1930-119\varepsilon_{p}(7) +40\varepsilon_{p}(3)-106\varepsilon_{p}(2)$ for all eigenvalues $\mu$ of $s$ on $V$. Therefore, $\scale[0.9]{\displaystyle \max_{s\in T\setminus\ZG(G)}\dim(V_{s}(\mu))}\leq 1930-119\varepsilon_{p}(7) +40\varepsilon_{p}(3)-106\varepsilon_{p}(2)$. Secondly, in the case of the unipotent elements, we argue as in the proof of Proposition \ref{C5om4} to show that $\scale[0.9]{\displaystyle \max_{u\in G_{u}\setminus \{1\}}\dim(V_{u}(1))}\leq 2363-136\varepsilon_{p}(7)+\varepsilon_{p}(2)$, where equality holds for $p\neq 2,3,5$, and so $\nu_{G}(V)\geq 544-16\varepsilon_{p}(7)-2\varepsilon_{p}(2)$, where equality holds for $p\neq 2,3,5$.
\end{proof}

\section{Proof of Theorem \ref{ResultsBl}}
In this section the characteristic of the algebraically closed field $k$ is different than $2$. We denote by $G$, respectively by $\tilde{G}$, a simple adjoint, respectively simply connected, linear algebraic group of type $B_{\ell}$, $\ell\geq 3$. We fix a central isogeny $\phi:\tilde{G}\to G$ with $\ker(\phi)\subseteq \ZG(\tilde{G})$ and $d\phi\neq 0$, and let $\tilde{T}$, respectively $\tilde{B}$, be a maximal torus, respectively Borel subgroup, in $\tilde{G}$ with the property that $\phi(\tilde{T})=T$, respectively $\phi(\tilde{B})=B$. We denote by $\X(\tilde{T})$, $\ZG(\tilde{G})$, $\tilde{G}_{u}$, $\tilde{\Delta}=\{\tilde{\alpha}_{1},\dots,\tilde{\alpha}_{\ell}\}$ and $\tilde{\omega}_{1},\dots, \tilde{\omega}_{\ell}$ the character group of $\tilde{T}$, the center of $\tilde{G}$, the set of unipotent elements in $\tilde{G}$, the set of simple roots in $\tilde{G}$ corresponding to $\tilde{B}$ and the fundamental dominant weights of $\tilde{G}$ corresponding to $\tilde{\Delta}$. Further, we denote by $\tilde{L}_{i}$ a Levi subgroup of the maximal parabolic subgroup $\tilde{P}_{i}$ of $\tilde{G}$ corresponding to $\tilde{\Delta}_{i}=\tilde{\Delta}\setminus\{\tilde{\alpha}_{i}\}$, $1\leq i\leq \ell$, and we let $\tilde{T}_{i}=\tilde{T}\cap [\tilde{L}_{i},\tilde{L}_{i}]$. Now, for $\tilde{\lambda}\in \X(\tilde{T})$ a $p$-restricted dominant weight, we let $\tilde{V}=L_{\tilde{G}}(\tilde{\lambda})$ and we have $\scale[0.9]{\displaystyle \tilde{V}\mid_{[\tilde{L}_{i}, \tilde{L}_{i}]}=\bigoplus_{j=0}^{e_{i}(\tilde{\lambda})}\tilde{V}^{j}}$, where $\scale[0.9]{\displaystyle \tilde{V}^{j}=\bigoplus_{\tilde{\gamma}\in \mathbb{N}_{\geq 0}\tilde{\Delta}_{i}}\tilde{V}_{\tilde{\lambda}-j\tilde{\alpha}_{i}-\tilde{\gamma}}}$, $\scale[0.9]{0\leq j\leq e_{i}(\tilde{\lambda})}$. Lastly, we fix the following hypothesis on the semisimple elements of $G$:
\begin{equation*}
\scale[0.9]{\begin{split}
(^{\dagger}H_{s}): & \text{ any }s\in T\setminus \ZG(G) \text{ is such that }s=\diag(\mu_{1}\cdot \I_{n_{1}},\dots,\mu_{m}\cdot \I_{n_{m}},1\cdot \I_{n},\mu_{m}^{-1}\cdot \I_{n_{m}},\dots,\mu_{1}^{-1}\cdot \I_{n_{1}}) \text{ with }\mu_{i}\neq \mu_{j}^{\pm 1},\ i<j, \\\
				 &\mu_{i}\neq 1, \ 1\leq i\leq m, \text{ and } \displaystyle n+2\sum_{i=1}^{m}n_{i}=2\ell+1, \text{ where }1\leq n\leq 2\ell-1 \text{ and }\ell\geq n_{1}\geq \cdots\geq n_{m}\geq 1.
\end{split}}
\end{equation*}

\begin{prop}\label{PropositionBlnatural}
Let $\tilde{V}=L_{\tilde{G}}(\tilde{\omega}_{1})$. Then $\nu_{\tilde{G}}(\tilde{V})=1$.  Moreover, we have $\scale[0.9]{\displaystyle \max_{\tilde{u}\in \tilde{G}_{u}\setminus \{1\}}\dim(\tilde{V}_{\tilde{u}}(1))=2\ell-1}$ and $\scale[0.9]{\displaystyle \max_{\tilde{s}\in \tilde{T}\setminus\ZG(\tilde{G})}\dim(\tilde{V}_{\tilde{s}}(\tilde{\mu}))=2\ell}$.
\end{prop}

\begin{proof}
First, we note that $\tilde{V}\cong L_{G}(\omega_{1})$ as $k\tilde{G}$-modules. To ease notation, let $V=L_{G}(\omega_{1})$. Further, we note that $V\cong W$ as $kG$-modules. Let $s\in T\setminus \ZG(G)$ and let $\mu$ be an eigenvalue of $s$ on $V$. If $\mu\neq \mu^{-1}$, then $\dim(V_{s}(\mu))\leq \ell$, as $\dim(V_{s}(\mu))=\dim(V_{s}(\mu^{-1}))$. On the other hand, for $\mu=1$, as $s\notin \ZG(G)$, we have $\dim(V_{s}(1))\leq 2\ell-1$. Lastly, for $\mu=-1$, we have $\dim(V_{s}(-1))\leq 2\ell$, where equality holds if and only if, up to conjugation, $s=\diag(-1,\dots, -1, 1,-1,\dots, -1)$. In the case of unipotent elements, in view of Lemma \ref{uniprootelems}, we have $\scale[0.9]{\displaystyle\max_{u\in G_{u}\setminus \{1\}} \dim(V_{u}(1))=\max_{i=1,\ell}\dim(V_{x_{\alpha_{i}}(1)}(1))=2\ell-1}$. Therefore, by Lemma \ref{LtildeGtildelambdaandLGlambda}, we conclude that $\scale[0.9]{\displaystyle \max_{\tilde{s}\in \tilde{T}\setminus\ZG(\tilde{G})}\dim(\tilde{V}_{\tilde{s}}(\tilde{\mu}))}=2\ell$ and $\scale[0.9]{\displaystyle \max_{\tilde{u}\in \tilde{G}_{u}\setminus \{1\}}\dim(\tilde{V}_{\tilde{u}}(1))=2\ell-1}$, thus $\scale[0.9]{\nu_{\tilde{G}}(\tilde{V})=1}$.
\end{proof}

\begin{prop}\label{PropositionBlwedge}
Let $\tilde{V}=L_{\tilde{G}}(\tilde{\omega}_{2})$. Then $\nu_{\tilde{G}}(\tilde{V})=2\ell$. Moreover, we have $\scale[0.9]{\displaystyle \max_{\tilde{u}\in \tilde{G}_{u}\setminus \{1\}}\dim(\tilde{V}_{\tilde{u}}(1))=2\ell^{2}-3\ell+4}$ and $\scale[0.9]{\displaystyle \max_{\tilde{s}\in \tilde{T}\setminus\ZG(\tilde{G})}\dim(\tilde{V}_{\tilde{s}}(\tilde{\mu}))= 2\ell^{2}-\ell}$.
\end{prop}

\begin{proof}
To begin, we note that $\tilde{V}\cong L_{G}(\omega_{2})$ as $k\tilde{G}$-modules. To ease notation, let $V=L_{G}(\omega_{2})$. Further, by \cite[Proposition $4.2.2$]{mcninch_1998}, we have $V\cong \wedge^{2}(W)$ as $kG$-modules.

We start with the semisimple elements. Let $s\in T\setminus \ZG(G)$ be as in hypothesis $(^{\dagger}H_{s})$. Since $V\cong \wedge^{2}(W)$, we deduce that the eigenvalues of $s$ on $V$, not necessarily distinct, are:
\begin{equation}\label{enum: Bl_P1}
\scale[0.9]{\begin{cases}
\mu_{i}^{2}$ and $\mu_{i}^{-2}$, where $1\leq i\leq m$, each with multiplicity at least $\frac{n_{i}(n_{i}-1)}{2};\\
\mu_{i}\mu_{j}$ and $\mu_{i}^{-1}\mu_{j}^{-1}$, where $1\leq i<j\leq m$, each with multiplicity at least $n_{i}n_{j};\\
\mu_{i}\mu_{j}^{-1}$ and $\mu_{i}^{-1}\mu_{j}$, where $1\leq i<j\leq m$, each with multiplicity at least $n_{i}n_{j};\\
\mu_{i}$ and $\mu_{i}^{-1}$, where $1\leq i\leq m$, each with multiplicity at least $nn_{i};\\
1$ with multiplicity at least $\frac{n(n-1)}{2}+\displaystyle \sum_{r=1}^{m}n_{r}^{2}.
\end{cases}}
\end{equation}

Let $\mu\in k^{*}$ be an eigenvalue of $s$ on $V$. If $\mu\neq \mu^{-1}$, then $\dim(V_{s}(\mu))\leq \dim(V)-\dim(V_{s}(\mu^{-1}))$, and, since $V$ is self-dual, we get $\dim(V_{s}(\mu))\leq \frac{2\ell^{2}+\ell}{2}<2\ell^{2}-\ell$. We thus assume that $\mu=\pm 1$.

First, consider the case of $m=1$. Then $\mu_{1}\neq 1$, as $s\notin \ZG(G)$, and, by \eqref{enum: Bl_P1}, the eigenvalues of $s$ on $V$, not necessarily distinct, are $\mu_{1}^{\pm 2}$ each with multiplicity at least $\frac{n_{1}(n_{1}-1)}{2}$; $\mu^{\pm 1}$ each with multiplicity at least $nn_{1}$; and $1$ with multiplicity at least $\frac{n(n-1)}{2}+n_{1}^{2}$. This gives $\dim(V_{s}(1))\leq 2\ell^{2}+\ell-2nn_{1}$. Suppose that $\dim(V_{s}(1))\geq 2\ell^{2}-\ell$. Then, keeping in mind that $2n_{1}=2\ell+1-n$, it follows that $(2\ell-n)(1-n)\geq 0$. Further, as $1\leq n\leq 2\ell-1$, the inequality $(2\ell-n)(1-n)\geq 0$ holds if and only if $n=1$ and $n_{1}=\ell$, i.e. $s=\diag(\mu_{1},\dots, \mu_{1},1,\mu_{1}^{-1},\dots,\mu_{1}^{-1})$. Then, we have $\dim(V_{s}(1))\leq 2\ell^{2}-\ell$, where equality holds if and only if all eigenvalues of $s$ on $V$ different than $\mu_{1}^{\pm 1}$ are equal to $1$, hence if and only if $\mu_{1}^{2}=1$, i.e. $\mu_{1}=-1$. Now, let $\mu=-1$. If $\mu_{1}=-1$, then $\dim(V_{s}(-1))=2nn_{1}<2\ell^{2}-\ell$, since $2n_{1}=2\ell+1-n$, $1\leq n\leq 2\ell-1$ and $\ell\geq 3$. Similarly, if $\mu_{1}^{2}=-1$, then $\dim(V_{s}(-1))=n_{1}(n_{1}-1)<2\ell^{2}-\ell$, since $1\leq n_{1}\leq \ell$. Therefore $\dim(V_{s}(-1))<2\ell^{2}-\ell$ for all $s\in T\setminus \ZG(G)$ with $m=1$. 

We thus assume that $m\geq 2$. We start with $\mu=1$. Since $\mu_{i}\neq \mu_{j}^{\pm 1}$ for all $ i<j$, by \eqref{enum: Bl_P1}, there are at least $\scale[0.9]{4\displaystyle \sum_{i<j}n_{i}n_{j}+2n\displaystyle \sum_{i=1}^{m}n_{i}}$ eigenvalues of $s$ on $V$ different than $1$. This gives $\scale[0.9]{\displaystyle\dim(V_{s}(1)) \leq 2\ell^{2}+\ell-4\sum_{i<j}n_{i}n_{j}-2n\sum_{i=1}^{m}n_{i}}$. If $\dim(V_{s}(1))\geq 2\ell^{2}-\ell$, then:
\begin{equation}\label{ineqwedgeforsymBl}
\scale[0.9]{\displaystyle 2\ell-4\sum_{i<j}n_{i}n_{j}-2n\sum_{i=1}^{m}n_{i}\geq 0}
\end{equation}
and, since $\scale[0.9]{2\displaystyle \sum_{i=1}^{m}n_{i}=2\ell+1-n}$, it follows that $\scale[0.9]{\displaystyle(2\ell-n)(1-n)-4\sum_{i<j}n_{i}n_{j}\geq 0}$. However, since $1\leq n\leq 2\ell-1$, we have $(2\ell-n)(1-n)\leq 0$, while, since $m\geq 2$, we have $\scale[0.9]{-4\displaystyle \sum_{i<j}n_{i}n_{j}<0}$, therefore $\scale[0.9]{\displaystyle(2\ell-n)(1-n)-4\sum_{i<j}n_{i}n_{j}<0}$ and, consequently, $\dim(V_{s}(1))<2\ell^{2}-\ell$. 

Lastly, let $\mu=-1$. If $\mu_{i}\neq -1$ for all $i$, then $\scale[0.9]{\displaystyle \dim(V_{s}(-1)) \leq 2\ell^{2}+\ell-\frac{n(n-1)}{2}-\sum_{r=1}^{m}n_{r}^{2}-2n\sum_{r=1}^{m}n_{r}}$. Suppose that $\dim(V_{s}(-1))\geq 2\ell^{2}-\ell$. Then $\scale[0.9]{\displaystyle 2\ell-\frac{n(n-1)}{2}-\sum_{r=1}^{m}n_{r}^{2}-2n\sum_{r=1}^{m}n_{r}\geq 0}$ and, since $\scale[0.9]{2\displaystyle \sum_{r=1}^{m}n_{r}=2\ell+1-n}$, we must have $\scale[0.9]{\displaystyle(2\ell-n)(1-n)-\frac{n(n-1)}{2}-\sum_{r=1}^{m}n_{r}^{2}\geq 0}$. However, $(2\ell-n)(1-n)\leq 0$, $-\frac{n(n-1)}{2}\leq 0$ and $\scale[0.9]{-\displaystyle \sum_{r=1}^{m}n_{r}^{2}<0}$, as $n\geq 1$ and $n_{r}\geq 1$. Therefore $\scale[0.9]{\displaystyle2\ell-\frac{n(n-1)}{2}-\sum_{r=1}^{m}n_{r}^{2}-2n\sum_{r=1}^{m}n_{r}<0}$ and, consequently, $\dim(V_{s}(-1))<2\ell^{2}-\ell$. We thus assume that there exist $i$ such that $\mu_{i}=-1$. Then, since the $\mu_{i}$'s are distinct and different than $1$, by \eqref{enum: Bl_P1}, we determine that $\scale[0.9]{\displaystyle\dim(V_{s}(-1))\leq 2\ell^{2}+\ell-\frac{n(n-1)}{2}-\sum_{r=1}^{m}n_{r}^{2}-2n\sum_{r\neq i}n_{r}-4n_{i}\sum_{r\neq i}n_{r}-n_{i}(n_{i}-1)}$. Suppose that  $\dim(V_{s}(-1))\geq 2\ell^{2}-\ell$. Then
\begin{equation}\label{Bl_-1}
\scale[0.9]{\displaystyle2\ell-\frac{n(n-1)}{2}-\sum_{r=1}^{m}n_{r}^{2}-2n\sum_{r\neq i}n_{r}-4n_{i}\sum_{r\neq i}n_{r}-n_{i}(n_{i}-1)\geq 0}.
\end{equation}
We have that $\scale[0.9]{\displaystyle \sum_{r=1}^{m}n_{r}^{2}\geq \displaystyle \sum_{r=1}^{m}n_{r}}$, as $n_{r}\geq 1$, and that $\scale[0.9]{2\ell=\displaystyle 2\sum_{r=1}^{m}n_{r}+n-1}$. Substituting in (\ref{Bl_-1}) gives $\scale[0.9]{\displaystyle \sum_{r\neq i}n_{r}(1-2n)+}$ $\scale[0.9]{\displaystyle n_{i}(2-n_{i}-4\sum_{r\neq i}n_{r})}-\frac{(n-1)(n-2)}{2}\geq 0$. However, as $n\geq 1$, we have $-\frac{(n-1)(n-2)}{2}\leq 0$ and $\scale[0.9]{\displaystyle \sum_{r\neq i}n_{r}(1-2n)<0}$. Further, since $m\geq 2$ and $n_{r}\geq 1$, we have $\scale[0.9]{n_{i}(2-n_{i}-4\displaystyle \sum_{r\neq i}n_{r})<0}$. Therefore $\scale[0.9]{\displaystyle\sum_{r\neq i}n_{r}(1-2n)+n_{i}(2-n_{i}-4\sum_{r\neq i}n_{r})}$ $-\frac{(n-1)(n-2)}{2}<0$ and, consequently, $\dim(V_{s}(-1))<2\ell^{2}-\ell$ for all $s\in T\setminus \ZG(G)$ with $m\geq 2$. Having considered all cases, we deduce that $\displaystyle \max_{s\in T\setminus\ZG(G)}\dim(V_{s}(\mu))= 2\ell^{2}-\ell$.

We now focus on the unipotent elements. By Lemma \ref{uniprootelems}, we have $\scale[0.9]{\displaystyle\max_{u\in G_{u}\setminus\{1\}}\dim(V_{u}(1))=\max_{i=1,\ell}\dim(V_{x_{\alpha_{i}}(1)}(1))}$. For $x_{\alpha_{1}}(1)$, we write $W=W_{1}\oplus W_{2}$, where $\dim(W_{1})=3$ and $x_{\alpha_{1}}(1)$ acts as $J_{3}$ on $W_{1}$ and $\dim(W_{2})=2\ell-2$ and $x_{\alpha_{1}}(1)$ acts trivially on $W_{2}$. Using \cite[Lemma $3.4$]{liebeck_2012unipotent}, we determine that $\dim(V_{x_{\alpha_{1}}(1)}(1))=2\ell^{2}-3\ell+2$. Similarly, for $x_{\alpha_{\ell}}(1)$, we write $W=W'_{1}\oplus W'_{2}$, where $\dim(W'_{1})=4$ and $x_{\alpha_{\ell}}(1)$ acts as $J_{2}^{2}$ on $W'_{1}$ and $\dim(W'_{2})=2\ell-3$ and $x_{\alpha_{\ell}}(1)$ acts trivially on $W'_{2}$. Once more, using  \cite[Lemma $3.4$]{liebeck_2012unipotent}, we determine that $\dim(V_{x_{\alpha_{\ell}}(1)}(1))=2\ell^{2}-3\ell+4$. Therefore, $\displaystyle \max_{u\in G_{u}\setminus \{1\}}\dim(V_{u}(1))=2\ell^{2}-3\ell+4$. To conclude, by Lemma \ref{LtildeGtildelambdaandLGlambda}, we have shown that $\scale[0.9]{\displaystyle \max_{\tilde{u}\in \tilde{G}_{u}\setminus \{1\}}\dim(\tilde{V}_{\tilde{u}}(1))=2\ell^{2}-3\ell+4}$ and $\scale[0.9]{\displaystyle \max_{\tilde{s}\in \tilde{T}\setminus\ZG(\tilde{G})}\dim(\tilde{V}_{\tilde{s}}(\tilde{\mu}))= 2\ell^{2}-\ell}$, and so $\nu_{\tilde{G}}(\tilde{V})=2\ell$.
\end{proof}

\begin{prop}\label{PropositionBlsymm}
Let $\tilde{V}=L_{\tilde{G}}(2\tilde{\omega}_{1})$. Then $\nu_{\tilde{G}}(\tilde{V})=2\ell$. Moreover, we have $\scale[0.9]{\displaystyle \max_{\tilde{u}\in \tilde{G}_{u}\setminus \{1\}}\dim(\tilde{V}_{\tilde{u}}(1))=2\ell^{2}-}$ $\scale[0.9]{\ell-\varepsilon_{p}(2\ell+1)}$ and $\scale[0.9]{\displaystyle \max_{\tilde{s}\in \tilde{T}\setminus\ZG(\tilde{G})}\dim(\tilde{V}_{\tilde{s}}(\tilde{\mu}))=2\ell^{2}+\ell-\varepsilon_{p}(2\ell+1)}$.
\end{prop}

\begin{proof}
First, let $V=L_{G}(2\omega_{1})$ and note that $\tilde{V}\cong L_{G}(2\omega_{1})$ as $k\tilde{G}$-modules. Further, by \cite[Proposition $4.7.3$]{mcninch_1998}, we have $\SW(W)\cong V\oplus L_{G}(0)$ if $\varepsilon_{p}(2\ell+1)=0$, and $\SW(W)\cong L_{G}(0)\mid V\mid L_{G}(0)$ if $\varepsilon_{p}(2\ell+1)=1$. 

We begin with the semisimple elements. Let $s\in T\setminus \ZG(G)$ be as in hypothesis $(^{\dagger}H_{s})$. By the structure of $\SW(W)$ as a $kG$-module, we determine that the eigenvalues of $s$ on $V$, not necessarily distinct, are:

\begin{equation}\label{enum: Bl_2}
\scale[0.9]{\begin{cases}
\mu_{i}^{2}$ and $\mu_{i}^{-2}$, $1\leq i\leq m$, each with multiplicity at least $\frac{n_{i}(n_{i}+1)}{2};\\
\mu_{i}\mu_{j}$ and $\mu_{i}^{-1}\mu_{j}^{-1}$, $1\leq i<j\leq m$, each with multiplicity at least $n_{i}n_{j};\\
\mu_{i}\mu_{j}^{-1}$ and $\mu_{i}^{-1}\mu_{j}$, $1\leq i<j\leq m$, each with multiplicity at least $n_{i}n_{j};\\
\mu_{i}$ and $\mu_{i}^{-1}$, $1\leq i\leq m$, each with multiplicity at least $nn_{i};\\
1$ with multiplicity at least $\displaystyle \sum_{r=1}^{m}n_{r}^{2}+$ $\frac{n(n+1)}{2}-1-\varepsilon_{p}(2\ell+1).
\end{cases}}
\end{equation}

Let $\mu\in k^{*}$ be an eigenvalue of $s$ on $V$. If $\mu\neq \mu^{-1}$, then $\dim(V_{s}(\mu))\leq \dim(V)-\dim(V_{s}(\mu^{-1}))$ and, since $V$ is self-dual, we have $\dim(V_{s}(\mu))\leq \frac{2\ell^{2}+3\ell-\varepsilon_{p}(2\ell+1)}{2}<2\ell^{2}+\ell-\varepsilon_{p}(2\ell+1)$. We thus assume that $\mu=\pm 1$.

First, consider the case of $m=1$. Then $\mu_{1}\neq 1$, as $s\notin \ZG(G)$, and by \eqref{enum: Bl_2}, the eigenvalues of $s$ on $V$ are $\mu_{1}^{\pm 2}$ each with multiplicity at least $\frac{n_{1}(n_{1}+1)}{2}$; $\mu_{1}^{\pm 1}$ each with multiplicity at least $nn_{1}$; and $1$ with multiplicity at least $\frac{n(n+1)}{2}+n_{1}^{2}-1-\varepsilon_{p}(2\ell+1)$. This gives $\dim(V_{s}(1))\leq 2\ell^{2}+3\ell-\varepsilon_{p}(2\ell+1)-2nn_{1}$. If $\dim(V_{s}(1))\geq 2\ell^{2}+\ell-\varepsilon_{p}(2\ell+1)$, then, keeping in mind that $2n_{1}=2\ell+1-n$, we must have $(2\ell-n)(1-n)\geq 0$. As $1\leq n\leq 2\ell-1$, this inequality holds if and only if $n=1$ and $n_{1}=\ell$, i.e. $s=\diag(\mu_{1},\dots,\mu_{1},1,\mu_{1}^{-1},\dots, \mu_{1}^{-1})$. In this case, we have $\dim(V_{s}(1))\leq 2\ell^{2}+\ell-\varepsilon_{p}(2\ell+1)$, where equality holds if and only if all eigenvalues of $s$ on $V$ different than $\mu_{1}^{\pm 1}$ are equal to $1$, hence if and only if $\mu_{1}^{2}=1$, i.e. $\mu_{1}=-1$. For $\mu=-1$ we have $\dim(V_{s}(-1))\leq 2\ell^{2}+3\ell+1-n_{1}^{2}-\frac{n(n+1)}{2}$. If $\dim(V_{s}(-1))\geq 2\ell^{2}+\ell-\varepsilon_{p}(2\ell+1)$, then $2\ell+1-n_{1}^{2}-\frac{n(n+1)}{2}+\varepsilon_{p}(2\ell+1)\geq 0$ and, since $2\ell+1=2n_{1}+n$, we have $-(n_{1}-1)^{2}-n\frac{1-n}{2}+1+\varepsilon_{p}(2\ell+1)\geq 0$, which does not hold. We deduce that $\dim(V_{s}(-1))<2\ell^{2}+\ell-\varepsilon_{p}(2\ell+1)$ for all $s\in T\setminus \ZG(G)$ with $m=1$.

We thus assume that $m\geq 2$. For $\mu=1$, as $\mu_{i}^{\pm 1}\mu_{j}^{\pm 1}\neq 1$ for all $ i<j$, by \eqref{enum: Bl_2}, we determine that $\scale[0.9]{\displaystyle \dim(V_{s}(1))\leq 2\ell^{2}+3\ell-\varepsilon_{p}(2\ell+1)-2n\sum_{i=1}^{m}n_{i}-4\sum_{i<j}n_{i}n_{j}}$. If $\dim(V_{s}(1))\geq 2\ell^{2}+\ell-\varepsilon_{p}(2\ell+1)$, then $\scale[0.9]{\displaystyle 2\ell-2n\sum_{i=1}^{m}n_{i}-4\sum_{i<j}n_{i}n_{j}\geq 0}$, which does not hold, see \eqref{ineqwedgeforsymBl}. Therefore, $\dim(V_{s}(1))<2\ell^{2}+\ell-\varepsilon_{p}(2\ell+1)$ for all $s\in T\setminus \ZG(G)$ with $m\geq 2$.

Lastly, let $\mu=-1$. Suppose that $\mu_{i}\neq -1$ for all $i$. Then, by \eqref{enum: Bl_2}, we have $\scale[0.88]{\displaystyle \dim(V_{s}(-1))\leq 2\ell^{2}+3\ell+1-\sum_{r=1}^{m}n_{r}^{2}}$ $-\frac{n(n+1)}{2}$ $\scale[0.9]{\displaystyle-2n\sum_{r=1}^{m}n_{r}}$. If $\dim(V_{s}(-1))\geq 2\ell^{2}+\ell-\varepsilon_{p}(2\ell+1)$, then $\scale[0.9]{\displaystyle 2\ell+1-\sum_{r=1}^{m}n_{r}^{2}}$ $-\frac{n(n+1)}{2}-$ $\scale[0.9]{\displaystyle 2n\sum_{r=1}^{m}n_{r}+\varepsilon_{p}(2\ell+1)\geq 0}$ and, since $\scale[0.9]{2\ell+1=n+\displaystyle 2\sum_{r=1}^{m}n_{r}}$, we get $n\frac{1-n}{2}+\scale[0.9]{\displaystyle \sum_{r=1}^{m}n_{r}(2-n_{r}-2n)+\varepsilon_{p}(2\ell+1)\geq 0}$, which, as $n\geq 1$ and $m\geq 2$, does not hold. We thus assume that there exists $i$ such that $\mu_{i}=-1$. Then $\mu_{r}^{\pm 1}\neq -1$ for all $r\neq i$ and, by \eqref{enum: Bl_2}, we determine that $\scale[0.9]{\displaystyle\dim(V_{s}(-1))\leq 2\ell^{2}+3\ell+1-\sum_{r=1}^{m}n_{r}^{2}}$ $-\frac{n(n+1)}{2}-$ $\scale[0.9]{\displaystyle 2n\sum_{r\neq i}n_{r}-n_{i}(n_{i}+1)}$. Assume $\dim(V_{s}(-1))\geq 2\ell^{2}+\ell-\varepsilon_{p}(2\ell+1)$. Then $\scale[0.9]{\displaystyle 2\ell+1-\sum_{r=1}^{m}n_{r}^{2}-}$ $\frac{n(n+1)}{2}-$ $\scale[0.9]{\displaystyle 2n\sum_{r\neq i}n_{r}-n_{i}(n_{i}+1)+\varepsilon_{p}(2\ell+1)\geq 0}$ and, since $\scale[0.9]{2\ell+1=n+2\displaystyle\sum_{r=1}^{m}n_{r}}$, we must have $n\frac{1-n}{2}+\scale[0.9]{\displaystyle \sum_{r\neq i}n_{r}(2-n_{r}-2n)+n_{i}(1-2n_{i})+\varepsilon_{p}(2\ell+1)\geq 0}$, which, as $n\geq 1$ and $m\geq 2$, does not hold. We conclude that $\dim(V_{s}(-1))<2\ell^{2}+\ell-\varepsilon_{p}(2\ell+1)$ for all $s\in T\setminus \ZG(G)$ with $m\geq 2$. Having considered all cases, we deduce that $\displaystyle \max_{s\in T\setminus\ZG(G)}\dim(V_{s}(\mu))=2\ell^{2}+\ell-\varepsilon_{p}(2\ell+1)$.

Lastly, for the unipotent elements, by Lemma \ref{uniprootelems}, we have $\scale[0.9]{\displaystyle \max_{u\in G_{u}\setminus \{1\}}\displaystyle \dim(V_{u}(1))=\max_{i=1,\ell}\dim(V_{x_{\alpha_{i}}(1)}(1))}$. For $x_{\alpha_{1}}(1)$, we write $W=W_{1}\oplus W_{2}$, where $\dim(W_{1})=3$ and $x_{\alpha_{1}}(1)$ acts as $J_{3}$ on $W_{1}$ and $\dim(W_{2})=2\ell-2$ and $x_{\alpha_{1}}(1)$ acts trivially on $W_{2}$. Using \cite[Lemma $3.4$]{liebeck_2012unipotent}, we determine that $\dim((\SW(W))_{x_{\alpha_{1}}(1)}(1))=2\ell^{2}-\ell+1$. It follows that $\dim(V_{x_{\alpha_{1}}(1)}(1))=2\ell^{2}-\ell-\varepsilon_{p}(2\ell+1)$, by the structure of $\SWT(W)$ as a $kG$-module and \cite[Corollary $6.3$]{Korhonen_2019}. Similarly, for $x_{\alpha_{\ell}}(1)$, we write $W=W'_{1}\oplus W'_{2}$, where $\dim(W'_{1})=4$ and $x_{\alpha_{\ell}}(1)$ acts as $J_{2}^{2}$ on $W'_{1}$ and $\dim(W'_{2})=2\ell-3$ and $x_{\alpha_{\ell}}(1)$ acts trivially on $W'_{2}$. Once more, by \cite[Lemma $3.4$]{liebeck_2012unipotent}, we have $\dim((\SW(W))_{x_{\alpha_{\ell}}(1)}(1))=2\ell^{2}-\ell+1$, and so $\dim(V_{x_{\alpha_{\ell}}(1)}(1))=2\ell^{2}-\ell-\varepsilon_{p}(2\ell+1)$, by the structure of $\SWT(W)$ as a $kG$-module and \cite[Corollary $6.3$]{Korhonen_2019}. Therefore $\displaystyle \max_{u\in G_{u}\setminus \{1\}}\dim(V_{u}(1))=2\ell^{2}-\ell-\varepsilon_{p}(2\ell+1)$. To conclude, by Lemma \ref{LtildeGtildelambdaandLGlambda}, we have shown that $\scale[0.9]{\displaystyle \max_{\tilde{u}\in \tilde{G}_{u}\setminus \{1\}}\dim(\tilde{V}_{\tilde{u}}(1))=2\ell^{2}-\ell-\varepsilon_{p}(2\ell+1)}$ and $\scale[0.9]{\displaystyle \max_{\tilde{s}\in \tilde{T}\setminus\ZG(\tilde{G})}\dim(\tilde{V}_{\tilde{s}}(\tilde{\mu}))=2\ell^{2}+\ell-\varepsilon_{p}(2\ell+1)}$, thus $\nu_{\tilde{G}}(\tilde{V})=2\ell$.
\end{proof}

In order to calculate $\nu_{\tilde{G}}(\tilde{V})$ for $\tilde{V}=L_{\tilde{G}}(\tilde{\omega}_{3})$, we need the following three preliminary results: Propositions \ref{PropositionBellomell} and \ref{PropositionB32om3}, and Lemma \ref{PropositionB4om3}. 

\begin{prop}\label{PropositionBellomell}
Let $\tilde{V}=L_{\tilde{G}}(\tilde{\omega}_{\ell})$. Then $\nu_{\tilde{G}}(\tilde{V})=2^{\ell-2}$. Moreover, we have $\displaystyle \max_{\tilde{u}\in \tilde{G}_{u}\setminus \{1\}}\dim(\tilde{V}_{\tilde{u}}(1))=3\cdot 2^{\ell-2}$ and $\displaystyle \max_{\tilde{s}\in \tilde{T}\setminus\ZG(\tilde{G})}\dim(\tilde{V}_{\tilde{s}}(\tilde{\mu}))=2^{\ell-1}$.
\end{prop}

\begin{proof}
Set $\tilde{\lambda}=\tilde{\omega}_{\ell}$ and $\tilde{L}=\tilde{L}_{1}$. By Lemma \ref{weightlevelBl}, we have $e_{1}(\tilde{\lambda})=1$, therefore $\displaystyle \tilde{V}\mid_{[\tilde{L},\tilde{L}]}=\tilde{V}^{0}\oplus \tilde{V}^{1}$. By \cite[Proposition]{Smith_82} and Lemma \ref{dualitylemma}, we have $\tilde{V}^{0}\cong  L_{\tilde{L}}(\tilde{\omega}_{\ell})$ and $\tilde{V}^{1}\cong  L_{\tilde{L}}(\tilde{\omega}_{\ell})$, therefore
\begin{equation}\label{DecompVBloml}
\tilde{V}\mid_{[\tilde{L},\tilde{L}]}\cong L_{\tilde{L}}(\tilde{\omega}_{\ell})\oplus L_{\tilde{L}}(\tilde{\omega}_{\ell}).
\end{equation}

We start with the semisimple elements. Let $\tilde{s}\in \tilde{T}\setminus\ZG(\tilde{G})$. If $\dim(\tilde{V}^{i}_{\tilde{s}}(\tilde{\mu}))=\dim(\tilde{V}^{i})$ for some eigenvalue $\tilde{\mu}$ of $\tilde{s}$ on $\tilde{V}$, where $i=0,1$, then $\tilde{s}\in \ZG(\tilde{L})^{\circ}\setminus \ZG(\tilde{G})$. In this case, as $\tilde{s}$ acts on each $\tilde{V}^{i}$ as scalar multiplication by $c^{1-2i}$ and $c^{2}\neq 1$, we determine that $\dim(\tilde{V}_{\tilde{s}}(\tilde{\mu}))=2^{\ell-1}$ and $\tilde{\mu}=c^{\pm 1}$. We thus assume that $\dim(\tilde{V}^{i}_{\tilde{s}}(\tilde{\mu}))<\dim(\tilde{V}^{i})$ for all eigenvalues $\tilde{\mu}$ of $\tilde{s}$ on $\tilde{V}$ and for both $i=0,1$. We write $\tilde{s}=\tilde{z}\cdot \tilde{h}$, where $\tilde{z}\in \ZG(\tilde{L})^{\circ}$ and $\tilde{h}\in [\tilde{L},\tilde{L}]$, and, using \eqref{DecompVBloml}, we have $\dim(\tilde{V}_{\tilde{s}}(\tilde{\mu}))\leq 2\dim(L_{\tilde{L}}(\tilde{\omega}_{\ell})_{\tilde{h}}(\tilde{\mu_{h}}))$. Recursively and using Proposition \ref{PropositionClnatural} for the base case of $\ell=3$, we deduce that $\dim(\tilde{V}_{\tilde{s}}(\tilde{\mu}))\leq 2^{\ell-1}$ for all eigenvalues $\tilde{\mu}$ of $\tilde{s}$ on $\tilde{V}$. Thus,$\scale[0.9]{\displaystyle \max_{\tilde{s}\in \tilde{T}\setminus\ZG(\tilde{G})}\dim(\tilde{V}_{\tilde{s}}(\tilde{\mu}))=2^{\ell-1}}$.

Now, for the unipotent elements, by Lemma \ref{uniprootelems}, we have $\scale[0.9]{\displaystyle\max_{\tilde{u}\in \tilde{G}_{u}\setminus \{1\}}\dim(\tilde{V}_{\tilde{u}}(1))=\max_{i=\ell-1,\ell}\dim(\tilde{V}_{x_{\tilde{\alpha}_{i}}(1)}(1))}$. Using \eqref{DecompVBloml}, we get $\dim(\tilde{V}_{x_{\tilde{\alpha}_{i}}(1)}(1))=2\dim(L_{\tilde{L}}(\tilde{\omega}_{\ell})_{x_{\tilde{\alpha}_{i}}(1)}(1))$, $i=\ell-1,\ell$. Recursively and by Proposition \ref{PropositionClnatural} for the base case of $\ell=3$, we determine that $\dim(\tilde{V}_{x_{\tilde{\alpha}_{i}}(1)}(1))\leq 3\cdot 2^{\ell-2}$, where equality holds for $x_{\tilde{\alpha}_{\ell-1}}(1)$. Therefore $\scale[0.9]{\displaystyle \max_{\tilde{u}\in \tilde{G}_{u}\setminus \{1\}}\dim(\tilde{V}_{\tilde{u}}(1))}=3\cdot 2^{\ell-2}$ and, as $\scale[0.9]{\displaystyle \max_{\tilde{s}\in \tilde{T}\setminus\ZG(\tilde{G})}\dim(\tilde{V}_{\tilde{s}}(\tilde{\mu}))}=2^{\ell-1}$, it follows that $\nu_{\tilde{G}}(\tilde{V})=2^{\ell-2}$.
\end{proof}

\begin{prop}\label{PropositionB32om3}
Let $\ell=3$ and $\tilde{V}=L_{\tilde{G}}(2\tilde{\omega}_{3})$. Then $\nu_{\tilde{G}}(\tilde{V})=14$. Moreover, we have $\scale[0.9]{\displaystyle \max_{\tilde{u}\in \tilde{G}_{u}\setminus \{1\}}\dim(\tilde{V}_{\tilde{u}}(1))}=21$ and $\scale[0.9]{\displaystyle \max_{\tilde{s}\in \tilde{T}\setminus\ZG(\tilde{G})}\dim(\tilde{V}_{\tilde{s}}(\tilde{\mu}))}=20$.
\end{prop}

\begin{proof}
Let $\tilde{\lambda}=2\tilde{\omega}_{3}$ and $\tilde{L}=\tilde{L}_{1}$. By Lemma \ref{weightlevelBl}, we have $e_{1}(\tilde{\lambda})=2$, therefore $\displaystyle \tilde{V}\mid_{[\tilde{L},\tilde{L}]}=\tilde{V}^{0}\oplus \tilde{V}^{1}\oplus \tilde{V}^{2}$. By \cite[Proposition]{Smith_82} and Lemma \ref{dualitylemma}, we have $\tilde{V}^{0}\cong L_{\tilde{L}}(2\tilde{\omega}_{3})$ and $\tilde{V}^{2}\cong L_{\tilde{L}}(2\tilde{\omega}_{3})$. Now, in $\tilde{V}^{1}$, the weight $\displaystyle (\tilde{\lambda}-\tilde{\alpha}_{1}-\tilde{\alpha}_{2}-\tilde{\alpha}_{3})\mid_{\tilde{T}_{1}}=2\tilde{\omega}_{3}$ admits a maximal vector, thus $\tilde{V}^{1}$ has a composition factor isomorphic to $L_{\tilde{L}}(2\tilde{\omega}_{3})$. Moreover, the weight $\displaystyle(\tilde{\lambda}-\tilde{\alpha}_{1}-\tilde{\alpha}_{2}-2\tilde{\alpha}_{3})\mid_{\tilde{T}_{1}}=\tilde{\omega}_{2}$ occurs with multiplicity $2$ and is a sub-dominant weight in the composition factor of $\tilde{V}^{1}$ isomorphic to $L_{\tilde{L}}(2\tilde{\omega}_{3})$, in which it has multiplicity $1$. By dimensional considerations and \cite[II.2.14]{Jantzen_2007representations}, we determine that
\begin{equation}\label{DecompVB32om3}
\tilde{V}\mid_{[\tilde{L},\tilde{L}]} \cong L_{\tilde{L}}(2\tilde{\omega}_{3}) \oplus L_{\tilde{L}}(2\tilde{\omega}_{3})\oplus L_{\tilde{L}}(\tilde{\omega}_{2}) \oplus L_{\tilde{L}}(2\tilde{\omega}_{3}).
\end{equation}

We start with the semisimple elements. Let $\tilde{s}\in \tilde{T}\setminus \ZG(\tilde{G})$. If $\dim(\tilde{V}^{i}_{\tilde{s}}(\tilde{\mu}))=\dim(\tilde{V}^{i})$ for some eigenvalue $\tilde{\mu}$ of $\tilde{s}$ on $\tilde{V}$, where $0\leq i\leq 2$, then $\tilde{s}\in \ZG(\tilde{L})^{\circ}\setminus \ZG(\tilde{G})$. In this case, as $\tilde{s}$ acts on each $\tilde{V}^{i}$ as scalar multiplication by $c^{2-2i}$ and $c^{2}\neq 1$, we determine that $\dim(\tilde{V}_{\tilde{s}}(\tilde{\mu}))\leq 20$, where equality holds for $c^{2}=-1$ and $\tilde{\mu}=-1$. We thus assume that $\dim(\tilde{V}^{i}_{\tilde{s}}(\tilde{\mu}))<\dim(\tilde{V}^{i})$ for all eigenvalues $\tilde{\mu}$ of $\tilde{s}$ on $\tilde{V}$ and all $0\leq i\leq 2$. We write $\tilde{s}=\tilde{z}\cdot \tilde{h}$, where $\tilde{z}\in \ZG(\tilde{L})^{\circ}$ and $\tilde{h}\in [\tilde{L},\tilde{L}]$. We have $\displaystyle \dim(\tilde{V}_{\tilde{s}}(\tilde{\mu}))\leq \sum_{i=0}^{2}\dim(\tilde{V}^{i}_{\tilde{h}}(\tilde{\mu}^{i}_{\tilde{h}}))$, where $\dim(\tilde{V}^{i}_{\tilde{h}}(\tilde{\mu}^{i}_{\tilde{h}}))<\dim(\tilde{V}^{i})$ for all eigenvalues $\tilde{\mu}^{i}_{\tilde{h}}$ of $\tilde{h}$ on $\tilde{V}^{i}$. We show that $\dim(\tilde{V}^{1}_{\tilde{h}}(\tilde{\mu}^{1}_{\tilde{h}}))\leq 8$ for all eigenvalues $\tilde{\mu}^{1}_{\tilde{h}}$ of $\tilde{h}$ on $\tilde{V}^{1}$. For this, we use \eqref{enum Cl_P2} and \eqref{enum Cl_P1} to determine that the eigenvalues of $\tilde{h}$ on $\tilde{V}^{1}$, not necessarily distinct, are $d^{\pm 2}$ and $e^{\pm 2}$ each with multiplicity at least $1$; $d^{\pm 1}e^{\pm 1}$ each with multiplicity at least $1$; and $1$ with multiplicity at least $3$, where $d,e\in k^{*}$ not both simultaneously equal to $1$. We see that $\dim(\tilde{V}^{1}_{\tilde{h}}(\tilde{\mu}^{1}_{\tilde{h}}))\leq 8$ for all eigenvalues $\tilde{\mu}^{1}_{\tilde{h}}$ of $\tilde{h}$ on $\tilde{V}^{1}$, and, by Proposition \ref{PropositionBlsymm}, we determine that $\dim(\tilde{V}_{\tilde{s}}(\tilde{\mu}))\leq 20$ for all eigenvalues $\tilde{\mu}$ of $\tilde{s}$ on $\tilde{V}$. Therefore, $\scale[0.9]{\displaystyle \max_{\tilde{s}\in \tilde{T}\setminus\ZG(\tilde{G})}\dim(\tilde{V}_{\tilde{s}}(\tilde{\mu}))}=20$.

We now focus on the unipotent elements. By Lemma \ref{uniprootelems}, we have $\scale[0.9]{\displaystyle \max_{\tilde{u}\in \tilde{G}_{u}\setminus \{1\}}\dim(\tilde{V}_{\tilde{u}}(1))=\max_{i=2,3}\dim(\tilde{V}_{x_{\tilde{\alpha}_{i}}(1)}(1))}$, and, by \eqref{DecompVB32om3} and Propositions \ref{PropositionClwedge} and \ref{PropositionClsymm}, it follows that $\dim(\tilde{V}_{x_{\tilde{\alpha}_{i}}(1)}(1))=3\dim((L_{\tilde{L}}(2\tilde{\omega}_{3}))_{x_{\tilde{\alpha}_{i}}(1)}(1))+$ $\dim((L_{\tilde{L}}(\tilde{\omega}_{2}))_{x_{\tilde{\alpha}_{i}}(1)}(1))\leq 21$, where equality holds for $x_{\tilde{\alpha}_{2}}(1)$. Therefore $\scale[0.9]{\displaystyle \max_{\tilde{u}\in \tilde{G}_{u}\setminus \{1\}}\dim(\tilde{V}_{\tilde{u}}(1))}=21$ and, as $\scale[0.9]{\displaystyle \max_{\tilde{s}\in \tilde{T}\setminus\ZG(\tilde{G})}\dim(\tilde{V}_{\tilde{s}}(\tilde{\mu}))}=20$, it follows that $\nu_{\tilde{G}}(\tilde{V})=14$.
\end{proof}

\begin{lem}\label{PropositionB4om3}
Let $\ell=4$ and  $\tilde{V}=L_{\tilde{G}}(\tilde{\omega}_{3})$. Then $\scale[0.9]{\displaystyle \max_{\tilde{s}\in \tilde{T}\setminus\ZG(\tilde{G})}\dim(\tilde{V}_{\tilde{s}}(\tilde{\mu}))}=56$.
\end{lem}

\begin{proof}
Set $\tilde{\lambda}=\tilde{\omega}_{3}$ and $\tilde{L}=\tilde{L}_{1}$. By Lemma \ref{weightlevelBl}, we have $e_{1}(\lambda)=2$, therefore $\displaystyle \tilde{V}\mid_{[\tilde{L},\tilde{L}]}=\tilde{V}^{0}\oplus \tilde{V}^{1}\oplus \tilde{V}^{2}$. By \cite[Proposition]{Smith_82} and Lemma \ref{dualitylemma}, we have $\tilde{V}^{0}\cong L_{\tilde{L}}(\tilde{\omega}_{3})$ and $\tilde{V}^{2}\cong L_{\tilde{L}}(\tilde{\omega}_{3})$. Now, the weight $\displaystyle (\lambda-\tilde{\alpha}_{1}-\tilde{\alpha}_{2}-\tilde{\alpha}_{3})\mid_{\tilde{T}_{1}}=2\tilde{\omega}_{4}$ admits a maximal vector in $\tilde{V}^{1}$, thus $\tilde{V}^{1}$ has a composition factor isomorphic to $L_{\tilde{L}}(2\tilde{\omega}_{4})$. Further, the weight $\displaystyle (\lambda-\tilde{\alpha}_{1}-\tilde{\alpha}_{2}-2\tilde{\alpha}_{3}-2\tilde{\alpha}_{4})\mid_{\tilde{T}_{1}}=\tilde{\omega}_{2}$ occurs with multiplicity $3$ and is a sub-dominant weight in the composition factor of $\tilde{V}^{1}$ isomorphic to $L_{\tilde{L}}(2\tilde{\omega}_{4})$, in which it has multiplicity $2$. By dimensional considerations and \cite[II.$2.14$]{Jantzen_2007representations}, we determine that
\begin{equation}\label{DecompVB4om3}
\tilde{V}\mid_{[\tilde{L},\tilde{L}]}\cong L_{\tilde{L}}(\tilde{\omega}_{3})\oplus L_{\tilde{L}}(2\tilde{\omega}_{4})\oplus L_{\tilde{L}}(\tilde{\omega}_{2}) \oplus L_{\tilde{L}}(\tilde{\omega}_{3}).
\end{equation}

Let $\tilde{s}\in \tilde{T}\setminus \ZG(\tilde{G})$. If $\dim(\tilde{V}^{i}_{\tilde{s}}(\tilde{\mu}))=\dim(\tilde{V}^{i})$ for some eigenvalue $\tilde{\mu}$ of $\tilde{s}$ on $\tilde{V}^{i}$, where $0\leq i\leq 2$, then $\tilde{s}\in \ZG(\tilde{L})^{\circ}\setminus \ZG(\tilde{G})$. In this case, as $\tilde{s}$ acts as scalar multiplication by $c^{2-2i}$ on $\tilde{V}^{i}$ and $c^{2}\neq 1$, we determine that $\dim(\tilde{V}_{\tilde{s}}(\tilde{\mu}))\leq 42$. We thus assume that $\dim(\tilde{V}^{i}_{\tilde{s}}(\tilde{\mu}))<\dim(\tilde{V}^{i})$ for all eigenvalues $\tilde{\mu}$ of $\tilde{s}$ on $\tilde{V}^{i}$ and all $0\leq i\leq 2$. We write $\tilde{s}=\tilde{z}\cdot \tilde{h}$, where $\tilde{z}\in \ZG(\tilde{L})^{\circ}$ and $\tilde{h}\in [\tilde{L},\tilde{L}]$, and, by \eqref{DecompVB4om3} and Propositions \ref{PropositionBlnatural}, \ref{PropositionBlwedge} and \ref{PropositionB32om3}, we get $\dim(\tilde{V}_{\tilde{s}}(\tilde{\mu}))\leq 2\dim((L_{\tilde{L}}(\tilde{\omega}_{3}))_{\tilde{h}}(\tilde{\mu}_{\tilde{h}}))+\dim((L_{\tilde{L}}(2\tilde{\omega}_{4}))_{\tilde{h}}(\tilde{\mu}_{\tilde{h}}))+\dim((L_{\tilde{L}}(\tilde{\omega}_{2}))_{\tilde{h}}(\tilde{\mu}_{\tilde{h}}))\leq 56$ for all eigenvalues $\tilde{\mu}$ of $\tilde{s}$ on $\tilde{V}$. Moreover, using Propositions \ref{PropositionBlnatural} and \ref{PropositionBlwedge} and the weight structure of $L_{\tilde{L}}(2\tilde{\omega}_{4})$, one can show that for $\tilde{s}=h_{\tilde{\alpha}_{1}}(-1)h_{\tilde{\alpha}_{3}}(-1)h_{\tilde{\alpha}_{4}}(d)$ with $d^{2}=1$ we have $\dim(\tilde{V}_{\tilde{s}}(-1))=56$. Therefore, $\scale[0.9]{\displaystyle \max_{\tilde{s}\in \tilde{T}\setminus\ZG(\tilde{G})}\dim(\tilde{V}_{\tilde{s}}(\tilde{\mu}))}=56$.
\end{proof}

\begin{prop}\label{PropositionBlwedgecube}
Let $\ell\geq 4$ and $\tilde{V}=L_{\tilde{G}}(\tilde{\omega}_{3})$. Then $\nu_{\tilde{G}}(\tilde{V})=2\ell^{2}-\ell$. Moreover, we have $\scale[0.9]{\displaystyle \max_{\tilde{u}\in \tilde{G}_{u}\setminus \{1\}}\dim(\tilde{V}_{\tilde{u}}(1))}$ $=\frac{4\ell^{3}-12\ell^{2}+29\ell-27}{3}$ and $\scale[0.9]{\displaystyle \max_{\tilde{s}\in \tilde{T}\setminus\ZG(\tilde{G})}\dim(\tilde{V}_{\tilde{s}}(\tilde{\mu}))}$ $=\frac{4\ell^{3}-6\ell^{2}+2\ell}{3}$.
\end{prop}

\begin{proof}
To begin, we note that $\tilde{V}\cong L_{G}(\omega_{3})$ as $k\tilde{G}$-modules. To ease notation, let $V=L_{G}(\omega_{3})$. Further, by \cite[Proposition $4.2.2$]{mcninch_1998}, we have $\wedge^{3}(W)\cong V$. 

We begin with the unipotent elements. By Lemma \ref{uniprootelems}, we have $\scale[0.9]{\displaystyle \max_{u\in G_{u}\setminus \{1\}}\dim(V_{u}(1))=\max_{i=1,\ell}\dim(V_{x_{\alpha_{i}}(1)}(1))}$. For $x_{\alpha_{1}}(1)$, we write $W=W_{1}\oplus W_{2}$, where $\dim(W_{1})=4$ and $x_{\alpha_{1}}(1)$ acts as $J_{2}^{2}$ on $W_{1}$ and $\dim(W_{2})=2\ell-3$ and $x_{\alpha_{1}}(1)$ acts trivially on $W_{2}$. As  $\wedge^{3}(W)=\wedge^{3}(W_{1})\oplus [\wedge^{2}(W_{1})\otimes W_{2}]\oplus [W_{1}\otimes \wedge^{2}(W_{2})]\oplus \wedge^{3}(W_{2})$, we determine that $\dim((\wedge^{3}(W))_{x_{\alpha_{1}}(1)}(1))=\frac{4\ell^{3}-12\ell^{2}+29\ell-27}{3}$. For $x_{\alpha_{\ell}}(1)$, we write $W=W'_{1}\oplus W'_{2}$, where $\dim(W'_{1})=3$ and $x_{\alpha_{\ell}}(1)$ acts as $J_{3}$ on $W'_{1}$ and $\dim(W'_{2})=2\ell-2$ and $x_{\alpha_{\ell}}(1)$ acts trivially on $W'_{2}$. Then $\dim((\wedge^{3}(W))_{x_{\alpha_{\ell}}(1)}(1))=\frac{4\ell^{3}-12\ell^{2}+17\ell-6}{3}$, and so $\displaystyle \max_{u\in G_{u}\setminus \{1\}}\dim(V_{u}(1))$ $=\frac{4\ell^{3}-12\ell^{2}+29\ell-27}{3}$. 

We now focus on the semisimple elements. As the case $\ell=4$ has been treated in Lemma \ref{PropositionB4om3}, we assume $\ell\geq 5$. Let $\lambda=\omega_{3}$ and $L=L_{1}$. By Lemma \ref{weightlevelBl}, we have $e_{1}(\lambda)=2$, therefore $\displaystyle V\mid_{[L,L]}=V^{0}\oplus V^{1}\oplus V^{2}$. By \cite[Proposition]{Smith_82} and Lemma \ref{dualitylemma}, we have $V^{0}\cong L_{L}(\omega_{3})$ and $V^{2}\cong L_{L}(\omega_{3})$. Now, the weight $\displaystyle (\lambda-\alpha_{1}-\alpha_{2}-\alpha_{3})\mid_{T_{1}}=\omega_{4}$ admits a maximal vector in $V^{1}$, thus $V^{1}$ has a composition factor isomorphic to $L_{L}(\omega_{4})$. Further, the weight $\displaystyle (\lambda-\alpha_{1}-\alpha_{2}-2\alpha_{3}-\cdots -2\alpha_{\ell})\mid_{T_{1}}=\omega_{2}$ occurs with multiplicity $\ell-1$ and is a sub-dominant weight in the composition factor of $V^{1}$ isomorphic to $L_{L}(\omega_{4})$, in which it has multiplicity $\ell-2$. By dimensional considerations and\cite[II.$2.14$]{Jantzen_2007representations}, we determine that
\begin{equation}\label{DecompVBlom3}
V\mid_{[L,L]}\cong L_{L}(\omega_{3})\oplus L_{L}(\omega_{4})\oplus L_{L}(\omega_{2}) \oplus L_{L}(\omega_{3}).
\end{equation}

Let $s\in T\setminus \ZG(G)$. If $\dim(V^{i}_{s}(\mu))=\dim(V^{i})$ for some eigenvalue $\mu$ of $s$ on $V^{i}$, where $0\leq i\leq 2$, then $s\in \ZG(L)^{\circ}\setminus \ZG(G)$. In this case, as $s$ acts as scalar multiplication by $c^{2-2i}$ on $V^{i}$ and $c^{2}\neq 1$, we determine that $\dim(V_{s}(\mu))\leq \binom{2\ell-1}{3}+2\ell-1$. We thus assume that $\dim(V^{i}_{s}(\mu))<\dim(V^{i})$ for all eigenvalues $\mu$ of $s$ on $V^{i}$ and all $0\leq i\leq 2$. We write $s=z\cdot h$, where $z\in \ZG(L)^{\circ}$ and $h\in [L,L]$, and, by \eqref{DecompVBlom3} and Propositions \ref{PropositionBlnatural} and \ref{PropositionBlwedge}, it follows that $\scale[0.9]{\dim(V_{s}(\mu))\leq 4(\ell-1)^{2}+\dim((L_{L}(\omega_{4}))_{h}(\mu_{h}))}$. Recursively and using Lemma \ref{PropositionB4om3}, we determine that $\scale[0.9]{\displaystyle \dim(V_{s}(\mu))\leq 4\sum_{j=4}^{\ell-1}j^{2}+56}=\frac{4\ell^{3}-6\ell^{2}+2\ell}{3}$. Consequently, $\dim(V_{s}(\mu))\leq \frac{4\ell^{3}-6\ell^{2}+2\ell}{3}$ for all eigenvalues $\mu$ of $s$ on $V$. Furthermore, recursively and using Propositions \ref{PropositionBlnatural} and \ref{PropositionBlwedge}, one can show that for $s=\diag(-1,\cdots,-1,1,-1,\cdots,-1)\in T\setminus \ZG(G)$ we have $\dim(V_{s}(-1))=\frac{4\ell^{3}-6\ell^{2}+2\ell}{3}$, therefore $\displaystyle \max_{s\in T\setminus\ZG(G)}\dim(V_{s}(\mu))$ $=\frac{4\ell^{3}-6\ell^{2}+2\ell}{3}$. To conclude, in view of Lemma \ref{LtildeGtildelambdaandLGlambda}, we have shown that $\scale[0.9]{\displaystyle \max_{\tilde{u}\in \tilde{G}_{u}\setminus \{1\}}\dim(\tilde{V}_{\tilde{u}}(1))}$ $=\frac{4\ell^{3}-12\ell^{2}+29\ell-27}{3}$ and $\scale[0.9]{\displaystyle \max_{\tilde{s}\in \tilde{T}\setminus\ZG(\tilde{G})}\dim(\tilde{V}_{\tilde{s}}(\tilde{\mu}))}$ $=\frac{4\ell^{3}-6\ell^{2}+2\ell}{3}$, thus $\nu_{\tilde{G}}(\tilde{V})=2\ell^{2}-\ell$.
\end{proof}

\begin{prop}\label{PropositionB33om1}
Let $p\neq 3$, $\ell=3$ and $\tilde{V}=L_{\tilde{L}}(3\tilde{\omega}_{1})$. Then $\nu_{\tilde{G}}(\tilde{V})=21$. Moreover, we have $\scale[0.9]{\displaystyle \max_{\tilde{s}\in \tilde{T}\setminus\ZG(\tilde{G})}\dim(\tilde{V}_{\tilde{s}}(\tilde{\mu}))}$ $=56$ and $\scale[0.9]{\displaystyle \max_{\tilde{u}\in \tilde{G}_{u}\setminus \{1\}}\dim(\tilde{V}_{\tilde{u}}(1))}=35$. 
\end{prop}

\begin{proof}
Set $\tilde{\lambda}=3\tilde{\omega}_{1}$ and $\tilde{L}=\tilde{L}_{1}$. By Lemma \ref{weightlevelBl}, we have $e_{1}(\tilde{\lambda})=6$, therefore $\displaystyle \tilde{V}\mid_{[\tilde{L},\tilde{L}]}=\tilde{V}^{0}\oplus \cdots \oplus \tilde{V}^{6}$. By \cite[Proposition]{Smith_82} and Lemma \ref{dualitylemma}, we have $\tilde{V}^{0}\cong L_{\tilde{L}}(0)$ and $\tilde{V}^{6}\cong L_{\tilde{L}}(0)$. Now, in $\tilde{V}^{1}$, the weight $\displaystyle (\tilde{\lambda}-\tilde{\alpha}_{1})\mid_{\tilde{T}_{1}}=\tilde{\omega}_{2}$ admits a maximal vector, thus $\tilde{V}^{1}$ has a composition factor isomorphic to $L_{\tilde{L}}(\tilde{\omega}_{2})$. Similarly, in $\tilde{V}^{2}$ the weight $\displaystyle (\tilde{\lambda}-2\tilde{\alpha}_{1})\mid_{\tilde{T}_{1}}=2\tilde{\omega}_{2}$ admits a maximal vector, thus $\tilde{V}^{2}$ has a composition factor isomorphic to $L_{\tilde{L}}(2\tilde{\omega}_{2})$. Moreover, the weight $\displaystyle(\tilde{\lambda}-2\tilde{\alpha}_{1}-2\tilde{\alpha}_{2}-2\tilde{\alpha}_{3})\mid_{\tilde{T}_{1}}=0$ occurs with multiplicity $3$ and is a sub-dominant weight in the composition factor of $\tilde{V}^{2}$ isomorphic to $L_{\tilde{L}}(2\tilde{\omega}_{2})$, in which it has multiplicity $2-\varepsilon_{p}(5)$. Lastly, in $\tilde{V}^{3}$ the weight $\displaystyle (\tilde{\lambda}-3\tilde{\alpha}_{1})\mid_{\tilde{T}_{1}}=3\tilde{\omega}_{2}$ admits a maximal vector, thus $\tilde{V}^{3}$ has a composition factor isomorphic to $L_{\tilde{L}}(3\tilde{\omega}_{2})$. Further, the weight $\displaystyle(\tilde{\lambda}-3\tilde{\alpha}_{1}-2\tilde{\alpha}_{2}-2\tilde{\alpha}_{3})\mid_{\tilde{T}_{1}}=\tilde{\omega}_{2}$ occurs with multiplicity $3$ and is a sub-dominant weight in the composition factor of $\tilde{V}^{3}$ isomorphic to $L_{\tilde{L}}(3\tilde{\omega}_{2})$, in which it has multiplicity $2-\varepsilon_{p}(7)$. Thus, as $\dim(\tilde{V}^{3})\leq 35$, it follows that $\tilde{V}^{3}$ has exactly $2+\varepsilon_{p}(7)$ composition factors: one isomorphic to $L_{\tilde{L}}(3\tilde{\omega}_{2})$ and $1+\varepsilon_{p}(7)$ to $L_{\tilde{L}}(\tilde{\omega}_{2})$. Further, we have $\tilde{V}^{1}\cong L_{\tilde{L}}(\tilde{\omega}_{2})$, $\tilde{V}^{5}\cong L_{\tilde{L}}(\tilde{\omega}_{2})$, and $\tilde{V}^{2}$ and $\tilde{V}^{4}$ each has exactly $2+\varepsilon_{p}(5)$ composition factors: one isomorphic to $L_{\tilde{L}}(2\tilde{\omega}_{2})$ and $1+\varepsilon_{p}(5)$ to $L_{\tilde{L}}(0)$.

We start with the semisimple elements. Let $\tilde{s}\in \tilde{T}\setminus\ZG(\tilde{G})$. If $\dim(\tilde{V}^{i}_{\tilde{s}}(\tilde{\mu}))=\dim(\tilde{V}^{i})$ for some eigenvalue $\tilde{\mu}$ of $\tilde{s}$ on $\tilde{V}$, where $0\leq i\leq 6$, then $\tilde{s}\in \ZG(\tilde{L})^{\circ}\setminus \ZG(\tilde{G})$. In this case, as $\tilde{s}$ acts on each $\tilde{V}^{i}$ as scalar multiplication by $c^{6-2i}$ and $c^{2}\neq 1$, we determine that $\dim(\tilde{V}_{\tilde{s}}(\tilde{\mu}))\leq 45$. We thus assume that $\dim(\tilde{V}^{i}_{\tilde{s}}(\tilde{\mu}))<\dim(\tilde{V}^{i})$ for all eigenvalues $\tilde{\mu}$ of $\tilde{s}$ on $\tilde{V}$ and for all $0\leq i\leq 6$. We write $\tilde{s}=\tilde{z}\cdot \tilde{h}$, where $\tilde{z}\in \ZG(\tilde{L})^{\circ}$ and $\tilde{h}\in [\tilde{L},\tilde{L}]$, and using the structure of $\tilde{V}\mid_{[\tilde{L},\tilde{L}]}$ and Propositions \ref{PropositionClwedge}, \ref{C22om2} and \ref{C23om2p=7}, we determine that $\dim(\tilde{V}_{\tilde{s}}(\tilde{\mu}))\leq (4+2\varepsilon_{p}(5))\dim((L_{\tilde{L}}(0))_{\tilde{h}}(\tilde{\mu}_{\tilde{h}}))+(3+\varepsilon_{p}(7))\dim((L_{\tilde{L}}(\tilde{\omega}_{2}))_{\tilde{h}}(\tilde{\mu}_{\tilde{h}}))+2\dim((L_{\tilde{L}}(2\tilde{\omega}_{2}))_{\tilde{h}}(\tilde{\mu}_{\tilde{h}}))+\dim((L_{\tilde{L}}(3\tilde{\omega}_{2}))_{\tilde{h}}(\tilde{\mu}_{\tilde{h}}))\leq 56$ for all eigenvalues $\tilde{\mu}$ of $\tilde{s}$ on $\tilde{V}$. Further, using Propositions \ref{PropositionClwedge}, \ref{C22om2} and \ref{C23om2p=7}, one can show that for $\tilde{s}=h_{\tilde{\alpha}_{1}}(-1)h_{\tilde{\alpha}_{3}}(c)$ with $c^{2}=-1$ we have $\dim(\tilde{V}_{\tilde{s}}(-1))=56$. Thus, we have shown that $\scale[0.9]{\displaystyle \max_{\tilde{s}\in \tilde{T}\setminus\ZG(\tilde{G})}\dim(\tilde{V}_{\tilde{s}}(\tilde{\tilde{\mu}}))}=56$. 

We focus on the unipotent elements. To begin, we note that $\tilde{V}\cong L_{G}(3\omega_{1})$ as $k\tilde{G}$-modules. To ease notation, let $V=L_{G}(3\omega_{1})$. Further, by \cite[Propositions $4.7.4$]{mcninch_1998}, we have $\SWT(W)\cong V\oplus W$, therefore $\dim(V_{u}(1))=\dim((\SWT(W))_{u}(1))-\dim(W_{u}(1))$ for all unipotent elements $u\in G_{u}$. Now, by Lemma \ref{uniprootelems}, we have $\displaystyle \max_{u\in G_{u}\setminus \{1\}}\dim(V_{u}(1))=\max_{i=2,3}\dim(V_{x_{\alpha_{i}}}(1))$. For $x_{\alpha_{2}}(1)$, we write $W=W_{1}\oplus W_{2}$, where $\dim(W_{1})=4$ and $x_{\alpha_{2}}(1)$ acts as $J_{2}^{2}$ on $W_{1}$, and $\dim(W_{2})=3$ and $x_{\alpha_{2}}(1)$ acts trivially on $W_{2}$. Then, as $\SWT(W)\cong \SWT(W_{1})\oplus [\SW(W_{1})\otimes W_{2}]\oplus [W_{1}\otimes \SW(W_{2})]\oplus \SWT(W_{2})$, we get $\dim((\SW(W))_{x_{\alpha_{2}}(1)}(1))=40$ and so $\dim(V_{x_{\alpha_{2}}(1)}(1))=35$. For $x_{\alpha_{3}}(1)$, we write $W=W'_{1}\oplus W'_{2}$, where $\dim(W'_{1})=3$ and $x_{\alpha_{3}}(1)$ acts as $J_{3}$ on $W'_{1}$, and $\dim(W'_{2})=4$ and $x_{\alpha_{3}}(1)$ acts trivially on $W'_{2}$. Then $\dim((\SWT(W))_{x_{\alpha_{3}}(1)}(1))=40$ and so $\dim(V_{x_{\alpha_{3}}(1)}(1))=35$. It follows that $\displaystyle \max_{u\in G_{u}\setminus \{1\}}\dim(V_{u}(1))=35$. To conclude, in view of Lemma \ref{LtildeGtildelambdaandLGlambda}, we have shown that $\scale[0.9]{\displaystyle \max_{\tilde{u}\in \tilde{G}_{u}\setminus \{1\}}\dim(\tilde{V}_{\tilde{u}}(1))}=35$ and, as $\scale[0.9]{\displaystyle \max_{\tilde{s}\in \tilde{T}\setminus\ZG(\tilde{G})}\dim(\tilde{V}_{\tilde{s}}(\tilde{\mu}))}$ $=56$, we get $\nu_{\tilde{G}}(\tilde{V})=21$. 
\end{proof}

\begin{prop}\label{PropositionBlsymmcube}
Let $p\neq 3$, $\ell\geq 4$ and $\tilde{V}=L_{\tilde{G}}(3\tilde{\omega}_{1})$. Then $\nu_{\tilde{G}}(\tilde{V})=2\ell^{2}+\ell-\varepsilon_{p}(2\ell+3)$. Moreover, we have $\scale[0.9]{\displaystyle \max_{\tilde{u}\in \tilde{G}_{u}\setminus \{1\}}\dim(\tilde{V}_{\tilde{u}}(1))}$ $=\binom{2\ell+1}{3}-(2\ell-1)\varepsilon_{p}(2\ell+3)$ and $\scale[0.9]{\displaystyle \max_{\tilde{s}\in \tilde{T}\setminus\ZG(\tilde{G})}\dim(\tilde{V}_{\tilde{s}}(\tilde{\mu}))}$ $=\binom{2\ell+1}{3}+2{\ell}^{2}+\ell-2\ell\varepsilon_{p}(2\ell+3)$.
\end{prop}

\begin{proof}
To begin, we note that $\tilde{V}\cong L_{G}(3\omega_{1})$ as $k\tilde{G}$-modules. To ease notation, let $V=L_{G}(3\omega_{1})$, $\lambda=3\omega_{1}$ and $L=L_{1}$. By Lemma \ref{weightlevelBl}, we have $e_{1}(\lambda)=6$, therefore $\displaystyle V\mid_{[L,L]}=V^{0}\oplus \cdots \oplus V^{6}$. By \cite[Proposition]{Smith_82} and Lemma \ref{dualitylemma}, we have $V^{0}\cong L_{L}(0)$ and $V^{6}\cong L_{L}(0)$. Now, in $V^{1}$, the weight $\displaystyle (\lambda-\alpha_{1})\mid_{T_{1}}=\omega_{2}$ admits a maximal vector, thus $V^{1}$ has a composition factor isomorphic to $L_{L}(\omega_{2})$. Similarly, in $V^{2}$ the weight $\displaystyle (\lambda-2\alpha_{1})\mid_{T_{1}}=2\omega_{2}$ admits a maximal vector, thus $V^{2}$ has a composition factor isomorphic to $L_{L}(2\omega_{2})$. Moreover, the weight $\displaystyle(\lambda-2\alpha_{1}-\cdots-2\alpha_{\ell})\mid_{T_{1}}=0$ occurs with multiplicity $\ell-\varepsilon_{p}(2\ell+3)$ and is a sub-dominant weight in the composition factor of $V^{2}$ isomorphic to $L_{L}(2\omega_{2})$, in which it has multiplicity $\ell-1-\varepsilon_{p}(2\ell-1)$. Lastly, in $V^{3}$ the weight $\displaystyle (\lambda-3\alpha_{1})\mid_{T_{1}}=3\omega_{2}$ admits a maximal vector, therefore $V^{3}$ has a composition factor isomorphic to $L_{L}(3\omega_{2})$. Further, the weight $\displaystyle(\lambda-3\alpha_{1}-2\alpha_{2}\cdots-2\alpha_{\ell})\mid_{T_{1}}=\omega_{2}$ occurs with multiplicity $\ell-\varepsilon_{p}(2\ell+3)$ and is a sub-dominant weight in the composition factor of $V^{3}$ isomorphic to $L_{L}(3\omega_{2})$, in which it has multiplicity $\ell-1-\varepsilon_{p}(2\ell+1)$. Thus, as $\dim(V^{3})\leq \binom{2\ell+1}{3}-(2\ell-1)\varepsilon_{p}(2\ell+3)$, it follows that $V^{3}$ has exactly $2+\varepsilon_{p}(2\ell+1)-\varepsilon_{p}(2\ell+3)$ composition factors: one isomorphic to $L_{L}(3\omega_{2})$ and $1+\varepsilon_{p}(2\ell+1)-\varepsilon_{p}(2\ell+3)$ to $L_{L}(\omega_{2})$. Further, we have $V^{1}\cong L_{L}(\omega_{2})$, $V^{5}\cong L_{L}(\omega_{2})$, and $V^{2}$ and $V^{4}$ each has $2+\varepsilon_{p}(2\ell-1)-\varepsilon_{p}(2\ell+3)$ composition factors: one isomorphic to $L_{L}(2\omega_{2})$ and $1+\varepsilon_{p}(2\ell-1)-\varepsilon_{p}(2\ell+3)$ to $L_{L}(0)$.

We start with the semisimple elements. Let $s\in T\setminus\ZG(G)$. If $\dim(V^{i}_{s}(\mu))=\dim(V^{i})$ for some eigenvalue $\mu$ of $s$ on $V$, where $0\leq i\leq 6$, then $s\in \ZG(L)^{\circ}\setminus \ZG(G)$. In this case, as $s$ acts on each $V^{i}$ as scalar multiplication by $c^{6-2i}$ and $c^{2}\neq 1$, we determine that $\dim(V_{s}(\mu))\leq \binom{2\ell+1}{3}+4\ell-2-(2\ell-1)\varepsilon_{p}(2\ell+3)$ for all eigenvalues $\mu$ of $s$ on $V$. We thus assume that $\dim(V^{i}_{s}(\mu))<\dim(V^{i})$ for all eigenvalues $\mu$ of $s$ on $V$ and for all $0\leq i\leq 6$. We write $s=z\cdot h$, where $z\in \ZG(L)^{\circ}$ and $h\in [L,L]$, and, using the structure of $V\mid_{[L,L]}$ and Propositions \ref{PropositionBlnatural} and \ref{PropositionBlsymm}, we determine that $\dim(V_{s}(\mu))\leq (4+2\varepsilon_{p}(2\ell-1)-2\varepsilon_{p}(2\ell+3))\dim((L_{L}(0))_{h}(\mu_{h}))+(3+\varepsilon_{p}(2\ell+1)-\varepsilon_{p}(2\ell+3))\dim((L_{L}(\omega_{2}))_{h}(\mu_{h}))+2\dim((L_{L}(2\omega_{2}))_{h}(\mu_{h}))+\dim((L_{L}(3\omega_{2}))_{h}(\mu_{h}))\leq 4(\ell-1)^{2}+8(\ell-1)+4+2(\ell-1)\varepsilon_{p}(2\ell+1)-[2(\ell-1)+2]\varepsilon_{p}(2\ell+3)+\dim((L_{L}(3\omega_{2}))_{h}(\mu_{h}))$. Recursively and using Proposition \ref{PropositionB33om1} for the base case of $\ell=4$, we get $\scale[0.9]{\displaystyle \dim(V_{s}(\mu))\leq 4\sum_{j=3}^{\ell-1}[j^{2}+2j+1]+\sum_{j=3}^{\ell-1}2j\varepsilon_{p}(2j+3)-\sum_{j=3}^{\ell-1}(2j+2)\varepsilon_{p}(2j+5)+56=}$ $\binom{2\ell+1}{3}\scale[0.9]{+2{\ell}^{2}+\ell-2\ell\varepsilon_{p}(2\ell+3)}$ for all eigenvalues $\mu$ of $s$ on $V$. Moreover, recursively and using Proposition \ref{PropositionB33om1} for the base case of $\ell=4$, one shows that for $s=\diag(-1,\dots,-1,1,-1,\dots.-1)\in T\setminus\ZG(G)$ we have $\dim(V_{s}(-1))=\binom{2\ell+1}{3}\scale[0.9]{+2{\ell}^{2}+\ell-2\ell\varepsilon_{p}(2\ell+3)}$. Therefore, $\scale[0.9]{\displaystyle \max_{s\in T\setminus\ZG(G)}\dim(V_{s}(\mu))}$ $= \binom{2\ell+1}{3}\scale[0.9]{+2{\ell}^{2}+\ell-2\ell\varepsilon_{p}(2\ell+3)}$.

For the unipotent elements, by Lemma \ref{uniprootelems}, we have $\scale[0.9]{\displaystyle \max_{u\in G_{u}\setminus \{1\}}\dim(V_{u}(1))=\max_{i=\ell-1,\ell}\dim(V_{x_{\alpha_{i}}(1)}(1))}$. Further, we note that, by \cite[Propositions $4.7.4$]{mcninch_1998}, we have $\SWT(W)\cong V\oplus W$ if $\varepsilon_{p}(2\ell+3)=0$, respectively $\SWT(W)\cong W\mid V\mid W$ if $\varepsilon_{p}(2\ell+3)=1$. For $x_{\alpha_{\ell-1}}(1)$, we write $W=W_{1}\oplus W_{2}$, where $\dim(W_{1})=4$ and $x_{\alpha_{\ell-1}}(1)$ acts as $J_{2}^{2}$ on $W_{1}$, and $\dim(W_{2})=2\ell-3$ and $x_{\alpha_{\ell-1}}(1)$ acts trivially on $W_{2}$. One shows that $\dim((\SWT(W))_{x_{\alpha_{\ell-1}}(1)}(1))=\frac{4\ell^{3}+5\ell-3}{3}$. For $x_{\alpha_{\ell}}(1)$, we write $W=W'_{1}\oplus W'_{2}$, where $\dim(W'_{1})=3$ and $x_{\alpha_{\ell}}(1)$ acts as $J_{3}$ on $W'_{1}$, and $\dim(W'_{2})=2\ell-2$ and $x_{\alpha_{\ell}}(1)$ acts trivially on $W'_{2}$. Then $\dim((\SWT(W))_{x_{\alpha_{\ell}}(1)}(1))=\frac{4\ell^{3}+5\ell-3}{3}$. Thus, by the structure of $\SWT(W)$ as a $kG$-module and Lemma \ref{LemmaonfiltrationofV}, we determine that $\displaystyle \max_{u\in G_{u}\setminus \{1\}}\dim(V_{u}(1))$ $\geq \frac{4\ell^{3}-\ell}{3}-(2\ell-1)\varepsilon_{p}(2\ell+3)$, where equality holds when $\varepsilon_{p}(2\ell+3)=0$. On the other hand, by the structure of $V\mid_{[L,L]}$ and Propositions \ref{PropositionBlnatural} and \ref{PropositionBlsymm}, we determine that $\dim(V_{x_{\alpha_{i}}(1)}(1))\leq 4(\ell-1)^{2}+4(\ell-1)+1+(2\ell-3)\varepsilon_{p}(2\ell+1)-(2\ell-1)\varepsilon_{p}(2\ell+3)+\dim((L_{L}(3\omega_{2}))_{x_{\alpha_{i}}(1)}(1))$, for $i=\ell-1,\ell$. Recursively and using Proposition \ref{PropositionB33om1} for the base case of $\ell=4$, it follows that $\scale[0.9]{\displaystyle \dim(V_{x_{\alpha_{i}}(1)}(1))\leq 4\sum_{j=3}^{\ell-1}[j^{2}+j]+\ell-3+\sum_{j=3}^{\ell-1}(2j-1)\varepsilon_{p}(2j+3)-\ \ \ \ \ }$ $\scale[0.9]{\displaystyle\sum_{j=3}^{\ell-1}(2j+1)\varepsilon_{p}(2j+5)+35}$ $=\binom{2\ell+1}{3}-(2\ell-1)\varepsilon_{p}(2\ell+3)$. Therefore, $\scale[0.9]{\displaystyle \max_{u\in G_{u}\setminus \{1\}}\dim(V_{u}(1))}$ $=\binom{2\ell+1}{3}-(2\ell-1)\varepsilon_{p}(2\ell+3)$. In conclusion, by Lemma \ref{LtildeGtildelambdaandLGlambda}, we have shown that $\scale[0.9]{\displaystyle \max_{\tilde{u}\in \tilde{G}_{u}\setminus \{1\}}\dim(\tilde{V}_{\tilde{u}}(1))}$ $=\binom{2\ell+1}{3}-(2\ell-1)\varepsilon_{p}(2\ell+3)$ and $\scale[0.9]{\displaystyle \max_{\tilde{s}\in \tilde{T}\setminus\ZG(\tilde{G})}\dim(\tilde{V}_{\tilde{s}}(\tilde{\mu}))}$ $=\binom{2\ell+1}{3}+2{\ell}^{2}+\ell-2\ell\varepsilon_{p}(2\ell+3)$, and so $\nu_{\tilde{G}}(\tilde{V})=2\ell^{2}+\ell-\varepsilon_{p}(2\ell+3)$.
\end{proof}

\begin{prop}\label{PropositionB3om1+om2}
Let $\ell=3$ and $\tilde{V}=L_{\tilde{G}}(\tilde{\omega}_{1}+\tilde{\omega}_{2})$. Then $\nu_{\tilde{G}}(\tilde{V})=35-16\varepsilon_{p}(3)$. Moreover, we have $\scale[0.9]{\displaystyle \max_{\tilde{s}\in \tilde{T}\setminus\ZG(\tilde{G})}\dim(\tilde{V}_{\tilde{s}}(\tilde{\mu}))}=70-26\varepsilon_{p}(3)$ and $\scale[0.9]{\displaystyle \max_{\tilde{u}\in \tilde{G}_{u}\setminus \{1\}}\dim(\tilde{V}_{\tilde{u}}(1))}\leq 68-24\varepsilon_{p}(3)$. 
\end{prop}

\begin{proof}
Set $\tilde{\lambda}=\tilde{\omega}_{1}+\tilde{\omega}_{2}$ and $\tilde{L}=\tilde{L}_{1}$. By Lemma \ref{weightlevelBl}, we have $e_{1}(\tilde{\lambda})=4$, therefore $\displaystyle \tilde{V}\mid_{[\tilde{L},\tilde{L}]}=\tilde{V}^{0}\oplus \cdots \oplus \tilde{V}^{4}$. By \cite[Proposition]{Smith_82} and Lemma \ref{dualitylemma}, we have $\tilde{V}^{0}\cong L_{\tilde{L}}(\tilde{\omega}_{2})$ and $\tilde{V}^{4}\cong L_{\tilde{L}}(\tilde{\omega}_{2})$. Now, in $\tilde{V}^{1}$, the weight $\displaystyle (\tilde{\lambda}-\tilde{\alpha}_{1})\mid_{\tilde{T}_{1}}=2\tilde{\omega}_{2}$ admits a maximal vector, thus $\tilde{V}^{1}$ has a composition factor isomorphic to $L_{\tilde{L}}(2\tilde{\omega}_{2})$. Moreover, the weight $\displaystyle(\tilde{\lambda}-\tilde{\alpha}_{1}-\tilde{\alpha}_{2})\mid_{\tilde{T}_{1}}=2\tilde{\omega}_{3}$ occurs with multiplicity $2-\varepsilon_{p}(3)$ and is a sub-dominant weight in the composition factor of $\tilde{V}^{1}$ isomorphic to $L_{\tilde{L}}(2\tilde{\omega}_{2})$, in which it has multiplicity $1$. Lastly, we note that the weight $\displaystyle(\tilde{\lambda}-\tilde{\alpha}_{1}-2\tilde{\alpha}_{2}-2\tilde{\alpha}_{3})\mid_{\tilde{T}_{1}}=0$ occurs with multiplicity $5-3\varepsilon_{p}(3)$ in $\tilde{V}^{1}$. Similarly, in $\tilde{V}^{2}$ the weight $\displaystyle (\tilde{\lambda}-2\tilde{\alpha}_{1}-\tilde{\alpha_{2}})\mid_{\tilde{T}_{1}}=\tilde{\omega}_{2}+2\tilde{\omega_{3}}$ admits a maximal vector, therefore $\tilde{V}^{2}$ has a composition factor isomorphic to $L_{\tilde{L}}(\tilde{\omega}_{2}+2\tilde{\omega}_{3})$. Further, the weight $\displaystyle(\tilde{\lambda}-2\tilde{\alpha}_{1}-2\tilde{\alpha}_{2}-2\tilde{\alpha}_{3})\mid_{\tilde{T}_{1}}=\tilde{\omega}_{2}$ occurs with multiplicity $5-3\varepsilon_{p}(3)$ and is a sub-dominant weight in the composition factor of $\tilde{V}^{2}$ isomorphic to $L_{\tilde{L}}(\tilde{\omega}_{2}+2\tilde{\omega}_{3})$, in which it has multiplicity $3-\varepsilon_{p}(3)$. Now, as $\dim(\tilde{V}^{2})\leq 45-20\varepsilon_{p}(3)$, it follows that $\tilde{V}^{2}$ has exactly $3-2\varepsilon_{p}(3)$ composition factors: one isomorphic to $L_{\tilde{L}}(\tilde{\omega}_{2}+2\tilde{\omega}_{3})$ and $2-2\varepsilon_{p}(3)$ to $L_{\tilde{L}}(\tilde{\omega}_{2})$. Furthermore, $\tilde{V}^{1}$ and $\tilde{V}^{3}$ each has $3+\varepsilon_{p}(5)-2\varepsilon_{p}(3)$ composition factors: one isomorphic to $L_{\tilde{L}}(2\tilde{\omega}_{2})$, $1-\varepsilon_{p}(3)$ to $L_{\tilde{L}}(2\tilde{\omega}_{3})$ and $1+\varepsilon_{p}(5)-\varepsilon_{p}(3)$ to $L_{\tilde{L}}(0)$. 

We start with the semisimple elements. Let $\tilde{s}\in \tilde{T}\setminus\ZG(\tilde{G})$. If $\dim(\tilde{V}^{i}_{\tilde{s}}(\tilde{\mu}))=\dim(\tilde{V}^{i})$ for some eigenvalue $\tilde{\mu}$ of $\tilde{s}$ on $\tilde{V}$, where $0\leq i\leq 4$, then $\tilde{s}\in \ZG(\tilde{L})^{\circ}\setminus \ZG(\tilde{G})$. In this case, as $\tilde{s}$ acts on each $\tilde{V}^{i}$ as scalar multiplication by $c^{4-2i}$ and $c^{2}\neq 1$, we determine that $\dim(\tilde{V}_{\tilde{s}}(\tilde{\mu}))\leq 55-20\varepsilon_{p}(3)$ for all eigenvalues $\tilde{\mu}$ of $\tilde{s}$ on $\tilde{V}$. We thus assume that $\dim(\tilde{V}^{i}_{\tilde{s}}(\tilde{\mu}))<\dim(\tilde{V}^{i})$ for all eigenvalues $\tilde{\mu}$ of $\tilde{s}$ on $\tilde{V}$ and for $0\leq i\leq 4$. We write $\tilde{s}=\tilde{z}\cdot \tilde{h}$, where $\tilde{z}\in \ZG(\tilde{L})^{\circ}$ and $\tilde{h}\in [\tilde{L},\tilde{L}]$, and, using the structure of $\tilde{V}\mid_{[\tilde{L},\tilde{L}]}$ and Propositions \ref{PropositionClwedge}, \ref{C22om2}, \ref{C22om1+om2} and \ref{PropositionBlsymm}, we determine that $\dim(\tilde{V}_{\tilde{s}}(\tilde{\mu}))\leq (4-2\varepsilon_{p}(3))\dim((L_{\tilde{L}}(\tilde{\omega}_{2}))_{\tilde{h}}(\tilde{\mu}_{\tilde{h}}))+2\dim((L_{\tilde{L}}(2\tilde{\omega}_{2}))_{\tilde{h}}(\tilde{\mu}_{\tilde{h}}))+2\dim((L_{\tilde{L}}(\tilde{\omega}_{2}+2\tilde{\omega}_{3}))_{\tilde{h}}(\tilde{\mu}_{\tilde{h}}))+(2-2\varepsilon_{p}(3))\dim((L_{\tilde{L}}(2\tilde{\omega}_{3}))_{\tilde{h}}(\tilde{\mu}_{\tilde{h}}))+(2+2\varepsilon_{p}(5)-2\varepsilon_{p}(3))\dim((L_{\tilde{L}}(0))_{\tilde{h}}(\tilde{\mu}_{\tilde{h}}))\leq 70-26\varepsilon_{p}(3)$ for all eigenvalues $\tilde{\mu}$ of $\tilde{s}$ on $\tilde{V}$. Further, using Propositions \ref{PropositionClwedge}, \ref{C22om2}, \ref{C22om1+om2} and \ref{PropositionBlsymm}, one can show that for $\tilde{s}=h_{\tilde{\alpha}_{1}}(-1)h_{\tilde{\alpha}_{3}}(c)$ with $c^{2}=-1$ we have $\dim(\tilde{V}_{\tilde{s}}(-1))=70-26\varepsilon_{p}(3)$, thus $\scale[0.9]{\displaystyle \max_{\tilde{s}\in \tilde{T}\setminus\ZG(\tilde{G})}\dim(\tilde{V}_{\tilde{s}}(\tilde{\mu}))}=70-26\varepsilon_{p}(3)$.

We focus on the unipotent elements. By Lemma \ref{uniprootelems}, we have $\scale[0.9]{\displaystyle \max_{\tilde{u}\in \tilde{G}_{u}\setminus\{1\}}\dim(\tilde{V}_{\tilde{u}}(1))=\max_{i=2,3}\dim(\tilde{V}_{x_{\tilde{\alpha}_{i}}(1)}(1))}$, and, using the structure of $\tilde{V}\mid_{[\tilde{L},\tilde{L}]}$ and Propositions \ref{PropositionClwedge}, \ref{C22om2}, \ref{C22om1+om2} and \ref{PropositionBlsymm}, we determine that $\dim(\tilde{V}_{x_{\tilde{\alpha}_{i}}(1)}(1))$ $\leq 68-24\varepsilon_{p}(3)$, where $i=2,3$. Therefore, $\scale[0.9]{\displaystyle \max_{\tilde{u}\in \tilde{G}_{u}\setminus \{1\}}\dim(\tilde{V}_{\tilde{u}}(1))}\leq 68-24\varepsilon_{p}(3)$ and, as $\scale[0.9]{\displaystyle \max_{\tilde{s}\in \tilde{T}\setminus\ZG(\tilde{G})}\dim(\tilde{V}_{\tilde{s}}(\tilde{\mu}))}=70-26\varepsilon_{p}(3)$, we get $\nu_{\tilde{G}}(\tilde{V})=35-16\varepsilon_{p}(3)$. 
\end{proof}

\begin{prop}\label{PropositionBlom1+om2}
Let $\ell\geq 4$ and $\tilde{V}=L_{\tilde{G}}(\tilde{\omega}_{1}+\tilde{\omega}_{2})$. Then $\nu_{\tilde{G}}(\tilde{V})=4\ell^{2}-1-\varepsilon_{p}(\ell)-(2\ell^{2}-\ell)\varepsilon_{p}(3)$. Moreover, we have $\scale[0.9]{\displaystyle \max_{\tilde{s}\in \tilde{T}\setminus\ZG(\tilde{G})}\dim(\tilde{V}_{\tilde{s}}(\tilde{\mu}))}$ $=2\binom{2\ell+1}{3}-2\ell\varepsilon_{p}(\ell)-\binom{2\ell}{3}\varepsilon_{p}(3)$ and $\scale[0.9]{\displaystyle \max_{\tilde{u}\in \tilde{G}_{u}\setminus \{1\}}\dim(\tilde{V}_{\tilde{u}}(1))}$ $\leq \frac{8\ell^{3}-12\ell^{2}+22\ell+30}{3}-$ $(2\ell-1)\varepsilon_{p}(\ell)-\big(\frac{4\ell^{3}-12\ell^{2}+29\ell-30}{3}\big)\varepsilon_{p}(3)$.
\end{prop}

\begin{proof}
Set $\tilde{\lambda}=\tilde{\omega}_{1}+\tilde{\omega}_{2}$ and $\tilde{L}=\tilde{L}_{1}$. By Lemma \ref{weightlevelBl}, we have $e_{1}(\tilde{\lambda})=4$, therefore $\displaystyle \tilde{V}\mid_{[\tilde{L},\tilde{L}]}=\tilde{V}^{0}\oplus \cdots \oplus \tilde{V}^{4}$. By \cite[Proposition]{Smith_82} and Lemma \ref{dualitylemma}, we have $\tilde{V}^{0}\cong L_{\tilde{L}}(\tilde{\omega}_{2})$ and $\tilde{V}^{4}\cong L_{\tilde{L}}(\tilde{\omega}_{2})$. Now, in $\tilde{V}^{1}$, the weight $\displaystyle (\tilde{\lambda}-\tilde{\alpha}_{1})\mid_{\tilde{T}_{1}}=2\tilde{\omega}_{2}$ admits a maximal vector, thus $\tilde{V}^{1}$ has a composition factor isomorphic to $L_{\tilde{L}}(2\tilde{\omega}_{2})$. Moreover, the weight $\displaystyle(\tilde{\lambda}-\tilde{\alpha}_{1}-\tilde{\alpha}_{2})\mid_{\tilde{T}_{1}}=2\tilde{\omega}_{3}$ occurs with multiplicity $2-\varepsilon_{p}(3)$ and is a sub-dominant weight in the composition factor of $\tilde{V}^{1}$ isomorphic to $L_{\tilde{L}}(2\tilde{\omega}_{2})$, in which it has multiplicity $1$. Lastly, we note that the weight $\displaystyle(\tilde{\lambda}-\tilde{\alpha}_{1}-2\tilde{\alpha}_{2}-\cdots-2\tilde{\alpha}_{\ell})\mid_{\tilde{T}_{1}}=0$ occurs with multiplicity $(2\ell-1)-\varepsilon_{p}(\ell)-(\ell-1)\varepsilon_{p}(3)$. Similarly, in $\tilde{V}^{2}$ the weight $\displaystyle (\tilde{\lambda}-2\tilde{\alpha}_{1}-\tilde{\alpha_{2}})\mid_{\tilde{T}_{1}}=\tilde{\omega}_{2}+\tilde{\omega_{3}}$ admits a maximal vector, thus $\tilde{V}^{2}$ has a composition factor isomorphic to $L_{\tilde{L}}(\tilde{\omega}_{2}+\tilde{\omega}_{3})$. Further, the weight $\displaystyle(\tilde{\lambda}-2\tilde{\alpha}_{1}-\cdots-2\tilde{\alpha}_{\ell})\mid_{\tilde{T}_{1}}=\tilde{\omega}_{2}$, which occurs with multiplicity $(2\ell-1)-\varepsilon_{p}(\ell)-(\ell-1)\varepsilon_{p}(3)$ in $\tilde{V}^{2}$, is a sub-dominant weight in the composition factor of $\tilde{V}^{2}$ isomorphic to $L_{\tilde{L}}(\tilde{\omega}_{2}+\tilde{\omega}_{3})$, in which it has multiplicity $2\ell-3-\varepsilon_{p}(\ell-1)-(\ell-2)\varepsilon_{p}(3)$. Now, as $\dim(\tilde{V}^{2})\leq 2^{4}\binom{\ell+\frac{1}{2}}{3}+2(2\ell-1)-(2\ell-1)\varepsilon_{p}(\ell)-\big[\binom{2\ell-1}{3}+2\ell-1\big]\varepsilon_{p}(3)$, it follows that $\tilde{V}^{2}$ has exactly $3-\varepsilon_{p}(\ell)+\varepsilon_{p}(\ell-1)-\varepsilon_{p}(3)$ composition factors: one isomorphic to $L_{\tilde{L}}(\tilde{\omega}_{2}+\tilde{\omega}_{3})$ and $2-\varepsilon_{p}(\ell)+\varepsilon_{p}(\ell-1)-\varepsilon_{p}(3)$ to $L_{\tilde{L}}(\tilde{\omega}_{2})$. Further, $\tilde{V}^{1}$ and $\tilde{V}^{3}$ each has exactly $3+\varepsilon_{p}(2\ell-1)-\varepsilon_{p}(\ell)-\varepsilon_{p}(3)$ composition factors: one isomorphic to $L_{\tilde{L}}(2\tilde{\omega}_{2})$, $1-\varepsilon_{p}(3)$ to $L_{\tilde{L}}(\tilde{\omega}_{3})$ and $1+\varepsilon_{p}(2\ell-1)-\varepsilon_{p}(\ell)$ to $L_{\tilde{L}}(0)$. 

We start with the semisimple elements. Let $\tilde{s}\in \tilde{T}\setminus\ZG(\tilde{G})$. If $\dim(\tilde{V}^{i}_{\tilde{s}}(\tilde{\mu}))=\dim(\tilde{V}^{i})$ for some eigenvalue $\tilde{\mu}$ of $\tilde{s}$ on $\tilde{V}$, where $0\leq i\leq 4$, then $\tilde{s}\in \ZG(\tilde{L})^{\circ}\setminus \ZG(\tilde{G})$. In this case, as $\tilde{s}$ acts on each $\tilde{V}^{i}$ as scalar multiplication by $c^{4-2i}$ and $c^{2}\neq 1$, we determine that $\dim(\tilde{V}_{\tilde{s}}(\tilde{\mu}))\leq 2^{4}\binom{\ell+\frac{1}{2}}{3}+4(2\ell-1)-(2\ell-1)\varepsilon_{p}(\ell)-\big[\binom{2\ell-1}{3}+2\ell-1\big]\varepsilon_{p}(3)$ for all eigenvalues $\tilde{\mu}$ of $\tilde{s}$ on $\tilde{V}$. We thus assume that $\dim(\tilde{V}^{i}_{\tilde{s}}(\tilde{\mu}))<\dim(\tilde{V}^{i})$ for all eigenvalues $\tilde{\mu}$ of $\tilde{s}$ on $\tilde{V}$ and for all $0\leq i\leq 4$. We write $\tilde{s}=\tilde{z}\cdot \tilde{h}$, where $\tilde{z}\in \ZG(\tilde{L})^{\circ}$ and $\tilde{h}\in [\tilde{L},\tilde{L}]$, and,using the structure of $\tilde{V}\mid_{[\tilde{L},\tilde{L}]}$ and Propositions \ref{PropositionBlnatural}, \ref{PropositionBlwedge} and \ref{PropositionBlsymm}, we deduce that $\dim(\tilde{V}_{\tilde{s}}(\tilde{\mu})) \leq (4-\varepsilon_{p}(\ell)+\varepsilon_{p}(\ell-1)-\varepsilon_{p}(3))\dim((L_{\tilde{L}}(\tilde{\omega}_{2}))_{\tilde{h}}(\tilde{\mu}_{\tilde{h}}))+2\dim((L_{\tilde{L}}(2\tilde{\omega}_{2}))_{\tilde{h}}(\tilde{\mu}_{\tilde{h}}))+(2-2\varepsilon_{p}(3))\dim((L_{\tilde{L}}(\tilde{\omega}_{3}))_{\tilde{h}}(\tilde{\mu}_{\tilde{h}}))+(2+2\varepsilon_{p}(2\ell-1)-2\varepsilon_{p}(\ell))\dim((L_{\tilde{L}}(0))_{\tilde{h}}(\tilde{\mu}_{\tilde{h}}))+\dim((L_{\tilde{L}}(\tilde{\omega}_{2}+\tilde{\omega}_{3}))_{\tilde{h}}(\tilde{\mu}_{\tilde{h}})) \leq 8(\ell-1)^{2}+8(\ell-1)+2-2\ell\varepsilon_{p}(\ell)+2(\ell-1)\varepsilon_{p}(\ell-1)-4(\ell-1)^{2}\varepsilon_{p}(3)+\dim((L_{\tilde{L}}(\tilde{\omega}_{2}+\tilde{\omega}_{3}))_{\tilde{h}}(\tilde{\mu}_{\tilde{h}}))$. Recursively and using Proposition \ref{PropositionB3om1+om2} for the base case of $\ell=4$, we determine that $\scale[0.9]{\displaystyle\dim(\tilde{V}_{\tilde{s}}(\tilde{\mu}))\leq 8\sum_{j=3}^{\ell-1}j^{2}+8\sum_{j=3}^{\ell-1}j+\sum_{j=3}^{\ell-1}2-2\sum_{j=3}^{\ell-1}(j+1)\varepsilon_{p}(j+1)+2\sum_{j=3}^{\ell-1}j\varepsilon_{p}(j)-\big(4\sum_{j=3}^{\ell-1}j^{2}\big)\varepsilon_{p}(3)+70-26\varepsilon_{p}(3)}$ $=2\binom{2\ell+1}{3}$ $-2\ell\varepsilon_{p}(\ell)-\binom{2\ell}{3}\varepsilon_{p}(3)$ for all eigenvalues $\tilde{\mu}$ of $\tilde{s}$ on $\tilde{V}$. Lastly, similar to the proof of Proposition \ref{PropositionB3om1+om2}, one shows that for $\ell$ even and $\tilde{s}=h_{\tilde{\alpha}_{1}}(-1)h_{\tilde{\alpha}_{3}}(-1)\cdots h_{\tilde{\alpha}_{\ell-1}}(-1)h_{\tilde{\alpha}_{\ell}}(d)$ with $d^{2}=1$, respectively for $\ell$ odd and $\tilde{s}=h_{\tilde{\alpha}_{1}}(-1)h_{\tilde{\alpha}_{3}}(-1)\cdots h_{\tilde{\alpha}_{\ell-2}}(-1)h_{\tilde{\alpha}_{\ell}}(c)$ with $c^{2}=-1$, we have $\dim(\tilde{V}_{\tilde{s}}(-1))=2\binom{2\ell+1}{3}-2\ell\varepsilon_{p}(\ell)-\binom{2\ell}{3}\varepsilon_{p}(3)$. Therefore $\scale[0.9]{\displaystyle \max_{\tilde{s}\in \tilde{T}\setminus\ZG(\tilde{G})}\dim(\tilde{V}_{\tilde{s}}(\tilde{\mu}))}$ $=2\binom{2\ell+1}{3}-2\ell\varepsilon_{p}(\ell)-\binom{2\ell}{3}\varepsilon_{p}(3)$. 

For the unipotent elements, by Lemma \ref{uniprootelems}, we have $\scale[0.9]{\displaystyle \max_{\tilde{u}\in\tilde{G}_{u}\setminus\{1\} }\dim(\tilde{V}_{\tilde{u}}(1))=\max_{i=\ell-1,\ell}\dim(\tilde{V}_{x_{\tilde{\alpha}_{i}}(1)}(1))}$. Using the structure of $\tilde{V}\mid_{[\tilde{L},\tilde{L}]}$ and Propositions \ref{PropositionBlnatural}, \ref{PropositionBlwedge}, and \ref{PropositionBlsymm}, we determine that $\dim(\tilde{V}_{x_{\tilde{\alpha}_{i}}(1)}(1)) \leq 8(\ell-1)^{2}+6-(2\ell-1)\varepsilon_{p}(\ell)+(2\ell-3)\varepsilon_{p}(\ell-1)-[4(\ell-1)^{2}-4\ell+11]\varepsilon_{p}(3)+\dim((L_{\tilde{L}}(\tilde{\omega}_{2}+\tilde{\omega}_{3}))_{x_{\tilde{\alpha}_{i}}(1)}(1))$ for $i=\ell-1,\ell$. Recursively and using Proposition \ref{PropositionB3om1+om2} for the base case of $\ell=4$, we determine that $\scale[0.9]{\displaystyle\dim(\tilde{V}_{x_{\tilde{\alpha}_{i}}(1)}(1))\leq 8\sum_{j=3}^{\ell-1}j^{2}+\sum_{j=3}^{\ell-1}6-\big[\sum_{j=3}^{\ell-1}(2j+1)\varepsilon_{p}(j+1)\big]+\big[\sum_{j=3}^{\ell-1}(2j-1)\varepsilon_{p}(j)\big]-\big[4\sum_{j=3}^{\ell-1}j^{2}-4\sum_{j=3}^{\ell-1}j+\sum_{j=3}^{\ell-1}7\big]\varepsilon_{p}(3)+68\ \ \ }$ $\scale[0.9]{-24\varepsilon_{p}(3)=\frac{8\ell^{3}-12\ell^{2}+22\ell+30}{3}-(2\ell-1)\varepsilon_{p}(\ell)-\big(\frac{4\ell^{3}- 12\ell^{2}+ 29\ell-30}{3}\big)\varepsilon_{p}(3)}$. Therefore, $\scale[0.9]{\displaystyle \max_{\tilde{u}\in \tilde{G}_{u}\setminus \{1\}}\dim(\tilde{V}_{\tilde{u}}(1))}$ $\scale[0.9]{\leq \frac{8\ell^{3}-12\ell^{2}+22\ell+30}{3}}$ $-(2\ell-1)\varepsilon_{p}(\ell)-\big(\frac{4\ell^{3}- 12\ell^{2}+ 29\ell-30}{3}\big)\varepsilon_{p}(3)$ and, as $\scale[0.9]{\displaystyle \max_{\tilde{s}\in \tilde{T}\setminus\ZG(\tilde{G})}\dim(\tilde{V}_{\tilde{s}}(\tilde{\mu}))}$ $=2\binom{2\ell+1}{3}-2\ell\varepsilon_{p}(\ell)-\binom{2\ell}{3}\varepsilon_{p}(3)$, we get $\nu_{\tilde{G}}(\tilde{V})=4\ell^{2}-1-\varepsilon_{p}(\ell)-(2\ell^{2}-\ell)\varepsilon_{p}(3)$. 
\end{proof}

\begin{prop}\label{B3om1+om3}
Let $\ell=3$ and $\tilde{V}=L_{\tilde{G}}(\tilde{\omega}_{1}+\tilde{\omega}_{3})$. Then $\nu_{\tilde{G}}(\tilde{V})\geq 20-2\varepsilon_{p}(7)$, where equality holds for $p\neq 5$. Moreover, we have $\scale[0.9]{\displaystyle \max_{\tilde{u}\in \tilde{G}_{u}\setminus \{1\}}\dim(\tilde{V}_{\tilde{u}}(1))}\leq 28-6\varepsilon_{p}(7)$, where equality holds for $p\neq 5$, and $\scale[0.9]{\displaystyle \max_{\tilde{s}\in \tilde{T}\setminus\ZG(\tilde{G})}\dim(\tilde{V}_{\tilde{s}}(\tilde{\mu}))}=24-4\varepsilon_{p}(7)$.
\end{prop}

\begin{proof}
Set $\tilde{\lambda}=\tilde{\omega}_{1}+\tilde{\omega}_{3}$ and $\tilde{L}=\tilde{L}_{1}$. By Lemma \ref{weightlevelBl}, we have $e_{1}(\tilde{\lambda})=3$, therefore $\displaystyle \tilde{V}\mid_{[\tilde{L},\tilde{L}]}=\tilde{V}^{0}\oplus \cdots \oplus \tilde{V}^{3}$. By \cite[Proposition]{Smith_82} and Lemma \ref{dualitylemma}, we have $\tilde{V}^{0}\cong L_{\tilde{L}}(\tilde{\omega}_{3})$ and $\tilde{V}^{3}\cong L_{\tilde{L}}(\tilde{\omega}_{3})$. Now, in $\tilde{V}^{1}$ the weight $\displaystyle (\tilde{\lambda}-\tilde{\alpha}_{1})\mid_{\tilde{T}_{1}}=\tilde{\omega}_{2}+\tilde{\omega}_{3}$ admits a maximal vector, thus $\tilde{V}^{1}$ has a composition factor isomorphic to $L_{\tilde{L}}(\tilde{\omega}_{2}+\tilde{\omega}_{3})$. Moreover, the weight $\displaystyle(\tilde{\lambda}-\tilde{\alpha}_{1}-\tilde{\alpha}_{2}-\tilde{\alpha}_{3})\mid_{\tilde{T}_{1}}=\tilde{\omega}_{3}$ has multiplicity $3-\varepsilon_{p}(7)$ and is a sub-dominant weight in the composition factor of $\tilde{V}^{1}$ isomorphic to $L_{\tilde{L}}(\tilde{\omega}_{2}+\tilde{\omega}_{3})$, in which it has multiplicity $2-\varepsilon_{p}(5)$. Now, as $\dim(\tilde{V}^{1})\leq 20-4\varepsilon_{p}(7)$, we deduce that $\tilde{V}^{1}$ and $\tilde{V}^{2}$ each has exactly $2-\varepsilon_{p}(7)+\varepsilon_{p}(5)$ composition factors: one isomorphic to $L_{\tilde{L}}(\tilde{\omega}_{2}+\tilde{\omega}_{3})$ and $1-\varepsilon_{p}(7)+\varepsilon_{p}(5)$ to $L_{\tilde{L}}(\tilde{\omega}_{3})$. Moreover, when $p\neq 5$, we have $\tilde{V}^{1}\cong L_{\tilde{L}}(\tilde{\omega}_{2}+\tilde{\omega}_{3})\oplus L_{\tilde{L}}(\tilde{\omega}_{3})^{1-\varepsilon_{p}(7)}$ and $\tilde{V}^{2}\cong L_{\tilde{L}}(\tilde{\omega}_{2}+\tilde{\omega}_{3})\oplus L_{\tilde{L}}(\tilde{\omega}_{3})^{1-\varepsilon_{p}(7)}$.

We start with the semisimple elements. Let $\tilde{s}\in \tilde{T}\setminus\ZG(\tilde{G})$. If $\dim(\tilde{V}^{i}_{\tilde{s}}(\tilde{\mu}))=\dim(\tilde{V}^{i})$ for some eigenvalue $\tilde{\mu}$ of $\tilde{s}$ on $\tilde{V}$, where $0\leq i\leq 3$, then $\tilde{s}\in \ZG(\tilde{L})^{\circ}\setminus \ZG(\tilde{G})$. In this case, as $\tilde{s}$ acts on each $\tilde{V}^{i}$ as scalar multiplication by $c^{3-2i}$ and $c^{2}\neq 1$, we determine that $\dim(\tilde{V}_{\tilde{s}}(\tilde{\mu}))\leq 24-4\varepsilon_{p}(7)$, where equality holds for $c^{2}=-1$ and $\tilde{\mu}=\pm c$. We thus assume that $\dim(\tilde{V}^{i}_{\tilde{s}}(\tilde{\mu}))<\dim(\tilde{V}^{i})$ for all eigenvalues $\tilde{\mu}$ of $\tilde{s}$ on $\tilde{V}$ and all $0\leq i\leq 3$. We write $\tilde{s}=\tilde{z}\cdot \tilde{h}$, where $\tilde{z}\in \ZG(\tilde{L})^{\circ}$ and $\tilde{h}\in [\tilde{L},\tilde{L}]$, and, using the structure of $\tilde{V}\mid_{[\tilde{L},\tilde{L}]}$ and Propositions \ref{PropositionClnatural} and \ref{PropositionC2om1+om2}, we deduce $\dim(\tilde{V}_{\tilde{s}}(\tilde{\mu}))\leq (4-2\varepsilon_{p}(7)+2\varepsilon_{p}(5))\dim((L_{\tilde{L}}(\tilde{\omega}_{3}))_{\tilde{h}}(\tilde{\mu}_{\tilde{h}}))+2\dim((L_{\tilde{L}}(\tilde{\omega}_{2}+\tilde{\omega}_{3}))_{\tilde{h}}(\tilde{\mu}_{\tilde{h}}))\leq 24-4\varepsilon_{p}(7)$ for all eigenvalues $\tilde{\mu}$ of $\tilde{s}$ on $\tilde{V}$. Therefore, $\scale[0.9]{\displaystyle \max_{\tilde{s}\in \tilde{T}\setminus\ZG(\tilde{G})}\dim(\tilde{V}_{\tilde{s}}(\tilde{\mu}))=24-4\varepsilon_{p}(7)}$.

We now focus on the unipotent elements. By Lemma \ref{uniprootelems}, we have $\scale[0.9]{\displaystyle \max_{\tilde{u}\in \tilde{G}_{u}\setminus \{1\}}\dim(\tilde{V}_{\tilde{u}}(1))=\max_{i=2,3}\dim(\tilde{V}_{x_{\tilde{\alpha}_{i}}(1)}(1))}$. Using the structure of $\tilde{V}\mid_{[\tilde{L},\tilde{L}]}$ and Propositions \ref{PropositionClnatural} and \ref{PropositionC2om1+om2}, we determine that $\displaystyle \max_{\tilde{u}\in \tilde{G}_{u}\setminus \{1\}}\dim(\tilde{V}_{\tilde{u}}(1))\leq 28-6\varepsilon_{p}(7)$, where equality holds for $p\neq 5$. In conclusion, we have shown that $\nu_{\tilde{G}}(\tilde{V})\geq 20-2\varepsilon_{p}(7)$, where equality hods for $p\neq 5$.
\end{proof}

\begin{prop}\label{B3om2+om3p=5}
Let $p=5$, $\ell=3$ and $\tilde{V}=L_{\tilde{G}}(\tilde{\omega}_{2}+\tilde{\omega}_{3})$. Then $\nu_{\tilde{G}}(\tilde{V})=32$. Moreover, we have $\scale[0.9]{\displaystyle \max_{\tilde{s}\in \tilde{T}\setminus\ZG(\tilde{G})}\dim(\tilde{V}_{\tilde{s}}(\tilde{\mu}))}=32$ and $\scale[0.9]{\displaystyle \max_{\tilde{u}\in \tilde{G}_{u}\setminus \{1\}}\dim(\tilde{V}_{\tilde{u}}(1))}=30$.
\end{prop}

\begin{proof}
Set $\tilde{\lambda}=\tilde{\omega}_{2}+\tilde{\omega}_{3}$ and $\tilde{L}=\tilde{L}_{1}$. By Lemma \ref{weightlevelBl}, we have $e_{1}(\tilde{\lambda})=3$, therefore $\displaystyle \tilde{V}\mid_{[\tilde{L},\tilde{L}]}=\tilde{V}^{0}\oplus \cdots \oplus \tilde{V}^{3}$. By \cite[Proposition]{Smith_82} and Lemma \ref{dualitylemma}, we have $\tilde{V}^{0}\cong L_{\tilde{L}}(\tilde{\omega}_{2}+\tilde{\omega}_{3})$ and $\tilde{V}^{3}\cong L_{\tilde{L}}(\tilde{\omega}_{2}+\tilde{\omega}_{3})$. In $\tilde{V}^{1}$ the weight $\displaystyle (\tilde{\lambda}-\tilde{\alpha}_{1}-\tilde{\alpha}_{2})\mid_{\tilde{T}_{1}}=3\tilde{\omega}_{3}$ admits a maximal vector, thus $\tilde{V}^{1}$ has a composition factor isomorphic to $L_{\tilde{L}}(3\tilde{\omega}_{3})$. By dimensional considerations, we determine that:
\begin{equation}\label{DecompVB3om2+om3p=5}
\tilde{V}\mid_{[\tilde{L},\tilde{L}]}\cong L_{\tilde{L}}(\tilde{\omega}_{2}+\tilde{\omega}_{3}) \oplus L_{\tilde{L}}(3\tilde{\omega}_{3}) \oplus L_{\tilde{L}}(3\tilde{\omega}_{3}) \oplus L_{\tilde{L}}(\tilde{\omega}_{2}+\tilde{\omega}_{3}).
\end{equation}

We start with the semisimple elements. Let $\tilde{s}\in \tilde{T}\setminus\ZG(\tilde{G})$. If $\dim(\tilde{V}^{i}_{\tilde{s}}(\tilde{\mu}))=\dim(\tilde{V}^{i})$ for some eigenvalue $\tilde{\mu}$ of $\tilde{s}$ on $\tilde{V}$, where $0\leq i\leq 3$, then $\tilde{s}\in \ZG(\tilde{L})^{\circ}\setminus \ZG(\tilde{G})$. In this case, as $\tilde{s}$ acts on each $\tilde{V}^{i}$ as scalar multiplication by $c^{3-2i}$ and $c^{2}\neq 1$, we determine that $\dim(\tilde{V}_{\tilde{s}}(\tilde{\mu}))\leq 32$, where equality holds for $c^{2}=-1$ and $\tilde{\mu}=c^{\pm 1}$. We thus assume that $\dim(\tilde{V}^{i}_{\tilde{s}}(\tilde{\mu}))<\dim(\tilde{V}^{i})$ for all eigenvalues $\tilde{\mu}$ of $\tilde{s}$ on $\tilde{V}$ and all $0\leq i\leq 3$. We write $\tilde{s}=\tilde{z}\cdot \tilde{h}$, where $\tilde{z}\in \ZG(\tilde{L})^{\circ}$ and $\tilde{h}\in [\tilde{L},\tilde{L}]$, and, using \eqref{DecompVB3om2+om3p=5} and Propositions \ref{PropositionClsymmcube} and \ref{PropositionC2om1+om2}, we determine that $\dim(\tilde{V}_{\tilde{s}}(\tilde{\mu}))\leq 2\dim((L_{\tilde{L}}(\tilde{\omega}_{2}+\tilde{\omega}_{3}))_{\tilde{h}}(\tilde{\mu}_{\tilde{h}}))+2\dim((L_{\tilde{L}}(3\tilde{\omega}_{3}))_{\tilde{h}}(\tilde{\mu}_{\tilde{h}}))\leq 32$ for all eigenvalues $\tilde{\mu}$ of $\tilde{s}$ on $\tilde{V}$. Therefore, $\scale[0.9]{\displaystyle \max_{\tilde{s}\in \tilde{T}\setminus\ZG(\tilde{G})}\dim(\tilde{V}_{\tilde{s}}(\tilde{\mu}))=32}$.

We now focus on the unipotent elements. By Lemma \ref{uniprootelems}, we have $\scale[0.9]{\displaystyle \max_{\tilde{u}\in \tilde{G}_{u}\setminus \{1\}}\dim(\tilde{V}_{\tilde{u}}(1))=\max_{i=2,3}\dim(\tilde{V}_{x_{\tilde{\alpha}_{i}}(1)}(1))}$. Using \eqref{DecompVB3om2+om3p=5} and Propositions \ref{PropositionClsymmcube} and \ref{PropositionC2om1+om2}, we deduce that $\scale[0.9]{\displaystyle \max_{\tilde{u}\in \tilde{G}_{u}\setminus \{1\}}\dim(\tilde{V}_{\tilde{u}}(1))=30}$, and so, as $\scale[0.9]{\displaystyle \max_{\tilde{s}\in \tilde{T}\setminus\ZG(\tilde{G})}\dim(\tilde{V}_{\tilde{s}}(\tilde{\mu}))}$ $=32$, we have $\nu_{\tilde{G}}(\tilde{V})=32$.
\end{proof}

\begin{prop}\label{B3om2+om3p=3}
Let $p=3$, $\ell=3$ and $\tilde{V}=L_{\tilde{G}}(\tilde{\omega}_{2}+\tilde{\omega}_{3})$. Then $\nu_{\tilde{G}}(\tilde{V})\geq 50$. Moreover, we have $\scale[0.9]{\displaystyle \max_{\tilde{s}\in \tilde{T}\setminus\ZG(\tilde{G})}\dim(\tilde{V}_{\tilde{s}}(\tilde{\mu}))}=52$ and $\scale[0.9]{\displaystyle \max_{\tilde{u}\in \tilde{G}_{u}\setminus \{1\}}\dim(\tilde{V}_{\tilde{u}}(1))}\leq 54$.
\end{prop}

\begin{proof}
Set $\tilde{\lambda}=\tilde{\omega}_{2}+\tilde{\omega}_{3}$ and $\tilde{L}=\tilde{L}_{1}$. By Lemma \ref{weightlevelBl}, we have $e_{1}(\tilde{\lambda})=3$, therefore $\displaystyle \tilde{V}\mid_{[\tilde{L},\tilde{L}]}=\tilde{V}^{0}\oplus \cdots \oplus \tilde{V}^{3}$. By \cite[Proposition]{Smith_82} and Lemma \ref{dualitylemma}, we have $\tilde{V}^{0}\cong L_{\tilde{L}}(\tilde{\omega}_{2}+\tilde{\omega}_{3})$ and $\tilde{V}^{3}\cong L_{\tilde{L}}(\tilde{\omega}_{2}+\tilde{\omega}_{3})$. In $\tilde{V}^{1}$ the weight $\displaystyle (\tilde{\lambda}-\tilde{\alpha}_{1}-\tilde{\alpha}_{2})\mid_{\tilde{T}_{1}}=3\tilde{\omega}_{3}$ admits a maximal vector, thus $\tilde{V}^{1}$ has a composition factor isomorphic to $L_{\tilde{L}}(3\tilde{\omega}_{3})$. Moreover, the weight $\displaystyle (\tilde{\lambda}-\tilde{\alpha}_{1}-\tilde{\alpha}_{2}-\tilde{\alpha}_{3})\mid_{\tilde{T}_{1}}=\tilde{\omega}_{2}+\tilde{\omega}_{3}$ occurs with multiplicity $2$ and is not a sub-dominant weight in the composition factor of $\tilde{V}^{1}$ isomorphic to $L_{\tilde{L}}(3\tilde{\omega}_{3})$. Thus, as $\dim(\tilde{V}^{1})=36$, we determine that $\tilde{V}^{1}$, hence $\tilde{V}^{2}$, has exactly $3$ composition factors: one isomorphic to $L_{\tilde{L}}(3\tilde{\omega}_{3})$ and two to $L_{\tilde{L}}(\tilde{\omega}_{2}+\tilde{\omega}_{3})$.

We start with the semisimple elements. Let $\tilde{s}\in \tilde{T}\setminus\ZG(\tilde{G})$. If $\dim(\tilde{V}^{i}_{\tilde{s}}(\tilde{\mu}))=\dim(\tilde{V}^{i})$ for some eigenvalue $\tilde{\mu}$ of $\tilde{s}$ on $\tilde{V}$, where $0\leq i\leq 3$, then $\tilde{s}\in \ZG(\tilde{L})^{\circ}\setminus \ZG(\tilde{G})$. In this case, as $\tilde{s}$ acts on each $\tilde{V}^{i}$ as scalar multiplication by $c^{3-2i}$ and $c^{2}\neq 1$, we determine that $\dim(\tilde{V}_{\tilde{s}}(\tilde{\mu}))\leq 52$, where equality holds for $c^{2}=-1$ and $\tilde{\mu}=c^{\pm 1}$. We thus assume that $\dim(\tilde{V}^{i}_{\tilde{s}}(\tilde{\mu}))<\dim(\tilde{V}^{i})$ for all eigenvalues $\tilde{\mu}$ of $\tilde{s}$ on $\tilde{V}$ and all $0\leq i\leq 3$. We write $\tilde{s}=\tilde{z}\cdot \tilde{h}$, where $\tilde{z}\in \ZG(\tilde{L})^{\circ}$ and $\tilde{h}\in [\tilde{L},\tilde{L}]$, and, using the structure of $\tilde{V}\mid_{[\tilde{L},\tilde{L}]}$ and Propositions \ref{PropositionClnatural} and \ref{PropositionC2om1+om2}, we determine that: $\dim(\tilde{V}_{\tilde{s}}(\tilde{\mu}))\leq 6\dim((L_{\tilde{L}}(\tilde{\omega}_{2}+\tilde{\omega}_{3}))_{\tilde{h}}(\tilde{\mu}_{\tilde{h}}))+2\dim((L_{\tilde{L}}(3\tilde{\omega}_{3}))_{\tilde{h}}(\tilde{\mu}_{\tilde{h}}))\leq 52$ for all eigenvalues $\tilde{\mu}$ of $\tilde{s}$ on $\tilde{V}$. Therefore, $\scale[0.9]{\displaystyle \max_{\tilde{s}\in \tilde{T}\setminus\ZG(\tilde{G})}\dim(\tilde{V}_{\tilde{s}}(\tilde{\mu}))=52}$

We focus on the unipotent elements. By Lemma \ref{uniprootelems}, we have $\scale[0.9]{\displaystyle \max_{\tilde{u}\in \tilde{G}_{u}\setminus\{1\}}\dim(\tilde{V}_{\tilde{u}}(1))=\max_{i=2,3}\dim(\tilde{V}_{x_{\tilde{\alpha}_{i}}(1)}(1))}$. Using the structure of $\tilde{V}\mid_{[\tilde{L},\tilde{L}]}$ and Propositions \ref{PropositionClnatural} and \ref{PropositionC2om1+om2}, we determine that $\scale[0.9]{\displaystyle \max_{\tilde{u}\in \tilde{G}_{u}\setminus\{1\}}\dim(\tilde{V}_{\tilde{u}}(1))\leq 54}$. In conclusion, we have shown that $\nu_{\tilde{G}}(\tilde{V})\geq 50$.
\end{proof}

\begin{prop}\label{B33om3p=5}
Let $p=5$, $\ell=3$ and $\tilde{V}=L_{\tilde{G}}(3\tilde{\omega}_{3})$. Then $\nu_{\tilde{G}}(\tilde{V})=52$. Moreover, we have $\scale[0.9]{\displaystyle \max_{\tilde{s}\in \tilde{T}\setminus\ZG(\tilde{G})}\dim(\tilde{V}_{\tilde{s}}(\tilde{\mu}))}$ $=52$ and $\scale[0.9]{\displaystyle \max_{\tilde{u}\in \tilde{G}_{u}\setminus \{1\}}\dim(\tilde{V}_{\tilde{u}}(1))}=50$.
\end{prop}

\begin{proof}
Set $\tilde{\lambda}=3\tilde{\omega}_{3}$ and $\tilde{L}=\tilde{L}_{1}$. By Lemma \ref{weightlevelBl}, we have $e_{1}(\tilde{\lambda})=3$, therefore $\displaystyle \tilde{V}\mid_{[\tilde{L},\tilde{L}]}=\tilde{V}^{0}\oplus \cdots \oplus \tilde{V}^{3}$. By \cite[Proposition]{Smith_82} and Lemma \ref{dualitylemma}, we have $\tilde{V}^{0}\cong L_{\tilde{L}}(3\tilde{\omega}_{3})$ and $\tilde{V}^{3}\cong L_{\tilde{L}}(3\tilde{\omega}_{3})$. In $\tilde{V}^{1}$ the weight $\displaystyle (\tilde{\lambda}-\tilde{\alpha}_{1}-\tilde{\alpha}_{2}-\tilde{\alpha}_{3})\mid_{\tilde{T}_{1}}=3\tilde{\omega}_{3}$ admits a maximal vector, thus $\tilde{V}^{1}$ has a composition factor isomorphic to $L_{\tilde{L}}(3\tilde{\omega}_{3})$. Moreover, the weight $\displaystyle (\tilde{\lambda}-\tilde{\alpha}_{1}-\tilde{\alpha}_{2}-2\tilde{\alpha}_{3})\mid_{\tilde{T}_{1}}=\tilde{\omega}_{2}+\tilde{\omega}_{3}$ occurs with multiplicity $2$ and is a sub-dominant weight in the composition factor of $\tilde{V}^{1}$ isomorphic to $L_{\tilde{L}}(3\tilde{\omega}_{3})$, in which it has multiplicity $1$. As $\dim(\tilde{V}^{1})=32$, by \cite[II.$2.14$]{Jantzen_2007representations}, we determine that $\tilde{V}^{1}\cong L_{\tilde{L}}(3\tilde{\omega}_{3})\oplus L_{\tilde{L}}(\tilde{\omega}_{2}+\tilde{\omega}_{3})$ and $\tilde{V}^{2}\cong L_{\tilde{L}}(3\tilde{\omega}_{3})\oplus L_{\tilde{L}}(\tilde{\omega}_{2}+\tilde{\omega}_{3})$, therefore
\begin{equation}\label{DecompVB33om3}
\tilde{V}\mid_{[\tilde{L},\tilde{L}]}\cong L_{\tilde{L}}(3\tilde{\omega}_{3}) \oplus  L_{\tilde{L}}(3\tilde{\omega}_{3})\oplus L_{\tilde{L}}(\tilde{\omega}_{2}+\tilde{\omega}_{3})\oplus  L_{\tilde{L}}(3\tilde{\omega}_{3})\oplus L_{\tilde{L}}(\tilde{\omega}_{2}+\tilde{\omega}_{3})  \oplus L_{\tilde{L}}(3\tilde{\omega}_{3}).
\end{equation}

We start with the semisimple elements. Let $\tilde{s}\in \tilde{T}\setminus\ZG(\tilde{G})$. If $\dim(\tilde{V}^{i}_{\tilde{s}}(\tilde{\mu}))=\dim(\tilde{V}^{i})$ for some eigenvalue $\tilde{\mu}$ of $\tilde{s}$ on $\tilde{V}$, where $0\leq i\leq 3$, then $\tilde{s}\in \ZG(\tilde{L})^{\circ}\setminus \ZG(\tilde{G})$. In this case, as $\tilde{s}$ acts on each $\tilde{V}^{i}$ as scalar multiplication by $c^{3-2i}$ and $c^{2}\neq 1$, we determine that $\dim(\tilde{V}_{\tilde{s}}(\tilde{\mu}))\leq 52$, where equality holds for $c^{2}=-1$ and $\tilde{\mu}=c^{\pm 1}$. We thus assume that $\dim(\tilde{V}^{i}_{\tilde{s}}(\tilde{\mu}))<\dim(\tilde{V}^{i})$ for all eigenvalues $\tilde{\mu}$ of $\tilde{s}$ on $\tilde{V}$ and all $0\leq i\leq 3$. We write $\tilde{s}=\tilde{z}\cdot \tilde{h}$, where $\tilde{z}\in \ZG(\tilde{L})^{\circ}$ and $\tilde{h}\in [\tilde{L},\tilde{L}]$, and, using \eqref{DecompVB33om3} and Propositions \ref{PropositionClsymmcube} and \ref{PropositionC2om1+om2}, we determine that $\dim(\tilde{V}_{\tilde{s}}(\tilde{\mu}))\leq 2\dim((L_{\tilde{L}}(\tilde{\omega}_{2}+\tilde{\omega}_{3}))_{\tilde{h}}(\tilde{\mu}_{\tilde{h}}))+4\dim((L_{\tilde{L}}(3\tilde{\omega}_{3}))_{\tilde{h}}(\tilde{\mu}_{\tilde{h}}))\leq 52$ for all eigenvalues $\tilde{\mu}$ of $\tilde{s}$ on $\tilde{V}$. Therefore, $\scale[0.9]{\displaystyle \max_{\tilde{s}\in \tilde{T}\setminus\ZG(\tilde{G})}\dim(\tilde{V}_{\tilde{s}}(\tilde{\mu}))}=52$. 

We focus on the unipotent elements. By Lemma \ref{uniprootelems}, we have $\scale[0.9]{\displaystyle \max_{\tilde{u}\in \tilde{G}_{u}\setminus\{1\}}\dim(\tilde{V}_{\tilde{u}}(1))=\max_{i=2,3}\dim(\tilde{V}_{x_{\tilde{\alpha}_{i}}(1)}(1))}$. Using \eqref{DecompVB33om3} and Propositions \ref{PropositionClsymmcube} and \ref{PropositionC2om1+om2}, we determine that $\scale[0.9]{\displaystyle \max_{\tilde{u}\in \tilde{G}_{u}\setminus\{1\}}\dim(\tilde{V}_{\tilde{u}}(1))=50}$. As $\scale[0.9]{\displaystyle \max_{\tilde{s}\in \tilde{T}\setminus\ZG(\tilde{G})}\dim(\tilde{V}_{\tilde{s}}(\tilde{\mu}))}$ $=52$, it follows that $\nu_{\tilde{G}}(\tilde{V})=52$.
\end{proof}

\begin{prop}\label{Bl2oml}
Let $\ell=4,5$ and $\tilde{V}=L_{\tilde{G}}(2\tilde{\omega}_{\ell})$. Then $\nu_{\tilde{G}}(\tilde{V})=183-133\varepsilon_{\ell}(4)$. Moreover, 
\begin{enumerate}
\item[\emph{$(1)$}] for $\ell=4$, we have $\scale[0.9]{\displaystyle \max_{\tilde{s}\in \tilde{T}\setminus\ZG(\tilde{G})}\dim(\tilde{V}_{\tilde{s}}(\tilde{\mu}))\leq 74}$ and $\scale[0.9]{\displaystyle \max_{\tilde{u}\in \tilde{G}_{u}\setminus \{1\}}\dim(\tilde{V}_{\tilde{u}}(1))=76}$.
\item[\emph{$(2)$}] for $\ell=5$, we have $\scale[0.9]{\displaystyle \max_{\tilde{s}\in \tilde{T}\setminus\ZG(\tilde{G})}\dim(\tilde{V}_{\tilde{s}}(\tilde{\mu}))\leq 278}$ and $\scale[0.9]{\displaystyle \max_{\tilde{u}\in \tilde{G}_{u}\setminus \{1\}}\dim(\tilde{V}_{\tilde{u}}(1))=279}$.
\end{enumerate}
\end{prop}

\begin{proof}
Set $\tilde{\lambda}=2\tilde{\omega}_{\ell}$ and $\tilde{L}=\tilde{L}_{1}$. By Lemma \ref{weightlevelBl}, we have $e_{1}(\tilde{\lambda})=2$, therefore $\displaystyle \tilde{V}\mid_{[\tilde{L},\tilde{L}]}=\tilde{V}^{0}\oplus \tilde{V}^{1} \oplus \tilde{V}^{2}$. By \cite[Proposition]{Smith_82} and Lemma \ref{dualitylemma}, we have $\tilde{V}^{0}\cong L_{\tilde{L}}(2\tilde{\omega}_{\ell})$ and $\tilde{V}^{2}\cong L_{\tilde{L}}(2\tilde{\omega}_{\ell})$. Now, in $\tilde{V}^{1}$ the weight $\displaystyle (\tilde{\lambda}-\tilde{\alpha}_{1}-\cdots-\tilde{\alpha}_{\ell})\mid_{\tilde{T}_{1}}=2\tilde{\omega}_{\ell}$ admits a maximal vector, therefore $\tilde{V}^{1}$ has a composition factor isomorphic to $L_{\tilde{L}}(2\tilde{\omega}_{\ell})$. Moreover, the weight $\displaystyle (\tilde{\lambda}-\tilde{\alpha}_{1}-\cdots-\tilde{\alpha}_{\ell-1}-2\tilde{\alpha}_{\ell})\mid_{\tilde{T}_{1}}=\tilde{\omega}_{\ell-1}$ occurs with multiplicity $2$ and is a sub-dominant weight in the composition factor of $\tilde{V}^{1}$ isomorphic to $L_{\tilde{L}}(2\tilde{\omega}_{\ell})$, in which it has multiplicity $1$. By dimensional considerations and  \cite[II.2.14]{Jantzen_2007representations}, we determine that $\tilde{V}^{1}\cong L_{\tilde{L}}(2\tilde{\omega}_{\ell})\oplus L_{\tilde{L}}(\tilde{\omega}_{\ell-1})$, therefore
\begin{equation}\label{DecompVBl2oml}
\tilde{V}\mid_{[\tilde{L},\tilde{L}]}\cong L_{\tilde{L}}(2\tilde{\omega}_{\ell})\oplus L_{\tilde{L}}(2\tilde{\omega}_{\ell})\oplus L_{\tilde{L}}(\tilde{\omega}_{\ell-1}) \oplus L_{\tilde{L}}(2\tilde{\omega}_{\ell}).
\end{equation}

We start with the semisimple elements. Let $\tilde{s}\in \tilde{T}\setminus\ZG(\tilde{G})$. If $\dim(\tilde{V}^{i}_{\tilde{s}}(\tilde{\mu}))=\dim(\tilde{V}^{i})$ for some eigenvalue $\tilde{\mu}$ of $\tilde{s}$ on $\tilde{V}$, where $0\leq i\leq 2$, then $\tilde{s}\in \ZG(\tilde{L})^{\circ}\setminus \ZG(\tilde{G})$. In this case, as $\tilde{s}$ acts on each $\tilde{V}^{i}$ as scalar multiplication by $c^{2-2i}$ and $c^{2}\neq 1$, we determine that $\dim(\tilde{V}_{s}(\tilde{\mu}))\leq 2^{\ell}$ for all eigenvalues $\tilde{\mu}$ of $\tilde{s}$ on $\tilde{V}$. We thus assume that $\dim(\tilde{V}^{i}_{\tilde{s}}(\tilde{\mu}))<\dim(\tilde{V}^{i})$ for all eigenvalues $\tilde{\mu}$ of $\tilde{s}$ on $\tilde{V}$ and all $0\leq i\leq 2$. We write $\tilde{s}=\tilde{z}\cdot \tilde{h}$, where $\tilde{z}\in \ZG(\tilde{L})^{\circ}$ and $\tilde{h}\in [\tilde{L},\tilde{L}]$. For $\ell=4$, using \eqref{DecompVBl2oml} and Propositions \ref{PropositionBlwedge} and \ref{PropositionB32om3}, we get $\dim(\tilde{V}_{\tilde{s}}(\tilde{\mu}))\leq 3\dim((L_{\tilde{L}}(2\tilde{\omega}_{4}))_{\tilde{h}}(\tilde{\mu}_{\tilde{h}}))+\dim((L_{\tilde{L}}(\tilde{\omega}_{3}))_{\tilde{h}}(\tilde{\mu}_{\tilde{h}}))\leq 75$ for all eigenvalues $\tilde{\mu}$ of $\tilde{s}$ on $\tilde{V}$. However, assume there exist $(\tilde{s},\tilde{\mu})\in \tilde{T}\setminus \ZG(\tilde{G})\times k^{*}$ with $\dim(\tilde{V}_{\tilde{s}}(\tilde{\mu}))=75$. Then, by the above arguments, we must have $\dim((L_{\tilde{L}}(\tilde{\omega}_{3}))_{\tilde{h}}(\tilde{\mu}_{\tilde{h}}))=15$, therefore $\tilde{h}=h_{\tilde{\alpha}_{2}}(-1)h_{\tilde{\alpha}_{4}}(d)$ with $d^{2}=-1$ and $\tilde{\mu}_{\tilde{h}}=1$, by Proposition \ref{PropositionBlwedge}. But, we have seen in the proof of Proposition \ref{PropositionB32om3} that $\dim((L_{\tilde{L}}(2\tilde{\omega}_{4}))_{\tilde{h}}(1))=15$ and so $\dim(\tilde{V}^{1}_{\tilde{h}}(1)=30$, making $\dim(\tilde{V}_{\tilde{s}}(\tilde{\mu}))\leq 70$. Therefore $\dim(\tilde{V}_{\tilde{s}}(\tilde{\mu}))\leq 74$ for all $(\tilde{s},\tilde{\mu})\in \tilde{T}\setminus \ZG(\tilde{G})\times k^{*}$. Similarly, for $\ell=5$, using \eqref{DecompVBl2oml}, the result for $\ell=4$ and Proposition \ref{PropositionBlwedgecube}, we determine that $\dim(\tilde{V}_{\tilde{s}}(\tilde{\mu}))\leq 278$ for all eigenvalues $\tilde{\mu}$ of $\tilde{s}$ on $\tilde{V}$. Therefore $\scale[0.9]{\displaystyle \max_{\tilde{s}\in \tilde{T}\setminus\ZG(\tilde{G})}\dim(\tilde{V}_{\tilde{s}}(\tilde{\mu}))\leq 74}$ for $\ell=4$, respectively $\scale[0.9]{\displaystyle \max_{\tilde{s}\in \tilde{T}\setminus\ZG(\tilde{G})}\dim(\tilde{V}_{\tilde{s}}(\tilde{\mu}))\leq 278}$ for $\ell=5$.

We focus on the unipotent elements. By Lemma \ref{uniprootelems}, we have $\scale[0.9]{\displaystyle \max_{\tilde{u}\in \tilde{G}_{u}\setminus\{1\}}\dim(\tilde{V}_{\tilde{u}}(1))=\max_{i=\ell-1,\ell}\dim(\tilde{V}_{x_{\tilde{\alpha}_{i}}(1)}(1))}$. For $\ell=4$, using \eqref{DecompVBl2oml} and Propositions \ref{PropositionBlwedge}  and \ref{PropositionB32om3},we determine that $\scale[0.9]{\displaystyle \max_{\tilde{u}\in \tilde{G}_{u}\setminus\{1\}}\dim(\tilde{V}_{\tilde{u}}(1))=76}$. Similarly, for $\ell=5$, using  \eqref{DecompVBl2oml}, the result for $\ell=4$ and Proposition \ref{PropositionBlwedgecube}, we get $\scale[0.9]{\displaystyle \max_{\tilde{u}\in \tilde{G}_{u}\setminus\{1\}}\dim(\tilde{V}_{\tilde{u}}(1))=279}$. Lastly, we note that $\nu_{\tilde{G}}(\tilde{V})=50$ for $\ell=4$ and $\nu_{\tilde{G}}(\tilde{V})=183$ for $\ell=5$.
\end{proof}

\begin{prop}\label{Blom1+oml}
Let $4\leq \ell\leq 6$ and $\tilde{V}=L_{\tilde{G}}(\tilde{\omega}_{1}+\tilde{\omega}_{\ell})$. Then $\nu_{\tilde{G}}(\tilde{V})\geq (\ell+2)\cdot 2^{\ell-1}-2^{\ell-2}\varepsilon_{p}(2\ell+1)$, where equality holds when $\varepsilon_{p}(2(\ell-i)+1)=0$ for all $1\leq i\leq \ell-3$. Moreover, we have $\scale[0.9]{\displaystyle \max_{\tilde{s}\in \tilde{T}\setminus\ZG(\tilde{G})}\dim(\tilde{V}_{\tilde{s}}(\tilde{\mu}))}=\ell\cdot 2^{\ell}-2^{\ell-1}\varepsilon_{p}(2\ell+1)$ and $\scale[0.9]{\displaystyle \max_{\tilde{u}\in \tilde{G}_{u}\setminus \{1\}}\dim(\tilde{V}_{\tilde{u}}(1))}\leq (3\ell-2)\cdot 2^{\ell-1}-3\cdot 2^{\ell-2}\varepsilon_{p}(2\ell+1)$, where equality holds when $\varepsilon_{p}(2(\ell-i)+1)=0$ for all $1\leq i\leq \ell-3$.
\end{prop}

\begin{proof}
Set $\tilde{\lambda}=\tilde{\omega}_{1}+\tilde{\omega}_{\ell}$ and $\tilde{L}=\tilde{L}_{1}$. By Lemma \ref{weightlevelBl}, we have $e_{1}(\tilde{\lambda})=3$, therefore $\displaystyle \tilde{V}\mid_{[\tilde{L},\tilde{L}]}=\tilde{V}^{0}\oplus \cdots \oplus \tilde{V}^{3}$. By \cite[Proposition]{Smith_82} and Lemma \ref{dualitylemma}, we have $\tilde{V}^{0}\cong L_{\tilde{L}}(\tilde{\omega}_{\ell})$ and $\tilde{V}^{3}\cong L_{\tilde{L}}(\tilde{\omega}_{\ell})$. Now, in $\tilde{V}^{1}$ the weight $\displaystyle (\tilde{\lambda}-\tilde{\alpha}_{1})\mid_{\tilde{T}_{1}}=\tilde{\omega}_{2}+\tilde{\omega}_{\ell}$ admits a maximal vector, thus $\tilde{V}^{1}$ has a composition factor isomorphic to $L_{\tilde{L}}(\tilde{\omega}_{2}+\tilde{\omega}_{\ell})$. Moreover, the weight $\displaystyle (\tilde{\lambda}-\tilde{\alpha}_{1}-\cdots-\tilde{\alpha}_{\ell})\mid_{\tilde{T}_{1}}=\tilde{\omega}_{\ell}$ occurs with multiplicity $\ell-\varepsilon_{p}(2\ell+1)$ and is a sub-dominant weight in the composition factor of $\tilde{V}^{1}$ isomorphic to $L_{\tilde{L}}(\tilde{\omega}_{2}+\tilde{\omega}_{\ell})$, in which it has multiplicity $\ell-1-\varepsilon_{p}(2\ell-1)$. Thus, as $\dim(\tilde{V}^{1})=(\ell-1)\cdot 2^{\ell}+2^{\ell-1}-2^{\ell-1}\varepsilon_{p}(2\ell+1)$, we determine that $\tilde{V}^{1}$, hence $\tilde{V}^{2}$, has exactly $2+\varepsilon_{p}(2\ell-1)-\varepsilon_{p}(2\ell+1)$ composition factors: one isomorphic to $L_{\tilde{L}}(\tilde{\omega}_{2}+\tilde{\omega}_{\ell})$ and $1+\varepsilon_{p}(2\ell-1)-\varepsilon_{p}(2\ell+1)$ to $L_{\tilde{L}}(\tilde{\omega}_{\ell})$. Further, when $\varepsilon_{p}(2\ell-1)=0$, by \cite[II.2.14]{Jantzen_2007representations}, we have $\tilde{V}^{1}\cong L_{\tilde{L}}(\tilde{\omega}_{2}+\tilde{\omega}_{\ell})\oplus L_{\tilde{L}}(\tilde{\omega}_{\ell})^{1-\varepsilon_{p}(2\ell+1)}$ and $\tilde{V}^{2}\cong L_{\tilde{L}}(\tilde{\omega}_{2}+\tilde{\omega}_{\ell})\oplus L_{\tilde{L}}(\tilde{\omega}_{\ell})^{1-\varepsilon_{p}(2\ell+1)}$.

We start with the semisimple elements. Let $\tilde{s}\in \tilde{T}\setminus\ZG(\tilde{G})$. If $\dim(\tilde{V}^{i}_{\tilde{s}}(\tilde{\mu}))=\dim(\tilde{V}^{i})$ for some eigenvalue $\tilde{\mu}$ of $\tilde{s}$ on $\tilde{V}$, where $0\leq i\leq 3$, then $\tilde{s}\in \ZG(\tilde{L})^{\circ}\setminus \ZG(\tilde{G})$. In this case, as $\tilde{s}$ acts on each $\tilde{V}^{i}$ as scalar multiplication by $c^{3-2i}$ and $c^{2}\neq 1$, we determine that $\dim(\tilde{V}_{\tilde{s}}(\tilde{\mu}))\leq \ell\cdot 2^{\ell}-2^{\ell-1}\varepsilon_{p}(2\ell+1)$, where equality holds for $c^{2}=-1$ and $\tilde{\mu}=c^{\pm 1}$. We thus assume that $\dim(\tilde{V}^{i}_{\tilde{s}}(\tilde{\mu}))<\dim(\tilde{V}^{i})$ for all eigenvalues $\tilde{\mu}$ of $\tilde{s}$ on $\tilde{V}$ and all $0\leq i\leq 3$. We write $\tilde{s}=\tilde{z}\cdot \tilde{h}$, where $\tilde{z}\in \ZG(\tilde{L})^{\circ}$ and $\tilde{h}\in [\tilde{L},\tilde{L}]$, and, using the structure of $\tilde{V}\mid_{[\tilde{L},\tilde{L}]}$ and Proposition \ref{PropositionBellomell} we determine that $\scale[0.9]{\dim(\tilde{V}_{\tilde{s}}(\tilde{\mu}))\leq (4+2\varepsilon_{p}(2\ell-1)-2\varepsilon_{p}(2\ell+1))\dim((L_{\tilde{L}}(\tilde{\omega}_{\ell}))_{\tilde{h}}(\tilde{\mu}_{\tilde{h}}))+2\dim((L_{\tilde{L}}(\tilde{\omega}_{2}+\tilde{\omega}_{\ell}))_{\tilde{h}}(\tilde{\mu}_{\tilde{h}}))\leq 2^{\ell}+}$ $\scale[0.9]{2^{\ell-1}(\varepsilon_{p}(2\ell-1)-\varepsilon_{p}(2\ell+1))+2\dim((L_{\tilde{L}}(\tilde{\omega}_{2}+\tilde{\omega}_{\ell}))_{\tilde{h}}(\tilde{\mu}_{\tilde{h}}))}$. Recursively and using Proposition \ref{B3om1+om3} for the base case of $\ell=4$, we determine that $\scale[0.9]{\displaystyle \dim(\tilde{V}_{s}(\tilde{\mu}))\leq \ell\cdot 2^{\ell}-2^{\ell-1}\varepsilon_{p}(2\ell+1)}$ for all eigenvalues $\tilde{\mu}$ of $\tilde{s}$ on $\tilde{V}$. Thus $\scale[0.9]{\displaystyle \max_{\tilde{s}\in \tilde{T}\setminus\ZG(\tilde{G})}\dim(\tilde{V}_{\tilde{s}}(\tilde{\mu}))=\ell\cdot 2^{\ell}-2^{\ell-1}\varepsilon_{p}(2\ell+1)}$.

For the unipotent elements, by Lemma \ref{uniprootelems}, we have $\scale[0.9]{\displaystyle \max_{\tilde{u}\in \tilde{G}_{u}\setminus\{1\}}\dim(\tilde{V}_{\tilde{u}}(1))=\max_{i=\ell-1,\ell}\dim(\tilde{V}_{x_{\tilde{\alpha}_{i}}(1)}(1))}$. Using the structure of $\tilde{V}\mid_{[\tilde{L},\tilde{L}]}$ and Proposition \ref{PropositionBellomell} we determine that $\scale[0.9]{\displaystyle \dim(\tilde{V}_{x_{\tilde{\alpha}_{i}}(1)}(1))\leq 3\cdot 2^{\ell-1}+ 3\cdot 2^{\ell-2}(\varepsilon_{p}(2\ell-1)}$ $\scale[0.9]{\displaystyle-\varepsilon_{p}(2\ell+1))+2\dim((L_{\tilde{L}}(\tilde{\omega}_{2}+\tilde{\omega}_{\ell}))_{x_{\tilde{\alpha}_{i}}(1)}(1))}$, where $i=\ell-1,\ell$. Recursively and using Proposition \ref{B3om1+om3} for the base case of $\ell=4$, we determine that $\scale[0.9]{\displaystyle \max_{\tilde{u}\in \tilde{G}_{u}\setminus \{1\}}\dim(\tilde{V}_{\tilde{u}}(1))\leq (3\ell-2)\cdot 2^{\ell-1}-3\cdot 2^{\ell-2}\varepsilon_{p}(2\ell+1)}$
where equality holds when $\varepsilon_{p}(2(\ell-i)+1)=0$ for all $1\leq i\leq \ell-3$. Consequently, we have $\nu_{\tilde{G}}(\tilde{V})\geq (\ell+2)2^{\ell-1}-2^{\ell-2}\varepsilon_{p}(2\ell+1)$, where equality holds when $\varepsilon_{p}(2(\ell-i)+1)=0$ for all $1\leq i\leq \ell-3$.
\end{proof}

\begin{prop}\label{B5om4}
Let $\ell=5$ and $\tilde{V}=L_{\tilde{G}}(\tilde{\omega}_{4})$. Then $\nu_{\tilde{G}}(\tilde{V})\geq 116$. Moreover, we have $\scale[0.9]{\displaystyle \max_{\tilde{s}\in \tilde{T}\setminus\ZG(\tilde{G})}\dim(\tilde{V}_{\tilde{s}}(\tilde{\mu}))}\leq 214$ and $\scale[0.9]{\displaystyle \max_{\tilde{u}\in \tilde{G}_{u}\setminus \{1\}}\dim(\tilde{V}_{\tilde{u}}(1))}=202$.
\end{prop}

\begin{proof}
Set $\tilde{\lambda}=\tilde{\omega}_{4}$ and $\tilde{L}=\tilde{L}_{1}$. By Lemma \ref{weightlevelBl}, we have $e_{1}(\tilde{\lambda})=2$, therefore $\displaystyle \tilde{V}\mid_{[\tilde{L},\tilde{L}]}=\tilde{V}^{0}\oplus \tilde{V}^{1} \oplus \tilde{V}^{2}$. By \cite[Proposition]{Smith_82} and Lemma \ref{dualitylemma}, we have $\tilde{V}^{0}\cong L_{\tilde{L}}(\tilde{\omega}_{4})$ and $\tilde{V}^{2}\cong L_{\tilde{L}}(\tilde{\omega}_{4})$. Now, in $\tilde{V}^{1}$ the weight $\displaystyle (\tilde{\lambda}-\tilde{\alpha}_{1}-\cdots-\tilde{\alpha}_{4})\mid_{\tilde{T}_{1}}=2\tilde{\omega}_{5}$ admits a maximal vector, thus $\tilde{V}^{1}$ has a composition factor isomorphic to $L_{\tilde{L}}(2\tilde{\omega}_{5})$. Moreover, the weight $\displaystyle (\tilde{\lambda}-\tilde{\alpha}_{1}-\tilde{\alpha}_{2}-\tilde{\alpha}_{3}-2\tilde{\alpha}_{4}-2\tilde{\alpha}_{5})\mid_{\tilde{T}_{1}}=\tilde{\omega}_{3}$ occurs with multiplicity $3$ and is a sub-dominant weight in the composition factor of $\tilde{V}^{1}$ isomorphic to $L_{\tilde{L}}(2\tilde{\omega}_{5})$, in which it has multiplicity $2$. By dimensional considerations and \cite[II.2.14]{Jantzen_2007representations}, we deduce that:
\begin{equation}\label{DecompVB5om4}
\tilde{V}\mid_{[L,\tilde{L}]}\cong L_{\tilde{L}}(\tilde{\omega}_{4}) \oplus L_{\tilde{L}}(2\tilde{\omega}_{5}) \oplus L_{\tilde{L}}(\tilde{\omega}_{3}) \oplus L_{\tilde{L}}(\tilde{\omega}_{4}).
\end{equation}

We start with the semisimple elements. Let $\tilde{s}\in \tilde{T}\setminus\ZG(\tilde{G})$. If $\dim(\tilde{V}^{i}_{\tilde{s}}(\tilde{\mu}))=\dim(\tilde{V}^{i})$ for some eigenvalue $\tilde{\mu}$ of $\tilde{s}$ on $\tilde{V}$, where $0\leq i\leq 2$, then $\tilde{s}\in \ZG(\tilde{L})^{\circ}\setminus \ZG(\tilde{G})$. In this case, as $\tilde{s}$ acts on each $\tilde{V}^{i}$ as scalar multiplication by $c^{2-2i}$ and $c^{2}\neq 1$, we determine that $\dim(\tilde{V}_{\tilde{s}}(\tilde{\mu}))\leq 168$ for all eigenvalues $\tilde{\mu}$ of $\tilde{s}$ on $\tilde{V}$. We thus assume that $\dim(\tilde{V}^{i}_{\tilde{s}}(\tilde{\mu}))<\dim(\tilde{V}^{i})$ for all eigenvalues $\tilde{\mu}$ of $\tilde{s}$ on $\tilde{V}$ and all $0\leq i\leq 2$. We write $\tilde{s}=\tilde{z}\cdot \tilde{h}$, where $\tilde{z}\in \ZG(\tilde{L})^{\circ}$ and $\tilde{h}\in [\tilde{L},\tilde{L}]$, and, using \eqref{DecompVB5om4} and Propositions \ref{PropositionBlwedge}, \ref{PropositionBlwedgecube} and \ref{Bl2oml}, we get $\dim(\tilde{V}_{\tilde{s}}(\tilde{\mu}))\leq 2\dim((L_{\tilde{L}}(\tilde{\omega}_{4}))_{\tilde{h}}(\tilde{\mu}_{\tilde{h}}))+\dim((L_{\tilde{L}}(\tilde{\omega}_{3}))_{\tilde{h}}(\tilde{\mu}_{\tilde{h}}))+\dim((L_{\tilde{L}}(2\tilde{\omega}_{5}))_{\tilde{h}}(\tilde{\mu}_{\tilde{h}}))\leq 214$ for all eigenvalues $\tilde{\mu}$ of $\tilde{s}$ on $\tilde{V}$. Therefore, $\scale[0.9]{\displaystyle \max_{\tilde{s}\in \tilde{T}\setminus\ZG(\tilde{G})}\dim(\tilde{V}_{\tilde{s}}(\tilde{\mu}))}\leq 214$.

We now focus on the unipotent elements. By Lemma \ref{uniprootelems}, we have $\scale[0.9]{\displaystyle \max_{\tilde{u}\in \tilde{G}_{u}\setminus \{1\}}\dim(\tilde{V}_{\tilde{u}}(1))=\max_{i=4,5}\dim(\tilde{V}_{x_{\tilde{\alpha}_{i}}(1)}(1))}$. Using \eqref{DecompVB5om4} and Propositions \ref{PropositionBlwedge}, \ref{PropositionBlwedgecube} and \ref{Bl2oml}, we get $\scale[0.9]{\displaystyle\max_{\tilde{u}\in \tilde{G}_{u}\setminus \{1\}}\dim(\tilde{V}_{\tilde{u}}(1))=202}$, and, as $\scale[0.9]{\displaystyle \max_{\tilde{s}\in \tilde{T}\setminus\ZG(\tilde{G})}\dim(\tilde{V}_{\tilde{s}}(\tilde{\mu}))}\leq 214$, we get $\nu_{\tilde{G}}(\tilde{V})\geq 116$.
\end{proof}

\begin{prop}\label{Blom4}
Let $\ell=6,7$ and $\tilde{V}=L_{\tilde{G}}(\tilde{\omega}_{4})$. Then $\nu_{\tilde{G}}(\tilde{V})\geq 360-144\varepsilon_{\ell}(6)$. Moreover
\begin{enumerate}
\item[\emph{$(1)$}] for $\ell=6$, we have $\scale[0.9]{\displaystyle \max_{\tilde{s}\in \tilde{T}\setminus\ZG(\tilde{G})}\dim(\tilde{V}_{\tilde{s}}(\tilde{\mu}))}\leq 499$ and $\scale[0.9]{\displaystyle \max_{\tilde{u}\in \tilde{G}_{u}\setminus \{1\}}\dim(\tilde{V}_{\tilde{u}}(1))}=453$.
\item[\emph{$(2)$}] for $\ell=7$, we have $\scale[0.9]{\displaystyle \max_{\tilde{s}\in \tilde{T}\setminus\ZG(\tilde{G})}\dim(\tilde{V}_{\tilde{s}}(\tilde{\mu}))}\leq 1005$ and $\scale[0.9]{\displaystyle \max_{\tilde{u}\in \tilde{G}_{u}\setminus \{1\}}\dim(\tilde{V}_{\tilde{u}}(1))}=897$.
\end{enumerate}
\end{prop}

\begin{proof}
Set $\tilde{\lambda}=\tilde{\omega}_{4}$ and $\tilde{L}=\tilde{L}_{1}$. By Lemma \ref{weightlevelBl}, we have $e_{1}(\tilde{\lambda})=2$, therefore $\displaystyle \tilde{V}\mid_{[\tilde{L},\tilde{L}]}=\tilde{V}^{0}\oplus \tilde{V}^{1} \oplus \tilde{V}^{2}$. By \cite[Proposition]{Smith_82} and Lemma \ref{dualitylemma}, we have $\tilde{V}^{0}\cong L_{\tilde{L}}(\tilde{\omega}_{4})$ and $\tilde{V}^{2}\cong L_{\tilde{L}}(\tilde{\omega}_{4})$. Now, in $\tilde{V}^{1}$ the weight $\displaystyle (\tilde{\lambda}-\tilde{\alpha}_{1}-\cdots-\tilde{\alpha}_{4})\mid_{\tilde{T}_{1}}=\tilde{\omega}_{5}$ admits a maximal vector, thus $\tilde{V}^{1}$ has a composition factor isomorphic to $L_{\tilde{L}}(\tilde{\omega}_{5})$. Moreover, the weight $\displaystyle (\tilde{\lambda}-\tilde{\alpha}_{1}-\tilde{\alpha}_{2}-\tilde{\alpha}_{3}-2\tilde{\alpha}_{4}-\cdots-2\tilde{\alpha}_{\ell})\mid_{\tilde{T}_{1}}=\tilde{\omega}_{3}$ occurs with multiplicity $\ell-2$ and is a sub-dominant weight in the composition factor of $\tilde{V}^{1}$ isomorphic to $L_{\tilde{L}}(\tilde{\omega}_{5})$, in which it has multiplicity $\ell-3$. By dimensional considerations and \cite[II.2.14]{Jantzen_2007representations}, we deduce that:
\begin{equation}\label{DecompVBlom4}
\tilde{V}\mid_{[\tilde{L},\tilde{L}]}\cong L_{\tilde{L}}(\tilde{\omega}_{4}) \oplus L_{\tilde{L}}(\tilde{\omega}_{5}) \oplus L_{\tilde{L}}(\tilde{\omega}_{3}) \oplus L_{\tilde{L}}(\tilde{\omega}_{4}).
\end{equation}

We start with the semisimple elements. Let $\tilde{s}\in \tilde{T}\setminus\ZG(\tilde{G})$. If $\dim(\tilde{V}^{i}_{\tilde{s}}(\tilde{\mu}))=\dim(\tilde{V}^{i})$ for some eigenvalue $\tilde{\mu}$ of $\tilde{s}$ on $\tilde{V}$, where $0\leq i\leq 2$, then $\tilde{s}\in \ZG(\tilde{L})^{\circ}\setminus \ZG(\tilde{G})$. In this case, as $\tilde{s}$ acts on each $\tilde{V}^{i}$ as scalar multiplication by $c^{2-2i}$ and $c^{2}\neq 1$, we determine that $\dim(\tilde{V}_{\tilde{s}}(\tilde{\mu}))\leq 572-242\varepsilon_{\ell}(6)$ for all eigenvalues $\tilde{\mu}$ of $\tilde{s}$ on $\tilde{V}$. We thus assume that $\dim(\tilde{V}^{i}_{\tilde{s}}(\tilde{\mu}))<\dim(\tilde{V}^{i})$ for all eigenvalues $\tilde{\mu}$ of $\tilde{s}$ on $\tilde{V}$ and all $0\leq i\leq 2$. We write $\tilde{s}=\tilde{z}\cdot \tilde{h}$, where $\tilde{z}\in \ZG(\tilde{L})^{\circ}$ and $\tilde{h}\in [\tilde{L},\tilde{L}]$. For $\ell=6$, using \eqref{DecompVBlom4} and Propositions \ref{PropositionBlwedge}, \ref{PropositionBlwedgecube} and \ref{B5om4}, we get $\dim(\tilde{V}_{\tilde{s}}(\tilde{\mu}))\leq 2\dim((L_{\tilde{L}}(\tilde{\omega}_{4}))_{\tilde{h}}(\tilde{\mu}_{\tilde{h}}))+\dim((L_{\tilde{L}}(\tilde{\omega}_{3}))_{\tilde{h}}(\tilde{\mu}_{\tilde{h}}))+\dim((L_{\tilde{L}}(\tilde{\omega}_{5}))_{\tilde{h}}(\tilde{\mu}_{\tilde{h}}))\leq 499$ for all eigenvalues $\tilde{\mu}$ of $\tilde{s}$ on $\tilde{V}$. Therefore, $\scale[0.9]{\displaystyle \max_{\tilde{s}\in \tilde{T}\setminus\ZG(\tilde{G})}\dim(\tilde{V}_{\tilde{s}}(\tilde{\mu}))}\leq 499$. Similarly, for $\ell=7$, using \eqref{DecompVBlom4}, the previous result and Propositions \ref{PropositionBlwedge} and \ref{PropositionBlwedgecube}, we get $\dim(\tilde{V}_{\tilde{s}}(\tilde{\mu}))\leq 1005$ for all eigenvalues $\tilde{\mu}$ of $\tilde{s}$ on $\tilde{V}$, and so $\scale[0.9]{\displaystyle \max_{\tilde{s}\in \tilde{T}\setminus\ZG(\tilde{G})}\dim(\tilde{V}_{\tilde{s}}(\tilde{\mu}))}\leq 1005$.

For the unipotent elements, by Lemma \ref{uniprootelems}, we have $\scale[0.9]{\displaystyle \max_{\tilde{u}\in \tilde{G}_{u}\setminus\{1\}}\dim(\tilde{V}_{\tilde{u}}(1))=\max_{i=\ell-1,\ell}\dim(\tilde{V}_{x_{\tilde{\alpha}_{i}}(1)}(1))}$. For $\ell=6$, using \eqref{DecompVBlom4} and Propositions \ref{PropositionBlwedge}, \ref{PropositionBlwedgecube} and \ref{B5om4} we get $\scale[0.9]{\displaystyle\max_{\tilde{u}\in \tilde{G}_{u}\setminus \{1\}}\dim(\tilde{V}_{\tilde{u}}(1))=453}.$ Similarly, for $\ell=7$, using \eqref{DecompVBlom4}, the previous result and Propositions \ref{PropositionBlwedge} and \ref{PropositionBlwedgecube}, we get $\scale[0.9]{\displaystyle\max_{\tilde{u}\in \tilde{G}_{u}\setminus \{1\}}\dim(\tilde{V}_{\tilde{u}}(1))=897}$. In conclusion, we have shown that $\nu_{\tilde{G}}(\tilde{V})\geq 216$ for $\ell=6$, respectively $\nu_{\tilde{G}}(V)\geq 360$ for $\ell=7$.
\end{proof}

\section{Proof of Theorem \ref{ResultsDl}}

In this section $W$ is a $2\ell$-dimensional $k$-vector space equipped with a nondegenerate quadratic form $Q$ and $G=\SO(W,Q)$. Note that $G$ is a simple algebraic group of type $D_{\ell}$, $\ell\geq 4$. We let $\tilde{G}$ be a simple simply connected cover of $G$, and we fix a central isogeny $\phi:\tilde{G}\to G$ with $\ker(\phi)\subseteq \ZG(\tilde{G})$ and $d\phi\neq 0$. As in Subsection \ref{ResultsBl}, we will denote by $\tilde{X}$ the object in $\tilde{G}$ corresponding to the object $X$ in $G$ under $\phi$. For example, $\tilde{T}$ is a maximal torus in $\tilde{G}$ with $\phi(\tilde{T})=T$. Now, let $a$ be a nondegenerate alternating bilinear form on $W$ and let $H=\Sp(W,a)$. Note that $H$ is a simple simply connected linear algebraic group of type $C_{\ell}$. Let $T_{H}$ be a maximal torus in $H$. Note that we can choose a symplectic basis and an orthogonal basis in $W$ such that $T=T_{H}$. Thus, any $s\in T$ is conjugate in $H$ to an element $s_{H}\in T_{H}$ of the form $s_{H}=\diag(\mu_{1}\cdot \I_{n_{1}},\mu_{2}\cdot \I_{n_{2}},\dots,\mu_{m}\cdot \I_{n_{m}},\mu_{m}^{-1}\cdot \I_{n_{m}},\dots, \mu_{2}^{-1}\cdot \I_{n_{2}}, \mu_{1}^{-1}\cdot \I_{n_{1}})$ with $\mu_{i}\neq \mu_{j}^{\pm 1}$ for all $i<j$, $\scale[0.9]{\displaystyle \sum_{i=1}^{m}n_{i}=\ell}$ and $\ell\geq n_{1}\geq \cdots \geq n_{m}\geq 1$. Lastly, we denote by $\omega_{1}^{H}, \dots, \omega_{\ell}^{H}$ the fundamental dominant weights of $H$.

\begin{prop}\label{PropositionDlnatural}
Let $\tilde{V}=L_{\tilde{G}}(\tilde{\omega}_{1})$. Then $\nu_{\tilde{G}}(\tilde{V})=2$.  Moreover, we have $\scale[0.9]{\displaystyle \max_{\tilde{u}\in \tilde{G}_{u}\setminus \{1\}}\dim(\tilde{V}_{\tilde{u}}(1))}=2\ell-2$ and $\scale[0.9]{\displaystyle \max_{\tilde{s}\in \tilde{T}\setminus\ZG(\tilde{G})}\dim(\tilde{V}_{\tilde{s}}(\tilde{\mu}))}=2\ell-2$.
\end{prop}

\begin{proof}
First, we note that $\tilde{V}\cong L_{G}(\omega_{1})$ as $k\tilde{G}$-modules. Further, we also have that $L_{G}(\omega_{1})\cong W$ as $kG$-modules. For the semisimple elements, the result follows by Proposition \ref{PropositionClnatural} and Lemma \ref{LtildeGtildelambdaandLGlambda}. For the unipotent elements, by Lemmas \ref{LtildeGtildelambdaandLGlambda} and \ref{uniprootelems}, we have $\scale[0.9]{\displaystyle \max_{\tilde{u}\in \tilde{G}_{u}\setminus \{1\}}\dim(\tilde{V}_{\tilde{u}}(1))=2\ell-2}$, and so $\nu_{\tilde{G}}(\tilde{V})=2$.
\end{proof}

\begin{lem}\label{DecompwedgeSOp=2}
Let $p=2$. If $\varepsilon_{2}(\ell)=0$, then $\wedge^{2}(W)\cong L_{G}(\omega_{2})\oplus L_{G}(0)$, as $kG$-modules. If $\varepsilon_{2}(\ell)=1$, then the $kG$-module $\wedge^{2}(W)$ has three composition factors one isomorphic to $L_{G}(\omega_{2})$ and two isomorphic to $L_{G}(0)$.
\end{lem}

\begin{proof}
By \cite[1.15]{seitz1987maximal}, the $kG$-module $\wedge^{2}(W)$ admits a unique nontrivial composition factor of highest weight $\omega_{2}$. Since $\dim(L_{G}(\omega_{2}))=2\ell^{2}-\ell-\gcd(2,\ell)$, see \cite[Table 2]{Lubeck_2001}, we determine that, if $\varepsilon_{2}(\ell)=0$, then $\wedge^{2}(W)$ has two composition factors: one isomorphic to $L_{G}(\omega_{2})$ and one isomorphic to $L_{G}(0)$, and $\wedge^{2}(W)\cong L_{G}(\omega_{2})\oplus L_{G}(0)$, by \cite[II.$2.14$]{Jantzen_2007representations}; while, if $\varepsilon_{2}(\ell)=1$, then $\wedge^{2}(W)$ has three composition factors: one isomorphic to $L_{G}(\omega_{2})$ and two isomorphic to $L_{G}(0)$. 
\end{proof}

\begin{prop}\label{PropositionDlwedge}
Let $\tilde{V}=L_{\tilde{G}}(\tilde{\omega}_{2})$. Then $\nu_{\tilde{G}}(\tilde{V})=4\ell-6$.  Moreover, we have $\scale[0.9]{\displaystyle \max_{\tilde{u}\in \tilde{G}_{u}\setminus \{1\}}\dim(\tilde{V}_{\tilde{u}}(1))=2\ell^{2}-5\ell+6}$ $-(1+\varepsilon_{2}(\ell))\varepsilon_{p}(2)$ and $\scale[0.9]{\displaystyle \max_{\tilde{s}\in \tilde{T}\setminus\ZG(\tilde{G})}\dim(\tilde{V}_{\tilde{s}}(\tilde{\mu}))}=2\ell^{2}-5\ell+4-(1+\varepsilon_{2}(\ell))\varepsilon_{p}(2)$.
\end{prop}

\begin{proof}
First, we note that $\tilde{V}\cong L_{G}(\omega_{2})$ as $k\tilde{G}$-modules. To ease notation, let $V=L_{G}(\omega_{2})$. Now, if $p\neq 2$, then $V\cong \wedge^{2}(W)$ as $kG$-modules, see \cite[Proposition $4.2.2$]{mcninch_1998}, while if $p=2$, the structure of $\wedge^{2}(W)$ as a $kG$-module is as in Lemma \ref{DecompwedgeSOp=2}. 

We start with the unipotent elements. By Lemmas \ref{LtildeGtildelambdaandLGlambda} and \ref{uniprootelems}, we have $\scale[0.88]{\displaystyle \max_{\tilde{u}\in \tilde{G}_{u}\setminus \{1\}}\dim(\tilde{V}_{\tilde{u}}(1))=\dim(V_{x_{\alpha_{1}}(1)}(1))}$. Further, by the proof of Proposition \ref{PropositionClwedge}, we have $\dim((\wedge^{2}(W))_{x_{\alpha_{1}}(1)}(1))=2\ell^{2}-5\ell+6$. Now, if $p\neq 2$, then $\scale[0.9]{\displaystyle \max_{\tilde{u}\in \tilde{G}_{u}\setminus \{1\}}\dim(\tilde{V}_{\tilde{u}}(1))=2\ell^{2}-5\ell+6}$, while, if $p=2$ and $\varepsilon_{2}(\ell)=0$, then $\scale[0.9]{\displaystyle \max_{\tilde{u}\in \tilde{G}_{u}\setminus \{1\}}\dim(\tilde{V}_{\tilde{u}}(1))=2\ell^{2}-5\ell+5}$. Lastly, if $p=2$ and $\varepsilon_{2}(\ell)=1$, we use \cite[Theorem B]{Korhonen_2020HesselinkNF}, to determine that $\scale[0.9]{\displaystyle \max_{\tilde{u}\in \tilde{G}_{u}\setminus \{1\}}\dim(\tilde{V}_{\tilde{u}}(1))=2\ell^{2}-5\ell+4}$. We conclude that $\scale[0.9]{\displaystyle \max_{\tilde{u}\in \tilde{G}_{u}\setminus \{1\}}\dim(\tilde{V}_{\tilde{u}}(1))}=2\ell^{2}-5\ell+6-(1+\varepsilon_{2}(\ell))\varepsilon_{p}(2)$.

We now focus on the semisimple elements. Let $s\in T\setminus \ZG(G)$ and note that, in particular, we have $s\in T_{H}\setminus \ZG(H)$. Let $\mu$ be an eigenvalue of $s$ on $V$. Now, when $p\neq 2$, as $V\cong \wedge^{2}(W)$, we have $\dim(V_{s}(\mu))=\dim((\wedge^{2}(W))_{s}(\mu))$. Further, by \cite[Lemma $4.8.2$]{mcninch_1998} which gives the structure of $\wedge^{2}(W)$ as a $kH$-module, we have $\dim(V_{s}(\mu))=\dim((L_{H}(\omega_{2}^{H}))_{s}(\mu))$, for $\mu\neq 1$, and $\dim(V_{s}(1))=\dim((L_{H}(\omega_{2}^{H}))_{s}(1))+1+\varepsilon_{p}(\ell)$; and we use Proposition \ref{PropositionClwedge} to get the result. Similarly, when $p=2$, by the structure of $\wedge^{2}(W)$ as a $kG$-module, we determine that $\dim(V_{s}(\mu))=\dim((\wedge^{2}(W))_{s}(\mu))$, for $\mu\neq 1$, and $\dim(V_{s}(1))=\dim((\wedge^{2}(W))_{s}(1))-1-\varepsilon_{2}(\ell)$. Once more, as in the $p\neq 2$ case, using the structure of $\wedge^{2}(W)$ as a $kH$-module, we deduce that $\dim(V_{s}(\mu))=\dim((L_{H}(\omega_{2}^{H}))_{s}(\mu))$ for all $\mu$; and the result follows by Proposition \ref{PropositionClwedge}. Thus, in view of Lemma \ref{LtildeGtildelambdaandLGlambda}, we have shown that $\scale[0.9]{\displaystyle \max_{\tilde{s}\in \tilde{T}\setminus\ZG(\tilde{G})}\dim(\tilde{V}_{\tilde{s}}(\tilde{\mu}))}=2\ell^{2}-5\ell+4-(1+\varepsilon_{2}(\ell))\varepsilon_{p}(2)$, and, as $\scale[0.9]{\displaystyle \max_{\tilde{u}\in \tilde{G}_{u}\setminus \{1\}}\dim(\tilde{V}_{\tilde{u}}(1))}=2\ell^{2}-5\ell+6-(1+\varepsilon_{2}(\ell))\varepsilon_{p}(2)$, it follows that $\nu_{\tilde{G}}(\tilde{V})=4\ell-6$.
\end{proof}

\begin{prop}\label{PropositionDlsymm}
Let $p\neq 2$ and $\tilde{V}=L_{\tilde{G}}(2\tilde{\omega}_{1})$. Then $\nu_{\tilde{G}}(\tilde{V})=4\ell-4$.  Moreover, we have $\scale[0.9]{\displaystyle \max_{\tilde{u}\in \tilde{G}_{u}\setminus \{1\}}\dim(\tilde{V}_{\tilde{u}}(1))}$ $=2\ell^{2}-3\ell+1-\varepsilon_{p}(\ell)$ and $\scale[0.9]{\displaystyle \max_{\tilde{s}\in \tilde{T}\setminus\ZG(\tilde{G})}\dim(\tilde{V}_{\tilde{s}}(\tilde{\mu}))}=2\ell^{2}-3\ell+3-\varepsilon_{p}(\ell)$.
\end{prop}

\begin{proof}
First, we note that $\tilde{V}\cong L_{G}(2\omega_{1})$ as $k\tilde{G}$-modules. To ease notation, let $V=L_{G}(2\omega_{1})$. Further, if $\varepsilon_{p}(\ell)=0$, then $\SW(W)=V\oplus L_{G}(0)$, while, if $\varepsilon_{p}(\ell)=1$, then $\SW(W)=L_{G}(0)\mid V\mid L_{G}(0)$, see \cite[Propositions $4.7.3$]{mcninch_1998}. 

We start with the unipotent elements. By Lemmas \ref{LtildeGtildelambdaandLGlambda} and \ref{uniprootelems}, we have $\scale[0.9]{\displaystyle \max_{\tilde{u}\in \tilde{G}_{u}\setminus \{1\}}\dim(\tilde{V}_{\tilde{u}}(1))=\dim(V_{x_{\alpha_{1}}(1)}(1))}$ and $\dim(V_{x_{\alpha_{1}}(1)}(1))=2\ell^{2}-3\ell+2$, by the proof of Proposition \ref{PropositionClsymm}. Now, using the structure of $\SW(W)$ as a $kG$-module and \cite[Corollary $6.2$]{Korhonen_2019}, we determine that $\scale[0.9]{\displaystyle \max_{\tilde{u}\in \tilde{G}_{u}\setminus \{1\}}\dim(\tilde{V}_{\tilde{u}}(1))}=2\ell^{2}-3\ell+1-\varepsilon_{p}(\ell)$.

We focus on the semisimple elements. Let $s\in T\setminus \ZG(G)$ and note that, in particular, $s\in T_{H}\setminus \ZG(H)$. Let $\mu$ be an eigenvalue of $s$ on $V$. Now,  by \cite[Proposition $4.2.2$]{mcninch_1998}, we have $\SW(W)\cong L_{H}(2\omega_{1}^{H})$, as $kH$-module, and so, using the structure of $\SW(W)$ as a $kG$-module, we deduce that $\dim(V_{s}(\mu))=\dim((L_{H}(2\omega_{1}^{H}))_{s}(\mu))$, for $\mu\neq 1$, and $\dim(V_{s}(1))=\dim((L_{H}(2\omega_{1}^{H}))_{s}(1))-1-\varepsilon_{p}(\ell)$. For $\mu=1$, we have $\dim(V_{s}(1))\leq 2\ell^{2}-3\ell+3-\varepsilon_{p}(\ell)$, by Proposition \ref{PropositionClsymm}. For $\mu$ such that $\mu\neq \mu^{-1}$ we have $\dim(V_{s}(\mu))\leq \frac{\dim(V)}{2}<2\ell^{2}-3\ell+3-\varepsilon_{p}(\ell)$, as $V$ is self-dual. Lastly, we let $\mu=-1$. As $s\in T_{H}\setminus \ZG(H)$, we write $s=\diag(\mu_{1}\cdot \I_{n_{1}},\dots,\mu_{m}\cdot \I_{n_{m}},\mu_{m}^{-1}\cdot \I_{n_{m}},\dots, \mu_{1}^{-1}\cdot \I_{n_{1}})$, where $\mu_{i}\neq \mu_{j}^{\pm 1}$ for all $i<j$,  $\scale[0.9]{\displaystyle \sum_{i=1}^{m}n_{i}=\ell}$ and $n_{1}\geq n_{2}\geq \cdots \geq n_{m}\geq 1$.  If $\mu_{i}\mu_{j}\neq -1$ for all $ i<j$, then, by inequality \eqref{eigen-1case1symminClforDl}, we have $\dim(V_{s}(-1))\leq \ell^{2}+\ell<2\ell^{2}-3\ell+3-\varepsilon_{p}(\ell)$. If there exists $i<j$ such that $\mu_{i}\mu_{j}=-1$, then, by inequality \eqref{eigen-1dimspacesymminClforDl}, we have $\scale[0.9]{\displaystyle \dim(V_{s}(-1))\leq 2\ell^{2}+\ell-\sum_{r=1}^{m}n_{r}^{2}-n_{i}(n_{i}+1)-n_{j}(n_{j}+1)-2(n_{i}+n_{j})(\ell-n_{i}-n_{j})}$. Assume that $\dim(V_{s}(-1))\geq 2\ell^{2}-3\ell+3-\varepsilon_{p}(\ell)$. Then $\scale[0.9]{\displaystyle \ell(4-n_{i}-n_{j})-3+\varepsilon_{p}(\ell)-\sum_{r\neq i,j}n_{r}^{2}-(n_{i}-n_{j})^{2}-(n_{i}+n_{j})(\ell+1-n_{i}-n_{j})\geq 0}$ and, as $\ell+1>n_{i}+n_{j}$, we must have $(n_{i},n_{j})\in \{(1,1),(2,1)\}$. If $(n_{i},n_{j})=(1,1)$, then $\scale[0.9]{\displaystyle -1+\varepsilon_{p}(\ell)-\sum_{r\neq i,j}n_{r}^{2}\geq 0}$, which does not hold as $\scale[0.9]{\displaystyle \sum_{r\neq i,j}n_{r}=\ell-2\geq 2}$. Similarly, if $(n_{i},n_{j})=(2,1)$, then $\scale[0.9]{\displaystyle -2\ell+2+\varepsilon_{p}(\ell)-\sum_{r\neq i,j}n_{r}^{2}\geq 0}$, which does not hold. Thus, $\dim(V_{s}(-1))<2\ell^{2}-3\ell+3-\varepsilon_{p}(\ell)$ for all $s\in T\setminus \ZG(G)$, therefore $\scale[0.9]{\displaystyle \max_{\tilde{s}\in \tilde{T}\setminus\ZG(\tilde{G})}\dim(\tilde{V}_{\tilde{s}}(\tilde{\mu}))}=2\ell^{2}-3\ell+3-\varepsilon_{p}(\ell)$. As $\scale[0.9]{\displaystyle \max_{\tilde{u}\in \tilde{G}_{u}\setminus \{1\}}\dim(\tilde{V}_{\tilde{u}}(1))}=2\ell^{2}-3\ell+1-\varepsilon_{p}(\ell)$, we have $\nu_{\tilde{G}}(\tilde{V})=4\ell-4$.
\end{proof}

In order to calculate $\nu_{\tilde{G}}(\tilde{V})$ for $\tilde{V}=L_{\tilde{G}}(\tilde{\omega}_{3})$, we need the following four auxiliary results: Propositions \ref{D4om3+om4p=2} and \ref{D4om3+om4pneq2} which treat the $k\tilde{G}$-module $L_{\tilde{G}}(\tilde{\omega}_{3}+\tilde{\omega}_{4})$, and Propositions \ref{D5om3p=2} and \ref{D5om3pneq2} for the $k\tilde{G}$-module $L_{\tilde{G}}(\tilde{\omega}_{3})$.

\begin{prop}\label{D4om3+om4p=2}
Let $p=2$, $\ell=4$ and $\tilde{V}=L_{\tilde{G}}(\tilde{\omega}_{3}+\tilde{\omega}_{4})$. Then $\nu_{\tilde{G}}(\tilde{V})=20$.  Moreover, we have $\scale[0.9]{\displaystyle \max_{\tilde{u}\in \tilde{G}_{u}\setminus \{1\}}\dim(\tilde{V}_{\tilde{u}}(1))}=28$ and $\scale[0.9]{\displaystyle \max_{\tilde{s}\in \tilde{T}\setminus\ZG(\tilde{G})}\dim(\tilde{V}_{\tilde{s}}(\tilde{\mu}))}=20$.
\end{prop}

\begin{proof}
First, we note that $\tilde{V}\cong L_{G}(\omega_{3}+\omega_{4})$ as $k\tilde{G}$-modules. Secondly, as $p=2$, by \cite[Table 1]{seitz1987maximal}, we have $L_{H}(\omega^{H}_{3})\cong L_{G}(\omega_{3}+\omega_{4})$, as $kG$-modules. Lastly, by \cite[Lemma 4.8.2]{mcninch_1998}, we have $\wedge^{3}(W) \cong L_{H}(\omega^{H}_{3})\oplus W \cong L_{G}(\omega_{3}+\omega_{4}) \oplus W$, as $kG$-modules.

For the unipotent elements, by Lemmas \ref{LtildeGtildelambdaandLGlambda} and \ref{uniprootelems}, we have $\scale[0.9]{\displaystyle \max_{\tilde{u}\in \tilde{G}_{u}\setminus \{1\}}\dim(\tilde{V}_{\tilde{u}}(1))=\dim(\tilde{V}_{x_{\tilde{\alpha}_{1}}(1)}(1))}=$ $\scale[0.9]{\displaystyle\dim(L_{G}(\omega_{3}+\omega_{4})_{x_{\alpha_{1}}(1)}(1))=\dim((\wedge^{3}(W))_{x_{\alpha_{1}}(1)}(1))-\dim(W_{x_{\alpha_{1}}(1)}(1))}=28$, see proof of Proposition \ref{C4om3}. In the case of the semisimple elements, by Lemma \ref{LtildeGtildelambdaandLGlambda} and Proposition \ref{C4om3}, and the arguments of the first paragraph, it follows that $\scale[0.9]{\displaystyle \max_{\tilde{s}\in \tilde{T}\setminus\ZG(\tilde{G})}\dim(\tilde{V}_{\tilde{s}}(\tilde{\mu}))=\max_{s_{H}\in T_{H}\setminus\ZG(H)}\dim((L_{H}(\omega_{3}^{H}))_{s_{H}}(\mu_{H}))}=20$. It follows that $\nu_{\tilde{G}}(V)=20$.
\end{proof}

\begin{prop}\label{D5om3p=2}
Let $p=2$, $\ell=5$ and $\tilde{V}=L_{\tilde{G}}(\tilde{\omega}_{3})$. Then $\nu_{\tilde{G}}(\tilde{V})=40$.  Moreover, we have $\scale[0.9]{\displaystyle \max_{\tilde{u}\in \tilde{G}_{u}\setminus \{1\}}\dim(\tilde{V}_{\tilde{u}}(1))}$ $=60$ and $\scale[0.9]{\displaystyle \max_{\tilde{s}\in \tilde{T}\setminus\ZG(\tilde{G})}\dim(\tilde{V}_{\tilde{s}}(\tilde{\mu}))}\leq 48$.
\end{prop}

\begin{proof}
We begin with the semisimple elements. First, we note that $\tilde{V}\cong L_{G}(\omega_{3})$ as $k\tilde{G}$-modules. Secondly, as $p=2$, by \cite[Table 1]{seitz1987maximal}, we have $L_{H}(\omega^{H}_{3})\cong L_{G}(\omega_{3})$, as $kG$-modules. Thus, in view of Lemma \ref{LtildeGtildelambdaandLGlambda}, we have $\scale[0.9]{\displaystyle \max_{\tilde{s}\in \tilde{T}\setminus\ZG(\tilde{G})}\dim(\tilde{V}_{\tilde{s}}(\tilde{\mu}))=\max_{s_{H}\in T_{H}\setminus \ZG(H)}\dim((L_{H}(\omega^{H}_{3}))_{s_{H}}(\mu_{H}))}\leq 48$, by Proposition \ref{PropositionClwedgecube}. 

We now focus on the unipotent elements. Set $\tilde{\lambda}=\tilde{\omega}_{3}$ and $\tilde{L}=\tilde{L}_{1}$. By Lemma \ref{weightlevelDl}, we determine that $e_{1}(\tilde{\lambda})=2$, therefore $\displaystyle \tilde{V}\mid_{[\tilde{L},\tilde{L}]}=\tilde{V}^{0}\oplus \tilde{V}^{1}\oplus \tilde{V}^{2}$. By \cite[Proposition]{Smith_82} and Lemma \ref{dualitylemma}, we have $\tilde{V}^{0}\cong L_{\tilde{L}}(\tilde{\omega}_{3})$ and $\tilde{V}^{2}\cong L_{\tilde{L}}(\tilde{\omega}_{3})$. The weight $\displaystyle (\tilde{\lambda}-\tilde{\alpha}_{1}-\tilde{\alpha}_{2}-\tilde{\alpha}_{3})\mid_{\tilde{T}_{1}}=\tilde{\omega}_{4}+\tilde{\omega}_{5}$ admits a maximal vector in $\tilde{V}^{1}$, thus $\tilde{V}^{1}$ has a composition factor isomorphic to $L_{\tilde{L}}(\tilde{\omega}_{4}+\tilde{\omega}_{5})$. By dimensional considerations, we determine that:
\begin{equation}\label{DecompVD5om3p=2}
\tilde{V}\mid_{[\tilde{L},\tilde{L}]}\cong L_{\tilde{L}}(\tilde{\omega}_{3}) \oplus L_{\tilde{L}}(\tilde{\omega}_{4}+\tilde{\omega}_{5}) \oplus L_{\tilde{L}}(\tilde{\omega}_{3}).
\end{equation}

Now,in view of Lemma \ref{uniprootelems}, we have $\scale[0.9]{\displaystyle \max_{\tilde{u}\in \tilde{G}_{u}\setminus \{1\}}\dim(\tilde{V}_{\tilde{u}}(1))=\dim(\tilde{V}_{x_{\tilde{\alpha}_{1}}(1)}(1))}$ and, by \eqref{DecompVD5om3p=2} and Propositions  \ref{PropositionDlwedge} and \ref{D4om3+om4p=2}, we get $\scale[0.9]{\displaystyle \max_{\tilde{u}\in \tilde{G}_{u}\setminus \{1\}}\dim(\tilde{V}_{\tilde{u}}(1))}= 60$. Lastly, we note that $\nu_{\tilde{G}}(\tilde{V})=40$.
\end{proof}

\begin{prop}\label{D4om3+om4pneq2}
Let $p\neq 2$, $\ell=4$ and $\tilde{V}=L_{\tilde{G}}(\tilde{\omega}_{3}+\tilde{\omega}_{4})$. Then $\nu_{\tilde{G}}(\tilde{V})=22$.  Moreover, we have $\scale[0.9]{\displaystyle \max_{\tilde{u}\in \tilde{G}_{u}\setminus \{1\}}\dim(\tilde{V}_{\tilde{u}}(1))}=34$ and $\scale[0.9]{\displaystyle \max_{\tilde{s}\in \tilde{T}\setminus\ZG(\tilde{G})}\dim(\tilde{V}_{\tilde{s}}(\tilde{\mu}))}\leq 34$.
\end{prop}

\begin{proof}
Set $\tilde{\lambda}=\tilde{\omega}_{3}+\tilde{\omega}_{4}$ and $\tilde{L}=\tilde{L}_{1}$. By Lemma \ref{weightlevelDl}, we have $e_{1}(\tilde{\lambda})=2$, therefore $\displaystyle \tilde{V}\mid _{[\tilde{L},\tilde{L}]}=\tilde{V}^{0}\oplus \tilde{V}^{1} \oplus \tilde{V}^{2}$. By \cite[Proposition]{Smith_82} and Lemma \ref{dualitylemma}, we have $\tilde{V}^{0}\cong L_{\tilde{L}}(\tilde{\omega}_{3}+\tilde{\omega}_{4})$ and $\tilde{V}^{2}\cong L_{\tilde{L}}(\tilde{\omega}_{3}+\tilde{\omega}_{4})$. Now, in $\tilde{V}^{1}$, both the weight $\displaystyle (\tilde{\lambda}-\tilde{\alpha}_{1}-\tilde{\alpha}_{2}-\tilde{\alpha}_{3})\mid_{\tilde{T}_{1}}=2\tilde{\omega}_{4}$ and the weight $\displaystyle (\tilde{\lambda}-\tilde{\alpha}_{1}-\tilde{\alpha}_{2}-\tilde{\alpha}_{4})\mid_{\tilde{T}_{1}}=2\tilde{\omega}_{3}$ admit a maximal vector, thus $\tilde{V}^{1}$ has a composition factor isomorphic to $L_{\tilde{L}}(2\tilde{\omega}_{3})$ and another isomorphic to $L_{\tilde{L}}(2\tilde{\omega}_{4})$. Moreover, the weight $\displaystyle (\tilde{\lambda}-\tilde{\alpha}_{1}-\tilde{\alpha}_{2}-\tilde{\alpha}_{3}-\tilde{\alpha}_{4})\mid_{\tilde{T}_{1}}=\tilde{\omega}_{2}$ occurs with multiplicity $3$ and is a sub-dominant weight in both the composition factors of $\tilde{V}^{1}$ we identified, where it has multiplicity $1$ in each. By dimensional considerations and \cite[II.2.14]{Jantzen_2007representations}, we deduce that 
\begin{equation}\label{DecompVD4om3+om4pneq2}
\tilde{V}\mid _{[\tilde{L},\tilde{L}]}=L_{\tilde{L}}(\tilde{\omega}_{3}+\tilde{\omega}_{4})\oplus L_{\tilde{L}}(2\tilde{\omega}_{3})\oplus L_{\tilde{L}}(2\tilde{\omega}_{4})\oplus L_{\tilde{L}}(\tilde{\omega}_{2}) \oplus L_{\tilde{L}}(\tilde{\omega}_{3}+\tilde{\omega}_{4}).
\end{equation}

We start with the semisimple elements. Let $\tilde{s}\in \tilde{T}\setminus \ZG(\tilde{G})$. If $\dim(\tilde{V}^{i}_{\tilde{s}}(\tilde{\mu}))=\dim(\tilde{V}^{i})$ for some eigenvalue $\tilde{\mu}$ of $\tilde{s}$ on $\tilde{V}$, where $0\leq i \leq 2$, then $\tilde{s}\in \ZG(\tilde{L})^{\circ}\setminus \ZG(\tilde{G})$. In this case, as $\tilde{s}$ acts on each $\tilde{V}^{i}$ as scalar multiplication by $c^{2-2i}$ and $c^{2}\neq 1$, it follows that $\dim(\tilde{V}_{\tilde{s}}(\tilde{\mu}))\leq 30$ for all eigenvalues $\tilde{\mu}$ of $\tilde{s}$ on $\tilde{V}$. We thus assume that $\dim(\tilde{V}^{i}_{\tilde{s}}(\tilde{\mu}))<\dim(\tilde{V}^{i})$ for all eigenvalues $\tilde{\mu}$ of $\tilde{s}$ on $\tilde{V}$ and all $0\leq i\leq 2$. We write $\tilde{s}=\tilde{z}\cdot \tilde{h}$, where $\tilde{z}\in \ZG(\tilde{L})^{\circ}$ and $\tilde{h}\in [\tilde{L},\tilde{L}]$. We have $\scale[0.9]{\displaystyle \dim(\tilde{V}_{\tilde{s}}(\tilde{\mu}))\leq \sum_{i=0}^{2}\dim(\tilde{V}^{i}_{\tilde{h}}(\tilde{\mu}^{i}_{\tilde{h}}))}$, where $\dim(\tilde{V}^{i}_{\tilde{h}}(\tilde{\mu}_{\tilde{h}}^{i}))<\dim(\tilde{V}^{i})$ for all eigenvalues $\tilde{\mu}_{\tilde{h}}^{i}$ of $\tilde{h}$ on $\tilde{V}^{i}$. First, assume that, up to conjugation, $\tilde{h}=h_{\tilde{\alpha}_{2}}(d)h_{\tilde{\alpha}_{3}}(d^{2})h_{\tilde{\alpha}_{4}}(d^{3})$ with $d^{4}\neq 1$. Then, by Proposition \ref{PropositionAlom1+oml}, we have $\dim(\tilde{V}^{0}_{\tilde{h}}(1))=\dim(\tilde{V}^{2}_{\tilde{h}}(1))=9$. Using \eqref{Al_enum_P2} and \eqref{Al_enum_P1}, we determine that the eigenvalues of $\tilde{h}$ on $\tilde{V}^{1}$ are $d^{\pm 2}$ each with multiplicity at least $12$; and $d^{\pm 6}$ each with multiplicity at least $1$. Thus, we have $\dim(\tilde{V}^{1}_{\tilde{h}}(\tilde{\mu}_{\tilde{h}}^{1}))\leq 13$ for all eigenvalues $\tilde{\mu}_{\tilde{h}}^{1}$ of $\tilde{h}$ on $\tilde{V}^{1}$, and so $\dim(\tilde{V}_{\tilde{s}}(\tilde{\mu}))\leq 31$ for all eigenvalues $\tilde{\mu}$ of $\tilde{s}$ on $\tilde{V}$. Secondly, if $\tilde{h}$ is not conjugate to $h_{\tilde{\alpha}_{2}}(d)h_{\tilde{\alpha}_{3}}(d^{2})h_{\tilde{\alpha}_{4}}(d^{3})$ with $d^{4}\neq 1$, then, by Propositions \ref{PropositionAlwedge}, \ref{PropositionAlsymmetric} and \ref{PropositionAlom1+oml}, we have $\dim(\tilde{V}_{\tilde{s}}(\tilde{\mu}))\leq 34$ for all eigenvalues $\tilde{\mu}$ of $\tilde{s}$ on $\tilde{V}$. Therefore, $\scale[0.9]{\displaystyle \max_{\tilde{s}\in \tilde{T}\setminus\ZG(\tilde{G})}\dim(\tilde{V}_{\tilde{s}}(\tilde{\mu}))}\leq 34$. 

For the unipotent elements, by Lemma \ref{uniprootelems}, decomposition \eqref{DecompVD4om3+om4pneq2} and Propositions \ref{PropositionAlwedge}, \ref{PropositionAlsymmetric} and \ref{PropositionAlom1+oml}, we have $\scale[0.9]{\displaystyle \max_{\tilde{u}\in \tilde{G}_{u}\setminus \{1\}}\dim(\tilde{V}_{\tilde{u}}(1))}$ $\scale[0.9]{\displaystyle=\dim(\tilde{V}_{x_{\tilde{\alpha}_{1}}(1)}(1))}=34$. If follows that $\nu_{\tilde{G}}(\tilde{V})=22$. 
\end{proof}

\begin{prop}\label{D5om3pneq2}
Let $p\neq 2$, $\ell=5$ and $\tilde{V}=L_{\tilde{G}}(\tilde{\omega}_{3})$. Then $\nu_{\tilde{G}}(\tilde{V})=44$.  Moreover, we have $\scale[0.9]{\displaystyle \max_{\tilde{u}\in \tilde{G}_{u}\setminus \{1\}}\dim(\tilde{V}_{\tilde{u}}(1))}$ $\scale[0.9]{\displaystyle}=76$ and $\scale[0.9]{\displaystyle \max_{\tilde{s}\in \tilde{T}\setminus\ZG(\tilde{G})}\dim(\tilde{V}_{\tilde{s}}(\tilde{\mu}))}\leq 72$.
\end{prop}

\begin{proof}
Set $\tilde{\lambda}=\tilde{\omega}_{3}$ and $\tilde{L}=\tilde{L}_{1}$. By Lemma \ref{weightlevelDl}, we have $e_{1}(\tilde{\lambda})=2$, therefore $\displaystyle \tilde{V}\mid_{[\tilde{L},\tilde{L}]}=\tilde{V}^{0}\oplus \tilde{V}^{1}\oplus \tilde{V}^{2}$. By \cite[Proposition]{Smith_82} and Lemma \ref{dualitylemma}, we have $\tilde{V}^{0}\cong L_{\tilde{L}}(\tilde{\omega}_{3})$ and $\tilde{V}^{2}\cong L_{\tilde{L}}(\tilde{\omega_{3}})$. Now, the weight $\displaystyle (\tilde{\lambda}-\tilde{\alpha}_{1}-\tilde{\alpha}_{2}-\tilde{\alpha}_{3})\mid_{\tilde{T}_{1}}=\tilde{\omega}_{4}+\tilde{\omega}_{5}$ admits a maximal vector in $\tilde{V}^{1}$, thus $\tilde{V}^{1}$ has a composition factor isomorphic to $L_{\tilde{L}}(\tilde{\omega}_{4}+\tilde{\omega}_{5})$. Moreover, the weight $\displaystyle (\tilde{\lambda}-\tilde{\alpha}_{1}-\tilde{\alpha}_{2}-2\tilde{\alpha}_{3}-\tilde{\alpha}_{4}-\tilde{\alpha}_{5})\mid_{\tilde{T}_{1}}=\tilde{\omega}_{2}$ occurs with multiplicity $4$ and is a sub-dominant weight in the composition factor of $\tilde{V}^{1}$ isomorphic to $L_{\tilde{L}}(\tilde{\omega}_{4}+\tilde{\omega}_{5})$, in which it has multiplicity $3$. Therefore, by dimensional considerations and \cite[II.2.14]{Jantzen_2007representations}, we determine that
\begin{equation}\label{DecompVD5om3pneq 2}
\tilde{V}\mid_{[\tilde{L},\tilde{L}]}\cong L_{\tilde{L}}(\tilde{\omega}_{3}) \oplus L_{\tilde{L}}(\tilde{\omega}_{4}+\tilde{\omega}_{5}) \oplus L_{\tilde{L}}(\tilde{\omega}_{2})\oplus L_{\tilde{L}}(\tilde{\omega}_{3}).
\end{equation}

We start with the semisimple elements. Let $\tilde{s}\in \tilde{T}\setminus \ZG(\tilde{G})$. If $\dim(\tilde{V}^{i}_{\tilde{s}}(\tilde{\mu}))=\dim(\tilde{V}^{i})$ for some eigenvalue $\tilde{\mu}$ of $\tilde{s}$ on $\tilde{V}$, where $0\leq i\leq 2$, then $\tilde{s}\in \ZG(\tilde{L})^{\circ}\setminus \ZG(\tilde{G})$. In this case, as $\tilde{s}$ acts on each $\tilde{V}^{i}$ as scalar multiplication by $c^{2-2i}$ and $c^{2}\neq 1$, it follows that $\dim(\tilde{V}_{\tilde{s}}(\tilde{\mu}))\leq 64$ for all eigenvalues $\tilde{\mu}$ of $\tilde{s}$ on $\tilde{V}$. We thus assume that $\dim(\tilde{V}^{i}_{\tilde{s}}(\tilde{\mu}))<\dim(\tilde{V}^{i})$ for all eigenvalues $\tilde{\mu}$ of $\tilde{s}$ on $\tilde{V}$ and all $0\leq i\leq 2$. We write $\tilde{s}=\tilde{z}\cdot \tilde{h}$, where $\tilde{z}\in \ZG(\tilde{L})^{\circ}$ and $\tilde{h}\in [\tilde{L},\tilde{L}]$. Now, using  \eqref{DecompVD5om3pneq 2} and Propositions \ref{PropositionDlnatural}, \ref{PropositionDlwedge} and \ref{D4om3+om4pneq2}, we determine that $\dim(\tilde{V}_{\tilde{s}}(\tilde{\mu}))\leq 2\dim((L_{\tilde{L}}(\tilde{\omega}_{3}))_{\tilde{h}}(\tilde{\mu}_{\tilde{h}}))+\dim((L_{\tilde{L}}(\tilde{\omega}_{2}))_{\tilde{h}}(\tilde{\mu}_{\tilde{h}}))+\dim((L_{\tilde{L}}(\tilde{\omega}_{4}+\tilde{\omega}_{5}))_{\tilde{h}}(\tilde{\mu}_{\tilde{h}}))\leq 72$. Therefore, $\scale[0.9]{\displaystyle \max_{\tilde{s}\in \tilde{T}\setminus\ZG(\tilde{G})}\dim(\tilde{V}_{\tilde{s}}(\tilde{\mu}))}\leq 72$.

For the unipotent elements, by Lemma \ref{uniprootelems}, decomposition \eqref{DecompVD5om3pneq 2} and Propositions \ref{PropositionDlnatural}, \ref{PropositionDlwedge} and \ref{D4om3+om4pneq2}, we have $\scale[0.9]{\displaystyle \max_{\tilde{u}\in \tilde{G}_{u}\setminus \{1\}}\dim(\tilde{V}_{\tilde{u}}(1))=\dim(\tilde{V}_{x_{\tilde{\alpha}_{1}}(1)}(1))}=76$. It follows that $\nu_{\tilde{G}}(\tilde{V})=44$.
\end{proof}

\begin{prop}\label{PropositionDlwedgecubep=2}
Let $p=2$, $\ell\geq 6$ and $\tilde{V}=L_{\tilde{G}}(\tilde{\omega}_{3})$. Then $\nu_{\tilde{G}}(\tilde{V})=4\ell^{2}-14\ell+12-2\varepsilon_{2}(\ell-1)$.  Moreover, we have $\scale[0.9]{\displaystyle \max_{\tilde{u}\in \tilde{G}_{u}\setminus \{1\}}\dim(\tilde{V}_{\tilde{u}}(1))}=\binom{2\ell-2}{3}+4\ell-8-(2\ell-2)\varepsilon_{2}(\ell-1)$ and $\scale[0.9]{\displaystyle \max_{\tilde{s}\in \tilde{T}\setminus\ZG(\tilde{G})}\dim(\tilde{V}_{\tilde{s}}(\tilde{\mu}))}\leq\binom{2\ell-2}{3}-(2\ell-2)\varepsilon_{2}(\ell-1)$.
\end{prop}

\begin{proof}
First, we have $\tilde{V}\cong L_{G}(\omega_{3})$ as $k\tilde{G}$-modules. To ease notation, let $V=L_{G}(\omega_{3})$. Secondly, as $p=2$, by \cite[Table 1]{seitz1987maximal}, we have $L_{H}(\omega_{3}^{H})\cong V$, as $kG$-modules. Lastly, by \cite[Lemma $4.8.2$]{mcninch_1998}, we have $\wedge^{3}(W)\cong L_{H}(\omega_{3}^{H})\oplus L_{H}(\omega_{1}^{H})$ if $\varepsilon_{2}(\ell-1)=0$, and $\wedge^{3}(W)\cong L_{H}(\omega_{1}^{H})\mid L_{H}(\omega_{3}^{H})\mid L_{H}(\omega_{1}^{H})$ if $\varepsilon_{2}(\ell-1)=1$.

We start with the semisimple elements. In view of Lemma \ref{LtildeGtildelambdaandLGlambda} and the above arguments, we deduce that $\scale[0.85]{\displaystyle \max_{\tilde{s}\in \tilde{T}\setminus\ZG(\tilde{G})}\dim(\tilde{V}_{\tilde{s}}(\tilde{\mu}))=\max_{s_{H}\in T_{H}\setminus \ZG(H)}\dim((L_{H}(\omega_{3}^{H}))_{s_{H}}(\mu))}\leq\binom{2\ell-2}{3}-(2\ell-2)\varepsilon_{2}(\ell-1)$, by Proposition \ref{PropositionClwedgecube}. 

For the unipotent elements, by Lemmas \ref{LtildeGtildelambdaandLGlambda} and \ref{uniprootelems}, we have $\scale[0.9]{\displaystyle \max_{\tilde{u}\in \tilde{G}_{u}\setminus \{1\}}\dim(\tilde{V}_{\tilde{u}}(1))=}$ $\scale[0.9]{\displaystyle \dim(V_{x_{\alpha_{\ell}}(1)}(1))}$. Further, by Lemma \ref{LemmaonfiltrationofV}, we have $\dim(V_{x_{\alpha_{\ell}}(1)}(1))\leq \dim((\wedge^{3}(W))_{x_{\alpha_{\ell}}(1)}(1))-(1+\varepsilon_{2}(\ell-1))\dim((L_{H}(\omega_{1}^{H}))_{x_{\alpha_{\ell}}(1)}(1))=\frac{4\ell^{3}-18\ell^{2}+44\ell-42}{3}-(1+\varepsilon_{2}(\ell-1))(2\ell-2)$ and equality holds when $\varepsilon_{2}(\ell-1)=0$. In what follows we show that equality also holds when $\varepsilon_{2}(\ell-1)=1$.

Let $\lambda=\omega_{3}$ and $L=L_{1}$. We have $e_{1}(\lambda)=2$, therefore $\displaystyle V\mid_{[L,L]}=V^{0}\oplus V^{1}\oplus V^{2}$. Now, by \cite[Proposition]{Smith_82} and Lemma \ref{dualitylemma}, we have $V^{0}\cong L_{L}(\omega_{3})$ and $V^{2}\cong L_{L}(\omega_{3})$. The weight $\displaystyle (\lambda-\alpha_{1}-\alpha_{2}-\alpha_{3})\mid_{T_{1}}=\omega_{4}$ admits a maximal vector in $V^{1}$, thus $V^{1}$ has a composition factor isomorphic to $L_{L}(\omega_{4})$. Moreover, the weight $\displaystyle (\lambda-\alpha_{1}-\alpha_{2}-2\alpha_{3}-\cdots-2\alpha_{\ell-2}-\alpha_{\ell-1}-\alpha_{\ell})\mid_{T_{1}}=\omega_{2}$ occurs with multiplicity $\ell-3+\varepsilon_{2}(\ell)$ and is a sub-dominant weight in the composition factor of $V^{1}$ isomorphic to $L_{L}(\omega_{4})$, in which it has multiplicity $\ell-4-\varepsilon_{2}(\ell-1)$. Therefore, by dimensional considerations, we determine that $V^{1}$ has exactly $2+\varepsilon_{2}(\ell)-\varepsilon_{2}(\ell-1)$ composition factors: one isomorphic to $L_{L}(\omega_{4})$ and $1+\varepsilon_{2}(\ell)-\varepsilon_{2}(\ell-1)$ to $L_{L}(\omega_{2})$. 

By the structure of $V\mid_{[L,L]}$ and Propositions \ref{PropositionDlnatural} and \ref{PropositionDlwedge}, it follows that $\scale[0.9]{\displaystyle \dim(V_{x_{\alpha_{\ell}}(1)}(1))\leq4(\ell-1)^{2}-8(\ell-1)}$ $\scale[0.9]{+8+[2(\ell-1)-2]\varepsilon_{2}(\ell)-2(\ell-1)\varepsilon_{2}(\ell-1)+\dim((L_{L}(\omega_{4}))_{x_{\alpha_{\ell}}(1)}(1))}$. Recursively and using Proposition \ref{D5om3p=2} for the base case of $\ell=6$, we determine that $\scale[0.9]{\displaystyle\dim(V_{x_{\alpha_{\ell}}(1)}(1))\leq 4\sum_{j=5}^{\ell-1}j^{2}-8\sum_{j=5}^{\ell-1}j+\sum_{j=5}^{\ell-1}8+\sum_{j=5}^{\ell-1}(2j-2)\varepsilon_{2}(j+1)-\sum_{j=5}^{\ell-1}2j\varepsilon_{2}(j)}$ $+60= \binom{2\ell-2}{3}+4\ell-8-(2\ell-2)\varepsilon_{2}(\ell-1)$. Therefore, $\scale[0.9]{\displaystyle \max_{\tilde{u}\in \tilde{G}_{u}\setminus \{1\}}\dim(\tilde{V}_{\tilde{u}}(1))}=\binom{2\ell-2}{3}+4\ell-8-(2\ell-2)\varepsilon_{2}(\ell-1)$, and so, we conclude that $\nu_{\tilde{G}}(\tilde{V})=4\ell^{2}-14\ell+12-2\varepsilon_{2}(\ell-1)$.
\end{proof}

\begin{prop}\label{PropositionDlwedgecubepneq2}
Let $p\neq 2$, $\ell\geq 6$ and $\tilde{V}=L_{\tilde{G}}(\tilde{\omega}_{3})$. Then $\nu_{\tilde{G}}(\tilde{V})=4\ell^{2}-14\ell+14$.  Moreover, we have $\scale[0.9]{\displaystyle \max_{\tilde{u}\in \tilde{G}_{u}\setminus \{1\}}\dim(\tilde{V}_{\tilde{u}}(1))}=\binom{2\ell-2}{3}+6\ell-10$ and $\scale[0.9]{\displaystyle \max_{\tilde{s}\in \tilde{T}\setminus\ZG(\tilde{G})}\dim(\tilde{V}_{\tilde{s}}(\tilde{\mu}))}\leq \binom{2\ell-2}{3}+2\ell+6$.
\end{prop}

\begin{proof}
First, we note that $\tilde{V}\cong L_{G}(\omega_{3})$ as $k\tilde{G}$-modules. To ease notation, let $V=L_{G}(\omega_{3})$. Secondly, as $p\neq 2$, by \cite[Proposition $4.2.2$]{mcninch_1998}, we have $V\cong \wedge^{3}(W)$ as $kG$-modules. 

We start with the unipotent elements. In view of Lemmas \ref{LtildeGtildelambdaandLGlambda} and \ref{uniprootelems}, we have $\scale[0.9]{\displaystyle \max_{\tilde{u}\in \tilde{G}_{u}\setminus \{1\}}\dim(\tilde{V}_{\tilde{u}}(1))=}$ $\dim(V_{x_{\alpha_{\ell}}(1)}(1))=\binom{2\ell-2}{3}+6\ell-10$, see the proof of Proposition \ref{PropositionClwedgecube}. 

Let $\lambda=\omega_{3}$ and $L=L_{1}$. By Lemma \ref{weightlevelDl}, we have $e_{1}(\lambda)=2$, therefore $\displaystyle V\mid_{[L,L]}=V^{0}\oplus V^{1}\oplus V^{2}$. By \cite[Proposition]{Smith_82} and Lemma \ref{dualitylemma}, we have $V^{0}\cong L_{L}(\omega_{3})$ and $V^{2}\cong L_{L}(\omega_{3})$. Now, the weight $\displaystyle (\lambda-\alpha_{1}-\alpha_{2}-\alpha_{3})\mid_{T_{1}}=\omega_{4}$ admits a maximal vector in $V^{1}$, thus $V^{1}$ has a composition factor isomorphic to $L_{L}(\omega_{4})$. Moreover, the weight $\displaystyle (\lambda-\alpha_{1}-\alpha_{2}-2\alpha_{3}-\cdots-2\alpha_{\ell-2}-\alpha_{\ell-1}-\alpha_{\ell})\mid_{T_{1}}=\omega_{2}$ occurs with multiplicity $\ell-1$ and is a sub-dominant weight in the composition factor of $V^{1}$ isomorphic to $L_{L}(\omega_{4})$, in which it has multiplicity $\ell-2$. Therefore, by dimensional considerations and \cite[II.$2.14$]{Jantzen_2007representations}, we determine that $V^{1}\cong L_{L}(\omega_{4})\oplus L_{L}(\omega_{2})$. 

Let $s\in T\setminus \ZG(G)$. If $\dim(V^{i}_{s}(\mu))=\dim(V^{i})$ for some eigenvalue $\mu$ of $s$ on $V$, where $0\leq i\leq 2$, then $s\in \ZG(L)^{\circ}\setminus \ZG(G)$. In this case, as $s$ acts on each $V^{i}$ as scalar multiplication by $c^{2-2i}$ and $c^{2}\neq 1$, it follows that $\dim(V_{s}(\mu))\leq \binom{2\ell-2}{3}+2\ell-2$ for all eigenvalues $\mu$ of $s$ on $V$. We thus assume that $\dim(V^{i}_{s}(\mu))<\dim(V^{i})$ for all eigenvalues $\mu$ of $s$ on $V$ and all $0\leq i\leq 2$. We write $s=z\cdot h$, where $z\in \ZG(L)^{\circ}$ and $h\in [L,L]$. By the structure of $V\mid_{[L,L]}$ and Propositions \ref{PropositionDlnatural} and \ref{PropositionDlwedge}, it follows that $\scale[0.9]{\displaystyle \dim(V_{s}(\mu))\leq 4(\ell-1)^{2}-8(\ell-1)+6+\dim((L_{L}(\omega_{4}))_{h}(\mu_{h}))}$. Recursively and using Proposition \ref{D5om3pneq2} for the base case of $\ell=6$, we determine that $\scale[0.9]{\displaystyle \dim(V_{s}(\mu))\leq 4\sum_{j=5}^{\ell-1}j^{2}-8\sum_{j=5}^{\ell-1}j+\sum_{j=5}^{\ell-1}6+72}$ $= \binom{2\ell-2}{3}+2\ell+6$. Therefore, $\scale[0.9]{\displaystyle \max_{\tilde{s}\in \tilde{T}\setminus\ZG(\tilde{G})}\dim(\tilde{V}_{\tilde{s}}(\tilde{\mu}))\leq} \binom{2\ell-2}{3}\scale[0.9]{+2\ell+6}$, and, as$\scale[0.9]{\displaystyle \max_{\tilde{u}\in \tilde{G}_{u}\setminus \{1\}}\dim(\tilde{V}_{\tilde{u}}(1))=}\binom{2\ell-2}{3}\scale[0.9]{+6\ell-10}$, we have $\scale[0.9]{\nu_{\tilde{G}}(\tilde{V})=4\ell^{2}-14\ell+14}$.
\end{proof}

\begin{prop}\label{D43om1}
Let $p\neq 2,3$, $\ell=4$ and $\tilde{V}=L_{\tilde{G}}(3\tilde{\omega}_{1})$. Then $\nu_{\tilde{G}}(\tilde{V})=44-2\varepsilon_{p}(5)$.  Moreover, we have  $\scale[0.9]{\displaystyle \max_{\tilde{u}\in \tilde{G}_{u}\setminus \{1\}}\dim(\tilde{V}_{\tilde{u}}(1))}=56-6\varepsilon_{p}(5)$ and $\scale[0.9]{\displaystyle \max_{\tilde{s}\in \tilde{T}\setminus\ZG(\tilde{G})}\dim(\tilde{V}_{\tilde{s}}(\tilde{\mu}))}=68-6\varepsilon_{p}(5)$.
\end{prop}

\begin{proof}
Let $\tilde{\lambda}=3\tilde{\omega}_{1}$ and $\tilde{L}=\tilde{L}_{1}$. We have $e_{1}(\tilde{\lambda})=6$, therefore $\displaystyle \tilde{V}\mid_{[\tilde{L},\tilde{L}]}=\tilde{V}^{0}\oplus \cdots\oplus \tilde{V}^{6}$. By \cite[Proposition]{Smith_82} and Lemma \ref{dualitylemma}, we have $\tilde{V}^{0}\cong L_{\tilde{L}}(0)$ and $\tilde{V}^{6}\cong L_{\tilde{L}}(0)$. Now, the weight $\displaystyle (\tilde{\lambda}-\tilde{\alpha}_{1})\mid_{\tilde{T}_{1}}=\tilde{\omega}_{2}$ admits a maximal vector in $\tilde{V}^{1}$, thus $\tilde{V}^{1}$ has a composition factor isomorphic to $L_{\tilde{L}}(\tilde{\omega}_{2})$. Similarly, the weight $\displaystyle (\tilde{\lambda}-2\tilde{\alpha}_{1})\mid_{\tilde{T}_{1}}=2\tilde{\omega}_{2}$ admits a maximal vector in $\tilde{V}^{2}$, thus $\tilde{V}^{2}$ has a composition factor isomorphic to $L_{\tilde{L}}(2\tilde{\omega}_{2})$. Moreover, the weight $\displaystyle (\tilde{\lambda}-2\tilde{\alpha}_{1}-\cdots-2\tilde{\alpha}_{\ell-2}-\tilde{\alpha}_{\ell-1}-\tilde{\alpha}_{\ell})\mid_{\tilde{T}_{1}}=0$ occurs with multiplicity $3-\varepsilon_{p}(5)$ and is a sub-dominant weight in the composition factor of $\tilde{V}^{2}$ isomorphic to $L_{\tilde{L}}(2\tilde{\omega}_{2})$, in which it has multiplicity $2$. Lastly, the weight $\displaystyle (\tilde{\lambda}-3\tilde{\alpha}_{1})\mid_{\tilde{T}_{1}}=3\tilde{\omega}_{2}$ admits a maximal vector in $\tilde{V}^{3}$, thus $\tilde{V}^{3}$ has a composition factor isomorphic to $L_{\tilde{L}}(3\tilde{\omega}_{2})$. Moreover, the weight $\displaystyle (\tilde{\lambda}-3\tilde{\alpha}_{1}-\cdots-2\tilde{\alpha}_{\ell-2}-\tilde{\alpha}_{\ell-1}-\tilde{\alpha}_{\ell})\mid_{\tilde{T}_{1}}=\tilde{\omega}_{2}$ occurs with multiplicity $3-\varepsilon_{p}(5)$ and is a sub-dominant weight in the composition factor of $\tilde{V}^{3}$ isomorphic to $L_{\tilde{L}}(3\tilde{\omega}_{2})$, in which it has multiplicity $2$. By dimensional considerations and \cite[II.$2.14$]{Jantzen_2007representations}, we have
\begin{equation}\label{DecompVDl3om1}
\tilde{V}\mid_{[\tilde{L},\tilde{L}]}\cong L_{\tilde{L}}(0)^{4-2\varepsilon_{p}(5)} \oplus L_{\tilde{L}}(\tilde{\omega}_{2})^{3-\varepsilon_{p}(5)} \oplus L_{\tilde{L}}(2\tilde{\omega}_{2})^{2} \oplus L_{\tilde{L}}(3\tilde{\omega}_{2}).
\end{equation}

We start with the semisimple elements. Let $\tilde{s}\in \tilde{T}\setminus \ZG(\tilde{G})$. If $\dim(\tilde{V}^{i}_{\tilde{s}}(\tilde{\mu}))=\dim(\tilde{V}^{i})$ for some eigenvalue $\tilde{\mu}$ of $\tilde{s}$ on $\tilde{V}$, where $0\leq i\leq 6$, then $\tilde{s}\in \ZG(\tilde{L})^{\circ}\setminus \ZG(\tilde{G})$. In this case, as $\tilde{s}$ acts on each $\tilde{V}^{i}$ as scalar multiplication by $c^{6-2i}$ and $c^{2}\neq 1$, it follows that $\dim(\tilde{V}_{\tilde{s}}(\tilde{\mu}))\leq 68-6\varepsilon_{p}(5)$, where equality holds for $c^{2}=-1$ and $\tilde{\mu}=1$. We thus assume that $\dim(\tilde{V}^{i}_{\tilde{s}}(\tilde{\mu}))<\dim(\tilde{V}^{i})$ for all eigenvalues $\tilde{\mu}$ of $\tilde{s}$ on $\tilde{V}$ and all $0\leq i\leq 6$. We write $\tilde{s}=\tilde{z}\cdot \tilde{h}$, where $\tilde{z}\in \ZG(\tilde{L})^{\circ}$ and $\tilde{h}\in [\tilde{L},\tilde{L}]$, and, by \eqref{DecompVDl3om1} and Propositions \ref{PropositionAlwedge}, \ref{A32om2} and \ref{A33om2}, it follows that $\dim(\tilde{V}_{\tilde{s}}(\tilde{\mu})) \leq (4-2\varepsilon_{p}(5))\dim((L_{\tilde{L}}(0))_{\tilde{h}}(\tilde{\mu}_{\tilde{h}}))+(3-\varepsilon_{p}(5))\dim((L_{\tilde{L}}(\tilde{\omega}_{2}))_{\tilde{h}}(\tilde{\mu}_{\tilde{h}}))+2\dim((L_{\tilde{L}}(2\tilde{\omega}_{2}))_{\tilde{h}}(\tilde{\mu}_{\tilde{h}}))+$ $\dim((L_{\tilde{L}}(3\tilde{\omega}_{2}))_{\tilde{h}}(\tilde{\mu}_{\tilde{h}})) =68-6\varepsilon_{p}(5)$. Therefore,  $\scale[0.9]{\displaystyle \max_{\tilde{s}\in \tilde{T}\setminus\ZG(\tilde{G})}\dim(\tilde{V}_{\tilde{s}}(\tilde{\mu}))}=68-6\varepsilon_{p}(5)$.

For the unipotent elements, in view of Lemma \ref{uniprootelems}, we have $\scale[0.9]{\displaystyle \max_{\tilde{u}\in \tilde{G}_{u}\setminus \{1\}}\dim(\tilde{V}_{\tilde{u}}(1))=}$ $\scale[0.9]{\dim(\tilde{V}_{x_{\tilde{\alpha}_{4}}(1)}(1))}$. By \eqref{DecompVDl3om1} and Propositions \ref{PropositionAlwedge}, \ref{A32om2} and \ref{A33om2}, we determine that $\scale[0.85]{\displaystyle \max_{\tilde{u}\in \tilde{G}_{u}\setminus \{1\}}\dim(\tilde{V}_{\tilde{u}}(1))=56-6\varepsilon_{p}(5)}$ and $\scale[0.85]{\nu_{\tilde{G}}(\tilde{V})=44-2\varepsilon_{p}(5)}$.
\end{proof}

\begin{prop}\label{PropositionDlsymmcube}
Let $p\neq 2,3$, $\ell\geq 5$ and $\tilde{V}=L_{\tilde{G}}(3\tilde{\omega}_{1})$. Then $\nu_{\tilde{G}}(\tilde{V})=4\ell^{2}-6\ell+4-2\varepsilon_{p}(\ell+1)$.  Moreover, we have $\scale[0.9]{\displaystyle \max_{\tilde{u}\in \tilde{G}_{u}\setminus \{1\}}\dim(\tilde{V}_{\tilde{u}}(1))=}\binom{2\ell}{3}\scale[0.9]{-(2\ell-2)\varepsilon_{p}(\ell+1)}$ and $\scale[0.9]{\displaystyle \max_{\tilde{s}\in \tilde{T}\setminus\ZG(\tilde{G})}\dim(\tilde{V}_{\tilde{s}}(\tilde{\mu}))=}\binom{2\ell}{3}\scale[0.9]{+4\ell-4-(2\ell-2)\varepsilon_{p}(\ell+1)}$.
\end{prop}

\begin{proof}
First, we note that $\tilde{V}\cong L_{G}(3\omega_{1})$ as $k\tilde{G}$-modules. To ease notation, let $V=L_{G}(3\omega_{1})$. Now, by \cite[Proposition $4.7.4$]{mcninch_1998}, we have $\SWT(W)\cong V\oplus L_{G}(\omega_{1})$ when $\varepsilon_{p}(\ell+1)=0$, and $\SWT(W)\cong L_{G}(\omega_{1})\mid V\mid L_{G}(\omega_{1})$ when $\varepsilon_{p}(\ell+1)=1$, as $kG$-modules. 

We start with the unipotent elements. In view of Lemmas \ref{LtildeGtildelambdaandLGlambda} and \ref{uniprootelems}, we have $\scale[0.9]{\displaystyle \max_{\tilde{u}\in \tilde{G}_{u}\setminus \{1\}}\dim(\tilde{V}_{\tilde{u}}(1))=}$ $\dim(V_{x_{\alpha_{\ell}}(1)}(1))$. Now, by the structure of $\SWT(W)$ as a $kG$-module and Lemma \ref{LemmaonfiltrationofV}, we get $\dim(V_{x_{\alpha_{\ell}}(1)}(1))$ $\geq \binom{2\ell}{3}-(2\ell-2)\varepsilon_{p}(\ell+1)$, see the proof of Proposition \ref{PropositionClsymmcube}. Further, equality holds when $\varepsilon_{p}(\ell+1)=0$. In what follows, we show that equality holds also when $\varepsilon_{p}(\ell+1)=1$. 

Let $\lambda=3\omega_{1}$ and $L=L_{1}$. We have $e_{1}(\lambda)=6$, therefore $\displaystyle V\mid_{[L,L]}=V^{0}\oplus \cdots\oplus V^{6}$. We argue as we did in the proof of Proposition \ref{D43om1}, to show that $V^{0}\cong L_{L}(0)$ and $V^{6}\cong L_{L}(0)$; $V^{1}\cong L_{L}(\omega_{2})$ and $V^{5}\cong L_{L}(\omega_{2})$; $V^{2}$ and $V^{4}$ each has $2-\varepsilon_{p}(\ell+1)+\varepsilon_{p}(\ell-1)$ composition factors: one isomorphic to $L_{L}(2\omega_{2})$ and $1-\varepsilon_{p}(\ell+1)+\varepsilon_{p}(\ell-1)$ to $L_{L}(0)$; and $V^{3}$ has exactly $2-\varepsilon_{p}(\ell+1)+\varepsilon_{p}(\ell)$ composition factors: one isomorphic to $L_{L}(3\omega_{2})$ and $1-\varepsilon_{p}(\ell+1)+\varepsilon_{p}(\ell)$ to $L_{L}(\omega_{2})$. 

Now, by the structure of $V\mid_{[L,L]}$ and Propositions \ref{PropositionDlnatural} and \ref{PropositionDlsymm}, it follows that $\dim(V_{x_{\alpha_{\ell}}(1)}(1))\leq (4-2\varepsilon_{p}(\ell+1)+2\varepsilon_{p}(\ell-1))\dim((L_{L}(0))_{x_{\alpha_{\ell}}(1)}(1))+(3-\varepsilon_{p}(\ell+1)+\varepsilon_{p}(\ell))\dim((L_{L}(\omega_{2}))_{x_{\alpha_{\ell}}(1)}(1))+2\dim((L_{L}(2\omega_{2}))_{x_{\alpha_{\ell}}(1)}(1))$ $+ \dim((L_{L}(3\omega_{2}))_{x_{\alpha_{\ell}}(1)}(1))=4(\ell-1)^{2}-2(\ell-1)\varepsilon_{p}(\ell+1)+[2(\ell-1)-2]\varepsilon_{p}(\ell)+\dim((L_{L}(3\omega_{2}))_{x_{\alpha_{\ell}}(1)}(1))$. Recursively and using Proposition \ref{D43om1} for the base case of $\ell=5$, we determine that $\scale[0.9]{\displaystyle \dim(V_{x_{\alpha_{\ell}}(1)}(1))}\leq \binom{2\ell}{3}-(2\ell-2)\varepsilon_{p}(\ell+1)$. Therefore, $\scale[0.9]{\displaystyle \max_{\tilde{u}\in \tilde{G}_{u}\setminus \{1\}}\dim(\tilde{V}_{\tilde{u}}(1))}=\binom{2\ell}{3}-(2\ell-2)\varepsilon_{p}(\ell+1)$. 

We now focus on the semisimple elements. Let $s\in T\setminus \ZG(G)$. If $\dim(V^{i}_{s}(\mu))=\dim(V^{i})$ for some eigenvalue $\mu$ of $s$ on $V$, where $0\leq i\leq 6$, then $s\in \ZG(L)^{\circ}\setminus \ZG(G)$. In this case, as $s$ acts on each $V^{i}$ as scalar multiplication by $c^{6-2i}$ and $c^{2}\neq 1$, it follows that $\dim(V_{s}(\mu))\leq \binom{2\ell}{3}+4\ell-4-(2\ell-2)\varepsilon_{p}(\ell+1)$, where equality holds for $c^{2}=-1$ and $\mu=1$. We thus assume that $\dim(V^{i}_{s}(\mu))<\dim(V^{i})$ for all eigenvalues $\mu$ of $s$ on $V$ and all $0\leq i\leq 6$. We write $s=z\cdot h$, where $z\in \ZG(L)^{\circ}$ and $h\in [L,L]$, and, by the structure of $V\mid_{[L,L]}$ and Propositions \ref{PropositionDlnatural} and \ref{PropositionDlsymm}, it follows that $\dim(V_{s}(\mu)) \leq 4(\ell-1)^{2}+4-2(\ell-1)\varepsilon_{p}(\ell+1)+[2(\ell-1)-2]\varepsilon_{p}(\ell)+\dim((L_{L}(3\omega_{2}))_{h}(\mu_{h}))$. Recursively and using Proposition \ref{D43om1} for the base case of $\ell=5$, we determine that $\dim(V_{s}(\mu)) \leq \binom{2\ell}{3}+4\ell-4-(2\ell-2)\varepsilon_{p}(\ell+1)$. Therefore, $\scale[0.9]{\displaystyle \max_{\tilde{s}\in \tilde{T}\setminus\ZG(\tilde{G})}\dim(\tilde{V}_{\tilde{s}}(\tilde{\mu}))}= \binom{2\ell}{3}+4\ell-4-(2\ell-2)\varepsilon_{p}(\ell+1)$, by Lemma \ref{LtildeGtildelambdaandLGlambda}, and so $\nu_{\tilde{G}}(\tilde{V})=4\ell^{2}-6\ell+4-2\varepsilon_{p}(\ell+1)$.
\end{proof}

\begin{prop}\label{D4om1+om2p=3}
Let $p=3$, $\ell=4$ and $\tilde{V}=L_{\tilde{G}}(\tilde{\omega}_{1}+\tilde{\omega}_{2})$. Then $\nu_{\tilde{G}}(\tilde{V})=42$. Moreover, we have $\scale[0.9]{\displaystyle \max_{\tilde{u}\in \tilde{G}_{u}\setminus \{1\}}\dim(\tilde{V}_{\tilde{u}}(1))}\leq 54$ and $\scale[0.9]{\displaystyle \max_{\tilde{s}\in \tilde{T}\setminus\ZG(\tilde{G})}\dim(\tilde{V}_{\tilde{s}}(\tilde{\mu}))}=62$.
\end{prop}

\begin{proof}
Let $\tilde{\lambda}=\tilde{\omega}_{1}+\tilde{\omega}_{2}$ and $\tilde{L}=\tilde{L}_{1}$. By Lemma \ref{weightlevelDl}, we have $e_{1}(\tilde{\lambda})=4$, therefore $\displaystyle \tilde{V}\mid_{[\tilde{L},\tilde{L}]}=\tilde{V}^{0}\oplus \cdots\oplus \tilde{V}^{4}$. By \cite[Proposition]{Smith_82} and Lemma \ref{dualitylemma}, we have $\tilde{V}^{0}\cong L_{\tilde{L}}(\tilde{\omega}_{2})$ and $\tilde{V}^{4}\cong L_{\tilde{L}}(\tilde{\omega}_{2})$. The weight $\displaystyle (\tilde{\lambda}-\tilde{\alpha}_{1})\mid_{\tilde{T}_{1}}=2\tilde{\omega}_{2}$ admits a maximal vector in $\tilde{V}^{1}$, thus $\tilde{V}^{1}$ has a composition factor isomorphic to $L_{\tilde{L}}(2\tilde{\omega}_{2})$. Moreover, the weight $\displaystyle (\tilde{\lambda}-\tilde{\alpha}_{1}-2\tilde{\alpha}_{2}-\tilde{\alpha}_{3}-\tilde{\alpha}_{4})\mid_{\tilde{T}_{1}}=0$ occurs with multiplicity $3$ and is a sub-dominant weight in the composition factor of $\tilde{V}^{1}$ isomorphic to $L_{\tilde{L}}(2\tilde{\omega}_{2})$, in which it has multiplicity $1$. Similarly, the weight $\displaystyle (\tilde{\lambda}-2\tilde{\alpha}_{1}-\tilde{\alpha}_{2})\mid_{\tilde{T}_{1}}=\tilde{\omega}_{2}+\tilde{\omega}_{3}+\tilde{\omega}_{4}$ admits a maximal vector in $\tilde{V}^{2}$, thus $\tilde{V}^{2}$ has a composition factor isomorphic to $L_{\tilde{L}}(\tilde{\omega}_{2}+\tilde{\omega}_{3}+\tilde{\omega}_{4})$. Moreover, the weight $\displaystyle (\tilde{\lambda}-2\tilde{\alpha}_{1}-2\tilde{\alpha}_{2}-\tilde{\alpha}_{3}-\tilde{\alpha}_{4})\mid_{\tilde{T}_{1}}=\tilde{\omega}_{2}$ occurs with multiplicity $3$ and is a sub-dominant weight in the composition factor of $\tilde{V}^{2}$ isomorphic to $L_{\tilde{L}}(\tilde{\omega}_{2}+\tilde{\omega}_{3}+\tilde{\omega}_{4})$, in which it has multiplicity $2$. By dimensional considerations and \cite[II.$2.14$]{Jantzen_2007representations}, we determine that $\tilde{V}^{2}\cong L_{\tilde{L}}(\tilde{\omega}_{2}+\tilde{\omega}_{3}+\tilde{\omega}_{4})\oplus L_{\tilde{L}}(\tilde{\omega}_{2})$; and that $\tilde{V}^{1}$ and $\tilde{V}^{3}$ each has $3$ composition factors: one isomorphic to $L_{\tilde{L}}(2\tilde{\omega}_{2})$ and two to $L_{\tilde{L}}(0)$.

We begin with the semisimple elements. Let $\tilde{s}\in \tilde{T}\setminus \ZG(\tilde{G})$. If $\dim(\tilde{V}^{i}_{\tilde{s}}(\tilde{\mu}))=\dim(\tilde{V}^{i})$ for some eigenvalue $\tilde{\mu}$ of $\tilde{s}$ on $\tilde{V}$, where $0\leq i\leq 4$, then $\tilde{s}\in \ZG(\tilde{L})^{\circ}\setminus \ZG(\tilde{G})$. In this case, as $\tilde{s}$ acts on each $\tilde{V}^{i}$ as scalar multiplication by $c^{4-2i}$ and $c^{2}\neq 1$, it follows that $\dim(\tilde{V}_{\tilde{s}}(\tilde{\mu}))\leq 62$, where equality holds for $c^{2}=-1$ and $\tilde{\mu}=1$. We thus assume that $\dim(\tilde{V}^{i}_{\tilde{s}}(\tilde{\mu}))<\dim(\tilde{V}^{i})$ for all eigenvalues $\tilde{\mu}$ of $\tilde{s}$ on $\tilde{V}$ and all $0\leq i\leq 4$. We write $\tilde{s}=\tilde{z}\cdot \tilde{h}$, where $\tilde{z}\in \ZG(\tilde{L})^{\circ}$ and $\tilde{h}\in [\tilde{L},\tilde{L}]$, and we have $\scale[0.9]{\displaystyle \dim(\tilde{V}_{\tilde{s}}(\tilde{\mu}))=\sum_{i=0}^{4}\dim(\tilde{V}^{i}_{\tilde{h}}(\tilde{\mu}^{i}_{\tilde{h}}))}$, where $\dim(\tilde{V}^{i}_{\tilde{h}}(\tilde{\mu}^{i}_{\tilde{h}}))<\dim(\tilde{V^{i}})$ for all eigenvalues $\tilde{\mu}_{\tilde{h}}^{i}$ of $\tilde{h}$ on $\tilde{V}^{i}$. First, assume that, up to conjugation, $\tilde{h}=h_{\tilde{\alpha}_{2}}(\pm 1)h_{\tilde{\alpha}_{3}}(d^{-1})h_{\tilde{\alpha}_{4}}(d)$ with $d^{2}\neq \pm 1$. Using the weight structure of $\tilde{V}^{1}$, one can show that the eigenvalues of $\tilde{h}$ on $\tilde{V}^{1}$ are $1$ with multiplicity $13$; and $\pm d^{2}$ and $\pm d^{-2}$, each with multiplicity at least $4$. Further, by the stricture of $\tilde{V}\mid_{[\tilde{L},\tilde{L}]}$ and Propositions \ref{PropositionAlwedge} and \ref{A3om1+om2+om3}, we determine that $\dim(\tilde{V}_{\tilde{s}}(\tilde{\mu}))\leq 62$ for all eigenvalues $\tilde{\mu}$ of $\tilde{s}$ on $\tilde{V}$. On the other hand, if $\tilde{h}$ is not conjugate to $h_{\tilde{\alpha}_{2}}(\pm 1)h_{\tilde{\alpha}_{3}}(d^{-1})h_{\tilde{\alpha}_{4}}(d)$ with $d^{2}\neq \pm 1$, then one shows that $\dim((L_{\tilde{L}}(\tilde{\omega}_{2}))_{\tilde{h}}(\tilde{\mu}_{\tilde{h}}))\leq 3$ for all eigenvalues $\tilde{\mu}_{\tilde{h}}$ of $\tilde{h}$ on $L_{\tilde{L}}(\tilde{\omega}_{2})$,and, consequently, $\dim(\tilde{V}_{\tilde{s}}(\tilde{\mu}))\leq 61$ for all eigenvalues $\tilde{\mu}$ of $\tilde{s}$ on $\tilde{V}$. It follows that $\scale[0.9]{\displaystyle \max_{\tilde{s}\in \tilde{T}\setminus\ZG(\tilde{G})}\dim(\tilde{V}_{\tilde{s}}(\tilde{\mu}))}=62$. 

For the unipotent elements, by Lemma \ref{uniprootelems}, the structure of $\tilde{V}\mid_{[\tilde{L},\tilde{L}]}$ and Propositions \ref{PropositionAlwedge}, \ref{A3om1+om2+om3} and \ref{A32om2}, we have $\scale[0.9]{\displaystyle \max_{\tilde{u}\in \tilde{G}_{u}\setminus \{1\}}\dim(\tilde{V}_{\tilde{u}}(1))=}$ $\dim(\tilde{V}_{x_{\tilde{\alpha}_{\ell}}(1)}(1))\leq 54$. Lastly, we note that $\nu_{\tilde{G}}(\tilde{V})=40$.
\end{proof}

\begin{prop}\label{PropositionDlom1+om2p=3}
Let $p=3$, $\ell\geq 5$ and $\tilde{V}=L_{\tilde{G}}(\tilde{\omega}_{1}+\tilde{\omega}_{2})$. Then $\scale[0.9]{\nu_{\tilde{G}}(\tilde{V})=4\ell^{2}-6\ell+2-2\varepsilon_{3}(2\ell-1)}$, and we have $\scale[0.9]{\displaystyle \max_{\tilde{u}\in \tilde{G}_{u}\setminus \{1\}}\dim(\tilde{V}_{\tilde{u}}(1))\leq} \binom{2\ell}{3}\scale[0.9]{-2\ell+6-(2\ell-2)\varepsilon_{3}(2\ell-1)}$ and $\scale[0.9]{\displaystyle \max_{\tilde{s}\in \tilde{T}\setminus\ZG(\tilde{G})}\dim(\tilde{V}_{\tilde{s}}(\tilde{\mu}))=}\binom{2\ell}{3}\scale[0.9]{+2\ell-2-(2\ell-2)\varepsilon_{3}(2\ell-1)}$.
\end{prop}

\begin{proof}
Let $\tilde{\lambda}=\tilde{\omega}_{1}+\tilde{\omega}_{2}$ and $\tilde{L}=\tilde{L}_{1}$. By Lemma \ref{weightlevelDl}, we have $e_{1}(\tilde{\lambda})=4$, therefore $\displaystyle \tilde{V}\mid_{[\tilde{L},\tilde{L}]}=\tilde{V}^{0}\oplus \cdots\oplus \tilde{V}^{4}$. We argue as we did in the proof of Proposition \ref{D4om1+om2p=3}, to show that $\tilde{V}^{0}\cong L_{\tilde{L}}(\tilde{\omega}_{2})$ and $\tilde{V}^{4}\cong L_{\tilde{L}}(\tilde{\omega}_{2})$; that $\tilde{V}^{1}$ and $\tilde{V}^{3}$ each has $2-\varepsilon_{3}(2\ell-1)+\varepsilon_{3}(\ell-1)$ composition factors: one isomorphic to $L_{\tilde{L}}(2\tilde{\omega}_{2})$ and $1-\varepsilon_{3}(2\ell-1)+\varepsilon_{3}(\ell-1)$ to $L_{\tilde{L}}(0)$; and that $\tilde{V}^{2}$ has exactly $2-\varepsilon_{3}(2\ell-1)+\varepsilon_{3}(2\ell-3)$ composition factors: one isomorphic to $L_{\tilde{L}}(\tilde{\omega}_{2}+\tilde{\omega}_{3})$ and $1-\varepsilon_{3}(2\ell-1)+\varepsilon_{3}(2\ell-3)$ to $L_{\tilde{L}}(\tilde{\omega}_{2})$.

We start with the semisimple elements. Let $\tilde{s}\in \tilde{T}\setminus \ZG(\tilde{G})$. If $\dim(\tilde{V}^{i}_{\tilde{s}}(\tilde{\mu}))=\dim(\tilde{V}^{i})$ for some eigenvalue $\tilde{\mu}$ of $\tilde{s}$ on $\tilde{V}$, where $0\leq i\leq 4$, then $\tilde{s}\in \ZG(\tilde{L})^{\circ}\setminus \ZG(\tilde{G})$. In this case, as $\tilde{s}$ acts on each $\tilde{V}^{i}$ as scalar multiplication by $c^{4-2i}$ and $c^{2}\neq 1$, it follows that $\dim(\tilde{V}_{\tilde{s}}(\tilde{\mu}))\leq \binom{2\ell}{3}+2\ell-2-(2\ell-2)\varepsilon_{3}(2\ell-1)$, where equality holds for $c^{2}=-1$ and $\tilde{\mu}=-1$. We thus assume that $\dim(\tilde{V}^{i}_{\tilde{s}}(\tilde{\mu}))<\dim(\tilde{V}^{i})$ for all eigenvalues $\tilde{\mu}$ of $\tilde{s}$ on $\tilde{V}$ and all $0\leq i\leq 4$. We write $\tilde{s}=\tilde{z}\cdot \tilde{h}$, where $\tilde{z}\in \ZG(\tilde{L})^{\circ}$ and $\tilde{h}\in [\tilde{L},\tilde{L}]$, and, by the structure of $\tilde{V}\mid_{[\tilde{L},\tilde{L}]}$ and Propositions \ref{PropositionDlnatural} and \ref{PropositionDlsymm}, it follows that $\dim(\tilde{V}_{\tilde{s}}(\tilde{\mu}))\leq (2-2\varepsilon_{3}(2\ell-1)+2\varepsilon_{3}(\ell-1))\dim((L_{\tilde{L}}(0))_{\tilde{h}}(\tilde{\mu}_{\tilde{h}}))+(3-\varepsilon_{3}(2\ell-1)+\varepsilon_{3}(2\ell-3))\dim((L_{\tilde{L}}(\tilde{\omega}_{2}))_{\tilde{h}}(\tilde{\mu}_{\tilde{h}}))+2\dim((L_{\tilde{L}}(2\tilde{\omega}_{2}))_{\tilde{h}}(\tilde{\mu}_{\tilde{h}}))+\dim((L_{\tilde{L}}(\tilde{\omega}_{2}+\tilde{\omega}_{3}))_{\tilde{h}}(\tilde{\mu}_{\tilde{h}}))\leq 4(\ell-1)^{2}+2-2(\ell-1)\varepsilon_{3}(2\ell-1)+[2(\ell-1)-2]\varepsilon_{3}(2\ell-3)+\dim((L_{\tilde{L}}(\tilde{\omega}_{2}+\tilde{\omega}_{3}))_{\tilde{h}}(\tilde{\mu}_{\tilde{h}}))$. 
Recursively and using Proposition \ref{D4om1+om2p=3} for the base case of $\ell=5$, we determine that $\dim(\tilde{V}_{\tilde{s}}(\tilde{\mu}))\leq \binom{2\ell}{3}+2\ell-2-(2\ell-2)\varepsilon_{3}(2\ell-1)$. Therefore, $\scale[0.9]{\displaystyle \max_{\tilde{s}\in \tilde{T}\setminus\ZG(\tilde{G})}\dim(\tilde{V}_{\tilde{s}}(\tilde{\mu}))}=\binom{2\ell}{3}+2\ell-2-(2\ell-2)\varepsilon_{3}(2\ell-1)$.

For the unipotent elements, by Lemma \ref{uniprootelems}, the structure of $\tilde{V}\mid_{[\tilde{L},\tilde{L}]}$ and Propositions \ref{PropositionDlnatural} and \ref{PropositionDlsymm}, we have $\scale[0.9]{\displaystyle \max_{\tilde{u}\in \tilde{G}_{u}\setminus \{1\}}\dim(\tilde{V}_{\tilde{u}}(1))=}$ $\dim(\tilde{V}_{x_{\tilde{\alpha}_{\ell}}(1)}(1))\leq 4(\ell-1)^{2}-2-2(\ell-1)\varepsilon_{3}(2\ell-1)+[2(\ell-1)-2]\varepsilon_{3}(2\ell-3)+\dim((L_{\tilde{L}}(\tilde{\omega}_{2}+\tilde{\omega}_{3}))_{x_{\tilde{\alpha}_{\ell}}(1)}(1))$. Recursively and using Proposition \ref{D4om1+om2p=3} for the base case of $\ell=5$, we determine that $\dim(\tilde{V}_{x_{\tilde{\alpha}_{\ell}}(1)}(1))\leq \binom{2\ell}{3}-2\ell+6-(2\ell-2)\varepsilon_{3}(2\ell-1)$. Therefore $\scale[0.9]{\displaystyle \max_{\tilde{u}\in \tilde{G}_{u}\setminus \{1\}}\dim(\tilde{V}_{\tilde{u}}(1))}$ $\leq \binom{2\ell}{3}-2\ell+6-(2\ell-2)\varepsilon_{3}(2\ell-1)$. We conclude that $\nu_{\tilde{G}}(\tilde{V})=4\ell^{2}-6\ell+2-2\varepsilon_{3}(2\ell-1)$.
\end{proof}

\begin{prop}\label{D4om1+om2pneq3}
Let $p\neq 3$, $\ell=4$ and $\tilde{V}=L_{\tilde{G}}(\tilde{\omega}_{1}+\tilde{\omega}_{2})$. Then $\nu_{\tilde{G}}(\tilde{V})\geq 72-2\varepsilon_{p}(7)-10\varepsilon_{p}(2)$, where equality holds for $p\neq 2$. Moreover, we have $\scale[0.9]{\displaystyle \max_{\tilde{u}\in \tilde{G}_{u}\setminus \{1\}}\dim(\tilde{V}_{\tilde{u}}(1))}\leq 86-6\varepsilon_{p}(7)+12\varepsilon_{p}(2)$ and $\scale[0.9]{\displaystyle \max_{\tilde{s}\in \tilde{T}\setminus\ZG(\tilde{G})}\dim(\tilde{V}_{\tilde{s}}(\tilde{\mu}))}= 88-6\varepsilon_{p}(7)-12\varepsilon_{p}(2)$.
\end{prop}

\begin{proof}
Let $\tilde{\lambda}=\tilde{\omega}_{1}+\tilde{\omega}_{2}$ and $\tilde{L}=\tilde{L}_{1}$. By Lemma \ref{weightlevelDl}, we have $e_{1}(\tilde{\lambda})=4$, therefore $\displaystyle \tilde{V}\mid_{[\tilde{L},\tilde{L}]}=\tilde{V}^{0}\oplus \cdots\oplus \tilde{V}^{4}$. By \cite[Proposition]{Smith_82} and Lemma \ref{dualitylemma}, we have $\tilde{V}^{0}\cong L_{\tilde{L}}(\tilde{\omega}_{2})$ and $\tilde{V}^{4}\cong L_{\tilde{L}}(\tilde{\omega}_{2})$. The weight $\displaystyle (\tilde{\lambda}-\tilde{\alpha}_{1})\mid_{\tilde{T}_{1}}=2\tilde{\omega}_{2}$ admits a maximal vector in $\tilde{V}^{1}$, thus $\tilde{V}^{1}$ has a composition factor isomorphic to $L_{\tilde{L}}(2\tilde{\omega}_{2})$. Moreover, the weight $\displaystyle (\tilde{\lambda}-\tilde{\alpha}_{1}-\tilde{\alpha}_{2})\mid_{\tilde{T}_{1}}=\tilde{\omega}_{3}+\tilde{\omega}_{4}$ occurs with multiplicity $2$ and is a sub-dominant weight in the composition factor of $\tilde{V}^{1}$ isomorphic to $L_{\tilde{L}}(2\tilde{\omega}_{2})$, in which it has multiplicity $1$, if and only if $p\neq 2$. Further, we note that the weight $\displaystyle (\tilde{\lambda}-\tilde{\alpha}_{1}-2\tilde{\alpha}_{2}-\tilde{\alpha}_{3}-\tilde{\alpha}_{4})\mid_{\tilde{T}_{1}}=0$ occurs with multiplicity $6-\varepsilon_{p}(7)$ in $\tilde{V}^{1}$. Similarly, the weight $\displaystyle (\tilde{\lambda}-2\tilde{\alpha}_{1}-\tilde{\alpha}_{2})\mid_{\tilde{T}_{1}}=\tilde{\omega}_{2}+\tilde{\omega}_{3}+\tilde{\omega}_{4}$ admits a maximal vector in $\tilde{V}^{2}$, thus $\tilde{V}^{2}$ has a composition factor isomorphic to $L_{\tilde{L}}(\tilde{\omega}_{2}+\tilde{\omega}_{3}+\tilde{\omega}_{4})$. Moreover, the weight $\displaystyle (\tilde{\lambda}-2\tilde{\alpha}_{1}-2\tilde{\alpha}_{2}-\tilde{\alpha}_{3}-\tilde{\alpha}_{4})\mid_{\tilde{T}_{1}}=\tilde{\omega}_{2}$ occurs with multiplicity $6-\varepsilon_{p}(7)$ and is a sub-dominant weight in the composition factor of $\tilde{V}^{2}$ isomorphic to $L_{\tilde{L}}(\tilde{\omega}_{2}+\tilde{\omega}_{3}+\tilde{\omega}_{4})$, in which it has multiplicity $4-\varepsilon_{p}(5)$. By dimensional considerations, we determine that $\tilde{V}^{2}$ has exactly $3-\varepsilon_{p}(7)+\varepsilon_{p}(5)$ composition factors: one isomorphic to $L_{\tilde{L}}(\tilde{\omega}_{2}+\tilde{\omega}_{3}+\tilde{\omega}_{4})$ and $2-\varepsilon_{p}(7)+\varepsilon_{p}(5)$ to $L_{\tilde{L}}(\tilde{\omega}_{2})$. Further, $\tilde{V}^{1}$ and $\tilde{V}^{3}$ each has $3-\varepsilon_{p}(7)+2\varepsilon_{p}(2)$ composition factors: one isomorphic to $L_{\tilde{L}}(2\tilde{\omega}_{2})$, $1+\varepsilon_{p}(2)$ to $L_{\tilde{L}}(\tilde{\omega}_{3}+\tilde{\omega}_{4})$ and $1-\varepsilon_{p}(7)+\varepsilon_{p}(2)$ to $L_{\tilde{L}}(0)$. 

We begin with the semisimple elements. Let $\tilde{s}\in \tilde{T}\setminus \ZG(\tilde{G})$. If $\dim(\tilde{V}^{i}_{\tilde{s}}(\tilde{\mu}))=\dim(\tilde{V}^{i})$ for some eigenvalue $\tilde{\mu}$ of $\tilde{s}$ on $\tilde{V}$, where $0\leq i\leq 4$, then $\tilde{s}\in \ZG(\tilde{L})^{\circ}\setminus \ZG(\tilde{G})$. In this case, as $\tilde{s}$ acts on each $\tilde{V}^{i}$ as scalar multiplication by $c^{4-2i}$ and $c^{2}\neq 1$, it follows that $\dim(\tilde{V}_{\tilde{s}}(\tilde{\mu}))\leq 88-6\varepsilon_{p}(7)-12\varepsilon_{p}(2)$ for all eigenvalues $\tilde{\mu}$ of $\tilde{s}$ on $\tilde{V}$. We thus assume that $\dim(\tilde{V}^{i}_{\tilde{s}}(\tilde{\mu}))<\dim(\tilde{V}^{i})$ for all eigenvalues $\tilde{\mu}$ of $\tilde{s}$ on $\tilde{V}$ and all $0\leq i\leq 4$. We write $\tilde{s}=\tilde{z}\cdot \tilde{h}$, where $\tilde{z}\in \ZG(\tilde{L})^{\circ}$ and $\tilde{h}\in [\tilde{L},\tilde{L}]$, and we have $\scale[0.9]{\displaystyle \dim(\tilde{V}_{\tilde{s}}(\tilde{\mu}))\leq \sum_{i=0}^{4}\dim(\tilde{V}^{i}_{\tilde{h}}(\tilde{\mu}^{i}_{\tilde{h}}))}$, where $\dim(\tilde{V}^{i}_{\tilde{h}}(\tilde{\mu}^{i}_{\tilde{h}}))<\dim(\tilde{V}^{i})$ for all eigenvalues $\tilde{\mu}^{i}_{\tilde{h}}$ of $\tilde{h}$ on $\tilde{V}^{i}$. Further, by the structure of $\tilde{V}\mid_{[\tilde{L},\tilde{L}]}$, we also have $\dim(\tilde{V}_{\tilde{s}}(\tilde{\mu}))\leq (4-\varepsilon_{p}(7)+\varepsilon_{p}(5))\dim((L_{\tilde{L}}(\tilde{\omega}_{2}))_{\tilde{h}}(\tilde{\mu}_{\tilde{h}}))+2\dim((L_{\tilde{L}}(2\tilde{\omega}_{2}))_{\tilde{h}}(\tilde{\mu}_{\tilde{h}}))+(2+2\varepsilon_{p}(2))\dim((L_{\tilde{L}}(\tilde{\omega}_{3}+\tilde{\omega}_{4}))_{\tilde{h}}(\tilde{\mu}_{\tilde{h}}))+(2-2\varepsilon_{p}(7)+2\varepsilon_{p}(2))\dim((L_{\tilde{L}}(0))_{\tilde{h}}(\tilde{\mu}_{\tilde{h}}))+\dim((L_{\tilde{L}}(\tilde{\omega}_{2}+\tilde{\omega}_{3}+\tilde{\omega}_{4}))_{\tilde{h}}(\tilde{\mu}_{\tilde{h}}))$. To begin, in view of Proposition \ref{PropositionAlwedge}, assume that $\dim((L_{\tilde{L}}(\tilde{\omega}_{2}))_{\tilde{h}}(\tilde{\mu}_{\tilde{h}}))=4$ for some eigenvalue $\tilde{\mu}_{\tilde{h}}$ of $\tilde{h}$ on $L_{\tilde{L}}(\tilde{\omega}_{2})$. Then, by the same result, we have that, up to conjugation, $\tilde{h}=h_{\tilde{\alpha}_{2}}(d^{2})h_{\tilde{\alpha}_{3}}(d)h_{\tilde{\alpha}_{4}}(d)$ with $d^{2}\neq 1$, or $\tilde{h}=h_{\tilde{\alpha}_{2}}(d^{2})h_{\tilde{\alpha}_{3}}(d)h_{\tilde{\alpha}_{4}}(-d)$ with $d^{2}\neq -1$. In the first, respectively second, case, using the weight structure of $\tilde{V}^{1}$, we see that the eigenvalues of $\tilde{h}$ on $\tilde{V}^{1}$, not necessarily distinct, are $1$ with multiplicity $18-\varepsilon_{p}(7)$; $d^{\pm 4}$ each with multiplicity $1$; and $d^{\pm 2}$, respectively $-d^{\pm 2}$, each with multiplicity $8$. As $d^{2}\neq 1$, respectively $d^{2}\neq -1$, it follows that $\dim(\tilde{V}^{1}_{\tilde{h}}(\tilde{\mu}^{1}_{\tilde{h}}))\leq 20-\varepsilon_{p}(7)-2\varepsilon_{p}(2)$ for all eigenvalues $\tilde{\mu}^{1}_{\tilde{h}}$ of $\tilde{h}$ on $\tilde{V}^{1}$. Similarly, we see that the eigenvalues of $\tilde{h}$ on $\tilde{V}^{2}$, not necessarily distinct, are $1$, respectively $-1$, with multiplicity $32-4\varepsilon_{p}(7)$; $d^{\pm 4}$, respectively $-d^{\pm 4}$, each with multiplicity $4$; and $d^{\pm 2}$ each with multiplicity $14-\varepsilon_{p}(7)$. As $d^{2}\neq 1$, respectively $d^{2}\neq -1$, it follows that $\dim(\tilde{V}^{2}_{\tilde{h}}(\tilde{\mu}^{2}_{\tilde{h}}))\leq 40-4\varepsilon_{p}(7)-8\varepsilon_{p}(2)$ for all eigenvalues $\tilde{\mu}^{2}_{\tilde{h}}$ of $\tilde{h}$ on $\tilde{V}^{2}$. Therefore, as $\tilde{V}^{3}\cong (\tilde{V}^{1})^{*}$, we see that $\dim(\tilde{V}_{\tilde{s}}(\tilde{\mu}))\leq 88-6\varepsilon_{p}(7)-12\varepsilon_{p}(2)$ for all eigenvalues $\tilde{\mu}$ of $\tilde{s}$ on $\tilde{V}$. Secondly, we assume that $\dim((L_{\tilde{L}}(\tilde{\omega}_{2}))_{\tilde{h}}(\tilde{\mu}_{\tilde{h}}))=3$ for some eigenvalue $\tilde{\mu}_{\tilde{h}}$ of $\tilde{h}$ on $L_{\tilde{L}}(\tilde{\omega}_{2})$. In this case, using \eqref{Al_enum_P2}, one shows that, up to conjugation, $\tilde{h}=h_{\tilde{\alpha}_{2}}(d^{2})h_{\tilde{\alpha}_{3}}(d)h_{\tilde{\alpha}_{4}}(d^{3})$ with $d^{4}\neq 1$. Then, the eigenvalues of $\tilde{h}$ on $\tilde{V}^{1}$, not necessarily distinct, are $d^{\pm 4}$ each with multiplicity $9$; and $1$ with multiplicity $18-\varepsilon_{p}(7)$. As $d^{4}\neq 1$, we have $\dim(\tilde{V}^{1}_{\tilde{h}}(\tilde{\mu}^{1}_{\tilde{h}}))\leq 18$ for all eigenvalues $\tilde{\mu}^{1}_{\tilde{h}}$ of $\tilde{h}$ on $\tilde{V}^{1}$. Similarly,  the eigenvalues of $\tilde{h}$ on $\tilde{V}^{2}$, not necessarily distinct, are $d^{\pm 6}$ each with multiplicity $7$; and $d^{\pm 2}$ each with multiplicity $27-3\varepsilon_{p}(7)$. As $d^{2}\neq 1$, we have $\dim(\tilde{V}^{2}_{\tilde{h}}(\tilde{\mu}^{2}_{\tilde{h}}))\leq 34-3\varepsilon_{p}(7)-7\varepsilon_{p}(2)$ for all eigenvalues $\tilde{\mu}^{2}_{\tilde{h}}$ of $\tilde{h}$ on $\tilde{V}^{2}$. Consequently, as $\tilde{V}^{3}\cong (\tilde{V}^{1})^{*}$, we see that $\dim(\tilde{V}_{\tilde{s}}(\tilde{\mu}))\leq 76-3\varepsilon_{p}(7)-7\varepsilon_{p}(2)$ for all eigenvalues $\tilde{\mu}$ of $\tilde{s}$ on $\tilde{V}$. Lastly, we assume that $\dim((L_{\tilde{L}}(\tilde{\omega}_{2}))_{\tilde{h}}(\tilde{\mu}_{\tilde{h}}))\leq 2$ for all eigenvalues $\tilde{\mu}_{\tilde{h}}$ of $\tilde{h}$ on $L_{\tilde{L}}(\tilde{\omega}_{2})$. Then, in particular, $\tilde{h}$ is not conjugate to $\tilde{h}=h_{\tilde{\alpha}_{2}}(d^{2})h_{\tilde{\alpha}_{3}}(d)h_{\tilde{\alpha}_{4}}(d^{3})$ with $d^{4}\neq 1$, thereby $\dim((L_{\tilde{L}}(\tilde{\omega}_{3}+\tilde{\omega}_{4}))_{\tilde{h}}(\tilde{\mu}_{\tilde{h}}))\leq 8-\varepsilon_{p}(2)$ for all eigenvalues $\tilde{\mu}_{\tilde{h}}$ of $\tilde{h}$ on $L_{\tilde{L}}(\tilde{\omega}_{3}+\tilde{\omega}_{4})$, see Proposition \ref{PropositionAlom1+oml}. Further, using Proposition \ref{A32om2}, we deduce that $\dim(\tilde{V}_{\tilde{s}}(\tilde{\mu}))\leq 50-4\varepsilon_{p}(7)+2\varepsilon_{p}(5)-6\varepsilon_{p}(2)+\dim((L_{\tilde{L}}(\tilde{\omega}_{2}+\tilde{\omega}_{3}+\tilde{\omega}_{4}))_{\tilde{h}}(\tilde{\mu}_{\tilde{h}}))$ for all eigenvalues $\tilde{\mu}$ of $\tilde{s}$ on $\tilde{V}$. Now, if $p\neq 2$, then by Proposition \ref{A3om1+om2+om3}, we get $\dim(\tilde{V}_{\tilde{s}}(\tilde{\mu}))\leq 86-4\varepsilon_{p}(7)-2\varepsilon_{p}(5)$ for all eigenvalues $\tilde{\mu}$ of $\tilde{s}$ on $\tilde{V}$. We thus assume $p=2$ and let $\tilde{L}^{1}_{1}$ be a Levi subgroup of the maximal parabolic subgroup $\tilde{P}^{1}_{1}$ of $\tilde{L}$ corresponding to $\tilde{\Delta}^{1}_{1}=\tilde{\Delta}_{1}\setminus\{\tilde{\alpha}_{3}\}$. We once again abuse notation and denote by $\tilde{\omega}_{2}$ and $\tilde{\omega}_{4}$ the fundamental dominant weights of $\tilde{L}^{1}_{1}$ corresponding to the simple roots $\tilde{\alpha}_{2}$ and $\tilde{\alpha}_{4}$. We saw in the proof of Proposition \ref{A3om1+om2+om3} that if $\tilde{h}\in \ZG(\tilde{L}^{1}_{1})^{\circ}\setminus \ZG(\tilde{L})$, then $\dim((L_{\tilde{L}}(\tilde{\omega}_{2}+\tilde{\omega}_{3}+\tilde{\omega}_{4}))_{\tilde{h}}(\tilde{\mu}_{\tilde{h}}))\leq 24$ for all eigenvalues $\tilde{\mu}_{\tilde{h}}$ of $\tilde{h}$, and so $\dim(\tilde{V}_{\tilde{s}}(\tilde{\mu}))\leq 68$ for all eigenvalues $\tilde{\mu}$ of $\tilde{s}$ on $\tilde{V}$. On the other hand, if $\tilde{h}\notin \ZG(\tilde{L}^{1}_{1})^{\circ}\setminus \ZG(\tilde{L})$, then $\dim((L_{\tilde{L}}(\tilde{\omega}_{2}+\tilde{\omega}_{3}+\tilde{\omega}_{4}))_{\tilde{h}}(\tilde{\mu}_{\tilde{h}}))\leq 2\dim((L_{\tilde{L}^{1}_{1}}(\tilde{\omega}_{2}+\tilde{\omega}_{4}))_{\tilde{h}^{1}_{1}}(\tilde{\mu}_{\tilde{h}^{1}_{1}}))+2\dim((L_{\tilde{L}^{1}_{1}}(2\tilde{\omega}_{2}+\tilde{\omega}_{4}))_{\tilde{h}^{1}_{1}}(\tilde{\mu}_{\tilde{h}^{1}_{1}}))+10\dim((L_{\tilde{L}^{1}_{1}}(\tilde{\omega}_{2}))_{\tilde{h}^{1}_{1}}(\tilde{\mu}_{\tilde{h}^{1}_{1}}))$, where $\tilde{h}=\tilde{z}^{1}_{1}\cdot \tilde{h}^{1}_{1}$ with $\tilde{z}^{1}_{1}\in \ZG(\tilde{L}^{1}_{1})^{\circ}$ and $\tilde{h}^{1}_{1}\in [\tilde{L}^{1}_{1}, \tilde{L}^{1}_{1}]$. In view of Proposition \ref{PropositionAlnatural}, assume $\dim((L_{\tilde{L}^{1}_{1}}(\tilde{\omega}_{2}))_{\tilde{h}^{1}_{1}}(\tilde{\mu}_{\tilde{h}^{1}_{1}}))=2$ for some eigenvalue $\tilde{\mu}_{\tilde{h}^{1}_{1}}$ of $\tilde{h}^{1}_{1}$ on $L_{\tilde{L}^{1}_{1}}(\tilde{\omega}_{2})$. Then, by the same result, we have that, up to conjugation, $\tilde{h}^{1}_{1}=h_{\tilde{\alpha}_{2}}(d)h_{\tilde{\alpha}_{4}}(d^{2})$ with $d^{3}\neq 1$. Using the weight structure of $L_{\tilde{L}}(\tilde{\omega}_{2}+\tilde{\omega}_{3}+\tilde{\omega}_{4})$, one shows that the eigenvalues of $\tilde{h}$ on $L_{\tilde{L}}(\tilde{\omega}_{2}+\tilde{\omega}_{3}+\tilde{\omega}_{4})$, not necessarily distinct, are $c^{\pm 6}$ each with multiplicity $4$; $c^{\pm 6}d^{\pm 3}$ each with multiplicity $2$; $c^{2}d$ and $c^{-2}d^{-1}$ each with multiplicity $10$; $c^{2}d^{-2}$ and $c^{-2}d^{2}$ each with multiplicity $8$; $c^{2}d^{4}$ and $c^{-2}d^{-4}$ each with multiplicity $4$; and $c^{2}d^{-5}$ and $c^{-2}d^{5}$ each with multiplicity $2$, where $c\in k^{*}$. As $d^{3}\neq 1$, it follows that $\dim((L_{\tilde{L}}(\tilde{\omega}_{2}+\tilde{\omega}_{3}+\tilde{\omega}_{4}))_{\tilde{h}}(\tilde{\mu}_{\tilde{h}}))\leq 28$ for all eigenvalues $\tilde{\mu}_{\tilde{h}}$ of $\tilde{h}$ on $L_{\tilde{L}}(\tilde{\omega}_{2}+\tilde{\omega}_{3}+\tilde{\omega}_{4})$, therefore $\dim(\tilde{V}_{\tilde{s}}(\tilde{\mu}))\leq 72$ for all eigenvalues $\tilde{\mu}$ of $\tilde{s}$ on $\tilde{V}$. Similarly, if $\dim((L_{\tilde{L}^{1}_{1}}(\tilde{\omega}_{2}))_{\tilde{h}^{1}_{1}}(\tilde{\mu}_{\tilde{h}^{1}_{1}}))=1$ for all eigenvalues $\tilde{\mu}_{\tilde{h}^{1}_{1}}$ of $\tilde{h}^{1}_{1}$ on $L_{\tilde{L}^{1}_{1}}(\tilde{\omega}_{2})$, then $\dim((L_{\tilde{L}}(\tilde{\omega}_{2}+\tilde{\omega}_{3}+\tilde{\omega}_{4}))_{\tilde{h}}(\tilde{\mu}_{\tilde{h}}))\leq 26$ for all eigenvalues $\tilde{\mu}_{\tilde{h}}$ of $\tilde{h}$ on $L_{\tilde{L}}(\tilde{\omega}_{2}+\tilde{\omega}_{3}+\tilde{\omega}_{4})$, therefore $\dim(\tilde{V}_{\tilde{s}}(\tilde{\mu}))\leq 70$ for all eigenvalues $\tilde{\mu}$ of $\tilde{s}$ on $\tilde{V}$. In conclusion, we have proven that $\scale[0.9]{\displaystyle \max_{\tilde{s}\in \tilde{T}\setminus\ZG(\tilde{G})}\dim(\tilde{V}_{\tilde{s}}(\tilde{\mu}))}= 88-6\varepsilon_{p}(7)-12\varepsilon_{p}(2)$.

For the unipotent elements, by Lemma \ref{uniprootelems}, the structure of $\tilde{V}\mid_{[\tilde{L},\tilde{L}]}$ and Propositions \ref{PropositionAlwedge}, \ref{PropositionAlom1+oml}, \ref{A3om1+om2+om3} and \ref{A32om2}, we have $\scale[0.9]{\displaystyle \max_{\tilde{u}\in \tilde{G}_{u}\setminus \{1\}}\dim(\tilde{V}_{\tilde{u}}(1))=}$ $\dim(\tilde{V}_{x_{\tilde{\alpha}_{4}}(1)}(1))\leq 86-6\varepsilon_{p}(7)+12\varepsilon_{p}(2)$. Lastly, we note that $\nu_{\tilde{G}}(\tilde{V})\geq 72-2\varepsilon_{p}(7)-10\varepsilon_{p}(2)$, where equality holds for $p\neq 2$.
\end{proof}

\begin{prop}\label{Dlom1+om2pneq3}
Let $p\neq 3$, $\ell\geq 5$ and $\tilde{V}=L_{\tilde{G}}(\tilde{\omega}_{1}+\tilde{\omega}_{2})$. Then $\nu_{\tilde{G}}(\tilde{V})\geq 8(\ell-1)^{2}-2\varepsilon_{p}(2\ell-1)-(4\ell-6)\varepsilon_{p}(2)$, where equality holds for $p\neq 2$. Moreover, we have $\scale[0.9]{\displaystyle \max_{\tilde{u}\in \tilde{G}_{u}\setminus \{1\}}\dim(\tilde{V}_{\tilde{u}}(1))}\leq 2^{4}\binom{\ell}{3}+8\ell-10-(2\ell-2)\varepsilon_{p}(2\ell-1)+(4\ell-4)\varepsilon_{p}(2)$ and $\scale[0.9]{\displaystyle \max_{\tilde{s}\in \tilde{T}\setminus\ZG(\tilde{G})}\dim(\tilde{V}_{\tilde{s}}(\tilde{\mu}))}=2^{4}\binom{\ell}{3}+8\ell-8-(2\ell-2)\varepsilon_{p}(2\ell-1)-(4\ell-4)\varepsilon_{p}(2)$.
\end{prop}

\begin{proof}
Let $\tilde{\lambda}=\tilde{\omega}_{1}+\tilde{\omega}_{2}$ and $\tilde{L}=\tilde{L}_{1}$. By Lemma \ref{weightlevelDl}, we have $e_{1}(\tilde{\lambda})=4$, therefore $\displaystyle \tilde{V}\mid_{[\tilde{L},\tilde{L}]}=\tilde{V}^{0}\oplus \cdots\oplus \tilde{V}^{4}$. We argue as in the proof of Proposition \ref{D4om1+om2pneq3} to show that  $\tilde{V}^{0}\cong L_{\tilde{L}}(\tilde{\omega}_{2})$ and $\tilde{V}^{4}\cong L_{\tilde{L}}(\tilde{\omega}_{2})$; $\tilde{V}^{1}$ and $\tilde{V}^{3}$ each has $\scale[0.95]{3-\varepsilon_{p}(2\ell-1)+\varepsilon_{p}(\ell-1)+\varepsilon_{p}(2)+(1+\varepsilon_{2}(\ell-1))\varepsilon_{p}(2)}$ composition factors: one isomorphic to $L_{\tilde{L}}(2\tilde{\omega}_{2})$, $1+\varepsilon_{p}(2)$ to $L_{\tilde{L}}(\tilde{\omega}_{3})$ and $\scale[0.95]{1-\varepsilon_{p}(2\ell-1)+\varepsilon_{p}(\ell-1)+[1+\varepsilon_{2}(\ell-1)]\varepsilon_{p}(2)}$ to $L_{\tilde{L}}(0)$; and $\tilde{V}^{2}$ has $3-\varepsilon_{p}(2\ell-1)+\varepsilon_{p}(2\ell-3)$ composition factors: one isomorphic to $L_{\tilde{L}}(\tilde{\omega}_{2}+\tilde{\omega}_{3})$ and $\scale[0.95]{2-\varepsilon_{p}(2\ell-1)+\varepsilon_{p}(2\ell-3)}$ to $L_{\tilde{L}}(\tilde{\omega}_{2})$.

We begin with the semisimple elements. Let $\tilde{s}\in \tilde{T}\setminus \ZG(\tilde{G})$. If $\dim(\tilde{V}^{i}_{\tilde{s}}(\tilde{\mu}))=\dim(\tilde{V}^{i})$ for some eigenvalue $\tilde{\mu}$ of $\tilde{s}$ on $\tilde{V}$, where $0\leq i\leq 4$, then $\tilde{s}\in \ZG(\tilde{L})^{\circ}\setminus \ZG(\tilde{G})$. In this case, as $\tilde{s}$ acts on each $\tilde{V}^{i}$ as scalar multiplication by $c^{4-2i}$ and $c^{2}\neq 1$, it follows that $\dim(\tilde{V}_{\tilde{s}}(\tilde{\mu}))\leq 2^{4}\binom{\ell}{3}+8\ell-8-(2\ell-2)\varepsilon_{p}(2\ell-1)-(4\ell-4)\varepsilon_{p}(2)$, where equality holds for $p\neq 2$, $c^{2}=-1$ and $\tilde{\mu}=1$, respectively for $p=2$ and $\tilde{\mu}=1$. We thus assume that $\dim(\tilde{V}^{i}_{\tilde{s}}(\tilde{\mu}))<\dim(\tilde{V}^{i})$ for all eigenvalues $\tilde{\mu}$ of $\tilde{s}$ on $\tilde{V}$ and all $0\leq i\leq 4$. We write $\tilde{s}=\tilde{z}\cdot \tilde{h}$, where $\tilde{z}\in \ZG(\tilde{L})^{\circ}$ and $\tilde{h}\in [\tilde{L},\tilde{L}]$, and, by the structure of $\tilde{V}\mid_{[\tilde{L},\tilde{L}]}$ and Propositions \ref{PropositionDlnatural}, \ref{PropositionDlwedge} and \ref{PropositionDlsymm}, it follows that $\dim(\tilde{V}_{\tilde{s}}(\tilde{\mu})) \leq (4-\varepsilon_{p}(2\ell-1)+\varepsilon_{p}(2\ell-3))\dim((L_{\tilde{L}}(\tilde{\omega}_{2}))_{\tilde{h}}(\tilde{\mu}_{\tilde{h}}))+2\dim((L_{\tilde{L}}(2\tilde{\omega}_{2}))_{\tilde{h}}(\tilde{\mu}_{\tilde{h}}))+(2+2\varepsilon_{p}(2))\dim((L_{\tilde{L}}(\tilde{\omega}_{3}))_{\tilde{h}}(\tilde{\mu}_{\tilde{h}}))+(2-2\varepsilon_{p}(2\ell-1)+2\varepsilon_{p}(\ell-1)+2\varepsilon_{p}(2)+2\varepsilon_{2}(\ell-1)\varepsilon_{p}(2))\dim((L_{\tilde{L}}(0))_{\tilde{h}}(\tilde{\mu}_{\tilde{h}}))+\dim((L_{\tilde{L}}(\tilde{\omega}_{2}+\tilde{\omega}_{3})_{\tilde{h}}(\tilde{\mu}_{\tilde{h}}))\leq 8(\ell-1)^{2}-8(\ell-1)+8-4\varepsilon_{p}(2)-2(\ell-1)\varepsilon_{p}(2\ell-1)+[2(\ell-1)-2]\varepsilon_{p}(2\ell-3)+\dim((L_{\tilde{L}}(\tilde{\omega}_{2}+\tilde{\omega}_{3})_{\tilde{h}}(\tilde{\mu}_{\tilde{h}}))$. Recursively and using Proposition \ref{D4om1+om2pneq3} for the base case of $\ell=5$, we determine that $\dim(\tilde{V}_{\tilde{s}}(\tilde{\mu}))\leq 2\binom{2\ell}{3}-4\ell^{2}+12\ell-8-(2\ell-2)\varepsilon_{p}(2\ell-1)-(4\ell-4)\varepsilon_{p}(2)=2^{4}\binom{\ell}{3}+8\ell-8-(2\ell-2)\varepsilon_{p}(2\ell-1)-(4\ell-4)\varepsilon_{p}(2)$ for all eigenvalues $\tilde{\mu}$ of $\tilde{s}$ on $\tilde{V}$. Therefore, $\scale[0.9]{\displaystyle \max_{\tilde{s}\in \tilde{T}\setminus\ZG(\tilde{G})}\dim(\tilde{V}_{\tilde{s}}(\tilde{\mu}))}=2^{4}\binom{\ell}{3}+8\ell-8-(2\ell-2)\varepsilon_{p}(2\ell-1)-(4\ell-4)\varepsilon_{p}(2)$.

For the unipotent elements, in view of Lemma \ref{uniprootelems}, we have $\scale[0.9]{\displaystyle \max_{\tilde{u}\in \tilde{G}_{u}\setminus \{1\}}\dim(\tilde{V}_{\tilde{u}}(1))=}$ $\dim(\tilde{V}_{x_{\tilde{\alpha}_{\ell}}(1)}(1))$. Now, by the structure of $\tilde{V}\mid_{[\tilde{L},\tilde{L}]}$ and Propositions \ref{PropositionDlnatural}, \ref{PropositionDlwedge} and \ref{PropositionDlsymm}, it follows that $\dim(\tilde{V}_{x_{\tilde{\alpha}_{\ell}}(1)}(1)) \leq 8(\ell-1)^{2}-8(\ell-1)+8+4\varepsilon_{p}(2)-2(\ell-1)\varepsilon_{p}(2\ell-1)+[2(\ell-1)-2]\varepsilon_{p}(2\ell-3)+\dim((L_{\tilde{L}}(\tilde{\omega}_{2}+\tilde{\omega}_{3})_{x_{\tilde{\alpha}_{\ell}}(1)}(1))$. Recursively and using Proposition \ref{D4om1+om2pneq3} for the base case of $\ell=5$, we determine that $\dim(\tilde{V}_{x_{\tilde{\alpha}_{\ell}}(1)}(1))\leq 2\binom{2\ell}{3}-4\ell^{2}+12\ell-10-(2\ell-2)\varepsilon_{p}(2\ell-1)+(4\ell-4)\varepsilon_{p}(2)=2^{4}\binom{\ell}{3}+8\ell-10-(2\ell-2)\varepsilon_{p}(2\ell-1)+(4\ell-4)\varepsilon_{p}(2)$. Therefore, $\scale[0.9]{\displaystyle \max_{\tilde{u}\in \tilde{G}_{u}\setminus \{1\}}\dim(\tilde{V}_{\tilde{u}}(1))}\leq 2^{4}\binom{\ell}{3}+8\ell-10-(2\ell-2)\varepsilon_{p}(2\ell-1)+(4\ell-4)\varepsilon_{p}(2)$. Lastly, we note that $\nu_{\tilde{G}}(\tilde{V})\geq 8(\ell-1)^{2}-2\varepsilon_{p}(2\ell-1)-(4\ell-6)\varepsilon_{p}(2)$, where equality holds for $p\neq 2$.
\end{proof}

\begin{prop}\label{Dloml}
Let $\ell\geq 5$ and $\tilde{V}=L_{\tilde{G}}(\tilde{\omega}_{\ell})$. Then $\nu_{\tilde{G}}(\tilde{V})=2^{\ell-3}$. Moreover, we have $\scale[0.88]{\displaystyle \max_{\tilde{u}\in \tilde{G}_{u}\setminus \{1\}}\dim(\tilde{V}_{\tilde{u}}(1))}=3\cdot 2^{\ell-3}$ and $\scale[0.9]{\displaystyle \max_{\tilde{s}\in \tilde{T}\setminus\ZG(\tilde{G})}\dim(\tilde{V}_{\tilde{s}}(\tilde{\mu}))}=5\cdot 2^{\ell-4}$.
\end{prop}

\begin{proof}
Set $\tilde{\lambda}=\tilde{\omega}_{\ell}$ and $\tilde{L}=\tilde{L}_{1}$. By Lemma \ref{weightlevelDl}, we have $e_{1}(\tilde{\lambda})=1$, therefore $\displaystyle \tilde{V}\mid _{[\tilde{L},\tilde{L}]}=\tilde{V}^{0}\oplus \tilde{V}^{1}$. By \cite[Proposition]{Smith_82}, we have $\tilde{V}^{0}\cong L_{\tilde{L}}(\tilde{\omega}_{\ell})$. Now, the weight $\displaystyle (\tilde{\lambda}- \tilde{\alpha}_{1}-\cdots-\tilde{\alpha}_{\ell-2}-\tilde{\alpha}_{\ell})\mid_{\tilde{T}_{1}}=\tilde{\omega}_{\ell-1}$ admits a maximal vector in $\tilde{V}^{1}$, thus $\tilde{V}^{1}$ has a composition factor isomorphic to $L_{\tilde{L}}(\tilde{\omega}_{\ell-1})$. As $\dim(\tilde{V}^{2})=2^{\ell-2}$, it follows that
\begin{equation}\label{DecompVDloml-1}
\tilde{V}\mid_{[\tilde{L},\tilde{L}]}\cong L_{\tilde{L}}(\tilde{\omega}_{\ell}) \oplus L_{\tilde{L}}(\tilde{\omega}_{\ell-1}).
\end{equation}

We start with the semisimple elements. Let $\tilde{s}\in \tilde{T}\setminus \ZG(\tilde{G})$. If $\dim(\tilde{V}^{i}_{\tilde{s}}(\tilde{\mu}))=\dim(\tilde{V}^{i})$ for some eigenvalue $\tilde{\mu}$ of $\tilde{s}$ on $\tilde{V}$, where $i=0,1$, then $\tilde{s}\in \ZG(\tilde{L})^{\circ}\setminus \ZG(\tilde{G})$. In this case, as $\tilde{s}$ acts as scalar multiplication by $c^{1-2i}$ on $\tilde{V}^{i}$ and $c^{2}\neq 1$, it follows that $\dim(\tilde{V}_{\tilde{s}}(\tilde{\mu}))\leq 2^{\ell-2}$ for all eigenvalues $\tilde{\mu}$ of $\tilde{s}$ on $\tilde{V}$. We thus assume that $\dim(\tilde{V}^{i}_{\tilde{s}}(\tilde{\mu}))<\dim(\tilde{V}^{i})$ for all eigenvalues $\tilde{\mu}$ of $\tilde{s}$ on $\tilde{V}$ and for both $i=0, 1$. We write $\tilde{s}=\tilde{z}\cdot \tilde{h}$, where $\tilde{z}\in \ZG(\tilde{L})^{\circ}$ and $\tilde{h}\in [\tilde{L},\tilde{L}]$. First, let $\ell=5$. Then, by \eqref{DecompVDloml-1} and Proposition \ref{PropositionDlnatural}, we get $\dim(\tilde{V}_{\tilde{s}}(\tilde{\mu}))\leq 12$ for all eigenvalues $\tilde{\mu}$ of $\tilde{s}$ on $\tilde{V}$. However, assume there exist $(\tilde{s},\tilde{\mu})\in \tilde{T}\setminus \ZG(\tilde{G})\times k^{*}$ with $\dim(\tilde{V}_{\tilde{s}}(\tilde{\mu}))>10$. Let $\tilde{\mu}_{\tilde{h}}^{i}$ be the eigenvalue of $\tilde{h}$ on $\tilde{V}^{i}$ given by $\tilde{\mu}=c^{1-2i}\tilde{\mu}_{h}^{i}$, where $i=0,2$. Since $\dim(\tilde{V}_{\tilde{s}}(\tilde{\mu}))>10$, we must have either $\dim(\tilde{V}^{0}_{\tilde{h}}(\tilde{\mu}_{\tilde{h}}^{0}))=6$, or $\dim(\tilde{V}^{1}_{\tilde{h}}(\tilde{\mu}_{\tilde{h}}^{1}))=6$. If $\dim(\tilde{V}^{0}_{\tilde{h}}(\tilde{\mu}_{\tilde{h}}^{0}))=6$, then $\tilde{\mu}_{\tilde{h}}^{0}=\pm 1$ and $\tilde{h}=h_{\tilde{\alpha}_{2}}(\pm 1)h_{\tilde{\alpha}_{4}}(d)h_{\tilde{\alpha}_{5}}(\pm d)$, where $d\neq \pm 1$, see Proposition \ref{PropositionDlnatural}. However, using the weight structure of $\tilde{V}^{1}$, we determine that, in both cases, we have $\dim(\tilde{V}^{1}_{\tilde{h}}(\tilde{\mu}_{\tilde{h}}))\leq 4$ for all eigenvalues $\tilde{\mu}_{\tilde{h}}^{1}$ of $\tilde{h}$ on $\tilde{V}^{1}$, therefore $\dim(\tilde{V}_{\tilde{s}}(\tilde{\mu}))\leq 10$ for all eigenvalues $\tilde{\mu}$ of $\tilde{s}$ on $\tilde{V}$. Similarly, if $\dim(\tilde{V}^{1}_{\tilde{h}}(\tilde{\mu}_{\tilde{h}}^{1}))=6$, arguing as before, one shows that $\dim(\tilde{V}_{\tilde{s}}(\tilde{\mu}))\leq 10$ for all eigenvalues $\tilde{\mu}$ of $\tilde{s}$ on $\tilde{V}$. It follows that $\dim(\tilde{V}_{\tilde{s}}(\tilde{\mu}))\leq 10$ for all eigenvalues $\tilde{\mu}$ of $\tilde{s}$ on $\tilde{V}$. Now, let $\ell\geq 6$. Recursively and using the result for $\ell=5$, one shows that $\dim(\tilde{V}_{\tilde{s}}(\tilde{\mu}))\leq 2\dim((L_{\tilde{L}}(\tilde{\omega}_{\ell}))_{\tilde{h}}(\tilde{\mu}_{\tilde{h}}))\leq 5\cdot 2^{\ell-4}$ for all eigenvalues $\tilde{\mu}$ of $\tilde{s}$ on $\tilde{V}$. Moreover, recursively, one shows that for $\tilde{s}=h_{\tilde{\alpha}_{\ell-1}}(d^{-1})h_{\tilde{\alpha}_{\ell}}(d)$ with $d^{2}\neq 1$, we have $\dim(\tilde{V}_{\tilde{s}}(1))=5\cdot 2^{\ell-4}$. Therefore, $\scale[0.9]{\displaystyle \max_{\tilde{s}\in \tilde{T}\setminus\ZG(\tilde{G})}\dim(\tilde{V}_{\tilde{s}}(\tilde{\mu}))}=5\cdot 2^{\ell-4}$. 

For the unipotent elements, in view of Lemma \ref{uniprootelems}, we have $\scale[0.9]{\displaystyle \max_{\tilde{u}\in \tilde{G}_{u}\setminus \{1\}}\dim(\tilde{V}_{\tilde{u}}(1))=}$ $\dim(\tilde{V}_{x_{\tilde{\alpha}_{\ell}}(1)}(1))$. First, for $\ell=5$, by \eqref{DecompVDloml-1} and Proposition \ref{PropositionDlnatural}, we get $\scale[0.9]{\displaystyle \max_{\tilde{u}\in \tilde{G}_{u}\setminus \{1\}}\dim(\tilde{V}_{\tilde{u}}(1))=}12$. Now, for $\ell\geq 6$, recursively and using the result for $\ell=5$, we deduce that $\scale[0.9]{\displaystyle \max_{\tilde{u}\in \tilde{G}_{u}\setminus \{1\}}\dim(\tilde{V}_{\tilde{u}}(1))}=3\cdot 2^{\ell-3}$, and so $\nu_{\tilde{G}}(\tilde{V})=2^{\ell-2}$. 
\end{proof}

\begin{prop}\label{D42om2p=3}
Let $p=3$, $\ell=4$ and $\tilde{V}=L_{\tilde{G}}(2\tilde{\omega}_{2})$. Then $\nu_{\tilde{G}}(\tilde{V})\geq 84$.  Moreover, we have $\scale[0.9]{\displaystyle \max_{\tilde{u}\in \tilde{G}_{u}\setminus \{1\}}\dim(\tilde{V}_{\tilde{u}}(1))}$ $\leq 93$ and $\scale[0.9]{\displaystyle \max_{\tilde{s}\in \tilde{T}\setminus\ZG(\tilde{G})}\dim(\tilde{V}_{\tilde{s}}(\tilde{\mu}))}\leq 111$.
\end{prop}

\begin{proof}
Let $\tilde{\lambda}=2\tilde{\omega}_{2}$ and $\tilde{L}=\tilde{L}_{1}$. By Lemma \ref{weightlevelDl}, we have $e_{1}(\tilde{\lambda})=4$, therefore $\displaystyle \tilde{V}\mid_{[\tilde{L},\tilde{L}]}=\tilde{V}^{0}\oplus \cdots\oplus \tilde{V}^{4}$. By \cite[Proposition]{Smith_82} and Lemma \ref{dualitylemma}, we have $\tilde{V}^{0}\cong L_{\tilde{L}}(2\tilde{\omega}_{2})$ and $\tilde{V}^{4}\cong L_{\tilde{L}}(2\tilde{\omega}_{2})$. The weight $\displaystyle (\tilde{\lambda}-\tilde{\alpha}_{1}-\tilde{\alpha}_{2})\mid_{\tilde{T}_{1}}=\tilde{\omega}_{2}+\tilde{\omega}_{3}+\tilde{\omega}_{4}$ admits a maximal vector in $\tilde{V}^{1}$, thus $\tilde{V}^{1}$ has a composition factor isomorphic to $L_{\tilde{L}}(\tilde{\omega}_{2}+\tilde{\omega}_{3}+\tilde{\omega}_{4})$. Similarly, the weight $\displaystyle (\tilde{\lambda}-2\tilde{\alpha}_{1}-2\tilde{\alpha}_{2})\mid_{\tilde{T}_{1}}=2\tilde{\omega}_{3}+2\tilde{\omega}_{4}$ admits a maximal vector in $\tilde{V}^{2}$, thus $\tilde{V}^{2}$ has a composition factor isomorphic to $L_{\tilde{L}}(2\tilde{\omega}_{3}+2\tilde{\omega}_{4})$. As $\dim(\tilde{V}^{2})\leq 69$, it follows that
\begin{equation}\label{DecompVD42om2p=3}
\tilde{V}\mid_{[\tilde{L},\tilde{L}]}\cong L_{\tilde{L}}(2\tilde{\omega}_{2}) \oplus L_{\tilde{L}}(\tilde{\omega}_{2}+\tilde{\omega}_{3}+\tilde{\omega}_{4}) \oplus L_{\tilde{L}}(2\tilde{\omega}_{3}+2\tilde{\omega}_{4}) \oplus L_{\tilde{L}}(\tilde{\omega}_{2}+\tilde{\omega}_{3}+\tilde{\omega}_{4}) \oplus L_{\tilde{L}}(2\tilde{\omega}_{2}).
\end{equation}

We begin with the semisimple elements. Let $\tilde{s}\in \tilde{T}\setminus \ZG(\tilde{G})$. If $\dim(\tilde{V}^{i}_{\tilde{s}}(\tilde{\mu}))=\dim(\tilde{V}^{i})$ for some eigenvalue $\tilde{\mu}$ of $\tilde{s}$ on $\tilde{V}$, where $0\leq i\leq 4$, then $\tilde{s}\in \ZG(\tilde{L})^{\circ}\setminus \ZG(\tilde{G})$. In this case, as $\tilde{s}$ acts on each $\tilde{V}^{i}$ as scalar multiplication by $c^{4-2i}$ and $c^{2}\neq 1$, it follows that $\dim(\tilde{V}_{\tilde{s}}(\tilde{\mu}))\leq 107$ for all eigenvalues $\tilde{\mu}$ of $\tilde{s}$ on $\tilde{V}$. We thus assume that $\dim(\tilde{V}^{i}_{\tilde{s}}(\tilde{\mu}))<\dim(\tilde{V}^{i})$ for all eigenvalues $\tilde{\mu}$ of $\tilde{s}$ on $\tilde{V}$ and all $0\leq i\leq 4$. We write $\tilde{s}=\tilde{z}\cdot \tilde{h}$, where $\tilde{z}\in \ZG(\tilde{L})^{\circ}$ and $\tilde{h}\in [\tilde{L},\tilde{L}]$, and by \eqref{DecompVD42om2p=3} and Propositions \ref{A3om1+om2+om3}, \ref{A32om2} and \ref{A32om1+2om3}, we determine that $\dim(\tilde{V}_{\tilde{s}}(\tilde{\mu}))\leq 2\dim((L_{\tilde{L}}(2\tilde{\omega}_{2}))_{\tilde{h}}(\tilde{\mu}_{\tilde{h}}))+2\dim((L_{\tilde{L}}(\tilde{\omega}_{2}+\tilde{\omega}_{3}+\tilde{\omega}_{4}))_{\tilde{h}}(\tilde{\mu}_{\tilde{h}}))+2\dim((L_{\tilde{L}}(2\tilde{\omega}_{3}+2\tilde{\omega}_{4}))_{\tilde{h}}(\tilde{\mu}_{\tilde{h}}))\leq 111$ for all eigenvalues $\tilde{\mu}$ of $\tilde{s}$ on $\tilde{V}$. Therefore, $\scale[0.9]{\displaystyle \max_{\tilde{s}\in \tilde{T}\setminus\ZG(\tilde{G})}\dim(\tilde{V}_{\tilde{s}}(\tilde{\mu}))}\leq 111$. 

For the unipotent elements, by Lemma \ref{uniprootelems}, decomposition \eqref{DecompVD42om2p=3} and Propositions \ref{A3om1+om2+om3}, \ref{A32om2} and \ref{A32om1+2om3}, we have $\scale[0.9]{\displaystyle \max_{\tilde{u}\in \tilde{G}_{u}\setminus \{1\}}\dim(\tilde{V}_{\tilde{u}}(1))=}$ $\dim(\tilde{V}_{x_{\tilde{\alpha}_{4}}(1)}(1))\leq 93$. Lastly, we note that $\nu_{\tilde{G}}(\tilde{V})\geq 84$.
\end{proof}

\begin{prop}\label{D4om1+om3+om4p=2}
Let $p=2$, $\ell=4$ and $\tilde{V}=L_{\tilde{G}}(\tilde{\omega}_{1}+\tilde{\omega}_{3}+\tilde{\omega}_{4})$. Then $\nu_{\tilde{G}}(\tilde{V})\geq 102$.  Moreover, we have $\scale[0.9]{\displaystyle \max_{\tilde{u}\in \tilde{G}_{u}\setminus \{1\}}\dim(\tilde{V}_{\tilde{u}}(1))}\leq 144$ and $\scale[0.9]{\displaystyle \max_{\tilde{s}\in \tilde{T}\setminus\ZG(\tilde{G})}\dim(\tilde{V}_{\tilde{s}}(\tilde{\mu}))}\leq 130$.
\end{prop}

\begin{proof}
First, we note that $\tilde{V}\cong L_{G}(\omega_{1}+\omega_{3}+\omega_{4})$, as $k\tilde{G}$-modules, Further, as $p=2$, by \cite[Table $1$]{seitz1987maximal}, we have $L_{G}(\omega_{1}+\omega_{3}+\omega_{4})\cong L_{H}(\omega_{1}^{H}+\omega_{3}^{H})$ as $kG$-modules. Thus, in view of Lemma \ref{LtildeGtildelambdaandLGlambda} and by Proposition \ref{C4om1+om3p=2}, it follows that $\scale[0.9]{\displaystyle \max_{\tilde{s}\in \tilde{T}\setminus\ZG(\tilde{G})}\dim(\tilde{V}_{\tilde{s}}(\tilde{\mu}))}$ $\scale[0.9]{\displaystyle=\max_{s_{H}\in T_{H}\setminus\ZG(H)}\dim((L_{H}(\omega_{1}^{H}+\omega_{3}^{H}))_{s_{H}}(\mu_{H}))\leq 130}$ and $\scale[0.9]{\displaystyle \max_{\tilde{u}\in \tilde{G}_{u}\setminus \{1\}}\dim(\tilde{V}_{\tilde{u}}(1))}$ $\scale[0.9]{\displaystyle=\max_{u\in H_{u}\setminus \{1\}}\dim((L_{H}(\omega_{1}^{H}+\omega_{3}^{H})_{u}(1))\leq 144}$. Lastly, we note that $\nu_{\tilde{G}}(\tilde{V})\geq 102$.
\end{proof}

\begin{prop}\label{D42om1+om3}
Let $p\neq 2$, $\ell=4$ and $\tilde{V}=L_{G}(2\tilde{\omega}_{1}+\tilde{\omega}_{3})$. Then $\nu_{\tilde{G}}(\tilde{V})\geq 96-22\varepsilon_{p}(5)-8\varepsilon_{p}(3)$.  Moreover, we have $\scale[0.9]{\displaystyle \max_{\tilde{u}\in \tilde{G}_{u}\setminus \{1\}}\dim(\tilde{V}_{\tilde{u}}(1))}$ $\leq 114-34\varepsilon_{p}(5)+6\varepsilon_{p}(3)$ and $\scale[0.9]{\displaystyle \max_{\tilde{s}\in \tilde{T}\setminus\ZG(\tilde{G})}\dim(\tilde{V}_{\tilde{s}}(\tilde{\mu}))}\leq 128-34\varepsilon_{p}(5)+8\varepsilon_{p}(3)$.
\end{prop}

\begin{proof}
Let $\tilde{\lambda}=2\tilde{\omega}_{1}+\tilde{\omega}_{3}$ and $\tilde{L}=\tilde{L}_{1}$. By Lemma \ref{weightlevelDl}, we have $e_{1}(\tilde{\lambda})=5$, therefore $\displaystyle \tilde{V}\mid_{[\tilde{L},\tilde{L}]}=\tilde{V}^{0}\oplus \cdots\oplus \tilde{V}^{5}$. By \cite[Proposition]{Smith_82} and Lemma \ref{dualitylemma}, we have $\tilde{V}^{0}\cong L_{\tilde{L}}(\tilde{\omega}_{3})$ and $\tilde{V}^{5}\cong L_{\tilde{L}}(\tilde{\omega}_{4})$. Now, the weight $\displaystyle (\tilde{\lambda}-\tilde{\alpha}_{1})\mid_{\tilde{T}_{1}}=\tilde{\omega}_{2}+\tilde{\omega}_{3}$ admits a maximal vector in $\tilde{V}^{1}$, thus $\tilde{V}^{1}$ has a composition factor isomorphic to $L_{\tilde{L}}(\tilde{\omega}_{2}+\tilde{\omega}_{3})$.  Moreover, the weight $\displaystyle (\tilde{\lambda}-\tilde{\alpha}_{1}-\tilde{\alpha}_{2}-\tilde{\alpha}_{3})\mid_{\tilde{T}_{1}}=\tilde{\omega}_{4}$ occurs with multiplicity $3-\varepsilon_{p}(5)$ and is a sub-dominant weight in the composition factor of $\tilde{V}^{1}$ isomorphic to $L_{\tilde{L}}(\tilde{\omega}_{2}+\tilde{\omega}_{3})$, in which it has multiplicity $2-\varepsilon_{p}(3)$. Similarly, in $\tilde{V}^{2}$ the weight $\displaystyle (\tilde{\lambda}-2\tilde{\alpha}_{1})\mid_{\tilde{T}_{1}}=2\tilde{\omega}_{2}+\tilde{\omega}_{3}$ admits a maximal vector, thus $\tilde{V}^{2}$ has a composition factor isomorphic to $L_{\tilde{L}}(2\tilde{\omega}_{2}+\tilde{\omega}_{3})$. Moreover, the weight $\displaystyle (\tilde{\lambda}-2\tilde{\alpha}_{1}-\tilde{\alpha}_{2}-\tilde{\alpha}_{3})\mid_{\tilde{T}_{1}}=\tilde{\omega}_{2}+\tilde{\omega}_{4}$ occurs with multiplicity $3-\varepsilon_{p}(5)$ and is a sub-dominant weight in the composition factor of $\tilde{V}^{2}$ isomorphic to $L_{\tilde{L}}(2\tilde{\omega}_{2}+\tilde{\omega}_{3})$, in which it has multiplicity $2$. Further, we note that the weight $\displaystyle (\tilde{\lambda}-2\tilde{\alpha}_{1}-2\tilde{\alpha}_{2}-\tilde{\alpha}_{3}-\tilde{\alpha}_{4})\mid_{\tilde{T}_{1}}=\tilde{\omega}_{3}$ occurs with multiplicity $6-3\varepsilon_{p}(5)$ in $\tilde{V}^{2}$. Now, as $\dim(\tilde{V}^{2})\leq 84-24\varepsilon_{p}(5)$, we determine that $\tilde{V}^{2}$, respectively $\tilde{V}^{3}$, has exactly $3-2\varepsilon_{p}(5)+\varepsilon_{p}(3)$ composition factors: one isomorphic to $L_{\tilde{L}}(2\tilde{\omega}_{2}+\tilde{\omega}_{3})$, respectively to $L_{\tilde{L}}(2\tilde{\omega}_{2}+\tilde{\omega}_{4})$, $1-\varepsilon_{p}(5)$ to $L_{\tilde{L}}(\tilde{\omega}_{2}+\tilde{\omega_{4}})$, respectively to $L_{\tilde{L}}(\tilde{\omega}_{2}+\tilde{\omega_{3}})$, and $1-\varepsilon_{p}(5)+\varepsilon_{p}(3)$ to $L_{\tilde{L}}(\tilde{\omega}_{3})$, respectively to $L_{\tilde{L}}(\tilde{\omega}_{3})$. Further, $\tilde{V}^{1}$, respectively $\tilde{V}^{4}$, has $2-\varepsilon_{p}(5)+\varepsilon_{p}(3)$ composition factors: one isomorphic to $L_{\tilde{L}}(\tilde{\omega}_{2}+\tilde{\omega}_{3})$, respectively to $L_{\tilde{L}}(\tilde{\omega}_{2}+\tilde{\omega}_{4})$, and $1-\varepsilon_{p}(5)+\varepsilon_{p}(3)$ to $L_{\tilde{L}}(\tilde{\omega_{4}})$, respectively to $L_{\tilde{L}}(\tilde{\omega_{3}})$.

We begin with the semisimple elements. Let $\tilde{s}\in \tilde{T}\setminus \ZG(\tilde{G})$. If $\dim(\tilde{V}^{i}_{\tilde{s}}(\tilde{\mu}))=\dim(\tilde{V}^{i})$ for some eigenvalue $\tilde{\mu}$ of $\tilde{s}$ on $\tilde{V}$, where $0\leq i\leq 5$, then $\tilde{s}\in \ZG(\tilde{L})^{\circ}\setminus \ZG(\tilde{G})$. In this case, as $\tilde{s}$ acts on each $\tilde{V}^{i}$ as scalar multiplication by $c^{5-2i}$ and $c^{2}\neq 1$, it follows that $\dim(\tilde{V}_{\tilde{s}}(\tilde{\mu}))\leq 112-28\varepsilon_{p}(5)$ for all eigenvalues $\tilde{\mu}$ of $\tilde{s}$ on $\tilde{V}$. We thus assume that $\dim(\tilde{V}^{i}_{\tilde{s}}(\tilde{\mu}))<\dim(\tilde{V}^{i})$ for all eigenvalues $\tilde{\mu}$ of $\tilde{s}$ on $\tilde{V}$ and all $0\leq i\leq 5$. We write $\tilde{s}=\tilde{z}\cdot \tilde{h}$, where $\tilde{z}\in \ZG(\tilde{L})^{\circ}$ and $\tilde{h}\in [\tilde{L},\tilde{L}]$, and, by the structure of $\tilde{V}\mid_{[\tilde{L},\tilde{L}]}$ and Propositions \ref{PropositionAlnatural}, \ref{PropositionAlom1+om2p=3}, \ref{PropositionAlom1+om2pneq3} and \ref{A3om1+2om2}, we determine that $\dim(\tilde{V}_{\tilde{s}}(\tilde{\mu}))\leq (6-4\varepsilon_{p}(5)+4\varepsilon_{p}(3))\dim((L_{\tilde{L}}(\tilde{\omega}_{3}))_{\tilde{h}}(\tilde{\mu}_{\tilde{h}}))+(4-2\varepsilon_{p}(5))\dim((L_{\tilde{L}}(\tilde{\omega}_{2}+\tilde{\omega}_{3}))_{\tilde{h}}(\tilde{\mu}_{\tilde{h}}))+2\dim((L_{\tilde{L}}(2\tilde{\omega}_{2}+\tilde{\omega}_{3}))_{\tilde{h}}(\tilde{\mu}_{\tilde{h}}))\leq 128-34\varepsilon_{p}(5)+8\varepsilon_{p}(3)$ for all eigenvalues $\tilde{\mu}$ of $\tilde{s}$ on $\tilde{V}$. Therefore $\scale[0.9]{\displaystyle \max_{\tilde{s}\in \tilde{T}\setminus\ZG(\tilde{G})}\dim(\tilde{V}_{\tilde{s}}(\tilde{\mu}))}\leq 128-34\varepsilon_{p}(5)+8\varepsilon_{p}(3)$. 

For the unipotent elements, by Lemma \ref{uniprootelems}, the structure of $\tilde{V}\mid_{[\tilde{L},\tilde{L}]}$ and Propositions \ref{PropositionAlnatural}, \ref{PropositionAlom1+om2p=3}, \ref{PropositionAlom1+om2pneq3} and \ref{A3om1+2om2}, we have $\scale[0.9]{\displaystyle \max_{\tilde{u}\in \tilde{G}_{u}\setminus \{1\}}\dim(\tilde{V}_{\tilde{u}}(1))=}$ $\dim(\tilde{V}_{x_{\tilde{\alpha}_{5}}(1)}(1))\leq 114-34\varepsilon_{p}(5)+6\varepsilon_{p}(3)$. Lastly, we note that $\nu_{\tilde{G}}(\tilde{V})\geq 96-22\varepsilon_{p}(5)-8\varepsilon_{p}(3)$.
\end{proof}

\begin{prop}\label{D52om5p}
Let $p\neq 2$, $\ell=5$ and $\tilde{V}=L_{G}(2\tilde{\omega}_{5})$. Then $\nu_{\tilde{G}}(\tilde{V})\geq 46 $.  Moreover, we have $\scale[0.9]{\displaystyle \max_{\tilde{u}\in \tilde{G}_{u}\setminus \{1\}}\dim(\tilde{V}_{\tilde{u}}(1))}$ $=76$ and $\scale[0.9]{\displaystyle \max_{\tilde{s}\in \tilde{T}\setminus\ZG(\tilde{G})}\dim(\tilde{V}_{\tilde{s}}(\tilde{\mu}))}\leq 80$.
\end{prop}

\begin{proof}
Let $\tilde{\lambda}=2\tilde{\omega}_{5}$ and $\tilde{L}=\tilde{L}_{1}$. By Lemma \ref{weightlevelDl}, we have $e_{1}(\tilde{\lambda})=2$, therefore $\displaystyle \tilde{V}\mid_{[\tilde{L},\tilde{L}]}=\tilde{V}^{0}\oplus \tilde{V}^{1}\oplus \tilde{V}^{2}$. By \cite[Proposition]{Smith_82}, we have $\tilde{V}^{0}\cong L_{\tilde{L}}(2\tilde{\omega}_{5})$. Now, the weight $\displaystyle (\tilde{\lambda}-\tilde{\alpha}_{1}-\tilde{\alpha}_{2}-\tilde{\alpha}_{3}-\tilde{\alpha}_{5})\mid_{\tilde{T}_{1}}=\tilde{\omega}_{4}+\tilde{\omega}_{5}$ admits a maximal vector in $\tilde{V}^{1}$, thus $\tilde{V}^{1}$ has a composition factor isomorphic to $L_{\tilde{L}}(\tilde{\omega}_{4}+\tilde{\omega}_{5})$. Similarly, the weight $\displaystyle (\tilde{\lambda}-2\tilde{\alpha}_{1}-2\tilde{\alpha}_{2}-2\tilde{\alpha}_{3}-2\tilde{\alpha}_{5})\mid_{\tilde{T}_{1}}=2\tilde{\omega}_{4}$ admits a maximal vector in $\tilde{V}^{2}$, thus $\tilde{V}^{2}$ has a composition factor isomorphic to $L_{\tilde{L}}(2\tilde{\omega}_{4})$. By dimensional considerations, we determine that:
\begin{equation}\label{DecompVD52om5}
\tilde{V}\mid_{[\tilde{L},\tilde{L}]}\cong L_{\tilde{L}}(2\tilde{\omega}_{5}) \oplus L_{\tilde{L}}(\tilde{\omega}_{4}+\tilde{\omega}_{4}) \oplus L_{\tilde{L}}(2\tilde{\omega}_{4}).
\end{equation}

We begin with the semisimple elements. Let $\tilde{s}\in \tilde{T}\setminus \ZG(\tilde{G})$. If $\dim(\tilde{V}^{i}_{\tilde{s}}(\tilde{\mu}))=\dim(\tilde{V}^{i})$ for some eigenvalue $\tilde{\mu}$ of $\tilde{s}$ on $\tilde{V}$, where $0\leq i\leq 2$, then $\tilde{s}\in \ZG(\tilde{L})^{\circ}\setminus \ZG(\tilde{G})$. In this case, as $\tilde{s}$ acts on each $\tilde{V}^{i}$ as scalar multiplication by $c^{2-2i}$ and $c^{2}\neq 1$, it follows that $\dim(\tilde{V}_{\tilde{s}}(\tilde{\mu}))\leq 70$ for all eigenvalues $\tilde{\mu}$ of $\tilde{s}$ on $\tilde{V}$. We thus assume that $\dim(\tilde{V}^{i}_{\tilde{s}}(\tilde{\mu}))<\dim(\tilde{V}^{i})$ for all eigenvalues $\tilde{\mu}$ of $\tilde{s}$ on $\tilde{V}$ and all $0\leq i\leq 2$. We write $\tilde{s}=\tilde{z}\cdot \tilde{h}$, where $\tilde{z}\in \ZG(\tilde{L})^{\circ}$ and $\tilde{h}\in [\tilde{L},\tilde{L}]$, and, by \eqref{DecompVD52om5} and Propositions \ref{PropositionDlsymm} and \ref{D4om3+om4pneq2}, we determine that $\dim(\tilde{V}_{\tilde{s}}(\tilde{\mu}))\leq 2\dim((L_{\tilde{L}}(2\tilde{\omega}_{4}))_{\tilde{h}}(\tilde{\mu}_{\tilde{h}}))+\dim((L_{\tilde{L}}(\tilde{\omega}_{4}+\tilde{\omega}_{5}))_{\tilde{h}}(\tilde{\mu}_{\tilde{h}}))\leq 80$ for all eigenvalues $\tilde{\mu}$ of $\tilde{s}$ on $\tilde{V}$. Therefore $\scale[0.9]{\displaystyle \max_{\tilde{s}\in \tilde{T}\setminus\ZG(\tilde{G})}\dim(\tilde{V}_{\tilde{s}}(\tilde{\mu}))}\leq 80$. 

For the unipotent elements, by Lemma \ref{uniprootelems}, decomposition \eqref{DecompVD52om5} and Propositions \ref{PropositionDlsymm} and \ref{D4om3+om4pneq2}, we have $\scale[0.9]{\displaystyle \max_{\tilde{u}\in \tilde{G}_{u}\setminus \{1\}}\dim(\tilde{V}_{\tilde{u}}(1))=}$ $\dim(\tilde{V}_{x_{\tilde{\alpha}_{5}}(1)}(1))=76$. Lastly, we note that $\nu_{\tilde{G}}(\tilde{V})\geq 46$.
\end{proof}

\begin{prop}\label{D5om1+om5}
Let $\ell=5$ and $\tilde{V}=L_{\tilde{G}}(\tilde{\omega}_{1}+\tilde{\omega}_{5})$. Then $\nu_{\tilde{G}}(\tilde{V})\geq 52-4\varepsilon_{p}(5)$, where equality holds for $p\neq 2$. Moreover, we have $\scale[0.9]{\displaystyle \max_{\tilde{u}\in \tilde{G}_{u}\setminus \{1\}}\dim(\tilde{V}_{\tilde{u}}(1))}\leq 92-12\varepsilon_{p}(5)$, where equality holds for $p\neq 2$, and $\scale[0.9]{\displaystyle \max_{\tilde{s}\in \tilde{T}\setminus\ZG(\tilde{G})}\dim(\tilde{V}_{\tilde{s}}(\tilde{\mu}))}\leq 92-12\varepsilon_{p}(5)-16\varepsilon_{p}(2)$.
\end{prop}

\begin{proof}
Let $\tilde{\lambda}=\tilde{\omega}_{1}+\tilde{\omega}_{5}$ and $\tilde{L}=\tilde{L}_{1}$. By Lemma \ref{weightlevelDl}, we have $e_{1}(\tilde{\lambda})=3$, therefore $\displaystyle \tilde{V}\mid_{[\tilde{L},\tilde{L}]}=\tilde{V}^{0}\oplus \cdots\oplus \tilde{V}^{3}$. By \cite[Proposition]{Smith_82}, we have $\tilde{V}^{0}\cong L_{\tilde{L}}(\tilde{\omega}_{5})$. Now, the weight $\displaystyle (\tilde{\lambda}-\tilde{\alpha}_{1})\mid_{\tilde{T}_{1}}=\tilde{\omega}_{2}+\tilde{\omega}_{5}$ admits a maximal vector in $\tilde{V}^{1}$, thus $\tilde{V}^{1}$ has a composition factor isomorphic to $L_{\tilde{L}}(\tilde{\omega}_{2}+\tilde{\omega}_{5})$. Moreover, the weight $\displaystyle (\tilde{\lambda}-\tilde{\alpha}_{1}-\tilde{\alpha}_{2}-\tilde{\alpha}_{3}-\tilde{\alpha}_{5})\mid_{\tilde{T}_{1}}=\tilde{\omega}_{4}$ occurs with multiplicity $4-\varepsilon_{p}(5)$ and is a sub-dominant weight in the composition factor of $\tilde{V}^{1}$ isomorphic to $L_{\tilde{L}}(\tilde{\omega}_{2}+\tilde{\omega}_{5})$, in which it has multiplicity $3-\varepsilon_{p}(2)$. Similarly, in $\tilde{V}^{2}$, the weight $\displaystyle (\tilde{\lambda}-2\tilde{\alpha}_{1}-\tilde{\alpha}_{2}-\tilde{\alpha}_{3}-\tilde{\alpha}_{5})\mid_{\tilde{T}_{1}}=\tilde{\omega}_{2}+\tilde{\omega}_{4}$ admits a maximal vector, thus $\tilde{V}^{2}$ has a composition factor isomorphic to $L_{\tilde{L}}(\tilde{\omega}_{2}+\tilde{\omega}_{4})$. Moreover, the weight $\displaystyle (\tilde{\lambda}-2\tilde{\alpha}_{1}-2\tilde{\alpha}_{2}-2\tilde{\alpha}_{3}-\tilde{\alpha}_{4}-\tilde{\alpha}_{5})\mid_{\tilde{T}_{1}}=\tilde{\omega}_{5}$ occurs with multiplicity $4-\varepsilon_{p}(5)$ and is a sub-dominant weight in the composition factor of $\tilde{V}^{2}$ isomorphic to $L_{\tilde{L}}(\tilde{\omega}_{2}+\tilde{\omega}_{4})$, in which it has multiplicity $3-\varepsilon_{p}(2)$. Lastly, the weight $\displaystyle (\tilde{\lambda}-3\tilde{\alpha}_{1}-3\tilde{\alpha}_{2}-3\tilde{\alpha}_{3}-\tilde{\alpha}_{4}-2\tilde{\alpha}_{5})\mid_{\tilde{T}_{1}}=\tilde{\omega}_{4}$ admits a maximal vector in $\tilde{V}^{3}$, thus $\tilde{V}^{3}$ has a composition factor isomorphic to $L_{\tilde{L}}(\tilde{\omega}_{4})$. By dimensional considerations, we determine that $\tilde{V}^{3}\cong L_{\tilde{L}}(\tilde{\omega}_{4})$, and that $\tilde{V}^{1}$, respectively $\tilde{V}^{2}$, has exactly $2-\varepsilon_{p}(5)+\varepsilon_{p}(2)$ composition factors: one isomorphic to $L_{\tilde{L}}(\tilde{\omega}_{2}+\tilde{\omega}_{5})$, respectively to $L_{\tilde{L}}(\tilde{\omega}_{2}+\tilde{\omega}_{4})$, and $1-\varepsilon_{p}(5)+\varepsilon_{p}(2)$ isomorphic to $L_{\tilde{L}}(\tilde{\omega}_{4})$, respectively to $L_{\tilde{L}}(\tilde{\omega}_{5})$.

We begin with the semisimple elements. Let $\tilde{s}\in \tilde{T}\setminus \ZG(\tilde{G})$. If $\dim(\tilde{V}^{i}_{\tilde{s}}(\tilde{\mu}))=\dim(\tilde{V}^{i})$ for some eigenvalue $\tilde{\mu}$ of $\tilde{s}$ on $\tilde{V}$, where $0\leq i\leq 3$, then $\tilde{s}\in \ZG(\tilde{L})^{\circ}\setminus \ZG(\tilde{G})$. In this case, as $\tilde{s}$ acts on each $\tilde{V}^{i}$ as scalar multiplication by $c^{3-2i}$ and $c^{2}\neq 1$, it follows that $\dim(\tilde{V}_{\tilde{s}}(\tilde{\mu}))\leq 72-8\varepsilon_{p}(5)$ for all eigenvalues $\tilde{\mu}$ of $\tilde{s}$ on $\tilde{V}$. We thus assume that $\dim(\tilde{V}^{i}_{\tilde{s}}(\tilde{\mu}))<\dim(\tilde{V}^{i})$ for all eigenvalues $\tilde{\mu}$ of $\tilde{s}$ on $\tilde{V}$ and all $0\leq i\leq 3$. We write $\tilde{s}=\tilde{z}\cdot \tilde{h}$, where $\tilde{z}\in \ZG(\tilde{L})^{\circ}$ and $\tilde{h}\in [\tilde{L},\tilde{L}]$, and, by  the structure of $\tilde{V}\mid_{[\tilde{L},\tilde{L}]}$ and Propositions \ref{PropositionDlnatural}, \ref{D4om3+om4p=2} and \ref{D4om3+om4pneq2}, we determine that $\dim(\tilde{V}_{\tilde{s}}(\tilde{\mu}))\leq (4-2\varepsilon_{p}(5)+2\varepsilon_{p}(2))\dim((L_{\tilde{L}}(\tilde{\omega}_{4}))_{\tilde{h}}(\tilde{\mu}_{\tilde{h}}))+2\dim((L_{\tilde{L}}(\tilde{\omega}_{2}+\tilde{\omega}_{4}))_{\tilde{h}}(\tilde{\mu}_{\tilde{h}}))\leq 92-12\varepsilon_{p}(5)-16\varepsilon_{p}(2)$ for all eigenvalues $\tilde{\mu}$ of $\tilde{s}$ on $\tilde{V}$. Therefore, $\scale[0.9]{\displaystyle \max_{\tilde{s}\in \tilde{T}\setminus\ZG(\tilde{G})}\dim(\tilde{V}_{\tilde{s}}(\tilde{\mu}))}\leq 92-12\varepsilon_{p}(5)-16\varepsilon_{p}(2)$. 

For the unipotent elements, by Lemma \ref{uniprootelems}, he structure of $\tilde{V}\mid_{[\tilde{L},\tilde{L}]}$ and Propositions \ref{PropositionDlnatural}, \ref{D4om3+om4p=2} and \ref{D4om3+om4pneq2}, we have $\scale[0.9]{\displaystyle \max_{\tilde{u}\in \tilde{G}_{u}\setminus \{1\}}\dim(\tilde{V}_{\tilde{u}}(1))=}$ $\dim(\tilde{V}_{x_{\tilde{\alpha}_{5}}(1)}(1))\leq 92-12\varepsilon_{p}(5)$, where equality holds for $p\neq 2$. Lastly, we note that $\nu_{\tilde{G}}(\tilde{V})\geq 52-4\varepsilon_{p}(5)$, where equality holds for $p\neq 2$.
\end{proof}

\begin{prop}\label{D5om4+om5}
Let $\ell=5$ and $\tilde{V}=L_{\tilde{G}}(\tilde{\omega}_{4}+\tilde{\omega}_{5})$. Then $\nu_{\tilde{G}}(\tilde{V})\geq 80-16\varepsilon_{p}(2)$. Moreover, we have $\scale[0.9]{\displaystyle \max_{\tilde{u}\in \tilde{G}_{u}\setminus \{1\}}\dim(\tilde{V}_{\tilde{u}}(1))}\leq 128-28\varepsilon_{p}(2)$ and $\scale[0.9]{\displaystyle \max_{\tilde{s}\in \tilde{T}\setminus\ZG(\tilde{G})}\dim(\tilde{V}_{\tilde{s}}(\tilde{\mu}))}\leq 130-50\varepsilon_{p}(2)$.
\end{prop}

\begin{proof}
Let $\tilde{\lambda}=\tilde{\omega}_{4}+\tilde{\omega}_{5}$ and $\tilde{L}=\tilde{L}_{1}$. By Lemma \ref{weightlevelDl}, we have $e_{1}(\tilde{\lambda})=2$, therefore $\displaystyle \tilde{V}\mid_{[\tilde{L},\tilde{L}]}=\tilde{V}^{0}\oplus \tilde{V}^{1}\oplus \tilde{V}^{2}$. By \cite[Proposition]{Smith_82} and Lemma \ref{dualitylemma}, we have $\tilde{V}^{0}\cong L_{\tilde{L}}(\tilde{\omega}_{4}+\tilde{\omega}_{5})$ and $\tilde{V}^{2}\cong L_{\tilde{L}}(\tilde{\omega}_{4}+\tilde{\omega}_{5})$. Now. in $\tilde{V}^{1}$ both the weight $\displaystyle (\tilde{\lambda}-\tilde{\alpha}_{1}-\tilde{\alpha}_{2}-\tilde{\alpha}_{3}-\tilde{\alpha_{4}})\mid_{\tilde{T}_{1}}=2\tilde{\omega}_{5}$ and the weight $\displaystyle (\tilde{\lambda}-\tilde{\alpha}_{1}-\tilde{\alpha}_{2}-\tilde{\alpha}_{3}-\tilde{\alpha}_{5})\mid_{\tilde{T}_{1}}=2\tilde{\omega}_{4}$ admits a maximal vector, thus $\tilde{V}^{1}$ has a composition factor isomorphic to $L_{\tilde{L}}(2\tilde{\omega}_{5})$ and another to $L_{\tilde{L}}(2\tilde{\omega}_{4})$. Moreover, the weight $\displaystyle (\tilde{\lambda}-\tilde{\alpha}_{1}-\tilde{\alpha}_{2}-\tilde{\alpha}_{3}-\tilde{\alpha}_{4}-\tilde{\alpha}_{5})\mid_{\tilde{T}_{1}}=\tilde{\omega}_{3}$ occurs with multiplicity $3-\varepsilon_{p}(2)$ and is a sub-dominant weight in each of the composition factors of $\tilde{V}^{1}$ we identified, where it has multiplicity $1-\varepsilon_{p}(2)$. By dimensional considerations, we determine that $\tilde{V}^{1}$ has exactly $3+\varepsilon_{p}(2)$ composition factors: one isomorphic to $L_{\tilde{L}}(2\tilde{\omega}_{4})$, one to $L_{\tilde{L}}(2\tilde{\omega}_{5})$ and $1+\varepsilon_{p}(2)$ to $L_{\tilde{L}}(\tilde{\omega}_{3})$. 

We begin with the semisimple elements. Let $\tilde{s}\in \tilde{T}\setminus \ZG(\tilde{G})$. If $\dim(\tilde{V}^{i}_{\tilde{s}}(\tilde{\mu}))=\dim(\tilde{V}^{i})$ for some eigenvalue $\tilde{\mu}$ of $\tilde{s}$ on $\tilde{V}$, where $0\leq i\leq 2$, then $\tilde{s}\in \ZG(\tilde{L})^{\circ}\setminus \ZG(\tilde{G})$. In this case, as $\tilde{s}$ acts on each $\tilde{V}^{i}$ as scalar multiplication by $c^{2-2i}$ and $c^{2}\neq 1$, it follows that $\dim(\tilde{V}_{\tilde{s}}(\tilde{\mu}))\leq 112-44\varepsilon_{p}(2)$ for all eigenvalues $\tilde{\mu}$ of $\tilde{s}$ on $\tilde{V}$. We thus assume that $\dim(\tilde{V}^{i}_{\tilde{s}}(\tilde{\mu}))<\dim(\tilde{V}^{i})$ for all eigenvalues $\tilde{\mu}$ of $\tilde{s}$ on $\tilde{V}$ and all $0\leq i\leq 2$. We write $\tilde{s}=\tilde{z}\cdot \tilde{h}$, where $\tilde{z}\in \ZG(\tilde{L})^{\circ}$ and $\tilde{h}\in [\tilde{L},\tilde{L}]$, and, by  the structure of $\tilde{V}\mid_{[\tilde{L},\tilde{L}]}$ and Propositions \ref{PropositionDlwedge}, \ref{PropositionDlsymm}, \ref{D4om3+om4p=2} and \ref{D4om3+om4pneq2}, we determine that $\dim(\tilde{V}_{\tilde{s}}(\tilde{\mu}))\leq 2\dim((L_{\tilde{L}}(\tilde{\omega}_{4}+\tilde{\omega}_{5}))_{\tilde{h}}(\tilde{\mu}_{\tilde{h}}))+2\dim((L_{\tilde{L}}(2\tilde{\omega}_{5}))_{\tilde{h}}(\tilde{\mu}_{\tilde{h}}))+(1+\varepsilon_{p}(2))\dim((L_{\tilde{L}}(\tilde{\omega}_{3}))_{\tilde{h}}(\tilde{\mu}_{\tilde{h}}))\leq 130-50\varepsilon_{p}(2)$ for all eigenvalues $\tilde{\mu}$ of $\tilde{s}$ on $\tilde{V}$. Therefore, $\scale[0.9]{\displaystyle \max_{\tilde{s}\in \tilde{T}\setminus\ZG(\tilde{G})}\dim(\tilde{V}_{\tilde{s}}(\tilde{\mu}))}\leq 130-50\varepsilon_{p}(2)$. 

For the unipotent elements, by Lemma \ref{uniprootelems}, the structure of $\tilde{V}\mid_{[\tilde{L},\tilde{L}]}$ and Propositions \ref{PropositionDlwedge}, \ref{PropositionDlsymm}, \ref{D4om3+om4p=2} and \ref{D4om3+om4pneq2}, we have $\scale[0.9]{\displaystyle \max_{\tilde{u}\in \tilde{G}_{u}\setminus \{1\}}\dim(\tilde{V}_{\tilde{u}}(1))=}$ $\dim(\tilde{V}_{x_{\tilde{\alpha}_{5}}(1)}(1))\leq 128-28\varepsilon_{p}(2)$. Lastly, we note that $\nu_{\tilde{G}}(\tilde{V})\geq 80-16\varepsilon_{p}(2)$.
\end{proof}

\begin{prop}\label{D5om2+om5p=2}
Let $p=2$, $\ell=5$ and $\tilde{V}=L_{\tilde{G}}(\tilde{\omega}_{2}+\tilde{\omega}_{5})$. Then $\nu_{\tilde{G}}(\tilde{V})\geq 164$. Moreover, we have $\scale[0.9]{\displaystyle \max_{\tilde{u}\in \tilde{G}_{u}\setminus \{1\}}\dim(\tilde{V}_{\tilde{u}}(1))}\leq 252$ and $\scale[0.9]{\displaystyle \max_{\tilde{s}\in \tilde{T}\setminus\ZG(\tilde{G})}\dim(\tilde{V}_{\tilde{s}}(\tilde{\mu}))}\leq 192$.
\end{prop}

\begin{proof}
Let $\tilde{\lambda}=\tilde{\omega}_{2}+\tilde{\omega}_{5}$ and $\tilde{L}=\tilde{L}_{1}$. By Lemma \ref{weightlevelDl}, we have $e_{1}(\tilde{\lambda})=3$, therefore $\displaystyle \tilde{V}\mid_{[\tilde{L},\tilde{L}]}=\tilde{V}^{0}\oplus\cdots\oplus \tilde{V}^{3}$. By \cite[Proposition]{Smith_82}, we have $\tilde{V}^{0}\cong L_{\tilde{L}}(\tilde{\omega}_{2}+\tilde{\omega}_{5})$. Now, in $\tilde{V}^{1}$ the weight $\displaystyle (\tilde{\lambda}-\tilde{\alpha}_{1}-\tilde{\alpha}_{2})\mid_{\tilde{T}_{1}}=\tilde{\omega}_{3}+\tilde{\omega}_{5}$ admits a maximal vector, thus $\tilde{V}^{1}$ has a composition factor isomorphic to $L_{\tilde{L}}(\tilde{\omega}_{3}+\tilde{\omega}_{5})$. Similarly, in $\tilde{V}^{2}$ the weight $\displaystyle (\tilde{\lambda}-2\tilde{\alpha}_{1}-2\tilde{\alpha}_{2}-\tilde{\omega}_{3}-\tilde{\omega}_{5})\mid_{\tilde{T}_{1}}=\tilde{\omega}_{3}+\tilde{\omega}_{4}$ admits a maximal vector, thus $\tilde{V}^{2}$ has a composition factor isomorphic to $L_{\tilde{L}}(\tilde{\omega}_{3}+\tilde{\omega}_{4})$. Lastly, the weight $\displaystyle (\tilde{\lambda}-3\tilde{\alpha}_{1}-3\tilde{\alpha}_{2}-3\tilde{\alpha}_{3}-\tilde{\alpha}_{4}-2\tilde{\alpha}_{5})\mid_{\tilde{T}_{1}}=\tilde{\omega}_{2}+\tilde{\omega}_{4}$ admits a maximal vector in $\tilde{V}^{3}$, thus $\tilde{V}^{3}$ has a composition factor isomorphic to $L_{\tilde{L}}(\tilde{\omega}_{2}+\tilde{\omega}_{4})$. As $\dim(\tilde{V}^{3})\leq 48$, it follows that:
\begin{equation}\label{DecompVD5om2+om5p=2}
\tilde{V}\mid_{[\tilde{L},\tilde{L}]}\cong  L_{\tilde{L}}(\tilde{\omega}_{2}+\tilde{\omega}_{5}) \oplus L_{\tilde{L}}(\tilde{\omega}_{3}+\tilde{\omega}_{5})\oplus L_{\tilde{L}}(\tilde{\omega}_{3}+\tilde{\omega}_{4}) \oplus L_{\tilde{L}}(\tilde{\omega}_{2}+\tilde{\omega}_{4}).
\end{equation}

We begin with the semisimple elements. Let $\tilde{s}\in \tilde{T}\setminus \ZG(\tilde{G})$. If $\dim(\tilde{V}^{i}_{\tilde{s}}(\tilde{\mu}))=\dim(\tilde{V}^{i})$ for some eigenvalue $\tilde{\mu}$ of $\tilde{s}$ on $\tilde{V}$, where $0\leq i\leq 3$, then $\tilde{s}\in \ZG(\tilde{L})^{\circ}\setminus \ZG(\tilde{G})$. In this case, as $\tilde{s}$ acts on each $\tilde{V}^{i}$ as scalar multiplication by $c^{3-2i}$ and $c^{2}\neq 1$, it follows that $\dim(\tilde{V}_{\tilde{s}}(\tilde{\mu}))\leq 160$ for all eigenvalues $\tilde{\mu}$ of $\tilde{s}$ on $\tilde{V}$. We thus assume that $\dim(\tilde{V}^{i}_{\tilde{s}}(\tilde{\mu}))<\dim(\tilde{V}^{i})$ for all eigenvalues $\tilde{\mu}$ of $\tilde{s}$ on $\tilde{V}$ and all $0\leq i\leq 3$. We write $\tilde{s}=\tilde{z}\cdot \tilde{h}$, where $\tilde{z}\in \ZG(\tilde{L})^{\circ}$ and $\tilde{h}\in [\tilde{L},\tilde{L}]$, and, by \eqref{DecompVD5om2+om5p=2} and Propositions \ref{D4om3+om4p=2} and \ref{D4om1+om2pneq3}, we determine that $\dim(\tilde{V}_{\tilde{s}}(\tilde{\mu}))\leq 2\dim((L_{\tilde{L}}(\tilde{\omega}_{2}+\tilde{\omega}_{5}))_{\tilde{h}}(\tilde{\mu}_{\tilde{h}}))+2\dim((L_{\tilde{L}}(\tilde{\omega}_{3}+\tilde{\omega}_{5}))_{\tilde{h}}(\tilde{\mu}_{\tilde{h}}))\leq 192$ for all eigenvalues $\tilde{\mu}$ of $\tilde{s}$ on $\tilde{V}$. Therefore, $\scale[0.9]{\displaystyle \max_{\tilde{s}\in \tilde{T}\setminus\ZG(\tilde{G})}\dim(\tilde{V}_{\tilde{s}}(\tilde{\mu}))}\leq 192$. 

For the unipotent elements, by Lemma \ref{uniprootelems}, decomposition \eqref{DecompVD5om2+om5p=2} and Propositions \ref{D4om3+om4p=2} and \ref{D4om1+om2pneq3}, we have $\scale[0.9]{\displaystyle \max_{\tilde{u}\in \tilde{G}_{u}\setminus \{1\}}\dim(\tilde{V}_{\tilde{u}}(1))=}$ $\dim(\tilde{V}_{x_{\tilde{\alpha}_{5}}(1)}(1))\leq 252$. Lastly, we note that $\nu_{\tilde{G}}(\tilde{V})\geq 164$.
\end{proof}

\begin{prop}\label{D6om1+om5}
Let $\ell=6$ and $\tilde{V}=L_{\tilde{G}}(\tilde{\omega}_{1}+\tilde{\omega}_{6})$. Then $\nu_{\tilde{G}}(\tilde{V})\geq 120-8\varepsilon_{p}(3)-8\varepsilon_{p}(2)$, where equality holds for $p\neq 2,5$. Moreover, we have $\scale[0.9]{\displaystyle \max_{\tilde{u}\in \tilde{G}_{u}\setminus \{1\}}\dim(\tilde{V}_{\tilde{u}}(1))}\leq 232-24\varepsilon_{p}(3)-24\varepsilon_{p}(2)$, where equality holds for $p\neq 2,5$, and $\scale[0.9]{\displaystyle \max_{\tilde{s}\in \tilde{T}\setminus\ZG(\tilde{G})}\dim(\tilde{V}_{\tilde{s}}(\tilde{\mu}))}\leq 224-4\varepsilon_{p}(5)-20\varepsilon_{p}(3)-52\varepsilon_{p}(2)$.
\end{prop}

\begin{proof}
Let $\tilde{\lambda}=\tilde{\omega}_{1}+\tilde{\omega}_{6}$ and $\tilde{L}=\tilde{L}_{1}$. By Lemma \ref{weightlevelDl}, we have $e_{1}(\tilde{\lambda})=3$, therefore $\displaystyle \tilde{V}\mid_{[\tilde{L},\tilde{L}]}=\tilde{V}^{0}\oplus\cdots\oplus \tilde{V}^{3}$. By \cite[Proposition]{Smith_82}, we have $\tilde{V}^{0}\cong L_{\tilde{L}}(\tilde{\omega}_{6})$. Now, in $\tilde{V}^{1}$ the weight $\displaystyle (\tilde{\lambda}-\tilde{\alpha}_{1})\mid_{\tilde{T}_{1}}=\tilde{\omega}_{2}+\tilde{\omega}_{6}$ admits a maximal vector, thus $\tilde{V}^{1}$ has a composition factor isomorphic to $L_{\tilde{L}}(\tilde{\omega}_{2}+\tilde{\omega}_{6})$. Moreover, the weight $\displaystyle (\tilde{\lambda}-\tilde{\alpha}_{1}-\cdots-\tilde{\alpha}_{4}-\tilde{\alpha}_{6})\mid_{\tilde{T}_{1}}=\tilde{\omega}_{5}$ occurs with multiplicity $5-\varepsilon_{p}(6)$ and is a sub-dominant weight in the composition factor of $\tilde{V}^{1}$ isomorphic to $L_{\tilde{L}}(\tilde{\omega}_{2}+\tilde{\omega}_{6})$, in which it has multiplicity $4-\varepsilon_{p}(5)$. Similarly, in $\tilde{V}^{2}$ the weight $\displaystyle (\tilde{\lambda}-2\tilde{\alpha}_{1}-\tilde{\alpha}_{2}-\tilde{\alpha}_{3}-\tilde{\alpha}_{4}-\tilde{\alpha}_{6})\mid_{\tilde{T}_{1}}=\tilde{\omega}_{2}+\tilde{\omega}_{5}$ admits a maximal vector, thus $\tilde{V}^{2}$ has a composition factor isomorphic to $L_{\tilde{L}}(\tilde{\omega}_{2}+\tilde{\omega}_{5})$. Moreover, the weight $\displaystyle (\tilde{\lambda}-2\tilde{\alpha}_{1}-\cdots-2\tilde{\alpha}_{4}-\tilde{\alpha}_{5}-\tilde{\alpha}_{6})\mid_{\tilde{T}_{1}}=\tilde{\omega}_{6}$ occurs with multiplicity $5-\varepsilon_{p}(3)-\varepsilon_{p}(2)$ and is a sub-dominant weight in the composition factor of $\tilde{V}^{2}$ isomorphic to $L_{\tilde{L}}(\tilde{\omega}_{2}+\tilde{\omega}_{5})$, in which it has multiplicity $4-\varepsilon_{p}(5)$. Lastly, the weight $\displaystyle (\tilde{\lambda}-3\tilde{\alpha}_{1}-\cdots-3\tilde{\alpha}_{4}-\tilde{\alpha}_{5}-2\tilde{\alpha}_{6})\mid_{\tilde{T}_{1}}=\tilde{\omega}_{5}$ admits a maximal vector in $\tilde{V}^{3}$, thus $\tilde{V}^{3}$ has a composition factor isomorphic to $L_{\tilde{L}}(\tilde{\omega}_{5})$. By dimensional considerations, we determine that $\tilde{V}^{3}\cong L_{\tilde{L}}(\tilde{\omega}_{5})$ and that $\tilde{V}^{1}$, respectively $\tilde{V}^{2}$, has exactly $2+\varepsilon_{p}(5)-\varepsilon_{p}(6)$ composition factors: one isomorphic to $L_{\tilde{L}}(\tilde{\omega}_{2}+\tilde{\omega}_{6})$, respectively to $L_{\tilde{L}}(\tilde{\omega}_{2}+\tilde{\omega}_{5})$, and $1+\varepsilon_{p}(5)-\varepsilon_{p}(6)$ isomorphic to $L_{\tilde{L}}(\tilde{\omega}_{5})$, respectively to $L_{\tilde{L}}(\tilde{\omega}_{6})$.

We begin with the semisimple elements. Let $\tilde{s}\in \tilde{T}\setminus \ZG(\tilde{G})$. If $\dim(\tilde{V}^{i}_{\tilde{s}}(\tilde{\mu}))=\dim(\tilde{V}^{i})$ for some eigenvalue $\tilde{\mu}$ of $\tilde{s}$ on $\tilde{V}$, where $0\leq i\leq 3$, then $\tilde{s}\in \ZG(\tilde{L})^{\circ}\setminus \ZG(\tilde{G})$. In this case, as $\tilde{s}$ acts on each $\tilde{V}^{i}$ as scalar multiplication by $c^{3-2i}$ and $c^{2}\neq 1$, it follows that $\dim(\tilde{V}_{\tilde{s}}(\tilde{\mu}))\leq 176-16\varepsilon_{p}(3)-32\varepsilon_{p}(2)$ for all eigenvalues $\tilde{\mu}$ of $\tilde{s}$ on $\tilde{V}$. We thus assume that $\dim(\tilde{V}^{i}_{\tilde{s}}(\tilde{\mu}))<\dim(\tilde{V}^{i})$ for all eigenvalues $\tilde{\mu}$ of $\tilde{s}$ on $\tilde{V}$ and all $0\leq i\leq 3$. We write $\tilde{s}=\tilde{z}\cdot \tilde{h}$, where $\tilde{z}\in \ZG(\tilde{L})^{\circ}$ and $\tilde{h}\in [\tilde{L},\tilde{L}]$, and, by the structure of $\tilde{V}\mid_{[\tilde{L},\tilde{L}]}$ and Propositions \ref{Dloml} and \ref{D5om1+om5}, we deduce that $\dim(\tilde{V}_{\tilde{s}}(\tilde{\mu}))\leq (4+2\varepsilon_{p}(5)-2\varepsilon_{p}(6))\dim((L_{\tilde{L}}(\tilde{\omega}_{5}))_{\tilde{h}}(\tilde{\mu}_{\tilde{h}}))+2\dim((L_{\tilde{L}}(\tilde{\omega}_{2}+\tilde{\omega}_{5}))_{\tilde{h}}(\tilde{\mu}_{\tilde{h}}))\leq 224-4\varepsilon_{p}(5)-20\varepsilon_{p}(3)-52\varepsilon_{p}(2)$ for all eigenvalues $\tilde{\mu}$ of $\tilde{s}$ on $\tilde{V}$. Therefore, $\scale[0.9]{\displaystyle \max_{\tilde{s}\in \tilde{T}\setminus\ZG(\tilde{G})}\dim(\tilde{\tilde{V}}_{\tilde{s}}(\tilde{\mu}))}\leq 224-4\varepsilon_{p}(5)-20\varepsilon_{p}(3)-52\varepsilon_{p}(2)$. 

For the unipotent elements, by Lemma \ref{uniprootelems}, the structure of $\tilde{V}\mid_{[\tilde{L},\tilde{L}]}$ and Propositions \ref{Dloml} and \ref{D5om1+om5}, we have $\scale[0.9]{\displaystyle \max_{\tilde{u}\in \tilde{G}_{u}\setminus \{1\}}\dim(\tilde{V}_{\tilde{u}}(1))=}$ $\dim(\tilde{V}_{x_{\tilde{\alpha}_{6}}(1)}(1))\leq 232-24\varepsilon_{p}(3)-24\varepsilon_{p}(2)$, where equality holds for $p\neq 2,5$. Lastly, we note that $\nu_{\tilde{G}}(\tilde{V})\geq 120-8\varepsilon_{p}(3)-8\varepsilon_{p}(2)$, where equality holds for $p\neq 2,5$.
\end{proof}

\begin{prop}\label{D6om4}
Let $\ell=6$ and $\tilde{V}=L_{\tilde{G}}(\tilde{\omega}_{4})$. Then $\nu_{\tilde{G}}(\tilde{V})\geq 184-40\varepsilon_{p}(2)$, where equality holds for $p\neq 2$. Moreover, we have $\scale[0.9]{\displaystyle \max_{\tilde{u}\in \tilde{G}_{u}\setminus \{1\}}\dim(\tilde{V}_{\tilde{u}}(1))}\leq 311-91\varepsilon_{p}(2)$, where equality holds for $p\neq 2$, and $\scale[0.9]{\displaystyle \max_{\tilde{s}\in \tilde{T}\setminus\ZG(\tilde{G})}\dim(\tilde{V}_{\tilde{s}}(\tilde{\mu}))}\leq 303-127\varepsilon_{p}(2)$.
\end{prop}

\begin{proof}
Let $\tilde{\lambda}=\tilde{\omega}_{4}$ and $\tilde{L}=\tilde{L}_{1}$. By Lemma \ref{weightlevelDl}, we have $e_{1}(\tilde{\lambda})=2$, therefore $\displaystyle \tilde{V}\mid_{[\tilde{L},\tilde{L}]}=\tilde{V}^{0}\oplus \tilde{V}^{1}\oplus \tilde{V}^{2}$. By \cite[Proposition]{Smith_82} and Lemma \ref{dualitylemma}, we have $\tilde{V}^{0}\cong L_{\tilde{L}}(\tilde{\omega}_{4})$ and $\tilde{V}^{2}\cong L_{\tilde{L}}(\tilde{\omega}_{4})$. Now, in $\tilde{V}^{1}$ the weight $\displaystyle (\tilde{\lambda}-\tilde{\alpha}_{1}-\cdots-\tilde{\alpha}_{4})\mid_{\tilde{T}_{1}}=\tilde{\omega}_{5}+\tilde{\omega}_{6}$ admits a maximal vector, thus $\tilde{V}^{1}$ has a composition factor isomorphic to $L_{\tilde{L}}(\tilde{\omega}_{5}+\tilde{\omega}_{6})$. Moreover, the weight $\displaystyle (\tilde{\lambda}-\tilde{\alpha}_{1}-\tilde{\alpha}_{2}-\tilde{\alpha}_{3}-2\tilde{\alpha}_{4}-\tilde{\alpha}_{5}-\tilde{\alpha}_{6})\mid_{\tilde{T}_{1}}=\tilde{\omega}_{3}$ occurs with multiplicity $4-2\varepsilon_{p}(2)$ and is a sub-dominant weight the composition factor of $\tilde{V}^{1}$ isomorphic to $L_{\tilde{L}}(\tilde{\omega}_{5}+\tilde{\omega}_{6})$, in which it has multiplicity $3-\varepsilon_{p}(2)$. Thus, by dimensional considerations, we determine that
\begin{equation}\label{DecompVD6om4}
\tilde{V}\mid_{[\tilde{L},\tilde{L}]}\cong L_{\tilde{L}}(\tilde{\omega}_{4}) \oplus L_{\tilde{L}}(\tilde{\omega}_{5}+\tilde{\omega}_{6})\oplus L_{\tilde{L}}(\tilde{\omega}_{3})^{1-\varepsilon_{p}(2)}\oplus L_{\tilde{L}}(\tilde{\omega}_{4}).
\end{equation}

We begin with the semisimple elements. Let $\tilde{s}\in \tilde{T}\setminus \ZG(\tilde{G})$. If $\dim(\tilde{V}^{i}_{\tilde{s}}(\tilde{\mu}))=\dim(\tilde{V}^{i})$ for some eigenvalue $\tilde{\mu}$ of $\tilde{s}$ on $\tilde{V}$, where $0\leq i\leq 2$, then $\tilde{s}\in \ZG(\tilde{L})^{\circ}\setminus \ZG(\tilde{G})$. In this case, as $\tilde{s}$ acts on each $\tilde{V}^{i}$ as scalar multiplication by $c^{2-2i}$ and $c^{2}\neq 1$, it follows that $\dim(\tilde{V}_{\tilde{s}}(\tilde{\mu}))\leq 255-91\varepsilon_{p}(2)$ for all eigenvalues $\tilde{\mu}$ of $\tilde{s}$ on $\tilde{V}$. We thus assume that $\dim(\tilde{V}^{i}_{\tilde{s}}(\tilde{\mu}))<\dim(\tilde{V}^{i})$ for all eigenvalues $\tilde{\mu}$ of $\tilde{s}$ on $\tilde{V}$ and all $0\leq i\leq 2$. We write $\tilde{s}=\tilde{z}\cdot \tilde{h}$, where $\tilde{z}\in \ZG(\tilde{L})^{\circ}$ and $\tilde{h}\in [\tilde{L},\tilde{L}]$, and, by \eqref{DecompVD6om4} and Propositions \ref{PropositionDlwedge}, \ref{D5om3p=2}, \ref{D5om3pneq2} and \ref{D5om4+om5}, we determine that $\dim(\tilde{V}_{\tilde{s}}(\tilde{\mu}))\leq 2\dim((L_{\tilde{L}}(\tilde{\omega}_{4}))_{\tilde{h}}(\tilde{\mu}_{\tilde{h}}))+\dim((L_{\tilde{L}}(\tilde{\omega}_{5}+\tilde{\omega}_{6}))_{\tilde{h}}(\tilde{\mu}_{\tilde{h}}))+(1-\varepsilon_{p}(2))\dim((L_{\tilde{L}}(\tilde{\omega}_{3}))_{\tilde{h}}(\tilde{\mu}_{\tilde{h}}))\leq 303-127\varepsilon_{p}(2)$ for all eigenvalues $\tilde{\mu}$ of $\tilde{s}$ on $\tilde{V}$. Therefore, $\scale[0.9]{\displaystyle \max_{\tilde{s}\in \tilde{T}\setminus\ZG(\tilde{G})}\dim(\tilde{V}_{\tilde{s}}(\tilde{\mu}))}\leq 303-127\varepsilon_{p}(2)$. 

For the unipotent elements, by Lemma \ref{uniprootelems}, decomposition \eqref{DecompVD6om4} and Propositions \ref{D5om3p=2}, \ref{D5om3pneq2}, \ref{D5om4+om5} and \ref{PropositionDlwedge}, we have $\scale[0.9]{\displaystyle \max_{\tilde{u}\in \tilde{G}_{u}\setminus \{1\}}\dim(\tilde{V}_{\tilde{u}}(1))=}$ $\dim(\tilde{V}_{x_{\tilde{\alpha}_{6}}(1)}(1))\leq 311-91\varepsilon_{p}(2)$, where equality holds for $p\neq 2$. Lastly, we note that $\nu_{\tilde{G}}(\tilde{V})\geq 184-40\varepsilon_{p}(2)$, where equality holds for $p\neq 2$.
\end{proof}

\begin{prop}\label{D62om5}
Let $p\neq 2$, $\ell=6$ and $\tilde{V}=L_{\tilde{G}}(2\tilde{\omega}_{5})$. Then $\nu_{\tilde{G}}(\tilde{V})\geq 172$. Moreover, we have $\scale[0.85]{\displaystyle \max_{\tilde{u}\in \tilde{G}_{u}\setminus \{1\}}\dim(\tilde{V}_{\tilde{u}}(1))}$ $=280$ and $\scale[0.9]{\displaystyle \max_{\tilde{s}\in \tilde{T}\setminus\ZG(\tilde{G})}\dim(\tilde{V}_{\tilde{s}}(\tilde{\mu}))}\leq 290$.
\end{prop}

\begin{proof}
Let $\tilde{\lambda}=2\tilde{\omega}_{5}$ and $\tilde{L}=\tilde{L}_{1}$. By Lemma \ref{weightlevelDl}, we have $e_{1}(\tilde{\lambda})=2$, therefore $\displaystyle \tilde{V}\mid_{[\tilde{L},\tilde{L}]}=\tilde{V}^{0}\oplus \tilde{V}^{1}\oplus \tilde{V}^{2}$. By \cite[Proposition]{Smith_82}, we have $\tilde{V}^{0}\cong L_{\tilde{L}}(2\tilde{\omega}_{5})$. Now, in $\tilde{V}^{1}$ the weight $\displaystyle (\tilde{\lambda}-\tilde{\alpha}_{1}-\cdots-\tilde{\alpha}_{5})\mid_{\tilde{T}_{1}}=\tilde{\omega}_{5}+\tilde{\omega}_{6}$ admits a maximal vector, thus $\tilde{V}^{1}$ has a composition factor isomorphic to $L_{\tilde{L}}(\tilde{\omega}_{5}+\tilde{\omega}_{6})$. Similarly, the weight $\displaystyle (\tilde{\lambda}-2\tilde{\alpha}_{1}-\cdots-2\tilde{\alpha}_{5})\mid_{\tilde{T}_{1}}=2\tilde{\omega}_{6}$ admits a maximal vector in $\tilde{V}^{3}$, thus $\tilde{V}^{3}$ has a composition factor isomorphic to $L_{\tilde{L}}(2\tilde{\omega}_{6})$. By dimensional considerations, we determine that 
\begin{equation}\label{DecompVD62om5}
\tilde{V}\mid_{[\tilde{L},\tilde{L}]}\cong L_{\tilde{L}}(2\tilde{\omega}_{5}) \oplus L_{\tilde{L}}(\tilde{\omega}_{5}+\tilde{\omega}_{6})\oplus L_{\tilde{L}}(2\tilde{\omega}_{6}).
\end{equation}

We begin with the semisimple elements. Let $\tilde{s}\in \tilde{T}\setminus \ZG(\tilde{G})$. If $\dim(\tilde{V}^{i}_{\tilde{s}}(\tilde{\mu}))=\dim(\tilde{V}^{i})$ for some eigenvalue $\tilde{\mu}$ of $\tilde{s}$ on $\tilde{V}$, where $0\leq i\leq 2$, then $\tilde{s}\in \ZG(\tilde{L})^{\circ}\setminus \ZG(\tilde{G})$. In this case, as $\tilde{s}$ acts on each $\tilde{V}^{i}$ as scalar multiplication by $c^{2-2i}$ and $c^{2}\neq 1$, it follows that $\dim(\tilde{V}_{\tilde{s}}(\tilde{\mu}))\leq 252$ for all eigenvalues $\tilde{\mu}$ of $\tilde{s}$ on $\tilde{V}$. We thus assume that $\dim(\tilde{V}^{i}_{\tilde{s}}(\tilde{\mu}))<\dim(\tilde{V}^{i})$ for all eigenvalues $\tilde{\mu}$ of $\tilde{s}$ on $\tilde{V}$ and all $0\leq i\leq 2$. We write $\tilde{s}=\tilde{z}\cdot \tilde{h}$, where $\tilde{z}\in \ZG(\tilde{L})^{\circ}$ and $\tilde{h}\in [\tilde{L},\tilde{L}]$, and, by \eqref{DecompVD62om5} and Propositions \ref{D52om5p} and \ref{D5om4+om5}, we determine that $\dim(\tilde{V}_{\tilde{s}}(\tilde{\mu}))\leq 2\dim((L_{\tilde{L}}(2\tilde{\omega}_{5}))_{\tilde{h}}(\tilde{\mu}_{\tilde{h}}))+\dim((L_{\tilde{L}}(\tilde{\omega}_{5}+\tilde{\omega}_{6}))_{\tilde{h}}(\tilde{\mu}_{\tilde{h}}))\leq 290$ for all eigenvalues $\tilde{\mu}$ of $\tilde{s}$ on $\tilde{V}$. Therefore, $\scale[0.9]{\displaystyle \max_{\tilde{s}\in \tilde{T}\setminus\ZG(\tilde{G})}\dim(\tilde{V}_{\tilde{s}}(\tilde{\mu}))}\leq 290$. 

For the unipotent elements, by Lemma \ref{uniprootelems}, decomposition \eqref{DecompVD62om5} and Propositions \ref{D52om5p} and \ref{D5om4+om5}, we have $\scale[0.9]{\displaystyle \max_{\tilde{u}\in \tilde{G}_{u}\setminus \{1\}}\dim(\tilde{V}_{\tilde{u}}(1))=}$ $\dim(\tilde{V}_{x_{\tilde{\alpha}_{6}}(1)}(1))= 280$. Lastly, we note that $\nu_{\tilde{G}}(\tilde{V})\geq 172$.
\end{proof}

\begin{prop}\label{D6om5+om6}
Let $\ell=6$ and $\tilde{V}=L_{\tilde{G}}(\tilde{\omega}_{5}+\tilde{\omega}_{6})$. Then $\nu_{\tilde{G}}(\tilde{V})\geq 300-84\varepsilon_{p}(2)$. Moreover, we have $\scale[0.9]{\displaystyle \max_{\tilde{u}\in \tilde{G}_{u}\setminus \{1\}}\dim(\tilde{V}_{\tilde{u}}(1))}\leq 484-140\varepsilon_{p}(2)$ and $\scale[0.9]{\displaystyle \max_{\tilde{s}\in \tilde{T}\setminus\ZG(\tilde{G})}\dim(\tilde{V}_{\tilde{s}}(\tilde{\mu}))}\leq 492-216\varepsilon_{p}(2)$.
\end{prop}

\begin{proof}
Let $\tilde{\lambda}=\tilde{\omega}_{5}+\tilde{\omega}_{6}$ and $\tilde{L}=\tilde{L}_{1}$. By Lemma \ref{weightlevelDl}, we have $e_{1}(\tilde{\lambda})=2$, therefore $\displaystyle \tilde{V}\mid_{[\tilde{L},\tilde{L}]}=\tilde{V}^{0}\oplus \tilde{V}^{1}\oplus \tilde{V}^{2}$. By \cite[Proposition]{Smith_82} and Lemma \ref{dualitylemma}, we have $\tilde{V}^{0}\cong L_{\tilde{L}}(\tilde{\omega}_{5}+\tilde{\omega}_{6})$ and $\tilde{V}^{2}\cong L_{\tilde{L}}(\tilde{\omega}_{5}+\tilde{\omega}_{6})$. Now, in $\tilde{V}^{1}$ both the weight $\displaystyle (\tilde{\lambda}-\tilde{\alpha}_{1}-\cdots-\tilde{\alpha}_{5})\mid_{\tilde{T}_{1}}=2\tilde{\omega}_{6}$ and the weight $\displaystyle (\tilde{\lambda}-\tilde{\alpha}_{1}-\cdots-\tilde{\alpha}_{4}-\tilde{\alpha}_{6})\mid_{\tilde{T}_{1}}=2\tilde{\omega}_{5}$ admits a maximal vector, thus $\tilde{V}^{1}$ has two composition factors one isomorphic to $L_{\tilde{L}}(2\tilde{\omega}_{6})$ and one to $L_{\tilde{L}}(2\tilde{\omega}_{5})$. Moreover, the weight $\displaystyle (\tilde{\lambda}-\tilde{\alpha}_{1}-\cdots-\tilde{\alpha}_{6})\mid_{\tilde{T}_{1}}=\tilde{\omega}_{4}$ occurs with multiplicity $3-\varepsilon_{p}(2)$ and is a sub-dominant weight in each of the composition factors of $\tilde{V}^{1}$ we identified, where it has multiplicity $1$, if and only if $p\neq 2$. Thus, by dimensional considerations, we determine that $\tilde{V}^{1}$ has exactly $3+\varepsilon_{p}(2)$ composition factors: one isomorphic to $L_{\tilde{L}}(2\tilde{\omega}_{6})$, one to $L_{\tilde{L}}(2\tilde{\omega}_{5})$ and $1+\varepsilon_{p}(2)$ to $L_{\tilde{L}}(\tilde{\omega}_{4})$. Further, when $p\neq 2$, by \cite[II.2.14]{Jantzen_2007representations},we have $\tilde{V}^{1}\cong L_{\tilde{L}}(2\tilde{\omega}_{6})\oplus L_{\tilde{L}}(2\tilde{\omega}_{5}) \oplus L_{\tilde{L}}(\tilde{\omega}_{4})$ and $\tilde{V}^{2}\cong L_{\tilde{L}}(2\tilde{\omega}_{6})\oplus L_{\tilde{L}}(2\tilde{\omega}_{5}) \oplus L_{\tilde{L}}(\tilde{\omega}_{4})$.

We begin with the semisimple elements. Let $\tilde{s}\in \tilde{T}\setminus \ZG(\tilde{G})$. If $\dim(\tilde{V}^{i}_{\tilde{s}}(\tilde{\mu}))=\dim(\tilde{V}^{i})$ for some eigenvalue $\tilde{\mu}$ of $\tilde{s}$ on $\tilde{V}$, where $0\leq i\leq 2$, then $\tilde{s}\in \ZG(\tilde{L})^{\circ}\setminus \ZG(\tilde{G})$. In this case, as $\tilde{s}$ acts on each $\tilde{V}^{i}$ as scalar multiplication by $c^{2-2i}$ and $c^{2}\neq 1$, it follows that $\dim(\tilde{V}_{\tilde{s}}(\tilde{\mu}))\leq 420-188\varepsilon_{p}(2)$ for all eigenvalues $\tilde{\mu}$ of $\tilde{s}$ on $\tilde{V}$. We thus assume that $\dim(\tilde{V}^{i}_{\tilde{s}}(\tilde{\mu}))<\dim(\tilde{V}^{i})$ for all eigenvalues $\tilde{\mu}$ of $\tilde{s}$ on $\tilde{V}$ and all $0\leq i\leq 2$. We write $\tilde{s}=\tilde{z}\cdot \tilde{h}$, where $\tilde{z}\in \ZG(\tilde{L})^{\circ}$ and $\tilde{h}\in [\tilde{L},\tilde{L}]$, and, by  the structure of $\tilde{V}\mid_{[\tilde{L},\tilde{L}]}$ and Propositions \ref{D5om3p=2}, \ref{D5om3pneq2}, \ref{Dloml}, \ref{D52om5p} and \ref{D5om4+om5}, we deduce that $\dim(\tilde{V}_{\tilde{s}}(\tilde{\mu}))\leq 2\dim((L_{\tilde{L}}(\tilde{\omega}_{5}+\tilde{\omega}_{6}))_{\tilde{h}}(\tilde{\mu}_{\tilde{h}}))+2\dim((L_{\tilde{L}}(2\tilde{\omega}_{5}))_{\tilde{h}}(\tilde{\mu}_{\tilde{h}}))+(1+\varepsilon_{p}(2))\dim((L_{\tilde{L}}(\tilde{\omega}_{4}))_{\tilde{h}}(\tilde{\mu}_{\tilde{h}}))\leq 492-216\varepsilon_{p}(2)$ for all eigenvalues $\tilde{\mu}$ of $\tilde{s}$ on $\tilde{V}$. Therefore, $\scale[0.9]{\displaystyle \max_{\tilde{s}\in \tilde{T}\setminus\ZG(\tilde{G})}\dim(\tilde{V}_{\tilde{s}}(\tilde{\mu}))}\leq 492-216\varepsilon_{p}(2)$. 

For the unipotent elements, by Lemma \ref{uniprootelems}, the structure of $\tilde{V}\mid_{[\tilde{L},\tilde{L}]}$ and Propositions \ref{D5om3p=2}, \ref{D5om3pneq2}, \ref{Dloml}, \ref{D52om5p} and \ref{D5om4+om5}, we have $\scale[0.9]{\displaystyle \max_{\tilde{u}\in \tilde{G}_{u}\setminus \{1\}}\dim(\tilde{V}_{\tilde{u}}(1))=}$ $\dim(\tilde{V}_{x_{\tilde{\alpha}_{6}}(1)}(1))\leq 484-140\varepsilon_{p}(2)$, where equality holds when $p\neq 2$. Lastly, we note that $\nu_{\tilde{G}}(\tilde{V})\geq 300-84\varepsilon_{p}(2)$.
\end{proof}

\begin{prop}\label{D7om1+om7}
Let $\ell=7$ and $\tilde{V}=L_{\tilde{G}}(\tilde{\omega}_{1}+\tilde{\omega}_{7})$. Then $\nu_{\tilde{G}}(\tilde{V})\geq 272-16\varepsilon_{p}(7)$, where equality holds for $p\neq 2,3,5$. Moreover, we have $\scale[0.9]{\displaystyle \max_{\tilde{u}\in \tilde{G}_{u}\setminus \{1\}}\dim(\tilde{V}_{\tilde{u}}(1))}\leq 560-48\varepsilon_{p}(7)$, where equality holds for $p\neq 2,3,5$, and $\scale[0.9]{\displaystyle \max_{\tilde{s}\in \tilde{T}\setminus\ZG(\tilde{G})}\dim(\tilde{V}_{\tilde{s}}(\tilde{\mu}))}\leq 528-40\varepsilon_{p}(7)-8\varepsilon_{p}(5)-64\varepsilon_{p}(2)$.
\end{prop}

\begin{proof}
Let $\tilde{\lambda}=\tilde{\omega}_{1}+\tilde{\omega}_{7}$ and $\tilde{L}=\tilde{L}_{1}$. By Lemma \ref{weightlevelDl}, we have $e_{1}(\tilde{\lambda})=3$, therefore $\displaystyle \tilde{V}\mid_{[\tilde{L},\tilde{L}]}=\tilde{V}^{0}\oplus \cdots \oplus \tilde{V}^{3}$. By \cite[Proposition]{Smith_82}, we have $\tilde{V}^{0}\cong L_{\tilde{L}}(\tilde{\omega}_{7})$. Now, in $\tilde{V}^{1}$ the weight $\displaystyle (\tilde{\lambda}-\tilde{\alpha}_{1})\mid_{\tilde{T}_{1}}=\tilde{\omega}_{2}+\tilde{\omega}_{7}$ admits a maximal vector, thus $\tilde{V}^{1}$ has a composition factor isomorphic to $L_{\tilde{L}}(\tilde{\omega}_{2}+\tilde{\omega}_{7})$. Moreover, the weight $\displaystyle (\tilde{\lambda}-\tilde{\alpha}_{1}-\cdots-\tilde{\alpha}_{5}-\tilde{\alpha}_{7})\mid_{\tilde{T}_{1}}=\tilde{\omega}_{6}$ occurs with multiplicity $6-\varepsilon_{p}(7)$ and is a sub-dominant weight in the composition factor of $\tilde{V}^{1}$ isomorphic to $L_{\tilde{L}}(\tilde{\omega}_{2}+\tilde{\omega}_{7})$, in which it has multiplicity $5-\varepsilon_{p}(3)-\varepsilon_{p}(2)$. Similarly, in $\tilde{V}^{2}$, the weight $\displaystyle (\tilde{\lambda}-2\tilde{\alpha}_{1}-\tilde{\alpha}_{2}-\cdots-\tilde{\alpha}_{5}-\tilde{\alpha}_{7})\mid_{\tilde{T}_{1}}=\tilde{\omega}_{2}+\tilde{\omega}_{6}$ admits a maximal vector, thus $\tilde{V}^{2}$ has a composition factor isomorphic to $L_{\tilde{L}}(\tilde{\omega}_{2}+\tilde{\omega}_{6})$. Moreover, the weight $\displaystyle (\tilde{\lambda}-2\tilde{\alpha}_{1}-\cdots-2\tilde{\alpha}_{5}-\tilde{\alpha}_{6}-\tilde{\alpha}_{7})\mid_{\tilde{T}_{1}}=\tilde{\omega}_{7}$ occurs with multiplicity $6-\varepsilon_{p}(7)$ and is a sub-dominant weight in the composition factor of $\tilde{V}^{2}$ isomorphic to $L_{\tilde{L}}(\tilde{\omega}_{2}+\tilde{\omega}_{6})$, in which it has multiplicity $5-\varepsilon_{p}(3)-\varepsilon_{p}(2)$. Lastly, in $\tilde{V}^{3}$ the weight $\displaystyle (\tilde{\lambda}-3\tilde{\alpha}_{1}-\cdots-3\tilde{\alpha}_{5}-\tilde{\alpha}_{6}-2\tilde{\alpha}_{7})\mid_{\tilde{T}_{1}}=\tilde{\omega}_{6}$ admits a maximal vector, thus $\tilde{V}^{3}$ has a composition factor isomorphic to $L_{\tilde{L}}(\tilde{\omega}_{6})$. By dimensional considerations, we determine that $\tilde{V}^{3}\cong L_{\tilde{L}}(\tilde{\omega}_{6})$ and that $\tilde{V}^{1}$, respectively $\tilde{V}^{2}$, has exactly $2-\varepsilon_{p}(7)+\varepsilon_{p}(3)+\varepsilon_{p}(2)$ composition factors: one isomorphic to $L_{\tilde{L}}(\tilde{\omega}_{2}+\tilde{\omega}_{7})$, respectively to $L_{\tilde{L}}(\tilde{\omega}_{2}+\tilde{\omega}_{6})$, and $1-\varepsilon_{p}(7)+\varepsilon_{p}(3)+\varepsilon_{p}(2)$ to $L_{\tilde{L}}(\tilde{\omega}_{6})$, respectively to $L_{\tilde{L}}(\tilde{\omega}_{7})$. Further, when $\varepsilon_{p}(6)=0$, by \cite[II.2.14]{Jantzen_2007representations},we have $\tilde{V}^{1}\cong L_{\tilde{L}}(\tilde{\omega}_{2}+\tilde{\omega}_{7})\oplus L_{\tilde{L}}(\tilde{\omega}_{6})^{1-\varepsilon_{p}(7)}$ and $\tilde{V}^{2}\cong L_{\tilde{L}}(\tilde{\omega}_{2}+\tilde{\omega}_{6})\oplus L_{\tilde{L}}(\tilde{\omega}_{7})^{1-\varepsilon_{p}(7)}$.

We begin with the semisimple elements. Let $\tilde{s}\in \tilde{T}\setminus \ZG(\tilde{G})$. If $\dim(\tilde{V}^{i}_{\tilde{s}}(\tilde{\mu}))=\dim(\tilde{V}^{i})$ for some eigenvalue $\tilde{\mu}$ of $\tilde{s}$ on $\tilde{V}$, where $0\leq i\leq 3$, then $\tilde{s}\in \ZG(\tilde{L})^{\circ}\setminus \ZG(\tilde{G})$. In this case, as $\tilde{s}$ acts on each $\tilde{V}^{i}$ as scalar multiplication by $c^{3-2i}$ and $c^{2}\neq 1$, it follows that $\dim(\tilde{V}_{\tilde{s}}(\tilde{\mu}))\leq 416-32\varepsilon_{p}(7)-32\varepsilon_{p}(2)$ for all eigenvalues $\tilde{\mu}$ of $\tilde{s}$ on $\tilde{V}$. We thus assume that $\dim(\tilde{V}^{i}_{\tilde{s}}(\tilde{\mu}))<\dim(\tilde{V}^{i})$ for all eigenvalues $\tilde{\mu}$ of $\tilde{s}$ on $\tilde{V}$ and all $0\leq i\leq 3$. We write $\tilde{s}=\tilde{z}\cdot \tilde{h}$, where $\tilde{z}\in \ZG(\tilde{L})^{\circ}$ and $\tilde{h}\in [\tilde{L},\tilde{L}]$, and, by the structure of $\tilde{V}\mid_{[\tilde{L},\tilde{L}]}$ and Propositions \ref{Dloml} and \ref{D6om1+om5}, we determine that $\dim(\tilde{V}_{\tilde{s}}(\tilde{\mu}))\leq (4-2\varepsilon_{p}(7)+2\varepsilon_{p}(3)+2\varepsilon_{p}(2))\dim((L_{\tilde{L}}(\tilde{\omega}_{6}))_{\tilde{h}}(\tilde{\mu}_{\tilde{h}}))+2\dim((L_{\tilde{L}}(\tilde{\omega}_{2}+\tilde{\omega}_{6}))_{\tilde{h}}(\tilde{\mu}_{\tilde{h}}))\leq 528-40\varepsilon_{p}(7)-8\varepsilon_{p}(5)-64\varepsilon_{p}(2)$ for all eigenvalues $\tilde{\mu}$ of $\tilde{s}$ on $\tilde{V}$. Therefore, $\scale[0.9]{\displaystyle \max_{\tilde{s}\in \tilde{T}\setminus\ZG(\tilde{G})}\dim(\tilde{V}_{\tilde{s}}(\tilde{\mu}))}\leq 528-40\varepsilon_{p}(7)-8\varepsilon_{p}(5)-64\varepsilon_{p}(2)$. 

For the unipotent elements, by Lemma \ref{uniprootelems}, the structure of $\tilde{V}\mid_{[\tilde{L},\tilde{L}]}$ and Propositions \ref{Dloml} and \ref{D6om1+om5}, we have $\scale[0.9]{\displaystyle \max_{\tilde{u}\in \tilde{G}_{u}\setminus \{1\}}\dim(\tilde{V}_{\tilde{u}}(1))=}$ $\dim(\tilde{V}_{x_{\tilde{\alpha}_{7}}(1)}(1))\leq 560-48\varepsilon_{p}(7)$, where equality holds for $p\neq 2,3,5$. Lastly, we note that $\nu_{\tilde{G}}(\tilde{V})\geq 272-16\varepsilon_{p}(7)$, where equality holds for $p\neq 2,3,5$.
\end{proof}

\begin{prop}\label{D7om4}
Let $\ell=7$ and $\tilde{V}=L_{\tilde{G}}(\tilde{\omega}_{4})$. Then $\nu_{\tilde{G}}(\tilde{V})\geq 350-26\varepsilon_{p}(2)$, where equality holds for $p\neq 2$. Moreover, we have $\scale[0.9]{\displaystyle \max_{\tilde{u}\in \tilde{G}_{u}\setminus \{1\}}\dim(\tilde{V}_{\tilde{u}}(1))}\leq 651-65\varepsilon_{p}(2)$, where equality holds for $p\neq 2$, and $\scale[0.9]{\displaystyle \max_{\tilde{s}\in \tilde{T}\setminus\ZG(\tilde{G})}\dim(\tilde{V}_{\tilde{s}}(\tilde{\mu}))}\leq 625-119\varepsilon_{p}(2)$.
\end{prop}

\begin{proof}
Let $\tilde{\lambda}=\tilde{\omega}_{4}$ and $\tilde{L}=\tilde{L}_{1}$. By Lemma \ref{weightlevelDl}, we have $e_{1}(\tilde{\lambda})=2$, therefore $\displaystyle \tilde{V}\mid_{[\tilde{L},\tilde{L}]}=\tilde{V}^{0}\oplus \tilde{V}^{1} \oplus \tilde{V}^{2}$. By \cite[Proposition]{Smith_82} and Lemma \ref{dualitylemma}, we have $\tilde{V}^{0}\cong L_{\tilde{L}}(\tilde{\omega}_{4})$ and $\tilde{V}^{2}\cong L_{\tilde{L}}(\tilde{\omega}_{4})$. Now, in $\tilde{V}^{1}$ the weight $\displaystyle (\tilde{\lambda}-\tilde{\alpha}_{1}-\cdots-\tilde{\alpha}_{4})\mid_{\tilde{T}_{1}}=\tilde{\omega}_{5}$ admits a maximal vector, thus $\tilde{V}^{1}$ has a composition factor isomorphic to $L_{\tilde{L}}(\tilde{\omega}_{5})$. Moreover, the weight $\displaystyle (\tilde{\lambda}-\tilde{\alpha}_{1}-\tilde{\alpha}_{2}-\tilde{\alpha}_{3}-2\tilde{\alpha}_{4}-2\tilde{\alpha}_{5}-\tilde{\alpha}_{6}-\tilde{\alpha}_{7})\mid_{\tilde{T}_{1}}=\tilde{\omega}_{3}$ occurs with multiplicity $5-\varepsilon_{p}(2)$ and is a sub-dominant weight in the composition factor of $\tilde{V}^{1}$ isomorphic to $L_{\tilde{L}}(\tilde{\omega}_{5})$, in which it has multiplicity $4-2\varepsilon_{p}(2)$. Further, we note that the weight $\displaystyle (\tilde{\lambda}-\tilde{\alpha}_{1}-2\tilde{\alpha}_{2}\cdots-2\tilde{\alpha}_{5}-\tilde{\alpha}_{6}-\tilde{\alpha}_{7})\mid_{\tilde{T}_{1}}=0$ occurs with multiplicity $21-7\varepsilon_{p}(2)$ in $\tilde{V}^{1}$. By dimensional considerations, we determine that $\tilde{V}^{1}$ has exactly $2+3\varepsilon_{p}(2)$ composition factors: one isomorphic to $L_{\tilde{L}}(\tilde{\omega}_{5})$, $1+\varepsilon_{p}(2)$ to $L_{\tilde{L}}(\tilde{\omega}_{3})$ and $2\varepsilon_{p}(2)$ to $L_{\tilde{L}}(0)$. Further, when $p\neq 2$, by \cite[II.2.14]{Jantzen_2007representations},we have $\tilde{V}^{1}\cong L_{\tilde{L}}(\tilde{\omega}_{5})\oplus L_{\tilde{L}}(\tilde{\omega}_{3})$.

We begin with the semisimple elements. Let $\tilde{s}\in \tilde{T}\setminus \ZG(\tilde{G})$. If $\dim(\tilde{V}^{i}_{\tilde{s}}(\tilde{\mu}))=\dim(\tilde{V}^{i})$ for some eigenvalue $\tilde{\mu}$ of $\tilde{s}$ on $\tilde{V}$, where $0\leq i\leq 2$, then $\tilde{s}\in \ZG(\tilde{L})^{\circ}\setminus \ZG(\tilde{G})$. In this case, as $\tilde{s}$ acts on each $\tilde{V}^{i}$ as scalar multiplication by $c^{2-2i}$ and $c^{2}\neq 1$, it follows that $\dim(\tilde{V}_{\tilde{s}}(\tilde{\mu}))\leq 561-67\varepsilon_{p}(2)$ for all eigenvalues $\tilde{\mu}$ of $\tilde{s}$ on $\tilde{V}$. We thus assume that $\dim(\tilde{V}^{i}_{\tilde{s}}(\tilde{\mu}))<\dim(\tilde{V}^{i})$ for all eigenvalues $\tilde{\mu}$ of $\tilde{s}$ on $\tilde{V}$ and all $0\leq i\leq 2$. We write $\tilde{s}=\tilde{z}\cdot \tilde{h}$, where $\tilde{z}\in \ZG(\tilde{L})^{\circ}$ and $\tilde{h}\in [\tilde{L},\tilde{L}]$, and, by  the structure of $\tilde{V}\mid_{[\tilde{L},\tilde{L}]}$ and Propositions \ref{PropositionDlwedge},  \ref{PropositionDlwedgecubep=2}, \ref{PropositionDlwedgecubepneq2} and \ref{D6om4}, we determine that $\dim(\tilde{V}_{\tilde{s}}(\tilde{\mu}))\leq 2\dim((L_{\tilde{L}}(\tilde{\omega}_{4}))_{\tilde{h}}(\tilde{\mu}_{\tilde{h}}))+\dim((L_{\tilde{L}}(\tilde{\omega}_{5}))_{\tilde{h}}(\tilde{\mu}_{\tilde{h}}))+(1+\varepsilon_{p}(2))\dim((L_{\tilde{L}}(\tilde{\omega}_{3}))_{\tilde{h}}(\tilde{\mu}_{\tilde{h}}))+2\varepsilon_{p}(2)\dim((L_{\tilde{L}}(0))_{\tilde{h}}(\tilde{\mu}_{\tilde{h}}))\leq 625-119\varepsilon_{p}(2)$ for all eigenvalues $\tilde{\mu}$ of $\tilde{s}$ on $\tilde{V}$. Therefore, $\scale[0.9]{\displaystyle \max_{\tilde{s}\in \tilde{T}\setminus\ZG(\tilde{G})}\dim(\tilde{V}_{\tilde{s}}(\tilde{\mu}))}\leq 625-119\varepsilon_{p}(2)$. 

For the unipotent elements, by Lemma \ref{uniprootelems},  the structure of $\tilde{V}\mid_{[\tilde{L},\tilde{L}]}$ and Propositions \ref{PropositionDlwedge},  \ref{PropositionDlwedgecubep=2}, \ref{PropositionDlwedgecubepneq2} and \ref{D6om4}, we have $\scale[0.9]{\displaystyle \max_{\tilde{u}\in \tilde{G}_{u}\setminus \{1\}}\dim(\tilde{V}_{\tilde{u}}(1))=}$ $\dim(\tilde{V}_{x_{\tilde{\alpha}_{7}}(1)}(1))\leq 651-65\varepsilon_{p}(2)$, where equality holds for $p\neq 2$. Lastly, we note that $\nu_{\tilde{G}}(V)\geq 350-26\varepsilon_{p}(2)$, where equality holds for $p\neq 2$.
\end{proof}

\begin{prop}\label{D7om5p=2}
Let $p=2$, $\ell=7$ and $\tilde{V}=L_{\tilde{G}}(\tilde{\omega}_{5})$. Then $\nu_{\tilde{G}}(\tilde{V})\geq 340$. Moreover, we have $\scale[0.9]{\displaystyle \max_{\tilde{u}\in \tilde{G}_{u}\setminus \{1\}}\dim(\tilde{V}_{\tilde{u}}(1))}$ $\leq 948$ and $\scale[0.9]{\displaystyle \max_{\tilde{s}\in \tilde{T}\setminus\ZG(\tilde{G})}\dim(\tilde{V}_{\tilde{s}}(\tilde{\mu}))}\leq 608$.
\end{prop}

\begin{proof}
First, we note that $\tilde{V}\cong L_{G}(\omega_{5})$, as $k\tilde{G}$-modules, Further, as $p=2$, by \cite[Table $1$]{seitz1987maximal}, we also have $L_{G}(\omega_{5})\cong L_{H}(\omega_{5}^{H})$ as $kG$-modules. Thus, in view of Lemma \ref{LtildeGtildelambdaandLGlambda} and by Proposition \ref{C7om5p=2}, we have $\scale[0.9]{\displaystyle \max_{\tilde{s}\in \tilde{T}\setminus\ZG(\tilde{G})}\dim(\tilde{V}_{\tilde{s}}(\tilde{\mu}))}$ $\scale[0.9]{\displaystyle=\max_{s_{H}\in T_{H}\setminus\ZG(H)}\dim((L_{H}(\omega_{5}^{H}))_{s_{H}}(\mu_{H}))}\leq 608$ and $\scale[0.9]{\displaystyle \max_{\tilde{u}\in \tilde{G}_{u}\setminus \{1\}}\dim(\tilde{V}_{\tilde{u}}(1))}$ $\scale[0.9]{\displaystyle=\max_{u\in H_{u}\setminus \{1\}}\dim((L_{H}(\omega_{5}^{H})_{u}(1))}$ $\leq 948$. Lastly, we note that $\nu_{\tilde{G}}(\tilde{V})\geq 340$.
\end{proof}

\begin{prop}\label{D8om4}
Let $\ell=8$ and $\tilde{V}=L_{\tilde{G}}(\tilde{\omega}_{4})$. Then $\nu_{\tilde{G}}(\tilde{V})\geq 596-56\varepsilon_{p}(2)$, where equality holds for $p\neq 2$. Moreover, we have $\scale[0.9]{\displaystyle \max_{\tilde{u}\in \tilde{G}_{u}\setminus \{1\}}\dim(\tilde{V}_{\tilde{u}}(1))}\leq 1224-182\varepsilon_{p}(2)$, where equality holds for $p\neq 2$, and $\scale[0.9]{\displaystyle \max_{\tilde{s}\in \tilde{T}\setminus\ZG(\tilde{G})}\dim(\tilde{V}_{\tilde{s}}(\tilde{\mu}))}\leq 1172-250\varepsilon_{p}(2)$.
\end{prop}

\begin{proof}
Let $\tilde{\lambda}=\tilde{\omega}_{4}$ and $\tilde{L}=\tilde{L}_{1}$. By Lemma \ref{weightlevelDl}, we have $e_{1}(\tilde{\lambda})=2$, therefore $\displaystyle \tilde{V}\mid_{[\tilde{L},\tilde{L}]}=\tilde{V}^{0}\oplus \tilde{V}^{1} \oplus \tilde{V}^{2}$. By \cite[Proposition]{Smith_82} and Lemma \ref{dualitylemma}, we have $\tilde{V}^{0}\cong L_{\tilde{L}}(\tilde{\omega}_{4})$ and $\tilde{V}^{2}\cong L_{\tilde{L}}(\tilde{\omega}_{4})$. Now, in $\tilde{V}^{1}$ the weight $\displaystyle (\tilde{\lambda}-\tilde{\alpha}_{1}-\cdots-\tilde{\alpha}_{4})\mid_{\tilde{T}_{1}}=\tilde{\omega}_{5}$ admits a maximal vector, thus $\tilde{V}^{1}$ has a composition factor isomorphic to $L_{\tilde{L}}(\tilde{\omega}_{5})$. Moreover, the weight $\displaystyle (\tilde{\lambda}-\tilde{\alpha}_{1}-\tilde{\alpha}_{2}-\tilde{\alpha}_{3}-2\tilde{\alpha}_{4}-2\tilde{\alpha}_{5}-2\tilde{\alpha}_{6}-\tilde{\alpha}_{7}-\tilde{\alpha}_{8})\mid_{\tilde{T}_{1}}=\tilde{\omega}_{3}$ occurs with multiplicity $6-2\varepsilon_{p}(2)$ and is a sub-dominant weight in the composition factor of $\tilde{V}^{1}$ isomorphic to $L_{\tilde{L}}(\tilde{\omega}_{5})$, in which it has multiplicity $5-\varepsilon_{p}(2)$. Thus, as $\dim(\tilde{V}^{1})=1092-182\varepsilon_{p}(2)$, by \cite[II.2.14]{Jantzen_2007representations}, it follows that $\tilde{V}^{1}\cong L_{\tilde{L}}(\tilde{\omega}_{5})\oplus L_{\tilde{L}}(\tilde{\omega}_{3})^{1-\varepsilon_{p}(2)}$.

We begin with the semisimple elements. Let $\tilde{s}\in \tilde{T}\setminus \ZG(\tilde{G})$. If $\dim(\tilde{V}^{i}_{\tilde{s}}(\tilde{\mu}))=\dim(\tilde{V}^{i})$ for some eigenvalue $\tilde{\mu}$ of $\tilde{s}$ on $\tilde{V}$, where $0\leq i\leq 2$, then $\tilde{s}\in \ZG(\tilde{L})^{\circ}\setminus \ZG(\tilde{G})$. In this case, as $\tilde{s}$ acts on each $\tilde{V}^{i}$ as scalar multiplication by $c^{2-2i}$ and $c^{2}\neq 1$, it follows that $\dim(\tilde{V}_{\tilde{s}}(\tilde{\mu}))\leq 1092-182\varepsilon_{p}(2)$ for all eigenvalues $\tilde{\mu}$ of $\tilde{s}$ on $\tilde{V}$. We thus assume that $\dim(\tilde{V}^{i}_{\tilde{s}}(\tilde{\mu}))<\dim(\tilde{V}^{i})$ for all eigenvalues $\tilde{\mu}$ of $\tilde{s}$ on $\tilde{V}$ and all $0\leq i\leq 2$. We write $\tilde{s}=\tilde{z}\cdot \tilde{h}$, where $\tilde{z}\in \ZG(\tilde{L})^{\circ}$ and $\tilde{h}\in [\tilde{L},\tilde{L}]$, and, by  the structure of $\tilde{V}\mid_{[\tilde{L},\tilde{L}]}$ and Propositions \ref{PropositionDlwedge}, \ref{PropositionDlwedgecubep=2}, \ref{PropositionDlwedgecubepneq2}, and \ref{D7om4}, we determine that $\dim(\tilde{V}_{\tilde{s}}(\tilde{\mu}))\leq 2\dim((L_{\tilde{L}}(\tilde{\omega}_{4}))_{\tilde{h}}(\tilde{\mu}_{\tilde{h}}))+\dim((L_{\tilde{L}}(\tilde{\omega}_{5}))_{\tilde{h}}(\tilde{\mu}_{\tilde{h}}))+(1-\varepsilon_{p}(2))\dim((L_{\tilde{L}}(\tilde{\omega}_{3}))_{\tilde{h}}(\tilde{\mu}_{\tilde{h}}))\leq 1172-250\varepsilon_{p}(2)$ for all eigenvalues $\tilde{\mu}$ of $\tilde{s}$ on $\tilde{V}$. Therefore, $\scale[0.9]{\displaystyle \max_{\tilde{s}\in \tilde{T}\setminus\ZG(\tilde{G})}\dim(\tilde{V}_{\tilde{s}}(\tilde{\mu}))}\leq 1172-250\varepsilon_{p}(2)$. 

For the unipotent elements, by Lemma \ref{uniprootelems}, the structure of $\tilde{V}\mid_{[\tilde{L},\tilde{L}]}$ and Propositions \ref{PropositionDlwedge}, \ref{PropositionDlwedgecubep=2}, \ref{PropositionDlwedgecubepneq2}, and \ref{D7om4}, we have $\scale[0.9]{\displaystyle \max_{\tilde{u}\in \tilde{G}_{u}\setminus \{1\}}\dim(\tilde{V}_{\tilde{u}}(1))=}$ $\dim(\tilde{V}_{x_{\tilde{\alpha}_{8}}(1)}(1))\leq 1224-182\varepsilon_{p}(2)$, where equality holds for $p\neq 2$. Lastly, we note that $\nu_{\tilde{G}}(\tilde{V})\geq 596-56\varepsilon_{p}(2)$, where equality holds for $p\neq 2$.
\end{proof}

\begin{prop}\label{D8om1+om8}
Let $\ell=8$ and $\tilde{V}=L_{\tilde{G}}(\tilde{\omega}_{1}+\tilde{\omega}_{8})$. Then $\nu_{\tilde{G}}(\tilde{V})\geq 508-142\varepsilon_{p}(2)$, where equality holds for $p\neq 2,3,5,7$. Moreover, we have $\scale[0.9]{\displaystyle \max_{\tilde{u}\in \tilde{G}_{u}\setminus \{1\}}\dim(\tilde{V}_{\tilde{u}}(1))}\leq 1312-96\varepsilon_{p}(2)$, where equality holds for $p\neq 2,3,5,7$, and $\scale[0.9]{\displaystyle \max_{\tilde{s}\in \tilde{T}\setminus\ZG(\tilde{G})}\dim(\tilde{V}_{\tilde{s}}(\tilde{\mu}))}\leq 1216-16\varepsilon_{p}(5)-208\varepsilon_{p}(2)$.
\end{prop}

\begin{proof}
Let $\tilde{\lambda}=\tilde{\omega}_{1}+\tilde{\omega}_{8}$ and $\tilde{L}=\tilde{L}_{1}$. By Lemma \ref{weightlevelDl}, we have $e_{1}(\tilde{\lambda})=3$, therefore $\displaystyle \tilde{V}\mid_{[\tilde{L},\tilde{L}]}=\tilde{V}^{0}\oplus \cdots \oplus \tilde{V}^{3}$. We argue as we did in the proof of Proposition \ref{D7om1+om7} to show that $\tilde{V}^{0}\cong L_{\tilde{L}}(\tilde{\omega}_{8})$, $\tilde{V}^{3}\cong L_{\tilde{L}}(\tilde{\omega}_{7})$, and that $\tilde{V}^{1}$, respectively $\tilde{V}^{2}$, has $2+\varepsilon_{p}(7)-\varepsilon_{p}(2)$ composition factors: one isomorphic to $L_{\tilde{L}}(\tilde{\omega}_{2}+\tilde{\omega}_{8})$, respectively to $L_{\tilde{L}}(\tilde{\omega}_{2}+\tilde{\omega}_{7})$, and $1+\varepsilon_{p}(7)-\varepsilon_{p}(2)$ isomorphic to $L_{\tilde{L}}(\tilde{\omega}_{7})$, respectively to $L_{\tilde{L}}(\tilde{\omega}_{8})$. Moreover, when $p\neq 7$, by \cite[II.2.14]{Jantzen_2007representations}, we have $\tilde{V}^{1}\cong L_{\tilde{L}}(\tilde{\omega}_{2}+\tilde{\omega}_{8})\oplus L_{\tilde{L}}(\tilde{\omega}_{7})^{1-\varepsilon_{p}(2)}$ and $\tilde{V}^{2}\cong L_{\tilde{L}}(\tilde{\omega}_{2}+\tilde{\omega}_{7})\oplus L_{\tilde{L}}(\tilde{\omega}_{8})^{1-\varepsilon_{p}(2)}$.

We begin with the semisimple elements. Let $\tilde{s}\in \tilde{T}\setminus \ZG(\tilde{G})$. If $\dim(\tilde{V}^{i}_{\tilde{s}}(\tilde{\mu}))=\dim(\tilde{V}^{i})$ for some eigenvalue $\tilde{\mu}$ of $\tilde{s}$ on $\tilde{V}$, where $0\leq i\leq 3$, then $\tilde{s}\in \ZG(\tilde{L})^{\circ}\setminus \ZG(\tilde{G})$. In this case, as $\tilde{s}$ acts on each $\tilde{V}^{i}$ as scalar multiplication by $c^{3-2i}$ and $c^{2}\neq 1$, it follows that $\dim(\tilde{V}_{\tilde{s}}(\tilde{\mu}))\leq 960-128\varepsilon_{p}(2)$ for all eigenvalues $\tilde{\mu}$ of $\tilde{s}$ on $\tilde{V}$. We thus assume that $\dim(\tilde{V}^{i}_{\tilde{s}}(\tilde{\mu}))<\dim(\tilde{V}^{i})$ for all eigenvalues $\tilde{\mu}$ of $\tilde{s}$ on $\tilde{V}$ and all $0\leq i\leq 3$. We write $\tilde{s}=\tilde{z}\cdot \tilde{h}$, where $\tilde{z}\in \ZG(\tilde{L})^{\circ}$ and $\tilde{h}\in [\tilde{L},\tilde{L}]$, and, by  the structure of $\tilde{V}\mid_{[\tilde{L},\tilde{L}]}$ and Propositions \ref{Dloml} and \ref{D7om1+om7}, we determine that $\dim(\tilde{V}_{\tilde{s}}(\tilde{\mu}))\leq (4+2\varepsilon_{p}(7)-2\varepsilon_{p}(2))\dim((L_{\tilde{L}}(\tilde{\omega}_{7}))_{\tilde{h}}(\tilde{\mu}_{\tilde{h}}))+2\dim((L_{\tilde{L}}(\tilde{\omega}_{2}+\tilde{\omega}_{7}))_{\tilde{h}}(\tilde{\mu}_{\tilde{h}}))\leq 1216-16\varepsilon_{p}(5)-208\varepsilon_{p}(2)$ for all eigenvalues $\tilde{\mu}$ of $\tilde{s}$ on $\tilde{V}$. Therefore, $\scale[0.9]{\displaystyle \max_{\tilde{s}\in \tilde{T}\setminus\ZG(\tilde{G})}\dim(\tilde{V}_{\tilde{s}}(\tilde{\mu}))}\leq 1216-16\varepsilon_{p}(5)-208\varepsilon_{p}(2)$. 

For the unipotent elements, by Lemma \ref{uniprootelems}, the structure of $\tilde{V}\mid_{[\tilde{L},\tilde{L}]}$ and Propositions \ref{Dloml} and \ref{D7om1+om7}, we have $\scale[0.9]{\displaystyle \max_{\tilde{u}\in \tilde{G}_{u}\setminus \{1\}}\dim(\tilde{V}_{\tilde{u}}(1))=}$ $\dim(\tilde{V}_{x_{\tilde{\alpha}_{8}}(1)}(1))\leq 1312-96\varepsilon_{p}(2)$, where equality holds for $p\neq 2,3,5,7$. Lastly, we note that $\nu_{\tilde{G}}(\tilde{V})\geq 508-142\varepsilon_{p}(2)$, where equality holds for $p\neq 2,3,5,7$.
\end{proof}

\begin{prop}\label{D9om4p=2}
Let $p=2$, $\ell=9$ and $\tilde{V}=L_{\tilde{G}}(\tilde{\omega}_{4})$. Then $\nu_{\tilde{G}}(\tilde{V})\geq 542$. Moreover, we have $\scale[0.9]{\displaystyle \max_{\tilde{u}\in \tilde{G}_{u}\setminus \{1\}}\dim(\tilde{V}_{\tilde{u}}(1))}$ $\leq 2364$ and $\scale[0.9]{\displaystyle \max_{\tilde{s}\in \tilde{T}\setminus\ZG(\tilde{G})}\dim(\tilde{V}_{\tilde{s}}(\tilde{\mu}))}\leq 1824$.
\end{prop}

\begin{proof}
First, we note that $\tilde{V}\cong L_{G}(\omega_{4})$, as $k\tilde{G}$-modules, Further, as $p=2$, by \cite[Table $1$]{seitz1987maximal}, we have $L_{G}(\omega_{4})\cong L_{H}(\omega_{4}^{H})$ as $kG$-modules. Thus, in view of Lemma \ref{LtildeGtildelambdaandLGlambda} and by Proposition \ref{C9om4}, it follows that $\scale[0.9]{\displaystyle \max_{\tilde{s}\in \tilde{T}\setminus\ZG(\tilde{G})}\dim(\tilde{V}_{\tilde{s}}(\tilde{\mu}))}$ $\scale[0.9]{\displaystyle=\max_{s_{H}\in T_{H}\setminus\ZG(H)}\dim((L_{H}(\omega_{4}^{H}))_{s_{H}}(\mu_{H}))}\leq 1824$ and $\scale[0.9]{\displaystyle \max_{\tilde{u}\in \tilde{G}_{u}\setminus \{1\}}\dim(\tilde{V}_{\tilde{u}}(1))}$ $\scale[0.86]{\displaystyle=\max_{u\in H_{u}\setminus \{1\}}\dim((L_{H}(\omega_{4}^{H})_{u}(1))}$ $\leq 2364$. Lastly, we note that $\nu_{\tilde{G}}(\tilde{V})\geq 542$.
\end{proof}

\section*{Acknowledgments}
The author is immensely grateful to Donna Testerman for her guidance. The author would also like to thank Simon Goodwin, Martin Liebeck and Adam Thomas for many helpful discussions and comments. This work was supported by the Swiss National Science Foundation, grant number FNS 200020 175571.

\bibliography{Article_v3}

\begin{thebibliography}{Hum72}

\bibitem[Car89]{carter1989simple}
R.W. Carter.
\newblock {\em {Simple groups of Lie type}}.
\newblock Wiley Classics Library. Wiley, 1989.

\bibitem[GL06]{GowLaffey_06}
R.~Gow and T.J. Laffey.
\newblock {On the decomposition of the exterior square of an indecomposable
  module of a cyclic p-group}.
\newblock {\em Journal of Group Theory}, 9(5):659--672, 2006.

\bibitem[GL19]{GurLaw19}
Robert~M. Guralnick and R.~Lawther.
\newblock Generic stabilizers in actions of simple algebraic groups i: modules
  and the first grassmanian varieties.
\newblock {\em arXiv: Group Theory}, 2019.

\bibitem[Gor91]{Gor90}
N.L. Gordeev.
\newblock {Coranks of elements of linear groups and the complexity of algebras
  of invariants}.
\newblock {\em Leningrad Mathematical Journal}, 2:245--267, 1991.

\bibitem[GS03]{GuralnickSaxl_2003}
R.M. Guralnick and J.~Saxl.
\newblock {Generation of finite almost simple groups by conjugates}.
\newblock {\em Journal of Algebra}, 268(2):519--571, 2003.

\bibitem[HLS92]{HLS_92}
J.I. Hall, M.W. Liebeck, and G.M. Seitz.
\newblock {Generators for finite simple groups, with applications to linear
  groups}.
\newblock {\em The Quarterly Journal of Mathematics}, 43(4):441--458, 12 1992.

\bibitem[Hum72]{humphreys_1972introduction}
J.E. Humphreys.
\newblock {\em {Introduction to Lie algebras and representation theory}}.
\newblock Graduate Texts in Mathematics. Springer, 1972.

\bibitem[Jan07]{Jantzen_2007representations}
J.C. Jantzen.
\newblock {\em {Representations of algebraic groups}}.
\newblock Mathematical surveys and monographs. American Mathematical Society,
  2007.

\bibitem[KM97]{Kemper1997}
G.~Kemper and G.~Malle.
\newblock The finite irreducible linear groups with polynomial ring of
  invariants.
\newblock {\em Transformation Groups}, 2:57--89, 1997.

\bibitem[Kor19]{Korhonen_2019}
M.~Korhonen.
\newblock {Jordan blocks of unipotent elements in some irreducible
  representations of classical groups in good characteristic}.
\newblock {\em Proceedings of the American Mathematical Society},
  147(10):4205–4219, 2019.

\bibitem[Kor20]{Korhonen_2020HesselinkNF}
M.~Korhonen.
\newblock {Hesselink normal forms of unipotent elements in some representations
  of classical groups in characteristic two}.
\newblock {\em Journal of Algebra}, 559:268--319, 2020.

\bibitem[KW82]{Kac1982}
V.G. Kac and K.~Watanabe.
\newblock Finite linear groups whose ring of invariants is a complete
  intersection.
\newblock {\em Bulletin of the American Mathematical Society}, 6:221--223,
  1982.

\bibitem[LS12]{liebeck_2012unipotent}
M.W. Liebeck and G.M. Seitz.
\newblock {\em {Unipotent and nilpotent classes in simple algebraic groups and
  Lie algebras}}.
\newblock Mathematical surveys and monographs. American Mathematical Society,
  2012.

\bibitem[Lü01a]{Lubeck_2001}
F.~Lübeck.
\newblock {Small degree representations of finite Chevalley groups in defining
  characteristic}.
\newblock {\em LMS Journal of Computation and Mathematics}, 4:135–169, 2001.

\bibitem[Lü01b]{LuTables}
F.~Lübeck.
\newblock {Tables of weight multiplicities}, 2001.

\bibitem[Mar18]{Martinez2018}
Álvaro~L. Martínez.
\newblock Low-dimensional irreducible rational representations of classical
  algebraic groups, 2018.

\bibitem[McN98]{mcninch_1998}
G.J. McNinch.
\newblock {Dimensional criteria for semisimplicity of representations}.
\newblock {\em Proceedings of the London Mathematical Society}, 76(1):95–149,
  1998.

\bibitem[MT11]{Malle_testerman_2011}
G.~Malle and D.~Testerman.
\newblock {\em {Linear algebraic groups and finite groups of Lie type}}.
\newblock Cambridge Studies in Advanced Mathematics. Cambridge University
  Press, 2011.

\bibitem[Sei87]{seitz1987maximal}
G.M. Seitz.
\newblock {The maximal subgroups of classical algebraic groups}.
\newblock (365), 1987.

\bibitem[Smi82]{Smith_82}
S.D. Smith.
\newblock {Irreducible modules and parabolic subgroups}.
\newblock {\em Journal of Algebra}, 75(1):286--289, 1982.

\bibitem[Ste16]{Steinberg:2016}
R.~Steinberg.
\newblock {\em {Lectures on Chevalley groups}}.
\newblock University lecture series volume 66. American Mathematical Society,
  2016.

\bibitem[Sup83]{suprunenko1983preservation}
I.~D. Suprunenko.
\newblock {Preservation of systems of weights of irreducible representations of
  an algebraic group and a Lie algebra of type A with bounded higher weights in
  reduction modulo p}.
\newblock {\em Vestsi Akad. Na uk BSSR, Ser. Fiz.-Mat. Na uk}, (2):18--22,
  1983.

\bibitem[Ver99]{Verbitsky1999}
M.~Verbitsky.
\newblock Holomorphic symplectic geometry and orbifold singularities.
\newblock {\em Asian Journal of Mathematics}, 4:553--564, 1999.

\end{thebibliography}
\bibliographystyle{alpha}

\end{document}